\documentclass[12pt]{amsart}
\usepackage{amssymb
,amsthm
,amsmath
,amscd
,mathtools
,mathdots
,url
,fancybox
}
\usepackage[all]{xy}
\usepackage[top=35truemm,bottom=35truemm,left=30truemm,right=30truemm]{geometry}
\usepackage[usenames]{color}
\usepackage[all]{xy}

\usepackage[
 setpagesize=false,
 bookmarks=true,
 bookmarksdepth=tocdepth,
 bookmarksnumbered=true,
 colorlinks=false,
 pdftitle={},
 pdfsubject={},
 pdfauthor={},
 pdfkeywords={}
]{hyperref}

\allowdisplaybreaks

\newcommand{\A}{\mathbb{A}}
\newcommand{\C}{\mathbb{C}}
\newcommand{\Q}{\mathbb{Q}}
\newcommand{\R}{\mathbb{R}}
\newcommand{\Z}{\mathbb{Z}}
\newcommand{\F}{\mathbb{F}}

\renewcommand{\AA}{\mathcal{A}}
\newcommand{\FF}{\mathcal{F}}
\newcommand{\GG}{\mathcal{G}}
\newcommand{\HH}{\mathcal{H}}
\newcommand{\II}{\mathcal{I}}
\newcommand{\RR}{\mathcal{R}}
\newcommand{\VV}{\mathcal{V}}
\newcommand{\Sc}{\mathcal{S}}

\renewcommand{\aa}{\mathfrak{a}}
\newcommand{\nn}{\mathfrak{n}}
\newcommand{\mm}{\mathfrak{m}}
\renewcommand{\ss}{\mathfrak{s}}
\renewcommand{\tt}{\mathfrak{t}}
\newcommand{\g}{\mathfrak{g}}
\newcommand{\oo}{\mathfrak{o}}
\newcommand{\pp}{\mathfrak{p}}
\newcommand{\rr}{\mathfrak{r}}
\newcommand{\uu}{\mathfrak{u}}
\newcommand{\ww}{\mathfrak{w}}
\newcommand{\ee}{\mathfrak{e}}
\newcommand{\zz}{\mathfrak{z}}

\renewcommand{\SS}{\mathfrak{S}}
\newcommand{\NN}{\mathfrak{N}}

\newcommand{\GL}{\mathrm{GL}}
\newcommand{\SL}{\mathrm{SL}}
\renewcommand{\O}{\mathrm{O}}
\newcommand{\SO}{\mathrm{SO}}
\newcommand{\Sp}{\mathrm{Sp}}
\newcommand{\U}{\mathrm{U}}
\newcommand{\M}{\mathrm{M}}
\newcommand{\SU}{\mathrm{SU}}

\newcommand{\Speh}{\mathrm{Speh}}
\newcommand{\Rep}{\mathrm{Rep}}
\newcommand{\Irr}{\mathrm{Irr}}
\newcommand{\temp}{\mathrm{temp}}
\newcommand{\tr}{\mathrm{tr}}
\newcommand{\Ind}{\mathrm{Ind}}
\newcommand{\Jac}{\mathrm{Jac}}
\newcommand{\Hom}{\mathrm{Hom}}
\newcommand{\End}{\mathrm{End}}
\newcommand{\unit}{\mathrm{unit}}

\newcommand{\St}{\mathrm{St}}
\newcommand{\good}{\mathrm{good}}
\newcommand{\bad}{\mathrm{bad}}
\newcommand{\im}{\mathrm{Im}}
\newcommand{\Cent}{\mathrm{Cent}}
\newcommand{\Sym}{\mathrm{Sym}}
\newcommand{\Asai}{\mathrm{Asai}}
\newcommand{\Supp}{\mathrm{Supp}}
\newcommand{\ev}{\mathrm{ev}}
\newcommand{\Trans}{\mathrm{Trans}}
\newcommand{\Lie}{\mathrm{Lie}}
\newcommand{\Res}{\mathrm{Res}}

\newcommand{\Gal}{\mathrm{Gal}}
\newcommand{\Rat}{\mathrm{Rat}}
\newcommand{\Ad}{\mathrm{Ad}}
\newcommand{\ad}{\mathrm{ad}}
\newcommand{\Sh}{\mathrm{Sh}}
\newcommand{\diag}{\mathrm{diag}}
\newcommand{\inv}{\mathrm{inv}}
\newcommand{\der}{\mathrm{der}}
\newcommand{\gen}{\mathrm{gen}}

\newcommand{\ab}{\mathrm{ab}}
\newcommand{\re}{\mathrm{Re}}
\newcommand{\sym}{\mathrm{sym}}
\newcommand{\st}{\mathrm{st}}
\renewcommand{\sc}{\mathrm{sc}}
\newcommand{\ari}{\mathrm{ari}}
\newcommand{\geo}{\mathrm{geo}}
\newcommand{\Frob}{\mathrm{Frob}}
\newcommand{\Br}{\mathrm{Br}}
\newcommand{\rec}{\mathrm{rec}}
\newcommand{\cts}{\mathrm{cts}}
\newcommand{\Tate}{\mathrm{Tate}}
\newcommand{\new}{\mathrm{new}}
\newcommand{\disc}{\mathrm{disc}}

\newcommand{\Ksph}{K\text{-}\mathrm{sph}}
\newcommand{\Calg}{\mathbb{C}\text{-}\mathrm{alg}}
\newcommand{\Int}{\mathrm{Int}}

\newcommand{\resp}{resp.~}
\renewcommand{\1}{\mathbf{1}}
\newcommand{\ep}{\varepsilon}
\newcommand{\bs}{\backslash}
\newcommand{\id}{\mathrm{id}}
\newcommand{\iif}{&\quad&\text{if }}
\newcommand{\other}{&\quad&\text{otherwise}}
\newcommand{\I}{\sqrt{-1}}
\newcommand{\spl}{\boldsymbol{spl}}

\newcommand{\half}[1]{\frac{#1}{2}}
\newcommand{\tl}[1]{\widetilde{#1}}
\newcommand{\pair}[1]{\langle #1 \rangle}

\def\<{\langle}
\def\>{\rangle}

\theoremstyle{plain}
\newtheorem{thm}{Theorem}[subsection]
\newtheorem{lem}[thm]{Lemma}
\newtheorem{prop}[thm]{Proposition}
\newtheorem{cor}[thm]{Corollary}
\newtheorem{defi}[thm]{Definition}
\newtheorem{conj}[thm]{Conjecture}

\theoremstyle{definition}
\newtheorem{rem}[thm]{Remark}
\newtheorem{ex}[thm]{Example}
\newtheorem{hyp}[thm]{Hypothesis}

\title[Local Intertwining Relations and Co-tempered $A$-packets]
{Local Intertwining Relations and \\ Co-tempered $A$-packets of Classical Groups}
\author[
H. Atobe, W. T. Gan, A. Ichino, T. Kaletha, A. M{\'i}nguez, S. W. Shin
]
{Hiraku Atobe, 
Wee Teck Gan, 
Atsushi Ichino, 
\\
Tasho Kaletha, 
Alberto M\'inguez
\and
Sug Woo Shin
}

\address{
Department of Mathematics, Kyoto University, 
Kitashirakawa-Oiwake-cho, Sakyo-ku, Kyoto, 606-8502, Japan
}
\email{
atobe@math.kyoto-u.ac.jp
}

\address{
Department of Mathematics, National University of Singapore, 
10 Lower Kent Ridge Road, Singapore 119076
}
\email{
matgwt@nus.edu.sg
}

\address{
Department of Mathematics, Kyoto University, 
Kitashirakawa-Oiwake-cho, Sakyo-ku, Kyoto, 606-8502, Japan
}
\email{
ichino@math.kyoto-u.ac.jp
}

\address{
Mathematics Center, University of Bonn, Endenicher Allee 60, Bonn, Germany
}
\email{
kaletha@math.uni-bonn.de
}

\address{
Departamento de Algebra and Instituto de Matematicas (IMUS), Universidad de Sevilla, C/ Tarfia s/n, 41012 Sevilla, Espa\~{n}a
/
University of Vienna,
Fakult{\"a}t f{\"u}r Mathematik,
Oskar-Morgenstern-Platz 1,
1090 Wien, Austria
}
\email{
aminguez@us.es
}

\address{
Department of Mathematics, UC Berkeley, 
Berkeley, CA 94720, USA 
/ 
Korea Institute for Advanced Study, 
Seoul 02455, Republic of Korea
}
\email{
sug.woo.shin@berkeley.edu
}

\begin{document}
\maketitle


\begin{abstract}
The local intertwining relation is an identity that gives precise information 
about the action of normalized intertwining operators on parabolically induced representations. 
We prove several instances of the local intertwining relation for quasi-split classical groups 
and the twisted general linear group, 
as they are required in the inductive proof of the endoscopic classification 
for quasi-split classical groups due to Arthur and Mok. 
In addition, we construct the co-tempered local $A$-packets by Aubert duality 
and verify their key properties by purely local means, 
thereby providing the seed cases needed as an input to the inductive proof. 
Together with further technical results that we establish, 
these results make the endoscopic classification 
conditional only on the validity of the twisted weighted fundamental lemma.
\end{abstract}

\tableofcontents

\section*{Introduction}
The theory of automorphic forms is a profound topic 
in number theory with broad applications to other areas of mathematics and theoretical physics. 
A fundamental problem, which is central in the Langlands program, 
is to classify automorphic representations of connected reductive groups $G$ over global fields 
and to obtain the analogous classification over local fields 
in terms of parameters pertaining to the Langlands $L$-group ${}^LG$. 
Such a classification should be consistent with the Langlands functoriality conjecture, 
whose rough form posits that a morphism of $L$-groups ${}^LH \rightarrow {}^LG$ 
should induce a functorial lifting of representations from $H$ to $G$, 
provided that $G$ is quasi-split.
The functoriality conjecture is beyond our current technology in general, 
but some special cases fit in the framework of endoscopy 
\`a la Langlands further developed by Arthur, Clozel, Kottwitz, Labesse, Shelstad, and others. 
\par

When $G$ is a quasi-split classical group, 
one can hope to study the classification problem for $G$ 
by relating it to the general linear groups and to quasi-split classical groups of smaller rank, 
via (ordinary and twisted) endoscopy. 
This was achieved in Arthur's book \cite{Ar}, which represents a crowning
achievement in the endoscopic approach to functoriality, and builds on
tremendous foundational work on the trace formula and related matters
spanning multiple decades. The results of \cite{Ar} were later extended to
quasi-split unitary groups by Mok \cite{Mok}, following the same arguments.
These results were partially extended to non-quasi-split classical groups in \cite{KMSW} and \cite{Ish}. 
Since classical groups are ubiquitous, 
the endoscopic classification for them has played an indispensable role 
in a number of arithmetic applications such as: 
\begin{itemize}
\item
new instances of the global Langlands reciprocity, 
see \cite[Section 5.1]{Sc2}, \cite[Section 1.4.1]{BCGP}, \cite{KS1, KS2};

\item
the $p$-adic Gross-Zagier formula and the Beilinson--Bloch--Kato conjecture, 
see \cite{DL}, \cite{LL}, \cite{LTXZZ}; 

\item
Euler systems, 
see \cite{GS}, \cite{LSk}, \cite{LTX};

\item
the Gan--Gross--Prasad conjecture, the Ichino--Ikeda conjecture
and their local analogues, 
see \cite{W6}, \cite{BP1, BP2, BP3}, \cite[Remark 1.7]{BPLZZ}, \cite{BPCZ}, \cite{BPC}; 

\item
the Sarnak--Xue density conjecture, 
see \cite{DGG}, \cite{EGGG}; 

\item
an extension of the Shimura--Waldspurger correspondence, 
see \cite{GI3}, \cite{Li}; 

\item
classification/counting of irreducible algebraic cuspidal automorphic
representations of $\GL_N$ or classical groups over $\Q$ of level one, 
see \cite{CR}, \cite{Tai}, \cite{CL}, \cite{CT};

\item
Harder's conjecture, 
see \cite{CL}, \cite{ACIKY1, ACIKY2}.
\end{itemize}
\par

The main theorems of \cite{Ar} and \cite{Mok} depend on several results 
which were unproven but expected at the time. 
Some of the results have become available over the past ten years. 
The most notable is the stabilization of the twisted trace formula
by M{\oe}glin--Waldspurger \cite{MW_Stab1, MW_Stab2},  
assuming the validity of the twisted weighted fundamental lemma; 
the latter is as yet not available and is discussed further in Section \ref{sec.TFL} below. 
M{\oe}glin--Waldspurger also established the local twisted trace formula in \cite{MW_LT}, 
which is one of the vital ingredients in \cite[Chapter 6]{Ar}.
The remaining issues were to be resolved in the projected papers by Arthur, 
which are named as \cite{A24, A25, A26, A27} in \cite{Ar}, 
at least in the symplectic and orthogonal cases. 
The problem to be addressed by \cite{A24} has been solved in \cite{MW_Stab1},
whereas the other three references \cite{A25, A26, A27} have not been treated yet.
For \cite{A24}, see also Section \ref{sec.A24} below.
The problems to be covered by \cite{A25} appear to be particularly challenging, 
and the Hecke algebra method mentioned below \cite[Lemma 7.1.2]{Ar} 
leads to rather complicated calculations 
that are delicate even in the (supposedly simplest) case pertaining to the Iwahori Hecke algebras.
\par

The goal of this paper is to prove all unproven assertions that \cite{Ar} and \cite{Mok} rely on, 
apart from the twisted weighted fundamental lemma, 
uniformly for quasi-split symplectic, special orthogonal, and unitary groups. 
Our three main theorems correspond to 
what should be expected from \cite{A27}, \cite{A26}, and \cite{A25}, respectively, 
including their analogues for unitary groups. 
In addition, we justify a few other technical results that are used in \cite{Ar} and \cite{Mok} 
without explicit reference. 
It is worth noting that we develop a novel method for \cite{A25} and \cite{A26} 
that is based on a careful study of local intertwining operators and is quite different from the approaches suggested by Arthur in \cite{Ar}. 
As a consequence of this paper, 
the main endoscopic classification for quasi-split classical groups will become unconditional 
as soon as the twisted weighted fundamental lemma is fully verified. 
The wide range of applications resting on the classification will also become unconditional.
\par

We remark that a weak global Langlands functoriality from split classical groups 
(as well as split general spin groups) to general linear groups 
was established by Cai--Friedberg--Kaplan \cite{CFK2} 
via doubling constructions and the converse theorem, 
extending the work by Cogdell--Kim--Piatetski-Shapiro--Shahidi 
et al.~(see \cite{CPSS} and the references therein) 
in the globally generic case. 
Even though the trace formula method leads to more precise results 
that are suitable for broader applications, 
their approach is unconditional and less demanding in terms of prerequisites 
because it avoids the trace formula. 
On the other hand, it is possible to deduce a weak global Langlands functoriality from 
(not necessarily quasi-split) classical groups to general linear groups 
by the trace formula in a way much softer than \cite{Ar} and \cite{Mok}, 
and conditional only on the twisted weighted fundamental lemma; 
see \cite{Shi}, cf.~\cite[Proposition 9.5.2]{Ar}.

\subsection{Context}\label{sec.context}
Now we partially review the outline of Arthur's inductive argument in \cite{Ar} to put our work in context. 
(The same applies to \cite{Mok} regarding unitary groups. 
Since the structure of \cite{Mok} closely follows that of \cite{Ar}, our review focuses on \cite{Ar}.)
We may organize the main theorems in \cite{Ar} as follows using the numbering from there.
\begin{itemize}
    \item Local classification theorems: Theorems 1.5.1, 2.2.1, 2.2.4.
    \item Local intertwining relations (LIR): Theorems 2.4.1, 2.4.4.
    \item Global seed theorems: Theorems 1.4.1, 1.4.2.
    \item Global classification theorems: Theorems 1.5.2, 1.5.3.
    \item Global stable multiplicity formulas: Theorems 4.1.2, 4.2.2.
\end{itemize}
The local classification includes the local Langlands correspondence, 
construction and internal parametrization of $A$-packets, 
and the endoscopic character relations. 
The global seed theorems support the formalism of Arthur's global parameters for classical groups; 
the global classification includes Arthur's multiplicity formula 
as well as the dichotomy for self-dual cuspidal automorphic representations of $\GL_N$, 
with $N$ even, into symplectic and orthogonal types.
The above theorems are proven simultaneously by induction on a positive integer $N$ 
for quasi-split symplectic and special orthogonal groups 
which are twisted endoscopic groups for $\GL_N$. 
\par

The most crucial ingredient of the proof is the stabilization of the trace formulas
for quasi-split classical groups and twisted general linear groups. 
Essential for its use is a precise understanding of the intertwining operators 
appearing in the trace formulas. 
Our main theorems corresponding to \cite{A26} and \cite{A27} 
(announced in \cite[Section 2.5]{Ar}) play this role by
clarifying the relationship between normalized intertwining operators and Whittaker models, 
as well as the part of \cite{A25} pertaining to LIR for $A$-parameters.
\par

In addition, the results corresponding to \cite{A25} serve as the cornerstone for Chapter 7 of \cite{Ar}. 
Let us provide more details. 
Arthur obtains the local theorems for tempered representations 
and bounded Langlands parameters, a.k.a., tempered $L$-parameters, 
by the end of Chapter 6 of \cite{Ar}. 
Chapter 7 is devoted to the local classification and LIR for non-tempered $A$-parameters. 
Building on the local results of Chapters 6 and 7, 
Arthur finishes the proof of the global theorems, 
thereby completing the inductive argument, in Chapter 8.
\par

Arthur's strategy in Chapter 7 broadly consists of two steps:
\begin{description}
\item[Step 1] 
Handle a certain class of $A$-parameters over $p$-adic fields by a local method.
\item[Step 2] 
Prove the general case via globalization.
\end{description}
\par

Step 2 propagates the results from Step 1 by carefully globalizing a given local $A$-parameter 
such that the local $A$-parameters at all the other places, 
apart from the place of interest, 
are essentially understood. 
With that said, we concentrate on Step 1 as this is what our main theorem is about.
\par

A key input in Step 1 is Aubert duality, 
which is defined on local $A$-parameters 
as well as on irreducible representations of $p$-adic reductive groups. 
On $A$-parameters, it is induced by simply permuting the two $\SL_2$-factors in the source group.
On representations, 
Aubert duality is defined in terms of parabolic inductions composed with Jacquet modules; 
for example, the trivial representation and the Steinberg representation are Aubert-dual to each other, 
and each supercuspidal representation is its own Aubert dual. 
Aubert duality on the two sides should be compatible with each other through the local classification. 
This is indeed shown to be so by the end of Chapter 7 in \cite{Ar}. 
\par

Arthur's strategy for Step 1 is to turn the tables around. 
Since the local theorems are known for tempered $L$-parameters, 
we can hope to construct $A$-packets and prove the local theorems via Aubert duality 
when the $A$-parameters are co-tempered, 
 i.e., when they are Aubert-dual to tempered $L$-parameters, 
 provided that we understand how Aubert duality interacts with the local classification and LIR. 
This is exactly what our third main theorem achieves.
\par

In \cite{Ar} the counterpart of this theorem is stated as Lemma 7.1.2 
and the penultimate paragraph on p.~428, whose proof was deferred to \cite{A25}. 
In fact Arthur asserted only a weaker statement for certain co-tempered $A$-parameters, 
which nevertheless suffices for Step 2 above. 
On the other hand, our method seems robust and optimal 
in that it allows us to deal with all co-tempered $A$-parameters. 
\par

\subsection{Results, proofs and organization}
Let us describe our results. Let $F$ be a local field of characteristic zero with the local Langlands group $L_F$, 
and let either $G = \GL_N(F)$ equipped with a non-trivial pinned outer automorphism $\theta$ 
(which coincides with $g \mapsto {}^tg^{-1}$ up to an inner automorphism), 
or let $G$ be a symplectic group $\Sp_{2n}(F)$ over $F$; for simplicity, in this introduction, we will not discuss orthogonal
or unitary groups, which require more notation. 
See also Section \ref{sec.main}, 
where we review Arthur's results and state our three main theorems in the general setting.
\par

Fix a proper standard parabolic subgroup $P=MN_P$ of $G$, 
and an $A$-parameter 
\[
\psi_M \colon L_F \times \SL_2(\C) \rightarrow {}^LM
\] 
for the Levi component $M$.
As an induction hypothesis, we assume that we have the $A$-packet $\Pi_{\psi_M}$, 
which is a multi-set of irreducible unitary representations of $M$. 
Let $w$ be a (twisted) Weyl element of $G$ that preserves $M$, i.e., an element of $N_G(M)/M$ when $G$ is the symplectic group, or an element of $N_{G \rtimes \theta}(M)/M$ when $G$ is the general linear group. After fixing some auxiliary data, 
for $\pi_M \in \Pi_{\psi_M}$, Arthur defines a normalized intertwining operator 
\[
R_P(w, \pi_M, \psi_M) \colon I_P(\pi_M) \rightarrow I_P(w\pi_M), 
\] 
where $I_P(\pi_M)$ is the normalized parabolic induction of $\pi_M$ and $(w\pi_M)(m) = \pi_M(\tl{w}^{-1}m\tl{w})$ with a carefully chosen representative $\tl{w}$ of $w$ (see Section \ref{sec.NIO} for more details). Our three main theorems concern these normalized intertwining operators.
\par

\subsubsection{Main Theorem 1}
In the first main theorem (Theorem \ref{main1}), 
we consider the tempered and generic case. 
Let $\psi_M = \phi_M$ be a tempered $L$-parameter, which means that $\phi_M|_{\{\1_{L_F}\} \times \SL_2(\C)} = \1$, 
and let $\pi_M$ be the generic representation lying in $\Pi_{\phi_M}$.
In this case, from a fixed Whittaker functional on $\pi_M$, 
one can get a Whittaker functional $\Omega(\pi_M)$ of $I_P(\pi_M)$ by the Jacquet integral 
(see Section \ref{sec.A27}). 
Then Theorem \ref{main1} claims that 
\[
\Omega(w\pi_M) \circ R_P(w, \pi_M, \phi_M) = \Omega(\pi_M) 
\]
if $w\pi_M \cong \pi_M$. 
This is \cite[Theorem 2.5.1 (b)]{Ar}, whose proof was deferred to \cite{A27}.
Note that a similar result is proven by Shahidi \cite{Sh7}, 
but Shahidi's definition of $R_P(w,\pi_M,\phi_M)$ is not the same as Arthur's because it uses different normalizing factors. 
As suggested in the proof of \cite[(2.5.5)]{Ar}, 
Theorem \ref{main1} will be proven by comparing Arthur's normalization factors 
with Shahidi's local coefficients. 
This is done in Section \ref{sec.main1}. 
In particular, if $G$ is a classical group, we need the coincidence of Shahidi's gamma factors (constructed in terms of representations of reductive groups) with those defined by Artin, Deligne, and Langlands (constructed in terms of Galois representations).
This is well-known to experts, but for completeness, 
we give a proof of this fact in Section \ref{sec.gamma}.

Another result whose proof was deferred to \cite{A27} is \cite[Lemma 2.5.5]{Ar}.
This lemma claims that the local intertwining relation 
(see Section \ref{intro.main3} below for more details) for the tempered case 
can be reduced to a slightly weaker statement. 
We will show this lemma in Appendix \ref{sec.255}. 
Theorems \ref{main1} and \ref{2.5.5} together with 
the twisted endoscopic character relations for the archimedean case, which are explained in Appendix \ref{sec.TECR}, make the discussion of \cite[Chapter 6]{Ar}, 
and hence the local classification in the tempered case (under the inductive hypotheses), 
conditional only on the twisted weighted fundamental lemma.
\par

\subsubsection{Main Theorem 2}
The second main theorem concerns $G = \GL_N(F)$ and its automorphism $\theta$. 
In this case, $\Pi_{\psi_M} = \{\pi_M\}$ is a singleton, 
and $I_P(\pi_M)$ is an irreducible unitary representation.
Let $w$ be a $\theta$-twisted Weyl element such that $w\pi_M \cong \pi_M$. 
Then we have a normalized isomorphism $\tl\pi_M(w) \colon w \pi_M \xrightarrow{\sim} \pi_M$, whose composition with the normalized intertwining operator $R_P(w,\pi_M,\psi_M)$ discussed above produces the intertwining operator 
\[
\tl{R}_P(w, \tl\pi_M) \colon I_P(\pi_M) \rightarrow I_P(\pi_M) \circ \theta,  
\]
see Section \ref{sec.main2}. On the other hand, the assumption $w\pi_M \cong \pi_M$ 
implies that $I_P(\pi_M)$ is self-dual, i.e., $I_P(\pi_M) \cong I_P(\pi_M) \circ \theta$. 
Using a Whittaker functional on the standard module of $I_P(\pi_M)$, 
one can define a normalized isomorphism 
$\theta_A \colon I_P(\pi_M) \xrightarrow{\sim} I_P(\pi_M) \circ \theta$
(see Section \ref{sec.thetaA}).
Our second main theorem (Theorem \ref{main2}) asserts that 
\[
\tl{R}_P(w, \tl\pi_M) = \theta_A.
\]
This is \cite[Theorem 2.5.3]{Ar}, whose proof was deferred to  \cite{A26}.
\par

The proof of Theorem \ref{main2} is given in Section \ref{LIR_GL}. 
When $\pi_M$ is tempered (and hence generic), 
the assertion immediately follows from the first main theorem (Theorem \ref{main1}). 
In the non-tempered case, 
Arthur notes below \cite[Theorem 2.5.3]{Ar} that 
it requires further techniques, based on some version of minimal $K$-types. While trying to follow this argument, we were led to heavy calculations. Therefore, in this paper we will prove Theorem \ref{main2} in the non-tempered case with a completely different approach.
One of the challenges is that the isomorphism $\theta_A$ is inexplicit, since it is defined through the Langlands quotient map from the standard module of $I_P(\pi_M)$, 
which is a priori known to exist only abstractly. 
Our novelty is to construct the standard module carefully and 
to realize the Langlands quotient map as a composition of normalized intertwining operators
(Lemma \ref{langlands}). 
This, together with Theorem \ref{main1}, describes $\theta_A$ using intertwining operators, 
and Theorem \ref{main2} is reduced to the commutativity of a certain diagram in Theorem \ref{diagramGL},
which we call the \emph{main diagram}. 
However, since $I_P(\pi_M)$ is a non-tempered unitary induction, 
this commutativity does not follow directly from the previous results, and requires further arguments.
A simple but non-trivial example for Theorem \ref{diagramGL} is given in Example \ref{ex_GL}.
\par

\subsubsection{Main Theorem 3}\label{intro.main3}
Let $F$ be a non-archimedean base field 
so that $L_F = W_F \times \SL_2(\C)$ with $W_F$ the Weil group of $F$.
We now discuss the results whose proofs were deferred to \cite{A25}. 
They pertain to a classical group $G$ over $F$
(in this introduction we are taking the example of $G = \Sp_{2n}(F)$), 
and are formulated as \cite[Lemma 7.1.2]{Ar}. 
This lemma builds on the inductive assumptions that the local theorems have been proved for 
\begin{itemize}
\item
all tempered $L$-parameters for $G$; and 
\item
all $A$-parameters for any classical group $G'$
such that $\dim(\St_{\widehat{G'}}) < \dim(\St_{\widehat{G}})$, 
\end{itemize}
where $\dim(\St_{\widehat{G}})$ is the dimension of the standard representation of the Langlands dual group $\widehat{G}$ of $G$. 
These inductive assumptions hold at the start of \cite[Chapter 7]{Ar}. The statements of this lemma concern (certain) co-tempered parameters. More precisely, for an $A$-parameter $\psi \colon W_F \times \SL_2(\C) \times \SL_2(\C) \rightarrow {}^LG$, we define its \emph{Aubert dual} parameter $\widehat\psi \colon W_F \times \SL_2(\C) \times \SL_2(\C) \rightarrow {}^LG$
by $\widehat\psi(w, g_1,g_2) = \psi(w,g_2,g_1)$. 
We say that $\psi$ is \emph{co-tempered} if its restriction to the first copy of $\SL_2(\C)$ is trivial. In other words, $\phi=\widehat\psi$ is a tempered $L$-parameter.

Under the above inductive hypotheses, 
our third main theorem (Theorem \ref{main3}) asserts that 
for every co-tempered $A$-parameter $\psi = \widehat\phi$ for $G$, 

\begin{enumerate}
\item
we can construct an $A$-packet $\Pi_\psi$ 
together with a character $\pair{\cdot, \pi}_\psi$ of the component group $\Sc_\psi$ 
assigned to each $\pi \in \Pi_\psi$
which satisfies the ordinary endoscopic character relations, as well as the twisted endoscopic character relations with respect to $\GL_N(F)$ (where $N=2n+1$ if $G = \Sp_{2n}(F)$); and

\item
it also satisfies the local intertwining relation, explained further below.
\end{enumerate}
Note that this is a stronger statement than what is claimed in \cite[Lemma 7.1.2]{Ar}, 
because it covers all co-tempered parameters, rather than the special class of tamely ramified quadratic co-tempered parameters.
\par

Let us be a bit more precise about the statement of our theorem. The construction of the $A$-packet $\Pi_\psi$ in (1) is actually very simple. We define
\[
\Pi_{\psi} = \{ \hat\pi \,|\, \pi \in \Pi_\phi\},
\]
where $\phi=\widehat\psi$ is the tempered $L$-parameter that is the Aubert dual of $\psi$, 
$\Pi_\phi$ is the $L$-packet constructed in \cite[Chapter 6]{Ar}, 
and $\hat\pi$ is the Aubert dual of $\pi$ (see \cite{Au}). 
The problem is to prove that this definition satisfies the twisted and ordinary endoscopic character identities. 
\par

For the twisted character identities, 
we need to show that an alternating sum
of the characters of the members of the packet $\Pi_\psi$ defined above 
is the twisted transfer of the twisted character of the representation of $\GL_N(F)$ associated to $\psi$ (now viewed as a parameter for $\GL_N(F)$ via the standard embedding of $^LG$ into $\GL_N(\C)$). 
For this we need to relate Aubert duality for the classical group $G$ 
to twisted Aubert duality for $\GL_N(F)$. 
Building on the work of Hiraga \cite{Hi}, 
this comes down to verifying that certain signs defined in terms of Aubert duality for representations agree with corresponding signs defined in terms of Aubert duality for parameters.
\par

For standard character identities, 
we need to establish a map $\pi \mapsto \pair{\cdot,\pi}_\psi$ 
between $\Pi_\psi$ and the set of characters on $\Sc_\psi$ in such a way that, for each $s \in \Sc_\psi$,  
the virtual character $\sum_{\pi \in \Pi_\psi} \pair{s \cdot s_\psi,\pi}_\psi \Theta_\pi$ matches 
its endoscopic counterpart.  
By construction there is an obvious bijection $\pi \mapsto \hat\pi$ between $\Pi_\psi$ and $\Pi_\phi$, 
while at the same time we have the identity $\Sc_\psi=\Sc_\phi$. 
However, it is \emph{not} true that $\pair{\cdot,\pi}_\psi=\pair{\cdot,\hat\pi}_\phi$, 
in the sense that, if we took the above identity as a definition, 
then the endoscopic character identities would \emph{not} hold\footnote{
It was also pointed out by Liu--Lo--Shahidi \cite{LLS} recently. 
Their ``anti-tempered'' is synonymous with ``co-tempered'' in this paper.
}.
\par

To see what the correct definition would be, 
we assume that the desired character relations hold for an arbitrary $A$-parameter $\psi$ and its dual $\widehat\psi$, 
and investigate the implication of this assumption on the relationship between 
$\pair{\cdot,\pi}_\psi$ and $\pair{\cdot,\hat\pi}_{\widehat\psi}$ in Lemma \ref{beta} (2), 
where we show that the quotient of these characters can be described 
by certain signs $\beta(\psi_+)$, $\beta(\psi_-)$, and $\beta(\psi)$; 
here $\beta(\psi)$ is the sign that occurs in twisted Aubert duality on $\GL_N$ 
for the representation corresponding to $\psi$ (see Section \ref{sec.TAD_GL})
and $\beta(\psi_+)$, $\beta(\psi_-)$ are the analogous signs 
for certain supplementary parameters $\psi_+$ and $\psi_-$ defined in terms of $\psi$ and the element $s$
(see Section \ref{sec.Arthur}). 
While we cannot compute these signs in full generality, 
we do compute them for tempered representations in Proposition \ref{c-temp}, 
and for a certain class of representations in Proposition \ref{psiA}.
Hence we obtain an explicit relation between 
$\pair{\cdot, \hat\pi}_{\widehat\psi}$ and $\pair{\cdot, \pi}_\psi$ 
in Corollaries \ref{+-temp} and \ref{+otimes-}. 
Note that Proposition \ref{psiA} is an application of the second main theorem (Theorem \ref{main2}), 
and Corollary \ref{+otimes-} will be applied in the final step of the proof of Theorem \ref{main3} (2).
\par

Turning the tables around (see Remark \ref{converse}), we now drop the assumption that the desired character identities hold (after all, these are the identities we want to prove) and for a co-tempered $A$-parameter $\psi = \widehat\phi$, 
we define the character $\pair{\cdot,\pi}_\psi$ 
such that the formula obtained in Corollary \ref{+-temp} holds.
In order to see  that this is well-defined, 
one has to check a certain equality \eqref{ast} in Proposition \ref{beta_phi_pi}. As a consequence, we can see that the $A$-packet $\Pi_{\psi}$ satisfies the endoscopic character identities (Theorem \ref{ECR-cotemp}).
\par

Finally, we shall explain the proof of Theorem \ref{main3} (2), 
the local intertwining relation for co-tempered $A$-parameters. 
Our approach is entirely different from the one suggested below the statement of \cite[Lemma 7.1.2]{Ar}.
While the suggestion there was based on the theory of Hecke algebras 
and an expected extension of results by Morris, 
our initial attempts in this direction quickly led to very difficult calculations. 
Therefore, we develop a new approach that is based on the study of certain special representations, 
motivated by a result of Tadi\'c \cite{T2}, 
as well as the study \cite{At} of Jacquet modules for tempered $L$-packets by one of the authors, 
to treat the case of maximal parabolic subgroups 
and parameters $\psi_M$ whose linear part is irreducible and self-dual, 
and then we use an induction procedure to generalize this 
to arbitrary parabolic subgroups and arbitrary co-tempered parameters.
\par

Let $P = MN_P$ be a proper standard parabolic subgroup of $G = \Sp_{2n}(F)$. 
If $\psi_M$ is an $A$-parameter for $M$, 
we denote by $\psi$ the $A$-parameter for $G$ 
given by $\psi_M$ and the embedding ${}^LM \hookrightarrow {}^LG$. 
Then the $A$-packet $\Pi_\psi$ is 
the multi-set of the irreducible components of $I_P(\pi_M)$ for $\pi_M \in \Pi_{\psi_M}$. 
Moreover, if a Weyl element $w$ of $G$ preserves $M$ and $\psi_M$, 
we can normalize an isomorphism $w\pi_M \xrightarrow{\sim} \pi_M$. 
This allows us to define for an element $u \in S_\psi = \Cent(\im(\psi), \widehat{G})$ that normalizes $\widehat{M}$, 
a normalized self-intertwining operator 
\[
\pair{\tl{u}, \tl\pi_M} R_P(w_u, \tl\pi_M, \psi_M) \colon I_P(\pi_M) \rightarrow I_P(\pi_M),
\]
where the Weyl element $w_u$ is determined by $u$. 
If we knew that $I_P(\pi_M)$ is multiplicity-free, 
then this operator, being  $G$-equivariant, 
would act on each irreducible summand $\pi \subset I_P(\pi_M)$ by a scalar. 
The \emph{local intertwining relation \eqref{LIR}} asserts that this scalar 
is equal to $\pair{s_u,\pi}_\psi$, 
where $s_u \in \Sc_\psi = \pi_0(S_\psi)$ is the image of $u \in S_\psi$. 
For a more precise statement, see Section \ref{sec.main3}.
When $\psi_M = \widehat\phi_M$ is co-tempered, 
the multiplicity-free statement follows from the corresponding statement for $\phi_M$, 
which has been established in \cite[Chapter 6]{Ar}, 
and Theorem \ref{main3} (2) claims that \eqref{LIR} holds. 
We prove Theorem \ref{main3} (2) in Sections \ref{sec.LIR} and \ref{computations}.
\par

As mentioned above, 
the proof can be reduced by an inductive argument (Lemma \ref{reduction}) 
to the case where $P=MN_P$ is a maximal parabolic subgroup, 
so that $M \cong \GL_k(F) \times G_0$ with $G_0 = \Sp_{2n_0}(F)$, 
and the $\GL$-part $\psi_\GL$ of $\psi_M = \psi_\GL \oplus \psi_0$ is irreducible and self-dual. 
After that, our approach to attack \eqref{LIR} is to extend our method for Theorem \ref{main2} 
from the case of $\GL_N(F)$ to the case of classical groups. 
The reason why our method works well for $\GL_N(F)$ is that 
the unitary induction $I_P(\pi_M)$ is irreducible, 
which implies that its Langlands data are easily described in terms of those of $\pi_M$.
When $G$ is a classical group, $I_P(\pi_M)$ need not be irreducible.
In Section \ref{high}, we isolate a certain irreducible summand of $I_P(\pi_M)$ on which our method would work; we call it a \emph{highly non-tempered summand}. 
The relevance of this notion is that 
one can infer the Langlands data of a highly non-tempered summand of $I_P(\pi_M)$ from that of $\pi_M$, 
while the Langlands data for general subquotients of parabolically induced representations are mysterious. 
It follows from the definition that for any $\pi_M \in \Pi_{\psi_M}$, 
there exists a highly non-tempered summand of $I_P(\pi_M)$, 
and in many cases it is unique. 
Our method would work only for highly non-tempered summands of $I_P(\pi_M)$. 
Since we are assuming that $P$ is maximal and that $\psi_\GL$ is irreducible, 
the unitary induction $I_P(\pi_M)$ for $\pi_M \in \Pi_{\psi_M}$ 
is a direct sum of at most two irreducible representations. 
In this sense, 
our method has a chance to prove \eqref{LIR} only for ``half'' of the cases.
\par

The good news is that Corollary \ref{NIOvsAD}
tells us that ``half'' is enough!
More precisely, if $I_P(\pi_M) = \pi_1 \oplus \pi_2$ is reducible, 
then this result, 
which says that Aubert duality commutes with the normalized intertwining operator
up to a nonzero scalar, 
implies that \eqref{LIR} for $\pi_1$ is equivalent to \eqref{LIR} for $\pi_2$ (see Lemma \ref{halve}).
Note that Corollary \ref{NIOvsAD} was established in an appendix in an arXiv version of \cite{KMSW}, 
but because it is a key input to our argument, 
we move this appendix to Appendix \ref{sec.AD} in this paper. 
\par

This allows us to focus on \eqref{LIR} for a highly non-tempered summand $\pi \subset I_P(\pi_M)$. 
By the definition of this summand, we have at our disposal the analogue of the \emph{main diagram} from the $\GL_N(F)$ case. 
However, since the intertwining operator on $I_P(\pi_M)$ is normalized using the $A$-parameter $\psi_M$, 
whereas the other operators appearing in the main diagram use $L$-parameters, 
the main diagram is commutative only up to an explicit scalar, 
which is a special value of the quotient of the normalizing factors 
defined using the $A$-parameter $\psi_M$ and the $L$-parameter $\phi_{\pi_M}$ of $\pi_M$.
See Theorem \ref{diagramG}. 
By \eqref{LIR} in the tempered case, 
we can relate this scalar to the eigenvalue of the normalized intertwining operator on $\pi$.
In conclusion, \eqref{LIR} for the highly non-tempered summand $\pi \subset I_P(\pi_M)$
is equivalent to a certain scalar equation \eqref{star} presented in Corollary \ref{summary}.
\par

The proof is therefore reduced to establishing this scalar equation \eqref{star}. 
Note that the left-hand side of this equation involves the $L$-parameter of $\pi_M$. 
In general, it is very difficult 
to completely list the $L$-parameters of representations in a given $A$-packet $\Pi_{\psi_M}$. 
Circumventing this problem, we give an inductive argument to show \eqref{star}. 
The initial step is where 
the classical part $\pi_0$ of $\pi_M = \pi_\GL \boxtimes \pi_0$ is \emph{almost supercuspidal}
(see Definition \ref{a-cusp}). 
In this case, we will compute everything by hand in Section \ref{cuspidal}. 
For the inductive steps (Sections \ref{(a)}, \ref{(b)} and \ref{(c)}), 
we use \cite[Theorem 4.2]{At}, which computes the Jacquet modules of tempered representations. 
According to this theorem, we need to consider three cases separately. 
The first and second cases are amenable, whereas 
in the last case in Section \ref{(c)}, we have to treat an $A$-parameter which is not co-tempered, 
but it fits the artificial assumption in Corollary \ref{+otimes-}. 
\par

In our proof of Theorem \ref{main3}, 
M{\oe}glin's work on the explicit construction of $A$-packets 
(see \cite{X2} and its references) plays a fundamental role. 
In particular, her results on the computation of the cuspidal support of discrete series representations 
and their reinterpretation in terms of enhanced $L$-parameters by Xu (see \cite{X1}), 
as well as the extension \cite{At} of these results to the tempered case by one of the authors, 
are an essential tool in our argument.
\par

However one has to be quite careful at this point. 
As explained in \cite[Section 8]{X2}, 
the original arguments presented in \cite{X2} and \cite{At}
follow a strategy that requires the validity of Arthur's results not only for the group $G$ itself, 
but also for groups of higher rank.
Therefore, such an approach cannot be taken in the middle of a proof by induction on the rank of $G$, 
which is the situation of \cite[Chapter 7]{Ar}.
This requires us to give new proofs of some of these results, 
in particular \cite[Theorem 4.2]{At}, 
which avoid appealing to groups of larger rank. 
We do this in Appendix \ref{xu-atobe}, 
and also extend the results to cover the case of unitary groups.
\par

Note also that, 
if we could use the full extent of these results, 
that would easily lead to a generalization of Corollary \ref{+otimes-} 
that does not use Theorem \ref{main2}. 
But because of the inductive constraints, 
we are forced to prove the stated version of Corollary \ref{+otimes-} 
and appeal to Theorem \ref{main2} in that proof.

\subsubsection{Supplementary results}
In addition to the main theorems, 
we prove a number of supplementary results in the appendices, 
most of which are also required for the inductive argument in \cite{Ar}, 
and some of which may be of independent interest.
\par

In Appendix \ref{sec.LF},
we prove that the local Langlands correspondence for classical groups, 
that is established by the inductive argument of \cite{Ar} and \cite{Mok}, 
matches the Galois-theoretic local factors defined by Artin, Deligne and Langlands 
with the automorphic local factors defined by Shahidi. 
This result is used in the proof of our first main theorem (Theorem \ref{main1}) and it is also of independent interest: 
for example, it is used in \cite{GI1} and \cite{GI2}, where it was taken as an assumption.
\par

In general, 
the matching of such local factors is a basic expectation for the local Langlands correspondence, 
at least in the setting in which automorphic local factors have been defined. 
In some constructions of the correspondence (such as those based on converse theorems), 
this matching is built into the construction. 
In the approach via endoscopy that is the subject of \cite{Ar} and \cite{Mok}, 
this matching does not follow directly from the construction. 
In Appendix \ref{sec.LF}, 
we verify it for groups of the form $\GL_k \times G_0$, 
where $G_0$ is a classical group. 
In this setting, the automorphic local factors were defined by Shahidi, 
and we verify that they match the Galois factors defined by Artin, Deligne and Langlands.
\par

In Appendix \ref{sec.AD}, 
we review the Aubert involution for connected reductive groups 
and prove that it is compatible with intertwining operators up to a scalar. 
This result is used in an essential way in the proof of our third main theorem (Theorem \ref{main3}).
This theorem, roughly speaking, explicitly determines the scalar. We also review the work \cite{CK26} on the Aubert involution for disconnected reductive groups, some aspects of which are used in the proof of the third main theorem (Theorem \ref{main3}). 
\par

In Appendix \ref{xu-atobe}, 
we extend and reprove some results of \cite{Moe11}, \cite{X1}, and \cite{At} to all classical groups. 
These results are essential in the proof of our third main theorem (Theorem \ref{main3}). 
On the one hand, some of these results were originally formulated only for symplectic and orthogonal groups, 
and we take the opportunity here to formulate and prove them also for unitary groups. 
As already mentioned, the earlier proofs of some of these results relied on some arguments 
that presuppose that the endoscopic classification of representations has been established for all groups, 
in particular for groups whose rank is higher than that of the group under consideration. Since such an assumption is problematic in the midst of a proof by induction on the rank, 
we present in this appendix alternative proofs that avoid this assumption.
\par

Appendix \ref{sec.255} is devoted to the proofs of two results which are based on the same argument. 
The first result is the strong form of Shahidi's tempered $L$-packet conjecture 
(a strengthening of \cite[Conjecture 9.4]{Sh7}), 
which states that every tempered $L$-packet contains exactly one member 
that is generic with respect to a given fixed Whittaker datum $\ww$, 
and moreover that the pairing of this $L$-packet with the centralizer group of its parameter 
matches the generic constituent with the trivial character. 
For classical groups, the existence and uniqueness of the generic constituent was proven by Varma \cite{V2}.
We prove this for general connected reductive groups under relevant assumptions, 
which the construction of the local Langlands correspondence in \cite{Ar} and \cite{Mok} 
does indeed satisfy. 
\par

The second result is of a more technical nature. 
It is formulated in \cite{Ar} as Lemma 2.5.5 and is used in the inductive proof. 
This lemma states that a weaker version of the local intertwining relation, 
where a certain unknown scalar is inserted, implies the stronger version, 
in which this scalar equals $1$. 
In \cite{Ar}, this statement is formulated for classical groups and for odd residual characteristic. 
Building on results of Kottwitz \cite{Kot} and Varma \cite{V1}, 
we are able to give a proof that works uniformly for all reductive groups and arbitrary residual characteristic.
\par

In Appendix \ref{sec.TECR}, 
we prove certain twisted character identities over the real numbers 
that are assumed at various places in the inductive argument of \cite{Ar}. 
Such identities are now available in the literature, 
so we simply collect the required references and supply the necessary additional arguments 
to adapt them to the form needed in \cite{Ar}.
\par

Finally, in Appendix \ref{sec.Frob}, 
we shall review two conventions for normalizing class field theory 
depending on whether arithmetic or geometric Frobenius elements 
correspond to uniformizers under the local non-archimedean reciprocity map, and discuss how they influence Arthur's endoscopic classification. 
It is an expanded version of the discussion of \cite[Section 4]{KoSh2}.
This appendix is independent of the other sections in this paper.

\subsubsection{A roadmap}
The following is a roadmap of our paper and Arthur's results:
\[
{\tiny
\xymatrix{
\fbox{Appendix \ref{sec.LF}} \ar@{=>}[r] \ar@{=>}[rrdddd] 
& \fbox{Section \ref{sec.main1}} \ar@{=>}[dd] \ar@{=>}[dr] 
& \fbox{Appendix \ref{sec.255}} \ar@{=>}[d] & \fbox{Appendix \ref{sec.TECR}} \ar@{=>}[dl] \\
&& \doublebox{Local classification for tempered $L$-parameters} 
\ar@{=>}[dddl] \ar@{=>}[dr] \ar@{=>}[ddr] \\
&\fbox{Section \ref{LIR_GL}} \ar@{=>}[dd] \ar@{=>}[dr] && \fbox{Appendix \ref{xu-atobe}} \ar@{=>}[d] \ar@{=>}[ddl]\\ 
\fbox{Appendix \ref{sec.AD}} \ar@{=>}[rr] \ar@{=>}[rd] && \fbox{Section \ref{sec.character}} \ar@{=>}[d] \ar@{=>}[r] 
&\fbox{Section \ref{sec.ECR.cotemp}} \ar@{=>}[lld] \ar@{=>}[ld] \ar@{=>}[ldd]\\
&\fbox{Section \ref{sec.LIR}} \ar@{=>}[rd] 
& \fbox{Section \ref{computations}} \ar@{=>}[d]\\
\fbox{Appendix \ref{sec.Frob}}&& \doublebox{Local classification for co-tempered $A$-parameters} \ar@{=>}[d]^{\text{globalization}}\\
&& \doublebox{Local classification for general $A$-parameters}
}}
\]

\subsection{On \cite{A24}}\label{sec.A24}
As alluded to above, the problem associated with \cite{A24} has been resolved. 
The main role of \cite{A24} is to justify \cite[Proposition 2.1.1]{Ar}, 
which asserts that choosing a test function on the twisted general linear group essentially amounts to 
choosing a family of test functions on (not necessarily elliptic) twisted endoscopic groups 
satisfying a certain compatibility condition. 
The proposition ``represents part of the stabilization of the twisted trace formula'' (see \cite[p.~57, line -8]{Ar}), 
and indeed, it has been proved in general (not only for twisted general linear groups) 
by M{\oe}glin--Waldspurger \cite[I.4.11]{MW_Stab1}. 
Note that \cite{A24} is used again in the proof of \cite[Corollary 8.4.5]{Ar} 
to back up an implicit fact in one of Arthur's papers; 
this fact is covered by \cite[I.4.11]{MW_Stab1} as well.

\subsection{On the twisted weighted fundamental lemma}\label{sec.TFL}
As mentioned above, 
the stabilization of the twisted trace formula depends on 
the twisted weighted fundamental lemma stated as a theorem in \cite[II.4.4]{MW_Stab1}, 
which remains conditional to the best of our knowledge; 
cf.~the third-to-last paragraph in the preface of \cite{MW_Stab2}. 
As explained in \cite[II.4.4]{MW_Stab1}, 
this theorem is reduced via \cite[Theorem 3.8]{W4} to 
\begin{enumerate}
\item[(i)] 
the weighted fundamental lemma for Lie algebras; and 
\item[(ii)] 
the non-standard version thereof; 
see \cite[Conjectures 3.6, 3.7]{W4} for the precise statements.
\end{enumerate}
We should point out that (i) is already needed to stabilize the untwisted trace formula.
The proof of (i) was completed for split groups by Chaudouard--Laumon \cite{CL1,CL2}. 
Even though their methods are expected to generalize beyond the split case, 
no written account has appeared on the proof of (i) for non-split groups, or on the proof of (ii). 
Such a generalization is necessary for the stabilized trace formulas considered in \cite{Ar} and \cite{Mok}. The good news is that a full proof of (i) and (ii) is on the horizon. Halleck-Dub{\'e}'s thesis \cite{HD} extends the cohomological version of the weighted fundamental lemma (amounting to \cite{CL1} and the first half of \cite{CL2}) to the general case. He is completing what remains to be done, that is, to carry out fixed-point and local-global arguments to obtain (i) and (ii) following the strategies of \cite{N,CL2}.
\par

It is worth noting that all versions of the unweighted fundamental lemma are theorems 
thanks to Ng\^o, Waldspurger, and others; 
see the introduction of \cite{N} and \cite{LMW} for the explanation and further references, 
and \cite{GWZ} and \cite{Wa} for other proofs. 

\subsection{Acknowledgments}
First and foremost, 
we express our most sincere gratitude to Jim Arthur for sharing with us his approach 
and thoughts on the subject of \cite{A25}
and for having been tremendously influential and inspirational for us. 
We are grateful to Freydoon Shahidi for explaining the historical
context of the issues discussed in Section \ref{sec.shahidi}. 
\par

H. Atobe is deeply grateful to Wen-Wei Li
for bringing the key result, Corollary A.3.3 in an arXiv version of \cite{KMSW} 
(which is Corollary \ref{NIOvsAD} in this paper), to his attention. 
This was the last piece needed to apply our approach for \cite{A26} to the classical groups. 
We would like to thank Yifeng Liu for suggesting several useful references. 
\par

We are grateful for the hospitality of Kyoto University and the support of Hokkaido University (to Atobe) 
during our stay 13.05.2023--27.05.2023, 
which marked the beginning of this project.
H. Atobe, A. Ichino, A. M\'inguez and S. W. Shin
thank 
\begin{itemize}
\item
the Korea Institute for Advanced Study;
\item
the National University of Singapore;
\item
the JSPS, Kyoto University and the University of Michigan; and
\item
the University of Michigan and the Korea Institute for Advanced Study, 
\end{itemize}
respectively, 
for hospitality and excellent working conditions
where parts of this work were done.
\par

H. Atobe, W. T. Gan, A. Ichino, T. Kaletha, A. M\'inguez, and S. W. Shin
were partially supported by
\begin{itemize}
\item
JSPS KAKENHI Grant Number 23K12946; 
\item
a NUS Tan Chin Tuan Centennial Professorship; 
\item
JSPS KAKENHI Grant Number 19H01781; 
\item
NSF grant DMS-2301507 and NSF RTG grant DMS-1840234; 
\item
the Principal Investigator project PAT4832423 of the Austrian Science Fund (FWF), 
the Croatian Science Foundation project (HRZZ-IP-2022-10-4615)
and the project ATR2024-154613 funded by MICIU/AEI/10.13039/501100011033; 
\item
NSF grant DMS-2401353, NSF RTG grant DMS-2342225, 
and a Simons Fellowship no.1013886,
\end{itemize}
respectively. 
\par

\section{Arthur's theory of the endoscopic classification}\label{sec.main}
In this section, we review Arthur's theory and state our main theorems in their appropriate generality.
\par

\subsection{Basic notation}\label{sec.notation}
Fix a local field $F$ of characteristic zero. 
Let $E$ be either $F$ or a quadratic field extension of $F$. 
We denote by $x \mapsto \overline{x}$ the generator of $\Gal(E/F)$. 
Fix a non-trivial unitary character $\psi_F \colon F \rightarrow \C^\times$, 
and put $\psi_E = \psi_F \circ \tr_{E/F}$.
The normalized absolute value of $E$ is denoted by $|\cdot|_E$.
In particular, if $E$ is non-archimedean, and if $\varpi_E$ is a uniformizer of $E$, 
then $|\varpi_E|_E = q_E^{-1}$, 
where $q_E$ is the cardinality of the residue field of $E$.
\par

Fix an algebraic closure $\overline{F}$ of $F$, which contains $E$. 
The absolute Galois group of $F$ is denoted by $\Gamma = \Gamma_F = \Gal(\overline{F}/F)$.
Let $W_E$ be the Weil group of $E$, and let
\[
L_E = \left\{
\begin{aligned}
&W_E \iif \text{$E$ is archimedean}, \\
&W_E \times \SL_2(\C) \iif \text{$E$ is non-archimedean}
\end{aligned}
\right. 
\]
be the local Langlands group of $E$. 
Finally, let $|\cdot|_E$ be the norm map of $L_E$, i.e., 
\[
|\cdot|_E \colon L_E \twoheadrightarrow L_E^{\mathrm{ab}} \cong E^\times \xrightarrow{|\cdot|_E} \R_{>0},
\]
where $L_E^{\mathrm{ab}} \cong E^\times$ is the isomorphism given by the local class field theory. 
Here, we normalize this isomorphism such that
an arithmetic Frobenius map corresponds to a uniformizer in $E^\times$.
\par

\subsection{Groups}\label{sec.groups}
In this paper, we often identify a connected reductive group over $F$ 
with the group of its $F$-points. 
\par

Let $G^\circ$ be a general linear group $\GL_N(E)$
or one of the following quasi-split classical groups
\[
\SO_{2n+1}(F), \quad \Sp_{2n}(F), \quad
\SO_{2n}(F), \quad \U_n.
\]
Here, when $G^\circ = \U_n$, 
it is an outer form of $\GL_n$ with respect to a specified quadratic extension $E$ of $F$, 
whereas when $G^\circ = \SO_N(F), \Sp_{2n}(F)$ we simply set $E=F$.
On the other hand, when $G^\circ = \SO_{2n}(F)$, 
we denote by $K$ the splitting field of $G^\circ$, 
which is equal to $F$ or a quadratic extension of $F$. 
Letting $\eta$ be the (possibly trivial) quadratic character of $F^\times$ associated to $K/F$, 
we sometimes write $\SO_{2n}(F) = \SO_{2n}^\eta(F)$.
\par

Fix an $F$-splitting $\spl = (B^\circ, T^\circ, \{X_\alpha\}_\alpha)$ of $G^\circ$.
Namely, $B^\circ = T^\circ U$ is an $F$-rational Borel subgroup, 
$T^\circ$ is a maximal torus, $U$ is the unipotent radical of $B^\circ$, 
and $\{X_\alpha\}_\alpha$ is a $\Gamma$-invariant set of root vectors, 
where $\alpha$ runs over simple roots of $T^\circ$ with respect to $B^\circ$.
Then $\spl$ and $\psi_F$ give rise to a Whittaker datum $\ww = (B^\circ, \chi)$, 
where $\chi$ is a non-degenerate character of $U$.
For any Levi subgroup $M^\circ$ of $G^\circ$ containing $T^\circ$, 
we take the $F$-splitting of $M^\circ$ induced by $\spl$, 
so that the associated Whittaker datum is given by $\ww_M = (B^\circ \cap M^\circ, \chi|_{U \cap M^\circ})$.
\par

For any subgroup $N$ of $U$ stable under the adjoint action of $T^\circ$, 
we take the Haar measure on $N$ determined by $\{X_\alpha\}_\alpha$ 
and the self-dual Haar measure on $F$ with respect to $\psi_F$.
\par

If $G^\circ = \SO_{2n}^\eta(F)$, we also consider the full orthogonal group $G = \O_{2n}^\eta(F)$
such that $G^\circ$ is the connected component of $\1 \in G$. 
We also abbreviate $\O_{2n}^\eta(F)$ as $\O_{2n}(F)$.
Let $T$ be the normalizer of $(T^\circ, B^\circ)$ in $G$. 
Then $T/T^\circ \cong G/G^\circ$.
Fix $\epsilon \in T \setminus T^\circ$ with $\epsilon^2 = \1$ such that $\epsilon$ preserves the splitting $\spl$.
(Cf., Section \ref{secA.groups} below.)
It gives an identification of $\O_{2n}(F)$ with a twisted group $\SO_{2n}(F) \rtimes \pair{\epsilon}$.
If $G^\circ \not= \SO_{2n}(F)$, we set $G = G^\circ$ and $T = T^\circ$. 
\par

Let $P^\circ = M^\circ N$ be a parabolic subgroup of $G^\circ$, 
where $M^\circ$ is a Levi component of $P^\circ$ and $N = N_P$ is the unipotent radical of $P^\circ$.
We say that $P^\circ$ (\resp $M^\circ$) is \emph{standard} (\resp \emph{semi-standard}) 
if it contains $B^\circ$ (\resp $T^\circ$).
If $P^\circ$ is stable under the adjoint action of $T$, 
we set $P = P^\circ \cdot T$ and $M = M^\circ \cdot T$. 
Otherwise, we put $P=P^\circ$ and $M = M^\circ$. 
We call the subgroup of the form $P=MN$ (\resp $M$) 
a \emph{(standard) parabolic subgroup} (\resp a \emph{(semi-standard) Levi subgroup}) of $G$.
\par

\subsection{Representations}\label{sec.rep}
Let $G$ be one of the groups $\GL_N(E)$, $\SO_{2n+1}(F)$, $\Sp_{2n}(F)$, $\O_{2n}(F)$ or $\U_n$
as in the previous subsection. 
We denote by $\Rep(G)$ the category of smooth (Fr\'echet) admissible
complex representations of $G$ (of moderate growth) of finite length. 
Here, the notions of Fr\'echet and moderate growth are meaningful only when $F$ is archimedean. 
Let $\Irr(G)$ be the set of equivalence classes of irreducible objects of $\Rep(G)$.
The subset of $\Irr(G)$ consisting of irreducible unitary (\resp tempered) representations 
is denoted by $\Irr_\unit(G)$ (\resp $\Irr_\temp(G)$).
\par

Let $P = MN$ be a standard parabolic subgroup of $G$ with semi-standard Levi component $M$.
Put $\aa_M = \Hom(\Rat(M), \R)$ and $\aa_M^* = \Rat(M) \otimes \R$, 
where $\Rat(M)$ is the group of algebraic characters of $M$ defined over $F$.
Let $\aa_{M,\C}^* = \Rat(M) \otimes \C$ be the complexification of $\aa_M^*$.
Define a homomorphism $H_M \colon M \rightarrow \aa_M$ by 
\[
 |\chi(m)|_F = e^{\pair{H_M(m), \chi}}
\]
for all $\chi \in \Rat(M)$ and $m \in M$, 
where $\pair{\cdot, \cdot} \colon \aa_M \times \aa_M^* \rightarrow \R$ is the natural pairing.
For an irreducible representation $\pi$ of $M$ and $\lambda \in \aa_{M,\C}^*$, 
we define a representation $\pi_\lambda$ of $M$ by $\pi_\lambda(m) = e^{\pair{H_M(m), \lambda}} \pi(m)$ 
realized on the space $\VV_\pi$ of $\pi$.
We denote by 
\[
I_P(\pi_\lambda) = \Ind^G_P(\pi_\lambda)
\]
the associated normalized parabolically induced representation of $G$.
As a $\C$-vector space, 
$I_P(\pi_\lambda)$ is the space of smooth functions $f_\lambda \colon G \rightarrow \VV_\pi$ such that 
\[
f_\lambda(nmg) = \delta_{P}^{\half{1}}(m) \pi_\lambda(m)f_\lambda(g)
\]
for $n \in N$, $m \in M$ and $g \in G$.
When $F$ is archimedean, 
this space is a Fr\'echet space with some natural semi-norms.
For details, see \cite[Section 4]{Cas}.
\par

On the other hand, if $F$ is non-archimedean and 
if $(\pi, \VV_\pi)$ is a smooth representation of $G$, 
set
$\VV_\pi(N)$ to be the subspace of $\VV_\pi$ generated by 
vectors of the form $v - \pi(n)v$ for $v \in \VV_\pi$ and $n \in N$.
Define an action $\overline{\pi}$ of $M$ on $\VV_\pi/\VV_\pi(N)$ by 
\[
\overline{\pi}(m) (v \bmod \VV_\pi(N)) = \delta_P^{-\half{1}}(m) \pi(m)v \bmod \VV_\pi(N).
\]
The representation $(\overline{\pi}, \VV_\pi/\VV_\pi(N))$ of $M$ is 
called the \emph{normalized Jacquet module} of $\pi$ along $P$, 
and is denoted by $\Jac_P(\pi)$.

\subsection{Arthur's extension of conjugate-self-dual representations}\label{sec.thetaA}
We define an involution $\theta$ on $\GL_N(E)$ by 
\[
\theta \colon x \mapsto 
\begin{pmatrix}
&&1\\
&\iddots&\\
(-1)^{N-1}&&
\end{pmatrix}
{}^t\overline{x}^{-1}
\begin{pmatrix}
&&1\\
&\iddots&\\
(-1)^{N-1}&&
\end{pmatrix}^{-1}.
\]
Set $\tl\GL_N(E) = \GL_N(E) \rtimes \pair{\theta}$.
\par

Let $\pi$ be an irreducible representation of $\GL_N(E)$. 
Suppose that $\pi$ is \emph{conjugate-self-dual}, i.e., $\pi \cong \pi \circ \theta$.
Then there is a linear isomorphism 
$T \colon \pi \xrightarrow{\sim} \pi$
such that 
\[
T \circ \pi(g) = \pi(\theta(g)) \circ T, \quad g \in \GL_N(E). 
\]
Note that $T$ is unique up to a nonzero scalar. 
We can normalize $T$ as follows. 
Let $\II_\pi$ be the standard module of $\GL_N(E)$ whose Langlands quotient is $\pi$. 
Since $\II_\pi$ is $\ww$-generic, we can fix a nonzero $\ww$-Whittaker functional $\Omega$ on $\II_\pi$. 
By $\pi \cong \pi \circ \theta$, 
there is a unique linear isomorphism $\theta_W \colon \II_\pi \xrightarrow{\sim} \II_\pi$ 
satisfying the equations 
\begin{align*}
\theta_W \circ \II_\pi(g) &= \II_\pi(\theta(g)) \circ \theta_W, \quad g \in \GL_N(E), \\
\Omega \circ \theta_W &= \Omega.
\end{align*}
Since $\Omega$ is unique up to a scalar, 
the definition of $\theta_W$ is independent of the choice of $\Omega$. 
Then we can define a linear isomorphism $\theta_A \colon \pi \xrightarrow{\sim} \pi$
satisfying 
$\theta_A \circ \pi(g) = \pi(\theta(g)) \circ \theta_A$ for any $g \in \GL_N(E)$
and making the diagram 
\[
\begin{CD}
\II_\pi @> \theta_W >> \II_\pi \\
@VVV @VVV \\
\pi @>\theta_A>> \pi
\end{CD}
\]
commutative, 
where the vertical map $\II_\pi \rightarrow \pi$ is the (fixed) Langlands quotient map, 
which is unique up to a scalar. 
In particular, $\theta_A$ gives an extension $\tl\pi = \pi \boxtimes \theta_A$ of $\pi$ to $\tl\GL_N(E)$.
We call it \emph{Arthur's extension} of $\pi$. 
Note that $\theta_A$ depends on $\pi$.

\subsection{$A$-parameters}\label{sec.parameters}
First, we consider $\GL_N(E)$. 
A \emph{representation} of $L_E \times \SL_2(\C)$ is a homomorphism 
\[
\psi \colon L_E \times \SL_2(\C) \rightarrow \GL_N(\C)
\]
such that 
\begin{itemize}
\item
$\psi(W_E)$ consists of semisimple elements; 
\item
$\psi|_{W_E}$ is continuous;
\item
$\psi|_{\SL_2(\C)}$ is algebraic if $E$ is archimedean, 
whereas
$\psi|_{\SL_2(\C) \times \SL_2(\C)}$ is algebraic if $E$ is non-archimedean.
\end{itemize}
An \emph{$A$-parameter} for $\GL_N(E)$
is an equivalence class of $N$-dimensional representations $\psi \colon L_E \times \SL_2(\C) \rightarrow \GL_N(\C)$
such that $\psi(W_E)$ is bounded.
By the \emph{local Langlands correspondence (LLC)} for $\GL_N(E)$
established by Langlands himself \cite{L2}, Harris--Taylor \cite{HT}, Henniart \cite{He2}, and Scholze \cite{Sc}, 
we obtain an irreducible representation $\pi_\psi$ of $\GL_N(E)$
associated to the $L$-parameter 
\[
\phi_\psi \colon L_E \ni w \mapsto \psi\left(w, 
\begin{pmatrix}
|w|_E^{\half{1}} & 0 \\ 0 & |w|_E^{-\half{1}}
\end{pmatrix}
\right)
\in \GL_N(\C).
\]
This is a unitary representation. 
\par

If $E \not= F$, fix $s \in W_F \setminus W_E$ and set ${}^c\psi(w, \alpha) = \psi(sws^{-1}, \alpha)$. 
We call ${}^c\psi$ the \emph{conjugate} of $\psi$. 
When $E=F$, we simply set ${}^c\psi = \psi$. 
We say that $\psi$ is \emph{conjugate-self-dual} if $\psi \cong {}^c\psi^\vee$.
In this case, $\pi_\psi$ is also conjugate-self-dual, 
i.e., $\pi_\psi \cong \pi_\psi \circ \theta$. 
Let $\tl\pi_\psi = \pi_\psi \boxtimes \theta_A$ be Arthur's extension of $\pi_\psi$ to $\tl\GL_N(E)$. 
We denote the character of $\tl\pi_\psi$ by 
$\Theta_{\tl\pi_\psi}(\tl{f}) = \tr(\tl\pi_\psi(\tl{f}))$
for $\tl{f} \in C_c^\infty(\GL_N(E) \rtimes \theta)$. 
\par

Next, let $G$ be one of the following quasi-split classical groups
\[
\SO_{2n+1}(F), \quad \Sp_{2n}(F), \quad
\O_{2n}^\eta(F), \quad \U_n.
\]
We denote by $G^\circ$ the connected component of the identity of $G$. 
Then $G=G^\circ$ unless $G = \O_{2n}^\eta(F)$ in which case $G^\circ = \SO_{2n}^\eta(F)$.
As explained in Section \ref{sec.groups}, 
we sometimes write $\SO_{2n}(F) = \SO_{2n}^\eta(F)$ and $\O_{2n}(F) = \O_{2n}^\eta(F)$.
We denote by $\St_{\widehat{G}}$ the standard representation of $\widehat{G^\circ}$, 
and set
\[
N = \dim(\St_{\widehat{G}}) = 
\left\{
\begin{aligned}
&2n \iif G = \SO_{2n+1}(F), \O_{2n}^\eta(F), \\
&2n+1 \iif G = \Sp_{2n}(F), \\
&n \iif G = \U_n.
\end{aligned}
\right. 
\]
\par

Let $\Psi(G)$ be the set of equivalence classes of conjugate-self-dual representations 
$\psi \colon L_E \times \SL_2(\C) \rightarrow \GL_N(\C)$
with 
\begin{align*}
\text{sign} &= \left\{
\begin{aligned}
&{-1} \iif G = \SO_{2n+1}(F), \\
&{+1} \iif G = \Sp_{2n}(F), \O_{2n}(F),\\
&(-1)^{n-1} \iif G = \U_n,
\end{aligned}
\right.
\\
\det(\psi) &= \left\{
\begin{aligned}
&\1 \iif G = \SO_{2n+1}(F), \Sp_{2n}(F),\\
&\eta \iif G = \O_{2n}^\eta(F).
\end{aligned}
\right.
\end{align*}
See \cite[Section 3]{GGP} for the notions of the signs of conjugate-self-dual representations. 
\par

An \emph{$A$-parameter} for $G^\circ$ is 
a $\widehat{G^\circ}$-conjugacy class of $L$-homomorphisms
\[
\psi \colon L_F \times \SL_2(\C) \rightarrow {}^LG^\circ
\]
such that $\psi(W_F)$ is bounded.
If $G = G^\circ$, then any $A$-parameter $\psi \colon L_F \times \SL_2(\C) \rightarrow {}^LG$ 
can be identified with an element of $\Psi(G)$ by $\psi \mapsto \psi|_{L_E \times \SL_2(\C)}$. 
On the other hand, when $G = \O_{2n}(F)$, 
if we denote by $\Psi(G^\circ)$ the set of $A$-parameters 
$\psi \colon L_F \times \SL_2(\C) \rightarrow {}^LG^\circ$, 
then there is a surjective map $\Psi(G^\circ) \rightarrow \Psi(G)$ 
whose fibers have cardinality $1$ or $2$.
We say that $\phi \in \Psi(G)$ is tempered if $\phi$ is trivial on the last $\SL_2(\C)$, i.e., 
$\phi \colon L_E \rightarrow \GL_N(\C)$. 
We denote by $\Phi_\temp(G)$ the subset of $\Psi(G)$ consisting of tempered $A$-parameters.
\par

Let $\psi \in \Psi(G)$. 
We decompose 
\[
\psi = \psi_{\bad} \oplus \psi_{\good} \oplus {}^c\psi_{\bad}^\vee, 
\]
where 
$\psi_{\good}$ is a sum of irreducible conjugate-self-dual representations of the same type as $\psi$, 
and $\psi_{\bad}$ is a sum of irreducible representations of other types
(cf., \cite[Section 4]{GGP}).
We say that $\psi$ is \emph{of good parity} if $\psi_\bad = 0$. 
In general, if we write 
\[
\psi_{\good} = \bigoplus_{i=1}^t \phi_i \boxtimes S_{d_i},
\]
where $\phi_i$ is an irreducible representation of $L_E$ with $\phi_i(W_E)$ bounded, 
and $S_d$ is the unique $d$-dimensional irreducible algebraic representation of $\SL_2(\C)$,
we set 
\[
A_{\psi} = \bigoplus_{i = 1}^t \Z/2\Z e(\phi_i, d_i).
\]
Namely, it is a free $\Z/2\Z$-module with a canonical basis $\{e(\phi_i, d_i)\}_{1 \leq i \leq t}$. 
Let $A_{\psi}^+$ be the kernel of the homomorphism 
\[
\det \colon A_{\psi} \rightarrow \Z/2\Z,\, 
e(\phi_i, d_i) \mapsto \dim(\phi_i \boxtimes S_{d_i}) \bmod 2.
\]
Define $A_{\psi}^0$ as the subgroup of $A_{\psi}$
generated by elements of the form
$e(\phi_i, d_i)+e(\phi_j, d_j)$ 
such that $\phi_i \boxtimes S_{d_i} \cong \phi_j \boxtimes S_{d_j}$. 
Note that $A_{\psi}^0 \subset A_{\psi}^+$.
Finally, set $z_{\psi} = \sum_{i=1}^t e(\phi_i, d_i)$. 
When $F$ is non-archimedean so that $L_E = W_E \times \SL_2(\C)$, 
if $\phi_i = \rho_i \boxtimes S_{a_i}$, then we also write $e(\phi_i,d_i) = e(\rho_i, a_i, d_i)$.
\par

Let $S_{\psi} = \Cent(\im(\psi),\widehat{G^\circ})$ be the centralizer of $\im(\psi)$ in $\widehat{G^\circ}$, 
and consider the component group 
\[
\Sc_{\psi} = \pi_0(S_{\psi}/Z(\widehat{G^\circ})^\Gamma).
\]
Then we have the following. 

\begin{itemize}
\item
If $G = \SO_{2n+1}(F)$, 
then $z_{\psi} \in A_{\psi}^+ = A_{\psi}$. 
Moreover, 
\[
\Sc_{\psi} \cong A_{\psi}/(A_{\psi}^0+\Z/2\Z z_{\psi}).
\]
\item
If $G = \Sp_{2n}(F)$, 
then $z_{\psi} \not\in A_{\psi}^+ \not= A_{\psi}$. 
Moreover, 
\[
\Sc_{\psi} \cong 
A_{\psi}^+/A_{\psi}^0 \cong 
A_{\psi}/(A_{\psi}^0+\Z/2\Z z_{\psi}).
\]
\item
If $G = \O_{2n}(F)$, 
then $z_{\psi} \in A_{\psi}^+$.
Moreover, 
\[
\Sc_{\psi} \cong 
A_{\psi}^+/(A_{\psi}^0+\Z/2\Z z_{\psi}).
\]
If we define $\tl\Sc_\psi^+$ similarly to $\Sc_\psi$, namely by replacing 
$\Cent(\im(\psi),\widehat{G^\circ})$ with $\Cent(\im(\psi),\O_{2n}(\C))$, 
then 
\[
\tl\Sc_{\psi}^+ \cong 
A_{\psi}/(A_{\psi}^0+\Z/2\Z z_{\psi}).
\]
\item
If $G = \U_n$, 
then 
\[
\Sc_{\psi} \cong A_{\psi}/(A_{\psi}^0+\Z/2\Z z_{\psi}).
\]
\end{itemize}
These isomorphisms are given by sending $e(\phi_i,d_i) \in A_\psi$ 
to the element in $\widehat{G^\circ}$ (or $\O_{2n}(\C)$)
which acts on $\phi_i \boxtimes S_{d_i}$ by $-1$, and on the other irreducible components trivially.

\subsection{$A$-packets via endoscopic character relations}\label{sec.Arthur}
Set $\AA_\psi = A_{\psi}/(A_{\psi}^0+\Z/2\Z z_{\psi})$ 
and denote its Pontryagin dual by $\widehat{\AA_\psi}$.
According to \cite[Theorems 2.2.1, 2.2.4]{Ar} and \cite[Theorem 3.2.1]{Mok} for $\psi$,
one has a multi-set $\Pi_\psi$ over $\Irr_\unit(G)$ 
and a map
\[
\Pi_\psi \rightarrow \widehat{\AA_\psi},\, 
\pi \mapsto \pair{\cdot, \pi}_\psi. 
\]
We call $\Pi_\psi$ the \emph{$A$-packet} associated to $\psi$. 
By identifying $\widehat{\AA_\psi}$ with a subgroup of $\widehat{A_{\psi}}$, 
one regards $\pair{\cdot, \pi}_\psi$ 
as a character of $A_{\psi}$ which factors through the surjection $A_\psi \twoheadrightarrow \AA_\psi$. 
In particular, we can consider a sign $\pair{s, \pi}_{\psi} \in \{\pm1\}$ for $s \in A_\psi$.
\par

The $A$-packet $\Pi_\psi$ together with the pairing $\pair{\cdot, \pi}_\psi$ 
will be determined by the following \emph{endoscopic character relations}.

\begin{enumerate}
\item
The equation
\[
\tag{{\bf ECR1}}\label{ECR1}
\Theta_{\tl\pi_\psi}(\tl{f}) = \frac{1}{(G:G^\circ)}\sum_{\pi \in \Pi_\psi} \pair{s_\psi, \pi}_\psi \Theta_\pi(f_G)
\]
holds 
whenever $\tl{f} \in C_c^\infty(\GL_N(E) \rtimes \theta)$ and $f_G \in C_c^\infty(G^\circ)$ have matching orbital integrals, 
where $\Theta_\pi$ is the character of $\pi$, and we set
\[
s_\psi = \sum_{\substack{1 \leq i \leq t \\ d_i \equiv 0 \bmod 2}} e(\phi_i,d_i).
\]

\item
For $s = \sum_{i \in I_-} e(\phi_i, d_i) \in A_\psi$, 
set $\psi_\pm$ as
\begin{align*}
\psi_- = \bigoplus_{i \in I_-} \phi_i \boxtimes S_{d_i}, 
\quad
\psi_+ = \psi - \psi_-. 
\end{align*}
Fix a conjugate-self-dual character $\eta_\pm$ of $E^\times$ 
such that there is a classical group $G_\pm$ satisfying that $\psi_\pm \otimes \eta_\pm \in \Psi(G_\pm)$.
For $f_G \in C_c^\infty(G)$, 
taking $f_{G_\pm} \in C_c^\infty(G^\circ_\pm)$ such that 
\begin{itemize}
\item
$f_G$ and $f_{G_+} \otimes f_{G_-}$ have matching orbital integrals; 
\item
when $G = \O_{2n}(F)$ and $s \in A_\psi^+$ (\resp $s \not\in A_\psi^+$), 
we further assume that $f_G(g) = 0$ for $g \in \O_{2n}(F) \setminus \SO_{2n}(F)$
(\resp for $g \in \SO_{2n}(F)$), 
\end{itemize}
we define 
\[
f'_G(\psi, s) 
= 
\prod_{\kappa \in \{\pm\}}
\frac{1}{(G_\kappa:G_\kappa^\circ)}
\left(\sum_{\pi_\kappa \in \Pi_{\psi_\kappa \otimes \eta_\kappa}} 
\pair{s_{\psi_\kappa \otimes \eta_\kappa}, \pi_\kappa}_{\psi_\kappa \otimes \eta_\kappa} \Theta_{\pi_\kappa}(f_{G_\kappa})\right).
\]
Then $f'_G(\psi,s)$ is independent of the choice of $f_{G_\pm}$, and 
the equation
\[
\tag{{\bf ECR2}}\label{ECR2}
f'_G(\psi, s) 
= 
\frac{1}{(G:G^\circ)}
\sum_{\pi \in \Pi_\psi} \pair{s \cdot s_\psi, \pi}_\psi \Theta_\pi(f_G)
\]
holds. 
\end{enumerate}

\begin{rem}\label{rem.ECR}
\begin{enumerate}
\item
We normalize the transfer factors such that they are consistent with Arthur's Whittaker normalization.
It gives a precise meaning of ``matching orbital integrals''.

\item
The pair of characters $(\eta_+, \eta_-)$ determines an $L$-embedding 
\[
\xi \colon {}^L(G_+^\circ \times G_-^\circ) \hookrightarrow {}^LG^\circ,
\] 
and the notion of matching orbital integrals depends on $(\eta_+, \eta_-)$ or $\xi$.

\item
A non-trivial pair of characters $(\eta_+, \eta_-)$ is necessary when $G = \Sp_{2n}(F)$ or $G = \U_n$ in general.
More precisely, 
if $G = \Sp_{2n}(F)$ and $\dim(\psi_\kappa) \equiv 1 \bmod 2$, 
we must take $\eta_{\kappa} = \det(\psi_\kappa)$. 
If $G = \U_n$ and $\dim(\psi_\kappa) \not\equiv n \bmod 2$, 
we need to choose a conjugate-symplectic character $\eta_\kappa$, 
for which there is no canonical choice.
\end{enumerate}
\end{rem}

\begin{rem}\label{packet_SO}
In \cite{Ar}, Arthur works exclusively with $\SO_{2n}(F)$, 
but we have elected to work with $\O_{2n}(F)$ in this paper. 
For a discussion of why this is natural and preferable, see the introduction of \cite{AG_O}. 
\par

Suppose that $G=\O_{2n}(F)$ and $\psi \in \Psi(G)$.
We shall explain how to reformulate Arthur's results in this setting. 
\begin{enumerate}
\item
Arthur defines a packet $\tl\Pi_\psi$ over $\Irr_\unit(\SO_{2n}(F))/\O_{2n}(F)$
together with a map $\tl\Pi_\psi \rightarrow \widehat{\Sc_\psi}$, $[\pi] \mapsto \pair{\cdot, [\pi]}_\psi$. 
We set $\Pi_\psi$ to be the inverse image of $\tl\Pi_\psi$
under the canonical map 
\[
\Irr_\unit(\O_{2n}(F)) \rightarrow \Irr_\unit(\SO_{2n}(F))/\O_{2n}(F)
\]
obtained by taking the orbit of an irreducible component of the restriction.
Then \cite[Theorem 2.2.4]{Ar} gives a map $\Pi_\psi \rightarrow \widehat{\AA_\psi}$ such that 
the diagram 
\[
\begin{CD}
\Pi_\psi @>>> \widehat{\AA_\psi} \\
@VVV @VVV \\
\tl\Pi_\psi @>>> \widehat{\Sc_\psi}
\end{CD}
\]
is commutative. 
In particular, $\Pi_\psi$ is stable under the determinant twist $\pi \mapsto \pi \otimes \det$. 

\item
Note that $\Pi_\psi$ is actually defined as a packet over 
$\Irr_\unit(\SO_{2n}(F) \rtimes \pair{\Ad(\epsilon)})$.
To fix an identification $\O_{2n}(F) \cong \SO_{2n}(F) \rtimes \pair{\Ad(\epsilon)}$,
we need to choose $\epsilon \in T \setminus T^\circ$. 
Hence the pairing $\pair{\cdot, \pi}_\psi$ for $\pi \in \Pi_\psi$ depends on this choice.

\item
As in \eqref{ECR1}, one can consider the distribution 
\[
\Theta_\psi(f_G) = \frac{1}{(G:G^\circ)} \sum_{\pi \in \Pi_\psi} \pair{s_\psi, \pi}_\psi \Theta_\pi(f_G), 
\quad f_G \in C_c^\infty(G). 
\]
On the other hand, 
if one denotes by $C_c^\infty(G^\circ)^\epsilon$ the subspace of $C_c^\infty(G^\circ)$
consisting of $f_G^\circ$ such that $f_G^\circ \circ \Ad(\epsilon) = f_G^\circ$, 
Arthur considers 
\[
\Theta_\psi^\circ(f_G^\circ) = \sum_{[\pi] \in \tl\Pi_\psi} \pair{s_\psi, [\pi]}_\psi \Theta_\pi(f_G^\circ), 
\quad f_G^\circ \in C_c^\infty(G^\circ)^\epsilon.
\]
They are related by the following commutative diagram 
\[
\begin{CD}
C_c^\infty(G^\circ) @>\Theta_\psi|_{C_c^\infty(G^\circ)}>> \C \\
@VVV @| \\
C_c^\infty(G^\circ)^\epsilon @>\Theta_\psi^\circ>> \C 
\end{CD}
\]
where the left arrow is defined by 
$f_G \mapsto \half{1}(f_G+f_G \circ \Ad(\epsilon))$.
Note that $\Theta_\psi$ is a distribution on $G$, 
but it is restricted to $C_c^\infty(G^\circ)$ in the diagram.
Since $f_G \circ \Ad(\epsilon)$ has the same transfer as $f_G$, 
our \eqref{ECR1} is the same as the ECR by Arthur \cite[Theorem 2.2.1]{Ar}. 
Similarly, \eqref{ECR2} is the same as \cite[Theorems 2.2.1, 2.2.4]{Ar}.
\end{enumerate}
\end{rem}

\subsection{Normalized intertwining operators}\label{sec.NIO}
For semi-standard Levi subgroups $M_1^\circ$ and $M_2^\circ$ of $G^\circ$, put
\[
N(M_1^\circ, M_2^\circ) 
= \{ g \in G \,|\, g M_1^\circ g^{-1} = M_2^\circ \}. 
\]
The group $M_1^\circ$ (\resp $M_2^\circ$) acts on this set by 
multiplication on the right (\resp the left).
We consider the Weyl set 
\[
W(M_1^\circ, M_2^\circ) 
= M_2^\circ \bs N(M_1^\circ, M_2^\circ) 
= N(M_1^\circ, M_2^\circ) /M_1^\circ.
\]
In particular, 
we write $W(M_1^\circ) = W(M_1^\circ, M_1^\circ) = N_{G}(M_1^\circ)/M_1^\circ$. 
We also set $W^{G^\circ} = N_{G^\circ}(T^\circ)/T^\circ$. 
\par

For $w \in W(M_1^\circ, M_2^\circ)$, 
there is a unique element $w_T \in N(M_1^\circ, M_2^\circ)/T^\circ$ such that 
it is a lift of $w$, and it maps the Borel pair $(B^\circ \cap M_1^\circ, T^\circ)$ of $M_1^\circ$
to the Borel pair $(B^\circ \cap M_2^\circ, T^\circ)$ of $M_2^\circ$. 
Unless $G = \O_{2n}(F)$ and $\det(w) = -1$, 
we have $w_T \in W^{G^\circ}$ and hence
the Langlands--Shelstad representative $\tl{w_T}$ of $w_T$ with respect to $\spl$. 
See \cite[p.~228]{LS}. 
We call $\tl{w} = \tl{w_T}$ the \emph{Tits lifting} of $w$. 
If $G = \O_{2n}(F)$ and $\det(w) = -1$, 
then $w_T \epsilon^{-1} \in W^{G^\circ}$ and we have Langlands--Shelstad representative $\tl{w_T \epsilon^{-1}}$.
Then we set $\tl{w} = \tl{w_T \epsilon^{-1}} \cdot \epsilon$ and call it the \emph{Tits lifting} of $w$. 
Note that $\tl{w}$ depends on the choice of $\epsilon$.
\par

Let $P = MN$ and $P' = M'N'$ be standard parabolic subgroups of $G$ 
with the semi-standard Levi components $M$ and $M'$, respectively, 
such that $W(M^\circ,M'^\circ) \not= \emptyset$.
For $w \in W(M^\circ,M'^\circ)$, an irreducible representation $\pi$ of $M$, and $\lambda \in \aa_{M,\C}^*$, 
we define a representation $w \pi_\lambda$ of $M'$ 
by $w \pi_\lambda(m') = \pi_\lambda(\tl{w}^{-1} m' \tl{w})$ realized on the space of $\pi$.

\begin{lem}\label{center}
For $w \in W(M^\circ, M'^\circ)$ and $w' \in W(M'^\circ, M''^\circ)$, 
if we write 
\[
\tl{w'w} = \tl{w}'\tl{w} \cdot z, 
\]
then $z$ belongs to the center $Z(M)$ of $M$. 
In particular, 
\[
(w'w)\pi_\lambda = w'(w\pi_\lambda).
\]
\end{lem}
\begin{proof}
Since $\tl{w'w}$ and $\tl{w}'\tl{w}$ are two representatives of $w'w \in W(M^\circ, M''^\circ)$, 
we have $z \in M^\circ$. 
Moreover, by definition, $z$ preserves the splitting of $M^\circ$.
Hence $z \in Z(M^\circ)$.
This completes the proof if $Z(M^\circ) = Z(M)$.
If $Z(M^\circ) \not= Z(M)$, then $M$ is of the form
\[
M = \GL_{k_1}(F) \times \dots \times \GL_{k_t}(F) \times \O_2(F).
\]
In this case, by calculating 
\[
\tl{w}_\alpha = \exp(X_\alpha) \exp(-X_{-\alpha}) \exp(X_\alpha)
\] 
for each simple root $\alpha$, where $w_\alpha$ is the simple reflection with respect to $\alpha$, 
one can check $z \in Z(M)$.
See Section \ref{secA.groups} for the root vectors $\{X_{\alpha}\}$. 
\end{proof}
By this lemma, for $w \in W(M^\circ, M'^\circ)$, we may write $M' = \tl{w}M\tl{w}^{-1}$ as $wMw^{-1}$.
\par

Now we have an intertwining operator
\[
J_P(w, \pi_\lambda) \colon I_P(\pi_\lambda) \rightarrow I_{P'}(w \pi_\lambda)
\]
given by (the meromorphic continuation of) the integral 
\[
(J_P(w, \pi_\lambda)f_\lambda)(g) = \int_{(\tl{w} N \tl{w}^{-1} \cap N') \backslash N'} f_\lambda(\tl{w}^{-1}ug) du.
\]

Suppose that $M$ is a proper Levi subgroup, 
and we are within an inductive argument.
Hence by inductive hypothesis, 
we have an $A$-packet $\Pi_{\psi}$ for $\psi \in \Psi(M)$ (with all the desired properties) at our disposal.  
Let $\pi \in \Pi_{\psi}$. 
In particular, $\pi$ is a unitary representation of $M$. 
Following Arthur \cite[Section 2.3]{Ar}, \cite[Section 3.3]{Mok}, 
we define the normalized intertwining operator $R_P(w,\pi_\lambda, \psi_\lambda)$ by
\[
R_P(w, \pi_\lambda, \psi_\lambda) = r_P(w, \psi_\lambda)^{-1} \cdot J_P(w,\pi_\lambda),
\]
with
\[
r_P(w, \psi_\lambda) = \lambda(w) \cdot \gamma_A(0,\psi_\lambda, \rho_{w^{-1} P'|P}^\vee, \psi_F)^{-1}
\]
where the notation is as follows.
\begin{itemize}
\item 
Put $\psi_{\lambda} = a_\lambda \cdot \psi$, 
where $a_\lambda \in Z^1(W_F, Z(\widehat{M^\circ}))$ is a $1$-cocycle 
whose class in $H^1(W_F, Z(\widehat{M^\circ}))$ corresponds to the character $m \mapsto e^{\pair{H_M(m), \lambda}}$ of $M^\circ$ 
(see \cite[Lemma A.1]{LMa}).

\item
For any finite dimensional representation $\rho$ of ${}^L M^\circ$, 
we write
\[
L(s, \psi, \rho) = L(s, \rho \circ \phi_{\psi}), \quad
\ep(s, \psi, \rho, \psi_F) = \ep(s, \rho \circ \phi_{\psi}, \psi_F)
\]
for the associated Artin $L$- and $\ep$-factors.
Then Arthur's modified gamma factor is defined by 
\[
\gamma_A(s, \psi, \rho, \psi_F) = \ep(s, \psi, \rho, \psi_F)
\frac{L(1+s, \psi, \rho)}{L(s, \psi, \rho)}.
\]
\item
Set $w^{-1} P' = \tl{w}^{-1} P' \tl{w}$ 
and write $\rho_{w^{-1} P'|P}$ for the adjoint representation of ${}^L M^\circ$ on
\[
\Ad(\tl{w})^{-1} \widehat{\nn}' / (\Ad(\tl{w})^{-1} \widehat{\nn}' \cap \widehat{\nn})
\]
with $\widehat\nn = \Lie(\widehat{N})$ and $\widehat\nn' = \Lie(\widehat{N}')$.
\item
Let $A_{T^\circ}$ be the split component of $T^\circ$.
For a reduced root $a$ of $A_{T^\circ}$ in $G^\circ$, 
we denote by $G_a$ the associated Levi subgroup of $G^\circ$ of semisimple rank $1$ 
and by $G_{a,\sc}$ the simply connected cover of the derived group of $G_a$.
Let $\Delta_1(w)$ (\resp $\Delta_2(w)$) be the set of reduced roots $a$ with $a>0$ and $wa < 0$ 
(with respect to $B^\circ$) such that 
$G_{a,\sc} \cong \Res_{F_a/F} \SL_2$ (\resp $\Res_{F_a/F} \SU_{E_a/F_a}(2,1)$), 
where $F_a$ is a finite extension of $F$ and $E_a$ is a quadratic extension of $F_a$.
Following \cite[(4.1)]{KeSh}, 
we define
\[
\lambda(w) = \lambda(w, \psi_F) = \prod_{a \in \Delta_1(w)} \lambda(F_a/F, \psi_F)
\prod_{a \in \Delta_2(w)} \lambda(E_a/F, \psi_F)^2 \lambda(F_a/F,\psi_F)^{-1},
\]
where $\lambda(F'/F,\psi_F)$ is the Langlands $\lambda$-factor
associated to a finite extension $F'$ of $F$.
\end{itemize}
\par

Recall that $R_P(w, \pi_\lambda, \psi_\lambda)$ is regular at $\lambda=0$ 
(\cite[Proposition 2.3.1]{Ar}, \cite[Proposition 3.3.1]{Mok}), 
and hence we have a well-defined operator 
\[
R_P(w, \pi, \psi) \colon I_P(\pi) \rightarrow I_{P'}(w\pi).
\]
When $\pi$ is tempered, taking its $L$-parameter $\phi$, 
we set $R_P(w, \pi_\lambda) = R_P(w, \pi_\lambda, \phi_\lambda)$. 
\par

Let $P'' = M'' N''$ be another standard parabolic subgroup of $G$ 
with the semi-standard Levi component $M''$ such that $W(M'^\circ,M''^\circ) \not= \emptyset$.
Then the normalized intertwining operators satisfy the following multiplicative property.

\begin{prop}\label{multiplicative}
Let $\pi$ be an irreducible tempered representation of $M$.
Then we have
\[
R_P(w'w, \pi_\lambda) = R_{P'}(w', w \pi_\lambda) \circ R_P(w, \pi_\lambda)
\]
for $w \in W(M^\circ, M'^\circ)$ and $w' \in W(M'^\circ, M''^\circ)$.
\end{prop}

\begin{proof}
When $G = \GL_N(E)$, the assertion was proved in \cite{Sh3}, \cite{Ar1}.
(Note that the local factors of Shahidi agree with those of Jacquet--Piatetski-Shapiro--Shalika by \cite{Sh4} and hence with the Artin factors by the desiderata of the local Langlands correspondence.)
\par

When $G$ is a classical group and $M_1 = M_2 = M_3$, 
the assertion was proved in \cite[(2.3.28)]{Ar}, \cite[Proposition 3.3.5]{Mok}
at least unless $G = \O_{2n}(F)$. 
The general case will be handled in Section \ref{sec.ell} below.
\end{proof}

The following is an important property of $\gamma_A$-factors throughout this paper. 
\begin{lem}\label{holomorphic}
Suppose that $F$ is non-archimedean.
Let $\phi_1$ and $\phi_2$ be two conjugate-self-dual representations of $W_E \times \SL_2(\C)$ of dimension $N$. 
For $i=1,2$, define a representation $\lambda_{\phi_i}$ of $W_E$ by 
\[
\lambda_{\phi_i}(w) = \phi_i\left(w, 
\begin{pmatrix}
|w|_E^{\half{1}} & 0 \\ 0 & |w|_E^{-\half{1}}
\end{pmatrix}
\right).
\]
If $\lambda_{\phi_1} \cong \lambda_{\phi_2}$, then the quotient 
\[
\frac{\gamma_A(s,\phi_1,\psi_E)}{\gamma_A(s,\phi_2,\psi_E)}
\]
is holomorphic at $s=0$, and its special value at $s=0$ is in $\{\pm1\}$.
\end{lem}
\begin{proof}
Let $\pi_{\phi_i}$ be the irreducible representation of $\GL_N(E)$ corresponding to $\phi_i$.
If $\lambda_{\phi_1} \cong \lambda_{\phi_2}$, 
then $\pi_{\phi_1}$ and $\pi_{\phi_2}$ share the same cuspidal support, 
and hence we obtain an equation of Godement--Jacquet $\gamma$-factors
\[
\gamma(s, \pi_{\phi_1}, \psi_E) = \gamma(s, \pi_{\phi_2}, \psi_E). 
\]
Since $\gamma(s, \pi_{\phi_i}, \psi_E)$ is equal to the usual $\gamma$-factor $\gamma(s, \phi_i, \psi_E)$ 
in the Galois side, 
it is enough to consider the quotient $\gamma_A(s, \phi_i ,\psi_E)/\gamma(s, \phi_i, \psi_E)$. 
Note that 
\[
\frac{\gamma_A(s,\phi_i,\psi_E)}{\gamma(s,\phi_i,\psi_E)}
= \frac{L(1+s, \phi_i)}{L(1-s, \phi_i^\vee)} 
= \frac{L(1+s, \phi_i)}{L(1-s, {}^c\phi_i)} 
= \frac{L(1+s, \phi_i)}{L(1-s, \phi_i)}. 
\]
If we write the Laurent expansion of $L(1+s, \phi_i)$ as
\[
L(1+s, \phi_i) = as^m + (\text{higher terms})
\]
with $a \not= 0$, 
then we conclude that 
\[
\left.
\frac{\gamma_A(s,\phi_i,\psi_E)}{\gamma(s,\phi_i,\psi_E)}
\right|_{s=0}
= 
\left.
\frac{L(1+s, \phi_i)}{L(1-s, \phi_i)}
\right|_{s=0}
= (-1)^m.
\]
This proves the lemma.
\end{proof}

In the rest of this section, we state three main theorems. 
Before stating them, let us clarify the dependence of these results. 
\begin{itemize}
\item
The first main theorem (Theorem \ref{main1}) depends on results in Sections \ref{sec.Arthur} and \ref{sec.NIO}
for proper Levi subgroups.
\item
The second main theorem (Theorem \ref{main2}) is for $G = \GL_N(E)$ so that 
it is independent of Section \ref{sec.Arthur}. 
However, it still uses results in Section \ref{sec.NIO}, in particular, Proposition \ref{multiplicative}. 
\item
In the third main theorem (Theorem \ref{main3}), 
we will use results in Sections \ref{sec.Arthur} and \ref{sec.NIO}
not only for proper Levi subgroups, but also for other classical groups $G'$ such that 
$\dim(\St_{\widehat{G'}}) < \dim(\St_{\widehat{G}})$.
For more precision, see Hypothesis \ref{hyp.arthur}. 
\end{itemize}
\par

\subsection{Main Theorem 1: [A27]}\label{sec.A27}
We set $G$ to be $\GL_N(E)$ or 
a (possibly disconnected) quasi-split classical group as in the previous subsection. 
Let $P=MN$ be a standard parabolic subgroup of $G$, 
and let $\overline{P} = M\overline{N}$ be the parabolic subgroup of $G$ opposite to $P$. 
Recall that we obtain a Whittaker datum $\ww_M$ of $M^\circ$ from the $F$-splitting $\spl$ of $G^\circ$.
\par

Let $\pi$ be an irreducible tempered representation of $M$. 
Suppose that $\pi$ admits a non-trivial $\ww_M$-Whittaker functional $\omega$. 
We may also regard $\omega$ as a $\ww_M$-Whittaker functional on $\pi_\lambda$ for all $\lambda \in \aa_{M,\C}^*$.
Then $\omega$ gives rise to a $\ww$-Whittaker functional $\Omega(\pi_\lambda)$ on $I_P(\pi_\lambda)$ given by (the holomorphic continuation of) the Jacquet integral
\[
\Omega(\pi_\lambda) f = \int_{N'} \omega(f(\tl{w}_0^{-1} n')) \chi(n')^{-1} dn'.
\]
Here $N' = \tl{w}_0 \overline{N} \tl{w}_0^{-1}$ and $w_0 = w_\ell w_\ell^M$, 
where $w_\ell$ and $w_\ell^M$ are the longest elements in $W^{G^\circ}$ and $W^{M^\circ}$, respectively.
\par

Let $P'=M'N$ be another standard parabolic subgroup of $G$, 
and $w \in W(M^\circ, M'^\circ)$. 
Then we may also regard $\omega$ as a $\ww_{M'}$-Whittaker functional on $w \pi_\lambda$.
Hence, it gives rise to a $\ww$-Whittaker functional $\Omega(w \pi_\lambda)$ on $I_{P'}(w \pi_\lambda)$. 
\par

The following is our first main theorem, which was supposed to be proven in \cite{A27}. 
\begin{thm}[cf. {\cite[Theorem 2.5.1 (b)]{Ar}, \cite[Proposition 3.5.3 (a)]{Mok}}]\label{main1}
Let $\pi$ be an irreducible $\ww_M$-generic tempered representation of $M$. 

\begin{enumerate}
\item
If $w \in W(M^\circ)$ satisfies that $w\pi \cong \pi$, 
then 
\[
\Omega(w\pi) \circ R_P(w, \pi) = 
\left\{
\begin{aligned}
&\Omega(\epsilon\pi) \circ L(\epsilon) \iif G = \O_{2n}(F), \, \det(w) = -1, \\
&\Omega(\pi) \other.
\end{aligned}
\right.
\]
Here 
\[
L(\epsilon) \colon I_P(\pi) \rightarrow I_{\epsilon P \epsilon^{-1}}(\epsilon \pi),\,
(L(\epsilon)f)(g) = f(\epsilon^{-1}g).
\]

\item
Suppose that $G = \GL_N(E)$. 
If $w \in W(M, \theta(M))$ satisfies that $w\pi \cong \pi \circ \theta$, 
then 
\[
\Omega(w\pi) \circ R_P(w, \pi) = \Omega(\pi).
\]
\end{enumerate}
\end{thm}

This theorem is regarded as the \emph{local intertwining relation} for generic tempered representations.
We will prove Theorem \ref{main1} in Section \ref{sec.main1}.
\par

Note that the proof of \cite[Lemma 2.5.5]{Ar} was also supposed to be included in \cite{A27}. 
We will show it in Appendix \ref{sec.255}. 

\subsection{Main Theorem 2: [A26]}\label{sec.main2}
The second main theorem concerns $G = \GL_N(E)$. 
Recall that we have an involution $\theta$ on $G$. 
For a function $f$ on $G$, 
define a new function $\theta^*(f)$ on $G$ by 
\[
\theta^*(f)(g) = f(\theta(g)).
\]
\par

Let $P=MN$ be a standard parabolic subgroup of $G$. 
Note that $\theta(P) = \theta(M)\theta(N)$ is also a standard parabolic subgroup. 
Let $\psi$ be an $A$-parameter for $M$, 
and let $\pi_\psi$ be the corresponding irreducible unitary representation of $M$. 
Suppose that the composition 
\[
L_E \times \SL_2(\C) \xrightarrow{\psi} \widehat{M} \hookrightarrow \widehat{G}
\]
is conjugate-self-dual. 
Then the corresponding representation, which is the irreducible induction $I_P(\pi_\psi)$, 
is an irreducible unitary conjugate-self-dual representation of $G$.
For the irreducibility of $I_P(\pi_\psi)$, see \cite{Ber-P}.
Recall from Section \ref{sec.thetaA} that we have a specific linear isomorphism 
\[
\theta_A \colon I_P(\pi_\psi) \xrightarrow{\sim} I_P(\pi_\psi). 
\]
\par

More specifically, we assume that
there is an element $w \in W(\theta(M),M)$ such that $w(\pi_\psi \circ \theta) \cong \pi_\psi$. 
Similar to the definition of $\theta_A$, 
we have a normalized isomorphism
\[
\tl\pi_\psi(w \rtimes \theta) \colon w(\pi_\psi \circ \theta) \xrightarrow{\sim} \pi_\psi
\]
as representations of $M$
by using a $\ww_M$-Whittaker functional on the standard module $\II_{\pi_\psi} \cong w(\II_{\pi_\psi} \circ \theta)$ of $M$
whose Langlands quotient is $\pi_\psi \cong w(\pi_\psi \circ \theta)$.
Then we can define a self-intertwining operator 
\[
\tl{R}_P(\theta \circ w, \tl{\pi}_\psi) \colon I_P(\pi_\psi) \rightarrow I_P(\pi_\psi)
\]
by the composition
\[
I_P(\pi_\psi) \xrightarrow{\theta^*} I_{\theta(P)}(\pi_\psi \circ \theta)
\xrightarrow{R_{\theta(P)}(w, \pi_\psi \circ \theta, \psi)} I_{P}(w(\pi_\psi \circ \theta))
\xrightarrow{I_P(\tl\pi_\psi(w \rtimes \theta))} I_P(\pi_\psi).
\]
If we write $(h\cdot f)(g) = f(gh)$ for $g,h \in G$, then we have
\[
\tl{R}_P(\theta \circ w, \tl{\pi}_\psi)(h \cdot f)(g) = \tl{R}_P(\theta \circ w, \tl{\pi}_\psi)f(g \cdot \theta(h)).  
\]
Hence $\tl{R}_P(\theta \circ w, \tl{\pi}_\psi)$ is a constant multiple of $\theta_A$.
\par

The second main theorem, which was supposed to be proven in \cite{A26}, is now stated as follows. 
\begin{thm}[cf. {\cite[Theorem 2.5.3]{Ar}, \cite[Proposition 3.5.1 (b)]{Mok}}]\label{main2}
Let $P=MN$ be a standard parabolic subgroup of $G=\GL_N(E)$, 
and let $\psi$ be an $A$-parameter for $M$. 
Then for any $w \in W(\theta(M),M)$ with $w(\pi_\psi \circ \theta) \cong \pi_\psi$, 
we have
\[
\tl{R}_P(\theta \circ w, \tl{\pi}_\psi) = \theta_A.
\]
\end{thm}
We can say that this theorem is the \emph{twisted local intertwining relation} for $\GL_N(E)$. 
See also \cite[Corollary 2.5.4]{Ar} for Arthur's form of local intertwining relation for twisted $\GL_N(E)$.
We will prove Theorem \ref{main2} in Section \ref{LIR_GL}.

\subsection{Main Theorem 3: [A25]}\label{sec.main3}
The third main theorem concerns classical groups over a non-archimedean local field.
\par

Assume that $F$ is a non-archimedean local field of characteristic zero. 
Hence $L_E = W_E \times \SL_2(\C)$. 
For a representation $\psi \colon W_E \times \SL_2(\C) \times \SL_2(\C) \rightarrow \GL_N(\C)$, 
we define its Aubert dual $\widehat\psi \colon W_E \times \SL_2(\C) \times \SL_2(\C) \rightarrow \GL_N(\C)$ by 
\[
\widehat\psi(w,\alpha_1,\alpha_2) = \psi(w,\alpha_2,\alpha_1).
\]
We say that $\psi$ is co-tempered if $\psi = \widehat\phi$ for some $\phi$ 
with $\phi|_{\{\1_{W_F}\} \times \{\1_{\SL_2(\C)}\} \times \SL_2(\C)} = \1$. 
\par

Let $G$ be one of the following quasi-split classical groups
\[
\SO_{2n+1}(F), \quad \Sp_{2n}(F), \quad
\O_{2n}(F), \quad \U_n.
\]
Fix a standard parabolic subgroup $P = MN_P$ of $G$ 
such that $M \cong \GL_{k_t}(E) \times \dots \times \GL_{k_1}(E) \times G_0$, 
where $G_0$ is a classical group of the same type as $G$. 
Let $\psi_M = \psi_t \oplus \dots \oplus \psi_1 \oplus \psi_0$ be an $A$-parameter for $M$, 
where $\psi_i$ (\resp $\psi_0$) is an $A$-parameter for $\GL_{k_i}(E)$ for $1 \leq i \leq t$ (\resp $G_0$).
It gives the $A$-parameter 
\[
\psi = \psi_t \oplus \dots \oplus \psi_1 \oplus \psi_0 \oplus {}^c\psi_1^\vee \oplus \dots \oplus {}^c\psi_t^\vee
\]
for $G$. 
We assume that $\psi_i = \rho_i \boxtimes S_{a_i} \boxtimes S_{b_i}$ is irreducible and conjugate-self-dual. 
\par

Let $V = \C^N$ and decompose 
\[
V = V_t \oplus \dots \oplus V_1 \oplus V_0 \oplus V_{-1} \oplus \dots \oplus V_{-t}
\]
such that $\dim(V_i) = \dim(\psi_i)$ for $0 \leq i \leq t$, 
together with a fixed isomorphism $I_i \colon V_i \xrightarrow{\sim} V_{-i}$ for $1 \leq i \leq t$. 
We regard $\psi_i$ as a homomorphism 
$\psi_i \colon W_E \times \SL_2(\C) \times \SL_2(\C) \rightarrow \GL(V_i)$. 
Let $\SS_t$ be the symmetric group on $\{1,\dots, t\}$. 
We identify $\sigma \in \SS_t$ with an element in 
$\GL(V_t \oplus \dots \oplus V_1) \subset \widehat{M}$ with entries in $\{0,1\}$ 
such that
$\sigma^{-1}(\GL(V_t) \times \dots \times \GL(V_1))\sigma
= \GL(V_{\sigma(t)}) \times \dots \times \GL(V_{\sigma(1)})$. 
Via $\widehat{M} \hookrightarrow \widehat{G}$, 
we also identify $\sigma \in \SS_t$ with an element of $\GL_N(\C) = \GL(V)$.
On the other hand, for $1 \leq i \leq t$, we set 
\[
u_i = \begin{pmatrix}
0& I_i \\
I_i^{-1} &0
\end{pmatrix}
\quad\text{or}\quad
u_i = \begin{pmatrix}
0& I_i \\
-I_i^{-1} &0
\end{pmatrix}
\]
in $\GL(V_i \oplus V_{-i})$
according to whether $\psi_i$ is of the same type as $\psi_0$ or not.
We regard $u_i$ as an element in $\GL(V)$ by setting $u_i|_{V_j} = \id_{V_j}$ for $j \not= \pm i$.
\par

\begin{rem}
When $G = \O_{2n}(F)$ and $k_i = \dim(V_i)$ is odd, 
since $\psi_i$ is an irreducible self-dual representation of dimension $k_i$, 
it must be of orthogonal type. 
Hence $u_i$ is always 
$\begin{pmatrix}
0&I_i \\
I_i^{-1}&0
\end{pmatrix}$
in $\GL(V_i \oplus V_{-i})$ so that $\det(u_i) = -1$.
\end{rem}

If $G \not= \O_{2n}(F)$, 
following \cite[Section 2.4]{Ar}, 
we write $N_\psi = N_\psi(G,M)$ for the normalizer of $A_{\widehat{M^\circ}}$ in $S_\psi$, 
and define $\NN_\psi$ by 
\[
\NN_\psi = \pi_0(N_\psi/Z(\widehat{G^\circ})^\Gamma). 
\]
When $G = \O_{2n}(F)$, 
we replace $S_{\psi} = \Cent(\im(\psi),\widehat{G^\circ})$ with $\Cent(\im(\psi),\O_{2n}(\C))$
to define $\NN_\psi$.
Note that $\NN_\psi$ is generated by $\AA_{\psi_0}$ and 
\[
\{u_1^{\epsilon_1} \dots u_t^{\epsilon_t} \,|\, \epsilon_i \in \Z/2\Z\} \rtimes 
\{\sigma \in \SS_t \,|\, \psi_{\sigma(i)} \cong \psi_i \, (1\leq i \leq t)\}.
\]
See \cite[(2.4.3)]{Ar}. 
We have two canonical maps $\NN_\psi \rightarrow W(M^\circ)$ and $\NN_\psi \rightarrow \AA_\psi$, 
which are denoted
by $u \mapsto w_u$ and $u \mapsto s_u$, respectively.
More precisely, 
for $u = s_0 u_1^{\epsilon_1} \dots u_t^{\epsilon_t} \sigma \in \NN_\psi$ with $s_0 \in \AA_{\psi_0}$, 
letting $I_u$ be the set of $1 \leq i \leq t$ such that $\epsilon_i = 1$ and $\psi_i$ is of the same type as $\psi_0$, 
we have
\[
s_u = s_0 + \sum_{i \in I_u} e(\rho_i , a_i, b_i) \in A_\psi.
\]
\par

Let $u \in \NN_\psi$ and let $w_u \in W(M^\circ)$ be its image. 
It satisfies that $w_u\pi_M \cong \pi_M$ for any $\pi_M \in \Pi_\psi$. 
Moreover, as in \cite[Section 2.4]{Ar} and \cite[Section 3.4]{Mok}, 
there is a linear isomorphism
\[
\pair{\tl{u}, \tl\pi_M}\tl{\pi}_M(w_u) \colon \pi_M \rightarrow \pi_M
\]
making the diagram
\[
\xymatrix{
\pi_M \ar@{->}[d]_{\pi_M(\tl{w}_u^{-1} m \tl{w}_u)} \ar@{->}[rrr]^{\pair{\tl{u}, \tl\pi_M}\tl{\pi}_M(w_u)} &&& \pi_M \ar@{->}[d]^{\pi_M(m)} \\
\pi_M  \ar@{->}[rrr]^{\pair{\tl{u}, \tl\pi_M}\tl{\pi}_M(w_u)} &&& \pi_M 
}
\]
commutative for any $m \in M$. 
In this paper, we will use the symbol $\pair{\tl{u}, \tl\pi_M}\tl{\pi}_M(w_u)$ to denote this map, 
and we do not separate it into two objects $\pair{\tl{u}, \tl\pi_M}$ and $\tl{\pi}_M(w_u)$.
\par

We recall the definition of the operator $\pair{\tl{u}, \tl\pi_M}\tl{\pi}_M(w_u) \colon \pi_M \rightarrow \pi_M$.
Note that $u \in \NN_\psi$ normalizes $\widehat{M^\circ}$. 
If we regard $u$ as an element of $W(\widehat{M^\circ})$, 
then we have its Tits lifting $\tl{u}$. 
Conjugation by $\tl{u}$ normalizes $\widehat{M^\circ}$ and preserves the pinning inherited from $\widehat{G^\circ}$.
Write $\widehat\theta$ for the resulting automorphism on $\widehat{M^\circ}$, 
and $\theta$ for its dual, an automorphism of $M^\circ$. 
Write $u = s \tl{u}$. 
Then $s$ lies in the $\theta$-twisted centralizer of $\psi_M$. 
The pair $(s, \psi_M)$ determines a twisted endoscopic datum 
$(M'^\circ, s, \xi)$ and a parameter $\psi_{M'}$ such that $\psi_M = \xi \circ \psi_{M'}$. 
The isomorphism $\pair{\tl{u}, \tl\pi_M}\tl{\pi}_M(w_u)$ is normalized by requiring that
the twisted endoscopic identity
\[
\sum_{\pi_{M'} \in \Pi_{\psi_{M'}}} \pair{s_{\psi_{M'}}, \pi_{M'}}_{\psi_{M'}} \Theta_{\pi_{M'}}(f_{M'})
= 
\sum_{\pi_M \in \Pi_{\psi_M}} \pair{s_{\psi_M}, \pi_M}_{\psi_M}
\tr(\pair{\tl{u}, \tl\pi_M}\tl{\pi}_M(w_u) \circ \pi_M(f_M))
\]
holds whenever $f_M \in C_c^\infty(M)$ and $f_{M'} \in C_c^\infty(M')$ have matching orbital integrals.
\par

We define the normalized self-intertwining operator 
\[
\pair{\tl{u}, \tl\pi_M} R_P(w_u, \tl\pi_M, \psi_M) \colon I_P(\pi_M) \rightarrow I_P(\pi_M)
\]
by 
\[
\pair{\tl{u}, \tl\pi_M} R_P(w_u, \tl\pi_M, \psi_M)f(g) 
= \pair{\tl{u}, \tl\pi_M}\tl{\pi}_M(w_u)\left(
R_P(w_u, \pi_M, \psi_M)f(g)
\right). 
\]

Now we assume that 
the (multi-)set $\Pi_\psi$ of irreducible components of $I_P(\pi_M)$ for $\pi_M \in \Pi_{\psi_M}$
is equipped with a pairing $\pair{\cdot, \pi}_\psi$ satisfying \eqref{ECR1}, \eqref{ECR2}
and that 
\[
\pair{\cdot, \pi}_\psi|_{A_{\psi_0}} = \pair{\cdot, \pi_0}_{\psi_0}
\]
if $\pi \subset I_P(\pi_M)$
where $\pi_M = \pi_{\psi_t} \boxtimes \dots \boxtimes \pi_{\psi_1} \boxtimes \pi_0$ with $\pi_0 \in \Pi_{\psi_0}$. 
See \cite[Proposition 2.4.3]{Ar} and \cite[Proposition 3.4.4]{Mok}. 
Define a distribution $f_G(\psi, u)$ on $G$ for $u \in \NN_\psi$ by 
\[
f_G(\psi,u) 
= \frac{1}{(G:G^\circ)}
\sum_{\pi_M \in \Pi_{\psi_M}} 
\tr( \pair{\tl{u}, \tl\pi_M} R_P(w_u, \tl\pi_M, \psi_M)I_P(\pi_M, f))
\]
for $f \in C_c^\infty(G)$.
Then Arthur's \emph{local intertwining relation}  
(\cite[Theorems 2.4.1, 2.4.4]{Ar}, \cite[Theorem 3.4.3]{Mok}) states that 
the equation
\[
\tag{{\bf A-LIR}} \label{A-LIR}
f'_G(\psi, s_\psi s_u) = f_G(\psi,u) 
\]
holds for $u \in \NN_\psi$, 
where the left-hand side is defined in \eqref{ECR2} in Section \ref{sec.Arthur}. 
Notice that if $G = \O_{2n}(F)$, then $\tl{w}_u$ can be in $G \setminus G^\circ$, in which case, 
\eqref{A-LIR} is \cite[Theorem 2.4.4]{Ar}.
\par

On the other hand, in this paper, we consider the following statement for our \emph{local intertwining relation}.
Fix an irreducible summand $\pi \subset I_P(\pi_M)$.
Then the equation
\[
\tag{{\bf LIR}}\label{LIR}
\pair{\tl{u}, \tl\pi_M} R_P(w_u, \tl\pi_M, \psi_M)|_\pi = \pair{s_u, \pi}_\psi \cdot \id_\pi
\]
holds for any $u \in \NN_\psi$. 
Notice that this statement is slightly different from \eqref{A-LIR}. 
We clarify the relation between \eqref{A-LIR} and our \eqref{LIR}.

\begin{lem}\label{A_LIR}
Assume the existence of the $A$-packet $\Pi_\psi$ together with the pairing $\pair{\cdot, \pi}_\psi$
satisfying \eqref{ECR1} and \eqref{ECR2}. 
We assume further that we know that 
\begin{itemize}
\item
$I_P(\pi_M)$ is multiplicity-free for any $\pi_M \in \Pi_{\psi_M}$; and
\item
for $\pi_M, \pi'_M \in \Pi_{\psi_M}$, if $\pi_M \not\cong \pi'_M$, 
then $I_P(\pi_M)$ and $I_P(\pi'_M)$ have no common irreducible summand.
\end{itemize}
Then \eqref{A-LIR} holds
if and only if 
our \eqref{LIR} holds for each irreducible summand $\pi \subset I_P(\pi_M)$ and $\pi_M \in \Pi_{\psi_M}$.
\end{lem}
\begin{proof}
Note that we do not assume whether $\Pi_{\psi_M}$ is multiplicity-free.
Only in this proof, we denote the canonical map $\Pi_\psi \rightarrow \Irr_\unit(G)$ by $\pi \mapsto [\pi]$, 
and the multiplicity of $\sigma \in \Irr_\unit(G)$ in $\Pi_\psi$, i.e., 
the cardinality of the fiber of $\sigma$ under this map, by $m_\psi(\sigma)$. 
By our assumptions, 
for any $\sigma \in \Irr(G)$ with $m_\psi(\sigma) > 0$,
there exists exactly one $\sigma_M \in \Irr(M)$ 
such that $\sigma \subset I_P(\sigma_M)$ and $m_{\psi_M}(\sigma_M) > 0$.
Moreover, $m_\psi(\sigma) = m_{\psi_M}(\sigma_M)$.
\par

Fix $u \in \NN_\psi$. 
By \eqref{ECR2}, $f'_G(\psi, s_\psi s_u)$ is equal to 
\[
\frac{1}{(G:G^\circ)}
\sum_{\pi \in \Pi_\psi} \pair{s_u, \pi}_\psi \Theta_\pi(f_G)
= 
\frac{1}{(G:G^\circ)} \sum_{\substack{\sigma \in \Irr(G) \\ m_\psi(\sigma) > 0}} 
\left(\sum_{\substack{\pi \in \Pi_\psi \\ [\pi] = \sigma}} \pair{s_u, \pi}_\psi \right) \Theta_\sigma(f_G).
\]
If $\pi \subset I_P(\pi_M)$, 
since we assume that $\pi$ appears in $I_P(\pi_M)$ with multiplicity one, 
$\pair{\tl{u}, \tl\pi_M} R_P(w_u, \tl\pi_M, \psi_M)$ acts on $\pi$ by a scalar $c_{[\pi]}$, 
which depends only on $[\pi]$ and $[\pi_M]$. 
Then one can write $f_G(\psi,u)$ as
\begin{align*}
f_G(\psi,u) &= \frac{1}{(G:G^\circ)} \sum_{\sigma \in \Irr(G)}
\left(\sum_{\substack{\pi_M \in \Pi_{\psi_M} \\ \exists \pi \subset I_P(\pi_M), [\pi] = \sigma}} 
c_{[\pi]} \right) \Theta_\sigma(f_G)
\\&= 
\frac{1}{(G:G^\circ)} \sum_{\substack{\sigma \in \Irr(G) \\ m_\psi(\sigma) > 0}} 
\left(
\sum_{\substack{\pi_M \in \Pi_{\psi_M} \\ [\pi_M] = \sigma_M}}
c_{\sigma} \right)
\Theta_\sigma(f_G)
\\&= 
\frac{1}{(G:G^\circ)} \sum_{\substack{\sigma \in \Irr(G) \\ m_\psi(\sigma) > 0}} 
m_\psi(\sigma) c_{\sigma} 
\Theta_\sigma(f_G).
\end{align*}
By the linear independence of the characters $\Theta_\sigma$, 
\eqref{A-LIR} holds if and only if
\[
c_{\sigma} = \frac{1}{m_\psi(\sigma)} \sum_{\substack{\pi \in \Pi_\psi \\ [\pi] = \sigma}} \pair{s_u, \pi}_\psi.
\]
Note that $\pair{s_u, \pi}_\psi \in \{\pm1\}$, whereas $c_{[\pi]}$ is a root of unity
since 
\[
\NN_\psi \ni u \mapsto \pair{\tl{u}, \tl\pi_M} R_P(w_u, \tl\pi_M, \psi_M)
\]
is multiplicative by Proposition \ref{multiplicative} and 
\cite[Lemma 2.5.3]{KMSW}
(see also the paragraph in \cite{Ar} containing (2.4.2)). 
Hence the above equation holds if and only if 
$\pair{s_u, \pi}_\psi = c_{\sigma}$ for all $\pi \in \Pi_\psi$ with $[\pi] = \sigma$.
This is our \eqref{LIR}.
\end{proof}

\begin{rem}\label{MF}
The multiplicity-free assumptions can be proven in generality using M{\oe}glin's explicit construction of $A$-packets. 
However, 
when $\psi_M$ is tempered or co-tempered, 
one can argue more directly as follows. 
The co-tempered case will be deduced from the tempered case by the construction 
(see Theorem \ref{ECR-cotemp}), 
and for the tempered case, the multiplicity-free assumptions are included 
in \cite[Theorem 1.5.1 (b)]{Ar} and \cite[Theorem 2.5.1 (b)]{Mok}. 
In Section \ref{(c)} below, 
we will need the multiplicity-free assumptions for slightly more general parameters, 
which we will prove in Lemma \ref{mult-free}. 
In this paper, except for Section \ref{(c)}, we may identify \eqref{A-LIR} with our \eqref{LIR}.
\end{rem}

In light of the inductive setting as in \cite[Chapter 7]{Ar} in which \cite{A25} is positioned, 
as explained in Sections \ref{sec.context} and \ref{intro.main3}, 
we are free to assume the following.
\begin{hyp}\label{hyp.arthur}
There are $A$-packets satisfying \eqref{ECR1}, \eqref{ECR2} and \eqref{A-LIR} associated to 
\begin{itemize}
\item
all tempered $L$-parameters for $G$; 
\item
all $A$-parameters for $G'$ with $G'$ any classical group 
such that $\dim(\St_{\widehat{G'}}) < \dim(\St_{\widehat{G}})$.
\end{itemize}
In particular, we have the $A$-packet $\Pi_{\psi_M}$ for $\psi_M \in \Psi(M)$,
where $M$ is an arbitrary proper Levi subgroup of $G$.
\end{hyp}

We state the third main theorem, 
which was supposed to be proven in \cite{A25}. 

\begin{thm}[cf. {\cite[Section 7.1]{Ar}, \cite[Section 8.2]{Mok}}]\label{main3}
Assume Hypothesis \ref{hyp.arthur}. 

\begin{enumerate}
\item
For any co-tempered $A$-parameter $\psi = \widehat\phi \in \Psi(G)$, 
we can construct an $A$-packet $\Pi_\psi$ 
together with a pairing $\pair{\cdot, \pi}_\psi$ for $\pi \in \Pi_\psi$
which satisfies \eqref{ECR1} and \eqref{ECR2}. 
Moreover, $\Pi_\psi$ is a (multiplicity-free) subset of $\Irr(G)$. 

\item
Let $P=MN_P$ be a parabolic subgroup of $G$ 
with $M \cong \GL_{k_t}(E) \times \dots \times \GL_{k_1}(E) \times G_0$, 
and let $\psi_M = \widehat\phi_M = \psi_t \oplus \dots \oplus \psi_1 \oplus \psi_0$ 
be a co-tempered $A$-parameter for $M$ 
such that $\psi_i$ is irreducible and conjugate-self-dual for $1 \leq i \leq t$.
Set $\psi = \iota \circ \psi_M \in \Psi(G)$ with $\iota \colon {}^LM \hookrightarrow {}^LG$. 
Then \eqref{LIR} also holds for every irreducible summand $\pi \subset I_P(\pi_M)$ for any $\pi_M \in \Pi_{\psi_M}$. 
\end{enumerate}
\end{thm}

\begin{rem}
The $A$-packet $\Pi_\psi$ for a given co-tempered $A$-parameter $\psi = \widehat\phi$ 
is constructed by Aubert duality from the tempered $L$-packet $\Pi_\phi$. 
To be precise, see Theorem \ref{ECR-cotemp} below.
At this stage, it is not known that each representation $\pi \in \Pi_\psi$ is unitary. 
The unitarity will be proven after establishing \cite[Proposition 7.4.3]{Ar} and \cite[Proposition 8.4.2]{Mok}.
\end{rem}

For co-tempered $A$-parameters
we will establish \eqref{ECR1} and \eqref{ECR2} in Section \ref{sec.ECR.cotemp}, 
whereas \eqref{LIR} will be proven in Sections \ref{sec.LIR} and \ref{computations} using results in Section \ref{sec.character}.
As a consequence of Theorem \ref{main3} together with Lemma \ref{A_LIR}, 
we have the following. 

\begin{cor}
Assume Hypothesis \ref{hyp.arthur}. 
For any co-tempered $A$-parameter $\psi \in \Psi(G)$ and for any $u \in \NN_\psi$, 
we have an identity 
\[
f'_G(\psi, s_\psi s_u) = f_G(\psi,u)
\]
of distributions on $G$. 
\end{cor}

\section{Normalizations of intertwining operators}\label{sec.main1}
The purpose of this section is to prove Theorem \ref{main1}. 
Notice that almost the same assertion was already proven by Shahidi \cite{Sh7}, 
but he used his own normalization of the intertwining operators.
So what we have to do is to compare Arthur's normalization with Shahidi's.
\par

\subsection{Local coefficients}\label{sec.C_P}
Let $\ww = (B^\circ, \chi)$ be a Whittaker datum for $G^\circ$.
We say that an irreducible representation $\pi^\circ$ of $G^\circ$ is $\ww$-generic if 
\[
\dim_\C \Hom_U(\pi^\circ, \chi) \neq 0,
\]
in which case 
\[
\dim_\C \Hom_U(\pi^\circ, \chi) = 1
\]
by the uniqueness of Whittaker functionals.
For $G = \O_{2n}(F)$, we also say that an irreducible representation $\pi$ of $G$ is $\ww$-generic if 
\[
\dim_\C \Hom_U(\pi, \chi) \neq 0,
\]
but as we will see in the next lemma, the uniqueness of Whittaker functionals does not necessarily hold.
Recall that we have chosen 
an element $\epsilon \in T \setminus T^\circ \subset G \setminus G^\circ$ 
with $\epsilon^2 = \1$ such that $\epsilon$ preserves $\spl$.
In particular, $\chi \circ \Ad(\epsilon) = \chi$.
For any representation $\pi^\circ$ of $G^\circ$, we define a representation $\epsilon \pi^\circ$ of $G^\circ$ by $\epsilon \pi^\circ(g) = \pi^\circ(\epsilon^{-1} g \epsilon)$.

\begin{lem}\label{mult2}
Suppose that $G = \O_{2n}(F)$.
Let $\pi$ be an irreducible $\ww$-generic representation of $G$.
\begin{itemize}
\item[(a)]
If $\pi|_{G^\circ}$ is irreducible, then $\pi|_{G^\circ}$ is $\ww$-generic and we have
\[
\dim_\C \Hom_U(\pi, \chi) = 1.
\]
\item[(b)]
If $\pi|_{G^\circ}$ is reducible, then any irreducible component of $\pi|_{G^\circ}$ is $\ww$-generic and we have
\[
\dim_\C \Hom_U(\pi, \chi) = 2.
\]
\end{itemize}
\end{lem}
\begin{proof}
Since 
\[
 \Hom_U(\pi, \chi) = \Hom_U(\pi|_{G^\circ}, \chi),
\]
the assertion (a) follows.

Assume that $\pi|_{G^\circ}$ is reducible.
Let $\pi^\circ$ be an irreducible component of $\pi|_{G^\circ}$.
Then we have $\pi|_{G^\circ} \cong \pi^\circ \oplus \epsilon \pi^\circ$ and 
\[
 \Hom_U(\pi, \chi) \cong \Hom_U(\pi^\circ, \chi) \oplus \Hom_U(\epsilon \pi^\circ, \chi).
\]
Since $\chi \circ \Ad(\epsilon) = \chi$, we see that $\pi^\circ$ is $\ww$-generic if and only if $\epsilon \pi^\circ$ is $\ww$-generic.
This implies the assertion (b). 
\end{proof}

\begin{rem}
In Lemma \ref{mult2} (b), the multiplicity two statement may be a bit disconcerting to readers. 
While for connected reductive groups $G^\circ$, 
the stabilizer of a Whittaker datum is the center of $G^\circ$, 
in the case of $G = \O_{2n}(F)$, 
the stabilizer of a Whittaker datum contains an extra group, generated by $\epsilon$. 
This group $\pair{\epsilon}$ acts on $\Hom_U(\pi ,\chi)$ with two possible eigenvalues 
and each eigenspace has dimension at most $1$. 
In other words, the multiplicity one statement is restored if we take into account the full symmetry of the situation.
\end{rem}

Let $P^\circ = M^\circ N$ be a standard parabolic subgroup of $G^\circ$. 
The heredity of Whittaker functionals asserts that 
if $\pi^\circ$ is an irreducible essentially unitary representation of $M^\circ$, 
then 
\[
\dim_\C \Hom_U(\Ind_{P^\circ}^{G^\circ}(\pi^\circ), \chi) 
=
\dim_\C \Hom_{U \cap M^\circ}(\pi^\circ, \chi).
\]
(See \cite{Rod}, \cite[Corollary 1.7]{CS}, \cite[Theorem 15.6.7]{Wl1}, \cite[Theorem 40]{Wl2}; 
note that this equality holds for all admissible representations $\pi^\circ$ when $F$ is non-archimedean.)
In particular, this dimension is at most $1$.
\par

For $G = \O_{2n}(F)$, we modify this property as follows.

\begin{lem}\label{Ovs.SO}
Let $G = \O_{2n}(F)$, and let $P=MN$ be a standard parabolic subgroup of $G$.
Set $P^\circ = P \cap G^\circ$ and $M^\circ = M \cap G^\circ$, so that $P=P^\circ \iff M=M^\circ$.
Let $\pi$ be an irreducible essentially unitary representation of $M$. 

\begin{enumerate}
\item[(a)]
If $M \not= M^\circ$, and $\pi|_{M^\circ}$ is irreducible, 
then 
\[
I_P(\pi)|_{G^\circ} \cong \Ind_{P^\circ}^{G^\circ}(\pi|_{M^\circ}).
\]
Moreover, 
\[
\dim_\C \Hom_U(I_P(\pi), \chi)
=
\dim_\C \Hom_{U \cap M}(\pi, \chi)
\leq 1.
\]

\item[(b)]
If $M = M^\circ$, or $\pi|_{M^\circ}$ is reducible, 
then for any irreducible component $\pi^\circ$ of $\pi|_{M^\circ}$, 
we have
\[
I_P(\pi) \cong I_{P^\circ}(\pi^\circ) = \Ind_{P^\circ}^G(\pi^\circ).
\]
Moreover, 
\[
\frac{1}{(G:G^\circ)} \dim_\C \Hom_U(I_P(\pi), \chi)
=
\frac{1}{(M:M^\circ)} \dim_\C \Hom_{U \cap M}(\pi, \chi)
\leq 1.
\]
(Note that the left-hand side is an integer.)
\end{enumerate}
\end{lem}
\begin{proof}
Suppose that $P \not= P^\circ$.
Then $G/G^\circ \cong M/M^\circ$, and the restriction map gives an isomorphism 
\[
\Ind_{P}^{G}(\pi)|_{G^\circ} \xrightarrow{\sim} \Ind_{P^\circ}^{G^\circ}(\pi|_{M^\circ}).
\]
This implies the assertion (a). 
\par

On the other hand, if $P=P^\circ$, or if $\pi|_{M^\circ}$ is reducible, 
for any irreducible component $\pi^\circ$ of $\pi|_{M^\circ}$, 
we have $\pi \cong \Ind_{M^\circ}^M(\pi^\circ)$.
Then $I_P(\pi) \cong I_{P^\circ}(\pi^\circ)$.
If we denote by $I_{P^\circ}^+(\pi^\circ)$ (\resp $I_{P^\circ}^-(\pi^\circ)$) 
the subspace of $I_{P^\circ}(\pi^\circ)$ consisting of functions $f$ on $G$ 
whose supports are contained in $G^\circ$ (\resp $G \setminus G^\circ$), 
then $I_{P^\circ}(\pi^\circ)|_{G^\circ} = I_{P^\circ}^+(\pi^\circ) \oplus I_{P^\circ}^-(\pi^\circ)$. 
Moreover, we have isomorphisms 
\begin{align*}
I_{P^\circ}^+(\pi^\circ) \xrightarrow{\sim} \Ind_{P^\circ}^{G^\circ}(\pi^\circ),\, 
f \mapsto f|_{G^\circ}, \\
I_{P^\circ}^-(\pi^\circ) \xrightarrow{\sim} \epsilon\Ind_{P^\circ}^{G^\circ}(\pi^\circ),\, 
f \mapsto (\epsilon^{-1} f)|_{G^\circ}
\end{align*}
as representations of $G^\circ$,
where $(\epsilon^{-1} f)(x) = f(x\epsilon^{-1})$.
In particular, we have 
\[
I_P(\pi)|_{G^\circ} \cong \Ind_{P^\circ}^{G^\circ}(\pi^\circ) \oplus \epsilon \Ind_{P^\circ}^{G^\circ}(\pi^\circ). 
\]
Since $\chi \circ \Ad(\epsilon) = \chi$, the following are equivalent: 
\begin{itemize}
\item
$\Hom_{U \cap M}(\pi, \chi) \not= 0$; 
\item
$\Hom_{U \cap M^\circ}(\pi^\circ, \chi) \not= 0$; 
\item
$\dim_\C \Hom_U(\Ind_{P^\circ}^{G^\circ}(\pi^\circ), \chi) \not= 0$; 
\item
$\dim_\C \Hom_U(\epsilon\Ind_{P^\circ}^{G^\circ}(\pi^\circ), \chi) \not= 0$.
\end{itemize}
Hence, in this case, we have
\[
\dim_\C \Hom_{U \cap M}(\pi, \chi) = (M:M^\circ), \quad
\dim_\C \Hom_U(I_P(\pi), \chi) = 2.
\]
This completes the proof.
\end{proof}

Fix two standard parabolic subgroups $P=MN_P$ and $P'=M'N_{P'}$ of $G$
such that $W(M^\circ, M'^\circ) \not= \emptyset$, 
an element $w \in W(M^\circ, M'^\circ)$, 
and an irreducible unitary $\ww_M$-generic representation $\pi$ of $M$.
Choose a non-trivial $\ww_M$-Whittaker functional $\omega$ on $\pi$.
For $\lambda \in \aa_{M,\C}^*$, define a $\ww$-Whittaker functional $\Omega(\pi_\lambda) = \Omega_\omega(\pi_\lambda)$ on $I_P(\pi_\lambda)$ as in Section \ref{sec.A27}.
Then $\Omega(\pi_\lambda)$ is holomorphic and nonzero for all $\lambda$
(see \cite{CS}, \cite[Theorem 15.6.7]{Wl1}, \cite[Theorem 40]{Wl2}).
Similarly, define a $\ww$-Whittaker functional $\Omega(w \pi_\lambda) = \Omega_\omega(w \pi_\lambda)$ on $I_{P'}(w\pi_\lambda)$, where we regard $\omega$ as a $\ww_{M'}$-Whittaker functional on $w \pi$.
When $G = G^\circ$, following Shahidi \cite[p.~333, Theorem 3.1]{Sh2}, we define a meromorphic function $C_P(w, \pi_\lambda)$ of $\lambda$, called the \emph{local coefficient}, 
by the equation
\[
C_P(w, \pi_\lambda) \cdot \Omega_\omega(w \pi_\lambda) \circ J_P(w,\pi_\lambda)
=
\Omega_\omega(\pi_\lambda).
\]
Note that such a function exists and does not depend on the choice of $\omega$ by the uniqueness of Whittaker functionals.
\par

When $G = \O_{2n}(F)$, this uniqueness may fail, but we can define an analogous function as follows.

\begin{lem}\label{CP-def}
Suppose that $G = \O_{2n}(F)$.
Then there exists a meromorphic function $C_P(w, \pi_\lambda)$ of $\lambda$ such that
\begin{align*}
&C_P(w, \pi_\lambda) \cdot \Omega_\omega(w \pi_\lambda) \circ J_P(w,\pi_\lambda)
\\&= \left\{
\begin{aligned}
&\Omega_\omega(\pi_\lambda) \iif \det(w) = 1, \\
&\Omega_\omega(\epsilon\pi_\lambda) \circ L(\epsilon) \iif \det(w) = -1,
\end{aligned}
\right. 
\end{align*}
where $L(\epsilon) \colon I_P(\pi_\lambda) \rightarrow I_{\epsilon P \epsilon^{-1}}(\epsilon \pi_\lambda)$ is given by $(L(\epsilon)f)(g) = f(\epsilon^{-1}g)$.
Moreover, $C_P(w, \pi_\lambda)$ does not depend on the choice of $\omega$.
\end{lem}

To prove this, we need the following.

\begin{lem}\label{CP-eps}
Suppose that $G^\circ = \SO_{2n}(F)$, $P \ne P^\circ$ (and hence $\epsilon P^\circ \epsilon^{-1} = P^\circ$), and $\det(w) = 1$.
Let $\pi^\circ$ be an irreducible unitary $\ww_M$-generic representation of $M^\circ$.
Then we have
\[
 C_{P^\circ}(w, \epsilon \pi^\circ_\lambda) = C_{P^\circ}(w, \pi^\circ_\lambda).
\]
\end{lem}

\begin{proof}
Since $P \ne P^\circ$, we have
\[
 \epsilon N_P \epsilon^{-1} = N_P, \quad
 \epsilon N_{P'} \epsilon^{-1} = N_{P'}, \quad
 \epsilon w \epsilon^{-1} = w, \quad
 \epsilon w_0 \epsilon^{-1} = w_0
\]
with $w_0 = w_\ell w_\ell^M$, where $w_\ell$ and $w_\ell^M$ are the longest elements in $W^{G^\circ}$ and $W^{M^\circ}$, respectively.
Fix a non-trivial $\ww_M$-Whittaker functional $\omega^\circ$ on $\pi^\circ$.
Since $\chi \circ \Ad(\epsilon) = \chi$, we may regard $\omega^\circ$ as a $\ww_M$-Whittaker functional on $\epsilon \pi^\circ$.
Then $\omega^\circ$ induces $\ww$-Whittaker functionals $\Omega(\pi^\circ_\lambda)$ and $\Omega(\epsilon \pi^\circ_\lambda)$ on $\Ind^{G^\circ}_{P^\circ}(\pi^\circ_\lambda)$ and $\Ind^{G^\circ}_{P^\circ}(\epsilon \pi^\circ_\lambda)$, respectively.
By definition, we have
\[
\Omega(\epsilon \pi^\circ_\lambda) = \Omega(\pi^\circ_\lambda) \circ \Ad(\epsilon)^*,
\]
where 
$\Ad(\epsilon)^* \colon 
\Ind^{G^\circ}_{P^\circ}(\epsilon \pi^\circ_\lambda) \rightarrow \Ind^{G^\circ}_{P^\circ}(\pi^\circ_\lambda)$ 
is the linear isomorphism given by $\Ad(\epsilon)^* f (g) = f(\epsilon g \epsilon^{-1})$.
Similarly, we have
\[
J_{P^\circ}(w, \pi^\circ_\lambda) \circ \Ad(\epsilon)^*  
= \Ad(\epsilon)^* \circ J_{P^\circ}(w, \epsilon \pi^\circ_\lambda).
\]
Hence we have
\begin{align*}
\Omega(\epsilon \pi^\circ_\lambda) & = \Omega(\pi^\circ_\lambda) \circ \Ad(\epsilon)^* \\
& = C_{P^\circ}(w, \pi_\lambda^\circ) \cdot \Omega(w \pi^\circ_\lambda) \circ J_{P^\circ}(w, \pi^\circ_\lambda) \circ \Ad(\epsilon)^* \\
& = C_{P^\circ}(w, \pi_\lambda^\circ) \cdot \Omega(w \pi^\circ_\lambda) \circ \Ad(\epsilon)^* \circ J_{P^\circ}(w, \epsilon \pi^\circ_\lambda) \\
& = C_{P^\circ}(w, \pi_\lambda^\circ) \cdot \Omega(\epsilon w \pi^\circ_\lambda) \circ J_{P^\circ}(w, \epsilon \pi^\circ_\lambda) \\
& = C_{P^\circ}(w, \pi_\lambda^\circ) \cdot \Omega(w \epsilon \pi^\circ_\lambda) \circ J_{P^\circ}(w, \epsilon \pi^\circ_\lambda).
\end{align*}
This implies the lemma.
\end{proof}

Now we prove Lemma \ref{CP-def}.

\begin{proof}[Proof of Lemma \ref{CP-def}]
First, we assume that $M \not= M^\circ$, and $\pi|_{M^\circ}$ is irreducible.
Then the existence of $C_P(w,\pi_\lambda)$ follows from Lemma \ref{Ovs.SO} (a).
Moreover, since $\omega$ is unique up to a scalar, $C_P(w,\pi_\lambda)$ does not depend on the choice of $\omega$.

Next, we assume that $M = M^\circ$, or $\pi|_{M^\circ}$ is reducible.
Fix an irreducible component $\pi^\circ$ of $\pi|_{M^\circ}$.
Note that $\pi^\circ$ is $\ww_M$-generic.
If $M=M^\circ$, then $\omega$ is unique up to a scalar.
If $M \not= M^\circ$ (so that $\epsilon \in M \setminus M^\circ$ and $\pi|_{M^\circ} = \pi^\circ \oplus \epsilon \pi^\circ$), 
then by Lemma \ref{mult2} (b), 
we may take a basis $\omega, \omega'$ of $\Hom_{U \cap M}(\pi, \chi)$ 
such that $\omega|_{\epsilon \pi^\circ} = 0$, $\omega'|_{\pi^\circ} = 0$.
We identify $\pi$ with $\Ind_{M^\circ}^M(\pi^\circ)$, 
so that $\pi^\circ$ (\resp $\epsilon \pi^\circ$) is the subspace of $\Ind_{M^\circ}^M(\pi^\circ)$ 
consisting of functions $f$ on $M$ whose supports are contained in $M^\circ$ (\resp $M \setminus M^\circ$).
Then $\omega$ can be realized by $\omega(f) = \omega^\circ(f(\1))$ for $f \in \Ind_{M^\circ}^M(\pi^\circ)$, 
where $\omega^\circ$ is a non-trivial $\ww_M$-Whittaker functional on $\pi^\circ$.
In both cases, we have $I_P(\pi) = I_{P^\circ}(\pi^\circ)$ and 
\[
 I_P(\pi)|_{G^\circ} = I_{P^\circ}^+(\pi^\circ) \oplus I_{P^\circ}^-(\pi^\circ)
\]
as in the proof of Lemma \ref{Ovs.SO}.
Then $\Omega_\omega(\pi_\lambda)$ (\resp $\Omega_\omega(\epsilon\pi_\lambda) \circ L(\epsilon)$) is a nonzero element in $\Hom_U(I_P(\pi_\lambda), \chi)$ which is identically zero on $I_{P^\circ}^-(\pi^\circ)$ (\resp $I_{P^\circ}^+(\pi^\circ)$).
In particular, $\Omega_\omega(\pi_\lambda)$ and $\Omega_\omega(\epsilon\pi_\lambda) \circ L(\epsilon)$ are linearly independent, and hence form a basis of $\Hom_U(I_P(\pi_\lambda), \chi)$ by Lemma \ref{Ovs.SO} (b).
On the other hand, $\Omega_\omega(w \pi_\lambda) \circ J_P(w,\pi_\lambda)$ is also an element in $\Hom_U(I_P(\pi_\lambda), \chi)$ (provided that $J_P(w,\pi_\lambda)$ is holomorphic at $\lambda$) which is identically zero on $I_{P^\circ}^-(\pi^\circ)$ (\resp $I_{P^\circ}^+(\pi^\circ)$) if $\det(w) = 1$ (\resp $\det(w) = -1$).
This proves the existence of the desired function $C_{P,\omega}(w,\pi_\lambda)$ with respect to $\omega$.
If $M=M^\circ$, 
then $C_{P,\omega}(w, \pi_\lambda)$ does not depend on the choice of $\omega$ 
by the uniqueness of Whittaker functionals.

\par
Finally, we assume that $M \ne M^\circ$, and $\pi|_{M^\circ}$ is reducible.
In particular, we have $\epsilon P \epsilon^{-1} = P$.
Recall the isomorphisms
\begin{align*}
I_{P^\circ}^+(\pi^\circ) \xrightarrow{\sim} \Ind_{P^\circ}^{G^\circ}(\pi^\circ),\, 
f \mapsto f|_{G^\circ}, \\
I_{P^\circ}^-(\pi^\circ) \xrightarrow{\sim} \Ind_{P^\circ}^{G^\circ}(\epsilon \pi^\circ),\, 
f \mapsto (L(\epsilon) f)|_{G^\circ}.
\end{align*}
If $\det(w) = 1$, then the restriction to $I^+_{P^\circ}(\pi^\circ)$ of the equality in the statement of the lemma yields
\[
 C_{P,\omega}(w,\pi_\lambda) = C_{P^\circ}(w, \pi_\lambda^\circ).
\]
Similarly, if $\det(w) = -1$, then the restriction to $I^-_{P^\circ}(\pi^\circ)$ yields
\[
 C_{P,\omega}(w,\pi_\lambda) = C_{P^\circ}(w \epsilon^{-1}, \epsilon \pi_\lambda^\circ),
\]
noting that $J_P(w, \pi_\lambda) = J_P(w\epsilon^{-1}, \epsilon \pi_\lambda) \circ L(\epsilon)$.
\par

Now we switch the roles of $\pi^\circ$ and $\epsilon \pi^\circ$, 
i.e., we identify $\pi$ with $\Ind_{M^\circ}^M (\epsilon \pi^\circ)$,
so that $\epsilon \pi^\circ$ (\resp $\pi^\circ$) is the subspace of $\Ind_{M^\circ}^M(\epsilon \pi^\circ)$ 
consisting of functions $f$ on $M$ whose supports are contained in $M^\circ$ (\resp $M \setminus M^\circ$).
Then $\omega'$ can be realized by $\omega'(f) = \omega^\circ(f(\1))$ for $f \in \Ind_{M^\circ}^M(\epsilon\pi^\circ)$,
where $\omega^\circ$ is now a non-trivial $\ww_M$-Whittaker functional on $\epsilon\pi^\circ$.
The same argument proves the existence of the desired function $C_{P,\omega'}(w, \pi_\lambda)$ with respect to $\omega'$ and shows that 
\[
C_{P,\omega'}(w,\pi_\lambda) = 
\left\{
\begin{aligned}
&C_{P^\circ}(w, \epsilon \pi_\lambda^\circ) \iif \det(w) = 1, \\
&C_{P^\circ}(w \epsilon^{-1}, \pi_\lambda^\circ) \iif \det(w) = -1.
\end{aligned}
\right. 
\]
\par

By Lemma \ref{CP-eps}, we have
\[
 C_{P,\omega'}(w, \pi_\lambda) = C_{P,\omega}(w, \pi_\lambda).
\]
This shows that $C_{P,\omega}(w, \pi_\lambda)$ does not depend on the choice of $\omega$ and completes the proof.
\end{proof}

In addition, from the proofs of Lemmas \ref{Ovs.SO} and \ref{CP-def}, we can deduce the following relation between local coefficients for $\O_{2n}(F)$ and those for $\SO_{2n}(F)$.

\begin{lem}\label{CPo}
Suppose that $G = \O_{2n}(F)$ and $\det(w) = 1$.
Then we have
\[
 C_P(w,\pi_\lambda) = C_{P^\circ}(w,\pi^\circ_\lambda),
\]
where $\pi^\circ$ is an arbitrary irreducible component of $\pi|_{M^\circ}$.
\end{lem}

Now we assume that $\pi$ is tempered.
Let $\phi$ be the tempered $L$-parameter of $\pi$. 
To show Theorem \ref{main1}, 
it suffices to prove the following.

\begin{prop}\label{cr=1}
In the setting of Theorem \ref{main1}, 
the function $C_P(w, \pi_\lambda) \cdot r_P(w,\phi_\lambda)$
is holomorphic and equal to $1$ at $\lambda = 0$. 
\end{prop}

We will prove Proposition \ref{cr=1} in Sections \ref{sec.A27_GL} and \ref{sec.A27_c}.
By the next lemma, which follows from the definitions, 
it suffices to consider Proposition \ref{cr=1} in the following two situations. 
\begin{enumerate}
\item
$G = G^\circ$, or $G = \O_{2n}(F)$ and $\det(w) = 1$, and $w\pi \cong \pi$; 
\item
$G = \O_{2n}(F)$, $\det(w) = 1$ and $w\epsilon\pi \cong \pi$.
\end{enumerate}

\begin{lem}\label{det=1}
Suppose that $G = \O_{2n}(F)$ and $\det(w) = -1$. 
\begin{itemize}
\item
We have 
\[
J_P(w, \pi_\lambda) = J_{\epsilon P \epsilon^{-1}}(w\epsilon^{-1}, \epsilon \pi_\lambda) \circ L(\epsilon)
\]
so that $C_P(w, \pi_\lambda) = C_{\epsilon P \epsilon^{-1}}(w\epsilon^{-1}, \epsilon\pi_\lambda)$.
\item
We have 
$r_P(w, \phi_\lambda) = r_{\epsilon P \epsilon^{-1}}(w\epsilon^{-1}, \phi_\lambda)$.
\end{itemize}
\end{lem}

\subsection{The maximal case}\label{sec.maximal}
Let $P = MN_P$ and $P' = M'N_{P'}$ be standard maximal parabolic subgroups of $G$
such that $W(M^\circ, M'^\circ) \not= \emptyset$. 
From Lemma \ref{det=1}, 
we may assume that 
\begin{itemize}
\item
if $G = \O_{2k}(F)$, $M = \GL_k(F)$, and $k > 1$ is odd, 
then $P' = \epsilon P \epsilon^{-1}$ and $M' = \epsilon M \epsilon^{-1}$ so that $W(M^\circ,M'^\circ) \not= \emptyset$; 
\item
otherwise, $G = \GL_N(E)$, or $M=M'$ so that $W(M^\circ,M'^\circ) \not= \{\1\}$.
\end{itemize}
In either case, let $w$ be the unique non-trivial element in $W(M^\circ,M'^\circ)$.
\par

First suppose that $G = \GL_N(E)$ 
and $M = \GL_{N_1}(E) \times \GL_{N_2}(E)$ so that $M' = \GL_{N_2}(E) \times \GL_{N_1}(E)$.
We write $\pi = \pi_1 \boxtimes \pi_2$ 
with irreducible tempered representations $\pi_1$ and $\pi_2$ of $\GL_{N_1}(E)$ and $\GL_{N_2}(E)$, 
respectively.
If we denote the $L$-parameters of $\pi_1$ and $\pi_2$ by $\phi_{\pi_1}$ and $\phi_{\pi_2}$, respectively, 
then we denote by $L(s, \phi_{\pi_1} \otimes \phi_{\pi_2})$ and $\ep(s,\phi_{\pi_1} \otimes \phi_{\pi_2}, \psi_E)$ 
the Artin factors associated to the tensor product of the standard representations.

\begin{lem}\label{GL-maximal}
The function $C_P(w, \pi_\lambda) \cdot r_P(w, \phi_\lambda)$ 
is holomorphic and equal to 
\[
\frac{L(1, \phi_{\pi_1}^\vee \otimes \phi_{\pi_2})}{L(1, \phi_{\pi_1} \otimes \phi_{\pi_2}^\vee)}
\]
at $\lambda = 0$.
\end{lem}
\begin{proof}
If we write $\pi_\lambda = \pi_1 |\cdot|_E^{s_1} \boxtimes \pi_2 |\cdot|_E^{s_2}$ 
with $s_1, s_2 \in \C$ and put $s=s_1-s_2$, then we have, by definition,
\[
r_P(w,\phi_\lambda) =
\frac{L(s, \phi_{\pi_1} \otimes \phi_{\pi_2}^\vee)}
{\ep(s, \phi_{\pi_1} \otimes \phi_{\pi_2}^\vee, \psi_E) L(1+s, \phi_{\pi_1} \otimes \phi_{\pi_2}^\vee)}.
\]
On the other hand, by \cite[Theorem 3.5]{Sh7} (see also Section \ref{sec.shahidi} below), 
we have
\[
C_P(w, \pi_\lambda) = 
\frac{\ep^\Sh(s, \pi_1 \times \pi_2^\vee, \psi_E) L^\Sh(1-s, \pi_1^\vee \times \pi_2)}{L^\Sh(s, \pi_1 \times \pi_2^\vee)},
\]
where the superscript Sh indicates Shahidi's local factors.
We know that 
\[
L^\Sh(s, \pi_1 \times \pi_2^\vee) = L(s, \phi_{\pi_1} \otimes \phi_{\pi_2}^\vee), 
\quad
\ep^\Sh(s, \pi_1 \times \pi_2^\vee, \psi_E) = \ep(s, \phi_{\pi_1} \otimes \phi_{\pi_2}^\vee, \psi_E)
\]
since these local factors agree with those of Jacquet--Piatetski-Shapiro--Shalika
by \cite{Sh4} and the desiderata of the local Langlands correspondence.
This implies the lemma.
\end{proof}

Next suppose that $G$ is a classical group, and write $M = \GL_k(E) \times G_0$.
Note that $M' = M$ unless $G = \O_{2k}(F)$, $M = \GL_k(F)$, and $k > 1$ is odd, 
in which case $M' = \epsilon M \epsilon^{-1}$.
We write $\pi = \tau \boxtimes \pi_0$ 
with irreducible tempered representations $\tau$ and $\pi_0$ of $\GL_k(E)$ and $G_0$, respectively.
If we denote the $L$-parameters of $\tau$ and $\pi_0$ by $\phi_\tau$ and $\phi_{\pi_0}$, respectively, 
we denote by $L(s, \phi_\tau \otimes \phi_{\pi_0})$ and $\ep(s, \phi_\tau \otimes \phi_{\pi_0}, \psi_E)$ 
the Artin factors over $E$ associated to the tensor product of the standard representations.
We also denote by $L(s, \phi_\tau, R)$ and $\ep(s,\phi_\tau, R, \psi_F)$ 
the Artin factors over $F$ associated to the representation $R$ of ${}^L\GL_k(E)$ given by 
\[
R = \left\{
\begin{aligned}
&\Sym^2 \iif G = \SO_{2n+1}(F), \\
&\wedge^2 \iif G = \Sp_{2n}(F),\, \O_{2n}(F), \\
&\Asai^+ \iif G = \U_n,\, n \equiv 0 \bmod 2, \\
&\Asai^- \iif G = \U_n,\, n \equiv 1 \bmod 2. 
\end{aligned}
\right.
\]
(See \cite[Section 7]{GGP} for the definition of the Asai representations $\Asai^+$ and $\Asai^-$.)

\begin{lem}
\label{classical-maximal}
The function $C_P(w, \pi_\lambda) \cdot r_P(w, \phi_\lambda)$ is holomorphic and equal to 
\[
\frac{L(1, \phi_\tau^\vee \otimes \phi_{\pi_0})}{L(1, \phi_\tau \otimes \phi_{\pi_0}^\vee)} 
\frac{L(1, \phi_\tau^\vee, R)}{L(1, \phi_\tau, R)}
\]
at $\lambda = 0$.
\end{lem}
\begin{proof}
If we write $\pi_\lambda = \tau|\cdot|_E^s \boxtimes \pi_0$ with $s\in \C$, 
then we have by definition
\begin{align*}
r_P(w,\phi_\lambda) = 
\lambda(w) 
&\times
\frac{L(s, \phi_\pi, \St \otimes \St^\vee)}
{\ep(s, \phi_\pi, \St \otimes \St^\vee, \psi_F) L(1+s, \phi_\pi, \St \otimes \St^\vee)} 
\\&\times
\frac{L(2s, \phi_\tau, R)}{\ep(2s, \phi_\tau, R, \psi_F) L(1+2s, \phi_\tau, R)}.
\end{align*}
Here $L(s, \phi_\pi, \St \otimes \St^\vee)$ and $\ep(s, \phi_\pi, \St \otimes \St^\vee, \psi_F)$
are the Artin factors associated to the $L$-parameter $\phi_\pi$ of $\pi$ and 
the tensor product representation $\St \otimes \St^\vee$, 
where $\St$ is the standard representation of ${}^L\GL_N(E)$ or ${}^L G_0^\circ$.
Note that 
\begin{align*}
L(s, \phi_\pi, \St \otimes \St^\vee) &= L(s, \phi_\tau \otimes \phi_{\pi_0}^\vee), 
\\
\ep(s, \phi_\pi, \St \otimes \St^\vee, \psi_F) &= \lambda_0 \cdot \ep(s, \phi_\tau \otimes \phi_{\pi_0}^\vee, \psi_E),
\end{align*}
where $\lambda_0 = 1$ unless $[E:F] = 2$ (so that $G = \U_n$ and $M = \GL_{k}(E) \times \U_{n_0}$), 
in which case, 
\[
\lambda_0 = \lambda(E/F, \psi_F)^{k n_0}.
\]
See e.g., \cite[Section 5.6]{D}.
On the other hand, by \cite[Theorem 3.5]{Sh7} (see also Section \ref{sec.shahidi} below), 
we have
\begin{align*}
C_P(w, \pi_\lambda) = \lambda(w)^{-1} 
&\times\frac{\ep^{\Sh}(s, \pi, \St \otimes \St^\vee, \psi_F) L^{\Sh}(1-s, \pi^\vee, \St \otimes \St^\vee)}
{L^{\Sh}(s, \pi, \St \otimes \St^\vee)} 
\\&\times
\frac{\ep^{\Sh}(2s, \tau, R, \psi_F) L^{\Sh}(1-2s, \tau^\vee, R)}{L^{\Sh}(2s, \tau, R)},
\end{align*}
where the superscript $\Sh$ indicates Shahidi's local factors. 
Here, when $G = \O_{2n}(F)$, we have $\det(w) = 1$ by the assumption at the beginning of this subsection, 
and by Lemma \ref{CPo}, the above equality holds if we set
\begin{align*}
L^{\Sh}(s, \pi, \St \otimes \St^\vee) & = L^{\Sh}(s, \pi^\circ, \St \otimes \St^\vee), \\
\ep^{\Sh}(s, \pi, \St \otimes \St^\vee, \psi_F) & = \ep^{\Sh}(s, \pi^\circ, \St \otimes \St^\vee, \psi_F),
\end{align*}
where $\pi^\circ$ is an arbitrary irreducible component of $\pi|_{M^\circ}$.
Putting 
\begin{align*}
L^{\Sh}(s, \tau \times \pi_0^\vee) & = L^{\Sh}(s, \pi, \St \otimes \St^\vee), \\
\ep^{\Sh}(s, \tau \times \pi_0^\vee, \psi_E) 
&= \lambda_0^{-1} \cdot \ep^{\Sh}(s, \pi, \St \otimes \St^\vee, \psi_F), 
\end{align*}
by Proposition \ref{gamma=} below, we have
\[
L^{\Sh}(s, \tau \times \pi_0^\vee) = L(s, \phi_\tau \otimes \phi_{\pi_0}^\vee), 
\quad
\ep^{\Sh}(s, \tau \times \pi_0^\vee, \psi_E) = \ep(s, \phi_\tau \otimes \phi_{\pi_0}^\vee, \psi_E).
\]
Also, by \cite{He4}, \cite{CST}, \cite{Shan}, \cite{He5}, we have
\[
L^{\Sh}(s, \tau, R) = L(s, \phi_\tau, R), 
\quad
\ep^{\Sh}(s, \tau, R, \psi_F) = \ep(s, \phi_\tau, R, \psi_F).
\]
This implies the lemma.
\end{proof}

\subsection{Preparation for the general case}\label{sec.factorial}
Now we consider the general case.
Recall that $U$ is the unipotent radical of the Borel subgroup $B^\circ$.
Let $P = MN_P$ and $P' = M'N_{P'}$ be standard parabolic subgroups of $G$ 
such that $W(M^\circ,M'^\circ)$ has an element $w$ whose representatives lie in $G^\circ$. 
Then by \cite[Lemma 2.1.2]{Sh2}, 
we may take standard parabolic subgroups $P_i = M_i N_i$ of $G$ for $1 \leq i \leq n+1$ 
with the following properties:
\begin{itemize}
\item 
$P = P_1$ and $P' = P_{n+1}$;
\item 
for each $1 \leq i \leq n$, 
there exists a semi-standard Levi subgroup $G_i$ of $G$ containing $M_i$ 
such that $P_i \cap G_i$ is a maximal parabolic subgroup of $G_i$ 
and such that the element $w_i = w_\ell^{G_i^\circ} w_\ell^{M_i^\circ}$ 
belongs to $W(M_i^\circ,M_{i+1}^\circ)$, 
where $w_\ell^{\bullet}$ is the longest element in $W^\bullet$;
\item 
$w = w_n \cdots w_1$;
\item 
for each $1 \leq i \leq n$, we have
\[
\overline{\nn}_{w_i'} = \overline{\nn}_{w_i} \oplus \Ad(\tl{w}_i)^{-1} \overline{\nn}_{w_{i+1}'},
\]
where $\overline{\nn}_w$ is the Lie algebra of 
$\overline{N}_w = \overline{N}_P \cap \tl{w}^{-1} U \tl{w}$ for $w \in W^{G^\circ}$, 
and $w_i' = w_n \cdots w_i$ (with the convention that $w_{n+1}' = \1$).
\end{itemize}
This gives rise to a factorization of the intertwining operator 
\[
J_P(w, \pi_\lambda) = 
J_{P_n}(w_n, w_{n-1} \cdots w_1 \pi_\lambda) \circ \dots \circ 
J_{P_2}(w_2, w_1 \pi_\lambda) \circ J_{P_1}(w_1, \pi_\lambda)
\]
(see \cite[Theorem 2.1.1]{Sh2})
and hence of the local coefficient
\[
C_P(w, \pi_\lambda) = \prod_{i=1}^n C_{P_i}(w_i, w_{i-1} \cdots w_1 \pi_\lambda)
\]
(see \cite[Proposition 3.2.1]{Sh2}).
This also gives rise to a factorization of the $\lambda$-factor
\[
\lambda(w) = \prod_{i=1}^n \lambda(w_i).
\]
Moreover, since 
\[
\widehat{\overline{\nn}}_w = 
\widehat{\overline{\nn}}_{w_1} \oplus \Ad(\tl{w}_1)^{-1} \widehat{\overline{\nn}}_{w_2} \oplus 
\cdots \oplus \Ad(\tl{w}_{n-1} \cdots \tl{w}_1)^{-1} \widehat{\overline{\nn}}_{w_n}, 
\]
we have
\begin{align*}
L(s, \pi_\lambda, \rho_{w^{-1}P' | P}) 
&= \prod_{i=1}^n L(s, w_{i-1} \cdots w_1 \pi_\lambda, \rho^\vee_{w_i^{-1} P_{i+1} | P_i}), \\
\ep(s, \pi_\lambda, \rho_{w^{-1}P' | P}, \psi_F)
&= \prod_{i=1}^n \ep(s, w_{i-1} \cdots w_1 \pi_\lambda, \rho^\vee_{w_i^{-1} P_{i+1} | P_i}, \psi_F).
\end{align*}
Hence we have
\[
r_P(w, \phi_\lambda) = \prod_{i=1}^n r_{P_i}(w_i, w_{i-1} \cdots w_1 \phi_\lambda).
\]
\par

Using that the local coefficient $C_{P_i}(w_i, w_{i-1} \cdots w_1 \pi_\lambda)$ 
and the normalizing factor $r_{P_i}(w_i, w_{i-1} \cdots w_1 \phi_\lambda)$ 
agree with those for the intertwining operator 
\[
\Ind^{G_i}_{P_i}(w_{i-1} \cdots w_1 \pi_\lambda) \rightarrow \Ind^{G_i}_{P_{i+1}}(w_i \cdots w_1 \pi_\lambda),
\] 
and that $G_i$ is the product of general linear groups and a (possibly trivial) classical group, 
we can use these computations to attack Proposition \ref{cr=1} in general.
We will see it in the next two subsections.

\subsection{The case of general linear groups}\label{sec.A27_GL}
Suppose that $G = \GL_N(E)$ and $M = \GL_{n_1}(E) \times \cdots \times \GL_{n_m}(E)$. 
Put $I = \{1, \dots, m\}$.
Write $\pi = \pi_1 \boxtimes \cdots \boxtimes \pi_m$ 
with irreducible tempered representations $\pi_1, \dots, \pi_m$ of $\GL_{n_1}(E), \dots, \GL_{n_m}(E)$, respectively.
We denote by $\phi_{\pi_i}$ the $L$-parameter of $\pi_i$. 
We regard $w$ as an automorphism of $I$ 
such that $w \pi = \pi_{w^{-1}(1)} \boxtimes \cdots \boxtimes \pi_{w^{-1}(m)}$.
Then by Lemma \ref{GL-maximal} and Section \ref{sec.factorial}, 
the function $C_P(w, \pi_\lambda) \cdot r_P(w, \phi_\lambda)$ is holomorphic and equal to
\[
\prod_{(i,j) \in \inv(w)} A(\pi_i, \pi_j)
\]
at $\lambda = 0$, 
where we set 
\[
\inv(w) = \{ (i,j) \in I \times I \,| \, i<j, \, w(i)> w(j) \}
\]
and 
\[
A(\pi_i, \pi_j) = \frac{L(1, \phi_{\pi_i}^\vee \otimes \phi_{\pi_j})}{L(1, \phi_{\pi_i} \otimes \phi_{\pi_j}^\vee)}.
\]
We may write
\[
\prod_{(i,j) \in \inv(w)} A(\pi_i, \pi_j) 
= 
\prod_{(\sigma_1, \sigma_2)} A(\sigma_1, \sigma_2)^{n(\sigma_1,\sigma_2)}, 
\]
where $(\sigma_1, \sigma_2)$ runs over ordered pairs of irreducible tempered representations of general linear groups and 
\[
 n(\sigma_1, \sigma_2) = |(I(\sigma_1) \times I(\sigma_2)) \cap \inv(w)|
\]
with $I(\sigma_k) = \{ i \in I \,|\, \pi_i \cong \sigma_k \}$.
Since $A(\sigma_1, \sigma_2) = 1$ if $\sigma_1 \cong \sigma_2$ 
and $A(\sigma_1, \sigma_2) A(\sigma_2, \sigma_1) = 1$, 
we have
\[
\prod_{(\sigma_1, \sigma_2)} A(\sigma_1, \sigma_2)^{n(\sigma_1,\sigma_2)}
= 
\prod_{\{ \sigma_1, \sigma_2 \}} A(\sigma_1, \sigma_2)^{n(\sigma_1,\sigma_2) - n(\sigma_2, \sigma_1)},
\]
where $\{ \sigma_1, \sigma_2 \}$ runs over unordered pairs.
\par

\begin{proof}[Proof of Theorem \ref{main1} (1) for $G=\GL_N(E)$] 
Assume that $P=P'$ and $w \pi \cong \pi$.
It suffices to show that
\[
n(\sigma_1,\sigma_2) = n(\sigma_2, \sigma_1). 
\]
We consider the set $I_0(\sigma_1, \sigma_2) = I_1(\sigma_1, \sigma_2) \sqcup I_2(\sigma_1, \sigma_2)$
with
\begin{align*}
I_0(\sigma_1, \sigma_2) &= \{ (i,j) \in I(\sigma_1) \times I(\sigma_2) \,|\, i < j\}, \\
I_1(\sigma_1, \sigma_2) &= \{ (i,j) \in I(\sigma_1) \times I(\sigma_2) \,|\, i < j, \, w(i) < w(j) \}, \\
I_2(\sigma_1, \sigma_2) &= \{ (i,j) \in I(\sigma_1) \times I(\sigma_2) \,|\, i < j, \, w(i) > w(j) \}.
\end{align*}
Since $w \pi \cong \pi$, we have $\pi_{w^{-1}(i)} \cong \pi_i$ for $1 \leq i \leq m$. 
Hence the map $(i,j) \mapsto (w^{-1}(i),w^{-1}(j))$ gives a bijection 
\[
I_0(\sigma_1, \sigma_2) 
\xrightarrow{1:1} 
I_0'(\sigma_1, \sigma_2) = \{ (i,j) \in I(\sigma_1) \times I(\sigma_2) \,|\, w(i)<w(j) \}.
\]
Note that $I_0'(\sigma_1, \sigma_2) = I_1(\sigma_1, \sigma_2) \sqcup I_2'(\sigma_1, \sigma_2)$ with 
\[
I'_2(\sigma_1, \sigma_2) =  \{ (i,j) \in I(\sigma_1) \times I(\sigma_2) \,|\, i>j, \, w(i) < w(j) \}. 
\]
Since the map $(i,j) \mapsto (j,i)$ gives a bijection $I'_2(\sigma_1, \sigma_2) \xrightarrow{1:1} I_2(\sigma_2, \sigma_1)$, 
we have
\[
n(\sigma_1,\sigma_2) = |I_2(\sigma_1,\sigma_2)| = |I'_2(\sigma_1, \sigma_2)| = n(\sigma_2, \sigma_1).
\]
This completes the proof of Proposition \ref{cr=1} in this case, 
and hence Theorem \ref{main1} (1) for $G=\GL_N(E)$. 
\end{proof}
\par

Next, we consider Theorem \ref{main1} (2).
For a representation $\sigma$ of $\GL_N(E)$, 
define its conjugate ${}^c\sigma$ by ${}^c\sigma(x) = \sigma(\overline{x})$. 
Hence $\sigma \circ \theta \cong {}^c\sigma^\vee$ if $\sigma$ is irreducible. 
Since $L(s, {}^c\phi_{\sigma_1} \otimes {}^c\phi_{\sigma_2}^\vee) = L(s, \phi_{\sigma_1} \otimes \phi_{\sigma_2}^\vee)$, 
we have $A({}^c\sigma_2^\vee, {}^c\sigma_1^\vee) = A(\sigma_1, \sigma_2)$, so that 
\[
\prod_{(\sigma_1, \sigma_2)} A(\sigma_1, \sigma_2)^{n(\sigma_1,\sigma_2)}
=
\prod_{[\sigma_1, \sigma_2]} A(\sigma_1, \sigma_2)^{n'(\sigma_1,\sigma_2)/e(\sigma_1,\sigma_2)},
\]
where $[\sigma_1,\sigma_2]$ runs over orbits under the action 
$(\sigma_1, \sigma_2) \mapsto ({}^c\sigma_2^\vee, {}^c\sigma_1^\vee)$ and 
\begin{align*}
n'(\sigma_1,\sigma_2) & = n(\sigma_1,\sigma_2) + n({}^c\sigma_2^\vee, {}^c\sigma_1^\vee), \\
e(\sigma_1,\sigma_2) & = \frac{2}{|[\sigma_1, \sigma_2]|}.
\end{align*}
Moreover, since 
\begin{itemize}
\item
$A(\sigma_1, \sigma_2) = 1$ if $[\sigma_1, \sigma_2] = [\sigma_2, \sigma_1]$; 
\item
$A(\sigma_1, \sigma_2) A(\sigma_2, \sigma_1) = 1$; 
\item
$e(\sigma_1,\sigma_2) = e(\sigma_2,\sigma_1)$,
\end{itemize}
we have
\[
\prod_{[\sigma_1, \sigma_2]} A(\sigma_1, \sigma_2)^{n'(\sigma_1,\sigma_2)/e(\sigma_1,\sigma_2)}
= 
\prod_{[[\sigma_1, \sigma_2]]} 
A(\sigma_1, \sigma_2)^{(n'(\sigma_1,\sigma_2) - n'(\sigma_2,\sigma_1))/e(\sigma_1, \sigma_2)},
\]
where $[[\sigma_1, \sigma_2]]$ runs over orbits under the action $[\sigma_1, \sigma_2] \mapsto [\sigma_2, \sigma_1]$. 
\par

\begin{proof}[Proof of Theorem \ref{main1} (2)]
Assume that $w \pi \cong \pi \circ \theta$.
This is equivalent to saying that $\pi_{w^{-1}(i)} \cong {}^c\pi_{m+1-i}^\vee$ for $1 \leq i \leq m$.
It suffices to show that
\[
n'(\sigma_1,\sigma_2) = n'(\sigma_2, \sigma_1). 
\]
Put
\begin{align*}
I_1(\sigma_1, \sigma_2) &= \{ (i,j) \in I(\sigma_1) \times I(\sigma_2) \,|\, i < j, \, w(i) < w(j)\}, \\
I_2(\sigma_1, \sigma_2) &= \{ (i,j) \in I(\sigma_1) \times I(\sigma_2) \,|\, i < j, \, w(i) > w(j)\}, \\
I_3(\sigma_1, \sigma_2) &= \{ (i,j) \in I(\sigma_1) \times I(\sigma_2) \,|\, i > j, \, w(i) < w(j)\}.
\end{align*}
Then the map $(i,j) \mapsto (i',j') = (w^{-1}(m+1-j), w^{-1}(m+1-i))$ gives a bijection 
\begin{align*}
\{ (i,j) \in I(\sigma_1) \times I(\sigma_2) \,|\, i < j\} 
\xrightarrow{1:1}&
\{(i',j') \in I({}^c\sigma_2^\vee) \times I({}^c\sigma_1^\vee) \,|\, w(i') < w(j')\}.
\end{align*}
Note that the left-hand side is $I_1(\sigma_1, \sigma_2) \sqcup I_2(\sigma_1, \sigma_2)$, 
whereas the right-hand side is 
$I_1({}^c\sigma_2^\vee, {}^c\sigma_1^\vee) \sqcup I_3({}^c\sigma_2^\vee, {}^c\sigma_1^\vee)$. 
Hence
\[
|I_1(\sigma_1, \sigma_2)|+|I_2(\sigma_1, \sigma_2)| 
= |I_1({}^c\sigma_2^\vee, {}^c\sigma_1^\vee)|+|I_3({}^c\sigma_2^\vee, {}^c\sigma_1^\vee)|.
\]
This implies that
\begin{align*}
&|I_1(\sigma_1, \sigma_2)|+|I_2(\sigma_1, \sigma_2)|
+|I_1({}^c\sigma_2^\vee, {}^c\sigma_1^\vee)|+|I_2({}^c\sigma_2^\vee, {}^c\sigma_1^\vee)|
\\&=
|I_1({}^c\sigma_2^\vee, {}^c\sigma_1^\vee)|+|I_3({}^c\sigma_2^\vee, {}^c\sigma_1^\vee)|
+|I_1(\sigma_1, \sigma_2)|+|I_3(\sigma_1, \sigma_2)|
\end{align*}
so that 
\begin{align*}
n'(\sigma_1, \sigma_2) 
&= 
|I_2(\sigma_1, \sigma_2)|+|I_2({}^c\sigma_2^\vee, {}^c\sigma_1^\vee)|
\\&= 
|I_3(\sigma_1, \sigma_2)|+|I_3({}^c\sigma_2^\vee, {}^c\sigma_1^\vee)|
= n'(\sigma_2, \sigma_1), 
\end{align*}
where the last equation follows from the map $(i,j) \mapsto (j,i)$. 
This completes the proof of Proposition \ref{cr=1} in this case, 
and hence Theorem \ref{main1} (2). 
\end{proof}

\subsection{The case of classical groups}\label{sec.A27_c}
We now consider Theorem \ref{main1} (1) for classical groups. 
Suppose that $G$ is a classical group
and $M = \GL_{n_1}(E) \times \cdots \times \GL_{n_m}(E) \times G_0$.
As explained after Proposition \ref{cr=1}, by Lemma \ref{det=1},
it suffices to consider the following two situations. 
\begin{enumerate}
\item
$G = G^\circ$, or $G = \O_{2n}(F)$ and $\det(w) = 1$, and $w\pi \cong \pi$; 
\item
$G = \O_{2n}(F)$, $\det(w) = 1$ and $w\epsilon\pi \cong \pi$.
\end{enumerate}
\par

We consider the general situation for a moment. 
Put $I^+ = \{ 1, \dots, m \}$, $I^- = \{ -1, \dots, -m \}$, and $I = I^+ \sqcup I^-$.
Write $\pi = \tau_1 \boxtimes \cdots \boxtimes \tau_m \boxtimes \pi_0$ 
with irreducible tempered representations $\tau_1, \dots, \tau_m, \pi_0$ 
of $\GL_{n_1}(E), \dots, \GL_{n_m}(E), G_0$, respectively, 
and put $\tau_{-i} = {}^c\tau_i^\vee$ for $i \in I^+$.
We denote by $\phi_{\tau_i}$ (\resp $\phi_{\pi_0}$)
the $L$-parameter of $\tau_i$ (\resp $\pi_0$).
We regard $w$ as an automorphism of $I$ such that $w(-i) = -w(i)$ for all $i \in I$ 
and such that $w \pi = \tau_{w^{-1}(1)} \boxtimes \cdots \boxtimes \tau_{w^{-1}(m)} \boxtimes \pi_0$.
Then by Lemmas \ref{GL-maximal}, \ref{classical-maximal} and Section \ref{sec.factorial}, 
the function $C_P(w, \pi_\lambda) \cdot r_P(w,\phi_\lambda)$ is holomorphic and equal to the product 
\[
\left(\prod_{(i,j) \in \inv_1(w)} A(\tau_i, \tau_j)\right) 
\left(\prod_{i \in \inv_2(w)} B(\tau_i, \pi_0)\right) 
\]
at $\lambda = 0$, 
where
\begin{align*}
\inv_1(w) & = \{ (i, j) \in I^+ \times I \,|\, i < |j|, \, w(i) > 0, \, w(j) > 0, \, w(i) > w(j) \} \\
&\sqcup \{ (i, j) \in I^+ \times I \,|\, i < |j|, \, w(i) < 0, \, w(j) > 0 \} \\
&\sqcup \{ (i, j) \in I^+ \times I \,|\, i < |j|, \, w(i) < 0, \, w(j) < 0, \, |w(i)| < |w(j)| \}, \\
\inv_2(w) & = \{ i \in I^+ \,|\, w(i) < 0 \}
\end{align*}
and
\[
A(\tau_i, \tau_j) = \frac{L(1, \phi_{\tau_i}^\vee \otimes \phi_{\tau_j})}{L(1, \phi_{\tau_i} \otimes \phi_{\tau_j}^\vee)}, \quad
B(\tau_i, \pi_0) = \frac{L(1, \phi_{\tau_i}^\vee \otimes \phi_{\pi_0})}{L(1, \phi_{\tau_i} \otimes \phi_{\pi_0}^\vee)} 
\frac{L(1, \phi_{\tau_i}^\vee, R)}{L(1, \phi_{\tau_i}, R)}.
\]
First, we may write 
\[
\prod_{(i,j) \in \inv_1(w)} A(\tau_i, \tau_j)
= 
\prod_{(\sigma_1, \sigma_2)} A(\sigma_1, \sigma_2)^{n_1(\sigma_1,\sigma_2)},
\]
where $(\sigma_1, \sigma_2)$ runs over pairs of irreducible tempered representations of general linear groups and 
\[
n_1(\sigma_1, \sigma_2) = |(I(\sigma_1) \times I(\sigma_2)) \cap \inv_1(w)|
\]
with $I(\sigma) = \{ i \in I \,|\, \tau_i \cong \sigma \}$.
Since $L(s, {}^c\phi_{\sigma_1} \otimes {}^c\phi_{\sigma_2}^\vee) = L(s, \phi_{\sigma_1} \otimes \phi_{\sigma_2}^\vee)$, 
we have $A({}^c\sigma_2^\vee, {}^c\sigma_1^\vee) = A(\sigma_1, \sigma_2)$, 
so that 
\[
\prod_{(\sigma_1, \sigma_2)} A(\sigma_1, \sigma_2)^{n_1(\sigma_1,\sigma_2)} 
=
\prod_{[\sigma_1, \sigma_2]} A(\sigma_1, \sigma_2)^{n_1'(\sigma_1,\sigma_2)/e(\sigma_1,\sigma_2)},
\]
where $[\sigma_1,\sigma_2]$ runs over orbits under the action 
$(\sigma_1, \sigma_2) \mapsto ({}^c\sigma_2^\vee, {}^c\sigma_1^\vee)$ 
and 
\begin{align*}
 n_1'(\sigma_1,\sigma_2) & = n_1(\sigma_1,\sigma_2) + n_1({}^c\sigma_2^\vee, {}^c\sigma_1^\vee), \\
 e(\sigma_1,\sigma_2) & = \frac{2}{|[\sigma_1, \sigma_2]|}.
\end{align*}
Moreover, since 
\begin{itemize}
\item
$A(\sigma_1, \sigma_2) = 1$ if $[\sigma_1, \sigma_2] = [\sigma_2, \sigma_1]$; 
\item
$A(\sigma_1, \sigma_2) A(\sigma_2, \sigma_1) = 1$;
\item
$e(\sigma_1,\sigma_2) = e(\sigma_2,\sigma_1)$, 
\end{itemize}
we have
\[
\prod_{[\sigma_1, \sigma_2]} A(\sigma_1, \sigma_2)^{n_1'(\sigma_1,\sigma_2)/e(\sigma_1,\sigma_2)}
=
\prod_{[[\sigma_1, \sigma_2]]} 
A(\sigma_1, \sigma_2)^{(n_1'(\sigma_1,\sigma_2) - n_1'(\sigma_2,\sigma_1))/e(\sigma_1, \sigma_2)},
\]
where $[[\sigma_1, \sigma_2]]$ runs over orbits under the action $[\sigma_1, \sigma_2] \mapsto [\sigma_2, \sigma_1]$.
Similarly, we may write 
\[
\prod_{i \in \inv_2(w)} B(\tau_i, \pi_0) 
= 
\prod_\sigma B(\sigma, \pi_0)^{n_2(\sigma)},
\]
where $\sigma$ runs over irreducible tempered representations of general linear groups and 
\[
 n_2(\sigma) = |I(\sigma) \cap \inv_2(w)|.
\]
Since $L(s, {}^c\phi_{\sigma} \otimes \phi_{\pi_0}) = L(s, \phi_{\sigma} \otimes \phi_{\pi_0}^\vee)$ 
and $L(s, {}^c\phi_{\sigma}, R) = L(s, \phi_{\sigma}, R)$, 
we have $B(\sigma, \pi_0) = 1$ if ${}^c\sigma^\vee \cong \sigma$, and $B(\sigma, \pi_0) B({}^c\sigma^\vee, \pi_0) = 1$, 
so that
\[
\prod_\sigma B(\sigma, \pi_0)^{n_2(\sigma)}
=
\prod_{[\sigma]} B(\sigma, \pi_0)^{n_2(\sigma) - n_2({}^c\sigma^\vee)},
\]
where $[\sigma]$ runs over orbits under the action $\sigma \mapsto {}^c\sigma^\vee$.

\begin{proof}[Proof of Theorem \ref{main1} (1) for classical groups $G$]
Assume that $w \pi \cong \pi$, or $w\epsilon \pi \cong \pi$ (in which case, $G=\O_{2n}(F)$).
It suffices to show that 
\[
n_1'(\sigma_1,\sigma_2) = n_1'(\sigma_2, \sigma_1), 
\quad 
n_2(\sigma) = n_2({}^c\sigma^\vee). 
\]
First we consider the set 
\[
I^\delta(\sigma_1,\sigma_2) = \{(i,j) \in I(\sigma_1) \times I(\sigma_2) \,|\, i>0, \,\delta j>0, \, i<|j| \} 
= \bigsqcup_{k=1}^8 I_k^\delta(\sigma_1,\sigma_2)
\]
for $\delta = \pm$, 
where
\begin{align*}
I_1^\delta(\sigma_1,\sigma_2) &=  \{ (i,j) \in I^+(\sigma_1) \times I^\delta(\sigma_2) \,|\, i<|j|, \, w(i)>0, \, w(j)>0, \, |w(i)|<|w(j)| \}, \\ 
I_2^\delta(\sigma_1,\sigma_2) &=  \{ (i,j) \in I^+(\sigma_1) \times I^\delta(\sigma_2) \,|\, i<|j|, \, w(i)>0, \, w(j)>0, \, |w(i)|>|w(j)| \}, \\ 
I_3^\delta(\sigma_1,\sigma_2) &=  \{ (i,j) \in I^+(\sigma_1) \times I^\delta(\sigma_2) \,|\, i<|j|, \, w(i)>0, \, w(j)<0, \, |w(i)|<|w(j)| \}, \\ 
I_4^\delta(\sigma_1,\sigma_2) &=  \{ (i,j) \in I^+(\sigma_1) \times I^\delta(\sigma_2) \,|\, i<|j|, \, w(i)>0, \, w(j)<0, \, |w(i)|>|w(j)| \}, \\ 
I_5^\delta(\sigma_1,\sigma_2) &=  \{ (i,j) \in I^+(\sigma_1) \times I^\delta(\sigma_2) \,|\, i<|j|, \, w(i)<0, \, w(j)>0, \, |w(i)|<|w(j)| \}, \\ 
I_6^\delta(\sigma_1,\sigma_2) &=  \{ (i,j) \in I^+(\sigma_1) \times I^\delta(\sigma_2) \,|\, i<|j|, \, w(i)<0, \, w(j)>0, \, |w(i)|>|w(j)| \}, \\ 
I_7^\delta(\sigma_1,\sigma_2) &=  \{ (i,j) \in I^+(\sigma_1) \times I^\delta(\sigma_2) \,|\, i<|j|, \, w(i)<0, \, w(j)<0, \, |w(i)|<|w(j)| \}, \\ 
I_8^\delta(\sigma_1,\sigma_2) &=  \{ (i,j) \in I^+(\sigma_1) \times I^\delta(\sigma_2) \,|\, i<|j|, \, w(i)<0, \, w(j)<0, \, |w(i)|>|w(j)| \}
\end{align*}
with $I^\delta(\sigma) = I^\delta \cap I(\sigma)$.
Since $w \pi \cong \pi$ or $w\epsilon \pi \cong \pi$, 
the map $(i,j) \mapsto (i',j') = (w^{-1}(i),w^{-1}(j))$ gives a bijection 
$I^\delta(\sigma_1,\sigma_2) \xrightarrow{1:1} J^\delta(\sigma_1,\sigma_2)$, 
where
\begin{align*}
J^\delta(\sigma_1,\sigma_2) 
&= \{ (i',j') \in I(\sigma_1) \times I(\sigma_2) \,|\, w(i')>0, \, \delta w(j') > 0, \, w(i')<|w(j')| \} 
\\&= \bigsqcup_{k=1}^8 J_k^\delta(\sigma_1,\sigma_2)
\end{align*} 
with
\begin{align*}
J_1^\delta(\sigma_1,\sigma_2) &= \{ (i,j) \in I^+(\sigma_1) \times I^+(\sigma_2) \,|\, 
|i|<|j|, \, w(i)>0, \, \delta w(j) > 0, \, w(i)<|w(j)| \}, \\ 
J_2^\delta(\sigma_1,\sigma_2) &= \{ (i,j) \in I^+(\sigma_1) \times I^+(\sigma_2) \,|\, 
|i|>|j|, \, w(i)>0, \, \delta w(j) > 0, \, w(i)<|w(j)| \}, \\ 
J_3^\delta(\sigma_1,\sigma_2) &= \{ (i,j) \in I^+(\sigma_1) \times I^-(\sigma_2) \,|\, 
|i|<|j|, \, w(i)>0, \, \delta w(j) > 0, \, w(i)<|w(j)| \}, \\ 
J_4^\delta(\sigma_1,\sigma_2) &= \{ (i,j) \in I^+(\sigma_1) \times I^-(\sigma_2) \,|\, 
|i|>|j|, \, w(i)>0, \, \delta w(j) > 0, \, w(i)<|w(j)| \}, \\ 
J_5^\delta(\sigma_1,\sigma_2) &= \{ (i,j) \in I^-(\sigma_1) \times I^+(\sigma_2) \,|\, 
|i|<|j|, \, w(i)>0, \, \delta w(j) > 0, \, w(i)<|w(j)| \}, \\ 
J_6^\delta(\sigma_1,\sigma_2) &= \{ (i,j) \in I^-(\sigma_1) \times I^+(\sigma_2) \,|\, 
|i|>|j|, \, w(i)>0, \, \delta w(j) > 0, \, w(i)<|w(j)| \}, \\ 
J_7^\delta(\sigma_1,\sigma_2) &= \{ (i,j) \in I^-(\sigma_1) \times I^-(\sigma_2) \,|\, 
|i|<|j|, \, w(i)>0, \, \delta w(j) > 0, \, w(i)<|w(j)| \}, \\ 
J_8^\delta(\sigma_1,\sigma_2) &= \{ (i,j) \in I^-(\sigma_1) \times I^-(\sigma_2) \,|\, 
|i|>|j|, \, w(i)>0, \, \delta w(j) > 0, \, w(i)<|w(j)| \}.
\end{align*}
Hence, noting that 
\begin{align*}
 J_1^+(\sigma_1,\sigma_2) & = I_1^+(\sigma_1,\sigma_2), &
 J_1^-(\sigma_1,\sigma_2) & = I_3^+(\sigma_1,\sigma_2), \\
 J_3^+(\sigma_1,\sigma_2) & = I_1^-(\sigma_1,\sigma_2), &
 J_3^-(\sigma_1,\sigma_2) & = I_3^-(\sigma_1,\sigma_2), 
\end{align*}
we have
\[
\sum_{k \in \{ 2,4,5,6,7,8 \}} |I_k(\sigma_1,\sigma_2)| = \sum_{k \in \{ 2,4,5,6,7,8 \}} |J_k(\sigma_1,\sigma_2)|,
\]
where $I_k(\sigma_1,\sigma_2) = I_k^+(\sigma_1,\sigma_2) \sqcup I_k^-(\sigma_1,\sigma_2)$ 
and $J_k(\sigma_1,\sigma_2) = J_k^+(\sigma_1,\sigma_2) \sqcup J_k^-(\sigma_1,\sigma_2)$.
Moreover, since the map $(i,j) \mapsto (-j,-i)$ gives bijections
\begin{align*}
J_4^+(\sigma_1,\sigma_2) &\xrightarrow{1:1} I_8^-({}^c\sigma_2^\vee, {}^c\sigma_1^\vee), &
J_4^-(\sigma_1,\sigma_2) &\xrightarrow{1:1} I_4^-({}^c\sigma_2^\vee, {}^c\sigma_1^\vee), \\
J_8^+(\sigma_1,\sigma_2) &\xrightarrow{1:1} I_8^+({}^c\sigma_2^\vee, {}^c\sigma_1^\vee), &
J_8^-(\sigma_1,\sigma_2) &\xrightarrow{1:1} I_4^+({}^c\sigma_2^\vee, {}^c\sigma_1^\vee), 
\end{align*}
we have
\[
\tag{$\sharp$}\label{sharp}
\sum_{k \in \{ 2,5,6,7 \}} \left( |I_k(\sigma_1,\sigma_2)| + |I_k({}^c\sigma_2^\vee, {}^c\sigma_1^\vee)| \right)
= \sum_{k \in \{ 2,5,6,7 \}} \left( |J_k(\sigma_1,\sigma_2)| + |J_k({}^c\sigma_2^\vee, {}^c\sigma_1^\vee)| \right).
\]
Since
\[
n_1(\sigma_1,\sigma_2) = \sum_{k \in \{ 2,5,6,7 \}} |I_k(\sigma_1,\sigma_2)|,
\]
the left-hand side of \eqref{sharp} is equal to $n_1'(\sigma_1,\sigma_2)$.
On the other hand, noting that 
the map $(i,j) \mapsto (j,i)$ gives bijections 
\begin{align*}
J_2^+(\sigma_1,\sigma_2) &\xrightarrow{1:1} I_2^+(\sigma_2, \sigma_1), &  
J_2^-(\sigma_1,\sigma_2) &\xrightarrow{1:1} I_6^+(\sigma_2, \sigma_1), \\ 
J_6^+(\sigma_1,\sigma_2) &\xrightarrow{1:1} I_2^-(\sigma_2, \sigma_1), &  
J_6^-(\sigma_1,\sigma_2) &\xrightarrow{1:1} I_6^-(\sigma_2, \sigma_1), 
\end{align*}
and the map $(i,j) \mapsto (-i,-j)$ gives bijections
\begin{align*}
J_5^+(\sigma_1,\sigma_2) &\xrightarrow{1:1} I_7^-({}^c\sigma_1^\vee, {}^c\sigma_2^\vee), &
J_5^-(\sigma_1,\sigma_2) &\xrightarrow{1:1} I_5^-({}^c\sigma_1^\vee, {}^c\sigma_2^\vee), \\
J_7^+(\sigma_1,\sigma_2) &\xrightarrow{1:1} I_7^+({}^c\sigma_1^\vee, {}^c\sigma_2^\vee), &
J_7^-(\sigma_1,\sigma_2) &\xrightarrow{1:1} I_5^+({}^c\sigma_1^\vee, {}^c\sigma_2^\vee), 
\end{align*}
the right-hand side of \eqref{sharp} is equal to $n_1'(\sigma_2,\sigma_1)$.
This proves
\[
n_1'(\sigma_1,\sigma_2) = n_1'(\sigma_2,\sigma_1).
\]
\par

Next we consider the set 
\[
\{ i \in I(\sigma) \,|\, i>0 \}
= \{ i \in I^+(\sigma) \,|\, w(i)>0 \} \sqcup \{ i \in I^+(\sigma) \,|\, w(i)<0 \}.
\]
The map $i \mapsto i' = w^{-1}(i)$ gives a bijection from this set to 
\[
\{ i' \in I(\sigma) \, | \, w(i')>0 \} 
= \{ i' \in I^+(\sigma) \, | \,  w(i')>0 \} \sqcup \{ i' \in I^-(\sigma) \, | \, w(i')>0 \}.
\]
Hence we have
\[
n_2(\sigma) = |\{ i \in I^+(\sigma) \, | \, w(i)<0 \}|
= |\{ i \in I^-(\sigma) \, | \, w(i)>0 \}| = n_2({}^c\sigma^\vee), 
\]
where the last equation follows from the map $i \mapsto -i$.
This completes the proof of Proposition \ref{cr=1} in this case, 
and hence Theorem \ref{main1} (1) for the classical group $G$. 
\end{proof}

\subsection{Shahidi's formula for local coefficients}\label{sec.shahidi}
In a series of influential papers \cite{Sh1, Sh2, Sh3, Sh4, Sh5, KeSh, Sh6, Sh7} 
stretching over $12$ years, 
Shahidi made a deep study of local coefficients 
and developed a theory of $\gamma$-factors for generic representations of connected reductive groups.
In particular, as a culmination of this study, 
he showed that local coefficients can be expressed in terms of these $\gamma$-factors.
These results play a key role in the study of automorphic forms, in
particular in Arthur's theory of endoscopic classification.
\par

The theory of $\gamma$-factors is very delicate 
because it requires a careful normalization of various quantities 
(such as Weyl group representatives) 
and a careful evaluation of pertinent integrals in the rank $1$ case 
(i.e., for $\SL_2$ and $\SU_{2,1}$).
As the theory evolved over a period of $12$ years, 
and it was not a priori clear what the precise shape of the final product should be, 
it is understandable that different normalizations may have been preferred at different points in time, 
for different reasons.
As a consequence, various formulas (of the same quantity) 
which appeared in \cite{Sh5, KeSh, Sh7} are not entirely consistent with each other, at least at first glance.
In this subsection, with the benefit of hindsight, 
we shall explain and resolve these discrepancies so that one has a consistent story.
We thank Shahidi for his help in clarifying the material in this subsection, and stress that the discrepancies pointed out mainly result from
different choices of normalizations and the occasional typos in translating from one setting to another, and not of any conceptual flaws
in Shahidi's arguments. 
\par

We shall first review the key formula of Shahidi expressing local coefficients in terms of gamma factors.  
Let $F$ be a local field of characteristic zero and 
fix a non-trivial unitary character $\psi_F$ of $F$.
In this subsection, 
we consider an arbitrary quasi-split connected reductive algebraic group $G$ over $F$.
Fix an $F$-splitting $\spl = (B,T,\{X_\alpha\})$ of $G$ and
let $\ww$ be the Whittaker datum for $G$ determined by $\spl$ and $\psi_F$.
\par

Let $P = MN$ be a standard maximal parabolic subgroup of $G$.
We denote by $a$ the simple root of $A_T$ 
that does not belong to the root system of $M$, 
where $A_T$ is the split component of $T$.
Following \cite[p.~552]{Sh6}, 
we define $\tl{a}$ as the restriction of $\pair{\rho, \alpha^\vee}^{-1} \rho$ to $A_T$, 
where $\rho$ is half the sum of the absolute roots of $T$ in $N$, 
$\alpha$ is an absolute root which restricts to $a$, 
and $\alpha^\vee$ is the coroot associated to $\alpha$.
(Note that $\tl{a}$ is the corresponding fundamental weight when $G$ is semisimple and split over $F$.)
We may regard $\tl{a}$ as an element in $\aa_M^*$.
Let $r$ be the adjoint representation of ${}^LM$ on $\widehat{\nn}$, 
where $\widehat{\nn}$ is the Lie algebra of the unipotent radical of the dual parabolic subgroup $\widehat{P}$.
As in \cite[p.~554]{Sh6}, we decompose it as $r = \bigoplus_{i=1}^m r_i$.
Let $P' = M'N'$ be another standard maximal parabolic subgroup of $G$ 
and assume that $W(M,M') \not= \{ \1 \}$ (\resp $W(M,M') \not= \emptyset$) 
if $M = M'$ (\resp $M \ne M'$).
Denoting by $w$  the unique non-trivial element in $W(M,M')$,
let $\lambda(w, \psi_F)$ be the $\lambda$-factor defined  as in \cite[(4.1)]{KeSh} 
(see also Section \ref{sec.NIO}).
\par

Let $\pi$ be an irreducible $\ww_M$-generic representation of $M$, 
where $\ww_M$ is the Whittaker datum for $M$ induced by $\ww$.
Recall that the local coefficient $C_P(w, \pi_{s \tl{a}}, \psi_F)$ for $s \in \C$ 
depends on the choice of a  representative of $w$.
We take the Langlands--Shelstad representative $\tl{w}$ of $w$ with respect to $\spl$ 
as in Section \ref{sec.NIO}.
For example, if $G = \SL_2$ and $\spl$ is the standard splitting, 
then we have
\[
\tl{w} = 
\begin{pmatrix}
0 & 1 \\
-1 & 0
\end{pmatrix}.
\]
We write $C_P(w, \pi_{s \tl{a}}, \psi_F) = C_P(\tl{w}, \pi_{s \tl{a}}, \psi_F)$ to indicate this dependence.
\par

Now  the key formula expressing these local coefficients in terms of $\gamma$-factors, cf.~\cite[Theorem 3.5]{Sh7}, is:
 \[
\tag{S1}\label{S1}
C_P(\tl{w}, \pi_{s \tl{a}}, \psi_F) 
= \lambda(w, \psi_F)^{-1} \prod_{i=1}^m \gamma^{\Sh}(is, \pi, r_i, \psi_F), 
\]
where the superscript $\Sh$ indicates Shahidi's $\gamma$-factors. 
The astute reader will, however, observe that the formula \eqref{S1} does not completely agree with 
what is stated in \cite[Theorem 3.5]{Sh7}. 
It is the goal of this subsection to explain this discrepancy.
\par

There are two main reasons for the apparent discrepancy 
between the formula (\ref{S1}) and that in \cite[Theorem 3.5]{Sh7}:
\par

\begin{itemize}
\item[(a)] 
Normalization of the Satake isomorphism and more generally of the local Langlands correspondence (LLC) for tori.  
We thank Shahidi for pointing out this crucial difference in conventions to us. 
For simplicity, assume that $G$ is split with maximal split torus $T$, Weyl group $W$, 
and hyperspecial maximal compact subgroup $K = G(\oo_F)$. 
Then one has the Satake isomorphism
\[  
\Sc \colon \HH(G,K)  \rightarrow \C[T(F)/ T(\mathfrak{o}_F)]^W 
\rightarrow \C[X_*(T)]^W = \C[X^*(\widehat{T})]^W 
\rightarrow R(\widehat{G}) 
\]
of the spherical Hecke algebra $\HH(G,K)$ 
with the representation ring $R(\widehat{G})$ of the Langlands dual group $\widehat{G}$.
Here, the first arrow is given by the Satake transform, 
whereas the inverse of the third arrow is given by the restriction of the character of a representation of $\widehat{G}$ 
to the dual torus $\widehat{T} = X^*(T) \otimes \C^\times$. 
On the other hand, the second arrow depends on an isomorphism
\[
T(F) / T(\oo_F) \cong X_*(T) 
\]
which needs to be explicated. 
\par

Now there is a natural isomorphism $i \colon X_*(T) \xrightarrow{\sim} T(F)/ T(\oo_F)$ 
given by $\lambda \mapsto \lambda( \varpi_F)$, 
where $\varpi_F$ is a uniformizer of $F$. 
This isomorphism $i$ is obtained as a composite of the natural isomorphisms
\[ 
\left\{
\begin{aligned}
&X_*(T) \otimes F^{\times} \cong T(F) &&\text{ given by $\lambda \otimes z \mapsto \lambda(z)$,} \\
&X_*(T) \otimes \mathfrak{o}_F^{\times} \cong T(\mathfrak{o}_F)
\end{aligned}
\right.
\]
with
\[ 
\mathrm{val}_F \colon F^{\times}/ \oo_F^{\times} \cong \Z. 
\]
As explained in Appendix \ref{sec.Frob}, 
the LLC for the split torus $T$ is based on the first isomorphism above. 
Thus, the resulting normalization of the Satake isomorphism $\Sc$ using $i$ is the natural one, 
in the sense that it is compatible with the LLC of tori.
This is the normalization we use in this paper. 
\par

However, as Shahidi pointed out to us, 
in his \emph{Euler Product} book \cite{L1}, 
Langlands had used \emph{the negative of the above natural isomorphism}, 
and  Shahidi had adopted Langlands' normalization in his work.
The effect of this choice of  normalization is as follows. 
The Satake isomorphism induces a parametrization of 
the set of isomorphism classes of irreducible unramified representations of $G(F)$ 
with the set of semisimple conjugacy classes of $\widehat{G}$:
\[ 
\Irr_{\Ksph}(G(F)) \leftrightarrow \mathrm{Spec}(\HH(G,K)) 
\leftrightarrow \mathrm{Spec}(R(\widehat{G})) 
\leftrightarrow  \widehat{G}_{\mathrm{ss}}/\Int(\widehat{G}). 
\]
The two different normalizations of $\Sc$ give two different bijections 
which differ from each other by the inverse operation on semisimple classes. 
This in turn leads to a different notion  of local $L$-factors. 
More generally, 
using Langlands' normalization would also necessitate modifying the LLC for tori (composing with the inverse map), 
so as to ensure compatibility of the classification of unramified characters with that of general characters. 
\par

In the setting of unramified representations, 
if we use the normalization of $\Sc$ preferred in this paper (instead of the one Langlands and Shahidi used), 
then in \cite[Theorem 3.5, (3.11)]{Sh7}, the representation $\tilde{\sigma}$ there should be replaced by $\sigma$.  
\par

\item[(b)] 
Choice of Weyl group representatives. 
We write $\tl{w}^\Sh$ for the representative of $w$ used in \cite{Sh7}. 
As explained in the middle of \cite[p.~281]{Sh7}, 
$\tl{w}^\Sh$ is the same as the one given in \cite[p.~979]{Sh5} and \cite[p.~74]{KeSh}. 
Note that $\tl{w}^\Sh$ differs from the representative $\tl{w}$ used in this paper 
and is in fact the Langlands--Shelstad representative with respect to the splitting $\spl^- = (B,T,\{-X_\alpha\})$.
For example, 
if $G = \SL_2$ and $\spl$ is the standard splitting, 
then we have
\[
\tl{w}^{\Sh} = 
\begin{pmatrix}
0 & -1 \\
1 & 0
\end{pmatrix}.
\]
A discussion of such issues about Weyl group representatives can also be found in \cite[Remark 8.2.1, p.~131--133]{Sh8}. 
\end{itemize}
\par

Let us examine if (a) and (b) explain the apparent discrepancy between \eqref{S1} and \cite[Theorem 3.5]{Sh7}.  
Writing $C_P(\tl{w}^\Sh, \pi_{s \tl{a}}, \psi_F)$ for the local coefficient associated to $\tl{w}^{\Sh}$ and noting that
$\spl^-$ and $\overline{\psi}_F$ give rise to the fixed Whittaker datum $\ww$,  the equality \eqref{S1} is equivalent to
\[
\tag{S1'}\label{S1'}
C_P(\tl{w}^\Sh, \pi_{s \tl{a}}, \psi_F) 
= \lambda(w, \overline{\psi}_F)^{-1} \prod_{i=1}^m \gamma^\Sh(is, \pi, r_i, \overline{\psi}_F).
\]
On the other hand,  the identity in \cite[Theorem 3.5]{Sh7} reads:
\[
\tag{S2}\label{S2}
C_P(\tl{w}^\Sh, \pi_{s \tl{a}}, \psi_F) 
= \lambda(w, \psi_F)^{-1} \prod_{i=1}^m \gamma^\Sh(is, \pi, r^{\vee}_i, \overline{\psi}_F).
\]
\emph{In the setting  of unramified representations}, 
if we adjust for the different convention of the Satake isomorphism used by Shahidi and us, 
as discussed in (a) above, 
then the representation $r_i^{\vee}$ on the right-hand side of \eqref{S2} should be replaced by $r_i$.
Hence, with the normalization used in this paper, 
\cite[Theorem 3.5, (3.11)]{Sh7} should be restated as:
\[
\tag{S2'}\label{S2'}
C_P(\tl{w}^\Sh, \pi_{s \tl{a}}, \psi_F) 
= \lambda(w, \psi_F)^{-1} \prod_{i=1}^m \gamma^\Sh(is, \pi, r_i, \overline{\psi}_F).
\]
Comparing this with \eqref{S1'}, 
we see that the discrepancy between \eqref{S1} and \cite[Theorem 3.5]{Sh7} is almost completely resolved: 
there is only a remaining ambiguity between $\psi_F$ and $\overline{\psi}_F$ in the $\lambda$-factor. 
Indeed, in the setting of unramified groups, 
the $\lambda$-factor is insensitive to whether we are using $\psi_F$ or $\overline{\psi}_F$, 
so that the discrepancy is fully resolved.  
However, the discrepancy between \eqref{S1'} and \eqref{S2'} persists for 
general (possibly ramified) quasi-split groups (such as for  $G = \Res_{E/F} G'$, 
with $E/F$ a ramified extension and $G'$ a split group over $E$).
\par

The presence of $\psi_F$ (instead of $\overline{\psi}_F$) in the $\lambda$-factor in \eqref{S2'} 
must surely be a typo (as Shahidi has also mentioned to us). 
In the interest of transparency, however, it will be pertinent for us to explain how this typo could have arisen, 
and to explain why \eqref{S1} (or equivalently \eqref{S1'}) is the right formula. 
Before that, we explain in the following remark 
why the formula \eqref{S2'} is not consistent with the restriction of scalars and hence is not the correct identity.
\par

\begin{rem}\label{rem_Sh7}
For simplicity, we assume that $G$ is split over $F$.
Let $F_0$ be a subfield of $F$ such that $F/F_0$ has finite degree 
and assume that $\psi_F = \psi_{F_0} \circ \tr_{F/F_0}$ 
for some non-trivial additive character $\psi_{F_0}$ of $F_0$.
Put $G_0 = \Res_{F/F_0} G$, $P_0 = \Res_{F/F_0} P$, and $M_0 = \Res_{F/F_0} M$.
Denote by $a_0$ the simple root corresponding to $M_0$.
Let $\spl_0$ be the $F_0$-splitting of $G_0$ induced by $\spl$.
Then $\tl{w}^\Sh$ (regarded as an element in $G_0(F_0)$) 
is the representative of $w$ with respect to $\spl_0^-$, 
and we can define the local coefficient $C_{P_0}(\tl{w}^\Sh, \pi_{s \tl{a_0}}, \psi_{F_0})$, 
where $\pi$ is regarded as a representation of $M_0(F_0)$.
Since $\pi_{s \tl{a_0}} = \pi_{s \tl{a}}$ under the identification $M_0(F_0) = M(F)$, 
and $\spl_0$ and $\psi_{F_0}$ give rise to the Whittaker datum $\ww$, 
it follows from the definition that 
\[
C_{P_0}(\tl{w}^\Sh, \pi_{s \tl{a_0}}, \psi_{F_0}) = C_P(\tl{w}^\Sh, \pi_{s \tl{a}}, \psi_F).
\]
From this and \eqref{S2'}, we can deduce that
\[
\lambda(w, \psi_{F_0})^{-1} 
\prod_{i=1}^m \gamma^\Sh(is, \pi, \Ind^{{}^L M_0}_{{}^L M}(r_i), \overline{\psi}_{F_0})
= \prod_{i=1}^m \gamma^\Sh(is, \pi, r_i, \overline{\psi}_F),
\]
noting that $\lambda(w,\psi_F) = 1$.
By the property of $\lambda$-factors (see \cite[Section 5.6]{D}), 
we should have
\[
\gamma^\Sh(is, \pi, \Ind^{{}^L M_0}_{{}^L M}(r_i), \overline{\psi}_{F_0})
= \lambda(F/F_0, \overline{\psi}_{F_0})^{\dim r_i} \gamma^\Sh(is, \pi, r_i, \overline{\psi}_F) 
\]
and hence 
\[
\lambda(w,\psi_{F_0})^{-1} \lambda(F/F_0, \overline{\psi}_{F_0})^{\dim N} = 1.
\]
However, for general ramified extensions $F/F_0$, this is not consistent with the definition of $\lambda(w,\psi_{F_0})$:
\[
\lambda(w,\psi_{F_0}) = \lambda(F/F_0,\psi_{F_0})^{\dim N}.
\]
This explains why one should expect to have
$\overline{\psi}_F$ in the $\lambda$-factor in \eqref{S2'} instead of $\psi_F$.
The above discussion also explains why the $\lambda$-factor is necessarily involved in the formula \eqref{S1}; 
this necessity was first realized by Shahidi in \cite[Lemma 3.2 and Remark, p.~992]{Sh5}. 
\end{rem}
\par

While the above remark serves to explain why \eqref{S2'} is not  the correct formula, 
one would still need to justify that \eqref{S1} (or equivalently \eqref{S1'}) is the right one, 
and also to shed some light on how the typo in the $\lambda$-factor arises. 
\par

For this, let us recall Shahidi's approach for showing the desired identity. 
His argument consists of the following two parts:

\begin{itemize}
\item[(i)] 
Local part: 
an explicit computation of the local coefficient in the archimedean setting 
and for principal series representations in the non-archimedean setting.
This local computation was carried out in \cite{Sh5} and \cite{KeSh} respectively, 
with the results contained in \cite[Theorem 3.1]{Sh5} and \cite[Proposition 3.4]{KeSh}.
In these settings, as the local Langlands correspondence for tori and real reductive groups is known, 
the right-hand side of the identity \eqref{S1} is interpreted as a product of Artin gamma factors 
attached to the $L$-parameter of the representation $\pi$ (i.e. arithmetic or Galois theoretic gamma factors).
\par

For principal series representations (in both archimedean and non-archimedean settings), 
the argument proceeded by a standard reduction to ``rank $1$ cases'' 
(via expressing a Weyl group element as a product of simple reflections 
and using the factorizability of standard intertwining operators). 
For these rank $1$ cases, namely for $G = \SL_2$ or $G = \SU_{2,1}$, 
one then had to explicitly evaluate the local coefficient. 
These rank $1$ results can be found in \cite[Lemma 1.4]{Sh5} 
for both $\SL_2$ and $\SU_{2,1}$ in the archimedean setting 
and in \cite[Proposition 3.2]{KeSh} for $\SU_{2,1}$ in the non-archimedean setting. 
\par

To extend the result to all irreducible representations $\pi$ in the archimedean setting, 
Shahidi appealed to the Casselman subrepresentation theorem 
to realize $\pi$ as a submodule of a principal series representation. 
The desired identity was then obtained by showing that both sides of the identity satisfy the property of multiplicativity 
(called the inductive property in \cite[Theorem 3.5 (3)]{Sh7}): 
if $\pi_1$ and $\pi_2$ are two irreducible subquotients of the same principal series representation, 
then the local coefficient (respectively Artin or Galois theoretic gamma factors) of $\pi_1$ and $\pi_2$ are equal, 
and can be expressed in terms of the analogous invariants of the inducing data. 
For the local coefficients, 
this multiplicativity property follows from the factorizability of standard intertwining operators 
and was shown in \cite[Proposition 3.2.1, p.~336]{Sh2}. 
The multiplicativity for Artin or Galois theoretic gamma factors was essentially what was verified in \cite[p.~993--1002]{Sh5}  
(though not formulated as such there). 
It is undoubtedly the most difficult part of the proof of \cite[Theorem 3.1]{Sh5}, 
making significant use of deep results in the representation theory 
and local Langlands correspondence for real reductive groups. 
(We also refer the reader to Andler's papers \cite{An1,An2} for related discussions.)

\item[(ii)] 
Global part: 
a global inductive argument in \cite{Sh7}, 
using the formulas for the local coefficients evaluated in (i) 
as inputs to deduce the general case at all non-archimedean places.
\end{itemize}

Given this overview of Shahidi's proof of the desired identity, 
we see that the answers in the rank $1$ cases are absolutely fundamental. 
More precisely, if the rank $1$ computations were to yield \eqref{S1} (or equivalently \eqref{S1'}) as the answer, 
then Shahidi's arguments developed in (i) and (ii) above would imply that \eqref{S1} holds in general.
\par

In examining \cite[Lemma 1.4]{Sh5} and \cite[Propositions 3.2 and 3.4]{KeSh}, 
however, the careful reader will notice that the rank $1$ results contained there do not always agree with \eqref{S1}. 
In the following two remarks, we explain this in greater detail.

\begin{rem}\label{rem_Sh5}
Consider the setting of  \cite{Sh5}, so that  $F$ is archimedean.
In \cite{Sh5},  the induced representation $I_P(\pi)$ 
was realized on the space of $\VV_\pi$-valued smooth functions $f$ on $G$ such that 
\[
f(gmn) = \delta_P(m)^{-\half{1}} \pi(m)^{-1} f(g)
\]
for $g \in G$, $m \in M$, and $n \in N$, 
with associated intertwining operators, Whittaker functionals, and local coefficients.
By the isomorphism $f \mapsto [g \mapsto f(g^{-1})]$ from this realization to our realization, 
we see that the local coefficients defined in \cite{Sh5} agree with ours.
Then \cite[Theorem 3.1]{Sh5} states that
\[
\tag{S3}\label{S3}
C_P(\tl{w}^\Sh, \pi_{s \tl{a}}, \psi_F) = \lambda(w, \psi_F) \prod_{i=1}^m \gamma(is, \pi, r_i, \psi_F).
\]
(Note that $-2s \rho_\theta$ on the left-hand side of \cite[(3.1.1)]{Sh5} is a typo and should be $2s \rho_\theta$.
The same sign already appears in \cite{Sh2} and seems to have been retained from the conventions of \cite{Sh1}, 
although the conventions in \cite{Sh2} are slightly different.)  
Moreover, this identity was deduced from the rank $1$ results in \cite[Lemma 1.4]{Sh5}.
\par

Observe that \eqref{S3} does not agree with  \eqref{S1'}. 
More precisely, specializing to the rank $1$ cases:

\begin{itemize}
\item  
When $G = \SL_2$, 
the identity \eqref{S3} (also stated in \cite[Lemma 1.4 (a)]{Sh5}) differs from \eqref{S1}, 
unless one replaces $\tl{w}^\Sh$ in \eqref{S3} by $\tl{w}$, 
noting that $\lambda(w, \psi_F) = 1$ in this case.

\item
When $F = \R$ and $G = \Res_{\C/\R} \SL_2$, 
the identity \eqref{S3} differs from \eqref{S1}, 
unless one replaces $\tl{w}^\Sh$ by $\tl{w}$ and $\lambda(\C/\R,\psi_\R)$ by $\lambda(\C/\R,\psi_\R)^{-1}$. 
The inverse which is missing from $\lambda(\C/\R,\psi_\R)$ is just a typo 
which occurs on the right-hand side of \cite[(3.1.1)]{Sh5} 
(and seems to have arisen when applying \cite[(3.10)]{Sh5}).

\item 
When $F = \R$ and $G = \SU_{2,1}$, the equality \eqref{S3} 
(which was also  stated in \cite[Lemma 1.4 (b)]{Sh5}) agrees with \eqref{S1}. 
In this case, we have $C_P(\tl{w}^\Sh, \pi_{s \tl{a}}, \psi_F) = C_P(\tl{w}, \pi_{s \tl{a}}, \psi_F)$ 
and $\lambda(w, \psi_F) = \lambda(w, \psi_F)^{-1}$ since 
\[
\tl{w}^{\Sh} = \tl{w} = 
\begin{pmatrix}
0 & 0 & 1 \\
0 & -1 & 0 \\
1 & 0 & 0 
\end{pmatrix}
\]
and $\lambda(w, \psi_F) = \lambda(\C/\R, \psi_{\R})^2 = -1$.
\end{itemize}
\end{rem}

\begin{rem}\label{rem_KeSh}
We now consider the setting of  \cite{KeSh}, 
so that $F$ is non-archimedean and $\pi$ is a principal series representation.
Then \cite[Proposition 3.4]{KeSh} states that 
\[
\tag{S4}\label{S4}
C_P(\tl{w}^\Sh, \pi_{s \tl{a}}, \psi_F) 
= \lambda(w, \psi_F)^{-1} \prod_{i=1}^m \gamma(is, \pi, r_i^\vee, \psi_F).
\]
Observe that this does not agree with \eqref{S1'}.  
As we have explained, 
one reason is the different conventions for the LLC for tori, as explained in (a) above. 
Using our convention (instead of Langlands'), 
one should replace $r_i^{\vee}$ on the right-hand side by $r_i$. 
After this, one notes:
\begin{itemize}
\item
When $G = \SL_2$, 
one needs to replace $\tl{w}^\Sh$ in \eqref{S4} by $\tl{w}$ to have agreement with \eqref{S1}, 
as in the archimedean case.

\item
When $G = \SU_{2,1}$, 
the identity \eqref{S4} (also  stated in \cite[Corollary 3.3]{KeSh}) agrees with \eqref{S1}.
In this case, we have $C_P(\tl{w}^\Sh, \pi_{s \tl{a}}, \psi_F) = C_P(\tl{w}, \pi_{s \tl{a}}, \psi_F)$ 
as in the archimedean case.
\end{itemize}
\end{rem}

To summarize, 
the rank $1$ results in \cite{Sh5, KeSh} agree with \eqref{S1} when $G = \SU_{2,1}$ 
but do not agree in the case of $\SL_2$. 
However, in the case of $\SL_2$, 
the discrepancy can be resolved if one replaces $\tl{w}^{\Sh}$ by $\tl{w}$ 
which does not affect the result for $\SU_{2,1}$ 
(see \cite[p.~133]{Sh8} for a separate discussion).
It is thus necessary to figure out what the truth is for $\SL_2$.
\par

For the case of $\SU_{2,1}$, 
the details of the computation of the local coefficient can be found  in 
\cite[proof of Lemma 1.4 (b)]{Sh5} and the paper \cite[Section 5]{Ke} of Keys. 
However, the details of the proof for the $\SL_2$ case were not provided in \cite{Sh5, KeSh}, 
but the reader was referred to \cite[Lemma 4.4]{Sh1}, 
which was based on results in Jacquet's foundational paper \cite{J}. 
Because \cite{J} was written 60 years ago, 
its formulation of principal series representations, 
intertwining operators, and Whittaker functionals differs considerably from the modern language 
(being based on the $K$-type description of principal series representations), 
making it somewhat non-trivial to translate the statement of \cite[Lemma 4.4]{Sh1} 
to the setting of \cite{Sh5,KeSh} and of this paper. 
\par

Thus, instead of reconstructing the translation from \cite[Lemma 4.4]{Sh1} to \cite{Sh5, KeSh}, 
we have given an independent computation of the local coefficient for $\SL_2$ 
in Proposition \ref{SL(2)} below.
What Proposition \ref{SL(2)} shows is that \eqref{S1} holds in the case of $\SL_2$.
We have also carried out the $\SU_{2,1}$ computation in Proposition \ref{SU(2,1)} below 
and our results agree with those of \cite{Sh5, KeSh} (and hence \eqref{S1}).
\par

\begin{rem} \label{rem_Sh5KeShSh7} 
Let us make two further remarks:

\begin{itemize}
\item
The careful reader will notice that there are some discrepancies 
between the formulas for local coefficients established in \cite{Sh5, KeSh} 
and the formulas that were  quoted and used in \cite{Sh7}, 
specifically in \cite[Propositions 3.2 and 3.4]{Sh7}. 
Shahidi has informed us that this is due to typos in transcription. 

\item
It is certainly important for the local results of \cite{Sh5} and \cite{KeSh} 
to be formulated uniformly as special cases of a universal identity. 
This is necessary for the global argument of \cite{Sh7} to go through 
and propagate the identity to the general case. 
In Propositions \ref{SL(2)} and \ref{SU(2,1)} below, 
we have carried out the computations for the local coefficients of $\SL_2$ and $\SU_{2,1}$ 
uniformly over all local fields (of characteristic $0$), 
without considering archimedean and non-archimedean cases separately. 
\end{itemize}
\end{rem}

Taken together, the differences in convention noted in (a) and (b) above, 
along with Remarks \ref{rem_Sh7}, \ref{rem_Sh5}, \ref{rem_KeSh} and Proposition \ref{SL(2)} 
explain the discrepancy between \eqref{S1} and \eqref{S2}, 
and why \eqref{S1} is the correct version. 
To summarize, besides some minor and easily rectifiable typos, 
the reasons for the various discrepancies are:

\begin{itemize}
\item[(a)] 
different normalizations of the Satake isomorphism (and the LLC for tori);
\item[(b)] 
different choices of Weyl group representatives;
\item[(c)]  
discrepancy in computation of the local coefficient for $\SL_2$.
\end{itemize} 
When adjusted for (a), (b), and (c), 
the results of \cite[Theorem 3.1]{Sh5} and \cite[Proposition 3.4]{KeSh} agree with \eqref{S1}. 
When used as inputs for Shahidi's inductive global argument in \cite{Sh7}, 
they lead to the identity \eqref{S1} in general.
\par

\subsection{The rank $1$ cases}
Here, we include the promised computation of the local coefficient in the rank $1$ cases.
Let $\Sc(F)$ be the space of Schwartz--Bruhat functions on $F$.
For $\phi \in \Sc(F)$, we define its Fourier transform $\widehat{\phi} \in \Sc(F)$ by
\[
\widehat{\phi}(x) = \int_{F} \phi(y) \psi_F(xy) dy,
\]
where $dy$ is the self-dual Haar measure on $F$ with respect to $\psi_F$.
For $s \in \C$, $\phi \in \Sc(F)$, and a unitary character $\chi$ of $F^{\times}$, 
put
\[
Z(s, \chi, \phi) = \int_{F^{\times}} \phi(x) \chi(x) |x|_F^s d^{\times}x,
\]
where $d^{\times} x = |x|_F^{-1} dx$.
This integral is absolutely convergent for $\re(s) > 0$ and admits a meromorphic continuation to $\C$.
Moreover, the functional equation 
\[
\tag{{\bf FE}}
\label{FE}
Z(1-s, \chi^{-1}, \widehat{\phi}) = \gamma(s, \chi, \psi_F) Z(s, \chi, \phi) 
\]
holds, 
where 
\[
\gamma(s, \chi, \psi_F) = \varepsilon(s, \chi, \psi_F) \frac{L(1-s, \chi^{-1})}{L(s, \chi)}.
\]
\par

\begin{lem}
\label{lem:zeta-fourier}
Assume that $\widehat{\phi} \in C_c^\infty(F^\times)$.
Then we have
\[
\int_{F} Z(1-s, \chi^{-1}, \widehat{\phi_b}) \psi_F(ab) db 
= \widehat{\phi}(a) \chi(a)^{-1} |a|_F^{-s}
\]
for $a \in F^{\times}$, where $\phi_b(x) = \phi(x+b)$.
\end{lem}
\begin{proof}
By assumption, we may define $\Phi \in \Sc(F)$ by 
\[
\Phi(x) = \widehat{\phi}(x) \chi(x)^{-1} |x|_F^{-s}
\]
so that $\Phi$ vanishes in a neighborhood of $0$.
Since $\widehat{\phi_b}(x) = \widehat{\phi}(x) \psi_F(-bx)$, 
we have
\[
Z(1-s, \chi^{-1}, \widehat{\phi_b}) 
= \int_{F^\times} \widehat{\phi}(x) \chi(x)^{-1} |x|_F^{1-s} \psi_F(-bx) d^{\times} x 
= \widehat{\Phi}(-b).
\]
Hence the left-hand side of the equation in the lemma is absolutely convergent for all $s$ 
and is equal to 
\[
\int_{F} \widehat{\Phi}(b) \psi_F(-ab) db = \Phi(a).
\]
This completes the proof.
\end{proof}

Let $E$ be a quadratic extension of $F$ 
and write $x \mapsto \overline{x}$ for the non-trivial element in $\Gal(E/F)$.
We denote by $E_0$ the set of trace zero elements in $E$ 
and define a non-degenerate symmetric bilinear form $\beta \colon E_0 \times E_0 \rightarrow F$ 
by $\beta(y,y') = yy'$.
Then we have
\[
\int_E \varphi(z) dz = |2|_E^{\half{1}} \int_{E_0} \int_F \varphi(x + y) dx dy
\]
for $\varphi \in L^1(E)$, 
where $dz$ (\resp $dx$, \resp $dy$) is the self-dual Haar measure on $E$ (\resp $F$, \resp $E_0$) 
with respect to $\psi_E = \psi_F \circ \tr_{E/F}$ (\resp $\psi_F$, \resp $\psi_F \circ \beta$).

\begin{lem}
\label{lem:zeta-hks}
Let $\chi$ be a unitary character of $E^{\times}$.
Then we have
\[
\int_{E_0} \chi(1 + y) |1 + y|_E^{s-1} dy
= \chi(2) |2|_E^{s-\half{1}} \frac{\gamma(2s, \chi_0, \psi_F)}{\gamma(s, \chi, \psi_E)}
\]
for $\re(s) < \half{1}$, 
where $\chi_0 = \chi|_{F^{\times}}$.
\end{lem}
\begin{proof}
This was proved in \cite[Proposition 8.2]{HKS}, 
but we include the proof for the convenience of the reader.
Put
\[
A(s, \chi) = \int_{E_0} \chi(1 + y) |1 + y|_E^{s-1} dy.
\]
This integral is absolutely convergent for $\re(s) < \half{1}$.
In fact, $\inf_{y \in E_0}|1+y|_E > 0$ and 
\begin{align*}
\int_{y \in E_0,\; |y|_E > 1} |1 + y|_E^{\re(s)-1} dy
&\leq \int_{y \in E_0,\; |y|_E > 1} |y|_E^{\re(s)-1} dy 
\\&= \int_{x \in F,\; |x|_F^2 > |\delta|_E^{-1}} |\delta|_E^{\re(s)-\half{1}} |x|_F^{2(\re(s)-1)} dx,
\end{align*}
where we write $y = \delta x$ for a fixed element $\delta \in E_0$, 
in which case $dy = |\delta|_E^{\half{1}}dx$.
The last integral converges when $2(\re(s)-1) < -1$.
\par

Fix $\varphi \in \Sc(E)$ and $\phi \in \Sc(F)$.
For $0 < \re(s) < \half{1}$, we have
\begin{align*}
Z(2s, \chi_0, \phi) A(s, \chi) 
 & = \int_F \phi(x) \chi(x) |x|_F^{2s-1} dx \int_{E_0} \chi(1 + y) |1 + y|_E^{s-1} dy \\
 & = \int_{E_0} \int_F \phi(x) \chi(x + xy) |x + x y|_E^{s-1} |x|_F dx dy \\
 & = \int_{E_0} \int_F \phi(x) \chi(x + y) |x + y|_E^{s-1} dx dy \\
 & = |2|_E^{-\frac{1}{2}} \int_E \phi(\tfrac{1}{2} \tr_{E/F}(z)) \chi(z) |z|_E^{s-1} dz.
\end{align*}
Hence we have
\begin{align*}
 & |2|_E^{\frac{1}{2}} Z(2s, \chi_0, \phi) A(s, \chi) Z(1-s, \chi^{-1}, \widehat{\varphi}) \\  
 & = \int_E \phi(\tfrac{1}{2} \tr_{E/F}(z)) \chi(z) |z|_E^{s-1} dz 
 \int_E \widehat{\varphi}(w) \chi(w)^{-1} |w|_E^{-s} dw \\
 & = \int_E \int_E \phi(\tfrac{1}{2} \tr_{E/F}(z)) \widehat{\varphi}(w) \chi(z w^{-1}) |z w^{-1}|_E^{s-1} |w|_E^{-1} dz dw \\
 & = \int_E \int_E \phi(\tfrac{1}{2} \tr_{E/F}(z w)) \widehat{\varphi}(w) \chi(z) |z|_E^{s-1} dz dw \\
 & = \int_E \int_E \left( \int_F \widehat{\phi}(x) \psi_F(- \tfrac{1}{2} \tr_{E/F}(z w) x) dx \right) \widehat{\varphi}(w) \chi(z) |z|_E^{s-1} dz dw \\
 & = \int_F \int_E \left( \int_E \widehat{\varphi}(w) \psi_E(- \tfrac{1}{2} x z w) dw \right) \widehat{\phi}(x) \chi(z) |z|_E^{s-1} dz dx \\
 & = \int_F \int_E \varphi(\tfrac{1}{2} x z) \widehat{\phi}(x) \chi(z) |z|_E^{s-1} dz dx \\
 & = \chi(2) |2|_E^s \int_E \varphi(z) \chi(z) |z|_E^{s-1} dz
 \int_F \widehat{\phi}(x) \chi(x)^{-1} |x|_F^{-2s} dx \\
 & = \chi(2) |2|_E^s Z(s, \chi, \varphi) Z(1-2s, \chi_0^{-1}, \widehat{\phi}).
\end{align*}
Thus the lemma follows from the functional equation \eqref{FE}.
\end{proof}
\par

Now, let us consider the case $G = \SL_2(F)$. 
Take the Borel subgroup $B$ of $G$ consisting of upper triangular matrices.
Put 
\[
m(a) = 
\begin{pmatrix} 
a & \\ & a^{-1} 
\end{pmatrix},
\quad
n(b) = 
\begin{pmatrix} 
1 & b \\ & 1 
\end{pmatrix},
\quad
\tl{w} = 
\begin{pmatrix} 
& 1 \\ -1 &  
\end{pmatrix} 
\]
for $a \in F^\times$ and $b \in F$.
Let $\chi$ be a unitary character of $F^\times$.
We consider the normalized parabolically induced representation $I(s,\chi) = \Ind^G_B(\chi{|\cdot|_F^s})$ of $G$ 
on the space of smooth functions $f$ on $G$ such that
\[
f(m(a) n(b) g) = \chi(a) |a|_F^{s+1} f(g)
\]
for all $a \in F^\times$, $b \in F$, and $g \in G$.
Recall that the intertwining operator
\[
J(s,\chi) = J_B(w,\chi|\cdot|_F^s) \colon I(s,\chi) \rightarrow I(-s,\chi^{-1})
\]
is given by (the meromorphic continuation of) 
\[
J(s,\chi) f(g) = \int_F f(\tl{w}^{-1} n(x) g) dx,
\]
where $dx$ is the self-dual Haar measure on $F$ with respect to $\psi_F$.
Note that this integral is absolutely convergent for $\re(s) > 0$.
Recall also that the Whittaker functional $\Omega(s,\chi) = \Omega(\chi|\cdot|_F^s)$ on $I(s,\chi)$ is given by (the holomorphic continuation of) 
\[
\Omega(s,\chi) f = \int_F f(\tl{w}^{-1} n(b)) \psi_F(-b) db.
\]
Then the local coefficient $C(s, \chi, \psi_F) = C_B(\widetilde{w},\chi|\cdot|_F^s,\psi_F)$ 
is given by 
\[
\Omega(s,\chi) = C(s,\chi,\psi_F) \cdot \Omega(-s,\chi^{-1}) \circ J(s,\chi) 
\]
(see \cite[p.~333, Theorem 3.1]{Sh2}).

\begin{prop}
\label{SL(2)}
We have
\[
C(s,\chi,\psi_F) = \gamma(s,\chi,\psi_F).
\]
\end{prop}
\begin{proof}
Fix $\phi \in \Sc(F)$ such that $\widehat{\phi} \in C_c^\infty(F^\times)$.
We may define $f \in I(s,\chi)$ uniquely by requiring that
\[
f(\tl{w}^{-1} n(b)) = \phi(b)
\]
for all $b \in F$.
Then we have
\[
\Omega(s,\chi) f = \int_F \phi(b) \psi_F(-b) db = \widehat{\phi}(-1).
\]
On the other hand, we have
\[
\tl{w}^{-1} n(x) \tl{w}^{-1} = \begin{pmatrix} x^{-1} & -1 \\ & x \end{pmatrix} \tl{w}^{-1} n(-x^{-1})
\]
for $x \in F^{\times}$.
For $\re(s) > 0$, we have
\begin{align*}
J(s,\chi) f(\tl{w}^{-1} n(b))  
& = \int_F f(\tl{w}^{-1} n(x) \tl{w}^{-1} n(b)) dx \\
& = \int_F \chi(x^{-1}) |x^{-1}|_F^{s+1} \phi(-x^{-1}+b) dx \\ 
& = \int_{F^\times} \phi(x+b) \chi(-x) |x|_F^s d^\times x \\ 
& = \chi(-1) Z(s,\chi,\phi_b) \\
& = \chi(-1) \frac{Z(1-s,\chi^{-1},\widehat{\phi_b})}{\gamma(s,\chi,\psi_F)},
\end{align*}
where $\phi_b(x) = \phi(x+b)$.
Hence, by Lemma \ref{lem:zeta-fourier}, we have 
\begin{align*}
\Omega(-s,\chi^{-1}) J(s,\chi) f 
&= 
\chi(-1) \frac{1}{\gamma(s,\chi,\psi_F)}\int_F Z(1-s,\chi^{-1},\widehat{\phi_b}) \psi_F(-b) db 
\\&= 
\widehat{\phi}(-1) \frac{1}{\gamma(s,\chi,\psi_F)}
= \frac{1}{\gamma(s,\chi,\psi_F)} \Omega(s,\chi) f.
\end{align*}
Choosing $\phi$ such that $\widehat\phi(-1) \not= 0$ and hence $\Omega(s,\chi) f \not= 0$, 
this together with the definition of $C(s,\chi,\psi_F)$ yields the desired assertion
\[
C(s,\chi,\psi_F) = \gamma(s,\chi,\psi_F).
\]
This completes the proof of Proposition \ref{SL(2)}. 
\end{proof}

Next, let us consider the case
\[
G = \SU_{2,1} 
= \left\{ g \in \SL_3(E) \, \middle| \, 
{}^t \overline{g} J g = J
\right\},
\quad J = \begin{pmatrix} 
& & 1 \\ 
& -1 & \\ 
1 & & 
\end{pmatrix},
\]
where $E/F$ is a quadratic extension. 
Note that $G_{\overline{F}} = \SL_3$ with the Galois action given by 
\[
\tilde\sigma \colon g \mapsto \left\{
\begin{aligned}
&\sigma(g) \iif \sigma \in \Gamma_E, \\
& J^{-1} \cdot {}^t\sigma(g)^{-1} \cdot J \iif \sigma \in \Gamma_F \setminus \Gamma_E, 
\end{aligned}
\right.
\]
where $\sigma(g)$ means the usual action of $\Gamma_F$ on the matrix $g$.
Take the Borel subgroup $B$ (\resp maximal torus $T$) of $G$ consisting of upper triangular (\resp diagonal) matrices.
If we denote by $E_{i,j}$ the square matrix of size $3$ with $1$ at the $(i,j)$-th entry and $0$ elsewhere, 
then $\{E_{1,2}, E_{2,3}\}$ is a set of simple root vectors which is $\Gamma_F$-stable. 
We use the $F$-splitting $\spl = (B,T, \{E_{1,2}, E_{2,3}\})$.
The unique non-trivial element $w$ in the rational Weyl group corresponds to 
$(1,3) = (1,2) (2,3) (1,2)$ in the symmetric group $\SS_3$, 
which is identified with the absolute Weyl group. 
Define representatives $\tl{s}_1$ and $\tl{s}_2$ of the simple reflections $(1,2)$ and $(2,3)$ by 
\begin{align*}
\tl{s}_1 &= \exp(E_{1,2}) \exp(-E_{2,1}) \exp(E_{1,2}) 
=\begin{pmatrix}
0 & 1 & 0 \\
-1 & 0 & 0 \\
0 & 0 & 1
\end{pmatrix}, \\
\tl{s}_2 &= \exp(E_{2,3}) \exp(-E_{3,2}) \exp(E_{2,3}) 
= \begin{pmatrix}
1 & 0 & 0 \\
0 & 0 & 1 \\
0 & -1 & 0
\end{pmatrix},
\end{align*}
respectively. 
Then the Langlands--Shelstad representative $\tl{w}$ of $w$ is given by 
\begin{align*}
\tl{w} &= \tl{s}_1 \tl{s}_2 \tl{s}_1 = \begin{pmatrix} 
& & 1 \\ 
& -1 & \\ 
1 & & 
\end{pmatrix}.
\end{align*}
\par

Put
\[
m(a) = 
\begin{pmatrix} 
a & & \\ 
& a^{-1} \overline{a} & \\ 
& & \overline{a}^{-1} 
\end{pmatrix}, 
\quad
n(b,c) = 
\begin{pmatrix} 
1 & b & \frac{1}{2} b \overline{b} + c \\ 
& 1 & \overline{b} \\ 
& & 1 
\end{pmatrix} 
\]
for $a \in E^\times$, $b \in E$, and $c \in E_0$.
Note that 
\[
n(b,c) n(b',c') = n \left( b+b',c+c'+ \frac{1}{2} (b \overline{b'} - \overline{b} b') \right).
\]
Let $\chi$ and $\mu$ be unitary characters of $E^\times$ and $E^1$, respectively, 
where $E^1$ is the set of norm one elements in $E$.
We consider the normalized parabolically induced representation 
$I(s,\chi,\mu) = \Ind^G_B(\chi{|\cdot|_E^s} \otimes \mu)$ of $G$ 
on the space of smooth functions $f$ on $G$ such that
\[
f(m(a) n(b,c) g) = \chi(a) \mu(a^{-1} \overline{a}) |a|_E^{s+1} f(g)
\]
for all $a \in E^\times$, $b \in E$, $c \in E_0$, and $g \in G$.
Recall that the intertwining operator
\[
J(s,\chi,\mu) = J_B(w,\chi|\cdot|_E^s \otimes \mu) \colon I(s,\chi,\mu) \rightarrow I(-s,{}^c \chi^{-1},\mu)
\]
is given by (the meromorphic continuation of) 
\[
J(s,\chi,\mu) f(g) = \int_E \int_{E_0} f(\tl{w}^{-1} n(x,y) g) dy dx,
\]
where $dx$ (\resp $dy$) is the self-dual Haar measure on $E$ (\resp $E_0$) 
with respect to $\psi_E$ (\resp $\psi_F \circ \beta$).
Note that this double integral is absolutely convergent for $\re(s) > 0$ 
and the inner integral is absolutely convergent for $\re(s) > -\half{1}$.
Recall also that the Whittaker functional $\Omega(s,\chi,\mu) = \Omega(\chi|\cdot|_E^s \otimes \mu)$ on $I(s,\chi,\mu)$ 
is given by (the holomorphic continuation of) 
\[
\Omega(s,\chi,\mu) f = \int_E \int_{E_0} f(\tl{w}^{-1} n(b,c)) \psi_E(-b) dc db.
\]
Then the local coefficient $C(s, \chi, \mu, \psi_F) = C_B(\widetilde{w},\chi|\cdot|_F^s \otimes \mu,\psi_F)$ 
is given by 
\[
\Omega(s,\chi,\mu) = C(s,\chi,\mu,\psi_F) \cdot \Omega(-s,{}^c\chi^{-1},\mu) \circ J(s,\chi,\mu) 
\]
(see \cite[p.~333, Theorem 3.1]{Sh2}).

\begin{prop}
\label{SU(2,1)}
We have
\[
C(s,\chi,\mu,\psi_F) = \lambda(E/F, \psi_F)^{-1} 
\gamma(s, \chi \tl{\mu}^{-1}, \psi_E) 
\gamma(2s, \chi_0 \eta_E, \psi_F),
\]
where $\tl{\mu}$ is the base change of $\mu$ to $E^\times$, so that
\[
\tl{\mu}(x) = \mu(x \overline{x}^{-1}),
\]
$\chi_0$ is the restriction of $\chi$ to $F^\times$, 
and $\eta_E$ is the quadratic character of $F^{\times}$ associated to $E/F$ by class field theory.
\end{prop}
\begin{proof}
Fix $\varphi \in \Sc(E)$ such that $\widehat{\varphi} \in C_c^\infty(E^\times)$ and $\phi \in \Sc(E_0)$.
We may define $f \in I(s,\chi,\mu)$ uniquely by requiring that
\[
f(\tl{w}^{-1} n(b,c)) = \varphi(b) \phi(c)
\]
for all $b \in E$ and $c \in E_0$.
Then we have
\[
\Omega(s,\chi,\mu) f = \int_E \int_{E_0} \varphi(b) \phi(c) \psi_E(-b) dc db 
= \widehat{\varphi}(-1) \widehat{\phi}(0).
\]
On the other hand, we have
\[
\tl{w}^{-1} n(x,y) \tl{w}^{-1} = 
\begin{pmatrix}
\overline{a}^{-1} & -a^{-1} x & 1 \\
& a^{-1} \overline{a} & -\overline{x} \\
& & a
\end{pmatrix}
\tl{w}^{-1}
n \left( -\frac{x}{a}, -\frac{y}{a \overline{a}} \right)
\]
for $x \in E^\times$ and $y \in E_0$, where
\[
a = \frac{1}{2} x \overline{x} + y.
\]
For $\re(s) > 0$, we have
\begin{align*}
&J(s,\chi,\mu) f(\tl{w}^{-1} n(b,c)) \\
&= \int_E \int_{E_0} f(\tl{w}^{-1} n(x,y) \tl{w}^{-1} n(b,c)) dy dx \\
&= \int_E \int_{E_0} f\left( m(\overline{a}^{-1}) \tl{w}^{-1} n \left( -\frac{x}{a}, -\frac{y}{a \overline{a}} \right) n(b,c) \right) dy dx \\
 & = |2|_E^{-\frac{1}{2}} \int_{E^\times} \int_{E_0} f\left( m\left( \frac{2}{x \bar{x} \bar{z}} \right) \tl{w}^{-1} n \left( -\frac{2}{\overline{x} z}, -\frac{2y}{x \overline{x}  z \overline{z}} \right) n(b,c) \right) |x|^2_E dy d^\times x, 
\end{align*}
where we have changed the variable $y \mapsto 2^{-1} x \overline{x} y$, 
and we set 
\[
z = 1+y.
\] 
By further changing the variable $x \mapsto - 2 (\overline{xz})^{-1}$, 
the last integral becomes 
\[
|2|_E^{\frac{3}{2}} \int_{E^\times} \int_{E_0} f\left( m\left(\frac{x \overline{x} z}{2} \right) \tl{w}^{-1} n \left( x, -\frac{x \overline{x} y}{2} \right) n(b,c) \right) |xz|^{-2}_E dy d^\times x.
\]
Thus, noting that 
\[
n \left( x, -\frac{x \overline{x} y}{2} \right) n(b,c)
= n \left( b+x, c -\frac{x \overline{x} y}{2} + \frac{1}{2}(\overline{b} x - b \overline{x}) \right),
\]
we have
\begin{align*}
& J(s,\chi,\mu) f(\tl{w}^{-1} n(b,c)) = \chi(2)^{-1} |2|_E^{-s+\frac{1}{2}} \\
\quad & \times \int_{E^\times} \int_{E_0}
\chi(x \overline{x} z) \mu(z^{-1} \overline{z}) |x \overline{x} z|_E^{s+1}
\varphi_b(x)
\phi \left( c -\frac{x \overline{x} y}{2} + \frac{1}{2}(\overline{b} x - b \overline{x}) \right)
|x z|_E^{-2} dy d^\times x,
\end{align*}
where $\varphi_b(x) = \varphi(x+b)$.
For $0 < \re(s) < \frac{1}{2}$, the triple integral 
\[
 \int_{E_0} J(s,\chi,\mu) f(\tl{w}^{-1} n(b,c)) dc  
 = \int_{E_0} \int_{E^\times} \int_{E_0} \cdots dy d^\times x dc
\]
is absolutely convergent.
Interchanging the order of integration and integrating first with respect to $c$, we obtain
\[
\int_{E_0} \phi \left( c -\frac{x \overline{x} y}{2} + \frac{1}{2}(\overline{b} x - b \overline{x}) \right) dc
= \int_{E_0} \phi(c) dc = \widehat{\phi}(0).
\]
Hence, by Lemma \ref{lem:zeta-hks}, we have
\begin{align*}
& \int_{E_0} J(s,\chi,\mu) f(\tl{w}^{-1} n(b,c)) dc \\ 
& = \chi(2)^{-1} |2|_E^{-s+\frac{1}{2}} \widehat{\phi}(0) 
\int_{E^\times} \int_{E_0} \varphi_b(x) \tl{\chi}_0(x) (\chi \tl{\mu}^{-1})(z) |x|_E^{2s} |z|_E^{s-1} dy d^\times x \\
& = \widehat{\phi}(0) 
\frac{\gamma(2s, \chi_0, \psi_F)}{\gamma(s, \chi \tl{\mu}^{-1}, \psi_E)}
\int_{E^\times} \varphi_b(x) \tl{\chi}_0(x) |x|_E^{2s} d^\times x \\
& = \widehat{\phi}(0) 
\frac{\gamma(2s, \chi_0, \psi_F)}{\gamma(s, \chi \tl{\mu}^{-1}, \psi_E)}
Z(2s, \tl{\chi}_0, \varphi_b) \\
& = \widehat{\phi}(0) 
\frac{\gamma(2s, \chi_0, \psi_F)}{\gamma(s, \chi \tl{\mu}^{-1}, \psi_E)}
\frac{Z(1-2s, \tl{\chi}_0, \widehat{\varphi_b})}{\gamma(2s, \tl{\chi}_0, \psi_E)},
\end{align*}
where $\tl{\chi}_0$ is the base change of $\chi_0$ to $E^\times$, so that
\[
\tl{\chi}_0(x) = \chi(x \overline{x}).
\]
Moreover, by Lemma \ref{lem:zeta-fourier} and the property of $\lambda$-factors
\[
\lambda(E/F, \psi_F) \gamma(2s, \tl{\chi}_0, \psi_E) = 
\gamma(2s, \chi_0, \psi_F) \gamma(2s, \chi_0 \eta_E, \psi_F),
\]
we have
\begin{align*}
& \Omega(-s,{}^c\chi^{-1},\mu) J(s,\chi,\mu) f \\
& = \widehat{\phi}(0) 
\frac{\gamma(2s, \chi_0, \psi_F)}{\gamma(s, \chi \tl{\mu}^{-1}, \psi_E) \gamma(2s, \tl{\chi}_0, \psi_E)}
\int_E Z(1-2s, \tl{\chi}_0, \widehat{\varphi_b}) \psi_E(-b) db \\
& = \widehat{\varphi}(-1) \widehat{\phi}(0) 
\frac{\gamma(2s, \chi_0, \psi_F)}{\gamma(s, \chi \tl{\mu}^{-1}, \psi_E) \gamma(2s, \tl{\chi}_0, \psi_E)} \\
& = \widehat{\varphi}(-1) \widehat{\phi}(0) 
\frac{\lambda(E/F, \psi_F)}{\gamma(s, \chi \tl{\mu}^{-1}, \psi_E) \gamma(2s, \chi_0 \eta_E, \psi_F)}
\\&= 
\frac{\lambda(E/F, \psi_F)}{\gamma(s, \chi \tl{\mu}^{-1}, \psi_E) \gamma(2s, \chi_0 \eta_E, \psi_F)}
\Omega(s,\chi,\mu) f.
\end{align*}
Choosing $\varphi$ and $\phi$ such that $\widehat{\varphi}(-1) \widehat{\phi}(0) \not= 0$
and hence $\Omega(s,\chi,\mu) f \not= 0$,
this together with the definition of $C(s,\chi,\mu,\psi_F)$ yields the desired assertion
\[
C(s,\chi,\mu,\psi_F) = \lambda(E/F, \psi_F)^{-1} \gamma(s, \chi \tl{\mu}^{-1}, \psi_E) \gamma(2s, \chi_0 \eta_E, \psi_F).
\]
This completes the proof of Proposition \ref{SU(2,1)}.
\end{proof}

\section{The twisted local intertwining relation}\label{LIR_GL}
The purpose of this section is to prove Theorem \ref{main2}. 
This theorem is stated in \cite[Theorem 2.5.3]{Ar} and \cite[Proposition 3.5.1 (b)]{Mok}.
Arthur expected that this theorem (for non-tempered representations) 
would be proven by an argument ``based on some version of minimal $K$-types''.
However, this idea might require a huge amount of computation 
even if $F$ is non-archimedean and $\pi_\psi$ is unramified.
\par

To show Theorem \ref{main2}, we shall use a new approach. 
The difficulty of this theorem is that 
the linear isomorphism $\theta_A \colon \pi_\psi \xrightarrow{\sim} \pi_\psi$
is defined through the Langlands quotient map from the standard module of $\pi_\psi$. 
Our idea is to realize this Langlands quotient map as a composition of normalized intertwining operators
(see Lemma \ref{langlands} below).
Then we can show Theorem \ref{main2} by the multiplicativity of normalized intertwining operators 
(Proposition \ref{multiplicative}). 

\subsection{Representations of general linear groups}\label{sec.GL}
Throughout this section, we write $G = \GL_N(E)$. 
For $\tau \in \Rep(G)$ and for a character $\chi$ of $E^\times$, 
we define $\tau\chi$ by $(\tau\chi)(g) = \tau(g)\chi(\det(g))$.
\par

Let $P = MN$ be a standard parabolic subgroup of $G$ with Levi subgroup 
$M \cong \GL_{k_1}(E) \times \dots \times \GL_{k_t}(E)$. 
For $\tau_i \in \Rep(\GL_{k_i}(E))$, 
we denote the normalized parabolic induction by
\[
\tau_1 \times \dots \times \tau_t = \Ind_P^G(\tau_1 \boxtimes \dots \boxtimes \tau_t). 
\]
It follows from Bernstein \cite{Ber-P} that if $\pi_M \in \Irr(M)$ is unitary, then $I_P(\pi_M) = \Ind_P^G(\pi_M)$
is an irreducible unitary representation of $G$. 
\par

Recall that a \emph{standard module} of $G$ is an induced representation of the form
\[
\tau_1|\cdot|_E^{e_1} \times \dots \times \tau_t|\cdot|_E^{e_t}, 
\]
where $\tau_i$ is an irreducible tempered representation of $\GL_{k_i}(E)$
and $e_i \in \R$ such that $e_1 > \dots > e_t$.
It has a unique irreducible quotient, called the \emph{Langlands quotient}. 
For example, if $\psi = \phi \boxtimes S_{a}$ is an $A$-parameter for $G$
with $\phi$ an irreducible representation of $L_E$ with $\phi(W_E)$ bounded, 
then $\pi_\phi$ is discrete series, and 
$\pi_\psi$ is the Langlands quotient of the standard module
\[
\II_\psi^G = 
\pi_{\phi}|\cdot|_E^{\half{a-1}} \times \pi_{\phi}|\cdot|_E^{\half{a-3}} \times \dots \times \pi_{\phi}|\cdot|_E^{-\half{a-1}}. 
\]
In this situation, we write 
\[
\pi_\psi = \Speh(\pi_\phi, a)
\]
and call it a \emph{Speh representation}.
More generally, 
if $\psi = \oplus_{i=1}^t \phi_i \boxtimes S_{a_i}$ is an irreducible decomposition of an $A$-parameter for $G$, 
then
\[
\pi_\psi = \bigtimes_{i=1}^t \Speh(\pi_{\phi_i}, a_i)
\]
and its standard module is 
\[
\II_\psi^G = \bigtimes_{i=1}^t \bigtimes_{e_i = 1}^{a_i} \pi_{\phi_i}|\cdot|_E^{\half{a_i+1}-e_i}, 
\]
where the product is taken in decreasing order of the exponents. 
\par

\begin{lem}\label{End}
Let $\Pi$ be a standard module of $G$. 
Then 
\[
\dim_\C (\End_G(\Pi)) = 1.
\]
\end{lem}
\begin{proof}
This is a consequence of the famous fact that 
the Langlands quotient $\pi$ of $\Pi$ appears in $\Pi$ as a subquotient with multiplicity one
(See, e.g., \cite[Chapter XI, Lemma 2.13]{BW}).
This fact induces an injective linear map
\[
\End_G(\Pi) \rightarrow \End_G(\pi).
\]
Since $\dim_\C (\End_G(\pi)) = 1$ by Schur's lemma, we obtain the assertion.
\end{proof}

For the rest of this subsection, we assume that $E$ is non-archimedean. 
We will identify an irreducible unitary supercuspidal representation $\rho$ of $\GL_d(E)$
with the irreducible $d$-dimensional bounded representation of $W_E$ by the LLC for $\GL_d(E)$. 
\par

Recall that a \emph{segment} is a set of the form 
\[
[x,y]_\rho = \{\rho|\cdot|_E^x, \rho|\cdot|_E^{x+1}, \dots, \rho|\cdot|_E^y\},
\]
where $\rho$ is an irreducible unitary supercuspidal representation of $\GL_d(E)$
and $x \leq y$ are real numbers such that $x \equiv y \bmod \Z$.
One can attach to it 
two irreducible representations $\Delta([x,y]_\rho)$ and $Z([x,y]_\rho)$ of $\GL_{d(y-x+1)}(E)$, 
which are the unique irreducible subrepresentation and the unique irreducible quotient of the standard module
\[
\rho|\cdot|_E^y \times \dots \times \rho|\cdot|_E^{x+1} \times \rho|\cdot|_E^x,
\]
respectively. 
We call $\Delta([x,y]_\rho)$ a \emph{(generalized) Steinberg representation}. 
Note that $\Delta([x,y]_\rho)$ is an essentially discrete series representation, 
and all essentially discrete series representations are of this form (see \cite[Theorem 9.3]{Z}).
Similarly, any irreducible tempered representation is a product of 
representations of the form $\Delta([-x,x]_\rho)$.
On the other hand, by definition, we have $Z([-x,x]_\rho) = \pi_{\rho \boxtimes S_1 \boxtimes S_{2x+1}}$ if $2x \in \Z$.
Recall from \cite[Theorem 9.7]{Z} that the following are equivalent: 
\begin{itemize}
\item
The induction $\Delta([x,y]_\rho) \times \Delta([x',y']_{\rho'})$ is reducible; 
\item
$\Delta([x,y]_\rho) \times \Delta([x',y']_{\rho'}) \not\cong \Delta([x',y']_{\rho'}) \times \Delta([x,y]_\rho)$; 
\item
the two segments $[x,y]_\rho$ and $[x',y']_{\rho'}$ are \emph{linked}, i.e., 
$[x,y]_\rho \cup [x',y']_{\rho'}$ is also a segment, 
and $[x,y]_\rho \not\subset [x',y']_{\rho'}$ and $[x,y]_\rho \not\supset [x',y']_{\rho'}$.
\end{itemize}
\par

A \emph{multi-segment} is a formal finite sum of segments.
For a multi-segment $\mm$, 
writing $\mm = [x_1,y_1]_{\rho_1} + \dots + [x_r,y_r]_{\rho_r}$ with $x_1+y_1 \geq \dots \geq x_r+y_r$, 
we set 
\[
\II(\mm) = \Delta([x_1,y_1]_{\rho_1}) \times \dots \times \Delta([x_r,y_r]_{\rho_r}).
\]
This is a standard module. 
For example, 
let $\psi = \rho \boxtimes S_{2\alpha+1} \boxtimes S_{2\beta+1}$ be an $A$-parameter for $G = \GL_N(E)$. 
If we set
\[
\mm = [-\alpha+\beta,\alpha+\beta]_{\rho} + [-\alpha+\beta-1,\alpha+\beta-1]_{\rho} 
+ \dots + [-\alpha-\beta,\alpha-\beta]_{\rho}, 
\]
then the standard module $\II_\psi^G$ of $\pi_{\psi}$ is equal to $\II(\mm)$. 

\begin{lem}\label{psiD}
Let $\psi$ be an $A$-parameter for $G = \GL_N(E)$. 
Then the standard module $\II_\psi^G$ of $\pi_\psi$ 
contains an irreducible tempered representation $\pi_{\psi^D}$ as a subrepresentation, 
where $\psi^D \colon W_E \times \SL_2(\C) \rightarrow \GL_N(\C)$ is given by 
$\psi^D(w,\alpha) = \psi(w, \alpha,\alpha)$.
Moreover, $\pi_{\psi^D}$ appears in $\II_\psi^G$ as a subquotient with multiplicity one.
\end{lem}
\begin{proof}
By \cite{JS}, 
the unique irreducible subrepresentation of $\II^G_\psi$ is generic. 
By \cite[Theorem 9.7]{Z}, it is an irreducible product of generalized Steinberg representations. 
In particular, it is determined uniquely by its cuspidal support.
As this cuspidal support is the same as that of $\pi_\psi$, 
we deduce that the unique irreducible subrepresentation of $\II^G_\psi$ is $\pi_{\psi^D}$.
Finally, the multiplicity one statement follows from \cite[Proposition 8.4]{Z}.
\end{proof}

\subsection{The tempered case}
Let $G = \GL_N(E)$ with an involution $\theta$ defined in Section \ref{sec.thetaA}. 
Fix a standard parabolic subgroup $P=MN_P$. 
In this subsection, 
we will prove Theorem \ref{main2} for tempered representations, i.e., for generic (or tempered) $\psi$. 
Thus we consider an irreducible tempered representation $\pi$ of $M$, 
and $w \in W(\theta(M),M)$ such that $w(\pi \circ \theta) \cong \pi$. 

\par
Since $\pi$ is tempered, it is $\ww_M$-generic. 
Fix a non-trivial $\ww_M$-Whittaker functional $\omega$ on $\pi$. 
Then $I_P(\pi)$ is an irreducible $\ww$-generic representation of $G$ 
with the $\ww$-Whittaker functional $\Omega(\pi)$ induced by $\omega$ as in Section \ref{sec.A27}. 
By definition (see Section \ref{sec.thetaA}), $\theta_A = \theta_W$ is the unique linear isomorphism 
$\theta_A \colon I_P(\pi) \xrightarrow{\sim} I_P(\pi)$
such that
\begin{align*}
\theta_A \circ I_P(\pi)(h) &= I_P(\pi)(\theta(h)) \circ \theta_A, \quad h \in \GL_N(E), \\
\Omega(\pi) \circ \theta_A &= \Omega(\pi).
\end{align*}
In Section \ref{sec.main2}, we already argued that 
$\tl{R}_P(\theta \circ w, \tl{\pi})$ is a constant multiple of $\theta_A$, 
so the equation $\tl{R}_P(\theta \circ w, \tl{\pi}) = \theta_A$
would follow from $\Omega(\pi) \circ \tl{R}_P(\theta \circ w, \tl{\pi}) = \Omega(\pi)$. 
\par

Recall from Section \ref{sec.main2} that 
\[
\tl{R}_P(\theta \circ w, \tl{\pi}) 
= 
I_P(\tl\pi(w \rtimes \theta)) \circ R_{\theta(P)}(w, \pi \circ \theta) \circ \theta^*. 
\]
So it suffices to check that the three squares in the diagram 
\[
\begin{CD}
I_P(\pi) @>\theta^*>> I_{\theta(P)}(\pi \circ \theta)
@>{R_{\theta(P)}(w, \pi \circ \theta)}>> I_{P}(w(\pi \circ \theta))
@>{I_P(\tl\pi(w \rtimes \theta))}>> I_P(\pi) \\
@V\Omega(\pi)VV @V\Omega(\pi \circ \theta)VV 
@VV\Omega(w(\pi \circ \theta))V @VV\Omega(\pi)V\\
\C @= \C @= \C @= \C
\end{CD}
\]
are all commutative.
Here, we note that all these $\Omega$ are induced by the same linear functional $\omega$.
The commutativity of the middle square is nothing but Theorem \ref{main1} (2), 
whereas the one for the right square follows from $\omega \circ \tl\pi(w \rtimes \theta) = \omega$, 
which is the definition of the normalization of $\tl\pi(w \rtimes \theta)$.
\par

We show the commutativity of the left square. 

\begin{lem}\label{tits}
For $w \in W^G$, we have 
\[
\theta(\tl{w}) =\tl{\theta(w)}.
\]
\end{lem}
\begin{proof}
We recall the definition of the Tits lifting $\tl{w} \in G$ of $w \in W^G$.
If $w = w_{\alpha_1} \cdots w_{\alpha_k}$ is a reduced decomposition relative to the simple roots of $(G,T)$, 
where $w_\alpha$ is a simple reflection with respect to a simple root $\alpha$, 
then $\tl{w}$ is defined by $\tl{w} = \tl{w}_{\alpha_1} \cdots \tl{w}_{\alpha_k}$. 
Hence we may assume that $w = w_\alpha$ for some simple root $\alpha$. 
Then $\tl{w}_\alpha$ is defined by
\[
\tl{w}_\alpha = \exp(X_\alpha) \exp(-X_{-\alpha}) \exp(X_\alpha), 
\]
where $X_\alpha$ is already given as we fix a splitting $\spl$, and 
$X_{-\alpha}$ is the root vector for $-\alpha$ such that 
$H_\alpha = [X_\alpha, X_{-\alpha}]$ is the coroot for $\alpha$.
Since $\spl$ is $\theta$-stable, the claim follows.
\end{proof}

Recall that $\Omega(\pi)$ is defined by the holomorphic continuation of the Jacquet integral
\[
\Omega(\pi_\lambda) f = \int_{N'} \omega(f(\tl{w}_0^{-1} n')) \chi(n')^{-1} dn'.
\]
Here $N' = \tl{w}_0 \overline{N} \tl{w}_0^{-1}$ with $w_0 = w_\ell w_\ell^M$, 
where $w_\ell$ and $w_\ell^M$ are the longest elements in $W^{G}$ and $W^{M}$, respectively, 
and $\chi$ is (the restriction of) the non-degenerate character of 
the unipotent radical $U$ of the Borel subgroup $B$
given by $\spl$ and $\psi_F$. 
Note that $\theta(w_\ell) = w_\ell$, $\theta(w_\ell^M) = w_\ell^{\theta(M)}$ 
and $\chi \circ \theta = \chi$.
Hence, Lemma \ref{tits} implies that 
$\Omega(\pi \circ \theta) \circ \theta^* = \Omega(\pi)$. 
This completes the proof of Theorem \ref{main2} for tempered representations. 

\subsection{Construction of the Langlands quotient map}
Fix a standard parabolic subgroup $P=MN_P$ of $G = \GL_N(E)$.
Let $\psi$ be an $A$-parameter for $M$, 
and let $\pi_\psi$ be the associated irreducible unitary representation of $M$.
Assume that there is $w \in W(\theta(M),M)$ such that $w(\pi_\psi \circ \theta) \cong \pi_\psi$, 
and we fix such an element $w$ in this and next subsections.
\par

Note that $I_P(\pi_\psi)$ is irreducible. 
Let $\II_\psi^G$ be a standard module of $G$ whose Langlands quotient is $I_P(\pi_\psi)$. 
Recall from Section \ref{sec.thetaA} that 
a Whittaker functional $\Omega$ on $\II_\psi^G$ defines 
a linear isomorphism $\theta_W \colon \II_\psi^G \xrightarrow{\sim} \II_\psi^G$, 
which induces $\theta_A \colon \pi_\psi \xrightarrow{\sim} \pi_\psi$. 
To show Theorem \ref{main2} for $\pi_\psi$, 
we shall carefully construct $\II_\psi^G$.
\par

Set $V = E^N$ so that $G = \GL(V)$. 
Decompose $V$ into a direct sum
\[
V = V^{(1)} \oplus \dots \oplus V^{(t)}
\]
such that $M$ is the subgroup of $G$ stabilizing $V^{(i)}$ for $i = 1,\dots,t$.
Hence $M \cong G^{(1)} \times \dots \times G^{(t)}$, where $G^{(i)} = \GL(V^{(i)})$. 
\par

Recall that $\psi$ is an $A$-parameter for $M$.
It can be decomposed as 
\[
\psi = \psi^{(1)} \oplus \dots \oplus \psi^{(t)}, 
\]
where $\psi^{(i)}$ is an $A$-parameter for $G^{(i)}$.
Consider the decomposition of $\psi^{(i)}$ into irreducible representations: 
\[
\psi^{(i)} = \bigoplus_{j=1}^{m_i} \psi_j^{(i)}.
\]
Corresponding to this decomposition, 
we can also decompose $V^{(i)}$ as
\[
V^{(i)} = \bigoplus_{j=1}^{m_i} V_j^{(i)}
\]
such that $\dim_\C(\psi_j^{(i)}) = \dim_E(V_j^{(i)})$.
\par

Define a representation $\phi_{\psi_j^{(i)}}$ of $L_E$ by 
\[
\phi_{\psi_j^{(i)}}(w) = \psi_j^{(i)}
\left(w, \begin{pmatrix}
|w|_E^{\half{1}} & 0 \\ 0 & |w|_E^{-\half{1}}
\end{pmatrix} \right).
\]
If we write $\psi_j^{(i)} = \phi_j^{(i)} \boxtimes S_d$, 
where $\phi_j^{(i)}$ is an irreducible representation of $L_E$
with $\phi_j^{(i)}(W_E)$ bounded, and $d = d_j^{(i)} \geq 1$, 
then we have
\[
\phi_{\psi_j^{(i)}} = \bigoplus_{k=1}^{d} \phi_j^{(i)}|\cdot|_E^{\half{d+1}-k}.
\]
Corresponding to this decomposition, 
we can decompose $V_j^{(i)}$ as
\[
V_j^{(i)} = \bigoplus_{k=1}^{d} V_j^{(i)}\left(\half{d+1}-k\right)
\]
such that $\dim_E(V_j^{(i)}(\half{d+1}-k)) = \dim_\C(\phi_j^{(i)}|\cdot|_E^{\half{d+1}-k}) = \dim_\C(\phi_j^{(i)})$.
Fix a linear isomorphism $V_j^{(i)}(\half{d+1}-k) \cong V_j^{(i)}(\half{d+1}-k')$ for $1 \leq k,k' \leq d$.
When $\alpha \in (1/2)\Z$ satisfies 
$\alpha \not\equiv \half{d+1} \bmod \Z$ or $|\alpha| \geq \half{d+1}$, 
we formally set $V_j^{(i)}(\alpha) = 0$.
\par

Finally, we define 
\[
V^{(i)}(\alpha) = \bigoplus_{j=1}^{m_i} V_j^{(i)}(\alpha).
\]
After all, we obtain a decomposition 
\[
V = \bigoplus_{i=1}^t \bigoplus_{\alpha \in (1/2)\Z} V^{(i)}(\alpha). 
\]
Consider the finite set
\[
\VV = \left\{ V^{(i)}(\alpha) \,\middle|\, 
1 \leq i \leq t,\, \alpha \in (1/2)\Z
\text{ such that } V^{(i)}(\alpha) \not= 0 \right\}.
\]
We can define two total orders $\prec_1$ and $\prec_2$ on $\VV$ as follows.
Firstly, $V^{(i)}(\alpha) \prec_1 V^{(i')}(\alpha')$ if and only if 
\begin{itemize}
\item
$i < i'$; or
\item
$i=i'$ and $\alpha > \alpha'$.
\end{itemize}
Secondly, $V^{(i)}(\alpha) \prec_2 V^{(i')}(\alpha')$ if and only if 
\begin{itemize}
\item
$\alpha > \alpha'$; or
\item
$\alpha = \alpha'$ and $i < i'$.
\end{itemize}
If we write $\VV = \{V_1, \dots, V_r\} = \{V'_1, \dots, V'_r\}$ 
with $V_1 \prec_1 \dots \prec_1 V_r$ and $V'_1 \prec_2 \dots \prec_2 V'_r$, 
we define two parabolic subgroups $P_1 = M_1 N_{P_1}$ and $P'_1 = M_1 N_{P'_1}$ as 
the stabilizers of the flags
\[
V_1 \subset V_1 \oplus V_2 \subset \dots \subset V_1 \oplus \dots \oplus V_r
\]
and 
\[
V'_1 \subset V'_1 \oplus V'_2 \subset \dots \subset V'_1 \oplus \dots \oplus V'_r, 
\]
respectively. 
Here, $M_1$ is the stabilizer of elements $V_i \in \VV$, 
which is a common Levi subgroup of $P_1$ and $P_1'$. 
Note that $P_1$ is contained in $P$.
We may assume that $P_1$ is standard, but $P_1'$ is not standard in general.
Let $w_2 \in W^{G}$ be such that 
$\tl{w}_2^{-1}P_1'\tl{w}_2 = P_2=M_2N_{P_2}$ is a standard parabolic subgroup, 
where $\tl{w}_2M_2\tl{w}_2^{-1} = M_1$.
We regard $w_2$ as an element in $W(M_2,M_1)$. 
\par

Let $\tau_j^{(i)}$ be the irreducible discrete series representation corresponding to $\phi_j^{(i)}$.
We regard $\tau_j^{(i)}|\cdot |_E^\alpha$
as a representation of $\GL(V_j^{(i)}(\alpha))$ if $V_j^{(i)}(\alpha) \not= 0$.
(If $V_j^{(i)}(\alpha) = 0$, then we interpret $\tau_j^{(i)}|\cdot|_E^\alpha$ 
to be the trivial representation of the trivial group $\GL(V_j^{(i)}(\alpha)) = \GL_0(E)$.)
By induction, we obtain an essentially tempered representation 
\[
\tau^{(i)}|\cdot|_E^\alpha = \bigtimes_{j=1}^{m_i} \tau_j^{(i)}|\cdot|_E^\alpha
\] 
of $\GL(V^{(i)}(\alpha))$.
We consider an essentially tempered representation 
\[
\tau_1 = \bigotimes_{i,\alpha} \tau^{(i)}|\cdot|_E^\alpha
\]
of $M_1$, 
and set 
\[
\tau_2 = w_2^{-1}\tau_1,
\]
which is an essentially tempered representation of $M_2$.
\par

Note that $P_1 \subset P$ and $M_1 \subset M$. 
Recall that $\pi_\psi \in \Irr(M)$ is the representation corresponding to $\phi_\psi$ via the LLC for $M$. 
It is the Langlands quotient of the standard module
\[
\II_\psi^M = \Ind_{P_1 \cap M}^{M}\left( \tau_1 \right)
\]
of $M$. 
In particular, $I_P(\pi_\psi)$ is a quotient of 
$I_P(\II_\psi^M) = I_{P_1}(\tau_1)$. 
Moreover, there is $w_1 \in W(M_1)$ with $\tl{w}_1 \in M$ such that 
the image of the normalized intertwining operator 
\[
R_{P_1}(w_1, \tau_1) \colon I_{P_1}(\tau_1) \rightarrow I_{P_1}(w_1\tau_1)
\] 
is isomorphic to $I_P(\pi_\psi)$. 
Hereafter, we identify $I_P(\pi_\psi)$ with this subspace, 
i.e., $I_P(\pi_\psi)$ is realized as the image of $R_{P_1}(w_1, \tau_1)$. 
Hence we obtain a surjection 
\[
R_{P_1}(w_1, \tau_1) \colon I_{P_1}(\tau_1) \twoheadrightarrow I_P(\pi_\psi) \subset I_{P_1}(w_1\tau_1).
\]
\par

On the other hand, 
as the unitary induction preserves the irreducibility for general linear groups, 
we see that $I_P(\pi_\psi)$ is the irreducible representation of $G$, 
which corresponds to $\phi_\psi$, regarded as an $L$-parameter of $G$, via the LLC for $G$. 
It is the Langlands quotient of the standard module $\II_\psi^G = I_{P_2}(\tau_2)$. 
Since $\tau_1 = w_2 \tau_2$, 
we have a normalized intertwining operator 
\[
R_{P_2}(w_2, \tau_2) \colon I_{P_2}(\tau_2) \rightarrow I_{P_1}(\tau_1).
\]

\begin{lem}\label{langlands}
The composition 
\[
I_{P_2}(\tau_2) \xrightarrow{R_{P_2}(w_2, \tau_2)} I_{P_1}(\tau_1)
\xrightarrow{R_{P_1}(w_1, \tau_1)} I_P(\pi_\psi)
\]
is well-defined and nonzero.
In particular, this composition realizes the Langlands quotient map. 
Namely, $I_{P_2}(\tau_2)$ is the standard module of $I_P(\pi_\psi)$, and 
the above composition is surjective.

\end{lem}
\begin{proof}
Recall that the normalized intertwining operators are defined 
by the meromorphic continuation of certain (normalized) integrals.
Note that the operators $R_{P_1}(w_1, \tau_1)$ and $R_{P_2}(w_2, \tau_2)$
are compositions of normalized intertwining operators of the form 
\[
\tau_j^{(i)}|\cdot|_E^\alpha \times \tau_{j'}^{(i')}|\cdot|_E^{\alpha'} 
\rightarrow 
\tau_{j'}^{(i')}|\cdot|_E^{\alpha'} \times \tau_j^{(i)}|\cdot|_E^\alpha
\]
with $\alpha \geq \alpha'$.
These displayed operators are regular and nonzero at the relevant points. 
Hence, $R_{P_1}(w_1, \tau_1)$ and $R_{P_2}(w_2, \tau_2)$ are well-defined. 
\par

We now verify that the composite $R_{P_1}(w_1, \tau_1) \circ R_{P_2}(w_2, \tau_2)$ is nonzero. 
Since $I_P(\pi_\psi)$ is the unique irreducible quotient of $I_{P_2}(\tau_2)$, 
and since $R_{P_2}(w_2, \tau_2)$ is nonzero, 
$I_P(\pi_\psi)$ appears in the image of $R_{P_2}(w_2, \tau_2)$.
If $R_{P_1}(w_1, \tau_1) \circ R_{P_2}(w_2, \tau_2) = 0$, 
then we would conclude that $I_P(\pi_\psi)$ appears in $I_{P_1}(\tau_1)$ 
as a subquotient with multiplicity greater than one. 
Since the semisimplification of $I_{P_1}(\tau_1)$ is 
the same as that of the standard module $I_{P_2}(\tau_2)$, 
this contradicts the fact that the Langlands quotient $I_P(\pi_\psi)$ appears in $I_{P_2}(\tau_2)$
with multiplicity one.
\end{proof}

\subsection{The main diagram}
Recall that 
we have fixed an element $w \in W(\theta(M),M)$ such that $w(\pi_\psi \circ \theta) \cong \pi_\psi$.
As in Section \ref{sec.NIO}, we regard $w$ as an element of $W^G$. 
By construction, $w\theta(M_1)w^{-1} = M_1$ 
and $w(\tau_1 \circ \theta) \cong \tau_1$.
Since $w_1$ is the longest element in the subset of $W(M_1)$ consisting of elements whose representatives are in $M$, 
we see that $\theta(w_1) = w^{-1} w_1 w$ in $W(\theta(M))$.
Set 
\[
w' = w_2^{-1} w_1^{-1} w \theta(w_1) \theta(w_2) = w_2^{-1} w \theta(w_2).
\]
\par

\begin{lem}\label{www}
With the above notations, we have the following.
\begin{enumerate}
\item
The canonical inclusion $N_{P} \hookrightarrow N_{P_1}$ induces 
a homeomorphism 
\[
N_P \cap \tl{w} N_{\theta(P)} \tl{w}^{-1} \bs N_P
\cong 
N_{P_1} \cap \tl{w} N_{\theta(P_1)} \tl{w}^{-1} \bs N_{P_1}.
\]

\item
We have $w' \theta(M_2) w'^{-1} = M_2$ and $w'(\tau_2 \circ \theta) \cong \tau_2$.
Let $\tl\tau_2(w' \rtimes \theta) \colon w'(\tau_2 \circ \theta) \xrightarrow{\sim} \tau_2$
be the isomorphism normalized by using a Whittaker functional on $\tau_2$.

\item
The normalized intertwining operator 
$\tl{R}_{P_2}(\theta \circ w', \tl\tau_2) \colon I_{P_2}(\tau_2) \rightarrow I_{P_2}(\tau_2)$
defined by the composition
\[
I_{P_2}(\tau_2) \xrightarrow{\theta^*} I_{\theta(P_2)}(\tau_2 \circ \theta)
\xrightarrow{R_{\theta(P_2)}(w', \tau_2 \circ \theta)} I_{P_2}(w'(\tau_2 \circ \theta))
\xrightarrow{I_{P_2}(\tl\tau_2(w' \rtimes \theta))} I_{P_2}(\tau_2)
\]
is bijective.
\end{enumerate}
\end{lem}
\begin{proof}
Assertion (1) follows from the equation
$\tl{w} N_{\theta(P_1)} \tl{w}^{-1} \cap M = N_{P_1} \cap M$.
\par

For (2), the first assertion follows by direct computation
\begin{align*}
w' \theta(M_2) w'^{-1} 
&= 
w_2^{-1} w \theta(w_2) \theta(M_2) \theta(w_2)^{-1} w^{-1} w_2
\\&= 
w_2^{-1} w \theta(M_1) w^{-1} w_2
\\&= 
w_2^{-1} M_1 w_2 = M_2. 
\end{align*}
The second assertion is proven similarly.
\par

For (3),
it is obvious that $\theta^* \colon I_{P_2}(\tau_2) \rightarrow I_{\theta(P_2)}(\tau_2 \circ \theta)$ is bijective. 
Note that $I_{\theta(P_2)}(\tau_2 \circ \theta)$ is a standard module of $G$
whose Langlands quotient is $I_{\theta(P)}(\pi_\psi \circ \theta) \cong I_P(\pi_\psi)$.
Hence $I_{\theta(P_2)}(\tau_2 \circ \theta)$ and $I_{P_2}(\tau_2)$
are standard modules whose Langlands quotients are isomorphic to each other. 

Since $\tau_2 \circ \theta$ and $w'(\tau_2 \circ \theta) \cong \tau_2$ 
are essentially tempered representations of $\theta(M_2)$ and $M_2$, respectively, 
such that two inductions $I_{\theta(P_2)}(\tau_2 \circ \theta)$ and $I_{P_2}(w'(\tau_2 \circ \theta))$ are both standard modules, 
we see that $R_{\theta(P_2)}(w', \tau_2 \circ \theta)$ is nonzero.
Since $I_{P_2}(\tl\tau_2(w' \rtimes \theta)) \circ R_{\theta(P_2)}(w', \tau_2 \circ \theta)$ is a nonzero element in 
$\Hom_G(I_{\theta(P_2)}(\tau_2 \circ \theta), I_{P_2}(\tau_2))$, 
which is one dimensional by Lemma \ref{End}, 
it must be bijective.
\end{proof}

Now we will prove a key result: 
\begin{thm}\label{diagramGL}
The \emph{``main diagram''} 
\[
\begin{CD}
I_{P_2}(\tau_2) @> \tl{R}_{P_2}(\theta \circ w', \tl\tau_2) >> I_{P_2}(\tau_2)\\
@VR_{P_2}(w_2, \tau_2)VV @VVR_{P_2}(w_2, \tau_2)V\\
I_{P_1}(\tau_1) @. I_{P_1}(\tau_1)\\
@VR_{P_1}(w_1, \tau_1)VV @VV R_{P_1}(w_1, \tau_1)V \\
I_{P}(\pi_\psi) @> \tl{R}_{P}(\theta \circ w, \tl\pi_\psi) >> I_{P}(\pi_\psi)
\end{CD}
\]
is commutative. 
\end{thm}

Admitting this result, we can complete the proof of Theorem \ref{main2} as follows. 
Recall that $\II_\psi^G = I_{P_2}(\tau_2)$ is the standard module of $I_P(\pi_\psi)$. 
Moreover, by Theorem \ref{main2} for the tempered case together with analytic continuation, 
we have 
\[
\tl{R}_{P_2}(\theta \circ w', \tl\tau_2) = \theta_W. 
\] 
By the definition of $\theta_A$ (see Section \ref{sec.thetaA}), 
Theorem \ref{diagramGL} together with Lemma \ref{langlands} implies that 
\[
\tl{R}_P(\theta \circ w, \tl\pi_\psi) = \theta_A. 
\]
This completes the proof of Theorem \ref{main2} in general.
\par

Now we show Theorem \ref{diagramGL}. 
\begin{proof}[Proof of Theorem \ref{diagramGL}]
Recall that the top and bottom maps of the main diagram are composites of three maps: 
\[
\tl{R}_{P_2}(\theta \circ w', \tl\tau_2)
=
I_{P_2}(\tl\tau_2(w' \rtimes \theta)) \circ R_{\theta(P_2)}(w', \tau_2 \circ \theta) \circ \theta^*
\]
and
\[
\tl{R_P}(\theta \circ w, \tl\pi_\psi) 
= I_P(\tl\pi_\psi(w \rtimes \theta)) \circ R_{\theta(P)}(w, \pi_\psi \circ \theta, \psi) \circ \theta^*.
\] 
Hence the main diagram written in its full glory has the form 
\[
\begin{CD}
I_{P_2}(\tau_2) @>\theta^*>> I_{\theta(P_2)}(\tau_2 \circ \theta) 
@>R_{\theta(P_2)}(w', \tau_2 \circ \theta)>> 
I_{P_2}(w' (\tau_2 \circ \theta)) @>I_{P_2}(\tl\tau_2(w' \rtimes \theta))>> I_{P_2}(\tau_2) \\
@VVR_{P_2}(w_2, \tau_2)V @. @. @VR_{P_2}(w_2, \tau_2)VV\\
I_{P_1}(\tau_1) @. @. @. I_{P_1}(\tau_1)\\
@VVR_{P_1}(w_1, \tau_1)V @. @. @V R_{P_1}(w_1, \tau_1)VV\\
I_{P}(\pi_\psi) @>\theta^*>> I_{\theta(P)}(\pi_\psi \circ \theta)
@>R_{\theta(P)}(w, \pi_\psi \circ \theta, \psi)>> 
I_{P}(w(\pi_\psi \circ \theta)) @>I_P(\tl\pi_\psi(w \rtimes \theta))>> I_{P}(\pi_\psi).
\end{CD}
\]
We shall enhance this diagram by introducing additional stepping stones in the second row, 
namely by introducing the representations
\[
I_{\theta(P_1)}(\tau_1 \circ \theta)
\quad\text{and}\quad
 I_{P_1}(w(\tau_1 \circ \theta)),
\]
and additional maps connecting them to their neighbors. 
Hence the enhanced diagram has the form:
\[
\begin{CD}
I_{P_2}(\tau_2) @>\theta^*>> I_{\theta(P_2)}(\tau_2 \circ \theta) 
@>R_{\theta(P_2)}(w', \tau_2 \circ \theta)>> 
I_{P_2}(w' (\tau_2 \circ \theta)) @>I_{P_2}(\tl\tau_2(w' \rtimes \theta))>> I_{P_2}(\tau_2) \\
@VVV @VVV @VVV @VVV \\
I_{P_1}(\tau_1) @>\theta^*>> I_{\theta(P_1)}(\tau_1 \circ \theta) 
@. I_{P_1}(w(\tau_1 \circ \theta)) @>I_{P_1}(\tl\tau_1(w \rtimes \theta))>> I_{P_1}(\tau_1)\\
@VVV @VVV @VVV @VVV \\
I_{P}(\pi_\psi) @>\theta^*>> I_{\theta(P)}(\pi_\psi \circ \theta)
@>R_{\theta(P)}(w, \pi_\psi \circ \theta, \psi)>> 
I_{P}(w(\pi_\psi \circ \theta)) @>I_P(\tl\pi_\psi(w \rtimes \theta))>> I_{P}(\pi_\psi).
\end{CD}
\]
Here, the vertical maps are of the form $R_{P_*}(w_*, \tau_*)$, see below for the details.
To prove the commutativity of the main diagram, 
we shall show that the three vertical rectangles are commutative. 
\par

The commutativity of the left rectangle
\[
\begin{CD}
I_{P_2}(\tau_2) @>\theta^*>> I_{\theta(P_2)}(\tau_2 \circ \theta) \\
@VR_{P_2}(w_2, \tau_2)VV @VVR_{\theta(P_2)}(\theta(w_2), \tau_2 \circ \theta)V \\
I_{P_1}(\tau_1) @>\theta^*>> I_{\theta(P_1)}(\tau_1 \circ \theta) \\
@VR_{P_1}(w_1, \tau_1)VV @VVR_{\theta(P_1)}(\theta(w_1), \tau_1 \circ \theta)V \\
I_{P}(\pi_\psi) @>\theta^*>> I_{\theta(P)}(\pi_\psi \circ \theta)
\end{CD}
\]
follows easily from Lemma \ref{tits}. 
\par

Now let us consider the right rectangle:
\[
\begin{CD}
I_{P_2}(w' (\tau_2 \circ \theta)) @>I_{P_2}(\tl\tau_2(w' \rtimes \theta))>> I_{P_2}(\tau_2) \\
@VR_{P_2}(w_2, w'(\tau_2 \circ \theta))VV @VVR_{P_2}(w_2, \tau_2)V \\
I_{P_1}(w(\tau_1 \circ \theta)) @>I_{P_1}(\tl\tau_1(w \rtimes \theta))>> I_{P_1}(\tau_1) \\
@VR_{P_1}(w_1, w(\tau_1 \circ \theta))VV @VVR_{P_1}(w_1, \tau_1)V \\
I_{P}(w(\pi_\psi \circ \theta)) @>I_P(\tl\pi_\psi(w \rtimes \theta))>> I_{P}(\pi_\psi).
\end{CD}
\]
If we realize $w' (\tau_2 \circ \theta)$, $\tau_2$, $w(\tau_1 \circ \theta)$ and $\tau_1$
on the same vector space, say $\VV$, 
then $\tl\tau_2(w' \rtimes \theta)$ is a linear isomorphism $\Phi \colon \VV \rightarrow \VV$ satisfying 
\[
\Phi \circ \tau_2(\theta(\tl{w}'^{-1} m_2 \tl{w}')) = \tau_2(m_2) \circ \Phi, \quad m_2 \in M_2.
\]
Since $\tau_2 = w_2^{-1}\tau_1$, $w = w_2w'\theta(w_2)^{-1}$ and $M_1 = w_2M_2w_2^{-1}$, 
by Lemma \ref{center}, 
the above property of $\Phi$ can be rewritten as
\[
\Phi \circ \tau_1(\theta(\tl{w}^{-1} m_1 \tl{w})) = \tau_1(m_1) \circ \Phi, \quad m_1 \in M_1.
\]
Since $\tl\tau_2(w' \rtimes \theta)$ and $\tl\tau_1(w \rtimes \theta)$ are both normalized using a Whittaker functional, 
we see that $\tl\tau_2(w' \rtimes \theta) = \tl\tau_1(w \rtimes \theta)$ as linear isomorphisms on $\VV$.
It implies the commutativity of the top square. 
On the other hand, notice that $\Ind_{P_1 \cap M}^M(\tau_1)$ is the standard module of $\pi_\psi$. 
By the definition of $\tl\pi_\psi(w \rtimes \theta)$ and the functoriality of $I_P$, 
we have the following commutative diagram
\[
\begin{CD}
I_P(w(\Ind_{P_1 \cap M}^M(\tau_1) \circ \theta))
@>>> I_P(\Ind_{P_1 \cap M}^M(\tau_1)) \\
@VI_P(R_{P_1 \cap M}(w_1, w(\tau_1 \circ \theta)))VV @VVI_P(R_{P_1 \cap M}(w_1, \tau_1))V \\
I_P(w(\pi_\psi \circ \theta)) @>I_P(\tl\pi_\psi(w \rtimes \theta))>> I_P(\pi_\psi)
\end{CD}
\]
where the top map is induced from the isomorphism 
$w(\Ind_{P_1 \cap M}^M(\tau_1) \circ \theta) \xrightarrow{\sim} \Ind_{P_1 \cap M}^M(\tau_1)$
normalized by using a Whittaker functional.
We shall show that this commutative diagram is canonically isomorphic to 
the bottom square of the right rectangle under consideration.
To see this, note that we can realize $I_P(w(\Ind_{P_1 \cap M}^M(\tau_1) \circ \theta))$ 
as a subspace of two-variable functions $f \colon G \times M \rightarrow \VV$ satisfying  
\[
f(m'g, m) = \delta_P^{\half{1}}(m') f(g, m \theta(\tl{w}^{-1}m'\tl{w}))
\]
for $m,m' \in M$ and $g \in G$.
We have a canonical isomorphism 
\[
I_P(w(\Ind_{P_1 \cap M}^M(\tau_1) \circ \theta)) \xrightarrow{\sim} I_{P_1}(w(\tau_1 \circ \theta)),\,
f(g,m) \mapsto f(g,\1).
\]
Via this isomorphism, $I_P(R_{P_1 \cap M}(w_1, w(\tau_1 \circ \theta)))$ corresponds to 
\[
R_{P_1}(w_1, w(\tau_1 \circ \theta)) \colon I_{P_1}(w(\tau_1 \circ \theta)) \rightarrow I_P(w(\pi_\psi \circ \theta)).
\] 
Similarly, we have a canonical isomorphism 
$I_P(\Ind_{P_1 \cap M}^M(\tau_1)) \cong I_{P_1}(\tau_1)$.
Via these canonical isomorphisms, the above commutative diagram is rewritten as
\[
\begin{CD}
I_{P_1}(w(\tau_1 \circ \theta)) @>I_{P_1}(\tl\tau_1(w \rtimes \theta))>> I_{P_1}(\tau_1) \\
@VR_{P_1}(w_1, w(\tau_1 \circ \theta))VV @VVR_{P_1}(w_1, \tau_1)V \\
I_{P}(w(\pi_\psi \circ \theta)) @>I_P(\tl\pi_\psi(w \rtimes \theta))>> I_{P}(\pi_\psi).
\end{CD}
\]
Since this is exactly the bottom square of the right rectangle under consideration, we have shown the desired commutativity of the bottom square.
\par

To prove Theorem \ref{diagramGL}, 
it remains to show that the middle rectangle 
\[
\begin{CD}
I_{\theta(P_2)}(\tau_2 \circ \theta) @>R_{\theta(P_2)}(w', \tau_2 \circ \theta)>> I_{P_2}(w' (\tau_2 \circ \theta)) \\
@VR_{\theta(P_2)}(\theta(w_2), \tau_2 \circ \theta)VV @VVR_{P_2}(w_2, w'(\tau_2 \circ \theta))V \\
I_{\theta(P_1)}(\tau_1 \circ \theta) @. I_{P_1}(w(\tau_1 \circ \theta)) \\
@VR_{\theta(P_1)}(\theta(w_1), \tau_1 \circ \theta)VV @VVR_{P_1}(w_1, w(\tau_1 \circ \theta))V \\
I_{\theta(P)}(\pi_\psi \circ \theta) @>R_{\theta(P)}(w, \pi_\psi \circ \theta, \psi) >> I_{P}(w(\pi_\psi \circ \theta))
\end{CD}
\]
commutes.
The proof of this commutativity will be carried out in four steps below. 
Roughly speaking, 
we will embed this diagram into a meromorphic family of diagrams. 
Working in the context of this meromorphic family of diagrams, 
we will complete the diagram by extending the last row of the diagram as 
\[
\xymatrix{
I_{\theta(P)}(\pi_\psi \circ \theta) 
\ar@{->}[rrr]^{R_{\theta(P)}(w, \pi_\psi \circ \theta, \psi)} 
\ar@{^{(}->}[d]
&&& I_{P}(w(\pi_\psi \circ \theta)) \ar@{^{(}->}[d]\\
I_{\theta(P_1)}(w_1\tau_1 \circ \theta) \ar@{..>}[rrr]^{R_{\theta(P_1)}(w, w_1\tau_1 \circ \theta)}  
&&& I_{P_1}(w(w_1\tau_1 \circ \theta)). 
}
\]
Then 
the commutativity of the meromorphic family of diagrams is 
a consequence of the multiplicativity property in Proposition \ref{multiplicative}.
We will then deduce the commutativity of the desired diagram 
by specializing at the point of interest. 
Note however that the additional normalized intertwining operator 
will generally have a singularity at the point of interest 
(which is the reason why we use the dotted arrow), 
so its purpose is only to help prove the commutativity of the whole meromorphic family, 
before specialization.
\par

Let us start the proof of the commutativity of the middle diagram. 
\begin{description}
\setlength{\itemsep}{10pt}
\item[Step 1]
Recall that 
\[
\tau_1 = \bigotimes_{\alpha \in (1/2)\Z}\bigotimes_{i=1}^t \tau^{(i)}|\cdot|_E^{\alpha}. 
\]
For tuples of complex numbers
\begin{align*}
\lambda &= (\lambda^{(i)})_{i} \in \bigoplus_{i=1}^{t} \C, \\
\mu &= (\mu^{(i)}(\alpha))_{\alpha, i} \in \bigoplus_{\alpha \in (1/2)\Z}\bigoplus_{i=1}^{t} \C,
\end{align*}
we set 
\[
\tau_{1,(\lambda,\mu)} 
= \bigotimes_{\alpha \in (1/2)\Z}\bigotimes_{i=1}^t 
\tau^{(i)}|\cdot|_E^{\alpha+\lambda^{(i)}+\mu^{(i)}(\alpha)}.
\]
We define $\tau_{2,(\lambda,\mu)}$ similarly. 
Hence 
\[
\tau_{1,(\lambda,\mu)} = w_2\tau_{2,(\lambda,\mu)}.
\]
Note that 
\begin{align*}
w_2w'(\tau_{2, (\lambda,\mu)} \circ \theta)
&= 
w_1^{-1}w \theta(w_1) \theta(w_2) (\tau_{2, (\lambda,\mu)} \circ \theta)
\\&=
w_1^{-1}w (w_1w_2\tau_{2, (\lambda,\mu)} \circ \theta)
\\&= 
w_1^{-1}w (w_1\tau_{1, (\lambda,\mu)} \circ \theta).
\end{align*}
We can consider the following diagram of meromorphic families of operators: 
\[
\begin{CD}
I_{\theta(P_2)}(\tau_{2, (\lambda,\mu)} \circ \theta) 
@> R_{\theta(P_2)}(w', \tau_{2, (\lambda,\mu)} \circ \theta) >> 
I_{P_2}(w'(\tau_{2, (\lambda,\mu)} \circ \theta))\\
@VR_{\theta(P_2)}(\theta(w_2), \tau_{2, (\lambda,\mu)} \circ \theta)VV 
@VV R_{P_2}(w_2, w'(\tau_{2, (\lambda,\mu)} \circ \theta))V \\
I_{\theta(P_1)}(\tau_{1, (\lambda,\mu)} \circ \theta) @. I_{P_1}(w(\tau_{1, (\lambda,\mu)} \circ \theta))\\
@VR_{\theta(P_1)}(\theta(w_1), \tau_{1, (\lambda,\mu)} \circ \theta)VV 
@VV R_{P_1}(w_1, w (\tau_{1, (\lambda,\mu)} \circ \theta))V \\
I_{\theta(P_1)}(w_1\tau_{1, (\lambda,\mu)} \circ \theta)
@> R_{\theta(P_1)}(w, w_1\tau_{1, (\lambda,\mu)} \circ \theta) >> 
I_{P_1}(w(w_1\tau_{1, (\lambda,\mu)} \circ \theta)).
\end{CD}
\]
By Proposition \ref{multiplicative}, 
this diagram commutes whenever $\alpha+\lambda^{(i)}+\mu^{(i)}(\alpha) \in \I\R$ for all $\alpha, i$. 
Hence, by analytic continuation, 
we see that this diagram is commutative for all $(\lambda,\mu)$
at which all operators are regular. 

\item[Step 2]
In the diagram in Step 1, we will specialize at $\mu = 0$.
We claim that all six intertwining operators are well-defined 
as meromorphic families of operators in $\lambda$.
\par

In fact, the bottom arrow $R_{\theta(P_1)}(w, w_1\tau_{1, (\lambda,0)} \circ \theta)$ 
is a composition of intertwining operators of the form
\[
\tau^{(i)}|\cdot|_E^{\alpha+\lambda^{(i)}}
\times \tau^{(i')}|\cdot|_E^{\alpha'+\lambda^{(i')}}
\rightarrow 
\tau^{(i')}|\cdot|_E^{\alpha'+\lambda^{(i')}}
\times
\tau^{(i)}|\cdot|_E^{\alpha+\lambda^{(i)}}
\]
for some $i \not= i'$.
Hence the subset $\{(\lambda,0)\} \subset \{(\lambda, \mu)\}$ is not contained in 
the subset consisting of $(\lambda, \mu)$ at which
$R_{\theta(P_1)}(w, w_1\tau_{1, (\lambda,\mu)} \circ \theta)$ is singular. 
Here, we notice that 
$R_{\theta(P_1)}(w, w_1\tau_{1, (\lambda,\mu)} \circ \theta)$
can have a singularity at $(\lambda,\mu) = (0,0)$.
On the other hand, 
as we have seen in the proofs of Lemmas \ref{langlands} and \ref{www} (3), 
the other five operators are indeed regular even at $(\lambda, \mu) = (0,0)$.
\par

Hence we can evaluate at $\mu = 0$ in the diagram in Step 1, 
and get the following commutative diagram:
\[
\begin{CD}
I_{\theta(P_2)}(\tau_{2, (\lambda,0)} \circ \theta) 
@> R_{\theta(P_2)}(w', \tau_{2, (\lambda,0)} \circ \theta) >> 
I_{P_2}(w'(\tau_{2, (\lambda,0)} \circ \theta))\\
@VR_{\theta(P_2)}(\theta(w_2), \tau_{2, (\lambda,0)} \circ \theta)VV 
@VV R_{P_2}(w_2, w'(\tau_{2, (\lambda,0)} \circ \theta))V \\
I_{\theta(P_1)}(\tau_{1, (\lambda,0)} \circ \theta) @. I_{P_1}(w(\tau_{1, (\lambda,0)} \circ \theta))\\
@VR_{\theta(P_1)}(\theta(w_1), \tau_{1, (\lambda,0)} \circ \theta)VV 
@VV R_{P_1}(w_1, w (\tau_{1, (\lambda,0)} \circ \theta))V \\
I_{\theta(P_1)}(w_1\tau_{1, (\lambda,0)} \circ \theta)
@> R_{\theta(P_1)}(w, w_1\tau_{1, (\lambda,0)} \circ \theta) >> 
I_{P_1}(w(w_1\tau_{1, (\lambda,0)} \circ \theta)).
\end{CD}
\]

\item[Step 3]
Recall that $\Speh(\tau_j^{(i)}, d_j^{(i)})$ is the unique irreducible quotient of 
\[
\tau_j^{(i)}|\cdot|_E^{\frac{d-1}{2}} \times \tau_j^{(i)}|\cdot|_E^{\frac{d-3}{2}} 
\times \dots \times \tau_j^{(i)}|\cdot|_E^{-\frac{d-1}{2}}
\]
with $d = d_j^{(i)}$. 
Note that
\[
I_P(\pi_\psi) = \bigtimes_{i=1}^t \bigtimes_{j=1}^{m_i} \Speh(\tau_j^{(i)}, d_j^{(i)}), 
\quad 
\psi = \bigoplus_{i=1}^t \bigoplus_{j=1}^{m_i} \psi_j^{(i)}.
\]
Set 
\[
I_P(\pi_{\psi, \lambda}) 
= 
\bigtimes_{i=1}^t \bigtimes_{j=1}^{m_i} \Speh(\tau_j^{(i)}, d_j^{(i)})|\cdot|_E^{\lambda^{(i)}},
\]
and
\[
\psi_\lambda = \bigoplus_{i=1}^t \bigoplus_{j=1}^{m_i} \psi_j^{(i)}|\cdot|_E^{\lambda^{(i)}}.
\]
Note that $I_P(\pi_{\psi, \lambda}) $ is an irreducible subrepresentation of $I_{P_1}(w_1\tau_{1, (\lambda,0)})$.
Moreover, it is equal to the image of 
$R_{P_1}(w_1, \tau_{1, (\lambda,0)}) \circ R_{P_2}(w_2, \tau_{2, (\lambda,0)})$ 
by Lemma \ref{langlands}.
We have now the diagram
\[
\begin{CD}
I_{\theta(P)}(\pi_{\psi, \lambda} \circ \theta) 
@>R_{\theta(P)}(w, \pi_{\psi, \lambda} \circ \theta, \psi_\lambda)>> 
I_P(w(\pi_{\psi, \lambda} \circ \theta))\\
@VVV @VVV\\
I_{\theta(P_1)}(w_1\tau_{1, (\lambda,0)} \circ \theta)
@> R_{\theta(P_1)}(w, w_1\tau_{1, (\lambda,0)} \circ \theta) >> 
I_{P_1}(w(w_1\tau_{1, (\lambda,0)} \circ \theta))
\end{CD}
\]
where the vertical maps are the canonical inclusions.
We claim that this diagram is commutative. 
Indeed, the defining integrals are the same by Lemma \ref{www} (1), 
and the normalizing factors also agree by definition.
In conclusion, we obtain the following commutative diagram of meromorphic families of operators: 
\[
\begin{CD}
I_{\theta(P_2)}(\tau_{2, (\lambda,0)} \circ \theta) 
@> R_{\theta(P_2)}(w', \tau_{2, (\lambda,0)} \circ \theta) >> 
I_{P_2}(w'(\tau_{2, (\lambda,0)} \circ \theta))\\
@VR_{\theta(P_2)}(\theta(w_2), \tau_{2, (\lambda,0)} \circ \theta)VV 
@VV R_{P_2}(w_2, w'(\tau_{2, (\lambda,0)} \circ \theta))V \\
I_{\theta(P_1)}(\tau_{1, (\lambda,0)} \circ \theta) @. I_{P_1}(w(\tau_{1, (\lambda,0)} \circ \theta))\\
@VR_{\theta(P_1)}(\theta(w_1), \tau_{1, (\lambda,0)} \circ \theta)VV 
@VV R_{P_1}(w_1, w (\tau_{1, (\lambda,0)} \circ \theta))V \\
I_{\theta(P)}(\pi_{\psi, \lambda} \circ \theta) 
@>R_{\theta(P)}(w, \pi_{\psi, \lambda} \circ \theta, \psi_\lambda)>> 
I_P(w(\pi_{\psi, \lambda} \circ \theta)).
\end{CD}
\]

\item[Step 4]
In this last step, we would like to specialize the commutative diagram above at $\lambda = 0$.
As we have seen in Step 2, 
the five operators that appear in the top, left and right of the last diagram 
are regular at $\lambda = 0$. 
Especially, the composition 
\[
R_{P_1}(w_1, w (\tau_{1, (\lambda,0)} \circ \theta))
\circ
R_{P_2}(w_2, w'(\tau_{2, (\lambda,0)} \circ \theta))
\circ
R_{\theta(P_2)}(w', \tau_{2, (\lambda,0)} \circ \theta)
\]
is regular at $\lambda = 0$, and so is 
\[
R_{\theta(P)}(w, \pi_{\psi, \lambda} \circ \theta, \psi_\lambda) 
\circ 
R_{\theta(P_1)}(\theta(w_1), \tau_{1, (\lambda,0)} \circ \theta)
\circ 
R_{\theta(P_2)}(\theta(w_2), \tau_{2, (\lambda,0)} \circ \theta).
\]
Moreover, 
since the composition 
$R_{\theta(P_1)}(\theta(w_1), \tau_{1, (\lambda,0)} \circ \theta)
\circ 
R_{\theta(P_2)}(\theta(w_2), \tau_{2, (\lambda,0)} \circ \theta)$
is surjective if $\lambda$ is sufficiently close to $0$ by Lemma \ref{langlands}, 
the operator 
\[
R_{\theta(P)}(w, \pi_{\psi, \lambda} \circ \theta, \psi_\lambda) \colon 
I_{\theta(P)}(\pi_{\psi, \lambda} \circ \theta) \rightarrow I_P(w(\pi_{\psi, \lambda} \circ \theta))
\]
is also regular at $\lambda = 0$.
\par

Therefore, we can specialize the last diagram in Step 3 at $\lambda= 0$, 
and obtain the following commutative diagram.
\[
\begin{CD}
I_{\theta(P_2)}(\tau_{2} \circ \theta) 
@> R_{\theta(P_2)}(w', \tau_{2} \circ \theta) >> 
I_{P_2}(w'(\tau_{2} \circ \theta))\\
@VR_{\theta(P_2)}(\theta(w_2), \tau_{2} \circ \theta)VV 
@VV R_{P_2}(w_2, w'(\tau_{2} \circ \theta))V \\
I_{\theta(P_1)}(\tau_{1} \circ \theta) @. I_{P_1}(w(\tau_1 \circ \theta))\\
@VR_{\theta(P_1)}(\theta(w_1), \tau_{1} \circ \theta)VV 
@VV R_{P_1}(w_1, w (\tau_{1} \circ \theta))V \\
I_{\theta(P)}(\pi_{\psi} \circ \theta) 
@>R_{\theta(P)}(w, \pi_{\psi} \circ \theta, \psi)>> 
I_P(w(\pi_{\psi} \circ \theta)).
\end{CD}
\]
Hence we obtain the claim.
\end{description}
This completes the proof of Theorem \ref{diagramGL}.
\end{proof}

The following example may help the reader to understand the argument. 
\begin{ex}\label{ex_GL}
Let us suppose that $M = \GL_3(E) \times \GL_1(E) \subset P=MN_P \subset G = \GL_4(E)$, 
and let us consider
\[
\psi = \1_{L_E} \boxtimes S_3 \oplus \1_{L_E} \boxtimes S_1 \in \Psi(M).
\]
Then 
\[
\pi_\psi = \1_{\GL_3(E)} \boxtimes \1_{\GL_1(E)} \in \Irr(M)
\]
is the corresponding representation.
Note that $I_P(\pi_\psi) = \1_{\GL_3(E)} \times \1_{\GL_1(E)} \in \Irr(G)$.
We realize it as a subrepresentation of 
\[
|\cdot|_E^{-1} \times |\cdot|_E^{0} \times |\cdot|_E^{1} \times |\cdot|_E^{0}. 
\]
\par

In what follows
we abbreviate the normalized intertwining operator $R_P(w, \pi)$ by $w \in W^G$. 
The main diagram becomes:
\[
{\tiny
\xymatrix{
|\cdot|_E^{1} \times |\cdot|_E^{0} \times |\cdot|_E^{0} \times |\cdot|_E^{-1}
\ar@{->}[r]^{\theta^*} 
\ar@{->}[dd]_{\left(\begin{smallmatrix}
1&&&\\
&1&&\\
&&0&1\\
&&1&0
\end{smallmatrix}\right)}
&|\cdot|_E^{1} \times |\cdot|_E^{0} \times |\cdot|_E^{0} \times |\cdot|_E^{-1}
\ar@{->}[rr]^{\left(\begin{smallmatrix}
1&&&\\
&0&1&\\
&1&0&\\
&&&1
\end{smallmatrix}\right)}
\ar@{->}[dd]_{\left(\begin{smallmatrix}
0&1&&\\
1&0&&\\
&&1&\\
&&&1
\end{smallmatrix}\right)}
&&|\cdot|_E^{1} \times |\cdot|_E^{0} \times |\cdot|_E^{0} \times |\cdot|_E^{-1}
\ar@{->}[dd]_{\left(\begin{smallmatrix}
1&&&\\
&1&&\\
&&0&1\\
&&1&0
\end{smallmatrix}\right)}
\\
&&&
\\
|\cdot|_E^{1} \times |\cdot|_E^{0} \times |\cdot|_E^{-1} \times |\cdot|_E^{0}
\ar@{->}[r]^{\theta^*}
\ar@{->>}[dd]_{\left(\begin{smallmatrix}
&&1&\\
&1&&\\
1&&&\\
&&&1
\end{smallmatrix}\right)}
&
|\cdot|_E^{0} \times |\cdot|_E^{1} \times |\cdot|_E^{0} \times |\cdot|_E^{-1}
\ar@{->>}[dd]_{\left(\begin{smallmatrix}
1&&&\\
&&&1\\
&&1&\\
&1&&
\end{smallmatrix}\right)}
&&
|\cdot|_E^{1} \times |\cdot|_E^{0} \times |\cdot|_E^{-1} \times |\cdot|_E^{0}
\ar@{->>}[dd]_{\left(\begin{smallmatrix}
&&1&\\
&1&&\\
1&&&\\
&&&1
\end{smallmatrix}\right)}
\\
&&&
\\
\1_{\GL_3(E)} \times \1_{\GL_1(E)} 
\ar@{->}[r]^{\theta^*}
&
\1_{\GL_1(E)} \times \1_{\GL_3(E)} 
\ar@{^{(}->}[dd]
\ar@{->}[rr]^{\left(\begin{smallmatrix}
&1&&\\
&&1&\\
&&&1\\
1&&&
\end{smallmatrix}\right)}
&&
\1_{\GL_3(E)} \times \1_{\GL_1(E)} 
\ar@{^{(}->}[dd]
\\
&&&
\\
&|\cdot|_E^0 \times |\cdot|_E^{-1} \times |\cdot|_E^{0} \times |\cdot|_E^{1} 
\ar@{..>}[rr]^{\left(\begin{smallmatrix}
&1&&\\
&&1&\\
&&&1\\
1&&&
\end{smallmatrix}\right)}
&&
|\cdot|_E^{-1} \times |\cdot|_E^{0} \times |\cdot|_E^{1} \times |\cdot|_E^{0}
}}
\]
Notice that we have added a bottom ``map''
\[
\xymatrix{
|\cdot|_E^0 \times |\cdot|_E^{-1} \times |\cdot|_E^{0} \times |\cdot|_E^{1} 
\ar@{..>}[rr]^{\left(\begin{smallmatrix}
&1&&\\
&&1&\\
&&&1\\
1&&&
\end{smallmatrix}\right)}
&&
|\cdot|_E^{-1} \times |\cdot|_E^{0} \times |\cdot|_E^{1} \times |\cdot|_E^{0}
}
\]
that is actually a singularity of a meromorphic operator. 
Hence this ``map'' is not well-defined so that 
we cannot consider it. 
\end{ex}

\subsection{Remark on the untwisted case}
Note that the argument for Theorem \ref{main2} works when we replace $\theta$ with the identity map on $G$. 
We state the analogue of Theorem \ref{main2} as in \cite[Proposition 3.5.1 (a)]{Mok}. 

\begin{thm}\label{untwistedGL}
Let $P = MN_P$ be a standard parabolic subgroup of $G = \GL_N(E)$, 
and let $\psi$ be an $A$-parameter for $M$. 
Then for any $w \in W(M)$ with $w\pi_\psi \cong \pi_\psi$, 
we have
\[
I_P(\tl\pi_\psi(w)) \circ R_P(w, \pi_\psi, \psi) = \id, 
\]
where $\tl\pi_\psi(w) \colon w\pi_\psi \rightarrow \pi_\psi$ is the isomorphism 
normalized by using a Whittaker functional on the standard module $\II_\psi^M$ of $\pi_\psi$. 
\end{thm}

When $G$ is a unitary group $\U_N$, 
this theorem is needed as a local input at split places for the global argument.

\section{Characters of component groups vs.~Aubert duality}\label{sec.character}
In this and the next two sections, 
we assume that $F$ is non-archimedean.
Fix a non-trivial additive character $\psi_F \colon F \rightarrow \C^\times$. 
Let $\oo_E$ be the ring of integers of $E$.
\par

In this section, assuming Arthur's theory (Hypothesis \ref{ECR}), 
we provide a formula for the action of Aubert duality on the characters of component groups in certain cases
(Corollaries \ref{+-temp} and \ref{+otimes-}). 
These results will be applied to tempered $L$-parameters for a given classical group $G$, 
and to certain $A$-parameters for a proper Levi subgroup $M$ of $G$:
in both cases, Hypothesis \ref{ECR} is known in the framework of Arthur's inductive argument. 

\subsection{Twisted Aubert duality for $\GL_N(E)$}\label{sec.TAD_GL}
In this and next subsections, we consider twisted Aubert duality for $\GL_N(E)$.
For a preview of the general theory of twisted Aubert duality, 
see Sections \ref{sec.twistedAD} and \ref{sec.twistedAD_G}.
\par

Let $\Rep(\tl\GL_N(E))$ be the category of smooth finite-length representations 
of the non-connected group $\tl\GL_N(E) = \GL_N(E) \rtimes \pair{\theta}$, 
and let $\RR(\tl\GL_N(E))$ be its Grothendieck group. 
For $\tl\pi \in \RR(\tl\GL_N(E))$, 
one can consider its character $\Theta_{\tl\pi}$, which is a linear form on $C_c^\infty(\tl\GL_N(E))$. 
For $\tl\pi_1, \tl\pi_2 \in \RR(\tl\GL_N(E))$, 
we write 
\[
\tl\pi_1 \overset{\theta}{=} \tl\pi_2
\]
if 
\[
\Theta_{\tl\pi_1}(\tl{f}) = \Theta_{\tl\pi_2}(\tl{f})
\]
for any $\tl{f} \in C_c^\infty(\GL_N(E) \rtimes \theta)$. 
For example, if $\pi \in \Irr(\GL_N(E))$, then 
\[
\Ind_{\GL_N(E)}^{\tl\GL_N(E)}(\pi) \overset{\theta}{=} 0.
\]
\par

Let
\[
\Rep(\tl\GL_N(E)) \rightarrow \Rep(\tl\GL_N(E)),\, \tl\pi \mapsto \widehat{\tl\pi}
\]
be the functor given by Definition \ref{def.TAD}. 

\begin{prop}\label{TAD_GL}
The functor $\tl\pi \mapsto \widehat{\tl\pi}$ satisfies the following properties. 
\begin{enumerate}
\item
The restriction $\widehat{\tl\pi}|_{\GL_N(E)}$ is equal to the Aubert dual of $\tl\pi|_{\GL_N(E)}$. 

\item
If $\tl\pi$ is an irreducible representation of $\tl\GL_N(E)$, 
then so is $\widehat{\tl\pi}$. 

\item
For $\tl\pi \in \RR(\tl\GL_N(E))$, set
\[
D_{\tl\GL_N(E)}(\tl\pi) = \sum_{P = \theta(P)} (-1)^{\dim (A_M^\theta)} 
\Ind_{\tl{P}}^{\tl\GL_N(E)}(\Jac_{\tl{P}}(\tl\pi)), 
\]
where $P$ runs over the set of standard parabolic subgroups of $\GL_N(E)$
which are stable under $\theta$ and we set $\tl{P} = P \rtimes \pair{\theta}$. 
Then 
\[
D_{\tl\GL_N(E)}(\tl\pi) \overset{\theta}{=} (-1)^{\dim (A_{M_0}/A_{\GL_N(E)})} \widehat{\tl\pi}
\]
for $\tl\pi \in \Irr(\tl\GL_N(E))$, 
where $P_0=M_0N_{P_0}$ is a minimal standard parabolic subgroup of $\GL_N(E)$
such that $\Jac_{P_0}(\tl\pi|_{\GL_N(E)}) \not= 0$.  
Note that such a $P_0$ is not unique but the sign $(-1)^{\dim (A_{M_0}/A_{\GL_N(E)})}$ is well-defined.
\end{enumerate}
\end{prop}
\begin{proof}
These properties are proven in Propositions \ref{TAD} and \ref{prop2.13}.
For the notations of parabolic inductions and Jacquet modules, see Section \ref{sec.twistedAD_G}.  
Here, note that $\dim((A_M/A_{\GL_N(E)})^\theta) = \dim(A_M^\theta)$.
\end{proof}

We call $\widehat{\tl\pi}$ the \emph{Aubert dual} of $\tl\pi$. 
We remark that $\tl\pi \mapsto \widehat{\tl\pi}$ is expected to be an involution, 
but we do not prove it and we will not use this property.
\par

Let $\psi \colon W_E \times \SL_2(\C) \times \SL_2(\C) \rightarrow \GL_N(\C)$ 
be an $A$-parameter for $\GL_N(E)$ 
and let $\pi_\psi$ be the corresponding irreducible representation of $\GL_N(E)$. 
Suppose that $\psi$ (or equivalently, $\pi_\psi$) is conjugate-self-dual. 
Then, as in Section \ref{sec.thetaA}, 
we have Arthur's extension $\tl\pi_\psi = \pi_\psi \boxtimes \theta_A$ of $\pi_\psi$ to $\tl\GL_N(E)$. 
The Aubert dual of $\pi_\psi$ is equal to $\pi_{\widehat\psi}$, 
where $\widehat\psi$ is defined by $\widehat\psi(w,g_1,g_2) = \psi(w,g_2,g_1)$.
However, $\widehat{\tl\pi}_\psi$ is not necessarily equal to $\tl\pi_{\widehat\psi}$. 
Namely, if we write $\widehat{\tl\pi}_\psi = \pi_{\widehat\psi} \boxtimes \widehat\theta_A$, 
then $\widehat\theta_A$ is not necessarily equal to $\theta_A$ as linear operators on $\pi_{\widehat\psi}$. 
Since $\theta^2 = 1$, we have $\widehat\theta_A = \pm \theta_A$.
Hence $D_{\tl\GL_N(E)}(\tl\pi_\psi) \overset{\theta}{=} \pm \tl\pi_{\widehat\psi}$. 
We define $\beta(\psi) \in \{\pm1\}$ such that
\[
D_{\tl\GL_N(E)}(\tl\pi_\psi) \overset{\theta}{=} \beta(\psi) \tl\pi_{\widehat\psi}. 
\]
\par

\begin{lem}\label{beta_twist}
Let $\psi$ be as above, 
and let $\eta$ be a conjugate-self-dual character of $E^\times$, 
which is regarded as a character of $W_E$ by the local class field theory.
Then $\beta(\psi \otimes \eta) = \beta(\psi)$. 
\end{lem}
\begin{proof}
Since $\eta$ is conjugate-self-dual, we have $\eta(\det(\theta(g))) = \eta(\det(g))$ for $g \in \GL_N(E)$. 
This implies that $\tl\pi_{\psi \otimes \eta} = \tl\pi_\psi \otimes \eta$, i.e., 
Arthur's action of $\theta$ on $\pi_{\psi \otimes \eta}$ is the same as that of $\pi_\psi$
as a linear operator on the same vector space.
By Lemma \ref{AD_twist}, $\widehat{\tl\pi}_{\psi \otimes \eta}
= \widehat{\tl\pi}_{\psi} \otimes \eta \overset{\theta}{=} \beta(\psi)\pi_{\widehat\psi} \otimes \eta$.
Hence we have $\beta(\psi \otimes \eta) = \beta(\psi)$. 
\end{proof}

On the other hand, 
since $\psi$ is conjugate-self-dual, 
we can decompose 
the representation 
\[
W_E \ni w \mapsto 
\psi\left(w, 
\begin{pmatrix}
|w|_E^{\half{1}} & 0 \\ 0 & |w|_E^{-\half{1}}
\end{pmatrix}, 
\begin{pmatrix}
|w|_E^{\half{1}} & 0 \\ 0 & |w|_E^{-\half{1}}
\end{pmatrix}
\right) \in \GL_N(\C)
\]
of $W_E$ into irreducible representations as 
\[
\rho_{-r} \oplus \dots \oplus \rho_{-1} 
\oplus \rho_1' \oplus \dots \oplus \rho'_t
\oplus \rho_1 \oplus \dots \oplus \rho_r, 
\]
where 
\begin{itemize}
\item
$\rho_i$ is an irreducible representation of $W_E$; 
\item
$\rho_{-i} \cong {}^c\rho_{i}^\vee$ is the conjugate-dual of $\rho_i$; 
\item
$\rho'_i$ is an irreducible conjugate-self-dual representation of $W_E$ 
such that $\rho'_i \not\cong \rho'_j$ for $1 \leq i < j \leq t$. 
\end{itemize}
Note that such a labeling of irreducible components is not unique, 
but the non-negative integer $r$ is uniquely determined from $\psi$. 
We write $r(\psi) = r$.

\begin{prop}\label{c-temp}
Suppose that $\psi = \phi$ is tempered and conjugate-self-dual. 
Then $\beta(\phi) = \beta(\widehat\phi) = (-1)^{r(\phi)}$.
\end{prop}
\begin{proof}
A more general assertion was stated by M{\oe}glin--Waldspurger (\cite[Lemma 3.2.2]{MW_twist}).
However, they gave a proof only when 
every irreducible component of $\phi|_{\SL_2(\C)}$ is even-dimensional. 
Here, we will give a proof for the general case.
\par

Let $\II_{\widehat\phi}$ be the standard module whose Langlands quotient is $\pi_{\widehat\phi}$. 
If we denote by $\theta_W$ the action of $\theta$ on $\II_{\widehat\phi}$ fixing a nonzero Whittaker functional, 
then the action $\theta_A$ on $\pi_{\widehat\phi}$ is induced from $\theta_W$ by definition. 
On the other hand, by Lemma \ref{psiD}, 
we see that $\II_{\widehat\phi}$ contains the tempered representation $\pi_\phi$ as a subrepresentation.
Moreover, the restriction of $\theta_W$ to $\pi_\phi$ gives Arthur's action $\theta_A$ on $\pi_\phi$
since the restriction of a nonzero Whittaker functional on $\II_{\widehat\phi}$ to $\pi_\phi$ is also nonzero.
Thus, as representations of $\tl\GL_N(E)$, we have
\[
\pi_\phi \boxtimes \theta_A \hookrightarrow \II_{\widehat{\phi}} \boxtimes \theta_W \twoheadrightarrow \pi_{\widehat\phi} \boxtimes \theta_A.
\]
\par

Applying the duality functor $\tl\pi \mapsto \widehat{\tl\pi}$, 
we see that 
$(\pi_\phi \boxtimes \theta_A)\widehat{\,} = \pi_{\widehat\phi} \boxtimes \widehat\theta_A$
and 
$(\pi_{\widehat\phi} \boxtimes \theta_A)\widehat{\,} = \pi_{\phi} \boxtimes \widehat\theta_A$
are contained in 
$(\II_{\widehat\phi} \boxtimes \theta_W)\widehat{\,} = \widehat{\II}_{\widehat{\phi}} \boxtimes \widehat\theta_W$
as irreducible subquotients. 
By Proposition \ref{TAD_GL} (3), they satisfy 
\begin{align*}
D_{\tl\GL_N(E)}(\II_{\widehat\phi} \boxtimes \theta_W) 
&\overset{\theta}{=} \ep \cdot \widehat{\II}_{\widehat{\phi}} \boxtimes \widehat\theta_W, \\
D_{\tl\GL_N(E)}(\pi_\phi \boxtimes \theta_A) 
&\overset{\theta}{=} \ep \cdot \pi_{\widehat\phi} \boxtimes \widehat\theta_A, \\
D_{\tl\GL_N(E)}(\pi_{\widehat\phi} \boxtimes \theta_A) 
&\overset{\theta}{=} \ep \cdot \pi_{\phi} \boxtimes \widehat\theta_A
\end{align*}
for some common sign $\ep \in \{\pm1\}$
since all irreducible subquotients of $\II_{\widehat\phi}$ share the same cuspidal support.
Since $\II_{\widehat\phi}$ is an induction from a cuspidal representation, 
its Aubert dual  $\widehat{\II}_{\widehat\phi}$ is equal to $\II_{\widehat\phi}$ in $\RR(\GL_N(E))$. 
Moreover, as $\dim_\C(\End_{\GL_N(E)}(\II_{\widehat\phi})) = 1$ by Lemma \ref{End}, 
we can find $\delta \in \{\pm1\}$ such that $\widehat\theta_W = \delta \theta_W$ so that 
$\widehat{\II}_{\widehat{\phi}} \boxtimes \widehat\theta_W 
= \delta \cdot \II_{\widehat{\phi}} \boxtimes \theta_W$
in $\RR(\tl\GL_N(E))$.
Hence 
\[
D_{\tl\GL_N(E)}(\II_{\widehat\phi} \boxtimes \theta_W) 
\overset{\theta}{=} \ep\delta \cdot \II_{\widehat{\phi}} \boxtimes \theta_W.
\]
\par

Since $\pi_\phi$ and $\pi_{\widehat\phi}$ each appear in $\II_{\widehat\phi}$ as a subquotient with multiplicity one, 
by Proposition \ref{TAD_GL} (1), we see that $\widehat{\II}_{\widehat{\phi}} \boxtimes \widehat\theta_W$
contains only one irreducible representation of the form 
$\pi_{\widehat\phi} \boxtimes \theta$ (\resp $\pi_{\phi} \boxtimes \theta$)
as subquotients. 
In particular, we get 
\[
D_{\tl\GL_N(E)}(\pi_\phi \boxtimes \theta_A) 
\overset{\theta}{=} \ep\delta \cdot \pi_{\widehat\phi} \boxtimes \theta_A, 
\quad
D_{\tl\GL_N(E)}(\pi_{\widehat\phi} \boxtimes \theta_A) 
\overset{\theta}{=} \ep\delta \cdot \pi_{\phi} \boxtimes \theta_A. 
\]
This means that $\ep\delta = \beta(\phi) = \beta(\widehat\phi)$. 
Therefore, what we have to show is $\ep\delta = (-1)^{r(\phi)}$, 
which we will prove in the next lemma.
\end{proof}

\begin{lem}\label{r_phi}
With the above notation, 
we have
\[
D_{\tl\GL_N(E)}(\II_{\widehat\phi} \boxtimes \theta_W) \overset{\theta}{=} 
(-1)^{r(\phi)} \II_{\widehat\phi} \boxtimes \theta_W.
\]
\end{lem}

\subsection{An example}
The proof of Lemma \ref{r_phi} is complicated. 
Before giving the proof generally, 
let us discuss a simple but non-trivial case. 
This example showcases the strategy of the general proof. 
\par

Suppose that $E=F$.
Let $\chi_1$, $\chi_2$ be two quadratic characters of $F^\times$, 
and consider $\phi = \chi_1 \oplus \chi_2$. 
Then
\[
\pi_\phi = \pi_{\widehat\phi} = \II_{\widehat\phi} = \chi_1 \times \chi_2 \in \Irr(\GL_2(F)).
\]
We denote the upper triangular Borel subgroup of $\GL_2(F)$ by $B = TU$, 
where $T$ is the diagonal torus. 
If we set $\tl{B} = B \rtimes \pair{\theta}$ and $\tl{T} = T \rtimes \pair{\theta}$,
then by definition
\[
D_{\tl\GL_2(F)}(\pi_\phi \boxtimes \theta_A) 
= \pi_\phi \boxtimes \theta_A - \Ind_{\tl{B}}^{\tl\GL_2(F)}(\Jac_{\tl{B}}(\pi_\phi \boxtimes \theta_A))
\]
in the Grothendieck group $\RR(\tl\GL_2(F))$. 
Since 
\[
(-1)^{r(\phi)} = \left\{
\begin{aligned}
&{-1} \iif \chi_1=\chi_2, \\
&1 \iif \chi_1 \not= \chi_2, 
\end{aligned}
\right. 
\]
the desired equation 
$D_{\tl\GL_2(F)}(\pi_\phi \boxtimes \theta_A) \overset{\theta}{=} (-1)^{r(\phi)}\pi_\phi \boxtimes \theta_A$
is equivalent to 
\[
\Ind_{\tl{B}}^{\tl\GL_2(F)}(\Jac_{\tl{B}}(\pi_\phi \boxtimes \theta_A)) 
\overset{\theta}{=}
\left\{
\begin{aligned}
&2 \cdot \pi_\phi \boxtimes \theta_A \iif \chi_1 = \chi_2, \\
&0 \iif \chi_1 \not= \chi_2.
\end{aligned}
\right. 
\]
\par

Note that $\Jac_{\tl{B}}(\pi_\phi \boxtimes \theta_A)|_{T} = \Jac_{B}(\pi_\phi)$. 
By the Geometric Lemma \cite[Theorem 5.2]{BZ}, 
we have an exact sequence 
\[
\begin{CD}
0 @>>> \chi_2 \boxtimes \chi_1 @>>> \Jac_{B}(\pi_\phi) @>>> \chi_1 \boxtimes \chi_2 @>>> 0
\end{CD}
\]
and hence $\Ind_B^{\GL_2(F)}(\Jac_B(\pi_\phi)) = 2\cdot \pi_\phi$ in $\RR(\GL_2(F))$.
To understand the induced action of $\theta$ on $\Jac_B(\pi_\phi)$, 
we recall the details for this exact sequence. 
The surjection $\Jac_B(\pi_\phi) \twoheadrightarrow \chi_1 \boxtimes \chi_2$
is induced by the evaluation map
\[
\pi_\phi = \chi_1 \times \chi_2 \ni f \mapsto f(\1) \in \chi_1 \boxtimes \chi_2.
\]
The kernel of this map is 
\[
\FF = \{f \in \chi_1 \times \chi_2 \,|\, \Supp(f) \subset Bw_0^{-1}B \}, 
\]
where $w_0 = \left(
\begin{smallmatrix}
0 & 1 \\ -1 & 0
\end{smallmatrix}
\right) 
\in \GL_2(F)$. 
We identify $w_0$ with its image in the Weyl group $W^{\GL_2(F)}$.  
Then for $f \in \FF$, the evaluation of the integral
\[
J_B(w_0, \chi_1 \boxtimes \chi_2)f(\1) = \int_{U}f(w_0^{-1}u)du
\]
converges absolutely. 
Moreover, the map $\FF \ni f \mapsto J_B(w_0, \chi_1 \boxtimes \chi_2)f(\1)$ 
induces an isomorphism 
$\Jac_B(\pi_\phi) \supset \Jac_B(\FF) \xrightarrow{\sim} \chi_2 \boxtimes \chi_1$. 
\par

However, this description of $\Jac_B(\pi_\phi)$ 
does not seem convenient for us. 
We give another description. 
Let $\FF'$ be the subspace of $\pi_\phi = \chi_1 \times \chi_2$ 
consisting of functions $f \in \chi_1 \times \chi_2$
such that $J_B(w_0, \chi_1 \boxtimes \chi_2)f(\1)$ is well-defined. 
This means that 
the meromorphic continuation of 
\[
\C^2 \ni \lambda = (\lambda_1, \lambda_2) \mapsto 
J_B(w_0, \chi_1|\cdot|_F^{\lambda_1} \boxtimes \chi_2|\cdot|_F^{\lambda_2})f_\lambda(\1)
\]
is holomorphic at $\lambda = (0,0)$, 
where $f_\lambda \in \chi_1|\cdot|_F^{\lambda_1} \boxtimes \chi_2|\cdot|_F^{\lambda_2}$
is such that $f_\lambda|_K = f|_K$ with $K = \GL_2(\oo_F)$. 
Then $\FF \subset \FF'$, and 
the map $\FF' \ni f \mapsto J_B(w_0, \chi_1 \boxtimes \chi_2)f(\1)$ also induces a surjection  
\[
\Jac_B(\pi_\phi) \supset \Jac_B(\FF') \twoheadrightarrow \chi_2 \boxtimes \chi_1.
\]
We set $\FF'' = \{f \in \FF'\,|\, J_B(w_0, \chi_1 \boxtimes \chi_2)f(\1) = 0\}$
so that $\Jac_B(\FF')/\Jac_B(\FF'') \cong \chi_2 \boxtimes \chi_1$.
\par

Now we consider two cases separately. 
\begin{description}
\setlength{\itemsep}{10pt}
\item[Case 1]
Suppose that $\chi_1 = \chi_2$.
Then the action $\theta_A$ on $\pi_\phi = \chi_1 \times \chi_2$ is given by 
\[
f \mapsto f \circ \theta.
\]
Since $(f \circ \theta)(\1) = f(\1)$, 
the induced action on $\Jac_B(\pi_\phi)$ preserves the quotient 
$\chi_1 \boxtimes \chi_2$ and acts on it trivially. 
Similarly, since $\theta(w_0) = w_0$ and $\theta(U)=U$, 
the same holds for the subrepresentation $\chi_2 \boxtimes \chi_1$ of $\Jac_B(\pi_\phi)$. 
Hence by inducing these two pieces, we obtain that 
$\Ind_{\tl{B}}^{\tl\GL_2(F)}(\Jac_{\tl{B}}(\pi_\phi \boxtimes \theta_A)) = 2 \cdot \pi_\phi \boxtimes \theta_A$.

\item[Case 2]
Suppose that $\chi_1 \not= \chi_2$.
Then the map $f \mapsto f \circ \theta$ is no longer an action of $\theta$ on $\pi_\phi$.
Instead of this map, we use Theorem \ref{main2}.
Namely, the action $\theta_A$ on $\pi_\phi = \chi_1 \times \chi_2$ can be realized by 
the normalized intertwining operator 
\[
R_B(w_0, \chi_2 \boxtimes \chi_1) 
\circ \theta^*
= 
\theta^* \circ R_B(w_0, \chi_1 \boxtimes \chi_2).
\]
Since $\chi_1 \not= \chi_2$, we see that the normalizing factor
\[
\gamma_A(s,\chi_1 \otimes \chi_2, \psi_F)
\]
is holomorphic and nonzero at $s=0$. 
Hence this action can be written as 
\[
\theta_A = \gamma_A(0,\chi_1 \otimes \chi_2, \psi_F) \cdot 
\theta^* \circ
J_B(w_0, \chi_1 \boxtimes \chi_2).
\]
In particular, the map 
\[
\FF' \ni f \mapsto \theta_A(f)(\1) = \gamma_A(0,\chi_1 \otimes \chi_2, \psi_F) \cdot 
J_B(w_0, \chi_1 \boxtimes \chi_2)f(\1) 
\]
factors through $f \mapsto J_B(w_0, \chi_1 \boxtimes \chi_2)f(\1)$, 
and hence, this map is zero on $\FF''$.
Conversely, for $f \in \chi_1 \times \chi_2$, 
since $R_B(w_0, \chi_1 \boxtimes \chi_2) \circ R_B(w_0, \chi_2 \boxtimes \chi_1) = \id$ 
and $\theta(\1) = \1$, we have 
\[
J_B(w_0, \chi_1 \boxtimes \chi_2)\theta_A(f)(\1) 
= \gamma_A(0,\chi_1 \otimes \chi_2, \psi_F)^{-1} \cdot f(\1).
\]
Hence $\theta_A(f) \in \FF'$ and the map $f \mapsto J_B(w_0, \chi_1 \boxtimes \chi_2)\theta_A(f)(\1)$
factors through $f \mapsto f(\1)$. 
In particular, if $f(\1) = 0$, 
then $J_B(w_0, \chi_1 \boxtimes \chi_2)\theta_A(f)(\1) = 0$.
(Here, it is not true in general that $\theta_A(f) \in \FF$.)
Therefore, the induced action on $\Jac_B(\pi_\phi)$
swaps $\chi_2 \boxtimes \chi_1$ and $\chi_1 \boxtimes \chi_2$. 
By inducing these two pieces, we obtain that 
$\Ind_{\tl{B}}^{\tl\GL_2(F)}(\Jac_{\tl{B}}(\pi_\phi \boxtimes \theta_A)) \overset{\theta}{=} 0$.
\end{description}

Note that Theorem \ref{main2} can be used even when $\chi_1 = \chi_2$. 
However, in this case, $\gamma_A(0,\chi_1 \otimes \chi_2, \psi_F) = 0$ 
and hence $\theta_A(f)(\1) = 0$ for $f \in \FF'$.
Hence the argument in Case 2 does not work for Case 1.
In Case 2, it is trivial that the induced action swaps $\chi_2 \boxtimes \chi_1$ and $\chi_1 \boxtimes \chi_2$
since $(\chi_1 \boxtimes \chi_2) \circ \theta \not= \chi_1 \boxtimes \chi_2$.
However, this idea would not work in general, 
e.g., for $\phi = \chi_1^{\oplus3} \oplus \chi_2$ with $\chi_1 \not= \chi_2$.
On the other hand, the analysis using intertwining operators can be generalized.
\par

\subsection{Proof of Lemma \ref{r_phi}}

Now we prove Lemma \ref{r_phi} generally. 
\begin{proof}[Proof of Lemma \ref{r_phi}]
One can write $\II_{\widehat\phi} = \Ind_{P_0}^{\GL_N(E)}(\pi_{M_0})$ with 
\[
\pi_{M_0} = \rho_{-r} \boxtimes \dots \boxtimes \rho_{-1} \boxtimes 
\rho'_1 \boxtimes \dots \boxtimes \rho'_t \boxtimes \rho_{1} \boxtimes \dots \boxtimes \rho_{r}, 
\]
where 
\begin{itemize}
\item
$P_0$ corresponds to a partition $(d_r, \dots, d_1, d'_1, \dots, d'_t, d_1, \dots, d_r)$ of $N$; 
\item
$\rho_i$ is an irreducible cuspidal representation of $\GL_{d_i}(E)$; 
\item
$\rho_{-i} \cong {}^c\rho_i^\vee$ is the conjugate-dual of $\rho_i$; 
\item
$\rho'_i$ is an irreducible conjugate-self-dual cuspidal representation of $\GL_{d'_i}(E)$ 
such that $\rho'_i \not\cong \rho'_j$ for $1 \leq i < j \leq t$. 
\end{itemize}
Then $r(\phi) = r$. 
\par

Set $W = W^{\GL_N(E)}$.
For a $\theta$-stable parabolic subgroup $P=MN_P$, 
by the geometric lemma (\cite[Theorem 5.2]{BZ}), 
up to a semisimplification, 
we can write
\[
\Jac_P(\II_{\widehat\phi}) = \bigoplus_{w \in W^{M} \bs W / W^{M_0}} J_{\widehat\phi}^w
\]
for some representation $J_{\widehat\phi}^w$ (possibly zero). 
We recall the relation between $\Jac_P(\II_{\widehat\phi})$ and $J_{\widehat\phi}^w$ more precisely. 
Fix a total order $\geq$ on $W^M \bs W / W^{M_0}$
 such that $w' \geq w \implies \dim(P_0 w'^{-1} P) \geq \dim(P_0 w^{-1} P)$. 
For $w \in W^M \bs W / W^{M_0}$, we define $\FF_w$ as the subspace of $\II_{\widehat\phi}$ consisting of 
functions $f$ such that 
\[
\Supp(f) \subset 
\bigcup_{\substack{w' \in W^M \bs W / W^{M_0} \\ w' \geq w}}P_0 w'^{-1} P.
\]
For $w \in W^M \bs W / W^{M_0}$ and $\lambda \in \aa^*_{M_0,\C}$, 
let 
\[
J_{P_0}(w, \pi_{M_0, \lambda}) 
\colon I_{P_0}(\pi_{M_0, \lambda}) \rightarrow I_{P_0^w}(w \pi_{M_0,\lambda})
\]
be the unnormalized intertwining operator defined in Section \ref{sec.NIO}, 
where $P_0^w$ is the standard parabolic subgroup of $\GL_N(E)$
such that $w M_0 w^{-1}$ is its Levi subgroup. 
If $f \in \FF_{w}$ and $\lambda = 0$, then the integral $(J_{P_0}(w, \pi_{M_0})f)(\1)$ converges absolutely. 
Moreover, the map $\FF_w \ni f \mapsto (J_{P_0}(w, \pi_{M_0})f)(\1)$ gives a surjection 
\[
\Jac_P(\II_{\widehat\phi}) \supset \Jac_P(\FF_{w}) \twoheadrightarrow J_{\widehat\phi}^w. 
\]
By varying $w$, 
we obtain a filtration of $\Jac_P(\II_{\widehat\phi})$. 
For details, see \cite[Section 5]{BZ}.
\par

We modify this description for $\Jac_P(\II_{\widehat\phi})$. 
Set $K = \GL_N(\oo_E)$ to be the standard maximal compact open subgroup of $\GL_N(E)$. 
Then for $f \in \II_{\widehat\phi} = I_{P_0}(\pi_{M_0})$, 
one can define $f_\lambda \in I_{P_0}(\pi_{M_0, \lambda})$ 
by requiring $f_\lambda|_K = f|_K$. 
Let $\FF'_w$ be the subspace of $\II_{\widehat\phi} = I_{P_0}(\pi_{M_0})$
consisting of functions $f$ such that 
the meromorphic function 
\[
\lambda \mapsto J_{P_0}(w, \pi_{M_0, \lambda})f_\lambda(\1)
\]
is holomorphic at $\lambda = 0$.
Then $\FF_w \subset \FF'_w$ so that we obtain a well-defined surjection 
\[
\Jac_P(\II_{\widehat\phi}) \supset \Jac_P(\FF'_{w}) \twoheadrightarrow J_{\widehat\phi}^w. 
\]
Its kernel is of the form $\Jac_P(\FF''_w)$, 
where $\FF''_w = \{f \in \FF'_w \,|\, J_{P_0}(w, \pi_{M_0})f(\1) = 0 \}$.
\par

Using Theorem \ref{main2}, 
we realize an action $\theta_W$ on $\II_{\widehat\phi}$ by a twisted intertwining operator. 
Set $N' = d_1'+\dots+d_t'$. 
Consider $\GL_{N'}(E)$ and its standard parabolic subgroup $P_0' = M_0'N_{P_0'}$
corresponding to the partition $(d_1',\dots,d_t')$. 
Define 
\[
\pi_{M_0'} = \rho_1' \boxtimes \dots \boxtimes \rho'_t. 
\]
Then $\Ind_{P_0'}^{\GL_{N'}(E)}(\pi_{M_0'})$
is an irreducible conjugate-self-dual representation of $\GL_{N'}(E)$. 
Let $w_0 \in W(\theta(M'_0),M'_0)$ be the unique element 
such that $w_0(\pi_{M'_0} \circ \theta) \cong \pi_{M'_0}$. 
We regard $w_0$ as an element in $W(\theta(M_0),M_0)$. 
By Theorem \ref{main2} for $\pi_{M_0'}$, we have 
\begin{align*}
\theta_W &= 
I_{P_0}(\tl\pi_{M_0}(w_0 \rtimes \theta)) \circ 
R_{\theta(P_0)}(w_0, \pi_{M_0} \circ \theta) \circ \theta^*
\\&= I_{P_0}(\tl\pi_{M_0}(w_0 \rtimes \theta)) \circ \theta^* \circ R_{P_0}(\theta(w_0), \pi_{M_0}).
\end{align*}
\par

Now the map $w \mapsto \theta(ww_0)$ gives a well-defined involutive action of $\theta$ 
on $W^{M} \bs W / W^{M_0}$ since $w_0^{-1}W^{M_0}w_0 = W^{\theta(M_0)} = \theta(W^{M_0})$
and $w_0\theta(w_0) \in W^{M_0}$. 
Since $P_0^w$ contains $w M_0 w^{-1}$, 
we see that $\theta(P_0^w)$ contains $\theta(ww_0) M_0 \theta(ww_0)^{-1}$ as a Levi subgroup.
One can check that
\[
\left.
\frac{\gamma_A(0, \pi_{M_0, \lambda}, \rho^\vee_{\theta(ww_0)^{-1}\theta(P_0^w) \mid P_0}, \psi_F)}
{\gamma_A(0, \theta(w_0)\pi_{M_0, \lambda}, 
\rho^\vee_{\theta(w)^{-1}\theta(P_0^w) \mid \theta(P_0)}, \psi_F)}
\right|_{\lambda=0}
= \prod_{1 \leq i < j \leq t} \gamma_A(0, \rho'_i \otimes {\rho'_j}^\vee, \psi_E)^{m_{i,j}}
\]
for some $m_{i,j} \in \Z$.
Since $\rho'_i \not\cong \rho'_j$ for $1 \leq i < j \leq t$, 
the right-hand side is in $\C^\times$. 
\par

We denote the evaluation map $f \mapsto f(\1)$ by $\ev_\1$.
For $w \in W^{M} \bs W / W^{M_0}$, we claim that 
\[
(\ev_\1 \circ J_{P_0}(w, \pi_{M_0})) \circ 
(I_{P_0}(\tl\pi_{M_0}(w_0 \rtimes \theta)) \circ \theta^* \circ R_{P_0}(\theta(w_0), \pi_{M_0}))
\]
factors through $\ev_\1 \circ J_{P_0}(\theta(ww_0), \pi_{M_0})$.
For simplicity, let $\VV$ be the space of $\pi_{M_0}$, 
and we regard $\tl\pi_{M_0}(w_0 \rtimes \theta)$ as a linear isomorphism $\Phi \colon \VV \rightarrow \VV$. 
Since 
$R_{\theta(P_0)}(\theta(w), \theta(w_0)\pi_{M_0}) \circ R_{P_0}(\theta(w_0), \pi_{M_0}) 
= R_{P_0}(\theta(ww_0), \pi_{M_0})$
by Proposition \ref{multiplicative}, 
as linear maps,
we have
\begin{align*}
&(\ev_\1 \circ J_{P_0}(w, \pi_{M_0})) \circ 
(I_{P_0}(\tl\pi_{M_0}(w_0 \rtimes \theta)) \circ \theta^* \circ R_{P_0}(\theta(w_0), \pi_{M_0}))
\\&=
\Phi
\circ (\ev_\1 \circ J_{P_0}(w, w_0(\pi_{M_0} \circ \theta)))
\circ (\theta^* \circ R_{P_0}(\theta(w_0), \pi_{M_0}))
\\&=
\Phi \circ (\ev_\1 \circ J_{\theta(P_0)}(\theta(w), \theta(w_0)\pi_{M_0}))
\circ R_{P_0}(\theta(w_0), \pi_{M_0})
\\&= 
\prod_{1 \leq i < j \leq t} \gamma_A(0, \rho'_i \otimes {\rho'_j}^\vee, \psi_E)^{m_{i,j}}
\cdot 
\Phi \circ (\ev_\1 \circ J_{P_0}(\theta(ww_0), \pi_{M_0})).
\end{align*}
This implies that 
$\theta_W(\FF'_{\theta(ww_0)}) \subset \FF'_{w}$ and $\theta_W(\FF''_{\theta(ww_0)}) \subset \FF''_{w}$. These inclusions make $\FF'_w$ preferable to $\FF_w$.
Hence the induced action $\theta_W$ on $\Jac_{P}(\II_{\widehat\phi})$ 
sends $J_{\widehat\phi}^{\theta(ww_0)}$ to $J_{\widehat\phi}^w$. 
Since $w \mapsto \theta(ww_0)$ is an involution on $W^{M} \bs W / W^{M_0}$, 
we see that $\theta_W$ swaps $J_{\widehat\phi}^w$ and $J_{\widehat\phi}^{\theta(ww_0)}$. 
Hence if $\theta(ww_0) \not= w$, then we have
\[
\left(
\Ind_P^{\GL_N(E)}(J_{\widehat\phi}^w) + \Ind_P^{\GL_N(E)}(J_{\widehat\phi}^{\theta(ww_0)}) 
\right) \boxtimes \theta_W
\overset{\theta}{=} 0.
\]
On the other hand, if $\theta(ww_0) = w$, then $\theta_W$ preserves $J_{\widehat\phi}^w$. 
Moreover, the same argument as above shows that  
$\Ind_{\tl{P}}^{\tl\GL_N(E)}(J_{\widehat\phi}^w \boxtimes \theta_W) = \II_{\widehat\phi} \boxtimes \theta_W$
in $\RR(\tl\GL_N(E))$.
Therefore
\[
\Ind_{\tl{P}}^{\tl\GL_N(E)}(\Jac_{\tl{P}}(\II_{\widehat\phi} \boxtimes \theta_W)) \overset{\theta}{=} 
\bigoplus_{\substack{w \in (W^{M} \bs W / W^{M_0})^{\theta} \\ J_{\widehat\phi}^w \not= 0}}
\II_{\widehat\phi} \boxtimes \theta_W, 
\]
where $(W^{M} \bs W / W^{M_0})^{\theta}$ is the subset of the double coset space 
fixed by the action $w \mapsto \theta(ww_0)$. 
\par

Suppose that $P = MN_P$ corresponds to a partition $(n_m, \dots, n_1, n_0, n_1, \dots, n_m)$ of $N$. 
Here, we assume that $n_i > 0$ for $i > 0$, but $n_0$ is possibly zero.
Note that $\dim(A_M^\theta) = m$. 
As in \cite[Section 1.6]{Z}, 
the double coset space $W^{M} \bs W / W^{M_0}$ is canonically identified with 
the set of matrices of the form
\[
A = \begin{pmatrix}
a_{-m,-r} & \ldots & a_{-m, -1} & a'_{-m,1} & \ldots & a'_{-m,t} & a_{-m, 1} & \ldots & a_{-m,r} \\
\vdots & \ddots & \vdots & \vdots & \ddots & \vdots & \vdots & \ddots & \vdots \\
a_{m,-r} & \ldots & a_{m, -1} & a'_{m,1} & \ldots & a'_{m,t} & a_{m, 1} & \ldots & a_{m,r} 
\end{pmatrix}
\in \M_{2m+1, 2r+t}(\Z)
\]
such that 
\begin{enumerate}
\item
all entries are non-negative integers; 
\item
$\sum_{i=-m}^m a_{i,\pm j} = d_j$ for $1 \leq j \leq r$; 
\item
$\sum_{i=-m}^m a'_{i,j} = d'_j$ for $1 \leq j \leq t$; 
\item
$\sum_{j=1}^r(a_{i,-j}+a_{i,j})+\sum_{j=1}^t a'_{i,j} = n_{|i|}$ for $-m \leq i \leq m$.
\end{enumerate}
Since $\rho_{\pm i}$ and $\rho'_j$ are cuspidal, 
we see that $J_{\widehat\phi}^w \not= 0$ if and only if 
\begin{enumerate}
\setcounter{enumi}{4}
\item
$a_{i,\pm j} \in \{0, d_{j}\}$ for $-m \leq i \leq m$ and $1 \leq j \leq r$; and 
\item
$a'_{i,j} \in \{0, d'_{j}\}$ for $-m \leq i \leq m$ and $1 \leq j \leq t$. 
\end{enumerate}
In particular, for each $\pm j$ (\resp $j$), 
there exists a unique $i$ such that $a_{i,\pm j} = d_{j}$ (\resp $a'_{i,j} = d'_j$).
Let $X_P$ be the set of matrices $A$ as above satisfying the conditions (1)--(6). 
For $A \in X_P$, the corresponding $J_{\widehat\phi}^w$ is given by 
\[
J_{\widehat\phi}^w = 
\bigotimes_{i=-m}^m
\left( 
\left(\bigtimes_{\substack{1 \leq  j \leq r \\ a_{i,j}\not=0}} \rho_{j}\right)
\times \left(\bigtimes_{\substack{1 \leq  j \leq t \\ a'_{i,j}\not=0}} \rho'_j \right)
\times \left(\bigtimes_{\substack{1 \leq  j \leq r \\ a_{i,-j}\not=0}} \rho_{-j} \right)
\right)
\in \bigotimes_{i=-m}^m \RR(\GL_{n_i}(E)).
\]
The action $w \mapsto \theta(ww_0)$ on $W^{M} \bs W / W^{M_0}$ 
gives an action of $\theta$ on $X_P$. 
This is given by 
\[
\left\{
\begin{aligned}
&a_{i,j} \mapsto a_{-i,-j} &&(-m \leq i \leq m, 1 \leq j \leq r), \\
&a'_{i,j} \mapsto a'_{-i,j} &&(-m \leq i \leq m, 1 \leq j \leq t).
\end{aligned}
\right.
\]
Therefore, $A$ is fixed by this action if and only if 
\begin{enumerate}
\setcounter{enumi}{6}
\item
$a_{-i,-j} = a_{i,j}$ for $-m \leq i \leq m$ and $1 \leq j \leq r$; and 
\item
$a'_{-i,j} = a'_{i,j}$ for $-m \leq i \leq m$ and $1 \leq j \leq t$. 
\end{enumerate}
Since for each $1 \leq j \leq t$, there is only one $i$ such that $a'_{i,j} \not= 0$, 
we must have $a'_{i,j} = 0$ for $i\not= 0$ and $a'_{0,j} = d'_j$.
\par

Therefore, for a fixed $0 \leq m \leq r$, 
there is a bijection between 
\[
\left\{(P,w) \,\middle|\, 
P= \theta(P) = MN_P,\, \dim (A_M^\theta) = m,\, 
w \in (W^{M} \bs W / W^{M_0})^{\theta},\, J_{\widehat\phi}^w \not= 0\right\}
\]
and 
\[
\left\{
(\{I_i\}_{i=1}^m, \epsilon) \,\middle|\, 
\emptyset \not= I_i \subset \{1,\dots,r\},\, I_i \cap I_{i'} = \emptyset \,(i\not=i'),\,
\epsilon \colon I_1 \sqcup \dots \sqcup I_m \rightarrow \{\pm1\}
\right\}.
\]
If $(P,w)$ corresponds to $(\{I_i\}_{i=1}^m, \epsilon)$, 
then the corresponding matrix $A \in X_P$ is given such that 
\begin{itemize}
\item
if $1 \leq j \leq r$ and $j \in I_i$ for some $1 \leq i \leq m$, 
then $a_{\epsilon(j)i,j} = d_j$ and $a_{i',j} = 0$ for $i' \not= \epsilon(j)i$; 
\item
if $1 \leq j \leq r$ and $j \not\in I_i$ for all $1 \leq i \leq m$, 
then $a_{0,j} = d_j$ and $a_{i',j} = 0$ for $i' \not= 0$; 
\item
if $1 \leq j \leq t$, then $a'_{0,j} = d'_j$ and $a'_{i',j} = 0$ for $i' \not= 0$.
\end{itemize}
Moreover, $J_{\widehat\phi}^w$ is equal to 
\[
\left(\bigtimes_{j \in I_m} \rho_{\epsilon(j)j}\right)
\boxtimes \dots \boxtimes 
\left(\bigtimes_{j \in I_1} \rho_{\epsilon(j)j}\right)
\boxtimes \sigma_0 \boxtimes
\left(\bigtimes_{j \in I_1} \rho_{-\epsilon(j)j}\right)
\boxtimes \dots \boxtimes 
\left(\bigtimes_{j \in I_m} \rho_{-\epsilon(j)j}\right)
\]
for some $\sigma_0 \in\Irr(\GL_{n_0}(E))$.
In particular, $P$ corresponds to the partition 
\[
\left(
\sum_{j \in I_m} d_j, \dots, \sum_{j \in I_1} d_j, n_0, \sum_{j \in I_1} d_j, \dots, \sum_{j \in I_m} d_j
\right). 
\]
By setting $k_i = |I_i|$, we see that 
\begin{align*}
&\left|\left\{
(\{I_i\}_{i=1}^m, \epsilon) \,\middle|\, 
\emptyset \not= I_i \subset \{1,\dots,r\},\, I_i \cap I_{i'} = \emptyset \,(i\not=i'),\,
\epsilon \colon I_1 \sqcup \dots \sqcup I_m \rightarrow \{\pm1\}
\right\}\right|
\\&=
\sum_{\substack{k_1,\dots,k_m \geq 1 \\ k_1+\dots+k_m \leq r}}
2^{k_1+\dots+k_m} \frac{r(r-1)\cdots(r-k_1-\dots-k_m+1)}{k_1! \cdots k_m!}
\\&= \left.\left(\frac{d}{dx}\right)^r
e^x(e^{2x}-1)^m \right|_{x=0}.
\end{align*}
Therefore, we have 
\begin{align*}
D_{\tl\GL_N(E)}(\II_{\widehat\phi} \boxtimes \theta_W) &\overset{\theta}{=}
\left(\sum_{m=0}^r (-1)^m
\left.\left(\frac{d}{dx}\right)^r
e^x(e^{2x}-1)^m \right|_{x=0}
\right) 
\II_{\widehat\phi} \boxtimes \theta_W
\\&= 
\left(
\left.\left(\frac{d}{dx}\right)^r
e^{-x} (1-(1-e^{2x})^{r+1})
\right|_{x=0}
\right)
\II_{\widehat\phi} \boxtimes \theta_W
\\&= 
\left(
\left.\left(\frac{d}{dx}\right)^r
e^{-x} 
\right|_{x=0}
\right)
\II_{\widehat\phi} \boxtimes \theta_W
= (-1)^r \II_{\widehat\phi} \boxtimes \theta_W.
\end{align*}
Here, we use the fact that $e^{-x}(1-e^{2x})^{r+1}$ has a zero at $x=0$ with order $r+1$.
Since $r(\phi) = r$, we obtain the assertion.
\end{proof}

\subsection{ECR vs.~Aubert duality}
Next, we consider Aubert duality for classical groups. 
We will use the notations in Section \ref{sec.groups}.
\par

Let $G$ be one of the following quasi-split classical groups
\[
\SO_{2n+1}(F), \quad \Sp_{2n}(F), \quad
\O_{2n}(F), \quad \U_n.
\]
For $\pi \in \Rep(G^\circ)$, its Aubert dual is defined by 
\[
D_{G^\circ}(\pi) = \sum_{P^\circ} (-1)^{\dim(A_{M^\circ})} \Ind_{P^\circ}^{G^\circ} (\Jac_{P^\circ}(\pi))
\]
in the Grothendieck group $\RR(G^\circ)$, 
where $P^\circ = M^\circ N_{P}$ runs over the set of standard parabolic subgroups of $G^\circ$. 
Note that $A_{G^\circ} = \{1\}$ unless $G^\circ = \SO_2(F)$.
If $\pi \in \Irr(G^\circ)$, then $D_{G^\circ}(\pi) = \beta(\pi) \hat\pi$ 
for an irreducible representation $\hat\pi$ with a sign $\beta(\pi) \in \{\pm1\}$
(see Theorem \ref{AD} (2), (3)).
As in Theorem \ref{AD} (2), 
this sign is given by $\beta(\pi) = (-1)^{\dim(A_{M^\circ})}$, 
where $P^\circ = M^\circ N_{P}$ is a minimal standard parabolic subgroup of $G^\circ$ 
such that $\Jac_{P^\circ}(\pi) \not= 0$. 
Such a $P^\circ$ may not be unique, but the sign $(-1)^{\dim(A_{M^\circ})}$ is well-defined.
\par

When $G=\O_{2n}(F)$, we also consider twisted Aubert duality defined in Definition \ref{def.TAD}. 
Fix a Borel subgroup $B^\circ =T^\circ U$ of $G^\circ = \SO_{2n}(F)$. 
If we denote the normalizer of $(T^\circ, B^\circ)$ in $G$ by $T$, 
then $T \cap G^\circ = T^\circ$ and $T/T^\circ \cong G/G^\circ$. 
Fix a representative $\epsilon \in T$ of the non-trivial coset in $T/T^\circ$ as in Section \ref{sec.groups}.  
For $\pi \in \Rep(G)$, we define 
\[
D_G(\pi) = \sum_{P^\circ} (-1)^{\dim(A_{M^\circ}^\epsilon)} \Ind_P^{G} (\Jac_P(\pi)), 
\] 
where 
\begin{itemize}
\item
$P^\circ = M^\circ N_{P}$ now runs over the set of standard parabolic subgroups of $G^\circ = \SO_{2n}(F)$
which are stable under the conjugation action of $\epsilon$; 
\item
$P = P^\circ \cdot T \subset G$; 
\item
$A_{M^\circ}^\epsilon$ is the subgroup of $A_{M^\circ}$ fixed by the action of $\epsilon$. 
\end{itemize}
The Jacquet module $\Jac_P(\pi)$ is defined as usual (see Section \ref{sec.rep}). 
It is a representation of $M = M^\circ \cdot T$ and satisfies 
\[
\Jac_P(\pi)|_{M^\circ} = \Jac_{P^\circ}(\pi|_{G^\circ}).
\]
By Propositions \ref{TAD} (3) and \ref{prop2.13},
for $\pi \in \Irr(G)$, one can find $\hat\pi \in \Irr(G)$ 
such that the trace of $D_G(\pi)-\beta(\pi)\hat\pi$ is zero on $f_G \in C_c^\infty(G \setminus G^\circ)$.
Here
\[
\beta(\pi) = \left\{
\begin{aligned}
&\beta(\pi_1)=\beta(\pi_2) \iif \pi|_{G^\circ} = \pi_1 \oplus \pi_2, \\
&\beta(\pi|_{G^\circ}) \iif \text{$\pi|_{G^\circ}$ is irreducible}.
\end{aligned}
\right. 
\]
Note that when $\pi|_{G^\circ} = \pi_1 \oplus \pi_2$, 
if $\Jac_{P}(\pi_1) \not= 0$, then $\Jac_{P'}(\pi_2) \not= 0$ with $P'$ the conjugate of $P$ by $\epsilon$.

\begin{ex}
Let us consider $G = \O_2(F)$. 
Then $G^\circ = \SO_2(F)$ is a torus so that it has only one parabolic subgroup $P^\circ =G^\circ$. 
For $\pi \in \Irr(G)$, we have
\begin{align*}
\beta(\pi) = (-1)^{\dim(A_{G^\circ})} &= \left\{
\begin{aligned}
&{-1} \iif \text{$G^\circ = \SO_2(F)$ is split}, \\
&1 \other, 
\end{aligned}
\right.
\\
(-1)^{\dim(A_{G^\circ}^\epsilon)} &= 1.
\end{align*}
Hence $D_G(\pi) = \pi$. 
Since $\hat\pi$ is equal to $\beta(\pi)D_G(\pi)$ on $\O_2(F) \setminus \SO_2(F)$, 
we see that $\hat\pi \not= \pi$ if and only if $G^\circ = \SO_2(F)$ is split.
\end{ex}
\par

Now we compare Aubert duality for $G$ 
with twisted Aubert duality for $\tl\GL_N(E)$. 
To do this, we consider the following hypothesis.
\begin{hyp}\label{ECR}
Fix an $A$-parameter $\psi$ for $G$. 
Then there exist multi-sets $\Pi_{\psi}$ and $\Pi_{\widehat\psi}$ over $\Irr(G)$ 
equipped with $\pair{\cdot,\pi}_{\psi}$ and $\pair{\cdot,\pi'}_{\widehat\psi}$
satisfying \eqref{ECR1} and \eqref{ECR2} in Section \ref{sec.Arthur}.
Moreover, we assume that for any proper Levi subgroup $M$ of $G$ and any $A$-parameter $\psi_M$ for $M$, 
there exists a multi-set $\Pi_{\psi_M}$ over $\Irr(M)$ equipped with $\pair{\cdot,\pi_M}_{\psi_M}$ 
satisfying \eqref{ECR1} and \eqref{ECR2} in Section \ref{sec.Arthur}.
\end{hyp}

\begin{rem}
Notice that Hypothesis \ref{ECR} does not require 
members in $A$-packets to be unitary. 
For a proper Levi subgroup $M$, 
Hypothesis \ref{ECR} assumes Arthur's results for all $A$-parameters $\psi_M$, 
whereas, for $G$, 
it assumes only for a fixed $A$-parameter $\psi$ and its dual $\widehat\psi$. 
In particular, if $\psi = \phi$ is a tempered $L$-parameter for $G$, 
after establishing \eqref{ECR1} and \eqref{ECR2} for $\widehat\phi$ in the next section, 
one can use results in this section for $\phi$ and $\widehat\phi$. 
\end{rem}

\begin{lem}\label{beta}
Fix $\psi \in \Psi(G)$. 
Assume Hypothesis \ref{ECR}.
\begin{enumerate}
\item
The $A$-packet $\Pi_{\widehat\psi}$ is given by 
\[
\Pi_{\widehat\psi} = \{\hat\pi \,|\, \pi \in \Pi_\psi\}.
\]
Moreover, 
for $\pi \in \Pi_\psi$, 
we have $\beta(\psi) \pair{s_{\widehat\psi}, \hat\pi}_{\widehat\psi} = \pair{s_\psi, \pi}_\psi \beta(\pi)$. 

\item
For $\pi \in \Pi_\psi$ and $s \in A_\psi$, we have
\[
\frac{\pair{\hat{s}, \hat\pi}_{\widehat\psi}}{\pair{s,\pi}_\psi}
= \frac{\beta(\psi)}{\beta(\psi_+)\beta(\psi_-)}, 
\]
where $\hat{s} \in A_{\widehat\psi}$ is the element corresponding to $s$ 
via the canonical identification 
\[
A_\psi \xrightarrow{\sim} A_{\widehat\psi}, \, e(\rho,a,b) \mapsto e(\rho,b,a), 
\]
and $\psi_\pm$ is given by $s$ as in \eqref{ECR2} in Section \ref{sec.Arthur}.
\end{enumerate}
\end{lem}
\begin{proof}
We show (1). 
Since (twisted) Aubert duality commutes with the twisted endoscopic character identity
(see \cite[(A.1)]{X2}), 
when $\tl{f} \in C_c^\infty(\GL_N(E) \rtimes \theta)$ and $f_G \in C_c^\infty(G^\circ)$ have matching orbital integrals, we have
\begin{align*}
\beta(\psi)\sum_{\hat\pi \in \Pi_{\widehat\psi}} 
\pair{s_{\widehat\psi}, \hat\pi}_{\widehat\psi}\Theta_{\hat\pi}(f_G)
&= \beta(\psi) (G:G^\circ)\Theta_{\tl\pi_{\widehat\psi}}(\tl{f})
\\&= (G:G^\circ)\Theta_{D_{\tl\GL_N(E)}(\tl\pi_{\psi})}(\tl{f})
\\&= \sum_{\pi \in \Pi_\psi} \pair{s_\psi, \pi}_\psi \Theta_{D_{G^\circ}(\pi)}(f_G)
\\&= \sum_{\pi \in \Pi_\psi} \pair{s_\psi, \pi}_\psi \beta(\pi) \Theta_{\hat\pi}(f_G).
\end{align*}
By the linear independence of the characters $\Theta_\pi$ 
together with the surjectivity of $\tl{f} \mapsto f_G$, 
we see that $\Pi_{\widehat\psi} = \{\hat\pi \,|\, \pi \in \Pi_\psi\}$.
Moreover, comparing the coefficients, 
we have $\beta(\psi) \pair{s_{\widehat\psi}, \hat\pi}_{\widehat\psi} = \pair{s_\psi, \pi}_\psi \beta(\pi)$. 
\par

Next, we show (2). 
Similar to (1), by \cite[Theorem 1.5]{Hi} or \cite[(A.1)]{X2}, 
when $f_G \in C_c^\infty(G)$ and $f_{G_+} \otimes f_{G_-} \in C_c^\infty(G^\circ_+ \times G^\circ_-)$ 
have matching orbital integrals, 
we have
\begin{align*}
&\frac{1}{(G:G^\circ)}\sum_{\pi \in \Pi_\psi} \pair{s \cdot s_\psi, \pi}_\psi \beta(\pi)\Theta_{\hat\pi}(f_G) 
\\&\quad= 
\frac{1}{(G:G^\circ)}\sum_{\pi \in \Pi_\psi} \pair{s \cdot s_\psi, \pi}_\psi \Theta_{D_{G^\bullet}(\pi)}(f_G) 
\\&\quad= 
\prod_{\kappa \in \{\pm\}} 
\frac{1}{(G_\kappa:G_\kappa^\circ)}
\sum_{\pi_\kappa \in \Pi_{\psi_\kappa \otimes \eta_\kappa}} 
\pair{s_{\psi_\kappa \otimes \eta_\kappa}, \pi_\kappa}_{\psi_\kappa \otimes \eta_\kappa} 
\Theta_{D_{G^\circ_\kappa}(\pi_\kappa)}(f_{G_\kappa}) 
\\&\quad= 
\prod_{\kappa \in \{\pm\}} 
\frac{1}{(G_\kappa:G_\kappa^\circ)}
\sum_{\pi_\kappa \in \Pi_{\psi_\kappa \otimes \eta_\kappa}} 
\pair{s_{\psi_\kappa \otimes \eta_\kappa}, \pi_\kappa}_{\psi_\kappa} \beta(\pi_\kappa)
\Theta_{\hat\pi_\kappa}(f_{G_\kappa}). 
\end{align*}
Here, we set 
\[
D_{G^\bullet} = \left\{
\begin{aligned}
&D_G \iif \text{$G = \O_{2n}(F)$ and $s \not\in A_\psi^+$}, \\
&D_{G^\circ} \other, 
\end{aligned}
\right. 
\]
and we assume that $f_G|_{G^\circ} = 0$ if $G = \O_{2n}(F)$ and $s \not\in A_\psi^+$, 
whereas $f_G \in C_c^\infty(G^\circ)$ otherwise.
Using (1) and Lemma \ref{beta_twist}, we have 
\begin{align*}
&\frac{\beta(\psi)}{(G:G^\circ)}
\sum_{\pi \in \Pi_\psi} \pair{s, \pi}_\psi \pair{s_{\widehat\psi}, \hat\pi}_{\widehat\psi}\Theta_{\hat\pi}(f_G) 
\\&\quad= 
\prod_{\kappa \in \{\pm\}} 
\frac{\beta(\psi_\kappa \otimes \eta_\kappa)}{(G_\kappa:G_\kappa^\circ)}
\sum_{\pi_\kappa \in \Pi_{\psi_\kappa \otimes \eta_\kappa}} 
\pair{s_{\widehat\psi_\kappa \otimes \eta_\kappa}, \hat\pi_\kappa}_{\widehat\psi_\kappa \otimes \eta_\kappa} 
\Theta_{\hat\pi_\kappa}(f_{G_\kappa})
\\&\quad= 
\frac{\beta(\psi_+)\beta(\psi_-)}{(G:G^\circ)}
\sum_{\hat\pi \in \Pi_{\widehat\psi}} \pair{\hat{s} \cdot s_{\widehat\psi}, \hat\pi}_{\widehat\psi}\Theta_{\hat\pi}(f_G). 
\end{align*}
Comparing the coefficients, 
we have $\beta(\psi)\pair{s, \pi}_\psi = \beta(\psi_+)\beta(\psi_-)\pair{\hat{s}, \hat\pi}_{\widehat\psi}$, 
as desired.
\end{proof}

By Proposition \ref{c-temp}, we know $\beta(\psi)$ for tempered parameters $\psi = \phi$. 
It is useful to state the following result. 

\begin{cor}\label{+-temp}
Let $\phi$ be a tempered $L$-parameter for $G$. 
Assume the existence of $A$-packets $\Pi_\phi$ and $\Pi_{\widehat\phi}$ 
associated to $\phi$ and $\widehat\phi$ which satisfy \eqref{ECR1} and \eqref{ECR2}.
Then we have
\[
\frac{\pair{\hat{s}, \hat\pi}_{\widehat\phi}}{\pair{s,\pi}_\phi}
= (-1)^{r(\phi)-r(\phi_+)-r(\phi_-)}. 
\]
In particular, if $s = e(\rho, d, 1)$, then 
\[
\frac{\pair{e(\rho, 1, d), \hat\pi}_{\widehat\phi}}{\pair{e(\rho, d, 1),\pi}_\phi}
= 
\left\{
\begin{aligned}
&1 \iif d \equiv 0 \bmod 2, \\
&{-(-1)^{m}} \iif d \equiv 1 \bmod 2, 
\end{aligned}
\right. 
\]
where $m$ is the number of irreducible constituents of $\phi$ of the form $\rho \boxtimes S_a$ for some $a \geq 1$
(counting with multiplicity). 
\end{cor}
\begin{proof}
The first assertion follows from Lemma \ref{beta} (2) and Proposition \ref{c-temp}. 
Suppose that $s = e(\rho, d, 1) \in A_{\phi}$. 
If we write 
\[
\phi = \phi' \oplus \bigoplus_{i=1}^t \rho \boxtimes S_{d_i}
\]
such that $d_i \equiv d \bmod 2$ and 
$\phi' \not\supset \rho \boxtimes S_a$ for any $a \equiv d \bmod 2$, 
then by definition
\[
r(\phi) = r(\phi') + \left\{
\begin{aligned}
&\sum_{i=1}^t\half{d_i} \iif d \equiv 0 \bmod 2, \\
&\sum_{i=1}^t\half{d_i-1} + \left[\half{m}\right] \iif d \equiv 1 \bmod 2,
\end{aligned}
\right. 
\]
where $[x]$ denotes the largest integer not greater than $x$.
Hence 
\[
r(\phi)-r(\phi_+)-r(\phi_-) = 
\left\{
\begin{aligned}
&0 \iif d \equiv 0 \bmod 2, \\
&\left[\half{m}\right] - \left[\half{m-1}\right] \iif d \equiv 1 \bmod 2.
\end{aligned}
\right. 
\]
This is equal to $1$ if and only if $d$ is odd and $m$ is even.
\end{proof}

\begin{rem}\label{converse}
Notice that the same proof shows 
the converse of Lemma \ref{beta} in the following sense.
Namely, if we assume \eqref{ECR1} and Lemma \ref{beta} (1) 
(\resp \eqref{ECR2} and Lemma \ref{beta} (2)) for an $A$-parameter $\psi$, 
then we obtain \eqref{ECR1} (\resp \eqref{ECR2}) for the dual $A$-parameter $\widehat\psi$. 
In fact, \eqref{ECR1} and \eqref{ECR2} for $A$-parameters of the form $\psi = \widehat\phi$,
where $\phi$ is a tempered $L$-parameter,
will be proven in this way. 
See Theorem \ref{ECR-cotemp} below. 
\end{rem}

Note that the same statement as Corollary \ref{+-temp}
was recently given by Liu--Lo--Shahidi \cite[Theorem 5.9]{LLS}. 
\par

We shall give an example of this corollary for $G=\O_4(F)$. 
In this example, for $A, B \in \RR(G_0)$, we write $A \leq B$ 
if $B-A$ is a non-negative combination of irreducible representations.

\begin{ex}
Write $G = \O_4(F)$.
We denote by $B^\circ = T^\circ U$ the standard Borel subgroup of $G^\circ$ 
so that its Levi is isomorphic to $T^\circ = \GL_1(F) \times \SO_2(F)$.
This is the unique proper standard parabolic subgroup 
which is stable under the conjugate action of $T = \GL_1(F) \times \O_2(F)$. 
Set $B = TU$. 
For simplicity, write $e_\chi = e(\chi,1,1) \in A_\phi$ for some $\phi$ containing $\chi$. 
\begin{enumerate}
\item
Let $\chi$ and $\chi'$ be quadratic characters of $F^\times$ with $\chi \not= \chi'$. 
Consider 
\[
\phi = \chi \oplus \chi \oplus \chi' \oplus \chi' \in \Phi(G)
\]
so that $G^\circ = \SO_4(F)$ is split. 
Then $|\AA_\phi| = 4$ and $|\Sc_\phi| = 2$. 
We can write $\Pi_\phi = \{\pi, \pi \otimes \det, \pi', \pi' \otimes \det\}$ such that
\begin{align*}
\pair{e_\chi, \pi}_\phi = 1, &\quad \pair{e_{\chi'}, \pi}_\phi = 1, \\
\pair{e_\chi, \pi \otimes \det}_\phi = -1, &\quad \pair{e_{\chi'}, \pi \otimes \det}_\phi = -1, \\
\pair{e_\chi, \pi'}_\phi = -1, &\quad \pair{e_{\chi'}, \pi'}_\phi = 1, \\
\pair{e_\chi, \pi' \otimes \det}_\phi = 1, &\quad \pair{e_{\chi'}, \pi' \otimes \det}_\phi = -1. 
\end{align*}
Similarly, write $\Pi_{\chi \oplus \chi} = \{\pi_\chi, \pi_\chi \otimes \det\}$ 
and $\Pi_{\chi' \oplus \chi'} = \{\pi_{\chi'}, \pi_{\chi'} \otimes \det\}$ 
such that $\pair{\cdot, \pi_\chi}_{\chi \oplus \chi} = \1$ and $\pair{\cdot, \pi_{\chi'}}_{\chi' \oplus \chi'} = \1$. 
Then 
\[
\Ind_{B}^{G}(\chi \boxtimes \pi_{\chi'}) = \pi \oplus \pi', 
\quad
\Ind_B^G(\chi' \boxtimes \pi_\chi) = \pi \oplus (\pi' \otimes \det).
\]
Hence, in the Grothendieck group $\RR(T)$, we have
\begin{align*}
\Jac_{B}(\pi) &\leq 2\chi \otimes \pi_{\chi'} + \chi' \otimes \pi_{\chi} + \chi' \otimes (\pi_{\chi} \otimes \det), \\
\Jac_{B}(\pi) &\leq 2\chi' \otimes \pi_{\chi} + \chi \otimes \pi_{\chi'} + \chi \otimes (\pi_{\chi'} \otimes \det)
\end{align*}
so that 
\[
\Jac_{B}(\pi) = \chi \otimes \pi_{\chi'} + \chi' \otimes \pi_{\chi}.
\]
Since $\beta(\pi) = 1$, we have
\[
\hat\pi \overset{\theta}{=} D_G(\pi) = \pi - (\pi + \pi') - (\pi + (\pi' \otimes \det)) \overset{\theta}{=} \pi \otimes \det. 
\]
One can check that 
\[
\frac{\pair{e_\chi, \pi \otimes \det}_\phi}{\pair{e_\chi, \pi}_\phi} 
= \frac{\pair{e_{\chi'}, \pi \otimes \det}_\phi}{\pair{e_{\chi'}, \pi}_\phi} 
= -1, 
\]
which is the statement of Corollary \ref{+-temp}. 
Similarly, we have $\hat\pi' = \pi' \otimes \det$, and one can also check Corollary \ref{+-temp} for this case.

\item
Let $\chi$, $\chi'$ and $\chi''$ be distinct quadratic characters of $F^\times$. 
Consider 
\[
\phi = \chi \oplus \chi \oplus \chi' \oplus \chi'' \in \Phi(G)
\]
so that $G^\circ = \SO_4(F)$ is not split. 
Then $|\AA_\phi| = 4$ and $|\Sc_\phi| = 2$. 
We can write $\Pi_\phi = \{\pi, \pi \otimes \det, \pi', \pi' \otimes \det\}$ such that
\begin{align*}
\pair{e_\chi, \pi}_\phi = 1, &\quad \pair{e_{\chi'}, \pi}_\phi = 1, \quad \pair{e_{\chi''}, \pi}_\phi = 1\\
\pair{e_\chi, \pi'}_\phi = -1, &\quad \pair{e_{\chi'}, \pi'}_\phi = 1, \quad \pair{e_{\chi''}, \pi'}_\phi = 1.
\end{align*}
Similarly, write $\Pi_{\chi' \oplus \chi''} = \{\pi_0, \pi_0 \otimes \det\}$ with $\pair{\cdot, \pi_0}_{\chi' \oplus \chi''} = \1$. 
Then $\Ind_B^G(\chi \boxtimes \pi_0) = \pi \oplus \pi'$ so that $\Jac_{B}(\pi) = \chi \otimes \pi_0$.
Since $\beta(\pi)=-1$, we have 
\[
\hat\pi \overset{\theta}{=} -D_G(\pi) = -\pi+(\pi+\pi') = \pi'.
\]
On the other hand, 
since 
\[
(-1)^{r(\phi)-r(\phi_+)-r(\phi_-)}
= \left\{
\begin{aligned}
&{-1} \iif s = e_\chi, \\
&1 \iif s = e_{\chi'}, e_{\chi''},
\end{aligned}
\right. 
\]
we get Corollary \ref{+-temp} for $\pi$.

\item
Let $\chi$ and $\chi'$ be quadratic characters of $F^\times$ with $\chi \not= \chi'$. 
Consider 
\[
\phi = \chi \oplus \chi \oplus \chi \oplus \chi' \in \Phi(G)
\]
so that $G^\circ = \SO_4(F)$ is not split. 
Then $|\AA_\phi| = 2$ and $|\Sc_\phi| = 1$. 
We can write $\Pi_\phi = \{\pi, \pi \otimes \det\}$ such that $\pair{\cdot, \pi}_\phi = \1$. 
Then $\pi = \Ind_B^G(\chi \boxtimes \pi_0)$ 
with $\pi_0 \in \Pi_{\chi \oplus \chi'}$ such that $\pair{\cdot, \pi_0}_{\chi \oplus \chi'} = \1$.
Hence $\Jac_{B}(\pi) = 2 \chi \otimes \pi_0$.
Since $\beta(\pi) = -1$, we have
\[
\hat\pi \overset{\theta}{=} -D_G(\pi) = -\pi+2\pi = \pi.
\]
On the other hand, 
since 
\[
(-1)^{r(\phi)-r(\phi_+)-r(\phi_-)} = 1
\]
for $s = e_\chi$ and $e_\chi'$, 
we obtain Corollary \ref{+-temp} for $\pi$.

\item
Let $\chi$ be a quadratic character of $F^\times$. 
Consider 
\[
\phi = \chi \oplus \chi \oplus \chi \oplus \chi \in \Phi(G)
\]
so that $G^\circ = \SO_4(F)$ is split. 
Then $|\AA_\phi| = 2$ and $|\Sc_\phi| = 1$. 
We can write $\Pi_\phi = \{\pi, \pi \otimes \det\}$ such that $\pair{\cdot, \pi}_\phi = \1$. 
Then $\pi = \Ind_B^G(\chi \boxtimes \pi_\chi)$ with $\pi_\chi \in \Pi_{\chi \oplus \chi}$ 
such that $\pair{\cdot, \pi_\chi}_{\chi \oplus \chi} = \1$.
Note that 
\[
\Jac_{B}(\pi) = 3 \chi \otimes \pi_\chi + \chi \otimes (\pi_\chi \otimes \det). 
\]
Since $\beta(\pi) = 1$, we have
\[
\hat\pi \overset{\theta}{=} D_G(\pi) = \pi-(3\pi+(\pi \otimes \det)) \overset{\theta}{=} -\pi \overset{\theta}{=} \pi \otimes \det.
\]
On the other hand, 
since 
\[
(-1)^{r(\phi)-r(\phi_+)-r(\phi_-)} = -1
\]
for $s = e_\chi$, 
we get Corollary \ref{+-temp} for $\pi$.
\end{enumerate}
\end{ex}

\subsection{Computation of $\beta(\psi)$}
To show Theorem \ref{main3} (2), we need 
to compare $\pair{\hat{s}, \hat\pi}_{\widehat\psi}$ with $\pair{s,\pi}_\psi$ 
in a little more general situation than the tempered case.
The results in this subsection will only be used in
the proof of Lemma \ref{8.10}. 
\par

In the next proposition, 
we compute $\beta(\psi)$ when $\psi$ is irreducible as a representation of $W_F \times \SL_2(\C) \times \SL_2(\C)$
assuming Hypothesis \ref{ECR}. 

\begin{prop}\label{c-irred}
Assume Hypothesis \ref{ECR}. 
If $\psi$ is irreducible and conjugate-self-dual, then $\beta(\psi) = (-1)^{r(\psi)}$.
\end{prop}
\begin{proof}
Using Lemma \ref{beta_twist},
up to replacing $\psi$ with $\psi \otimes \eta$ for some $\eta$ if necessary, 
we may assume that $\psi \in \Psi(G)$ for some classical group $G$. 
\par

Since $\psi$ is irreducible, we have $\AA_\psi \cong \AA_{\widehat\psi} = \{1\}$. 
By Lemma \ref{beta} (1), we have $\beta(\psi) = \beta(\pi)$ for every $\pi \in \Pi_\psi$.
By \cite[Proposition 7.4.1]{Ar} and \cite[Proposition 8.4.1]{Mok} 
which we can use because we are assuming Hypothesis \ref{ECR}, 
we can take $\pi$ to be an element in the associated $L$-packet $\Pi_{\phi_\psi}$. 
If $\pi$ is the unique element in this $L$-packet corresponding to the trivial character of $\AA_{\phi_\psi}$, 
and if we write $\psi = \rho \boxtimes S_a \boxtimes S_b$, 
then by Theorem \ref{Dk}, we have 
\[
\pi \hookrightarrow \rho|\cdot|_E^{x_1} \times \dots \times \rho|\cdot|_E^{x_n} \rtimes \sigma
\]
for some supercuspidal representation $\sigma$ and real numbers $x_1,\dots,x_n \in \R$, 
where $n = [\half{ab}]$.
For the notation of the parabolically induced representations, see Section \ref{sec.derivative}. 
Since $r(\psi) = [\half{ab}]$, we have $\beta(\psi) = \beta(\pi) = (-1)^{r(\psi)}$.
\end{proof}

By Propositions \ref{c-temp} and \ref{c-irred}, 
we know $\beta(\psi) = (-1)^{r(\psi)}$ when $\psi$ is irreducible or (co-)tempered. 
We will need a few more cases.
In the following proposition, for $\psi \colon W_E \times \SL_2(\C) \times \SL_2(\C) \rightarrow \GL_N(\C)$,
we define $\psi^D$ and $\psi^A$ by 
\[
\psi^D(w,g_1,g_2) = \psi(w,g_1,g_1), 
\quad
\psi^A(w,g_1,g_2) = \psi(w,g_2,g_2).
\]
Recall that $S_a$ is the unique $a$-dimensional irreducible algebraic representation of $\SL_2(\C)$.

\begin{prop}\label{psiA}
Fix an $A$-parameter $\psi$ of $G$ of good parity, and assume Hypothesis \ref{ECR}. 
Suppose that $\psi = \phi_1 \boxtimes S_1 \oplus \phi_2 \boxtimes S_2$
for some representations $\phi_1,\phi_2$ of $W_E \times \SL_2(\C)$
(which can be zero).
If we decompose $\psi = \oplus_{i=1}^t \psi_i$ into irreducible representations, 
then we have
\[
\beta(\psi) = (-1)^{r(\psi)}
\prod_{1 \leq i < j \leq t}
\left.\frac{\gamma_A(s, \psi_i^A \otimes {}^c\psi_j^A, \psi_E)}
{\gamma_A(s, \widehat\psi_i \otimes {}^c\widehat\psi_j, \psi_E)}\right|_{s=0}.
\]
\end{prop}
\begin{proof}
Note that the product of gamma factors is independent of the order of the irreducible components of $\psi$
by Proposition \ref{gamma_c}.
\par

We prove the assertion by induction on $\dim(\phi_2)$. 
If $\phi_2 = 0$, then $\psi = \phi_1$ is tempered and the assertion is Proposition \ref{c-temp}
since $\psi_i^A = \widehat\psi_i$ for any $1 \leq i \leq t$ in this case.
\par

Suppose that $\phi_2 \not= 0$. 
Choose an irreducible subrepresentation $\phi_0 \subset \phi_2$. 
Set $\psi_0 = \phi_0 \boxtimes S_2$ and $\psi' = \psi- \phi_0 \boxtimes S_2$. 
Hence $\pi_\psi = \pi_{\psi'} \times \pi_{\psi_0}$.
We denote the standard modules of $\pi_{\psi_0}$ and $\pi_{\psi'}$ by $\II_{\psi_0}$ and $\II_{\psi'}$, respectively. 
By Lemma \ref{psiD}, we have the following diagram
\[
\xymatrix{
&\II_{\psi'} \times \pi_{\psi_0^D} \ar@{->>}[d] \ar@{^(_->}[r] 
& \II_{\psi'} \times \II_{\psi_0} \ar@{->>}[d] \\
0 \ar@{->}[r] &\pi_{\psi'} \times \pi_{\psi_0^D} \ar@{->}[r]
& \pi_{\psi'} \times \II_{\psi_0} \ar@{->}[r]
& \pi_{\psi'} \times \pi_{\psi_0} \ar@{->}[r] & 0. 
}
\]
Here, the bottom sequence is exact 
since $\II_{\psi_0} = \pi_{\phi_0}|\cdot|_E^{\half{1}} \times \pi_{\phi_0}|\cdot|_E^{-\half{1}}$ is of length two.
By the same argument as Lemma \ref{End}, we have
\[
\dim_\C(\End_{\GL_N(E)}(\pi_{\psi'} \times \II_{\psi_0})) = 1.
\]
Indeed, the canonical map 
\[
\End_{\GL_N(E)}(\pi_{\psi'} \times \II_{\psi_0}) \rightarrow 
\End_{\GL_N(E)}(\pi_{\psi'} \times \pi_{\psi_0}) \cong \C
\]
is injective since $\pi_{\psi'} \times \pi_{\psi_0}$ is the unique irreducible quotient of $\pi_{\psi'} \times \II_{\psi_0}$
and appears in $\pi_{\psi'} \times \II_{\psi_0}$ as a subquotient with multiplicity one. 
By the functoriality of the Aubert involution (Theorem \ref{AD} (1)), 
we see that $\End_{\GL_N(E)}(\hat\pi_{\psi'} \times \widehat{\II}_{\psi_0})$
is also one-dimensional.
\par

The action $\theta_W$ of $\theta$ on the standard module $\II_{\psi'} \times \II_{\psi_0}$
which fixes a Whittaker functional 
gives an action $\theta_W$ on $\pi_{\psi'} \times \II_{\psi_0}$. 
This is the unique action inducing Arthur's actions $\theta_A$ both on $\pi_{\psi'} \times \pi_{\psi_0^D}$ 
and on $\pi_{\psi'} \times \pi_{\psi_0}$. 
If we denote by $\widehat\theta_W$ and $\widehat\theta_A$ the actions induced 
by $\tl\pi \mapsto \widehat{\tl\pi}$, 
then by Proposition \ref{TAD} (2) and Theorem \ref{AD} (4), 
we see that 
\[
(\pi_{\widehat\psi'} \times \widehat{\II}_{\psi_0}) \boxtimes \widehat\theta_W
= 
(\pi_{\widehat\psi'} \times \pi_{\widehat\psi_0}) \boxtimes \widehat\theta_A 
+
(\pi_{\widehat\psi'} \times \pi_{\psi_0^A}) \boxtimes \widehat\theta_A 
\]
in the Grothendieck group $\RR(\tl\GL_N(E))$.
Since $D_{\tl\GL_N(E)}(\pi_\psi \boxtimes \theta_A) \overset{\theta}{=} \beta(\psi) \pi_{\widehat\psi} \boxtimes \theta_A$,
if we take an action $\theta_W'$ of $\theta$ on $\hat\pi_{\psi'} \times \widehat{\II}_{\psi_0}$ such that 
\[
(\pi_{\widehat\psi'} \times \widehat{\II}_{\psi_0}) \boxtimes \theta_W'
= 
(\pi_{\widehat\psi'} \times \pi_{\widehat\psi_0}) \boxtimes \theta_A 
+
(\pi_{\widehat\psi'} \times \pi_{\psi_0^A}) \boxtimes \theta'_A 
\]
in $\RR(\tl\GL_N(E))$ for some $\theta'_A$, 
then we have
\[
\theta_A' = \frac{\beta(\psi' \oplus \psi_0^D)}{\beta(\psi)}\theta_A.
\]
\par

We realize this action $\theta'_W$ on $\hat\pi_{\psi'} \times \widehat{\II}_{\psi_0}$ using Theorem \ref{main2}. 
Namely, if we let $P=MN_P$ be the standard maximal parabolic subgroup of $\GL_N(E)$ such that 
$\pi_{\widehat\psi'} \boxtimes \pi_{\widehat\psi_0}$ is an irreducible representation of $M$, 
and if we let $w \in W(\theta(M),M)$ be the unique non-trivial element, 
then we have  
$w(\pi_{\widehat\psi'} \boxtimes \pi_{\widehat\psi_0} \circ \theta) \cong \pi_{\widehat\psi'} \boxtimes \pi_{\widehat\psi_0}$, and 
the normalized intertwining operator $\tl{R}_P(\theta \circ w, \tl\pi_{\widehat\psi'} \boxtimes \tl\pi_{\widehat\psi_0})$ realizes 
Arthur's action $\theta_A$ on $\pi_{\widehat\psi} = \pi_{\widehat\psi'} \times \pi_{\widehat\psi_0}$.
Recall that 
\[
\tl{R}_P(\theta \circ w, \tl\pi_{\widehat\psi'} \boxtimes \tl\pi_{\widehat\psi_0}) 
= I_P(\tl\pi_{\widehat\psi'} \boxtimes \tl\pi_{\widehat\psi_0}(w \rtimes \theta))
\circ \theta^* \circ R_{P}(\theta(w), \pi_{\widehat\psi'} \boxtimes \pi_{\widehat\psi_0}, \psi). 
\]
The operator $R_{P}(\theta(w), \pi_{\widehat\psi'} \boxtimes \pi_{\widehat\psi_0}, \psi)$
can be extended to a meromorphic family of normalized intertwining operators 
on $\pi_{\widehat\psi'}|\cdot|_E^{s'} \times \widehat{\II}_{\psi_0}|\cdot|_E^{s}$ for $(s,s') \in \C^2$. 
It can be decomposed into the composition of the normalized intertwining operators 
\begin{align*}
\hat\pi_{\psi'}|\cdot|_E^{s'} \times \hat\pi_{\phi_0}|\cdot|_E^{s-\half{1}} \times \hat\pi_{\phi_0}|\cdot|_E^{s+\half{1}}
&\rightarrow
\hat\pi_{\phi_0}|\cdot|_E^{s-\half{1}} \times \hat\pi_{\psi'}|\cdot|_E^{s'} \times \hat\pi_{\phi_0}|\cdot|_E^{s+\half{1}}
\\&\rightarrow
\hat\pi_{\phi_0}|\cdot|_E^{s-\half{1}} \times \hat\pi_{\phi_0}|\cdot|_E^{s+\half{1}} \times \hat\pi_{\psi'}|\cdot|_E^{s'}
\end{align*}
up to a scalar-valued meromorphic function which is holomorphic at $(s,s') = (0,0)$
(see Lemma \ref{holomorphic}).
Since these two intertwining operators are holomorphic at $(s,s') = (0,0)$ by \cite[I.6.3, Lemma (ii)]{MW_GL}, 
we see that $R_{P}(\theta(w), \pi_{\widehat\psi'} \boxtimes \pi_{\widehat\psi_0}, \psi)$
can be considered as a well-defined operator on $\pi_{\widehat\psi'} \times \widehat{\II}_{\psi_0}$.
On the other hand, the linear isomorphism 
\[
\tl\pi_{\widehat\psi'} \boxtimes \tl\pi_{\widehat\psi_0}(w \rtimes \theta) 
\colon w(\pi_{\widehat\psi'} \boxtimes \pi_{\widehat\psi_0} \circ \theta)
\xrightarrow{\sim} \pi_{\widehat\psi'} \boxtimes \pi_{\widehat\psi_0}
\]
can be extended to a linear isomorphism
\[
w(\pi_{\widehat\psi'} \boxtimes \widehat{\II}_{\psi_0} \circ \theta) 
\xrightarrow{\sim} \pi_{\widehat\psi'} \boxtimes \widehat{\II}_{\psi_0}.
\]
Therefore, the action $\theta_W'$ on $\pi_{\widehat\psi'} \times \widehat{\II}_{\psi_0}$
can be realized by $\tl{R}_P(\theta \circ w, \tl\pi_{\widehat\psi'} \boxtimes \tl\pi_{\widehat\psi_0})$. 
\par

By Propositions \ref{c-temp} and \ref{c-irred} together with $r(\psi_0^D) = r(\psi_0)$, 
if we apply the functor $\tl\pi \mapsto \widehat{\tl\pi}$ to 
the equation 
\[
\II_{\psi_0} \boxtimes \theta_W = \pi_{\psi_0^D} \boxtimes \theta_A + \pi_{\psi_0} \boxtimes \theta_A
\]
in $\RR(\tl\GL_N(E))$, 
then we obtain 
\[
\widehat{\II}_{\psi_0} \boxtimes \widehat\theta_W = \pi_{\psi_0^A} \boxtimes \theta_A + \pi_{\widehat\psi_0} \boxtimes \theta_A
\]
for some $\widehat\theta_W$.
Hence the above isomorphism 
$w(\pi_{\widehat\psi'} \boxtimes \widehat{\II}_{\psi_0} \circ \theta) 
\xrightarrow{\sim} \pi_{\widehat\psi'} \boxtimes \widehat{\II}_{\psi_0}$ 
induces 
\[
\tl\pi_{\widehat\psi'} \boxtimes \tl\pi_{\psi_0^A}(w \rtimes \theta) 
\colon w(\pi_{\widehat\psi'} \boxtimes \pi_{\psi_0^A} \circ \theta)
\xrightarrow{\sim} \pi_{\widehat\psi'} \boxtimes \pi_{\psi_0^A}. 
\]
Therefore the normalized intertwining operator 
$\tl{R}_P(\theta \circ w, \tl\pi_{\widehat\psi'} \boxtimes \tl\pi_{\widehat\psi_0})$ 
on $\pi_{\widehat\psi'} \times \widehat{\II}_{\psi_0}$
induces 
\[
\tl{R}_P(\theta \circ w, \tl\pi_{\widehat\psi'} \boxtimes \tl\pi_{\psi_0^A}) 
\times 
\left.
\frac{\gamma_A(s, \widehat\psi' \otimes {}^c\widehat\psi_0, \psi_E)}
{\gamma_A(s, \widehat\psi' \otimes {}^c\psi_0^A, \psi_E)}
\right|_{s=0}
\]
on $\pi_{\widehat\psi'} \times \pi_{\psi_0^A}$. 
This means that 
\[
\frac{\beta(\psi' \oplus \psi_0^D)}{\beta(\psi)} 
= 
\left.
\frac{\gamma_A(s, \widehat\psi' \otimes {}^c\widehat\psi_0, \psi_E)}
{\gamma_A(s, \widehat\psi' \otimes {}^c\psi_0^A, \psi_E)}
\right|_{s=0}.
\]
\par

Now, we write $\psi' = \oplus_{i=1}^t \psi_i$ for the irreducible decomposition. 
Then by induction, we see that $\beta(\psi' \oplus \psi_0^D)$ is equal to 
\[
(-1)^{r(\psi' \oplus \psi_0^D)}
\left(\prod_{1 \leq i < j \leq t} 
\left.
\frac{\gamma_A(s, \psi_i^A \otimes {}^c\psi_j^A, \psi_E)}
{\gamma_A(s, \widehat\psi_i \otimes {}^c\widehat\psi_j, \psi_E)}
\right|_{s=0}
\right)
\left(
\left.
\frac{\gamma_A(s, \psi'^A \otimes {}^c\psi_0^A, \psi_E)}
{\gamma_A(s, \widehat\psi' \otimes {}^c\psi_0^A, \psi_E)}
\right|_{s=0}
\right).
\]
Since $r(\psi' \oplus \psi_0^D) = r(\psi)$, we have
\[
\beta(\psi) = (-1)^{r(\psi)}
\left(\prod_{0 \leq i < j \leq t} 
\left.
\frac{\gamma_A(s, \psi_i^A \otimes {}^c\psi_j^A, \psi_E)}
{\gamma_A(s, \widehat\psi_i \otimes {}^c\widehat\psi_j, \psi_E)}
\right|_{s=0}
\right)
\left(
\left.
\frac{\gamma_A(s, \psi'^A \otimes {}^c\psi_0^A, \psi_E)}
{\gamma_A(s, \widehat\psi' \otimes {}^c\widehat\psi_0, \psi_E)}
\right|_{s=0}
\right).
\]
This completes the proof.
\end{proof}

\begin{cor}\label{+otimes-}
Fix an $A$-parameter $\psi$ of $G$ of good parity, and assume Hypothesis \ref{ECR}. 
Suppose that $\psi = \phi_1 \boxtimes S_1 \oplus \phi_2 \boxtimes S_2$
for some representations $\phi_1,\phi_2$ of $W_E \times \SL_2(\C)$. 
For $s = \sum_{i \in I_-} e(\rho_i, a_i, b_i) \in A_\psi$, 
define $\psi_\pm$ by 
\[
\psi_- = \bigoplus_{i \in I_-} \rho_i \boxtimes S_{a_i} \boxtimes S_{b_i}, \quad
\psi_+ = \psi- \psi_-
\]
as in Section \ref{sec.Arthur}. 
Then 
\[
\frac{\pair{\hat{s}, \hat\pi}_{\widehat\psi}}{\pair{s,\pi}_\psi}
= 
(-1)^{r(\psi)-r(\psi_+)-r(\psi_-)}
\left.
\frac{\gamma_A(s, \psi_+^A \otimes {}^c\psi_-^A, \psi_E)}
{\gamma_A(s, \widehat\psi_+ \otimes {}^c\widehat\psi_-, \psi_E)}
\right|_{s=0}.
\]
\end{cor}
\begin{proof}
This follows from Lemma \ref{beta} (2) and Proposition \ref{psiA}. 
\end{proof}

\section{Endoscopic character relations for co-tempered parameters}\label{sec.ECR.cotemp}
The purpose of this section is to prove Theorem \ref{main3} (1). 
In this section, we do not impose Hypothesis \ref{ECR}.

\subsection{Equation \eqref{ast}}
Let $\psi = \widehat\phi$ be a co-tempered $A$-parameter for $G$.
We will construct the $A$-packet $\Pi_\psi$ together with the pairing $\pair{\cdot, \pi}_\psi$ for $\pi \in \Pi_\psi$.
According to Lemma \ref{beta} and Corollary \ref{+-temp}, 
the correct definitions should be as follows: 

\begin{itemize}
\item
$\Pi_{\widehat\phi} = \{\hat\pi \,|\, \pi \in \Pi_\phi\}$; 
\item
$\pair{\hat{s}, \hat\pi}_{\widehat\phi} = (-1)^{r(\phi)-r(\phi_+)-r(\phi_-)}\pair{s,\pi}_\phi$ for $s \in \AA_\phi$
corresponding to $\hat{s} \in \AA_{\widehat\phi}$.
\end{itemize}
\par

As explained in Remark \ref{converse}, 
to show \eqref{ECR1} for this packet $\Pi_{\widehat\phi}$, 
we need Lemma \ref{beta} (1), i.e., the equation
$\beta(\phi) \pair{s_{\widehat\phi}, \hat\pi}_{\widehat\phi} = \pair{s_\phi, \pi}_\phi \beta(\pi)$.
It is not trivial that this equation is compatible with the definition of $\pair{\hat{s}, \hat\pi}_{\widehat\phi}$ above.
\par

If $s = \widehat{s_{\widehat{\phi}}} \in \AA_\phi$ corresponds to $\hat{s} = s_{\widehat{\phi}} \in \AA_{\widehat\phi}$, 
we see that $r(\phi) = r(\phi_+) + r(\phi_-)$.
Hence our definition shows that $\pair{s_{\widehat\phi}, \hat\pi}_{\widehat\phi} = \pair{\widehat{s_{\widehat{\phi}}}, \pi}_\phi$.
Using this together with $s_\phi = 1$, 
the equation $\beta(\phi) \pair{s_{\widehat\phi}, \hat\pi}_{\widehat\phi} = \pair{s_\phi, \pi}_\phi \beta(\pi)$
can be rewritten as 
\[
\beta(\phi)\beta(\pi) = \pair{\widehat{s_{\widehat{\phi}}}, \pi}_\phi.
\]
\par

In this section, we only assume Hypothesis \ref{ECR} for tempered $L$-parameters.
To clarify our situation, we state this hypothesis explicitly. 

\begin{hyp}\label{ECR-temp}
For any quasi-split classical group $G'$ with $\dim(\St_{\widehat{G'}}) \leq \dim(\St_{\widehat{G}})$, 
and for any tempered $L$-parameter $\phi'$ for $G'$,
there exists a subset $\Pi_{\phi'}$ of $\Irr_{\temp}(G')$ equipped with $\pair{\cdot,\pi'}_{\phi'}$ 
satisfying \eqref{ECR1} and \eqref{ECR2} of Section \ref{sec.Arthur}.
\end{hyp}

This hypothesis is the same as Hypothesis \ref{ECR-tempD}
and hence, we can use the results in Appendix \ref{xu-atobe}. 
We also notice that the proof of Proposition \ref{c-temp} did not use Hypothesis \ref{ECR}. 
Therefore we know that $\beta(\phi) = (-1)^{r(\phi)}$ even now.
\par

Now we can state what we have to show in this section. 
\begin{prop}\label{beta_phi_pi}
Assume Hypothesis \ref{ECR-temp} (but not \ref{ECR}). 
Let $\phi$ be a tempered $L$-parameter for $G$. 
Then for $\pi \in \Pi_\phi$, we have
\[
\tag{$\ast$}\label{ast}
\beta(\phi)\beta(\pi) = \pair{\widehat{s_{\widehat{\phi}}}, \pi}_\phi. 
\]
\end{prop}

\begin{rem}
This statement was recently proven by Liu--Lo--Shahidi \cite{LLS}. 
However, it is not obvious to us whether they use any results that are unavailable in the setting we need, 
i.e., at the point of Arthur's argument in \cite[Section 7.1]{Ar}. For safety, we will give a proof of Proposition \ref{beta_phi_pi}. 
\end{rem}

Let us explain our strategy for the proof of Proposition \ref{beta_phi_pi}. 
One can check that both $\beta(\phi)$ and $\beta(\pi)$ depend only on the cuspidal support of $\pi$. 
However, it is complicated to list the cuspidal support of $\pi$ for all $\pi \in \Pi_\phi$. 
Instead of this, we give several reductions by using Theorem \ref{Dk}.
The final case is where $\pi$ is supercuspidal, 
which we can treat by a direct computation. 
\par

Let $P = MN_P$ be a standard parabolic subgroup of $G$
with Levi subgroup $M \cong \GL_{k_1}(E) \times \dots \times \GL_{k_t}(E) \times G_0$. 
For $\tau_i \in \Rep(\GL_{k_i}(E))$ and $\pi_0 \in \Rep(G_0)$, 
we denote the normalized parabolic induction by
\[
\tau_1 \times \dots \times \tau_t \rtimes \pi_0
= 
\Ind_P^G(\tau_1 \boxtimes \dots \boxtimes \tau_t \boxtimes \pi_0).
\]

\subsection{Reductions}
As the first reduction, 
we assume that we can decompose $\phi$ as
\[
\phi = \phi_1 \oplus \phi_0 \oplus {}^c\phi_1^\vee. 
\]
Then we have a canonical inclusion $\AA_{\phi_0} \hookrightarrow \AA_\phi$, 
which satisfies $\widehat{s_{\widehat{\phi}_0}} \mapsto \widehat{s_{\widehat{\phi}}}$.
If we denote by $\tau$ the irreducible tempered representation of $\GL_m(E)$
corresponding to $\phi_1$ with $m = \dim(\phi_1)$, 
then we have 
\begin{itemize}
\item
$\tau \rtimes \pi_0$ is semisimple and multiplicity-free for $\pi_0 \in \Pi_{\phi_0}$; 
\item
$\pi \in \Pi_\phi$ if and only if $\pi \hookrightarrow \tau \rtimes \pi_0$ for some $\pi_0 \in \Pi_{\phi_0}$; 
\item
if $\pi \hookrightarrow \tau \rtimes \pi_0$, then 
\[
\pair{\cdot, \pi}_\phi|_{\AA_{\phi_0}} = \pair{\cdot, \pi_0}_{\phi_0}.
\]
\end{itemize}
For these statements, see the proofs of \cite[Proposition 2.4.3]{Ar} and \cite[Proposition 3.4.4]{Mok}.
For such $\pi \hookrightarrow \tau \rtimes \pi_0$, 
by definition of $\beta(\phi)$ and $\beta(\pi)$, we have 
\[
\frac{\beta(\phi)}{\beta(\phi_0)} = (-1)^{r} = \frac{\beta(\pi)}{\beta(\pi_0)},
\]
where $r$ is the length of 
\[
W_E \ni w \mapsto \phi_1\left(
w, \begin{pmatrix}
|w|_E^{\half{1}} & 0 \\ 0 & |w|_E^{-\half{1}}
\end{pmatrix}
\right)
\] 
as a representation of $W_E$.
Hence if we knew \eqref{ast} for $\pi_0$, 
we would have 
\begin{align*}
\beta(\phi)\beta(\pi) 
= \beta(\phi_0)\beta(\pi_0) 
= \pair{\widehat{s_{\widehat{\phi}_0}}, \pi_0}_{\phi_0} 
= \pair{\widehat{s_{\widehat{\phi}}}, \pi}_\phi, 
\end{align*}
which is \eqref{ast} for $\pi$. 
Therefore we may assume that: 
\begin{enumerate}
\item[(1)]
$\phi$ is a discrete parameter, i.e., $\phi$ is of good parity and multiplicity-free.
\end{enumerate}
\par

When $\phi$ is a discrete parameter for $G$, 
we write
\[
\phi = \bigoplus_\rho \bigoplus_{i=1}^t\rho \boxtimes S_{2a_{\rho,i}+1}, 
\]
where $a_{\rho,i} \in (1/2)\Z$ and $0 \leq a_{\rho,1} < \dots < a_{\rho,t}$ with $t=t_\rho \geq 0$.
Then $\rho$ is conjugate-self-dual, and the parity of $2a_{\rho,i}+1$ depends only on the sign of $\rho$.
Recall that $\AA_\phi$ is a quotient of 
\[
A_\phi =  \bigoplus_\rho \bigoplus_{i=1}^t \Z/2\Z e(\rho, 2a_{\rho,i}+1). 
\]
Here, as we abbreviated $\rho \boxtimes S_d \boxtimes S_1$ to $\rho \boxtimes S_d$, 
we write $e(\rho, d) = e(\rho, d,1)$ for simplicity.
Moreover, $\widehat{s_{\widehat{\phi}}} \in \AA_\phi$ is the image of 
\[
\sum_{\substack{\rho \\ a_{\rho,1} \not\in \Z}}\sum_{i=1}^t e(\rho,2a_{\rho,i}+1).
\]
\par

Let $\pi \in \Pi_\phi$. 
As the second reduction, we assume that one can find $\rho$ and $1 < i \leq t_\rho$ such that 
\[
\pair{e(\rho,2a_{\rho,i}+1), \pi}_\phi = \pair{e(\rho,2a_{\rho,i-1}+1), \pi}_\phi.
\]
In this case, by applying Theorem \ref{Dk} repeatedly 
together with \cite[Proposition 2.4.3]{Ar} and \cite[Proposition 3.4.4]{Mok}, 
we have 
\[
\pi
\hookrightarrow 
\rho|\cdot|_E^{a_{\rho,i}} \times \rho|\cdot|_E^{a_{\rho,i}-1} \times \dots \times \rho|\cdot|_E^{-a_{\rho,i-1}} \rtimes \pi_0, 
\]
where $\pi_0 \in \Pi_{\phi_0}$ with 
\[
\phi_0 = \phi - \rho \boxtimes (S_{2a_{\rho,i}+1} \oplus S_{2a_{\rho,i-1}+1}), 
\]
and 
\[
\pair{\cdot, \pi}_\phi|_{\AA_{\phi_0}} = \pair{\cdot, \pi_0}_{\phi_0}.
\]
Hence 
\[
\frac{\beta(\phi)}{\beta(\phi_0)} = (-1)^{a_{\rho,i}+a_{\rho,i-1}+1} = \frac{\beta(\pi)}{\beta(\pi_0)}.
\]
Note that via the canonical inclusion $\AA_{\phi_0} \hookrightarrow \AA_{\phi}$, 
we have
\[
\widehat{s_{\widehat{\phi}_0}} \mapsto \left\{
\begin{aligned}
&\widehat{s_{\widehat{\phi}}} \iif a_{\rho,i} \in \Z, \\
&\widehat{s_{\widehat{\phi}}}-(e(\rho,2a_{\rho,i}+1)+e(\rho,2a_{\rho,i-1}+1)) \iif a_{\rho,i} \not\in \Z.
\end{aligned}
\right. 
\]
Hence if we knew \eqref{ast} for $\pi_0$, 
using $\pair{e(\rho,2a_{\rho,i}+1), \pi}_\phi = \pair{e(\rho,2a_{\rho,i-1}+1), \pi}_\phi$,
we would have 
\begin{align*}
\beta(\phi)\beta(\pi) 
= \beta(\phi_0)\beta(\pi_0) 
= \pair{\widehat{s_{\widehat{\phi}_0}}, \pi_0}_{\phi_0} 
= \pair{\widehat{s_{\widehat{\phi}}}, \pi}_\phi,
\end{align*}
which is \eqref{ast} for $\pi$. 
Therefore, we may assume that: 
\begin{enumerate}
\item[(2)]
$\pair{e(\rho,2a_{\rho,i}+1), \pi}_\phi \not= \pair{e(\rho,2a_{\rho,i-1}+1), \pi}_\phi$ 
for any $\rho$ and $1 < i \leq t_\rho$. 
\end{enumerate}
\par

As the third reduction, we assume that one can find $\rho$ with $t_\rho \geq 1$ and $a_{\rho,1} \not\in \Z$ such that 
\[
\pair{e(\rho,2a_{\rho,1}+1), \pi}_\phi = 1.
\]
In this case, by applying Theorem \ref{Dk} repeatedly, we have 
\[
\pi
\hookrightarrow 
\rho|\cdot|_E^{a_{\rho,1}} \times \rho|\cdot|_E^{a_{\rho,1}-1} \times \dots \times \rho|\cdot|_E^{\half{1}} \rtimes \pi_0, 
\]
where $\pi_0 \in \Pi_{\phi_0}$ with 
\[
\phi_0 = \phi - \rho \boxtimes S_{2a_{\rho,1}+1},
\]
and 
\[
\pair{\cdot, \pi}_\phi|_{\AA_{\phi_0}} = \pair{\cdot, \pi_0}_{\phi_0}.
\]
Hence 
\[
\frac{\beta(\phi)}{\beta(\phi_0)} = (-1)^{a_{\rho,1}+\half{1}} = \frac{\beta(\pi)}{\beta(\pi_0)}.  
\]
Note that via the canonical inclusion $\AA_{\phi_0} \hookrightarrow \AA_{\phi}$, 
we have
\[
\widehat{s_{\widehat{\phi}_0}} \mapsto \widehat{s_{\widehat{\phi}}}-e(\rho,2a_{\rho,1}+1).
\]
Hence the equation \eqref{ast} for $\pi_0$ implies \eqref{ast} for $\pi$. 
Therefore we may assume that: 
\begin{enumerate}
\item[(3)]
$\pair{e(\rho,2a_{\rho,1}+1), \pi}_\phi = -1$ if $a_{\rho,1} \not\in \Z$.
\end{enumerate}
Note that an irreducible tempered representation $\pi$ satisfying (1), (2) and (3)
is called \emph{strongly positive discrete series} in M{\oe}glin--Tadi{\'c}'s terminology. 
\par

Let $\pi \in \Pi_\phi$ be a strongly positive discrete series representation, 
i.e., it satisfies the conditions (1)--(3).
Write again 
\[
\phi = \bigoplus_\rho \bigoplus_{i=1}^t\rho \boxtimes S_{2a_{\rho,i}+1}, 
\]
where $a_{\rho,i} \in (1/2)\Z$ and $0 \leq a_{\rho,1} < \dots < a_{\rho,t}$ with $t=t_\rho \geq 0$.
\par

As the fourth reduction, 
we assume that one can find $\rho$ with $t_\rho \geq 1$ such that 
$a_{\rho,i}-a_{\rho,i-1} >1$ for some $1 \leq i \leq t_\rho$.
Here, formally we set $a_{\rho,0} = -\half{1}$.
In this case, by applying Theorem \ref{Dk}, we have 
\[
\pi \hookrightarrow \rho|\cdot|_E^{a_{\rho,i}} \rtimes \pi_0, 
\]
where $\pi_0 \in \Pi_{\phi_0}$ with 
\[
\phi_0 = \phi - \rho \boxtimes S_{2a_{\rho,i}+1} \oplus \rho \boxtimes S_{2a_{\rho,i}-1},
\]
and $\pair{\cdot, \pi}_\phi = \pair{\cdot, \pi_0}_{\phi_0}$ 
via the canonical identification $\AA_{\phi} \cong \AA_{\phi_0}$ (see Proposition \ref{DvsSECR}).
Hence 
\[
\frac{\beta(\phi)}{\beta(\phi_0)} = -1 = \frac{\beta(\pi)}{\beta(\pi_0)}. 
\]
Note that $\widehat{s_{\widehat{\phi}_0}} = \widehat{s_{\widehat{\phi}}}$ 
via $\AA_{\phi_0} \cong \AA_{\phi}$.
Hence if we knew \eqref{ast} for $\pi_0$, 
we would have 
\begin{align*}
\beta(\phi)\beta(\pi) 
= \beta(\phi_0)\beta(\pi_0) 
= \pair{\widehat{s_{\widehat{\phi}_0}}, \pi_0}_{\phi_0} 
= \pair{\widehat{s_{\widehat{\phi}}}, \pi}_\phi,
\end{align*}
which is \eqref{ast} for $\pi$. 
\par

Therefore, we may finally assume that 
\begin{enumerate}
\item
$\phi$ is a discrete parameter; 
\item
$\pair{e(\rho,2a_{\rho,i}+1), \pi}_\phi \not= \pair{e(\rho,2a_{\rho,i-1}+1), \pi}_\phi$ 
for any $\rho$ and $1 < i \leq t_\rho$; 
\item
$\pair{e(\rho,2a_{\rho,1}+1), \pi}_\phi = -1$ if $a_{\rho,1} \not\in \Z$; 
\item
$a_{\rho,1} = 0$ or $a_{\rho,1} = \half{1}$; 
\item
$a_{\rho,i} = a_{\rho,i-1}+1$ for $i > 1$.
\end{enumerate}
Such a representation $\pi$ is \emph{supercuspidal} by Corollary \ref{sc}.

\subsection{The case of supercuspidal representations}\label{*-sc}
Now we assume that $\pi \in \Pi_{\phi}$ is supercuspidal, i.e., 
it satisfies the conditions (1)--(5) in the last subsection.
We check \eqref{ast} for $\pi$ directly. 
\par

Write  
\[
\phi = \bigoplus_\rho \bigoplus_{i=1}^t\rho \boxtimes S_{2a_{\rho,i}+1}, 
\]
where $a_{\rho,i} \in (1/2)\Z$ and $0 \leq a_{\rho,1} < \dots < a_{\rho,t}$ with $t=t_\rho \geq 0$.
Then 
\begin{align*}
\phi\left(
w, \begin{pmatrix}
|w|_E^{\half{1}} & 0 \\ 0 & |w|_E^{-\half{1}}
\end{pmatrix}
\right) \cong 
\bigoplus_\rho\bigoplus_{i=1}^t \rho \otimes 
(|\cdot|_E^{a_{\rho,i}} \oplus |\cdot|_E^{a_{\rho,i}-1} \oplus \dots \oplus|\cdot|_E^{-a_{\rho,i}}).
\end{align*}
Note that if $a_{\rho,1} \in \Z$, then $\rho$ appears in the right-hand side with multiplicity $t = t_\rho$. 
Hence, 
\[
r(\phi) = \sum_{\substack{\rho \\ a_{\rho,1} \not\in \Z}} \sum_{i=1}^t \left(a_{\rho,i}+\half{1}\right)
+ \sum_{\substack{\rho \\ a_{\rho,1} \in \Z}} \left(\left[\half{t}\right]+\sum_{i=1}^t a_{\rho,i}\right), 
\]
where $[x]$ denotes the largest integer not greater than $x$. 
By the conditions (4) and (5) in the last subsection, 
we note that when $a_{\rho,1} \in \Z$, we have $a_{\rho,i} = i-1$ so that 
\[
\left[\half{t}\right]+\sum_{i=1}^t a_{\rho,i} = \left[\half{t}\right]+\sum_{i=1}^t (i-1) 
= \left[\half{t}\right]+\half{t(t-1)} \equiv 0 \bmod 2.
\]
Similarly, when $a_{\rho,1} \not\in \Z$, we have $a_{\rho,i} + \half{1} = i$ so that 
\[
\sum_{i=1}^t \left(a_{\rho,i}+\half{1}\right) = \sum_{i=1}^t i = \half{t(t+1)}.
\]
Hence
\[
\beta(\phi) = (-1)^{r(\phi)} 
= \prod_{\substack{\rho \\ a_{\rho,1} \not\in \Z}} (-1)^{\half{t(t+1)}}.
\]
On the other hand, since $\pi$ is supercuspidal, we have $\beta(\pi) = 1$.
Finally, by the conditions (2) and (3), we have
\[
\pair{\widehat{s_{\widehat{\phi}}}, \pi}_\phi 
= \prod_{\substack{\rho \\ a_{\rho,1} \not\in \Z}} \prod_{i=1}^t \pair{e(\rho, 2a_{\rho,i}+1), \pi}_\phi
= \prod_{\substack{\rho \\ a_{\rho,1} \not\in \Z}} \prod_{i=1}^t (-1)^i
= \prod_{\substack{\rho \\ a_{\rho,1} \not\in \Z}} (-1)^{\half{t(t+1)}}. 
\]
Therefore, we conclude that 
\[
\beta(\phi)\beta(\pi) = \pair{\widehat{s_{\widehat{\phi}}}, \pi}_\phi,
\]
as desired. 
We obtain \eqref{ast} for $\pi$, and hence
we complete the proof of Proposition \ref{beta_phi_pi}. 
\par

\subsection{ECR for co-tempered $A$-parameters}
Now we can prove Theorem \ref{main3} (1). 
More precisely, we have the following: 
\begin{thm}\label{ECR-cotemp}
Assume Hypothesis \ref{ECR-temp} (but not \ref{ECR}). 
Let $\phi$ be a tempered $L$-parameter for $G$. 
Define $\Pi_{\widehat\phi}$ and $\pair{\hat{s}, \hat\pi}_{\widehat{\phi}}$ for $\hat{s} \in A_{\widehat\phi}$ by 
\[
\Pi_{\widehat\phi} = \{\hat\pi \,|\, \pi \in \Pi_\phi\}
\]
and
\[
\pair{\hat{s}, \hat\pi}_{\widehat{\phi}} = 
(-1)^{r(\phi)-r(\phi_+)-r(\phi_-)}
\pair{s,\pi}_\phi.
\]
Then $\pair{\cdot, \hat\pi}_{\widehat\phi}$ factors through $A_{\widehat\phi} \twoheadrightarrow \AA_{\widehat\phi}$. 
Moreover, \eqref{ECR1} and \eqref{ECR2} hold for $\widehat\phi$. 
\end{thm}
\begin{proof}
Since $r(\phi)=r(\phi_+)+r(\phi_-)$ if $s = \widehat{s_{\widehat\phi}}$, 
we have 
$\pair{s_{\widehat\phi}, \hat\pi}_{\widehat\phi} = \pair{\widehat{s_{\widehat\phi}}, \pi}_{\phi} = \beta(\phi)\beta(\pi)$
by Proposition \ref{beta_phi_pi}. 
Hence, 
when $\tl{f} \in C_c^\infty(\GL_N(E) \rtimes \theta)$ and $f_G \in C_c^\infty(G^\circ)$ have matching orbital integrals, by \cite[(A.1)]{X2}, we have
\begin{align*}
\sum_{\hat\pi \in \Pi_{\widehat\phi}} \pair{s_{\widehat\phi}, \hat\pi}_{\widehat\phi} \Theta_{\hat\pi}(f_G)
&= 
\beta(\phi) \sum_{\hat\pi \in \Pi_{\widehat\phi}} \beta(\pi) \Theta_{\hat\pi}(f_G)
\\&=
\beta(\phi) \sum_{\pi \in \Pi_{\phi}} \Theta_{D_{G^\circ}(\pi)}(f_G) 
\\&=
(G:G^\circ) \beta(\phi) \Theta_{D_{\tl\GL_N(E)}(\tl\pi_\phi)}(\tl{f})
\\&=
(G:G^\circ) \Theta_{\tl\pi_{\widehat\phi}}(\tl{f}).
\end{align*}
This shows \eqref{ECR1}. 
\par

Similarly, since $\beta(\phi) = (-1)^{r(\phi)}$ and $\beta(\phi_\pm) = \beta(\phi_\pm \otimes \eta_\pm)$
by Proposition \ref{c-temp} and Lemma \ref{beta_twist}, 
when $f_G \in C_c^\infty(G)$ and $f_{G_+} \otimes f_{G_-} \in C_c^\infty(G^\circ_+ \times G^\circ_-)$ 
have matching orbital integrals, 
by \cite[Theorem 1.5]{Hi} or \cite[(A.1)]{X2}, 
we have
\begin{align*}
&\frac{1}{(G:G^\circ)}\sum_{\hat\pi \in \Pi_{\widehat\phi}} 
\pair{\hat{s} \cdot s_{\widehat\phi}, \hat\pi}_{\widehat\phi} \Theta_{\hat\pi}(f_G) 
\\&\quad= 
\frac{1}{(G:G^\circ)}\sum_{\hat\pi \in \Pi_{\widehat\phi}} 
(-1)^{r(\phi)-r(\phi_+)-r(\phi_-)}\pair{s, \pi}_\phi \cdot \beta(\phi)\beta(\pi)
\Theta_{\hat\pi}(f_G) 
\\&\quad= 
\frac{1}{(G:G^\circ)} \beta(\phi_+)\beta(\phi_-)\sum_{\pi \in \Pi_{\phi}} 
\pair{s, \pi}_\phi  \Theta_{D_{G^\bullet}(\pi)}(f_G) 
\\&\quad= 
\prod_{\kappa \in \{\pm\}} 
\frac{1}{(G_\kappa:G_\kappa^\circ)} \beta(\phi_\kappa \otimes \eta_\kappa)
\sum_{\pi_\kappa \in \Pi_{\phi_\kappa \otimes \eta_\kappa}} 
\Theta_{D_{G^\circ_\kappa}(\pi_\kappa)}(f_{G_\kappa}) 
\\&\quad= 
\prod_{\kappa \in \{\pm\}} 
\frac{1}{(G_\kappa:G_\kappa^\circ)}
\sum_{\hat\pi_\kappa \in \Pi_{\widehat\phi_\kappa \otimes \eta_\kappa}} 
\beta(\phi_\kappa \otimes \eta_\kappa) \beta(\pi_\kappa)
\Theta_{\hat\pi_\kappa}(f_{G_\kappa})
\\&\quad= 
\prod_{\kappa \in \{\pm\}} 
\frac{1}{(G_\kappa:G_\kappa^\circ)}
\sum_{\hat\pi_\kappa \in \Pi_{\widehat\phi_\kappa \otimes \eta_\kappa}} 
\pair{s_{\widehat{\phi}_\kappa \otimes \eta_\kappa}, \hat\pi_\kappa}_{{\widehat\phi_\kappa} \otimes \eta_\kappa}
\Theta_{\hat\pi_\kappa}(f_{G_\kappa}).
\end{align*}
Here, the definition of $D_{G^\bullet}$ and the assumption on $f_G$ are the same as in the proof of Lemma \ref{beta} (2). 
This shows \eqref{ECR2}. 
\end{proof}

\section{Local intertwining relations for classical groups}\label{sec.LIR}
The purpose of this and next sections is to prove Theorem \ref{main3} (2).
This is a key result that is essential in the global method for establishing Arthur's theory of endoscopic classification. 
Arthur's initial approach was to prove some special cases of Theorem \ref{main3} (2), 
which would suffice for the global method, 
by an argument based on Hecke algebras.
However, our attempts to realize this approach led to very complicated computations.
\par

To show Theorem \ref{main3} (2), 
we will instead adapt the method for Theorem \ref{main2} to the case of classical groups.
However, unlike the $\GL_N(E)$ case, 
the unitary parabolic inductions of classical groups are not necessarily irreducible. 
Because of this fact, we can apply our method only to ``half'' of the cases.
The final key ingredient is Corollary \ref{NIOvsAD}, 
which was obtained in an arXiv version of \cite{KMSW}.
This result tells us that ``half'' of the cases is enough.

\subsection{Hypothesis}
Let $F$ be a non-archimedean local field of characteristic zero. 
Fix a non-trivial character $\psi_F \colon F \rightarrow \C^\times$.
Recall that 
$G$ is one of the following quasi-split classical groups
\[
\SO_{2n+1}(F), \quad \Sp_{2n}(F), \quad
\O_{2n}(F), \quad \U_n.
\]

In this section, we assume Hypothesis \ref{hyp.arthur}, 
which we restate here for the reader's convenience.
\begin{hyp}\label{arthur}
We assume \eqref{ECR1}, \eqref{ECR2} and \eqref{A-LIR} for 
\begin{itemize}
\item
all tempered $L$-parameters $\phi$ for $G$; 
\item
all $A$-parameters $\psi$ for $G'$ with $G'$ any classical group 
such that $\dim(\St_{\widehat{G'}}) < \dim(\St_{\widehat{G}})$.
\end{itemize}
In particular, we have the $A$-packet $\Pi_{\psi_M}$ for $\psi_M \in \Psi(M)$
for any proper Levi subgroup $M$ of $G$, 
since $M$ is a product of such a $G'$ and general linear groups.
\end{hyp}

\begin{rem}
\begin{enumerate}
\item
By Theorem \ref{ECR-cotemp}, one has an $A$-packet $\Pi_{\widehat\phi}$ 
associated to co-tempered $A$-parameters $\widehat\phi$ for $G$, 
satisfying \eqref{ECR1} and \eqref{ECR2}.

\item
For any proper Levi subgroup $M$ of $G$, 
Hypothesis \ref{arthur} is stronger than Hypotheses \ref{ECR} and \ref{hyp_appendixB}. 
Hence we can use Corollaries \ref{+otimes-} and \ref{NIOvsAD} for $\psi_M \in \Psi(M)$. 

\item
Note that we assume \eqref{ECR1} and \eqref{ECR2} not only for proper Levi subgroups $M$ of $G$, 
but also for any classical group $G'$ with $\dim(\St_{\widehat{G'}}) < \dim(\St_{\widehat{G}})$.
Hence Hypothesis \ref{arthur} contains Hypothesis \ref{ECR-tempD} so that 
we can use Theorems \ref{Dk}, \ref{Dk-1} and Corollary \ref{sc}. 

\end{enumerate}
\end{rem}

\subsection{Reduction to the maximal parabolic case}
Let $P = MN_P$ be a standard parabolic subgroup of $G$. 
Write $M = \GL_{k_t}(E) \times \dots \times \GL_{k_1}(E) \times G_0$. 
Let $\psi_M = \widehat\phi_M = \psi_t \oplus \dots \oplus \psi_1 \oplus \psi_0$
be a co-tempered $A$-parameter for $M$ 
such that $\psi_i$ is an $A$-parameter for $\GL_{k_i}(E)$ for $1 \leq i \leq t$, 
and $\psi_0 \in \Psi(G_0)$.
Suppose that $\psi_i$ is irreducible and conjugate-self-dual for any $1 \leq i \leq t$. 
\par

In this subsection, we reduce Theorem \ref{main3} (2) to the case where $t=1$. 
Namely, we prove the following. 
\begin{lem}\label{reduction}
Assume Hypothesis \ref{arthur}. 
We further assume that \eqref{LIR} holds for any irreducible component $\pi \subset I_{P'}(\pi_{M'})$, 
where 
\begin{itemize}
\item
$P'=M'N_{P'}$ is a maximal parabolic subgroup of $G$ so that $M' \cong \GL_k(E) \times G_0'$; 
\item
$\psi_{M'} = \widehat\phi_{M'} = \psi_\GL \oplus \psi_0'$ is a co-tempered $A$-parameter for $M'$
with $\psi_\GL$ irreducible and conjugate-self-dual; 
\item
$\pi_{M'} \in \Pi_{\psi_{M'}}$. 
\end{itemize}
Then \eqref{LIR} holds for any irreducible component $\pi \subset I_{P}(\pi_{M})$ for $\pi_M \in \Pi_{\psi_M}$.
\end{lem}
\begin{proof}
By Proposition \ref{multiplicative} and \cite[Lemma 2.5.3]{KMSW}, the map 
\[
\NN_\psi \ni u \mapsto \pair{\tl{u}, \tl\pi_M} R_P(w_u, \tl\pi_M, \psi_M)
\]
is multiplicative. 
(See also the paragraph in \cite{Ar} containing (2.4.2).)
Hence we may assume one of the following:
\begin{itemize}
\item
$u = s_0 \in \AA_{\psi_0}$;
\item
$u = u_i$ for $1 \leq i \leq t$; or
\item
$u = \sigma \in \SS_t$ such that $\psi_{\sigma(i)} \cong \psi_i$ for $1 \leq i \leq t$,
\end{itemize}
since $\NN_\psi$ is generated by these elements. 
See Section \ref{sec.main3}. 
\par

First, we assume that $u = s_0 \in \AA_{\psi_0}$. 
Then $\tl{w}_u = \1$ or $\tl{w}_u = \epsilon$ (which occurs only when $G = \O_{2n}(F)$), and 
\[
\pair{\tl{u}, \tl\pi_M} R_P(w_u, \tl\pi_M, \psi_M) = 
\left(\pair{\tl{u}, \tl\pi_M}\tl\pi_M(w_u)\right) \circ \pi_M(\tl{w}_u)^{-1}.
\]
By the definition of the operator $\pair{\tl{u}, \tl\pi_M}\tl\pi_M(w_u)$ 
recalled in Section \ref{sec.main3}
together with \eqref{ECR2}, 
we see that the right-hand side is equal to $\pair{s_0,\pi_0}_{\psi_0} = \pair{s_0,\pi}_{\psi}$.
Hence we obtain \eqref{LIR} for this case. 
\par

Next, we assume that $u = \sigma \in \SS_t$ such that $\psi_{\sigma(i)} \cong \psi_i$ for $1 \leq i \leq t$. 
Let $P'=M'N_{P'}$ be a maximal parabolic subgroup of $G$
such that $M' \cong \GL_{k_1+\dots+k_t}(E) \times G_0$, 
and let $\psi_{M'} = (\psi_t \oplus \dots \oplus \psi_1) \oplus \psi_0$ be the $A$-parameter for $M'$ given by $\psi_M$. 
Then 
\[
\Pi_{\psi_{M'}} = \left\{ \Ind_{P \cap M'}^{M'}(\pi_M) \,\middle|\, \pi_M \in \Pi_{\psi_M} \right\}. 
\]
By \cite[Lemma 2.7.2]{KMSW}, 
$\pair{\tl{u}, \tl\pi_M} R_P(w_u, \tl\pi_M, \psi_M)$
descends to the normalized intertwining operator for $\GL_{k_1+\dots+k_t}(E)$. 
(See also the proof of Lemma 2.4.2 in \cite{Ar}.)
Hence by Theorem \ref{untwistedGL}, this operator must be the identity map.
Since $s_u = 1$ in this case, we obtain \eqref{LIR}.
\par

Finally, we assume that $u = u_i$ for $1 \leq i \leq t$. 
We will prove the assertion by induction on $t$. 
We can assume that $t>1$.
Let $P'=M'N_{P'}$ be a maximal parabolic subgroup of $G$
such that $M' \cong \GL_{k_t}(E) \times G_0'$, 
and let
\[
\psi_{M'} 
= 
\psi_t \oplus 
(\psi_{t-1} \oplus \dots \oplus \psi_1 \oplus \psi_0 \oplus {}^c\psi_1^\vee \oplus \dots \oplus {}^c\psi_{t-1}^\vee)
\]
be the $A$-parameter for $M'$ given by $\psi_M$. 
Then for any $\pi_M \in \Pi_{\psi_M}$ and any irreducible component $\pi \subset I_P(\pi_M)$, 
there is a unique $\pi_{M'} \in \Pi_{\psi_{M'}}$ such that 
$\pi_{M'} \subset \Ind_{P \cap M'}^{M'}(\pi_M)$ and $\pi \subset I_{P'}(\pi_{M'})$. 
Moreover, there is a canonical injection $A_{\psi_{M'}} \hookrightarrow A_\psi$ 
and we have 
\[
\pair{\cdot, \pi}_\psi|_{A_{\psi_{M'}}} = \pair{\cdot, \pi_{M'}}_{\psi_{M'}}. 
\]
These facts follow from the tempered case by taking Aubert duality and using Corollary \ref{+-temp}. 
\par

Suppose that $u = u_i$ with $i \not= t$ so that $s_{u_i} \in A_{\psi_{M'}}$. 
Then by \cite[Lemma 2.7.2]{KMSW}, 
$\pair{\tl{u}, \tl\pi_M} R_P(w_u, \tl\pi_M, \psi_M)$
descends to the normalized intertwining operator for $G_0'$, i.e., 
\[
\pair{\tl{u_i}, \tl\pi_M} R_P(w_{u_i}, \tl\pi_M, \psi_M)
= \Ind_{P'}^{G}\left(
\id_{\pi_{\psi_t}} \otimes \pair{\tl{u_i}, \tl\pi_{M_0'}} R_{P \cap G_0'}(w_{u_i}, \tl\pi_{M_0'}, \psi_{M_0'})
\right),
\]
where $M_0' = M \cap G_0'$, $\psi_{M_0'} = \psi_M - \psi_t \in \Psi(M_0')$
and we write $\pi_M = \pi_{\psi_t} \boxtimes \pi_{M_0'}$ with $\pi_{M_0'} \in \Pi_{\psi_{M_0'}}$. 
By the induction hypothesis, 
\eqref{LIR} is known for $G_0'$ in place of $G$, 
so this operator acts on $I_{P'}(\pi_{M'})$ 
by the scalar $\pair{s_{u_i}, \pi_{M'}}_{\psi_{M'}} = \pair{s_{u_i}, \pi}_\psi$.
Hence we obtain \eqref{LIR} for this case. 
\par

Suppose finally that $u = u_t$. 
Then by \cite[Lemma 2.7.2]{KMSW}, 
\[
\pair{\tl{u_t}, \tl\pi_M} R_P(w_{u_t}, \tl\pi_M, \psi_M)|_{I_{P'}(\pi_{M'})}
= 
\pair{\tl{u_t}, \tl\pi_{M'}} R_{P'}(w_{u_t}, \tl\pi_{M'}, \psi_{M'}).
\]
Since we are assuming \eqref{LIR} for the maximal parabolic case, 
we know that 
\[
\pair{\tl{u_t}, \tl\pi_{M'}} R_{P'}(w_{u_t}, \tl\pi_{M'}, \psi_{M'})|_{\pi} = \pair{s_{u_t}, \pi}_\psi.
\]
Hence we obtain \eqref{LIR} for this case. 
This completes the proof.
\end{proof}

\subsection{Halving the problem}
In the rest of this section, 
we will focus on the maximal parabolic case.
\par

Fix a standard maximal parabolic subgroup $P=MN_P$ of $G$ with $M \cong \GL_k(E) \times G_0$. 
Let $\psi_M = \psi_\GL \oplus \psi_0$ be an $A$-parameter for $M$, 
and let $\psi = \psi_\GL \oplus \psi_0 \oplus {}^c\psi_\GL^\vee$ be the $A$-parameter for $G$ given by $\psi_M$. 
Suppose that $\psi_\GL$ is irreducible and conjugate-self-dual.
The $A$-packet $\Pi_\psi$ is given by the (multi-)set of irreducible components of $I_P(\pi_M)$ for $\pi_M \in \Pi_{\psi_M}$.
Let $u \in \NN_\psi$. 
Recall that the normalized self-intertwining operator 
\[
\pair{\tl{u}, \tl\pi_M} R_P(w_u, \tl\pi_M, \psi_M) \colon I_P(\pi_M) \rightarrow I_P(\pi_M)
\]
is defined by 
\[
\pair{\tl{u}, \tl\pi_M} R_P(w_u, \tl\pi_M, \psi_M)f(g) 
= \pair{\tl{u}, \tl\pi_M}\tl{\pi}_M(w_u)\left(
R_P(w_u, \pi_M, \psi_M)f(g)
\right)
\]
with a linear isomorphism
\[
\pair{\tl{u}, \tl\pi_M}\tl{\pi}_M(w_u) \colon \pi_M \xrightarrow{\sim} \pi_M
\]
satisfying that the diagram
\[
\xymatrix{
\pi_M \ar@{->}[d]_{\pi_M(\tl{w}_u^{-1} m \tl{w}_u)} \ar@{->}[rrr]^{\pair{\tl{u}, \tl\pi_M}\tl{\pi}_M(w_u)} &&& \pi_M \ar@{->}[d]^{\pi_M(m)} \\
\pi_M  \ar@{->}[rrr]^{\pair{\tl{u}, \tl\pi_M}\tl{\pi}_M(w_u)} &&& \pi_M 
}
\]
is commutative for any $m \in M$. 
The definition of this isomorphism will be recalled in the proof of the next lemma.
\par

We can write $\pi_M = \pi_\GL \boxtimes \pi_0$, 
where $\pi_\GL$ is an irreducible conjugate-self-dual representation of $\GL_k(E)$. 
Recall that $\NN_\psi$ is generated by $\AA_{\psi_0}$ and an element $u_1$. 
When $u = u_1$, 
similar to the definition of $\theta_A$ in Section \ref{sec.thetaA}, 
by using a Whittaker functional on the standard module of $\pi_\GL$, 
one can normalize a linear isomorphism $\AA_{w_u} \colon \pi_\GL \xrightarrow{\sim} \pi_\GL$ such that the diagram
\[
\xymatrix{
\pi_M \ar@{->}[d]_{\pi_M(\tl{w}_u^{-1} m \tl{w}_u)} \ar@{->}[rrr]^{\AA_{w_u} \otimes \id_{\pi_0}} &&& \pi_M \ar@{->}[d]^{\pi_M(m)} \\
\pi_M  \ar@{->}[rrr]^{\AA_{w_u} \otimes \id_{\pi_0}} &&& \pi_M 
}
\]
is commutative for any $m \in M$.

\begin{lem}\label{c3}
Let $u \in \NN_\psi$. 
\begin{enumerate}
\item
If $u = s_0 \in \AA_{\psi_0}$, then $\pair{\tl{u}, \tl\pi_M}\tl{\pi}_M(w_u) = \pair{s_0,\pi_0}_{\psi_0} \pi_M(\tl{w}_u)$. 
\item
If $u = u_1$, then $\pair{\tl{u}, \tl\pi_M}\tl{\pi}_M(w_u) = \AA_{w_u} \otimes \id_{\pi_0}$.
\end{enumerate}
\end{lem}
\begin{proof}
Recall from Section \ref{sec.main3} that the operator $\pair{\tl{u}, \tl\pi_M}\tl{\pi}_M(w_u)$ 
is defined by a twisted endoscopic identity. 
We will check that this identity also holds after replacing the operator
$\pair{\tl{u}, \tl\pi_M}\tl{\pi}_M(w_u)$ by either 
$\pair{s_0,\pi_0}_{\psi_0} \pi_M(\tl{w}_u)$ or
$\AA_{w_u} \otimes \id_{\pi_0}$
depending on whether $u = s_0 \in \AA_{\psi_0}$ or $u=u_1$.
\par

First, we assume that $u = s_0 \in \AA_{\psi_0}$.
Then the required identity follows from the standard (or twisted) endoscopic character identity \eqref{ECR2} for $\psi_0$. 
\par

Next we assume that $u=u_1$.
Since the conjugation action of $\tl{u}$ on $\widehat{M^\circ}$ preserves the standard pinning,
and it acts on $\GL_k(\C)$ by the pinned outer automorphism 
and on $\widehat{G_0^\circ}$ by the identity, 
we see that $s \in Z(\widehat{M^\circ})$ and $M'^\circ = G_1 \times G_0^\circ$ 
with $G_1$ a twisted endoscopic datum for $\GL_k(E)$.
\par

We may assume that $f_M = f_\GL \otimes f_0$ and $f_{M'} = f_{1} \otimes f_0$. 
Then 
\begin{align*}
&\sum_{\pi_M \in \Pi_{\psi_M}} \pair{s_{\psi_M}, \pi_M}_{\psi_M}
\tr(\AA_{w_u} \otimes \id_{\pi_0} \circ \pi_M(f_M))
\\&=
\tr(\AA_{w_u} \circ \pi_\GL(f_\GL))
\sum_{\pi_0 \in \Pi_{\psi_0}} \pair{s_{\psi_0}, \pi_0}_{\psi_0} \Theta_{\pi_0}(f_0), 
\end{align*}
whereas 
\begin{align*}
&\sum_{\pi_{M'} \in \Pi_{\psi_{M'}}} \pair{s_{\psi_{M'}}, \pi_{M'}}_{\psi_{M'}} \Theta_{\pi_{M'}}(f_{M'})
\\&=
\left(\sum_{\pi_1 \in \Pi_{\psi_1}} \pair{s_{\psi_1}, \pi_1}_{\psi_1} \Theta_{\pi_1}(f_1)\right)
\left(\sum_{\pi_0 \in \Pi_{\psi_0}} \pair{s_{\psi_0}, \pi_0}_{\psi_0} \Theta_{\pi_0}(f_0)\right),
\end{align*}
where we write $\psi_{M'} = \psi_1 \oplus \psi_0 \in \Psi(G_1 \times G_0)$. 
Since $\psi_\GL = \xi \circ \psi_1$ and $\dim(\St_{\widehat{G_1}}) < \dim(\St_{\widehat{G}})$, 
by \eqref{ECR1} for $\psi_1$, we have
\[
\tr(\AA_{w_u} \circ \pi_\GL(f_\GL)) 
= \sum_{\pi_1 \in \Pi_{\psi_1}} \pair{s_{\psi_1}, \pi_1}_{\psi_1} \Theta_{\pi_1}(f_1). 
\]
Therefore, we obtain the desired identity. 
\end{proof}

Lemma \ref{c3} (1) immediately implies \eqref{LIR} for $u=s_0 \in \AA_{\psi_0}$.
In the rest of this and next sections, we assume that $u=u_1$.
\par

By Lemma \ref{reduction}, 
we may assume that
$\psi_M = \widehat\phi_M$ 
where $\phi_M = \phi_\GL \oplus \phi_0$ is a tempered $L$-parameter for $M$ such that 
$\phi_\GL = \rho_\GL \boxtimes S_{2\alpha+1}$ is irreducible and conjugate-self-dual. 
Then $\Pi_{\psi_M}$ is given by Aubert duality from $\Pi_{\phi_M}$. 
In particular, for $\pi_M \in \Pi_{\psi_M}$, 
the parabolically induced representation $I_P(\pi_M) = \Ind_P^G(\pi_M)$
is a direct sum of at most two irreducible unitary representations.
Indeed, by taking Aubert duality, 
this fact is reduced to the tempered case, 
which follows from 
\cite[Theorem 1.5.1]{Ar}, \cite[Theorem 2.5.1]{Mok}
and $[\AA_\phi: \AA_{\phi_0}] \leq 2$.
\par

\begin{lem}\label{halve}
Assume Hypothesis \ref{arthur}.
Let $\psi_M = \widehat\phi_M$ be as above, and $\pi_M \in \Pi_{\psi_M}$. 
Assume that $I_P(\pi_M)$ is reducible, hence $I_P(\pi_M) = \pi_1 \oplus \pi_2$. 
Write 
\[
\pair{\tl{u}, \tl\pi_M} R_P(w_u, \tl\pi_M, \psi_M)|_{\pi_i} = \ep_i \cdot \id_{\pi_i}
\]
for $\ep_i \in \C^\times$.
Then we have $\ep_2 = -\ep_1$.
\end{lem}
\begin{proof}
Let $\overline{P}$ be the parabolic subgroup of $G$ opposite to $P$. 
If we set $\sigma_M = \widehat\pi_M$ and $\sigma_i = \widehat\pi_i$, 
then $I_{\overline{P}}(\sigma_M) = \sigma_1 \oplus \sigma_2$ is a direct sum of irreducible tempered representations.
Moreover, we have a normalized intertwining operator 
\[
\pair{\tl{u}, \tl\sigma_M} R_{\overline{P}}(w_u, \tl\sigma_M, \phi_M) 
\colon I_{\overline{P}}(\sigma_M) \rightarrow I_{\overline{P}}(\sigma_M).
\]
If we write 
$\pair{\tl{u}, \tl\sigma_M} R_{\overline{P}}(w_u, \tl\sigma_M, \phi_M)|_{\sigma_i} = \ep'_i \cdot \id_{\sigma_i}$, 
then by \eqref{LIR} for the tempered case, 
we know that $\ep'_2 = -\ep'_1$.
\par

Recall that Aubert duality is a functor. 
We claim that 
the two intertwining operators
\[ 
\pair{\tl{u}, \tl\pi_M} R_P(w_u, \tl\pi_M, \psi_M)
\quad\text{and}\quad
\pair{\tl{u}, \tl\sigma_M} R_{\overline{P}}(w_u, \tl\sigma_M, \phi_M)
\]
are dual to each other 
up to a nonzero constant $c \in \C^\times$.
This means that $\ep_i = c \ep'_i$ for $i=1,2$.
Therefore, we have $\ep_2 = c \ep'_2 = -c\ep'_1 = -\ep_1$.
\par

This claim is Corollary \ref{NIOvsAD} (1) in almost all cases. 
The exceptional case is where 
$G = \O_{2n}(F)$ and $I_P(\pi_M) = \Ind_{P^\circ}^{G}(\pi_M^\circ)$ 
for some irreducible representation $\pi_M^\circ$ of $M^\circ$. 
Then we can write $I_{P}(\pi_M) = I_{P^\circ}^+(\pi_M^\circ) \oplus I_{P^\circ}^-(\pi_M^\circ)$, 
where $I_{P^\circ}^+(\pi_M^\circ)$ (\resp $I_{P^\circ}^-(\pi_M^\circ)$) 
is the subspace of $\Ind_{P^\circ}^{G}(\pi_M^\circ)$ consisting of functions $f$ on $G$ 
whose supports are contained in $G^\circ$ (\resp $G \setminus G^\circ$). 
Corollary \ref{NIOvsAD} (2) says that 
\[
\pair{\tl{u}, \tl\pi_M} R_P(w_u, \tl\pi_M, \psi_M)|_{I_{P^\circ}^{\pm}(\pi_{M^\circ})}
\quad\text{and}\quad
\pair{\tl{u}, \tl\sigma_M} R_{\overline{P}}(w_u, \tl\sigma_M, \phi_M)|_{(I_{P^\circ}^{\pm}(\pi_{M^\circ}))^{\widehat{\ }}}
\]
are dual to each other 
up to a nonzero constant $c_\pm \in \C^\times$.
\par

Now for $f \in \pi_i$, write $f = f_+ + f_-$ with $f_\pm \in I_{P^\circ}^{\pm}(\pi_M^\circ)$.
Then 
\begin{align*}
&\pair{\tl{u}, \tl\pi_M} R_P(w_u, \tl\pi_M, \psi_M)f
\\&= c_+ (\pair{\tl{u}, \tl\sigma_M} R_{\overline{P}}(w_u, \tl\sigma_M, \phi_M))^{\widehat{\,}}f_+
+ c_- (\pair{\tl{u}, \tl\sigma_M} R_{\overline{P}}(w_u, \tl\sigma_M, \phi_M))^{\widehat{\,}}f_-
\\&= c_+ (\pair{\tl{u}, \tl\sigma_M} R_{\overline{P}}(w_u, \tl\sigma_M, \phi_M))^{\widehat{\,}}f
+ (c_--c_+) (\pair{\tl{u}, \tl\sigma_M} R_{\overline{P}}(w_u, \tl\sigma_M, \phi_M))^{\widehat{\,}}f_-.
\end{align*}
Since Aubert duality is a functor, 
we have
\[
\pair{\tl{u}, \tl\pi_M} R_P(w_u, \tl\pi_M, \psi_M)f, \quad
c_+ (\pair{\tl{u}, \tl\sigma_M} R_{\overline{P}}(w_u, \tl\sigma_M, \phi_M))^{\widehat{\,}}f 
\in \pi_i
\]
so that 
\[
(c_--c_+) (\pair{\tl{u}, \tl\sigma_M} R_{\overline{P}}(w_u, \tl\sigma_M, \phi_M))^{\widehat{\,}}f_- 
\in \pi_i \cap I_{P^\circ}^\delta(\pi_{M}^\circ), 
\]
where $\delta = -\det(\tl{w}_u)$.
\par

Suppose for the sake of contradiction that $c_+ \not= c_-$. 
Then the above argument shows that $f_\pm \in \pi_i$ for any $f = f_+ + f_- \in \pi_i$.
It implies that $\pi_i \cap I_{P^\circ}^{\pm}(\pi_M^\circ) \not= \{0\}$.  
By looking at the actions of 
$\pair{\tl{u}, \tl\pi_M} R_P(w_u, \tl\pi_M, \psi_M)$
and the dual of 
$\pair{\tl{u}, \tl\sigma_M} R_{\overline{P}}(w_u, \tl\sigma_M, \phi_M)$ 
on this subspace, 
we have $\ep_i = c_\pm \ep_i'$. 
This shows that $c_+ \ep_i' = c_- \ep_i'$ and hence $\ep_i' = 0$. 
This contradicts the fact that $\ep_i' \in \{\pm1\}$.
Hence we obtain that $c_+=c_-$, which shows the claim. 
\end{proof}

Lemma \ref{halve} is a key step to prove Theorem \ref{main3} (2). 
It ``halves'' the problem, i.e., 
by this lemma, it is enough to prove \eqref{LIR} 
for only one direct summand $\pi \subset I_P(\pi_M)$ for each $\pi_M \in \Pi_{\psi_M}$.

\subsection{Highly non-tempered summands}\label{high}
In the previous subsection, we showed that it is enough to prove \eqref{LIR} for one irreducible summand of $I_P(\pi_M)$. 
Here, we introduce a notion that will isolate a suitable summand. 
\par

Let $P = MN_P$ be a standard maximal parabolic subgroup of $G$
with $M \cong \GL_k(E) \times G_0$, 
and let $\psi_M = \psi_\GL \oplus \psi_0$ be an $A$-parameter for $M$. 
In this and next subsections, 
we only assume that $\psi_\GL$ is irreducible and conjugate-self-dual. 
Namely, $\psi_M$ is not necessarily co-tempered here.
Note that then $I_P(\pi_M)$ could have more than two irreducible summands for $\pi_M \in \Pi_{\psi_M}$.
\par

Recall from Section \ref{sec.GL} that 
for a multi-segment $\mm$, 
we denote by $\II(\mm)$ the standard module associated to $\mm$. 
For a segment $\ss = [x,y]_\rho$ with $\rho$ unitary supercuspidal, 
we call the value $\half{x+y}$ the \emph{midpoint} of $\ss$.
The Langlands classification for $G_0$ says that for $\pi_0 \in \Irr(G_0)$, 
one has a multi-segment $\mm_0$ and an irreducible tempered representation $\tau_0$ such that 
every segment $\ss \in \mm_0$ has a positive midpoint, 
and $\II(\mm_0) \rtimes \tau_0$ is the standard module of $\pi_0$. 
Thus, $\pi_0$ is the unique irreducible quotient of $\II(\mm_0) \rtimes \tau_0$.
When $G_0 = \O_{2n_0}(F)$, it follows from the Langlands classification for $\SO_{2n_0}(F)$. 
Note that if $\pi_0 \in \Irr(\O_{2n_0}(F))$ is the Langlands quotient of $\II(\mm_0) \rtimes \tau_0$ with $\tau_0 \in \Irr(\O_{2n_0'}(F))$, 
then $\tau_0|_{\SO_{2n_0'}(F)}$ is reducible if and only if $n_0' > 0$ and 
$\pi_0|_{\SO_{2n_0}(F)}$ is reducible. 
\par

Recall that $\psi_\GL$ is assumed to be irreducible. 
We write $\psi_\GL = \rho_\GL \boxtimes S_{2\alpha+1} \boxtimes S_{2\beta+1}$, 
and set 
\[
\mm_\GL = [-\alpha+\beta,\alpha+\beta]_{\rho_\GL} + [-\alpha+\beta-1,\alpha+\beta-1]_{\rho_\GL} 
+ \dots + [-\alpha-\beta,\alpha-\beta]_{\rho_\GL}. 
\]
Then the Langlands quotient $\pi_\GL$ of $\II(\mm_\GL)$ is the representation corresponding to $\psi_\GL$. 
\par

\begin{defi}\label{HNT}
Define $\kappa \in \{1, \half{1}\}$ such that $\beta-\kappa \in \Z$.
For $\pi_M = \pi_\GL \boxtimes \pi_0 \in \Pi_{\psi_M}$, 
write $\II(\mm_0) \rtimes \tau_0$ for the standard module of $\pi_0$.
We say that an irreducible summand $\pi$ of $I_P(\pi_M)$ is \emph{highly non-tempered}
if there is a tempered representation $\tau$ with 
\[
\left\{
\begin{aligned}
&\tau = \tau_0 \iif \kappa = \half{1}, \\
&\tau \hookrightarrow \Delta([-\alpha,\alpha]_{\rho_\GL}) \rtimes \tau_0 \iif \kappa = 1
\end{aligned}
\right.
\]
such that $\II(\mm) \rtimes \tau$ is the standard module of $\pi$, 
where
\[
\mm = \mm_0 + 2[-\alpha+\beta,\alpha+\beta]_{\rho_\GL} + 2[-\alpha+\beta-1,\alpha+\beta-1]_{\rho_\GL}
+ \dots + 2[-\alpha+\kappa,\alpha+\kappa]_{\rho_\GL}.
\]
\end{defi}

\begin{lem}\label{existence}
For any $\pi_M = \pi_\GL \boxtimes \pi_0 \in \Pi_{\psi_M}$, 
there exists a highly non-tempered summand $\pi$ of $I_P(\pi_M)$. 
Moreover, it is unique if $\kappa = \half{1}$, and there are at most two such summands if $\kappa = 1$.
\end{lem}
\begin{proof}
The lemma is a special case of \cite[Proposition 1.3]{T2}. 
But for completeness, we give a proof.
\par

We use the notations as above. 
Write $\II(\mm) = \tau_1|\cdot|_E^{s_1} \times \dots \times \tau_r|\cdot|_E^{s_r}$ 
with $\tau_i$ being tempered and $s_1 > \dots > s_r > 0$.
By Tadi\'c's formula (Theorem \ref{Tformula}), 
with a suitable parabolic subgroup $P'=M'N_{P'}$ and an irreducible tempered representation $\tau$, 
we have
\[
\Jac_{P'}(I_P(\pi_M)) \geq {}^c\tau_1^\vee|\cdot|_E^{-s_1} \otimes \dots \otimes {}^c\tau_r^\vee|\cdot|_E^{-s_r} \otimes \tau
\]
in the Grothendieck group $\RR(M')$.
Here, for $A, B \in \RR(M')$, we write $A \leq B$ 
if $B-A$ is a non-negative combination of irreducible representations.
Conversely, if 
$\Jac_{P'}(I_P(\pi_M)) \geq \tau'_1|\cdot|_E^{-s_1} \otimes \dots \otimes \tau'_r|\cdot|_E^{-s_r} \otimes \tau$
for some irreducible tempered representations $\tau_1', \dots, \tau_r'$ and $\tau$, 
then we must have $\tau_i' \cong {}^c\tau_i^\vee$ for $i=1,\dots,r$, 
and $\tau = \tau_0$ if $\kappa = \half{1}$, 
whereas $\tau \hookrightarrow \Delta([-\alpha,\alpha]_{\rho_\GL}) \rtimes \tau_0$ if $\kappa = 1$. 
Moreover, such an irreducible representation appears in $\Jac_{P'}(I_P(\pi_M))$ with multiplicity one. 
\par

Suppose that an irreducible subquotient $\pi$ of $I_P(\pi_M)$ satisfies that 
$\Jac_{P'}(\pi) \geq {}^c\tau_1^\vee|\cdot|_E^{-s_1} \otimes \dots \otimes {}^c\tau_r^\vee|\cdot|_E^{-s_r} \otimes \tau$.
Then by looking at the central character, 
(after replacing $\tau$ if necessary) 
we see that the right-hand side is a quotient of $\Jac_{P'}(\pi)$
in the category $\Rep(M')$, 
which is equivalent by Frobenius reciprocity to saying that 
\[
\pi \hookrightarrow {}^c\tau_1^\vee|\cdot|_E^{-s_1} \times \dots \times {}^c\tau_r^\vee|\cdot|_E^{-s_r} \rtimes \tau.
\]
By \cite[Lemma 2.2]{AG}, one can see that the standard module of $\pi$ is $\II(\mm) \rtimes \tau$. 
In particular, since $I_P(\pi_M)$ is semisimple, 
we conclude that $\pi$ is a highly non-tempered summand. 
\par

On the other hand, $\Jac_{P'}(I_P(\pi_M))$ contains 
${}^c\tau_1^\vee|\cdot|_E^{-s_1} \otimes \dots \otimes {}^c\tau_r^\vee|\cdot|_E^{-s_r} \otimes \tau$ 
with multiplicity at most one for each $\tau$. 
The number of highly non-tempered summands is at most the number of the choices of $\tau$.
By definition of $\tau$, 
this number is $1$ or $2$ according to $\kappa = \half{1}$ or $\kappa = 1$. 
\end{proof}

For the rest of this subsection, we fix a highly non-tempered summand $\pi \subset I_P(\pi_M)$. 
Let $\II(\mm_\GL)$, $\II(\mm) \rtimes \tau$ and $\II(\mm_0) \rtimes \tau_0$ be as above. 
Write 
\begin{align*}
\II(\mm_\GL) 
&= 
\Delta([-\alpha,\alpha]_{\rho_\GL})|\cdot|_E^{\beta} \times \Delta([-\alpha,\alpha]_{\rho_\GL})|\cdot|_E^{\beta-1} 
\times \dots \times \Delta([-\alpha,\alpha]_{\rho_\GL})|\cdot|_E^{-\beta}, 
\\
\II(\mm_0) 
&= \tau_r|\cdot|_E^{e_r} \times \dots \times \tau_1|\cdot|_E^{e_1}, 
\end{align*}
where $\tau_i$ is an irreducible tempered representation of $\GL_{k_i}(E)$ 
and $e_r > \dots > e_1 > 0$. 
\par

To define several objects,
we realize $G$ as an isometry group $G(W)$ (or its identity component)
of a vector space $W$ over $E$ equipped with a non-degenerate sesquilinear form.
We write
\[
W = V_+ \oplus W_0 \oplus V_-, 
\]
where $V_{\pm}$ is a totally isotropic subspace with $V_+ \oplus V_-$ non-degenerate, 
and where $W_0$ is the orthogonal complement of $V_+ \oplus V_-$. 
Suppose that the standard parabolic subgroup $P=MN_P$ is the stabilizer of $V_+$, 
and the Levi subgroup $M$ of $P$ is the stabilizer of $V_+$ and $V_-$. 
Hence $M = \GL(V_+) \times G(W_0)$.
We decompose 
\begin{align*}
V_\pm &= V_\pm^{(\beta)} \oplus V_\pm^{(\beta-1)} \oplus \dots \oplus V_\pm^{(-\beta)}, 
\\
W_0 &= \left(\bigoplus_{i=1}^r W_0^{(e_i)}\right) 
\oplus W_0^{(0)} \oplus \left(\bigoplus_{i=1}^r W_0^{(-e_i)}\right)
\end{align*}
such that 
\begin{itemize}
\item
$\dim(V_\pm^{(b)}) = d_0$ with $d_0 = \dim(\rho_\GL \boxtimes S_{2\alpha+1})$ for $b \in \{\beta, \beta-1, \dots, -\beta\}$;
\item
$V_+^{(b)} \oplus V_-^{(-b)}$ is non-degenerate for $b \in \{\beta, \beta-1, \dots, -\beta\}$;

\item
$W_0^{(\pm e_i)}$ is a totally isotropic subspace of dimension $k_i$ for $1 \leq i \leq r$; 
\item
$W_0^{(e_i)} \oplus W_0^{(-e_i)}$ is non-degenerate for $1 \leq i \leq r$.
\end{itemize}
For $j = 0,1,2$, we define a total order $\prec_j$ on the set 
\[
\VV = 
\left\{V_+^{(b)}, V_-^{(b)}\right\}_{-\beta \leq b \leq \beta} \cup 
\left\{W_0^{(e_i)}, W_0^{(-e_i)}\right\}_{1 \leq i \leq r} \cup \{W_0^{(0)}\}
\] 
as follows. 
\begin{enumerate}
\item[(0)] 
When $j=0$, 
\begin{itemize}
\item
$V_+^{(b)} \prec_0 W_0^{(e)} \prec_0 V_-^{(b)}$; 
\item
if $b > b'$, then $V_\pm^{(b')} \prec_0 V_\pm^{(b)}$; 
\item
if $e > e'$, then $W_0^{(e')} \prec_0 W_0^{(e)}$.
\end{itemize}

\item[(1)]
When $j=1$, 
\begin{itemize}
\item
$V_+^{(b)} \prec_1 W_0^{(e)} \prec_1 V_-^{(b)}$; 
\item
if $b > b'$, then $V_\pm^{(b)} \prec_1 V_\pm^{(b')}$; 
\item
if $e > e'$, then $W_0^{(e)} \prec_1 W_0^{(e')}$.
\end{itemize}

\item[(2)]
When $j=2$, 
\begin{itemize}
\item
for $X, Y \in \{V_\pm, W_0\}$, 
if $b > b'$, then $X^{(b)} \prec_2 Y^{(b')}$; 
\item
if $b = e$, then $V_+^{(b)} \prec_2 W_0^{(e)} \prec_2 V_-^{(b)}$.
\end{itemize}
\end{enumerate}
For $j = 0,1,2$, 
if we write 
\[
\{V \in \VV \;|\; V \prec_j W_0^{(0)}\} = \{V_1, \dots, V_t\}
\]
with $V_1 \prec_j \dots \prec_j V_t \prec_j W_0^{(0)}$, 
then we define a parabolic subgroup $P_j' = M_1 N_{P_j'}$ as the stabilizer of the flag 
\[
V_1 \subset V_1 \oplus V_2 \subset \dots \subset V_1 \oplus \dots \oplus V_t, 
\]
where $M_1$ is the common Levi subgroup of $G$ stabilizing all $V \in \VV$.
We may assume that $P_1 = P'_1$ is a standard parabolic subgroup. 
Note that $P'_1 \subset P$.
Let $w_1,w_2 \in W^G$ be such that $\tl{w}_1 P_1 \tl{w}_1^{-1} = P'_0$ 
and such that $\tl{w}_2^{-1}P'_2\tl{w}_2 = P_2 = M_2N_{P_2}$ is a standard parabolic subgroup.
We may regard $w_1$ and $w_2$ as 
\[
w_1 \in W(M_1^\circ), \quad 
w_2 \in W(M_2^\circ,M_1^\circ).
\] 
Note that $M_1 \subset M$, and the adjoint action of $w_u$ preserves $M_1$. 
Since $w_	1$ is the longest element in the subset of $W(M_1)$ consisting of elements whose representatives are in $M$, 
we see that $w_u w_1 w_u^{-1} = w_1$ in $W(M_1)$, 
and hence $w_1^{-1} w_u w_1 = w_u$. 
\par

If we consider that 
\begin{itemize}
\item
$\GL(V_\pm^{(b)})$ acts on $\Delta([-\alpha,\alpha]_{\rho_\GL})|\cdot|_E^b$; 
\item
$\GL(W_0^{(e_i)})$ acts on $\tau_i|\cdot|_E^{e_i}$ if $e_i > 0$ (\resp ${}^c\tau_i^\vee|\cdot|_E^{e_i}$ if $e_i < 0$); 
\item
$G(W_0^{(0)})$ acts on $\tau_0$, 
\end{itemize}
then we obtain irreducible representations $\pi_{M_0}$, $\pi_{M_1}$ of $M_1$ 
and $\pi_{M_2}$ of $M_2$ such that 
\begin{align*}
I_{P_0}(\pi_{M_0}) &= {}^c\II(\mm_\GL)^\vee \times {}^c\II(\mm_0)^\vee \rtimes \tau_0, \\
I_{P_1}(\pi_{M_1}) &= \II(\mm_\GL) \times \II(\mm_0) \rtimes \tau_0, \\
I_{P_2}(\pi_{M_2}) &= \left\{
\begin{aligned}
&\II(\mm) \times \Delta([-\alpha,\alpha]_{\rho_\GL}) \rtimes \tau_0 \iif 2\beta+1 \equiv 1 \bmod 2, \\
&\II(\mm) \rtimes \tau_0 \iif 2\beta+1 \equiv 0 \bmod 2,
\end{aligned}
\right. 
\end{align*}
respectively, where $P_0 = P_1$ is the standard parabolic subgroup 
with Levi subgroup $\tl{w}_1 M_1 \tl{w}_1^{-1}$.
Here, we note that $\rho_\GL$ is conjugate-self-dual.
Then we have
\[
w_2 \pi_{M_2} = \pi_{M_1}, \quad w_1 \pi_{M_1} = \pi_{M_0}.
\]
See Section \ref{sec.NIO} for these notations.
\par

Recall that 
\begin{align*}
\pi_{M_1} 
&= 
\left(
\Delta([-\alpha,\alpha]_{\rho_\GL})|\cdot|_E^{\beta} \times \dots \times \Delta([-\alpha,\alpha]_{\rho_\GL})|\cdot|_E^{-\beta} 
\right) 
\\&\quad\times 
\left(
\tau_r|\cdot|_E^{e_r} \times \dots \times \tau_1|\cdot|_E^{e_1} \rtimes \tau_0
\right).
\end{align*}
For $\lambda = (\lambda_\beta, \lambda_{\beta-1}, \dots, \lambda_{-\beta}) \in \C^{2\beta+1}$ 
and $\mu = (\mu_1, \dots, \mu_r) \in \C^r$, 
we set
\begin{align*}
\pi_{M_1, (\lambda,\mu)} 
&= 
\left(
\Delta([-\alpha,\alpha]_{\rho_\GL})|\cdot|_E^{\lambda_\beta} \times \dots \times \Delta([-\alpha,\alpha]_{\rho_\GL})|\cdot|_E^{\lambda_{-\beta}} 
\right) 
\\&\quad\times 
\left(
\tau_r|\cdot|_E^{\mu_r} \times \dots \times \tau_1|\cdot|_E^{\mu_1} \rtimes \tau_0
\right).
\end{align*}
We define $\pi_{M_0,(\lambda,\mu)}$ and $\pi_{M_2,(\lambda,\mu)}$ similarly.
Let $\phi_{M_j, (\lambda, \mu)}$ be the $L$-parameter of $\pi_{M_j, (\lambda,\mu)}$. 
Recall that the intertwining operators 
\begin{align*}
R_{P_1}(w_1,\pi_{M_1,(\lambda,\mu)}, \phi_{M_1, (\lambda, \mu)}) 
&\colon I_{P_1}(\pi_{M_1, (\lambda,\mu)}) \rightarrow I_{P_0}(\pi_{M_0, (\lambda,\mu)}), \\
R_{P_2}(w_2,\pi_{M_2,(\lambda,\mu)}, \phi_{M_2, (\lambda, \mu)}) 
&\colon I_{P_2}(\pi_{M_2, (\lambda,\mu)}) \rightarrow I_{P_1}(\pi_{M_1, (\lambda,\mu)})
\end{align*}
are defined by the meromorphic continuation of 
\begin{align*}
R_{P_j}(w_j,\pi_{M_j,(\lambda,\mu)}, \phi_{M_j, (\lambda, \mu)})f_{j,(\lambda,\mu)}(g)
&= \gamma_A(0,\phi_{M_j, (\lambda,\mu)}, \rho^\vee_{w_j^{-1}P_{j-1}|P_j}, \psi_F)
\\&\times
\lambda(w_j)^{-1} \int_{N_{P_{j-1}} \cap \tl{w}_j N_{P_j}\tl{w}_j^{-1} \bs N_{P_{j-1}}} 
f_{j,(\lambda,\mu)}(\tl{w}_j^{-1} n g) dn
\end{align*}
for $j = 1,2$.
As in \cite[Chapter XI, Proposition 2.6 (1)]{BW}, this integral converges 
when we specialize it at $\lambda = (\beta,\beta-1,\dots,-\beta)$ and $\mu = (e_1,\dots,e_r)$. 
Hence $R(w_1) = R_{P_1}(w_1,\pi_{M_1}, \phi_{M_1})$ 
and $R(w_2) = R_{P_2}(w_2,\pi_{M_2}, \phi_{M_2})$ are well-defined and nonzero.
Moreover, the image of $R(w_1)$ is exactly equal to $I_P(\pi_M)$.
\par

Set $w_u' = w_2^{-1}w_1^{-1} w_u w_1 w_2 = w_2^{-1} w_u w_2$. 
Since $w_u$ and $w_u'$ preserve the Levi subgroups $M_0$ and $M_2$, respectively, 
we obtain normalized intertwining operators
\begin{align*}
R_{P_0}(w_u,\pi_{M_0,(\lambda,\mu)}, \phi_{M_0, (\lambda, \mu)}) 
&\colon I_{P_0}(\pi_{M_0, (\lambda,\mu)}) \rightarrow I_{P_0}(w_u\pi_{M_0, (\lambda,\mu)}), \\
R_{P_2}(w_u',\pi_{M_2,(\lambda,\mu)}, \phi_{M_2, (\lambda, \mu)}) 
&\colon I_{P_2}(\pi_{M_2, (\lambda,\mu)}) \rightarrow I_{P_2}(w_u'\pi_{M_2, (\lambda,\mu)}).
\end{align*}
We will show in Lemma \ref{identity} (2) and (3) below
that $R_{P_2}(w_u',\pi_{M_2,(\lambda,\mu)}, \phi_{M_2, (\lambda, \mu)})$
is holomorphic at $\lambda = (\beta, \dots, -\beta)$ and $\mu = (e_1, \dots, e_r)$. 
Moreover, since $w_u\pi_{M_0} \cong \pi_{M_0}$ and hence $w_u'\pi_{M_2} \cong \pi_{M_2}$, 
one can normalize isomorphisms 
\[
A_{w_u} \otimes \id \colon w_u\pi_{M_0} \xrightarrow{\sim} \pi_{M_0}, 
\quad
A_{w_u'} \otimes \id \colon w_u'\pi_{M_2} \xrightarrow{\sim} \pi_{M_2}
\]
by Whittaker functionals on standard modules of general linear groups.
By composing $I_{P_2}(A_{w_u'} \otimes \id)$, we have a self-intertwining operator 
\[
R_{P_2}(w_u',\tl\pi_{M_2}, \phi_{M_2}) 
\colon I_{P_2}(\pi_{M_2}) \rightarrow I_{P_2}(\pi_{M_2}).
\]
On the other hand, 
$R_{P_0}(w_u,\pi_{M_0}, \phi_{M_0})$ might be 
a singularity of the meromorphic family $R_{P_0}(w_u,\pi_{M_0,(\lambda,\mu)}, \phi_{M_0, (\lambda, \mu)})$ at $\lambda = (\beta, \dots, -\beta)$ and $\mu = (e_1, \dots, e_r)$. 
Hence we just write 
\[
\xymatrix{
\ar@{..>}[rrr]^{R_{P_0}(w_u,\tl\pi_{M_0}, \phi_{M_0})} I_{P_0}(\pi_{M_0}) &&& I_{P_0}(\pi_{M_0}).
}
\]
\begin{lem}\label{identity}
The notation is as above. 
\begin{enumerate}
\item
The image of the map $R(w_1) \circ R(w_2) \colon \II(\mm) \rtimes \tau \rightarrow I_P(\pi_M)$ is exactly equal to $\pi$. 
\item
If $2\beta+1$ is even, then $R_{P_2}(w_u',\tl\pi_{M_2}, \phi_{M_2}) $ is the identity map.
\item
If $2\beta+1$ is odd, then 
\[
R_{P_2}(w_u',\tl\pi_{M_2}, \phi_{M_2}) = \pair{e(\rho_\GL, 2\alpha+1,1), \tau}_{\phi_\tau} \cdot \id
\]
on $\II(\mm) \rtimes \tau$,
where $\phi_\tau$ is the $L$-parameter of $\tau$. 
\end{enumerate}
\end{lem}
\begin{proof}
For (1), we claim that $\pi$ appears in $I_{P_1}(\pi_{M_1})$ as a subquotient with multiplicity one. 
Since the Jordan--H{\"o}lder series of $I_{P_1}(\pi_{M_1})$ and $I_{P_2}(\pi_{M_2})$
are the same, it is enough to consider $I_{P_2}(\pi_{M_2})$.
If $I_{P_2}(\pi_{M_2})$ is a standard module, then the claim follows from the famous fact that 
the Langlands quotient appears in its standard module with multiplicity one
(See e.g., \cite[Chapter XI, Lemma 2.13]{BW}).
Otherwise, $2\beta+1$ is odd and $I_{P_2}(\pi_{M_2})$ is a direct sum of two standard modules. 
Since the general linear parts of these two standard modules are the same, 
by computing Jacquet modules, one sees that $\pi$ appears only in one of them. 
Hence $\pi$ must appear in $I_{P_1}(\pi_{M_1})$ with multiplicity one. 
\par

Now, since $I_{P}(\pi_M)$ is a unitary induction, and hence semisimple, 
the image of $R(w_1) \circ R(w_2)$ is isomorphic to a subrepresentation of 
the maximal semisimple quotient of $\II(\mm) \rtimes \tau$. 
Since this maximal semisimple quotient is equal to $\pi$, 
we see that the image is equal to $\pi$, or $R(w_1) \circ R(w_2) = 0$.
Since $R(w_2)$ is nonzero, $\pi$ appears in its image. 
Hence if $R(w_1) \circ R(w_2)$ were to be zero, then $\pi$ would appear in the kernel of $R(w_1)$. 
On the other hand, since the image of $R(w_1)$ is equal to $I_P(\pi_M)$, which contains $\pi$, 
it would imply that $\pi$ must appear in $I_{P_1}(\pi_{M_1})$ with multiplicity greater than one. 
This contradicts the claim. 
Therefore we obtain (1).
\par

Next, we prove (2) and (3). 
For $b \in \R$, set 
\[
V^{(b)} = \bigoplus_{X \in \{V_+, V_-,W_0\}} X^{(b)}. 
\]
Since $\pi_{M_2}$ and $w_u'\pi_{M_2}$ are essentially tempered 
and since $I_{P_2}(\pi_{M_2})$ and $I_{P_2}(w_u'\pi_{M_2})$ are standard modules, 
we see that $w_u'$ preserves $V^{(b)}$ for each $b \in \R$. 
We write $\{b > 0\,|\, V^{(b)} \not= 0 \} = \{b_1, \dots, b_t\}$ with $b_1 < \dots < b_t$.
For $i=1,\dots,t$, 
we set $w_i$ to be $w_u'$ on $V^{(b_i)} \oplus V^{(-b_i)}$, and trivial on $V^{(b)}$ for $b \not= \pm b_i$.
Similarly, we define $w_{u_0}$ by $w_u'$ on $V^{(0)}$, and trivial on $V^{(b)}$ for $b \not= 0$. 
Then $w_u' = w_t \cdots w_1 \cdot w_{u_0}$. 
According to this decomposition, 
we can decompose 
$R(w'_u) = R_{P_2}(w_u',\tl\pi_{M_2}, \phi_{M_2})$ 
into a product $R(w'_u) = R(w_t) \circ \dots \circ R(w_1) \circ R(w_{u_0})$ 
by the multiplicativity of the normalized intertwining operators. 
For $i = 1,\dots,t$, 
since the operator $R(w_i)$ is induced from a normalized intertwining operator for $\GL(V^{(b_i)})$, 
we can apply Theorem \ref{untwistedGL} and obtain that $R(w_i) = \id$. 
In particular, if we are in the case (2), 
we have $V^{(0)} = W_0^{(0)}$ so that $w_{u_0} = \1$, 
and hence $R(w'_u) = R(w_{u_0}) = \id$. 
On the other hand, if $2\beta+1$ is odd, 
by \eqref{LIR} for the tempered representation $\tau$, we see that 
\[
R(w_{u_0}) = \pair{e(\rho_\GL, 2\alpha+1,1), \tau}_{\phi_\tau} \cdot \id
\]
on $\II(\mm) \rtimes \tau$.
This proves (3). 
\end{proof}

\subsection{The main diagram}
Using the notations in the previous subsection, 
with a constant $c \in \C^\times$, we now consider the following \emph{main diagram}. 
When $\kappa = 1$, it is:
\[
\xymatrix{
\II(\mm) \rtimes \tau 
\ar@{_(->}[d] 
\ar@[->][rrrr]^{R_{P_2}(w_u', \tl\pi_{M_2}, \phi_{M_2})} 
&&&& \II(\mm) \rtimes \tau 
\ar@{_(->}[d]\\
\II(\mm) \times \Delta([-\alpha,\alpha]_{\rho_\GL}) \rtimes \tau_0 
\ar@{->}[d]_{R(w_2)} 
\ar@[->][rrrr]^{R_{P_2}(w_u', \tl\pi_{M_2}, \phi_{M_2})} 
&&&& \II(\mm) \times \Delta([-\alpha,\alpha]_{\rho_\GL}) \rtimes \tau_0 
\ar@{->}[d]_{R(w_2)}\\
\II(\mm_\GL) \times \II(\mm_0) \rtimes \tau_0 
\ar@{->}[d]_{R(w_1)}
&&&&\II(\mm_\GL) \times \II(\mm_0) \rtimes \tau_0 
\ar@{->}[d]_{R(w_1)}\\
I_P(\pi_M) \ar@{_(->}[d] \ar@{->}[rrrr]^{c^{-1} \pair{\tl{u}, \tl\pi_M} R_P(w_u, \tl\pi_M, \psi_M)} 
&&&&I_P(\pi_M) 
\ar@{_(->}[d]\\
{}^c\II(\mm_\GL)^\vee \times {}^c\II(\mm_0)^\vee \rtimes \tau_0 
\ar@{.>}[rrrr]^{R_{P_0}(w_u, \tl\pi_{M_0}, \phi_{M_0})}
&&&&{}^c\II(\mm_\GL)^\vee \times {}^c\II(\mm_0)^\vee \rtimes \tau_0. 
}
\]
When $\kappa = \half{1}$, 
we replace $\Delta([-\alpha,\alpha]_{\rho_\GL}) \rtimes \tau_0$ with $\tau_0$ in the second line, 
or equivalently, we remove the second line.
The following is the main result in this section.

\begin{thm}\label{diagramG}
The main diagram is commutative
with 
\[
c = \left.\frac{\gamma_A(s,{}^c \psi_\GL \otimes \psi_0, \psi_E)}
{\gamma_A(s, {}^c\psi_\GL \otimes \phi_{\pi_0}, \psi_E)}\right|_{s=0},
\]
where we write $\pi_M = \pi_\GL \boxtimes \pi_0$, 
and $\phi_{\pi_0}$ is the $L$-parameter of $\pi_0$.
\end{thm}
\begin{proof}
As in the proof of Theorem \ref{diagramGL},
we will prove the assertion in five steps. 

\begin{description}
\setlength{\itemsep}{10pt}
\item[Step 1]
With complex parameters $\lambda \in \C^{2\beta+1}$ and $\mu \in \C^r$, 
we can consider the following diagram of 
meromorphic families of operators: 
\[
\xymatrix{
I_{P_2}(\pi_{M_2, (\lambda,\mu)})
\ar@{->}[d]_{R_{P_2}(w_2,\pi_{M_2,(\lambda,\mu)})} 
\ar@[->][rrrr]^{R_{P_2}(w_u', \pi_{M_2, (\lambda,\mu)})} 
&&&& I_{P_2}(w_u'\pi_{M_2, (\lambda,\mu)})
\ar@{->}[d]^{R_{P_2}(w_2,w_u'\pi_{M_2,(\lambda,\mu)})}\\
I_{P_1}(\pi_{M_1, (\lambda,\mu)})
\ar@{->}[d]_{R_{P_1}(w_1,\pi_{M_1,(\lambda,\mu)})}
&&&& I_{P_1}(w_u\pi_{M_1, (\lambda,\mu)})
\ar@{->}[d]^{R_{P_1}(w_1,w_u\pi_{M_1,(\lambda,\mu)})}\\
I_{P_0}(\pi_{M_0, (\lambda,\mu)})
\ar@{->}[rrrr]^{R_{P_0}(w_u, \pi_{M_0, (\lambda,\mu)})}
&&&&I_{P_0}(w_u\pi_{M_0, (\lambda,\mu)}).
}
\]
Here, we omit the $L$-parameters in the notations of the normalized intertwining operators.
If $\lambda \in (\I\R)^{2\beta+1}$ and $\mu \in (\I\R)^r$, 
then by Proposition \ref{multiplicative}, 
all maps in this diagram are regular and the diagram is commutative.
By analytic continuation, we see that 
this diagram is commutative whenever all maps are regular. 
\par

\item[Step 2]
Let $s \in \C$ be a new complex parameter. 
We will specialize the diagram in Step 1 at $\lambda = (s+\beta, s+\beta-1, \dots, s-\beta)$
and $\mu = (e_1,\dots,e_r)$. 
We write $\pi_{M_i,s}$ for the corresponding $\pi_{M_i,(\lambda,\mu)}$.
We claim that all operators in the diagram are well-defined as meromorphic families of operators in $s$.
\par

In fact, the bottom operator $R_{P_0}(w_u, \pi_{M_0, s})$
is not always a pole of the family of the operators $R_{P_0}(w_u, \pi_{M_0, (\lambda,\mu)})$
(but might be a pole at $s=0$). 
On the other hand, 
as we have seen in Lemma \ref{identity}, the other five operators are still regular at $s = 0$.
\par

Hence we can specialize the diagram in Step 1 at $\lambda = (s+\beta, s+\beta-1, \dots, s-\beta)$
and $\mu = (e_1,\dots,e_r)$, 
and obtain the following commutative diagram: 
\[
\xymatrix{
I_{P_2}(\pi_{M_2, s})
\ar@{->}[d]_{R_{P_2}(w_2,\pi_{M_2,s})} 
\ar@[->][rrrr]^{R_{P_2}(w_u', \pi_{M_2, s})} 
&&&& I_{P_2}(w_u'\pi_{M_2, s})
\ar@{->}[d]^{R_{P_2}(w_2,w_u'\pi_{M_2,s})}\\
I_{P_1}(\pi_{M_1, s})
\ar@{->}[d]_{R_{P_1}(w_1,\pi_{M_1,s})}
&&&& I_{P_1}(w_u\pi_{M_1, s})
\ar@{->}[d]^{R_{P_1}(w_1,w_u\pi_{M_1,s})}\\
I_{P_0}(\pi_{M_0, s})
\ar@{->}[rrrr]^{R_{P_0}(w_u, \pi_{M_0, s})}
&&&&I_{P_0}(w_u\pi_{M_0, s}).
}
\]

\item[Step 3]
Note that the image of 
$R_{P_1}(w_1,\pi_{M_1,s}) \colon I_{P_1}(\pi_{M_1, s}) \rightarrow I_{P_0}(\pi_{M_0, s})$
is equal to $I_P(\pi_{M,s}) = I_P(\pi_\GL|\cdot|_E^s \boxtimes \pi_0)$.
If we set 
\[
c(s) = 
\frac{\gamma_A(0,\psi_{M,s}, \rho_{w_u^{-1} P|P}^\vee, \psi_F)}
{\gamma_A(0,\phi_{M_0,s}, \rho_{w_u^{-1} P_0|P_0}^\vee, \psi_F)},
\]
then using a canonical homeomorphism
\[
N_P \cap \tl{w}_u N_P \tl{w}_u^{-1} \bs N_P 
\cong 
N_{P_0} \cap \tl{w}_u N_{P_0} \tl{w}_u^{-1} \bs N_{P_0},
\]
we obtain a commutative diagram of meromorphic families of operators
\[
\xymatrix{
I_{P}(\pi_{M, s})
\ar@{_(->}[d]
\ar@{->}[rrrr]^{c(s)^{-1}R_{P}(w_u, \pi_{M, s}, \psi_{M,s})}
&&&& I_{P}(w_u\pi_{M, s})
\ar@{_(->}[d]\\
I_{P_0}(\pi_{M_0, s})
\ar@{->}[rrrr]^{R_{P_0}(w_u, \pi_{M_0, s})}
&&&&I_{P_0}(w_u\pi_{M_0, s})
}
\]
where the vertical maps are the canonical inclusions.
Here, we explain these canonical inclusions more precisely. 
The representation $\pi_{M,s}$ is realized as 
the unique irreducible subrepresentation of $\Ind_{P_0 \cap M}^M(\pi_{M_0,s})$, 
and the functor $I_P$ induces a canonical inclusion 
$I_P(w_u\pi_{M,s}) \hookrightarrow I_P(w_u\Ind_{P_0 \cap M}^M(\pi_{M_0,s}))$. 
The latter induced representation can be realized as a subspace of two-variable functions 
$f \colon G \times M \rightarrow \VV$, 
where $\VV$ is a space of $\pi_{M_0,s}$, satisfying 
\[
f(m'g, m) = \delta_P^{\half{1}}(m')f(g, m \tl{w}_u^{-1}m' \tl{w}_u)
\]
for $m,m' \in M$ and $g \in G$.
Now the right inclusion in the diagram is induced from the isomorphism 
\[
I_P(w_u\Ind_{P_0 \cap M}^M(\pi_{M_0,s})) \xrightarrow{\sim} I_{P_0}(w_u\pi_{M_0, s}), 
\quad f(g,m) \mapsto f(g,\1).
\] 
The left inclusion is similar.
\par

Combining the above diagram with the one in Step 2, 
we obtain the following commutative diagram: 
\[
\xymatrix{
I_{P_2}(\pi_{M_2, s})
\ar@{->}[d]_{R_{P_2}(w_2,\pi_{M_2,s})} 
\ar@[->][rrrr]^{R_{P_2}(w_u', \pi_{M_2, s})} 
&&&& I_{P_2}(w_u'\pi_{M_2, s})
\ar@{->}[d]^{R_{P_2}(w_2,w_u'\pi_{M_2,s})}\\
I_{P_1}(\pi_{M_1, s})
\ar@{->}[d]_{R_{P_1}(w_1,\pi_{M_1,s})}
&&&& I_{P_1}(w_u\pi_{M_1, s})
\ar@{->}[d]^{R_{P_1}(w_1,w_u\pi_{M_1,s})}\\
I_{P}(\pi_{M, s})
\ar@{->}[rrrr]^{c(s)^{-1}R_{P}(w_u, \pi_{M, s}, \psi_{M,s})}
&&&& I_{P}(w_u\pi_{M, s}).
}
\]
We note that 
\[
c(s) = 
\frac{\gamma_A(s, \psi_\GL \otimes \psi_0^\vee, \psi_E)}
{\gamma_A(s, \psi_\GL \otimes \phi_{\pi_0}^\vee, \psi_E)}
\]
so that $c(0) = c$ by Proposition \ref{gamma_c}.
We will see that $c(s)$ is regular and nonzero at $s=0$ in the next step.

\item[Step 4]
We would like to specialize the commutative diagram above at $s = 0$.
As we have noted in Step 2, 
the five operators appearing in the top, left and right of the last diagram are regular at $s=0$.
In particular, the composition 
\[
R_{P_1}(w_1, w_u\pi_{M_1,s}) 
\circ R_{P_2}(w_2, w_u'\pi_{M_2,s})
\circ R_{P_2}(w_u', \pi_{M_2, s}), 
\]
and hence 
\[
c(s)^{-1}R_{P}(w_u, \pi_{M, s}, \psi_{M,s})
\circ 
R_{P_1}(w_1,\pi_{M_1,s}) 
\circ R_{P_2}(w_2,\pi_{M_2,s})
\]
are regular and nonzero at $s=0$.
We can specialize the last diagram at $s=0$, 
and obtain the commutative diagram
\[
\xymatrix{
I_{P_2}(\pi_{M_2})
\ar@{->}[d]_{R_{P_2}(w_2,\pi_{M_2})} 
\ar@[->][rrrr]^{R_{P_2}(w_u', \pi_{M_2})} 
&&&& I_{P_2}(w_u'\pi_{M_2})
\ar@{->}[d]^{R_{P_2}(w_2,w_u'\pi_{M_2})}\\
I_{P_1}(\pi_{M_1})
\ar@{->}[d]_{R_{P_1}(w_1,\pi_{M_1})}
&&&& I_{P_1}(w_u\pi_{M_1})
\ar@{->}[d]^{R_{P_1}(w_1,w_u\pi_{M_1})}\\
I_{P}(\pi_{M})
\ar@{->}[rrrr]^{c^{-1}R_{P}(w_u, \pi_{M}, \psi_M)}
&&&& I_{P}(w_u\pi_{M}).
}
\]
Note that $R_{P}(w_u, \pi_{M}, \psi_{M})$ is well-defined and nonzero
by \cite[Proposition 2.3.1]{Ar} and \cite[Proposition 3.3.1]{Mok}. 
Since $c^{-1}R_{P}(w_u, \pi_{M}, \psi_M)$ is nonzero, 
we conclude that $c(s)$ is regular and nonzero at $s=0$.

\item[Step 5]
If we realize $w_u'\pi_{M_2}$, $\pi_{M_2}$, $w_u\pi_{M_1}$ and $\pi_{M_1}$
on the same vector space, say $\VV$, 
then $A_{w_u'} \otimes \id$ is a linear isomorphism $\Phi \colon \VV \rightarrow \VV$ satisfying  
\[
\Phi \circ \pi_{M_2}(\tl{w}_u'^{-1} m_2 \tl{w}_u') = \pi_{M_2}(m_2) \circ \Phi, \quad m_2 \in M_2.
\]
Since $w_u' = w_2^{-1}w_uw_2$, 
by Lemma \ref{center}, 
the above property of $\Phi$ can be rewritten as 
\[
\Phi \circ \pi_{M_1}(\tl{w}_u^{-1} m_1 \tl{w}_u) = \pi_{M_1}(m_1) \circ \Phi, \quad m_1 \in M_1.
\]
Therefore, $\Phi$ is also equal to the normalized isomorphism 
$A_{w_u} \otimes \id \colon w_u \pi_{M_1} \xrightarrow{\sim} \pi_{M_1}$, 
and hence the diagram 
\[
\xymatrix{
I_{P_2}(w_u'\pi_{M_2})
\ar@{->}[d]_{R_{P_2}(w_2,w_u'\pi_{M_2})} 
\ar@[->][rrrr]^{I_{P_2}(A_{w_u'} \otimes \id)} 
&&&& I_{P_2}(\pi_{M_2})
\ar@{->}[d]^{R_{P_2}(w_2,\pi_{M_2})}\\
I_{P_1}(w_u\pi_{M_1})
\ar@[->][rrrr]^{I_{P_1}(A_{w_u} \otimes \id)} 
&&&& I_{P_1}(\pi_{M_1})
}
\]
is commutative.
On the other hand, 
since $\Ind_{P_1\cap M}^M(\pi_{M_1})$ is the standard module of $\pi_M$, 
by the definition of $\AA_{w_u}$ together with the functoriality of $I_P$, 
we have the following commutative diagram
\[
\xymatrix{
I_P(w_u\Ind_{P_1\cap M}^M(\pi_{M_1}))
\ar@{->}[d]_{I_P(R_{P_1\cap M}(w_1,\pi_{M_1}))}
\ar@[->][rrrr]
&&&& I_P(\Ind_{P_1\cap M}^M(\pi_{M_1}))
\ar@{->}[d]^{I_P(R_{P_1 \cap M}(w_1,\pi_{M_1}))}\\
I_P(w_u\pi_{M})
\ar@{->}[rrrr]^{I_P(\AA_{w_u} \otimes \id)}
&&&&I_P(\pi_{M}).
}
\]
Here the top map is induced from the isomorphism 
\[
\AA \otimes \id \colon 
w_u\Ind_{P_1\cap M}^M(\pi_{M_1}) \xrightarrow{\sim} \Ind_{P_1\cap M}^M(\pi_{M_1}),
\]
where $\AA$ is the isomorphism of standard modules of $\GL_k(E)$
normalized by using a Whittaker functional. 
Via the analogous identification $I_P(\Ind_{P_1\cap M}^M(\pi_{M_1})) \cong I_{P_1}(\pi_{M_1})$
(\resp $I_P(w_u\Ind_{P_1\cap M}^M(\pi_{M_1})) \cong I_{P_1}(w_u\pi_{M_1})$) in Step 3, 
the above commutative diagram is rewritten as 
\[
\xymatrix{
I_{P_1}(w_u\pi_{M_1})
\ar@{->}[d]_{R_{P_1}(w_1,w_u\pi_{M_1})}
\ar@[->][rrrr]^{I_{P_1}(A_{w_u} \otimes \id)} 
&&&& I_{P_1}(\pi_{M_1})
\ar@{->}[d]^{R_{P_1}(w_1,\pi_{M_1})}\\
I_{P}(w_u\pi_{M})
\ar@{->}[rrrr]^{I_P(\AA_{w_u} \otimes \id)}
&&&&I_{P}(\pi_{M}).
}
\]
Combining this with the first diagram in Step 5, 
and using Lemma \ref{c3}, 
we obtain the commutative diagram 
\[
\xymatrix{
I_{P_2}(w_u'\pi_{M_2})
\ar@{->}[d]_{R_{P_2}(w_2,w_u'\pi_{M_2})} 
\ar@[->][rrrr]^{I_{P_2}(A_{w_u'} \otimes \id)} 
&&&& I_{P_2}(\pi_{M_2})
\ar@{->}[d]^{R_{P_2}(w_2,\pi_{M_2})}\\
I_{P_1}(w_u\pi_{M_1})
\ar@{->}[d]_{R_{P_1}(w_1,w_u\pi_{M_1})}
&&&& I_{P_1}(\pi_{M_1})
\ar@{->}[d]^{R_{P_1}(w_1,\pi_{M_1})}\\
I_{P}(w_u\pi_{M})
\ar@{->}[rrrr]^{I_{P}(\pair{\tl{u}, \tl\pi_M}\tl{\pi}_M(w_u))}
&&&&I_{P}(\pi_{M}).
}
\]
This together with the diagram obtained in Step 4
implies that the main diagram is commutative.
\end{description}
This completes the proof of Theorem \ref{diagramG}.
\end{proof}

By the commutativity of the main diagram together with Lemma \ref{identity}, we have
\[
\pair{\tl{u}, \tl\pi_M} R_P(w_u, \tl\pi_M, \psi_M)|_\pi 
= \left\{
\begin{aligned}
&c \cdot \id_\pi \iif 2\beta+1 \equiv 0 \bmod 2, \\
&c \pair{e(\rho_\GL, 2\alpha+1,1), \tau}_{\phi_\tau} \cdot \id_\pi 
\iif 2\beta+1 \equiv 1 \bmod 2. 
\end{aligned}
\right. 
\]
Therefore, we obtain the following summary. 

\begin{cor}\label{summary}
Assume Hypothesis \ref{arthur}.
Let $\psi_M = \psi_{\GL} \oplus \psi_{0}$ be an $A$-parameter for $M$
such that $\psi_\GL = \rho_\GL \boxtimes S_{2\alpha+1} \boxtimes S_{2\beta+1}$ 
is irreducible and conjugate-self-dual. 
We assume further 
the existence of the $A$-packet $\Pi_\psi$ together with the pairing $\pair{\cdot, \pi}_\psi$
satisfying \eqref{ECR1} and \eqref{ECR2}. 
For $\pi_M \in \Pi_{\psi_M}$, 
let $\pi \subset I_P(\pi_M)$ be a highly non-tempered summand, 
with the standard module $\II(\mm) \rtimes \tau$. 
Then \eqref{LIR} holds for $\pi \subset I_P(\pi_M)$ 
if and only if the equation
\[
\tag{$\star$}\label{star}
\left.\frac{\gamma_A(s, {}^c\psi_\GL \otimes \psi_0, \psi_E)}
{\gamma_A(s, {}^c\psi_\GL \otimes \phi_{\pi_0}, \psi_E)}\right|_{s=0} 
= 
\left\{
\begin{aligned}
&\pair{s_u, \pi}_\psi \iif 2\beta+1 \equiv 0 \bmod 2, \\
&\frac{\pair{s_u,\pi}_\psi}{\pair{e(\rho_\GL, 2\alpha+1,1), \tau}_{\phi_\tau}} \iif 2\beta+1 \equiv 1 \bmod 2
\end{aligned}
\right.
\]
holds, 
where $\phi_{\pi_0}$ (\resp $\phi_\tau$) is the $L$-parameter of $\pi_0$ (\resp $\tau$). 
\end{cor}

\begin{rem}
We will apply this corollary to co-tempered $A$-parameters $\psi$ for $G$, 
and certain $A$-parameters $\psi'$ for $G'$ with $\dim(\St_{\widehat{G'}}) < \dim(\St_{\widehat{G}})$.
Hence we have $A$-packets $\Pi_{\psi}$ and $\Pi_{\psi'}$ together with their pairings by Hypothesis \ref{arthur} and Theorem \ref{ECR-cotemp}. 
However, Hypothesis \ref{arthur} assumes \eqref{A-LIR}, rather than \eqref{LIR}, for $\psi'$. 
To deduce \eqref{LIR} for a certain $\psi'$ which is not co-tempered, 
we need Lemma \ref{mult-free} below. 
We also notice that 
we cannot use M{\oe}glin's result claiming that $\Pi_{\psi'}$ is multiplicity-free
because 
M{\oe}glin's explicit construction of $A$-packets uses endoscopic character relations for higher rank cases. 
\end{rem}

\section{Computations of local factors}\label{computations}
This section continues the work of Section \ref{sec.LIR}. 
Our goal is to show \eqref{LIR} for each irreducible summand $\pi \subset I_P(\pi_M)$ and each $\pi_M \in \Pi_{\psi_M}$, 
where $P=MN_P$ is a standard parabolic subgroup of a classical group $G$, 
and $\psi_M$ is an arbitrary co-tempered $A$-parameter for $M$ (Theorem \ref{main3} (2)). 
Lemma \ref{reduction} reduces the problem to the case where $P$ is maximal so that $M \cong \GL_k(E) \times G_0$. 
By the key lemma (Lemma \ref{halve}), 
we may then assume that $\pi \subset I_P(\pi_M)$ is a highly non-tempered summand, 
which exists by Lemma \ref{existence}. 
For such a representation, \eqref{LIR} is equivalent to the scalar equation \eqref{star} in Corollary \ref{summary}.
In this section, we check by hand the validity of the equation \eqref{star} for our case.

\subsection{Preliminaries}
Let $P=MN$ be a standard maximal parabolic subgroup of $G$, 
and let $\psi_M = \widehat\phi_M = \psi_\GL \oplus \psi_0$ be a co-tempered $A$-parameter for $M$. 
Write $\phi_M = \phi_\GL \oplus \phi_0$ so that $\widehat\phi_\GL = \psi_\GL$ and $\widehat\phi_0 = \psi_0$.
By Lemma \ref{reduction}, we may assume that $\phi_\GL$ is irreducible and conjugate-self-dual. 
Let $\phi$ (\resp $\psi$) be the $L$-parameter (\resp $A$-parameter) for $G$ given by $\phi_M$ (\resp $\psi_M$).
\par

First, we reduce the problem to the case of good parity. 
\begin{lem}
Write 
\[
\phi_0 = \phi_{0,\bad} \oplus \phi_{0, \good} \oplus {}^c\phi_{0,\bad}^\vee, 
\]
where 
$\phi_{0,\good}$ is the sum of irreducible conjugate-self-dual representations of the same type as $\phi_0$, 
and $\phi_{0,\bad}$ is a sum of irreducible representations of other types.
Set $\psi_0 = \widehat\phi_0$ and $\psi_{0,\good} = \widehat\phi_{0,\good}$. 
For $\pi_0 \in \Pi_{\psi_0}$, let $\pi_{0,\good} \in \Pi_{\psi_{0,\good}}$ be the representation determined by 
$\pair{\cdot, \pi_0}_{\psi_0} = \pair{\cdot, \pi_{0,\good}}_{\psi_{0,\good}}$ via the canonical identification $A_{\psi_0} = A_{\psi_{0,\good}}$. 
Then we have
\[
\phi_{\psi_0}-\phi_{\psi_{0,\good}} = \phi_{\pi_0}-\phi_{\pi_{0,\good}}. 
\]
On the other hand, if we take $\psi_\good$ and $\pi_\good$ similarly, 
then $\pair{s_u,\pi}_\psi = \pair{s_u, \pi_\good}_{\psi_\good}$. 
In particular, the equation \eqref{star} holds for $\pi_\GL \boxtimes \pi_0$ 
if and only if \eqref{star} holds for $\pi_\GL \boxtimes \pi_{0,\good}$.
\end{lem}
\begin{proof}
By applying Aubert duality to the tempered case, 
we see that 
\[
\pi_0 = \tau_{0,\bad} \rtimes \pi_{0,\good},
\]
where $\tau_{0,\bad}$ is the representation of a general linear group 
corresponding to $\psi_{0,\bad} = \widehat\phi_{0,\bad}$. 
Since it is an irreducible parabolic induction, 
by \cite[Proposition 1.3]{T2},
one can describe the Langlands data of $\pi_0$ using those of $\pi_{0,\good}$ 
similar to Definition \ref{HNT}. 
It shows that $(\phi_{\pi_0}-\phi_{\pi_{0,\good}})(w, \alpha)$ is equal to 
\[
(\phi_{\psi_0}-\phi_{\psi_{0,\good}})(w, \alpha) 
= 
(\phi_{0,\bad} \oplus \phi_{0,\bad}^\vee)\left(w, 
\begin{pmatrix}
|w|_E^{\half{1}} & 0 \\
0 & |w|_E^{-\half{1}}
\end{pmatrix}
\right)
\]
for $(w,\alpha) \in W_E \times \SL_2(\C)$.
\par

On the other hand, note that 
$\pair{\cdot, \pi}_{\psi} = \pair{\cdot, \pi_{\good}}_{\psi_{\good}}$ 
via the canonical identification $A_{\psi} = A_{\psi_{\good}}$. 
Since $s_u \in A_\psi$ is the same as $s_u \in A_{\psi_\good}$ via this identification, 
we conclude that $\pair{s_u,\pi}_\psi = \pair{s_u, \pi_\good}_{\psi_\good}$. 
This implies the last assertion. 
\end{proof}
Hereafter, we always assume that $\phi_0$ is of good parity, i.e., $\phi_0 = \phi_{0,\good}$.
\par

To show the equation \eqref{star}, we need to know the pairing $\pair{s_u,\pi}_\psi$ for $\pi \in \Pi_\psi$. 
Since $\psi = \widehat\phi$ is co-tempered, 
this pairing is defined such that the statement of Corollary \ref{+-temp} holds. 
For convenience, we write this corollary once more. 
\begin{lem}\label{6.19}
Let $\phi = \oplus_{j=1}^t \rho_j \boxtimes S_{d_j}$ be a tempered $L$-parameter for $G$ of good parity, 
where $\rho_j$ is an irreducible representation of $W_F$. 
Set $\psi = \widehat\phi$. 
Then for $\pi \in \Pi_{\psi}$ and $\sigma = \hat\pi \in \Pi_{\phi}$, we have
\[
\frac{\pair{e(\rho, 1, d), \pi}_{\psi}}{\pair{e(\rho, d, 1), \sigma}_{\phi}}
= 
\left\{
\begin{aligned}
&1 \iif d \equiv 0 \bmod 2, \\
&{-(-1)^{|\{j \,|\, \rho_j \cong \rho\}|}} \iif d \equiv 1 \bmod 2.
\end{aligned}
\right. 
\]
In particular, if $d \equiv d' \bmod 2$, 
then we have
\[
\frac{\pair{e(\rho, 1, d), \pi}_{\psi}}{\pair{e(\rho, 1, d'), \pi}_{\psi}}
= 
\frac{\pair{e(\rho, d, 1), \sigma}_{\phi}}{\pair{e(\rho, d', 1), \sigma}_{\phi}}.
\]
\end{lem}
\begin{proof}
This follows from Corollary \ref{+-temp}. 
\end{proof}
\par

We write $\phi_\GL = \rho_\GL \boxtimes S_{2\alpha+1}$.
Then $s_u = e(\rho_\GL,1,2\alpha+1) \in A_\psi$ if $\phi_\GL$ is of the same type as $\phi_0$. 
We compute $\pair{s_u, \pi}_\psi$ for a highly non-tempered summand $\pi \subset I_P(\pi_M)$.
\begin{lem}\label{<s_u>}
Suppose that $\phi_\GL = \rho_\GL \boxtimes S_{2\alpha+1}$ is of the same type as $\phi_0$.
Set $\psi_M = \widehat{\phi}_M$ with $\phi_M = \phi_\GL \oplus \phi_0$.
For $\pi_M = \pi_\GL \boxtimes \pi_0 \in \Pi_{\psi_M}$, 
let $\pi \subset I_P(\pi_M)$ be a highly non-tempered summand, 
and let $\II(\mm) \rtimes \tau$ be its standard module. 

\begin{enumerate}
\item
If $\phi_0$ contains $\rho_\GL \boxtimes S_d$ with $d \geq 2\alpha+1$, 
then 
\[
\pair{s_u, \pi}_\psi = \pair{e(\rho_\GL, 1, d_+), \pi_0}_{\psi_0}, 
\]
where $d_+ = \min\{d \geq 2\alpha+1 \,|\, \rho_\GL \boxtimes S_d \subset \phi_0\}$. 

\item
Otherwise, 
if $\phi_0$ contains $\rho_\GL \boxtimes S_d$ with $d < 2\alpha+1$, 
then 
\[
\pair{s_u, \pi}_\psi = \pair{e(\rho_\GL, 1, d_-), \pi_0}_{\psi_0}, 
\]
where $d_- = \max\{d < 2\alpha+1 \,|\, \rho_\GL \boxtimes S_d \subset \phi_0\}$. 
Here, when $2\alpha+1 \equiv 0 \bmod 2$, 
we formally understand that 
$\rho_\GL \boxtimes S_0 \subset \phi_0$ and $\pair{e(\rho_\GL, 1, 0), \pi_0}_{\psi_0} = 1$.

\item
Otherwise, i.e., if $2\alpha+1 \equiv 1 \bmod 2$ and 
if $\phi_0$ does not contain $\rho_\GL \boxtimes S_d$ for any $d \geq 1$, 
then there are precisely two choices of $\pi$, 
and $\pair{s_u, \pi}_\psi$ can take either value in $\{\pm1\}$ depending on this choice.

\end{enumerate}
\end{lem}
\begin{proof}
Write $\hat\pi_M = \sigma_M = \sigma_\GL \boxtimes \sigma_0 \in \Pi_{\phi_M}$ and $\hat\pi = \sigma$. 
Recall that if $\II(\mm_0) \rtimes \tau_0$ is the standard module of $\pi_0$, 
then the standard module of $\pi$ is $\II(\mm) \rtimes \tau$, 
where 
\[
\mm = \mm_0 
+ 2[\alpha,\alpha]_{\rho_\GL} + 2[\alpha-1,\alpha-1]_{\rho_\GL} + \dots + 2[\kappa,\kappa]_{\rho_\GL}
\]
with $\kappa \in \{1,\half{1}\}$ such that $\alpha-\kappa \in \Z$.
By \cite[Lemma 2.2]{AG}, we have
\[
\pi \hookrightarrow {}^cL(\mm)^\vee \rtimes \tau, 
\quad
\pi_0 \hookrightarrow {}^cL(\mm_0)^\vee \rtimes \tau_0, 
\]
where $L(\mm)$ (\resp $L(\mm_0)$) is the Langlands quotient of $\II(\mm)$ (\resp $\II(\mm_0)$).
Set $d = \dim(\rho_\GL)$. 
For $k > 0$, we denote by $P_{dk}$ (\resp $P_{dk,0}$) the standard maximal parabolic subgroup of $G$ (\resp $G_0$)
with Levi subgroup of the form $\GL_{dk}(E) \times G'$ (\resp $\GL_{dk}(E) \times G'_0$). 
\par

Suppose that we are in the case (1). 
If $d_+ = 2\alpha+1$, then $\sigma = I_P(\sigma_M)$ is irreducible 
and $\pair{e(\rho_\GL,2\alpha+1,1),\sigma}_\phi = \pair{e(\rho_\GL, d_+,1), \sigma_0}_{\phi_0}$.
If $d_+ > 2\alpha+1$, then by Theorem \ref{Dk}, 
with $x = \half{d_+-1}$, we have 
\[
\Jac_{P_{dkl,0}}(\sigma_0) \geq
(\rho_\GL|\cdot|_E^{x})^k \times \dots \times (\rho_\GL|\cdot|_E^{\alpha+1})^k \otimes (\text{nonzero}) \geq 0
\]
in the appropriate Grothendieck group
with $k$ being the multiplicity of $\rho_\GL \boxtimes S_{d_+}$ in $\phi_0$, 
and $l = x-\alpha$. 
By a property of Aubert duality (Lemma \ref{AvsD}), 
it implies that $\Jac_{P_{dkl,0}}(\pi_0)$ and hence $\Jac_{P_{dkl,0}}({}^cL(\mm_0)^\vee \rtimes \tau_0)$
contain
\[
(\rho_\GL|\cdot|_E^{-x})^k \times \dots \times (\rho_\GL|\cdot|_E^{-(\alpha+1)})^k \otimes (\text{nonzero}).
\]
By Tadi\'c's formula and Casselman's criterion (Theorems \ref{Tformula} and \ref{Ccriterion}),
we see that 
\[
\Jac_{R_{dkl,0}}({}^cL(\mm_0)^\vee) 
\geq (\rho_\GL|\cdot|_E^{-x})^k \times \dots \times (\rho_\GL|\cdot|_E^{-(\alpha+1)})^k \otimes (\text{nonzero}) \geq 0,
\]
where $R_{dkl,0}$ is a suitable standard maximal parabolic subgroup of a general linear group.
Then by the definition of $\mm$ together with \cite[Theorem 2.2.1, Proposition 2.1.4]{Jan}, 
we have 
\[
\Jac_{R_{dkl}}({}^cL(\mm)^\vee) 
\geq (\rho_\GL|\cdot|_E^{-x})^k \times \dots \times (\rho_\GL|\cdot|_E^{-(\alpha+1)})^k \otimes (\text{nonzero}) \geq 0
\]
with analogous notations.
By reversing the analogous argument as above, we have
\[
\Jac_{P_{dkl}}(\pi) \geq
(\rho_\GL|\cdot|_E^{-x})^k \times \dots \times (\rho_\GL|\cdot|_E^{-(\alpha+1)})^k \otimes (\text{nonzero}) \geq 0
\]
and hence
\[
\Jac_{P_{dkl}}(\sigma) \geq
(\rho_\GL|\cdot|_E^{x})^k \times \dots \times (\rho_\GL|\cdot|_E^{\alpha+1})^k \otimes (\text{nonzero}) \geq 0
\]
in the appropriate Grothendieck group.
By Theorem \ref{Dk}, 
this happens exactly when 
\[
\pair{e(\rho_\GL,2\alpha+1,1),\sigma}_\phi = \pair{e(\rho_\GL, d_+,1), \sigma_0}_{\phi_0}.
\]
\par

If we are in the case (2) or (3), 
then by \cite[Theorem 2.2.1, Proposition 2.1.4]{Jan}, 
with $y = \half{d_--1}$, 
we have 
\[
\Jac_{R_{2dl}}({}^cL(\mm)^\vee) \geq 
(\rho_\GL|\cdot|_E^{-\alpha})^2 \times \dots \times (\rho_\GL|\cdot|_E^{-(y+1)})^2 \otimes (\text{nonzero})
\geq 0
\]
in the appropriate Grothendieck group with $l = \alpha-y$.
Here, in the case (3), we set $d_- = 1$ so that $y=0$.
Then we have 
\[
\Jac_{P_{2dl}}(\pi) \geq 
(\rho_\GL|\cdot|_E^{-\alpha})^2 \times \dots \times (\rho_\GL|\cdot|_E^{-(y+1)})^2 \otimes (\text{nonzero})
\geq 0
\]
and hence
\[
\Jac_{P_{2dl}}(\sigma) \geq 
(\rho_\GL|\cdot|_E^{\alpha})^2 \times \dots \times (\rho_\GL|\cdot|_E^{y+1})^2 \otimes (\text{nonzero})
\geq 0
\]
by Lemma \ref{AvsD}. 
By Theorem \ref{Dk}, 
this occurs exactly when we are in the case (3), or 
\[
\pair{e(\rho_\GL,2\alpha+1,1),\sigma}_\phi = \pair{e(\rho_\GL, d_-,1), \sigma_0}_{\phi_0}.
\]
Hence we obtain (1)--(3) by Lemma \ref{6.19}. 
\end{proof}

\subsection{Strategy of the proof}\label{strategy}
We will prove \eqref{star} in Corollary \ref{summary} for $\pi_M = \pi_\GL \boxtimes \pi_0 \in \Pi_{\psi_M}$ 
with $\psi_M = \widehat\phi_M$ a co-tempered $A$-parameter.
Notice that the left-hand side of \eqref{star} involves the $L$-parameter $\phi_{\pi_0}$ of $\pi_0$.
In general, it is very difficult to list $\phi_{\pi_0}$ for $\pi_0 \in \Pi_{\psi_0}$.
Instead of computing $\phi_{\pi_0}$ explicitly, 
we will give an inductive argument as follows.
\par

The initial case is where $\pi_0$ is \emph{almost supercuspidal}, 
which is defined as follows.

\begin{defi}\label{a-cusp}
We say that an irreducible representation $\pi_0$ of $G_0$ is \emph{almost supercuspidal} 
if the following condition holds for every maximal parabolic subgroup $P_0$ of $G_0$.
If $\Jac_{P_0}(\pi_0)$ contains an irreducible subquotient of the form 
$\rho \boxtimes \sigma$ with $\rho$ a supercuspidal representation of $\GL_k(E)$, 
then $\rho$ is unitary.
\end{defi}

Note that $\pi_0$ is almost supercuspidal if and only if so is $\hat\pi_0$ by Lemma \ref{AvsD}. 
The assumption that $\pi_0$ is almost supercuspidal implies the following strong properties. 

\begin{lem}\label{initial}
Let $\psi_0 = \widehat\phi_0$ be a co-tempered $A$-parameter of good parity for a classical group $G_0$. 
Suppose that $\pi_0 \in \Pi_{\psi_0}$ is almost supercuspidal. 

\begin{enumerate}
\item
The following conditions hold: 
\begin{itemize}
\item
If we denote the multiplicity of $\rho \boxtimes S_d$ in $\phi_0$ by $m_{\phi_0}(\rho,d)$, 
then $m_{\phi_0}(\rho,d) \leq 1$ for any $d \geq 2$;

\item
if $\rho \boxtimes S_{d} \subset \phi_0$ with $d > 2$, 
then $\rho \boxtimes S_{d-2} \subset \phi_0$ and 
\[
\pair{e(\rho, 1, d), \pi_0}_{\psi_0} = - \pair{e(\rho, 1, d-2), \pi_0}_{\psi_0};
\]
\item
if $\rho \boxtimes S_{2} \subset \phi_0$, 
then 
\[
\pair{e(\rho, 1, 2), \pi_0}_{\psi_0} = -1. 
\]
\end{itemize}

\item
Let $\{\rho'_1 \boxtimes S_1, \dots, \rho'_r \boxtimes S_1\}$ be 
the set of irreducible representations of $W_E \times \SL_2(\C)$ 
appearing in $\phi_0$ with even multiplicity, 
and set 
\[
2y_i+1 = \max\{d \geq 1\,|\, \rho'_i \boxtimes S_d \subset \phi_0\}. 
\]
Then the $L$-parameter $\phi_{\pi_0}$ of $\pi_0$ is given by 
\[
\phi_{\pi_0} = \phi_0 - \bigoplus_{i=1}^r \rho'_i \boxtimes (S_1 \oplus S_{2y_i+1})
\oplus \bigoplus_{i=1}^r \rho'_i (|\cdot|_E^{\half{y_i}} \oplus |\cdot|_E^{-\half{y_i}}) \boxtimes S_{y_i+1}.
\]

\item
Suppose that $\phi_\GL = \rho_\GL \boxtimes S_{2\alpha+1}$ is of the same type as $\phi_0$, 
and that $2\alpha+1 \equiv 1 \bmod 2$. 
Set $\pi_M = \pi_\GL \boxtimes \pi_0$ and 
let $\pi \subset I_P(\pi_M)$ be a highly non-tempered summand.
We denote by $\II(\mm) \rtimes \tau$ the standard module of $\pi$. 
Then 
\[
\pair{e(\rho_\GL, 1,1), \tau}_{\phi_\tau} = \pair{e(\rho_\GL, 1, d_0), \pi}_{\psi}, 
\]
where $\psi = \widehat\phi$ is the $A$-parameter for $G$ given by the dual of 
$\phi = \phi_\GL \oplus \phi_0 \oplus {}^c \phi_\GL^\vee$, 
and $d_0 = \max\{d \geq 1 \,|\, \rho_\GL \boxtimes S_d \subset \phi\}$. 

\end{enumerate}
\end{lem}
\begin{proof}
\begin{enumerate}
\item
Recall that $\hat\pi_0$ is tempered since it is in $\Pi_{\phi_0}$. 
Since $\pi_0$ is almost supercuspidal, so is $\hat\pi_0$. 
Hence by Corollary \ref{sc}, we obtain several properties of $m_{\phi_0}(\rho,d)$ and $\pair{\cdot, \hat\pi_0}_{\phi_0}$.
Then Lemma \ref{6.19} implies the desired properties of $\pair{\cdot, \pi_0}_{\psi_0}$.

\item
As in \cite[Proposition 5.4]{AM}, 
Aubert duality together with Theorem \ref{Dk} and Remark \ref{0-derivative} shows that
if we write $y_1 \geq \dots \geq y_t > 0 = y_{t+1} = \dots = y_r$, 
then the standard module of $\pi_0$ is 
\[
\II([0,y_1]_{\rho'_1}+\dots+[0,y_t]_{\rho'_t}) \rtimes \tau_0
\]
where $\tau_0 \in \Pi_{\phi_{\tau_0}}$ with 
\[
\phi_{\tau_0} = \phi_0 - \bigoplus_{i=1}^t \rho'_i \boxtimes (S_1 \oplus S_{2y_i+1})
\]
and 
\[
\pair{e(\rho, d, 1), \tau_0}_{\phi_{\tau_0}} = \left\{
\begin{aligned}
&{-\pair{e(\rho,d,1), \hat\pi_0}_{\phi_0}} \iif \rho \in \{\rho'_1, \dots, \rho'_r\}, \\
&\pair{e(\rho,d,1), \hat\pi_0}_{\phi_0} \other.
\end{aligned}
\right. 
\]
Here, we notice that \cite[Proposition 5.4]{AM} does not use M{\oe}glin's construction of $A$-packets. 
From this, we obtain the description for $\phi_{\pi_0}$ in (2).
\par

\item
If $\pi \subset I_P(\pi_M)$ is a highly non-tempered summand, 
then by Theorems \ref{Dk}, \ref{Dk-1} and Lemma \ref{<s_u>},
we see that 
\[
\hat\pi \hookrightarrow (\rho_\GL|\cdot|_E^{\alpha})^2 \times \dots \times (\rho_\GL|\cdot|_E^{1})^2 \rtimes \hat\pi', 
\]
where $\hat\pi' \in \Pi_{\phi'}$ with $\phi' = \phi_0 \oplus (\rho_\GL \boxtimes S_1)^{\oplus2}$,
and $\pair{\cdot, \hat\pi'}_{\phi'}$ is determined by 
\begin{align*}
\pair{\cdot, \hat\pi'}_{\phi'}|_{\AA_{\phi_0}} &= \pair{\cdot, \hat\pi_0}_{\phi_0}, \\
\pair{e(\rho_\GL,1,1),\hat\pi'}_{\phi'} &= 
\left\{\begin{aligned}
&\pair{e(\rho_\GL, 1,1), \hat\pi}_\phi \iif \rho_\GL \boxtimes S_1 \subset \phi_0, \\
&\pair{e(\rho_\GL, d_0,1), \hat\pi}_\phi \other. 
\end{aligned}
\right. 
\end{align*}
Hence $\pi'$ is almost supercuspidal, 
and if we denote by $\II(\mm') \rtimes \tau'$ its standard module, 
then $\mm = \mm' + 2([\alpha, \alpha]_{\rho_\GL}+\dots+[1, 1]_{\rho_\GL})$ and $\tau' = \tau$.
In particular, by the proof of (2), we have
\[
\pair{e(\rho_\GL,1,1), \tau}_{\phi_\tau} = \left\{
\begin{aligned}
&{-\pair{e(\rho_\GL,1,1),\hat\pi'}_{\phi'}} \iif m_{\phi'}(\rho_\GL) \equiv 0 \bmod 2, \\
&\pair{e(\rho_\GL,1,1),\hat\pi'}_{\phi'}\iif m_{\phi'}(\rho_\GL) \equiv 1 \bmod 2,
\end{aligned}
\right.
\]
where $m_{\phi'}(\rho_\GL)$ is the multiplicity of $\rho_\GL \boxtimes S_1$ in $\phi'$.
\par

If $\rho_\GL \boxtimes S_1 \not\subset \phi_0$, 
then $\rho_\GL \boxtimes S_d \not\subset \phi_0$ for any $d \geq 1$, 
and by Lemma \ref{6.19}, 
we have 
\[
\pair{e(\rho_\GL,1,1), \tau}_{\phi_\tau} 
= -\pair{e(\rho_\GL,1,1),\hat\pi'}_{\phi'}
= -\pair{e(\rho_\GL, d_0,1), \hat\pi}_\phi
= \pair{e(\rho_\GL, 1,d_0), \pi}_\psi.
\]
\par

From now on, we suppose that $\rho_\GL \boxtimes S_1 \subset \phi_0$. 
Assume first that $\rho_\GL \boxtimes S_{d_0} \subset \phi_0$ so that $2\alpha+1 \leq d_0$. 
In this case, by assertion (1), 
we have $\rho_\GL \boxtimes (S_1 \oplus S_3 \oplus \dots \oplus S_{d_0}) \subset \phi_0$, 
and $m_{\phi_0}(\rho_\GL \boxtimes S_d) = 1$ for $3 \leq d \leq d_0$ with $d$ odd. 
If we write $m = m_{\phi_0}(\rho_\GL)$, then $m_{\phi'}(\rho_\GL) = m+2$, 
and by Lemma \ref{6.19}, 
we have
\begin{align*}
\pair{e(\rho_\GL,1,1), \tau}_{\phi_\tau} 
&= (-1)^{m-1}\pair{e(\rho_\GL,1,1),\hat\pi'}_{\phi'}
\\&= (-1)^{m-1}\pair{e(\rho_\GL, 1,1), \hat\pi}_\phi
\\&= (-1)^{m-1} \cdot (-1)^{\half{d_0-1}+m+1} \pair{e(\rho_\GL, 1,1), \pi}_\psi
\\&= \pair{e(\rho_\GL, 1,d_0), \pi}_\psi.
\end{align*}
Next, assume that $\rho_\GL \boxtimes S_{d_0} \not\subset \phi_0$ so that $d_0 = 2\alpha+1$. 
If we set $d_0' = \max\{d \geq 1 \;|\; \rho_\GL \boxtimes S_{d} \subset \phi_0\}$, 
by the same argument as above, we have 
\begin{align*}
\pair{e(\rho_\GL,1,1), \tau}_{\phi_\tau} 
&= (-1)^{m-1}\pair{e(\rho_\GL,1,1),\hat\pi'}_{\phi'}
\\&= (-1)^{m-1}\pair{e(\rho_\GL, 1,1), \hat\pi}_\phi
\\&= (-1)^{m-1}\pair{e(\rho_\GL, 1,1), \hat\pi_0}_{\phi_0}
\\&= (-1)^{m-1} \cdot (-1)^{\half{d'_0-1}+m-1} \pair{e(\rho_\GL, 1,1), \pi_0}_{\psi_0}
\\&= \pair{e(\rho_\GL, 1,d'_0), \pi_0}_{\psi_0}.
\end{align*}
By Lemma \ref{<s_u>}, we have
\[
\pair{e(\rho_\GL, 1,d'_0), \pi_0}_{\psi_0} = \pair{s_u, \pi}_\psi = \pair{e(\rho_\GL, 1, d_0), \pi}_\psi.
\]
\end{enumerate}
This completes the proof of Lemma \ref{initial}.
\end{proof}

By this lemma, we can compute all terms of \eqref{star} when $\pi_0$ is almost supercuspidal. 
The details are given in Section \ref{cuspidal}.
\par

For the general case, suppose that $\pi_0$ is not almost supercuspidal. 
Then using Corollary \ref{Lparameters}, 
we will find a classical group $G'_0$ with $\dim(\St_{\widehat{G'_0}}) < \dim(\St_{\widehat{G_0}})$, 
an $A$-parameter $\psi'_0$ for $G'_0$, and $\pi'_0 \in \Pi_{\psi_0'}$ such that 
the difference
\[
\phi_{\pi_0} - \phi_{\pi_0'}
\]
is explicitly known (although we might not know $\phi_{\pi_0}$ nor $\phi_{\pi_0'}$ themselves). 
By the induction hypothesis, we can assume \eqref{LIR}, 
equivalently equation \eqref{star}, for $\pi'_{M'} = \pi_\GL \boxtimes \pi'_0$.
Therefore, what we have to check is the equation 
\[
\tag{$\star\star$}\label{starstar}
\left.
\left(\frac{\gamma_A(s, {}^c\psi_\GL \otimes \psi_0, \psi_E)}{\gamma_A(s, {}^c\psi_\GL \otimes \phi_{\pi_0}, \psi_E)}\right)
\left(\frac{\gamma_A(s, {}^c\psi_\GL \otimes \psi'_0, \psi_E)}{\gamma_A(s, {}^c\psi_\GL \otimes \phi_{\pi'_0}, \psi_E)}\right)^{-1}
\right|_{s=0}
= \frac{\pair{s_u,\pi}_\psi}{\pair{s_u,\pi'}_{\psi'}}, 
\]
where $\pi \subset I_P(\pi_M)$ and $\pi' \subset I_{P'}(\pi'_{M'})$ are highly non-tempered summands.
\par

To give $\pi_0'$, we consider Jacquet modules of the tempered representation $\hat\pi_0$.

\begin{lem}\label{abc}
Let $\psi_0 = \widehat\phi_0$ be a co-tempered $A$-parameter of good parity for a classical group $G_0$. 
Suppose that $\pi_0 \in \Pi_{\psi_0}$ is not almost supercuspidal. 
Then one can find an irreducible conjugate-self-dual representation $\rho_1$ of $W_E$
and positive half-integers $x \leq y$ with $x \equiv y \bmod \Z$ such that 
\begin{enumerate}
\item
$\rho_1 \boxtimes (S_{2x+1} \oplus S_{2x+3} \oplus \dots \oplus S_{2y+1}) \subset \phi_0$; 
\item
if we denote the multiplicity of $\rho \boxtimes S_d$ in $\phi_0$ by $m_{\phi_0}(\rho,d)$, 
then 
\[
m_{\phi_0}(\rho_1,2i+1)
= \left\{ 
\begin{aligned}
&1 \iif x < i \leq y, \\
&0 \iif i > y
\end{aligned}
\right. 
\]
for $i \in (1/2)\Z$ with $i \equiv x \bmod \Z$; 
\item
$\pair{e(\rho_1,1,2i+1), \pi_0}_{\psi_0} = -\pair{e(\rho_1,1,2i-1), \pi_0}_{\psi_0}$ 
for $x < i \leq y$ with $i \equiv x \bmod \Z$; 
\item
one of the following holds: 
\begin{enumerate}
\item
$\rho_1 \boxtimes S_{2x-1} \not\subset \phi_0$, 
or $\rho_1 \boxtimes S_{2x-1} \subset \phi_0$ and 
$\pair{e(\rho_1,1,2x+1), \pi_0}_{\psi_0} = \pair{e(\rho_1,1,2x-1), \pi_0}_{\psi_0}$; 
\item
$\rho_1 \boxtimes S_{2x-1} \subset \phi_0$,  
$\pair{e(\rho_1,1,2x+1), \pi_0}_{\psi_0} = -\pair{e(\rho_1,1,2x-1), \pi_0}_{\psi_0}$
and $m = m_{\phi_0}(\rho_1,2x+1) > 1$ is odd; 
\item
$\rho_1 \boxtimes S_{2x-1} \subset \phi_0$,  
$\pair{e(\rho_1,1,2x+1), \pi_0}_{\psi_0} = -\pair{e(\rho_1,1,2x-1), \pi_0}_{\psi_0}$
and $m = m_{\phi_0}(\rho_1,2x+1) > 1$ is even. 
\end{enumerate}
Here, when $x = 1/2$, 
we formally set $\rho_1 \boxtimes S_0 \subset \phi_0$ and $\pair{e(\rho_1,1,0), \pi_0}_{\psi_0} = 1$. 
\end{enumerate}
\end{lem}
\begin{proof}
Since $\pi_0$ is not almost supercuspidal, in the appropriate Grothendieck group,
$\Jac_{P}(\pi_0)$ contains a representation of the form $\rho_1|\cdot|_E^{-x} \otimes \pi_0'$
for some parabolic subgroup $P$, some irreducible bounded representation $\rho_1$ of $W_E$ and 
some real number $x > 0$. 
Since $\phi_0$ is of good parity, $\rho_1$ must be conjugate-self-dual and $x$ must be a half-integer.
Fix such a $\rho_1$, and take the maximal $x$ with this condition.
Then by Theorems \ref{Dk} and \ref{Dk-1}, we have $\rho_1 \boxtimes S_{2x+1} \subset \phi_0$. 
Moreover, if we set
\[
2y+1 = \max\{d \,|\, \rho_1 \boxtimes S_d \subset \phi_0\}, 
\]
then the same theorems together with Lemma \ref{6.19} imply the desired conditions (1)--(4).
\end{proof}

We will treat the cases (a), (b) and (c) in Sections \ref{(a)}, \ref{(b)} and \ref{(c)}, respectively. 
In the cases (a) and (b), the new $A$-parameter $\psi'_0$ is also co-tempered. 
However, in the case (c), $\psi'_0$ is no longer co-tempered, 
and it is more difficult to compute $\pair{s_u, \pi'}_{\psi'}$.
This is where Corollary \ref{+otimes-} will be used, 
and where we have to separate \eqref{A-LIR} and \eqref{LIR}.

\subsection{The initial case}\label{cuspidal}
Let $\psi_M = \widehat{\phi}_M$ be a co-tempered $A$-parameter for $M$, 
where $\phi_M = \phi_\GL \oplus \phi_0$. 
Fix $\pi_0 \in \Pi_{\psi_0}$. 
In this subsection, we assume that $\pi_0 \in \Pi_{\psi_0}$ is almost supercuspidal (see Definition \ref{a-cusp}). 
Write
\[
\phi_\GL = \rho_\GL \boxtimes S_{2\alpha+1}, 
\quad
\phi_0 = \bigoplus_{i = 1}^t \rho_i \boxtimes S_{2\beta_i+1},
\]
where $\rho_i$ is an irreducible conjugate-self-dual representation of $W_E$.
As in Lemma \ref{initial}, write $\{\rho_1', \dots, \rho_r'\}$ for the subset of $\{\rho_1, \dots, \rho_t\}$
consisting of $\rho_i$'s such that $\rho_i \boxtimes S_1$ appears in $\phi_0$ with even multiplicity.  
Set $2y_i+1 = \max\{d \geq 1\,|\, \rho'_i \boxtimes S_d \subset \phi_0\}$. 
\par

Let us check the equation \eqref{star} for $\pi$. 
By Lemma \ref{initial} (2), 
the left-hand side of \eqref{star} is 
\[
\left.\frac{\gamma_A(s, {}^c\widehat{\phi}_\GL \otimes \widehat{\phi}_0, \psi_E)}
{\gamma_A(s, {}^c\widehat{\phi}_\GL \otimes \phi_{0}, \psi_E)}
\cdot
\prod_{i=1}^r
\frac{\gamma_A(s, {}^c\widehat{\phi}_\GL \otimes (\rho'_i \boxtimes (S_1\oplus S_{2y_i+1})), \psi_E)}
{\gamma_A(s, {}^c\widehat{\phi}_\GL \otimes (\rho'_i(|\cdot|_E^{\half{y_i}} \oplus |\cdot|_E^{-\half{y_i}}) \boxtimes S_{y_i+1}), \psi_E)}
\right|_{s=0}.
\]
In the rest of this section, 
we write $\prod_{-\alpha \leq a \leq \alpha}$ for the product with respect to $a = -\alpha, -\alpha+1, \dots, \alpha$
(even if $\alpha \in (1/2)\Z \setminus \Z$).
First, we compute the quotient involving $\rho_i'$. 
To simplify the notation, we drop the subscript $i$.
Then by the formulas for local factors in Section \ref{sec.factors}, 
we have
\begin{align*}
&\frac{\gamma_A(s, {}^c\widehat{\phi}_\GL \otimes (\rho' \boxtimes (S_1\oplus S_{2y+1})), \psi_E)}
{\gamma_A(s, {}^c\widehat{\phi}_\GL \otimes (\rho'(|\cdot|_E^{\half{y}} \oplus |\cdot|_E^{-\half{y}}) \boxtimes S_{y+1}), \psi_E)}
\\&=
\prod_{-\alpha \leq a \leq \alpha}
\frac{\gamma_A(s, ({}^c\rho_\GL \otimes \rho')|\cdot|_E^a \boxtimes (S_1\oplus S_{2y+1}), \psi_E)}
{\gamma_A(s, ({}^c\rho_\GL \otimes \rho')(|\cdot|_E^{a+\half{y}} \oplus |\cdot|_E^{a-\half{y}}) \boxtimes S_{y+1}, \psi_E)}
=1.
\end{align*}
\par

On the other hand, using the notations and the formulas for local factors in Section \ref{sec.factors}, we have
\begin{align*}
&\frac{\gamma_A(s, {}^c\widehat{\phi}_\GL \otimes \widehat{\phi}_0, \psi_E)}
{\gamma_A(s, {}^c\widehat{\phi}_\GL \otimes \phi_{0}, \psi_E)}
\\&= \prod_{i=1}^t\prod_{-\alpha \leq a \leq \alpha} 
\frac{\gamma_A(s, {}^c\rho_\GL|\cdot|_E^a \otimes (\rho_i|\cdot|_E^{\beta_i} \oplus \dots \oplus \rho_i|\cdot|_E^{-\beta_i}), \psi_E)}
{\gamma_A(s, ({}^c\rho_\GL|\cdot|_E^a \otimes \rho_i) \boxtimes S_{2\beta_i+1}, \psi_E)}
\\&= 
\prod_{i=1}^t 
\prod_{-\alpha \leq a \leq \alpha}\prod_{\mu \in X({}^c\rho_\GL \otimes \rho_i)}
(-q_E^{-(\half{1}-s-\mu-a)})^{2\beta_i}
\prod_{-\beta_i \leq b \leq \beta_i-1} 
\frac{\zeta_E(s+\mu+a+b+1)}{\zeta_E(s+\mu+a+b)}
\\&= 
\prod_{i=1}^t
\prod_{-\alpha \leq a \leq \alpha}\prod_{\mu \in X({}^c\rho_\GL \otimes \rho_i)}
(-q_E^{-(\half{1}-s-\mu-a)})^{2\beta_i}
\frac{\zeta_E(s+\mu+a+\beta_i)}{\zeta_E(s+\mu+a-\beta_i)}
\\&= 
\prod_{i=1}^t
\prod_{-\alpha \leq a \leq \alpha}\prod_{\mu \in X({}^c\rho_\GL \otimes \rho_i)}
(-q_E^{-(\half{1}-s-\mu-a)})^{2\beta_i}
\frac{\zeta_E(s+\mu-a+\beta_i)}{\zeta_E(s+\mu+a-\beta_i)}.
\end{align*}
\par

\begin{lem}\label{zeta}
For $\mu \in \C/2\pi\I(\log q_E)^{-1}\Z$ and $a, \beta \in (1/2)\Z$, 
set 
\[
f_{\mu,a,\beta}(s) = (-q_E^{-(\half{1}-s-\mu-a)})^{2\beta}
\frac{\zeta_E(s+\mu-a+\beta)}{\zeta_E(s+\mu+a-\beta)}.
\]

\begin{enumerate}
\item
If $\re(\mu) = 0$ and $\mu \not= -\mu$, then 
\[
\left. f_{\mu,a,\beta}(s)f_{-\mu,a,\beta}(s) \right|_{s=0} = q_E^{2a(2\beta-1)}.
\]

\item
Let $\mu_0 \in \C/2\pi\I(\log q_E)^{-1}\Z$ be the unique nonzero element such that $\mu_0 = -\mu_0$.
Then 
\[
\left. f_{\mu_0,a,\beta}(s) \right|_{s=0} = q_E^{a(2\beta-1)}.
\]

\item
If $\mu = 0$, then 
\[
\left. f_{0,a,\beta}(s) \right|_{s=0} = 
\left\{
\begin{aligned}
&(-1)^{2\beta}q_E^{a(2\beta-1)} \iif a=\beta, \\
&(-1)^{2\beta+1}q_E^{a(2\beta-1)} \iif a \not= \beta.
\end{aligned}
\right. 
\]
\end{enumerate}
\end{lem}
\begin{proof}
If $\re(\mu) = 0$ and $\mu \not= -\mu$, 
then
\begin{align*}
\left. f_{\mu,a,\beta}(s)f_{-\mu,a,\beta}(s) \right|_{s=0} 
&= q_E^{-2(1-2a)\beta}\frac{1-q_E^{-\mu-a+\beta}}{1-q_E^{\mu+a-\beta}}
\frac{1-q_E^{\mu-a+\beta}}{1-q_E^{-\mu+a-\beta}} 
\\&= q_E^{-2(1-2a)\beta} \cdot (-q_E^{-\mu-a+\beta}) \cdot (-q_E^{\mu-a+\beta})
= q_E^{2a(2\beta-1)}.
\end{align*}
\par

Similarly, if $\mu = \mu_0$, then $q_E^{-{\mu_0}} = -1$ so that 
\[
\left. f_{\mu_0,a,\beta}(s) \right|_{s=0}
= q_E^{-(1-2a)\beta} \frac{1+q_E^{-a+\beta}}{1+q_E^{a-\beta}} 
= q_E^{a(2\beta-1)}. 
\]
\par

Finally, suppose that $\mu=0$. 
If $a \not= \beta$, then 
\[
\left. f_{0,a,\beta}(s) \right|_{s=0} 
= (-1)^{2\beta}q_E^{-(1-2a)\beta} \frac{1-q_E^{-a+\beta}}{1-q_E^{a-\beta}} 
= (-1)^{2\beta+1}q_E^{a(2\beta-1)}.
\]
On the other hand, if $a = \beta$, then 
\[
\left. f_{0,a,\beta}(s) \right|_{s=0} 
= \left.(-1)^{2\beta}q_E^{-(1-2s-2a)a}\frac{1-q_E^{-s}}{1-q_E^{-s}} \right|_{s=0} 
= (-1)^{2\beta}q_E^{a(2\beta-1)}.
\]
This completes the proof. 
\end{proof}

We have proven that 
\[
\left.
\frac{\gamma_A(s, {}^c\widehat{\phi}_\GL \otimes \widehat{\phi}_0, \psi_E)}
{\gamma_A(s, {}^c\widehat{\phi}_\GL \otimes \phi_{0}, \psi_E)}
\right|_{s=0}
=
\prod_{i=1}^t
\prod_{-\alpha \leq a \leq \alpha}\prod_{\mu \in X({}^c\rho_\GL \otimes \rho_i)}
\left. f_{\mu,a,\beta_i}(s) \right|_{s=0}.
\]
Note that for $\mu \in X({}^c\rho_\GL \otimes \rho_i)$, 
if $\mu \not= -\mu$, then $-\mu$ appears in $X({}^c\rho_\GL \otimes \rho_i)$
with the same multiplicity as $\mu$.
Since $\prod_{-\alpha \leq a \leq \alpha}q_E^{a(2\beta_i-1)} = 1$, 
by Lemma \ref{zeta},
only $\mu = 0$ can contribute. 
Note that $0 \in X({}^c\rho_\GL \otimes \rho_i)$ 
if and only if $\rho_i \cong {}^c\rho_\GL^\vee \cong \rho_\GL$. 
In this case, $0$ appears in $X({}^c\rho_\GL \otimes \rho_i)$ with multiplicity one. 
Hence
\[
\left.\frac{\gamma_A(s, {}^c\widehat{\phi}_\GL \otimes \widehat{\phi}_0, \psi_E)}
{\gamma_A(s, {}^c\widehat{\phi}_\GL \otimes \phi_{0}, \psi_E)}\right|_{s=0}
= \prod_{\substack{1 \leq i \leq t \\ \rho_i \cong \rho_\GL}}
\prod_{-\alpha \leq a \leq \alpha} 
\left. f_{0,a,\beta_i}(s) \right|_{s=0}, 
\]
and by Lemma \ref{zeta}, we have 
\[
\prod_{-\alpha \leq a \leq \alpha} 
\left. f_{0,a,\beta_i}(s) \right|_{s=0}
= \left\{
\begin{aligned}
&(-1)^{(2\alpha+1)(2\beta_i+1)-1} \iif \beta_i \leq \alpha, \, \beta_i \equiv \alpha \bmod \Z, \\
&(-1)^{(2\alpha+1)(2\beta_i+1)} \other.
\end{aligned}
\right. 
\]
\par

Suppose first that $\phi_\GL = \rho_\GL \boxtimes S_{2\alpha+1}$ is not of the same type as $\phi_0$. 
Then if $\rho_i \cong \rho_\GL$, then $\beta_i \not\equiv \alpha \bmod \Z$ so that $(2\alpha+1)(2\beta_i+1) \in 2\Z$.
Hence
\[
\left.\frac{\gamma_A(s, {}^c\widehat{\phi}_\GL \otimes \widehat{\phi}_0, \psi_E)}
{\gamma_A(s, {}^c\widehat{\phi}_\GL \otimes \phi_{0}, \psi_E)}\right|_{s=0}
= \prod_{\substack{1 \leq i \leq t \\ \rho_i \cong \rho_\GL}} (-1)^{(2\alpha+1)(2\beta_i+1)} = 1.
\]
Since $\pair{s_u,\pi}_\psi = 1$ in this case, we obtain the equation \eqref{star}. 
In the rest of this subsection, 
we assume that $\phi_\GL = \rho_\GL \boxtimes S_{2\alpha+1}$ is of the same type as $\phi_0$
so that $\beta_i \equiv \alpha \bmod \Z$ if $\rho_i \cong \rho_\GL$. 
\par

Suppose that $2\alpha+1$ is even. 
Then we conclude that 
\[
\left.\frac{\gamma_A(s, {}^c\widehat{\phi}_\GL \otimes \widehat{\phi}_0, \psi_E)}
{\gamma_A(s, {}^c\widehat{\phi}_\GL \otimes \phi_{0}, \psi_E)}\right|_{s=0}
= (-1)^{|\{i \,|\, \rho_i \cong \rho_\GL,\, \beta_i \leq \alpha\}|}.
\]
If $\rho_\GL \boxtimes S_{2\alpha+1} \subset \phi_0$, 
then 
\begin{itemize}
\item
$\rho_\GL \boxtimes (S_2 \oplus S_4 \oplus \dots \oplus S_{2\alpha+1}) \subset \phi_0$; and
\item
$m_{\phi_0}(\rho_\GL \boxtimes S_{2x}) = 1$ and $\pair{e(\rho_\GL,1,2x), \pi_0}_{\psi_0} = (-1)^{x}$ 
for $1 \leq x \leq \alpha+\half{1}$. 
\end{itemize}
In particular, 
\[
\pair{s_u, \pi}_\psi =
\pair{e(\rho_\GL, 1, 2\alpha+1), \pi_0}_{\psi_0} 
= (-1)^{\alpha+\half{1}} 
= (-1)^{|\{i \,|\, \rho_i \cong \rho_\GL,\, \beta_i \leq \alpha\}|}.
\]
On the other hand, if $\rho_\GL \boxtimes S_{2\alpha+1} \not\subset \phi_0$, 
then setting $d_- = \max\{d \geq 0 \,|\, \rho_\GL \boxtimes S_{d} \subset \phi_0\}$, 
we have
\[
\pair{s_u, \pi}_\psi =
\pair{e(\rho_\GL,1,d_-), \pi_0}_{\psi_0} = (-1)^{\half{d_-}} 
= (-1)^{|\{i \,|\, \rho_i \cong \rho_\GL,\, \beta_i \leq \alpha\}|}.
\]
Therefore, we obtain the equation \eqref{star} when $2\alpha+1$ is even.
\par

Next, we suppose that $2\alpha+1$ is odd.
If $\rho_\GL \boxtimes S_d \not\subset \phi_0$ for any $d \geq 1$, 
then noting that $\pair{s_u, \pi}_\psi = \pair{e(\rho_\GL, 1,1), \tau}_{\phi_\tau}$ by Lemma \ref{initial} (3), 
both sides of \eqref{star} are equal to $1$. 
Hence we assume that $\rho_\GL \boxtimes S_d \subset \phi_0$ for some $d \geq 1$.
Set
\[
d_0 = \max\{d \geq 1 \,|\, \rho_\GL \boxtimes S_{d} \subset \phi_0\}, 
\]
and write $d_0 = 2\beta_0+1$.
If $\alpha > \beta_0$, then 
since $\pair{s_u,\pi}_\psi = \pair{e(\rho_\GL,1,d_0),\pi_0}_{\psi_0} = \pair{e(\rho_\GL,1,1),\tau}_{\phi_\tau}$, 
we have
\[
\left.\frac{\gamma_A(s, {}^c\widehat{\phi}_\GL \otimes \widehat{\phi}_0, \psi_E)}
{\gamma_A(s, {}^c\widehat{\phi}_\GL \otimes \phi_{0}, \psi_E)}\right|_{s=0}
= 1 = \frac{\pair{s_u,\pi}_\psi}{\pair{e(\rho_\GL,1,1),\tau}_{\phi_\tau}}. 
\]
If $\alpha \leq \beta_0$, then 
\[
\left.\frac{\gamma_A(s, {}^c\widehat{\phi}_\GL \otimes \widehat{\phi}_0, \psi_E)}
{\gamma_A(s, {}^c\widehat{\phi}_\GL \otimes \phi_{0}, \psi_E)}\right|_{s=0}
= (-1)^{\beta_0-\alpha}. 
\]
On the other hand, we note that
\begin{itemize}
\item
$\rho_\GL \boxtimes (S_1 \oplus S_3 \oplus \dots \oplus S_{2\beta_0+1}) \subset \phi_0$; and 
\item
$\pair{e(\rho_\GL,1,2x+1), \pi_0}_{\psi_0} = -\pair{e(\rho_\GL,1,2x-1), \pi_0}_{\psi_0}$
for $1 \leq x \leq \beta_0$. 
\end{itemize}
Hence
\[
\frac{\pair{s_u,\pi}_\psi}{\pair{e(\rho_\GL,1,1), \tau}_{\phi_\tau}}
= \frac{\pair{e(\rho_\GL,1,2\alpha+1),\pi_0}_{\psi_0}}{\pair{e(\rho_\GL,1,2\beta_0+1), \pi_0}_{\psi_0}}
= (-1)^{\beta_0-\alpha}. 
\]
Therefore, we obtain the equation \eqref{star} when $2\alpha+1$ is odd.
This completes the proof of \eqref{star} when $\pi_0 \in \Pi_{\psi_0}$ is almost supercuspidal. 

\subsection{The inductive case (a)}\label{(a)}
Let $\phi_0$ be a tempered $L$-parameter for $G_0$ of good parity, 
and set $\psi_0 = \widehat\phi_0$. 
Fix $\pi_0 \in \Pi_{\psi_0}$. 
In this subsection, we assume the conditions (1)--(3) and (4a) in Lemma \ref{abc}.
Set $m = m_{\phi_0}(\rho_1,2x+1) > 0$. 
\par

In this case, by applying Lemma \ref{AvsD} and Theorem \ref{Dk} repeatedly, 
we see that
\[
\pi_0 \hookrightarrow (\rho_1|\cdot|_E^{-x})^m \times \rho_1|\cdot|_E^{-(x+1)} \times \dots \times \rho_1|\cdot|_E^{-y} \rtimes \pi'_0, 
\]
where $\pi_0' \in \Pi_{\psi_0'}$ is characterized such that 
\begin{itemize}
\item
$\psi_0' = \widehat\phi_0'$ is a co-tempered $A$-parameter with 
\[
\phi_0' = \phi_0- \rho_1 \boxtimes (S_{2x+1}^{\oplus m-1} \oplus S_{2y+1}) \oplus \rho_1 \boxtimes S_{2x-1}^{\oplus m}; 
\]
\item
the character $\pair{\cdot, \pi_0'}_{\psi_0'}$ is given by 
\[
\pair{e(\rho, 1, 2i-1), \pi_0'}_{\psi_0'} = \left\{
\begin{aligned}
&\pair{e(\rho_1,1,2i+1), \pi_0}_{\psi_0} \iif \rho \cong \rho_1,\, x \leq i \leq y, \\
&\pair{e(\rho, 1, 2i-1), \pi_0}_{\psi_0} \other. 
\end{aligned}
\right. 
\]
\end{itemize}
Moreover, 
since $\pi_0$ and $\pi'_0$ are in the situation in Section \ref{sec.Ldata}, 
by Corollary \ref{Lparameters},
the $L$-parameters $\phi_{\pi_0}$ and $\phi_{\pi'_0}$ are related by 
\[
\phi_{\pi_0} = \phi_{\pi'_0} \oplus \rho_1(|\cdot|_E^{\half{x+y}} \oplus |\cdot|_E^{-\half{x+y}}) \boxtimes S_{y-x+1}
\oplus (\rho_1(|\cdot|_E^x \oplus |\cdot|_E^{-x}) \boxtimes S_1)^{\oplus m-1}.
\]
\par

Set $\phi_\GL = \rho_\GL \boxtimes S_{2\alpha+1}$ and $\psi_\GL = \widehat\phi_\GL$.
We denote the irreducible representation of $\GL_k(E)$ corresponding to $\psi_\GL$ by $\pi_\GL$. 
Taking a suitable classical group $G'$ (with $\dim(\St_{\widehat{G'}}) < \dim(\St_{\widehat{G}})$)
and its maximal parabolic subgroup $P' = M'N_{P'}$, 
write $\psi'_{M'} = \psi_\GL \oplus \psi_0'$. 
By the induction hypothesis, we may assume that the equation \eqref{star} holds for 
a highly non-tempered representation $\pi' \subset I_{P'}(\pi'_{M'})$ 
with $\pi'_{M'} = \pi_\GL \boxtimes \pi_0'$. 
Take a highly non-tempered representation $\pi \subset I_{P}(\pi_{M})$
such that the tempered part $\tau$ of the Langlands data of $\pi$ coincides with the one for $\pi'$. 
\par

By Lemma \ref{<s_u>}, 
we see that 
\[
\frac{\pair{s_u,\pi}_\psi}{\pair{s_u,\pi'}_{\psi'}}  
= \left\{
\begin{aligned}
&{-1} \iif \rho_\GL \cong \rho_1,\, x \leq \alpha < y, \alpha \equiv x \bmod \Z, \\
&1 \other.
\end{aligned}
\right. 
\]
Note that if there is a choice of $\tau$, then $\rho_\GL \not\cong \rho_1$ by Lemma \ref{<s_u>} (3).
We will check that
the left-hand side of \eqref{starstar}
is equal to the right-hand side of this equation.
\par

Since $\gamma_A(s,\phi,\psi_E)$ is multiplicative, we have
\begin{align*}
&\left(\frac{\gamma_A(s, {}^c\psi_\GL \otimes \psi_0, \psi_E)}
{\gamma_A(s, {}^c\psi_\GL \otimes \phi_{\pi_0}, \psi_E)}\right)
\left(\frac{\gamma_A(s, {}^c\psi_\GL \otimes \psi'_0, \psi_E)}
{\gamma_A(s, {}^c\psi_\GL \otimes \phi_{\pi'_0}, \psi_E)}\right)^{-1}
\\&= 
\frac{\gamma_A(s, {}^c\psi_\GL \otimes (\rho_1 \boxtimes S_1 \boxtimes S_{2x+1}), \psi_E)^{m-1}
\gamma_A(s, {}^c\psi_\GL \otimes (\rho_1 \boxtimes S_1 \boxtimes S_{2y+1}),\psi_E)}
{\gamma_A(s, {}^c\psi_\GL \otimes (\rho_1 \boxtimes S_1 \boxtimes S_{2x-1}), \psi_E)^{m}}
\\&\quad\times 
\prod_{\epsilon \in \{\pm1\}} \gamma_A(s, {}^c\psi_\GL \otimes (\rho_1|\cdot|_E^{\epsilon \half{x+y}} \boxtimes S_{y-x+1}), \psi_E)^{-1}
\gamma_A(s, {}^c\psi_\GL \otimes \rho_1|\cdot|_E^{\epsilon x}, \psi_E)^{-(m-1)}
\\&= 
\prod_{-\alpha \leq a \leq \alpha} 
\prod_{-x \leq b \leq x} 
\gamma_A(s, {}^c\rho_\GL|\cdot|_E^a \otimes \rho_1|\cdot|_E^b, \psi_E)^{m-1}
\\&\quad \times 
\prod_{-\alpha \leq a \leq \alpha} 
\prod_{-y \leq b \leq y} 
\gamma_A(s, {}^c\rho_\GL|\cdot|_E^a \otimes \rho_1|\cdot|_E^b, \psi_E)
\\&\quad \times
\prod_{-\alpha \leq a \leq \alpha} 
\prod_{-x+1 \leq b \leq x-1} 
\gamma_A(s, {}^c\rho_\GL|\cdot|_E^a \otimes \rho_1|\cdot|_E^b, \psi_E)^{-m}
\\&\quad\times 
\prod_{-\alpha \leq a \leq \alpha}\prod_{\epsilon \in \{\pm1\}} 
\gamma_A(s, ({}^c\rho_\GL|\cdot|_E^a \otimes \rho_1|\cdot|_E^{\epsilon \half{x+y}}) \boxtimes S_{y-x+1}, \psi_E)^{-1}
\\&\quad\times 
\prod_{-\alpha \leq a \leq \alpha}\prod_{\epsilon \in \{\pm1\}} 
\gamma_A(s, {}^c\rho_\GL|\cdot|_E^a \otimes \rho_1|\cdot|_E^{\epsilon x}, \psi_E)^{-(m-1)}
\\&= 
\prod_{-\alpha \leq a \leq \alpha}\prod_{\substack{-y \leq b \leq -x\\\text{or}\\ x \leq b \leq y}}
\gamma_A(s, {}^c\rho_\GL|\cdot|_E^a \otimes \rho_1|\cdot|_E^b, \psi_E)
\\&\quad\times 
\prod_{-\alpha \leq a \leq \alpha}\prod_{\epsilon \in \{\pm1\}} 
\gamma_A(s, ({}^c\rho_\GL|\cdot|_E^a \otimes \rho_1|\cdot|_E^{\epsilon \half{x+y}}) \boxtimes S_{y-x+1}, \psi_E)^{-1}
\\&=
\prod_{\mu \in X({}^c\rho_\GL \otimes \rho_1)} \prod_{-\alpha \leq a \leq \alpha}
q_E^{-(1-2s-2\mu-2a)(y-x)} 
\prod_{\substack{-y \leq b \leq -x-1\\\text{or}\\ x \leq b \leq y-1}} \frac{\zeta_E(s+\mu+a+b+1)}{\zeta_E(s+\mu+a+b)}
\\&=
\prod_{\mu \in X({}^c\rho_\GL \otimes \rho_1)} \prod_{-\alpha \leq a \leq \alpha}
q_E^{-(1-2s-2\mu-2a)(y-x)} 
\frac{\zeta_E(s+\mu+a-x)}{\zeta_E(s+\mu+a-y)}\frac{\zeta_E(s+\mu+a+y)}{\zeta_E(s+\mu+a+x)}
\\&=
\prod_{\mu \in X({}^c\rho_\GL \otimes \rho_1)} \prod_{-\alpha \leq a \leq \alpha}
q_E^{-(1-2s-2\mu-2a)(y-x)} 
\frac{\zeta_E(s+\mu+a-x)}{\zeta_E(s+\mu+a-y)}\frac{\zeta_E(s+\mu-a+y)}{\zeta_E(s+\mu-a+x)}
\\&= \prod_{\mu \in X({}^c\rho_\GL \otimes \rho_1)} \prod_{-\alpha \leq a \leq \alpha}
\frac{f_{\mu, a, y}(s)}{f_{\mu, a, x}(s)}.
\end{align*}
As in the previous subsection, 
Lemma \ref{zeta} implies that 
only $\mu = 0$ can contribute to this product after evaluating at $s = 0$.
Hence it is $1$ unless $\rho_\GL \cong \rho_1$. 
Moreover, since 
\[\left.
\frac{f_{0, a, y}(s)}{f_{0, a, x}(s)}
\right|_{s=0}
= \left\{
\begin{aligned}
&{-q_E^{2a(y-x)}} \iif \text{$a=x \not= y$ or $a=y \not= x$}, \\
&q_E^{2a(y-x)} \other
\end{aligned}
\right. 
\]
by Lemma \ref{zeta}, 
we conclude that 
\begin{align*}
&\left.
\left(\frac{\gamma_A(s, {}^c\psi_\GL \otimes \psi_0, \psi_E)}
{\gamma_A(s, {}^c\psi_\GL \otimes \phi_{\pi_0}, \psi_E)}\right)
\left(\frac{\gamma_A(s, {}^c\psi_\GL \otimes \psi'_0, \psi_E)}
{\gamma_A(s, {}^c\psi_\GL \otimes \phi_{\pi'_0}, \psi_E)}\right)^{-1}
\right|_{s=0}
\\&= \left\{
\begin{aligned}
&{-1} \iif \rho_\GL \cong \rho_1,\, x \leq \alpha < y,\, \alpha \equiv x \bmod \Z, \\
&1 \other,
\end{aligned}
\right. 
\end{align*}
as desired.

\subsection{The inductive case (b)}\label{(b)}
We use the same notation as in the previous subsection. 
In this subsection, we assume (1)--(3) and (4b) in Lemma \ref{abc}.
Set $m = m_{\phi_0}(\rho_1,2x+1) > 1$. 
\par

In this case, by Lemma \ref{AvsD} and Theorems \ref{Dk}, \ref{Dk-1} (1), 
we see that
\[
\pi_0 \hookrightarrow 
(\rho_1|\cdot|_E^{-x})^{m-1} \rtimes \pi_0',
\]
where $\pi_0' \in \Pi_{\psi_0'}$ is characterized such that 
\begin{itemize}
\item
$\psi_0' = \widehat\phi_0'$ is a co-tempered $A$-parameter with 
\[
\phi_0' = \phi_0- \rho_1 \boxtimes S_{2x+1}^{\oplus m-1} \oplus \rho_1 \boxtimes S_{2x-1}^{\oplus m-1}; 
\]
\item
$\pair{\cdot, \pi_0'}_{\psi_0'} = \pair{\cdot, \pi_0}_{\psi_0}$ via the canonical identification $\AA_{\psi_0'} \cong \AA_{\psi_0}$. 
\end{itemize}
Moreover, 
since $\pi_0$ and $\pi'_0$ are in the situation in Section \ref{sec.Ldata}, 
by Corollary \ref{Lparameters},
the $L$-parameters $\phi_{\pi_0}$ and $\phi_{\pi'_0}$ are related by 
\[
\phi_{\pi_0} = \phi_{\pi'_0} \oplus (\rho_1(|\cdot|_E^x \oplus |\cdot|_E^{-x}) \boxtimes S_1)^{\oplus m-1}.
\]
\par

We take highly non-tempered summands $\pi \subset I_P(\pi_M)$ and $\pi' \subset I_{P'}(\pi'_{M'})$
as in the previous subsection.
By Lemma \ref{<s_u>}, we have $\pair{s_u,\pi}_\psi = \pair{s_u,\pi'}_{\psi'}$. 
On the other hand, 
by the multiplicativity of $\gamma_A$-factors, we have
\begin{align*}
&\left(\frac{\gamma_A(s, {}^c\psi_\GL \otimes \psi_0, \psi_E)}
{\gamma_A(s, {}^c\psi_\GL \otimes \phi_{\pi_0}, \psi_E)}\right)
\left(\frac{\gamma_A(s, {}^c\psi_\GL \otimes \psi'_0, \psi_E)}
{\gamma_A(s, {}^c\psi_\GL \otimes \phi_{\pi'_0}, \psi_E)}\right)^{-1}
\\&= 
\left(\frac{\gamma_A(s, {}^c\psi_\GL \otimes (\rho_1 \boxtimes S_1 \boxtimes S_{2x+1}), \psi_E)}
{\gamma_A(s, {}^c\psi_\GL \otimes (\rho_1 \boxtimes S_1 \boxtimes S_{2x-1}), \psi_E)}
\prod_{\epsilon \in \{\pm1\}} \gamma_A(s, {}^c\psi_\GL \otimes \rho_1|\cdot|_E^{\epsilon x}, \psi_E)^{-1}
\right)^{m-1}
\\&= 1.
\end{align*}
Therefore, we conclude that 
\[
\left.
\left(\frac{\gamma_A(s, {}^c\psi_\GL \otimes \psi_0, \psi_E)}
{\gamma_A(s, {}^c\psi_\GL \otimes \phi_{\pi_0}, \psi_E)}\right)
\left(\frac{\gamma_A(s, {}^c\psi_\GL \otimes \psi'_0, \psi_E)}
{\gamma_A(s, {}^c\psi_\GL \otimes \phi_{\pi'_0}, \psi_E)}\right)^{-1}
\right|_{s=0}
= \frac{\pair{s_u,\pi}_\psi}{\pair{s_u,\pi'}_{\psi'}}.
\]

\subsection{The inductive case (c)}\label{(c)}
We continue to consider the inductive case. 
Let $\phi_0$ be a tempered $L$-parameter for $G_0$ of good parity, and set $\psi_0 = \widehat\phi_0$. 
Fix $\pi_0 \in \Pi_{\psi_0}$. 
In this subsection, we assume the conditions (1)--(3) and (4c) in Lemma \ref{abc}.
Set $m = m_{\phi_0}(\rho_1,2x+1) > 1$. 
\par

In this case, by Theorems \ref{Dk} and \ref{Dk-1} (2), 
we see that 
\begin{align*}
\hat\pi_0 
&\hookrightarrow 
(\rho_1|\cdot|_E^x)^{m-1} \times \Delta([-x,x-1]_{\rho_1}) 
\times \rho_1|\cdot|_E^{x+1} \times \dots \times \rho_1|\cdot|_E^y \rtimes \sigma_0'
\\&\cong 
(\rho_1|\cdot|_E^x)^{m-1} \times \rho_1|\cdot|_E^{x+1} \times \dots \times \rho_1|\cdot|_E^y 
\times \Delta([-x,x-1]_{\rho_1}) \rtimes \sigma_0',
\end{align*}
where $\sigma'_0$ is tempered, and its $L$-parameter is given by 
\[
\phi_{\sigma'_0} 
= \phi_0- \rho_1 \boxtimes (S_{2x+1}^{\oplus m-1} \oplus S_{2y+1} ) \oplus \rho_1 \boxtimes S_{2x-1}^{\oplus m-2}
\]
and 
\[
\pair{e(\rho, 2i-1,1), \sigma_0'}_{\phi_{\sigma_0'}} = 
\left\{
\begin{aligned}
&\pair{e(\rho_1, 2i+1,1), \hat\pi_0}_{\phi_0} \iif \rho \cong \rho_1,\, x < i \leq y, \\
&\pair{e(\rho, 2i-1,1), \hat\pi_0}_{\phi_0} \other.
\end{aligned}
\right. 
\]
This inclusion factors through 
\[
\hat\pi_0 \hookrightarrow 
(\rho_1|\cdot|_E^x)^{m-1} \times \rho_1|\cdot|_E^{x+1} \times \dots \times \rho_1|\cdot|_E^y \rtimes \hat\pi_0'
\]
where $\hat\pi_0'$ is the unique irreducible subrepresentation of $\Delta([-x,x-1]_{\rho_1}) \rtimes \sigma_0'$, 
i.e., the Langlands quotient of $\Delta([-(x-1),x]_{\rho_1}) \rtimes \sigma_0'$.
In particular, $\hat\pi'_0$ belongs to the $L$-packet associated to the $A$-parameter
\[
\widehat\psi_0' = \phi_{\sigma'_0} \oplus \rho_1 \boxtimes S_{2x} \boxtimes S_{2}, 
\]
and the character $\pair{\cdot, \hat\pi'_0}_{\widehat\psi_0'}$ is given by
\begin{align*}
\pair{e(\rho,2i-1,1), \hat\pi'_0}_{\widehat\psi_0'} &= 
\left\{
\begin{aligned}
&\pair{e(\rho_1, 2i+1,1), \hat\pi_0}_{\widehat\psi_0} \iif \rho \cong \rho_1,\, x < i \leq y, \\
&\pair{e(\rho, 2i-1,1), \hat\pi_0}_{\widehat\psi_0} \other,
\end{aligned}
\right. 
\\
\pair{e(\rho_1,2x,2), \hat\pi'_0}_{\widehat\psi_0'} &= 1.
\end{align*}
\par

Taking the Aubert dual, we obtain 
\[
\pi_0 \hookrightarrow 
(\rho_1|\cdot|_E^{-x})^{m-1} \times \rho_1|\cdot|_E^{-(x+1)} \times \dots \times \rho_1|\cdot|_E^{-y} \rtimes \pi_0', 
\]
where $\pi_0' \in \Pi_{\psi'_0}$ with 
\[
\psi_0' = \psi_0-\rho_1 \boxtimes S_1 \boxtimes (S_{2x+1}^{\oplus m-1} \oplus S_{2y+1}) 
\oplus \rho_1 \boxtimes S_1 \boxtimes S_{2x-1}^{\oplus m-2}
\oplus \rho_1 \boxtimes S_2 \boxtimes S_{2x}.
\]
Since $\pi_0$ and $\pi'_0$ are in the situation in Section \ref{sec.Ldata}, 
by Corollary \ref{Lparameters},
the $L$-parameters $\phi_{\pi_0}$ and $\phi_{\pi'_0}$ are related by 
\[
\phi_{\pi_0} = \phi_{\pi'_0} \oplus \rho_1(|\cdot|_E^{\half{x+y}} \oplus |\cdot|_E^{-\half{x+y}}) \boxtimes S_{y-x+1}
\oplus (\rho_1(|\cdot|_E^x \oplus |\cdot|_E^{-x}) \boxtimes S_1)^{\oplus m-2}.
\]
\par

We take $\pi \subset I_P(\pi_M)$ and $\pi' \subset I_{P'}(\pi'_{M'})$ as in the previous subsections. 
By Hypothesis \ref{arthur}, we know the local intertwining relation for $\psi_{M'} = \psi_\GL \oplus \psi_0'$. 
However, it is \eqref{A-LIR}.
To deduce our \eqref{LIR} for $\pi' \subset I_{P'}(\pi'_{M'})$, 
we need to check the conditions in Lemma \ref{A_LIR}. 
We can check them by using the next lemma together with Aubert duality.
In this lemma, for simplicity, we refresh the notations for parameters.

\begin{lem}\label{mult-free}
Assume Hypothesis \ref{arthur}. 
Let $\psi_M = \phi_\GL \oplus \psi_0$ be an $A$-parameter for $M = \GL_d(E) \times G_0$ such that 
\begin{itemize}
\item
$\phi_\GL = \rho_\GL \boxtimes S_{2\alpha+1}$ is a discrete $L$-parameter for $\GL_d(E)$; 
\item
$\psi_0$ is an $A$-parameter for $G_0$ of the form 
\[
\psi_0 = \phi_0 \oplus \rho_1 \boxtimes S_{2x} \boxtimes S_2
\]
where $\phi_0$ is tempered and $\rho_1$ is irreducible (and conjugate-self-dual).
\end{itemize}
Then 
\begin{itemize}
\item
$I_P(\pi_M)$ is multiplicity-free for any $\pi_M \in \Pi_{\psi_M}$; and
\item
for $\pi_M, \pi'_M \in \Pi_{\psi_M}$, if $\pi_M \not\cong \pi'_M$, 
then $I_P(\pi_M)$ and $I_P(\pi'_M)$ have no common irreducible summand.
\end{itemize}
\end{lem}
\begin{proof}
This is a special case of M{\oe}glin's multiplicity one theorem (see \cite[Theorem 8.12]{X2}). 
However, her proof relies on \eqref{ECR1} and \eqref{ECR2} for all classical groups,
and hence we cannot use this result. 
We shall give another proof.
\par

Note that we can use Arthur's theory for $\psi_0$ by Hypothesis \ref{arthur}. 
We apply the formal construction of the $A$-packets in \cite[p.~416--417]{Ar} to $\psi_0$ and $s_{\psi_0}$. 
See also the proof of \cite[Lemma 2.2.2]{Ar}. 
Since $\psi_0 = \phi_0 \oplus \psi_1$ with $\psi_1 = \rho_1 \boxtimes S_{2x} \boxtimes S_2$, 
we have an exact sequence 
\[
\begin{CD}
0 @>>> \pi_{\phi_0} \boxtimes \pi_{\psi_1^D} @>>> \pi_{\phi_0} \boxtimes \II_{\psi_1} 
@>>> \pi_{\phi_0} \boxtimes \pi_{\psi_1} @>>> 0
\end{CD}.
\]
If we denote the standard module for $\pi_0 \in \Irr(G)$ by $\II_{\pi_0}$, 
then the twisted and ordinary endoscopic transfers of this exact sequence show that 
\[
\sum_{\pi_0 \in \Pi_{\phi_{\psi_0}}} \II_{\pi_0} 
= \sum_{\pi_0 \in \Pi_{\psi_0}} \pi_0 
+ \sum_{\pi_0' \in \Pi_{\psi_0^D}} \pair{s', \pi_0'}_{\psi_0^D} \pi_0'
\]
in the Grothendieck group $\RR(G_0)$,
where $s' = e(\rho_1, 2x-1,1) + e(\rho_1, 2x+1,1)$
(or $s' = e(\rho_1,2,1)$ if $x=1/2$).
Hence every $\pi_0 \in \Pi_{\psi_0}$ is 
\begin{enumerate}
\item[(1)]
tempered (with $L$-parameter $\psi_0^D$); or 
\item[(2)]
an irreducible subquotient of the standard module 
$\II_{\pi_0} = \Delta([-(x-1),x]_{\rho_1}) \rtimes \sigma_0$
for $\sigma_0 \in \Pi_{\phi_0}$. 
\end{enumerate}
\par

If $\sigma_0 \in \Pi_{\phi_0}$, then 
by Casselman's criterion (Theorem \ref{Ccriterion}), 
we see that all irreducible constituents of $\Delta([-(x-1),x]_{\rho_1}) \rtimes \sigma_0$ 
other than the Langlands quotient are tempered.
(See the proof of \cite[Proposition 5.2]{At}.)
Hence every $\pi_0 \in \Pi_{\psi_0}$ is 
\begin{enumerate}
\item[(1)]
tempered; or 
\item[(2')]
the Langlands quotient of $\Delta([-(x-1),x]_{\rho_1}) \rtimes \sigma_0$
for $\sigma_0 \in \Pi_{\phi_0}$. 
\end{enumerate}
\par

We denote by $\pi_\GL$ 
the irreducible discrete series representation of $\GL_d(E)$ corresponding to $\phi_\GL$, 
and set $\pi_M = \pi_\GL \boxtimes \pi_0$.
If $\pi_0$ is tempered, then $I_P(\pi_M)$ is a multiplicity-free sum of irreducible tempered representations
by \cite[Theorem 1.5.1]{Ar} and \cite[Theorem 2.5.1]{Mok}.
Suppose that $\pi_0$ is in the case (2')
so that $\Delta([-(x-1),x]_{\rho_1}) \rtimes \sigma_0 \twoheadrightarrow \pi_0$
for some $\sigma_0 \in \Pi_{\phi_0}$. 
Since $\pi_\GL \times \Delta([-(x-1),x]_{\rho_1}) \cong \Delta([-(x-1),x]_{\rho_1}) \times \pi_\GL$ 
by \cite[Theorem 9.7]{Z}, we have 
\[
\Delta([-(x-1),x]_{\rho_1}) \times \pi_\GL \rtimes \sigma_0
\twoheadrightarrow I_P(\pi_M). 
\]
Since $\pi_\GL \rtimes \sigma_0$ is a multiplicity-free sum of irreducible tempered representations, 
$\Delta([-(x-1),x]_{\rho_1}) \times \pi_\GL \rtimes \sigma_0$
is a sum of standard modules of the form $\Delta([-(x-1),x]_{\rho_1}) \rtimes \pi_i$
for tempered representations $\pi_i$ which are not isomorphic to each other.
Therefore, $I_{P}(\pi_M)$ is a semisimple quotient of 
a sum of distinct standard modules, and hence $I_{P}(\pi_M)$ is multiplicity-free.
Moreover, all irreducible summands of $I_{P}(\pi_M)$ are non-tempered.
\par

Next, let $\pi_M = \pi_\GL \boxtimes \pi_0$ and $\pi_M' = \pi_\GL \boxtimes \pi_0'$ be in $\Pi_{\psi_M}$.
Suppose that $\pi_M \not\cong \pi_M'$.
Note that all irreducible summands of $I_{P}(\pi_M)$ are tempered (\resp non-tempered) 
if $\pi_0$ is tempered (\resp non-tempered).
Hence if $I_P(\pi_M)$ and $I_P(\pi_M')$ have a common irreducible summand, 
then both of $\pi_0$ and $\pi_0'$ are tempered, or both of them are non-tempered. 
\par

In the former case, we denote the tempered $L$-parameters of $\pi_0$ and $\pi_0'$ by $\phi_{\pi_0}$ and $\phi_{\pi_0'}$, respectively. 
If $I_P(\pi_M)$ and $I_P(\pi_M')$ have a common irreducible summand, 
then denoting by $\phi$ the $L$-parameter of this summand, 
we have 
\[
\phi = \phi_\GL \oplus \phi_{\pi_0} \oplus {}^c\phi_\GL^\vee
= \phi_\GL \oplus \phi_{\pi_0'} \oplus {}^c\phi_\GL^\vee
\]
and hence $\phi_{\pi_0} = \phi_{\pi_0'}$. 
Moreover, since $\pi_M \not\cong \pi_{M'}$, 
by \cite[Proposition 2.4.3]{Ar} and \cite[Proposition 3.4.4]{Mok},
we have
\[
I_P(\pi_M) \oplus I_P(\pi_M') \subset \bigoplus_{\pi \in \Pi_{\phi}} \pi. 
\]
By \cite[Theorem 1.5.1]{Ar} and \cite[Theorem 2.5.1]{Mok},
the right-hand side, and hence the left-hand side are multiplicity-free. 
This is a contradiction. 
In other words, $I_P(\pi_M)$ and $I_P(\pi_M')$ have no common irreducible summand.
\par

Now we assume that both of $\pi_0$ and $\pi_0'$ are in the case (2'). 
We denote by $\sigma_0, \sigma_0' \in \Pi_{\phi_0}$ the tempered representations 
corresponding to $\pi_0$, $\pi_0'$ as in (2'), respectively. 
Since $\pi_0 \not\cong \pi_0'$, we have $\sigma_0 \not\cong \sigma_0'$.
Moreover, 
since tempered $L$-packets are multiplicity-free (\cite[Theorem 1.5.1]{Ar}, \cite[Theorem 2.5.1]{Mok}), 
we see that $\pi_\GL \rtimes \sigma_0$ and $\pi_\GL \rtimes \sigma_0'$ 
have no common irreducible summand.
Hence $I_P(\pi_M)$ and $I_P(\pi_M')$ are semisimple quotients of 
sums of standard modules which have no common standard module. 
Therefore, $I_P(\pi_M)$ and $I_P(\pi_M')$ have no common irreducible summand.
\end{proof}

Let us go back to the situation at the beginning of this subsection. 
We have defined 
\begin{align*}
\phi_{\sigma_0'} 
&= \phi_0- \rho_1 \boxtimes (S_{2x+1}^{\oplus m-1} \oplus S_{2y+1} ) \oplus \rho_1 \boxtimes S_{2x-1}^{\oplus m-2}, \\
\widehat\psi_0' &= \phi_{\sigma'_0} \oplus \rho_1 \boxtimes S_{2x} \boxtimes S_{2}.
\end{align*}
Set 
\[
\psi' = \psi_\GL \oplus \psi_0' \oplus {}^c\psi_\GL^\vee \in \Psi(G'), 
\]
where $G'$ is a classical group such that $\dim(\St_{\widehat{G'}}) < \dim(\St_{\widehat{G}})$.
Using Corollary \ref{+otimes-}, 
we compare $\pair{s_u,\pi}_\psi$ with $\pair{s_u', \pi'}_{\psi'}$.

\begin{lem}\label{8.10}
Write $\phi_\GL = \rho_\GL \boxtimes S_{2\alpha+1}$ and $\psi_\GL = \widehat\phi_\GL$.
If $\psi_\GL$ is of the same type as $\psi_0$, then 
we have
\begin{align*}
\frac{\pair{s_u,\pi}_\psi}{\pair{s_u,\pi'}_{\psi'}}
&= \left.\frac
{\gamma_A(s, {}^c\psi_\GL \otimes (\rho_1 \boxtimes S_{2} \boxtimes S_{2x}), \psi_E)}
{\gamma_A(s, {}^c\psi_\GL \otimes (\rho_1 \boxtimes S_1 \boxtimes (S_{2x-1} \oplus S_{2x+1})), \psi_E)}
\right|_{s=0}
\\&\quad\times 
 \left\{
\begin{aligned}
&{-1} \iif \rho_\GL \cong \rho_1,\, x \leq \alpha < y, \alpha \equiv x \bmod \Z, \\
&1 \other.
\end{aligned}
\right. 
\end{align*}
Otherwise, $\pair{s_u,\pi}_\psi = \pair{s_u,\pi'}_{\psi'} = 1$.
\end{lem}
\begin{proof}
Since $\phi = \widehat\psi$ is tempered, 
we can apply Corollary \ref{+-temp} and get 
\[
\frac{\pair{s_u,\pi}_\psi}{\pair{\widehat{s_u}, \hat\pi}_{\phi}} 
= 
(-1)^{r(\phi)-r(\phi_+)-r(\phi_-)}. 
\]
Note that $\widehat\psi'$ satisfies the assumption of Corollary \ref{+otimes-}. 
Since $\dim(\widehat\psi') < \dim(\phi)$, 
we can use \eqref{ECR1} and \eqref{ECR2} for $\widehat\psi'$, 
and hence we can apply Corollary \ref{+otimes-} to $\widehat\psi'$. 
Since $\psi_\GL = \widehat\phi_\GL = \phi_\GL^A$ 
and 
\[
\widehat\psi'_+ = \phi_\GL \oplus \phi_{\sigma_0'} \oplus \rho_1 \boxtimes S_{2x} \boxtimes S_2, 
\quad \widehat\psi'_- = \phi_\GL,
\]
we have
\begin{align*}
\frac{\pair{s_u, \pi'}_{\psi'}}{\pair{\widehat{s_u},\hat\pi'}_{\widehat\psi'}}
&= 
(-1)^{r(\widehat\psi')-r(\widehat\psi'_+)-r(\widehat\psi'_-)}
\left.\frac{\gamma_A(s, {}^c\psi_\GL \otimes (\rho_1 \boxtimes S_1 \boxtimes (S_{2x-1} \oplus S_{2x+1})), \psi_E)}
{\gamma_A(s, {}^c\psi_\GL \otimes (\rho_1 \boxtimes S_{2} \boxtimes S_{2x}), \psi_E)}\right|_{s=0}.
\end{align*}
By a similar argument to the proof of Corollary \ref{+-temp}, 
we have
\[
(-1)^{r(\phi)-r(\phi_+)-r(\phi_-)} = (-1)^{r(\widehat\psi')-r(\widehat\psi'_+)-r(\widehat\psi'_-)}. 
\]
Therefore, the assertion is reduced to showing
\[
\frac{\pair{\widehat{s_u},\hat\pi}_{\phi}}{\pair{\widehat{s_u},\hat\pi'}_{\widehat\psi'}}  
= \left\{
\begin{aligned}
&{-1} \iif \rho_\GL \cong \rho_1,\, x \leq \alpha < y, \alpha \equiv x \bmod \Z, \\
&1 \other.
\end{aligned}
\right. 
\]
Recall that $\Delta([-(x-1),x]_{\rho_1}) \times \hat\pi_\GL \rtimes \sigma_0' \twoheadrightarrow I_{P'}(\hat\pi'_{M'})$, 
where $\Delta([-(x-1),x]_{\rho_1}) \rtimes \sigma_0' \twoheadrightarrow \hat\pi'_{0}$. 
In particular, if $\rho_\GL \cong \rho_1$, $x \leq \alpha < y$ and $\alpha \equiv x \bmod \Z$, 
then $\phi_\GL \subset \phi_{\sigma_{0}'}$ so that 
\begin{align*}
\pair{\widehat{s_u},\hat\pi'}_{\widehat\psi'} 
&= \pair{e(\rho_\GL, 2\alpha+1,1), \sigma'_0}_{\phi_{\sigma'_0}}
\\&= \pair{e(\rho_\GL, 2\alpha+3,1), \hat\pi_0}_{\phi_0}
\\&= -\pair{e(\rho_\GL, 2\alpha+1,1), \hat\pi_0}_{\phi_0}
= -\pair{\widehat{s_u},\hat\pi}_{\phi}. 
\end{align*}
Otherwise, the same proof of Lemma \ref{<s_u>} works, 
and we obtain that $\pair{\widehat{s_u},\hat\pi'}_{\widehat\psi'} = \pair{\widehat{s_u},\hat\pi}_{\phi}$.
This completes the proof.
\end{proof}

Now, what we have to show is that 
\begin{align*}
\frac{\pair{s_u, \pi}_\psi}{\pair{s_u,\pi'}_{\psi'}}
\left.
\left(\frac{\gamma_A(s, {}^c\psi_\GL \otimes \psi_0, \psi_E)}
{\gamma_A(s, {}^c\psi_\GL \otimes \phi_{\pi_0}, \psi_E)}\right)
\left(\frac{\gamma_A(s, {}^c\psi_\GL \otimes \psi'_0, \psi_E)}
{\gamma_A(s, {}^c\psi_\GL \otimes \phi_{\pi'_0}, \psi_E)}\right)^{-1}
\right|_{s=0}
= 1.
\end{align*}
Since 
\begin{align*}
&\left(\frac{\gamma_A(s, {}^c\psi_\GL \otimes \psi_0, \psi_E)}
{\gamma_A(s, {}^c\psi_\GL \otimes \phi_{\pi_0}, \psi_E)}\right)
\left(\frac{\gamma_A(s, {}^c\psi_\GL \otimes \psi'_0, \psi_E)}
{\gamma_A(s, {}^c\psi_\GL \otimes \phi_{\pi'_0}, \psi_E)}\right)^{-1}
\\&= 
\frac{\gamma_A(s, {}^c\psi_\GL \otimes (\rho_1 \boxtimes S_1 \boxtimes S_{2x+1}), \psi_E)^{m-1}
\gamma_A(s, {}^c\psi_\GL \otimes (\rho_1 \boxtimes S_1 \boxtimes S_{2y+1}),\psi_E)}
{\gamma_A(s, {}^c\psi_\GL \otimes (\rho_1 \boxtimes S_1 \boxtimes S_{2x-1}), \psi_E)^{m-2}
\gamma_A(s, {}^c\psi_\GL \otimes (\rho_1 \boxtimes S_2 \boxtimes S_{2x}), \psi_E)}
\\&\quad\times 
\prod_{\epsilon \in \{\pm1\}} \gamma_A(s, {}^c\psi_\GL \otimes (\rho_1|\cdot|_E^{\epsilon \half{x+y}} \boxtimes S_{y-x+1}), \psi_E)^{-1}
\gamma_A(s, {}^c\psi_\GL \otimes \rho_1|\cdot|_E^{\epsilon x}, \psi_E)^{-(m-2)}, 
\end{align*}
it is equivalent to checking that 
\begin{align*}
&\frac{\gamma_A(s, {}^c\psi_\GL \otimes (\rho_1 \boxtimes S_1 \boxtimes S_{2x+1}), \psi_E)^{m-2}
\gamma_A(s, {}^c\psi_\GL \otimes (\rho_1 \boxtimes S_1 \boxtimes S_{2y+1}),\psi_E)}
{\gamma_A(s, {}^c\psi_\GL \otimes (\rho_1 \boxtimes S_1 \boxtimes S_{2x-1}), \psi_E)^{m-1}}
\\&\quad\times 
\left.\prod_{\epsilon \in \{\pm1\}} \gamma_A(s, {}^c\psi_\GL \otimes (\rho_1|\cdot|_E^{\epsilon \half{x+y}} \boxtimes S_{y-x+1}), \psi_E)^{-1}
\gamma_A(s, {}^c\psi_\GL \otimes \rho_1|\cdot|_E^{\epsilon x}, \psi_E)^{-(m-2)}
\right|_{s=0}
\\&= 
 \left\{
\begin{aligned}
&{-1} \iif \rho_\GL \cong \rho_1,\, x \leq \alpha < y, \alpha \equiv x \bmod \Z, \\
&1 \other.
\end{aligned}
\right. 
\end{align*}
\par

Let us check this equation. 
We have
\begin{align*}
&\frac{\gamma_A(s, {}^c\psi_\GL \otimes (\rho_1 \boxtimes S_1 \boxtimes S_{2x+1}), \psi_E)^{m-2}
\gamma_A(s, {}^c\psi_\GL \otimes (\rho_1 \boxtimes S_1 \boxtimes S_{2y+1}),\psi_E)}
{\gamma_A(s, {}^c\psi_\GL \otimes (\rho_1 \boxtimes S_1 \boxtimes S_{2x-1}), \psi_E)^{m-1}}
\\&\quad\times 
\prod_{\epsilon \in \{\pm1\}} \gamma_A(s, {}^c\psi_\GL \otimes (\rho_1|\cdot|_E^{\epsilon \half{x+y}} \boxtimes S_{y-x+1}), \psi_E)^{-1}
\gamma_A(s, {}^c\psi_\GL \otimes \rho_1|\cdot|_E^{\epsilon x}, \psi_E)^{-(m-2)}
\\&=
\prod_{-\alpha \leq a \leq \alpha} 
\prod_{-x \leq b \leq x} 
\gamma_A(s, {}^c\rho_\GL|\cdot|_E^a \otimes \rho_1|\cdot|_E^b, \psi_E)^{m-2}
\\&\quad \times 
\prod_{-\alpha \leq a \leq \alpha} 
\prod_{-y \leq b \leq y} 
\gamma_A(s, {}^c\rho_\GL|\cdot|_E^a \otimes \rho_1|\cdot|_E^b, \psi_E)
\\&\quad \times
\prod_{-\alpha \leq a \leq \alpha} 
\prod_{-x+1 \leq b \leq x-1} 
\gamma_A(s, {}^c\rho_\GL|\cdot|_E^a \otimes \rho_1|\cdot|_E^b, \psi_E)^{-(m-1)}
\\&\quad\times 
\prod_{-\alpha \leq a \leq \alpha}\prod_{\epsilon \in \{\pm1\}} 
\gamma_A(s, ({}^c\rho_\GL|\cdot|_E^a \otimes \rho_1|\cdot|_E^{\epsilon \half{x+y}}) \boxtimes S_{y-x+1}, \psi_E)^{-1}
\\&\quad\times 
\prod_{-\alpha \leq a \leq \alpha}\prod_{\epsilon \in \{\pm1\}} 
\gamma_A(s, {}^c\rho_\GL|\cdot|_E^a \otimes \rho_1|\cdot|_E^{\epsilon x}, \psi_E)^{-(m-2)}
\\&=
\prod_{-\alpha \leq a \leq \alpha} 
\prod_{-y \leq b \leq y} 
\gamma_A(s, {}^c\rho_\GL|\cdot|_E^a \otimes \rho_1|\cdot|_E^b, \psi_E)
\\&\quad \times
\prod_{-\alpha \leq a \leq \alpha} 
\prod_{-x+1 \leq b \leq x-1} 
\gamma_A(s, {}^c\rho_\GL|\cdot|_E^a \otimes \rho_1|\cdot|_E^b, \psi_E)^{-1}
\\&\quad\times 
\prod_{-\alpha \leq a \leq \alpha}\prod_{\epsilon \in \{\pm1\}} 
\gamma_A(s, ({}^c\rho_\GL|\cdot|_E^a \otimes \rho_1|\cdot|_E^{\epsilon \half{x+y}}) \boxtimes S_{y-x+1}, \psi_E)^{-1}
\\&= 
\prod_{\mu \in X({}^c\rho_\GL \otimes \rho_1)}\prod_{-\alpha \leq a \leq \alpha}
q_E^{-(1-2s-2\mu-2a)(y-x)}
\prod_{\substack{-y \leq b \leq -x-1 \\ \text{or} \\ x \leq b \leq y-1}}
\frac{\zeta_E(s+\mu+a+b+1)}{\zeta_E(s+\mu+a+b)} 
\\&= 
\prod_{\mu \in X({}^c\rho_\GL \otimes \rho_1)}\prod_{-\alpha \leq a \leq \alpha}
q_E^{-(1-2s-2\mu-2a)(y-x)}
\frac{\zeta_E(s+\mu+a-x)}{\zeta_E(s+\mu+a-y)} 
\frac{\zeta_E(s+\mu-a+y)}{\zeta_E(s+\mu-a+x)}
\\&= 
\prod_{\mu \in X({}^c\rho_\GL \otimes \rho_1)}\prod_{-\alpha \leq a \leq \alpha}
\frac{f_{\mu,a,y}(s)}{f_{\mu,a,x}(s)}.
\end{align*}
By the same argument as in Section \ref{(a)}, 
this is equal to 
\[
 \left\{
\begin{aligned}
&{-1} \iif \rho_\GL \cong \rho_1,\, x \leq \alpha < y, \alpha \equiv x \bmod \Z, \\
&1 \other
\end{aligned}
\right. 
\]
after evaluating at $s=0$.
This completes case (c), and the proof of Theorem \ref{main3} (2).

\appendix
\section{Local factors}\label{sec.LF}
In this appendix, we recall some facts about local factors.
In particular, we shall show that 
the local Langlands correspondence for classical groups $G_0$ identifies 
the standard Shahidi local factors for irreducible generic representations of $\GL_k(E) \times G_0$
with the tensor product Artin local factors of the corresponding $L$-parameters. 
This result was used in Lemma \ref{classical-maximal}, 
which is for the first main theorem (Theorem \ref{main1}).
We also give a proof of Proposition \ref{multiplicative}.

\subsection{Formulas for Artin local factors}\label{sec.factors}
We use the notation in Section \ref{sec.notation}. 
Assume that $E$ is non-archimedean, and write $q_E$ for the cardinality of the residue field of $E$. 
\par

Let $I_E$ be the inertia subgroup of the Weil group $W_E$. 
For a representation $(\phi, V)$ of $W_E$, 
we write 
\[
\phi^{I_E} = \{v \in V \,|\, \phi(w)v=v, \, \forall w \in I_E\}. 
\]
This is a subrepresentation of $\phi$. 
Moreover, any irreducible component of $\phi^{I_E}$ is unramified 
so that there is a finite multi-set 
\[
X(\phi) \subset \C/ 2\pi\I(\log q_E)^{-1}\Z
\]
such that 
\[
\phi^{I_E} \cong \bigoplus_{\mu \in X(\phi)} |\cdot|_E^\mu.
\]
Note that 
\begin{itemize}
\item
$X(\phi|\cdot|_E^{s_0}) = \{\mu+s_0 \,|\, \mu \in X(\phi)\}$; 
\item
if $\phi(W_E)$ is bounded, then $\re(\mu) = 0$ for any $\mu \in X(\phi)$; 
\item
if $\phi$ is conjugate-self-dual, then $X(\phi)$ is invariant under $\mu \mapsto -\mu$.
\end{itemize}
We denote by $\mu_0 \in \C/ 2\pi\I(\log q_E)^{-1}\Z$ 
the unique nonzero element such that $\mu_0 = -\mu_0$. 
It satisfies that $q_E^{-\mu_0} = -1$. 
\par

We recall the formulas for the local factors.
Let $\phi$ be a representation of $W_E \times \SL_2(\C)$, 
and decompose it as
\[
\phi \cong \bigoplus_{d \geq 1} \phi_d \boxtimes S_{d}
\]
with $\phi_d$ a representation of $W_E$. 
Let $\zeta_E(s) = (1-q_E^{-s})^{-1}$ be the local zeta function associated to $E$.
Then there are constants $\ep(\phi_d) \in \C^\times$ and $c(\phi_d) = c(\phi_d, \psi_E) \in \Z$ such that 
\begin{align*}
L(s,\phi) &= \prod_{d \geq 1} \prod_{\mu \in X(\phi_d)}\zeta_E\left(s+\mu+\frac{d-1}{2} \right), \\
\ep(s,\phi,\psi_E) &= 
\prod_{d \geq 1} \left(\ep(\phi_d)q_E^{c(\phi_d)(\half{1}-s)} \right)^d 
\prod_{\mu \in X(\phi_d)} (-q_E^{\half{1}-s-\mu})^{d-1}.
\end{align*}
In particular, we have 
\begin{align*}
\gamma_A(s, \phi, \psi_E) &= \ep(s,\phi,\psi_E) \frac{L(1+s,\phi)}{L(s,\phi)}
\\&= \prod_{d \geq 1} \left(\ep(\phi_d)q_E^{c(\phi_d)(\half{1}-s)} \right)^d 
\prod_{\mu \in X(\phi_d)} (-q_E^{(\half{1}-s-\mu)})^{d-1}\frac{\zeta_E(s+\mu+\half{d+1})}{\zeta_E(s+\mu+\half{d-1})}.
\end{align*}
Moreover, $L(s, \phi|\cdot|_E^{s_0}) = L(s+s_0, \phi)$ and $\ep(s,\phi|\cdot|_E^{s_0}, \psi_E) = \ep(s+s_0, \phi, \psi_E)$ hold. 
Hence we have $c(\phi_d|\cdot|_E^{s_0}) = c(\phi_d)$ and $\ep(\phi_d |\cdot|_E^{s_0}) = \ep(\phi_d) q_E^{-c(\phi_d)s_0}$. 
\par

When $[E:F] = 2$, we denote by ${}^c\phi$ the conjugate of $\phi$. 
We note that $X({}^c\phi) = X(\phi)$ and $c({}^c\phi) = c(\phi)$
for any representation $\phi$ of $W_E$.
\par

\begin{lem}
Suppose that $[E:F]=2$. 
Let $\psi_E'$ be a non-trivial additive character of $E$ which is trivial on $F$. 
If $\phi$ is a conjugate-self-dual representation of $W_E \times \SL_2(\C)$, then we have
\[
\frac{\gamma_A(s, {}^c\phi, \psi_E')}{\gamma_A(s, \phi, \psi_E')} = \det\phi(-1). 
\]
\end{lem}
\begin{proof}
Since $X({}^c\phi) = X(\phi)$, we have $L(s, {}^c\phi) = L(s, \phi)$. 
On the other hand, since 
\[
\ep(s, \phi^\vee, \psi_E') \ep(1-s, \phi, \psi_E') = \det\phi^\vee(-1) = \det\phi(-1), 
\]
by using ${}^c\phi \cong \phi^\vee$, 
we see that 
\begin{align*}
\frac{\ep(s, {}^c\phi, \psi_E')}{\ep(s,\phi,\psi_E')} 
= \frac{\ep(s, {}^c\phi, \psi_E')\ep(1-s,\phi,\psi_E')}{\ep(s,\phi,\psi_E')\ep(1-s,\phi,\psi_E')} 
= \frac{\det\phi(-1)}{\ep(\half{1},\phi,\psi_E')^2}. 
\end{align*}
Since $\phi$ is conjugate-self-dual, 
by \cite[Proposition 5.1 (2)]{GGP}, we know that $\ep(\half{1},\phi,\psi_E')^2 = 1$.
Hence we obtain the assertion.
\end{proof}

\begin{prop}\label{gamma_c}
Suppose that $[E:F]=2$. 
Let $\phi_1$ and $\phi_2$ be two conjugate-self-dual representations of $W_E \times \SL_2(\C)$.
If $\det\phi_1 = \det\phi_2$,
then we have
\[
\frac{\gamma_A(s, \phi_1, \psi_E)}{\gamma_A(s, \phi_2, \psi_E)} 
= 
\frac{\gamma_A(s, {}^c\phi_1, \psi_E)}{\gamma_A(s, {}^c\phi_2, \psi_E)}.
\]
\end{prop}
\begin{proof}
For $a \in E^\times$, define an additive character $a\psi_E$ of $E$ by $(a\psi_E)(x) = \psi_E(ax)$. 
Then 
\[
\ep(s, \phi_i, a\psi_E) = \det(\phi_i)(a)|a|_E^{\dim(\phi_i)(s-\half{1})} \ep(s,\phi_i,\psi_E)
\]
holds.
Hence we may replace $\psi_E$ by any non-trivial additive character $\psi_E'$ of $E$.
If we use $\psi_E'$ such that $\psi_E'|_F = \1$, then the assertion follows from the previous lemma.
\end{proof}

\subsection{Comparison of $\gamma$-factors}\label{sec.gamma}
Now we drop the assumption that the field $E$ is non-archimedean.
Let $G^\circ$ be a connected quasi-split classical group. 
Fix a Whittaker datum $\ww$ for $G^\circ$. 
Let $P^\circ = M^\circ N_P$ be a standard maximal parabolic subgroup of $G^\circ$
so that $M^\circ = \GL_{k}(E) \times G_0^\circ$. 
Consider irreducible $\ww$-generic representations $\tau$ and $\sigma$ 
of $\GL_k(E)$ and $G_0^\circ$, respectively. 
(By abuse of language, we say that $\tau$ or $\sigma$ is $\ww$-generic 
if it is generic with respect to the Whittaker datum induced by $\ww$.)
We denote by 
\[
\gamma^{\Sh}(s, \tau \times \sigma, \psi_E) 
= \ep^{\Sh}(s, \tau \times \sigma, \psi_E) 
\frac{L^{\Sh}(1-s, \tau^\vee \times \sigma^\vee)}{L^{\Sh}(s, \tau \times \sigma)}
\]
the associated $\gamma$-factor of Shahidi \cite{Sh7}.
On the other hand, 
let $\phi_\tau$ and $\phi_\sigma$ be the $L$-parameters associated to $\tau$ and $\sigma$, respectively,
and define the Artin $\gamma$-factor over $E$ associated to $\phi_\tau \otimes \phi_\sigma$ by 
\[
\gamma(s, \phi_\tau \otimes \phi_\sigma, \psi_E) 
= \ep(s, \phi_\tau \otimes \phi_\sigma, \psi_E) 
\frac{L(1-s, \phi_\tau^\vee \otimes \phi_\sigma^\vee)}{L(s, \phi_\tau \otimes \phi_\sigma)}. 
\]
Here, we assume the local classification theorem for $G_0^\circ$ by the induction hypothesis, so that there is a commutative diagram
\[
\begin{CD}
\Irr(G^\circ_0) @>>> \Phi(G^\circ_0) \\
@VVV @VVV \\
\Irr(\GL_{N_0}(E)) @>>> \Phi(\GL_{N_0}(E))
\end{CD}
\]
where the horizontal arrows are the local Langlands correspondence
with $\Phi(H)$ the set of $L$-parameters for $H$, 
the left vertical arrow is the twisted endoscopic transfer (for tempered representations), 
and the right vertical arrow is the natural map.
In other words, if $\pi$ is the functorial lift of $\sigma$ to $\GL_{N_0}(E)$ (in terms of the twisted endoscopic character relations), then $\phi_\sigma$ (regarded as a representation of $L_E$) is defined as the $L$-parameter $\phi_\pi$ of $\pi$.
In particular, we have
\[
\gamma(s, \phi_\tau \otimes \phi_\sigma, \psi_E) = \gamma(s, \phi_\tau \otimes \phi_\pi, \psi_E).
\]

\begin{prop}\label{gamma=}
For any irreducible $\ww$-generic representation $\tau \boxtimes \sigma$ of $M^\circ$, 
we have
\[
\gamma^{\Sh}(s, \tau \times \sigma, \psi_E) 
= 
\gamma(s, \phi_\tau \otimes \phi_\sigma, \psi_E). 
\]
In particular, 
if $\tau$ and $\sigma$ are tempered, then 
the equalities 
\[
L^{\Sh}(s, \tau \times \sigma) = L(s, \phi_\tau \otimes \phi_\sigma),
\quad
\ep^{\Sh}(s, \tau \times \sigma, \psi_E) = \ep(s, \phi_\tau \otimes \phi_\sigma, \psi_E)
\]
also hold. 
\end{prop}

The desired equality of $\gamma$-factors in Proposition \ref{gamma=} seems to be well-known to experts, 
but we briefly review the argument in Section \ref{sec.gamma=} below.
Before it, we give realizations of our quasi-split groups and splittings.

\subsection{Groups and splittings}\label{secA.groups}
We denote by $E_{i,j}$ the square matrix (of a certain size) with $1$ at the $(i,j)$-th entry and $0$ elsewhere. 
Define an $n \times n$ anti-diagonal matrix $J_n$ by 
\[
J_n =
\left(\begin{smallmatrix}
&&1 \\
&\iddots&\\
1&&
\end{smallmatrix}\right).
\]
\par

For reductive algebraic groups $G,T,\dots$ over $F$, 
we use the corresponding Gothic letters $\g, \tt, \ldots$ for the associated Lie algebras.
We consider the following quasi-split reductive algebraic group $G$ over $F$ 
and the $F$-splitting $\spl = (B^\circ,T^\circ,\{X_\alpha\})$ of $G^\circ$.

\begin{description}
\setlength{\itemsep}{10pt}

\item[Symplectic groups]
Suppose that $E = F$.
Let $G = \Sp_{2n}(F)$ be the symplectic group defined by 
\[
\Sp_{2n}(F) = \{ g \in \GL_{2n}(F) \,|\, {}^t g J_{2n}' g = J_{2n}' \}, \quad
J_{2n}' = 
\begin{pmatrix}
& J_n \\
-J_n &
\end{pmatrix}.
\]
Take the Borel subgroup $B$ of $G$ consisting of upper triangular matrices 
and the maximal torus $T$ of $G$ consisting of diagonal matrices.
Then the corresponding positive simple roots are given by 
$\alpha_i = e_i - e_{i+1}$ for $1 \leq i \leq n-1$ and $\alpha_n = 2 e_n$, 
where $\{ e_1, \dots, e_{n} \}$ is the standard basis of $X^*(T)$.
Take the root vectors given by 
$X_{\alpha_i} = E_{i,i+1} - E_{2n-i,2n+1-i}$ for $1 \leq i \leq n-1$ and $X_{\alpha_n} = E_{n,n+1}$.

\item[Odd special orthogonal groups]
Suppose that $E = F$.
For an extension $K$ of $F$, 
let $\O_{N}(K)$ be the orthogonal group defined by 
\[
\O_{N}(K) = \{ g \in \GL_{N}(K) \,|\, {}^t g J_{N} g = J_{N} \}.
\]
Take the Borel subgroup $B^\circ$ of $\SO_{N}(\overline{F})$ consisting of upper triangular matrices 
and the maximal torus $T^\circ$ of $\SO_{N}(\overline{F})$ consisting of diagonal matrices.
\par

Suppose that $N=2n+1$ and consider $G = \SO_{2n+1}(F)$. 
Then the corresponding positive simple roots are given by 
$\alpha_i = e_i - e_{i+1}$ for $1 \leq i \leq n-1$ and $\alpha_n = e_n$, 
where $\{ e_1, \dots, e_n \}$ is the standard basis of $X^*(T)$.
Take the root vectors given by 
$X_{\alpha_i} = E_{i,i+1} - E_{2n+1-i,2n+2-i}$ for $1 \leq i \leq n$. 

\item[Even special orthogonal groups]
Suppose that $E = F$.
Let $\O_{N}(\overline{F})$ be the orthogonal group as above.
Suppose that $N=2n$ and consider $\SO_{2n}(\overline{F})$. 
Then the corresponding positive simple roots are given by 
$\alpha_i = e_i - e_{i+1}$ for $1 \leq i \leq n-1$ and $\alpha_n = e_{n-1} + e_n$, 
where $\{ e_1, \dots, e_n \}$ is the standard basis of $X^*(T)$.
Take the root vectors given by $X_{\alpha_i} = E_{i,i+1} - E_{2n-i,2n+1-i}$ for $1 \leq i \leq n-1$ 
and $X_{\alpha_n} = E_{n-1,n+1} - E_{n,n+2}$.
For example, if $n = 2$, then 
\[
{\tiny
X_{\alpha_{n-1}}= 
\left(
\begin{array}{c|cc|c}
0&1&&\\
\hline
&0&&\\
&&0&-1\\
\hline
&&&0
\end{array}
\right), 
\quad 
X_{\alpha_{n}}= 
\left(
\begin{array}{c|cc|c}
0&&1&\\
\hline
&0&&-1\\
&&0&\\
\hline
&&&0
\end{array}
\right).
}
\]
\par

Now for a (possibly trivial) quadratic character $\eta$ of $F^\times$, 
we define a form $G = \O_{2n}^\eta(F)$ as follows. 
When $\eta$ is trivial, we put $\O_{2n}^\eta(F) = \O_{2n}(F)$ as above.
When $\eta$ is non-trivial, 
we denote by $K$ the quadratic extension of $F$ associated to $\eta$ by the local class field theory.
Then we define $\O_{2n}^\eta(F)$ as the subgroup $\O_{2n}(K)$ 
consisting of matrices $g$ such that
\[
g = \epsilon g^\rho \epsilon^{-1}, \quad
\epsilon =
\left(
\begin{smallmatrix}
\1_{n-1}&&& \\
&0&1& \\
&1&0& \\
&&&\1_{n-1}
\end{smallmatrix}
\right), 
\]
where $\rho$ is the non-trivial element in $\Gal (K/F)$. 
Note that $B^\circ$ and $T^\circ$ are defined over $F$, 
and $\Ad(\epsilon)(X_{\alpha_i}) = X_{\alpha_i}$ for $1 \leq i \leq n-2$ 
and $\Ad(\epsilon)(X_{\alpha_{n-1}}) = X_{\alpha_n}$.
In particular, $\epsilon$ fixes the $F$-splitting $\spl = (B^\circ, T^\circ, \{X_{\alpha}\})$.
\par

We remark that any $F$-splitting of $G^\circ$ is conjugate to the splitting $\spl' = (B^\circ,T^\circ,\{X'_\alpha\})$, 
where $X'_{\alpha_i} = X_{\alpha_i}$ for $1 \le i \le n-2$ 
and 
\[
(X'_{\alpha_{n-1}}, X'_{\alpha_n}) = 
\left\{
\begin{aligned}
&(X_{\alpha_{n-1}}, a X_{\alpha_n}) \iif \eta = \1, \\
&(b X_{\alpha_{n-1}}, b^\rho X_{\alpha_n}) \iif \eta \not= \1
\end{aligned}
\right.
\]
for some $a \in F^\times$ and $b \in K^\times$.
Then $\spl'$ is fixed by $\epsilon' = \epsilon t_0$ with
\[
t_0 = 
\left\{
\begin{aligned}
&\diag(\1_{n-1}, a^{-1}, a, \1_{n-1}) \iif \eta = \1, \\
&\diag(\1_{n-1}, b (b^\rho)^{-1}, b^{-1} b^\rho, \1_{n-1}) \iif \eta \not= \1.
\end{aligned}
\right.
\]
Finally, we notice that $\g = \Lie(G) = \Lie(G^\circ)$.

\item[Unitary groups]
Suppose that $[E:F] = 2$.
Let $G = \U_n$ be the unitary group defined by 
\[
\U_n = \{ g \in \GL_n(E) \,|\, {}^t \overline{g} J''_n g = J''_n \}, 
\quad J_n'' = \diag(1,-1, \dots, (-1)^{n-1}) \cdot J_n.
\]
Take the Borel subgroup $B$ of $G$ consisting of upper triangular matrices 
and the maximal torus $T$ of $G$ consisting of diagonal matrices.
Then the corresponding positive simple roots are given by $\alpha_i = e_i - e_{i+1}$ for $1 \leq i \leq n-1$, 
where $\{ e_1, \dots, e_n \}$ is the standard basis of $X^*(T)$.
Take the root vectors given by $X_{\alpha_i} = E_{i,i+1}$ for $1 \leq i \leq n-1$.
\end{description}

\subsection{Proof of Proposition \ref{gamma=}}\label{sec.gamma=}
Now we shall prove Proposition \ref{gamma=}.
\begin{proof}[Proof of Proposition \ref{gamma=}]
If $\tau$ and $\sigma$ are tempered, 
then the $L$-factors and the $\ep$-factors are uniquely determined by the corresponding $\gamma$-factors.
Hence the equations for the $L$-factors and the $\ep$-factors 
are derived from the corresponding equality for the $\gamma$-factors. 
Moreover, since $\phi_\sigma = \phi_\pi$ as representations of $L_E$, 
where $\pi$ is the functorial lift of $\sigma$, 
it suffices to show that 
\[
\gamma^\Sh(s, \tau \times \sigma, \psi_E) = \gamma(s, \phi_\tau \otimes \phi_\pi, \psi_E).
\]
\par

The equation for the $\gamma$-factors easily follows from 
the characterizing properties of $\gamma^{\Sh}(s, \tau \times \sigma, \psi_E)$, 
proved in \cite[Theorem 3.5]{Sh7}, 
which are some of the ``Ten Commandments'' given in \cite{LR}.
For the convenience of the readers, we recall the properties we need. 
\begin{enumerate}
\item
\label{TC-1}
(unramified twisting)
For $s_0 \in \C$, we have
\[
\gamma^{\Sh}(s, \tau |\cdot|_E^{s_0} \times \sigma, \psi_E) 
= \gamma^{\Sh}(s+s_0, \tau \times \sigma, \psi_E).
\]
\item
\label{TC-2}
(dependence on $\psi_F$)
Let $\psi'_F$ be another non-trivial additive character of $F$, 
so that $\psi'_F(x) = \psi_F(ax)$ for some $a \in F^\times$, 
and put $\psi_E' = \psi_F' \circ \tr_{E/F}$.
Then we have
\[
\gamma^{\Sh}(s, \tau \times \sigma, \psi'_E) 
= \eta(a)^{k} \omega_\tau(a)^{N_0} |a|_E^{kN_0(s-\frac{1}{2})} 
\gamma^{\Sh}(s, \tau \times \sigma, \psi_E),
\]
where $\eta$ is the trivial character of $F^\times$ unless $G^\circ$ is an even special orthogonal group, 
in which case $\eta$ is the (possibly trivial) quadratic character of $F^\times$ 
associated to the splitting field of $G^\circ$, 
$\omega_\tau$ is the central character of $\tau$, and $N_0 = \dim(\St_{\widehat{G_0}})$. 
See Section \ref{sec.property2}.

\item
\label{TC-3}
(multiplicativity)
Assume that $\tau$ is a subrepresentation of $I_{P_1}(\tau_1 \boxtimes \tau_2)$, 
where $P_1$ is a standard parabolic subgroup of $\GL_k(E)$ 
with Levi component $\GL_{k_1}(E) \times \GL_{k_2}(E)$, 
and $\tau_1$ and $\tau_2$ are irreducible $\ww$-generic representations 
of $\GL_{k_1}(E)$ and $\GL_{k_2}(E)$, respectively.
Then we have
\[
 \gamma^{\Sh}(s, \tau \times \sigma, \psi_E)
 = \gamma^{\Sh}(s, \tau_1 \times \sigma, \psi_E) \cdot \gamma^{\Sh}(s, \tau_2 \times \sigma, \psi_E).
\]
Similarly, assume that $\sigma$ is a subrepresentation of $I_{P_0}(\tau_0 \boxtimes \sigma_0)$, 
where $P_0$ is a standard parabolic subgroup of $G_0^\circ$ 
with Levi component $\GL_{k_0}(E) \times G_{00}^\circ$, 
and $\tau_0$ and $\sigma_0$ are irreducible $\ww$-generic representations 
of $\GL_{k_0}(E)$ and $G_{00}^\circ$, respectively.
Then we have
\[
\gamma^{\Sh}(s, \tau \times \sigma, \psi_E) 
= \gamma(s, \tau \times \tau_0, \psi_E) \cdot
\gamma(s, \tau \times {}^c\tau_0^\vee, \psi_E) \cdot 
\gamma^{\Sh}(s, \tau \times \sigma_0, \psi_E).
\]
Here, $\gamma(s, \tau \times \tau_0, \psi_E)$ is the Rankin--Selberg $\gamma$-factor 
which is equal to the Artin $\gamma$-factor $\gamma(s, \phi_\tau \otimes \phi_{\tau_0}, \psi_E)$.
\item
\label{TC-4}
(unramified factors)
Assume that $F$ is non-archimedean, 
and $\tau$ and $\sigma$ are spherical 
(in the sense that they have nonzero fixed vectors under good special maximal compact subgroups).
Then we have
\[
 \gamma^{\Sh}(s, \tau \times \sigma, \psi_E)
  = \gamma(s, \phi_\tau \otimes \phi_\sigma, \psi_E).
\]
\item
\label{TC-5}
(archimedean property)
Assume that $F$ is archimedean.
Then we have
\[
 \gamma^{\Sh}(s, \tau \times \sigma, \psi_E)
 = \gamma(s, \phi_\tau \otimes \phi_\sigma, \psi_E).
\]
\item
\label{TC-6}
(global property)
Let $\dot{F}$ be a number field with ring of ad\`{e}les $\dot{\A} = \dot\A_{\dot{F}}$, 
and let $\dot{E}$ be either $\dot{F}$ or a quadratic field extension of $\dot{F}$.
Fix a non-trivial additive character $\psi_{\dot{F}}$ of $\dot{\A}/\dot{F}$ 
and put $\psi_{\dot{E}} = \psi_{\dot{F}} \circ \tr_{\dot{E}/\dot{F}}$.
Let $\dot{G}^\circ$ be a quasi-split classical group defined over $\dot{F}$, 
and let $\dot{M}^\circ = \Res_{\dot{E}/\dot{F}} \GL_{k} \times \dot{G}_0^\circ$ 
be a maximal semi-standard Levi subgroup of $\dot{G}^\circ$. 
We denote by $\dot{\ww}$ the Whittaker datum induced by the $\dot{F}$-splitting of $\dot{G}^\circ$ 
and $\psi_{\dot{F}}$.
Let $\dot{\tau}$ and $\dot{\sigma}$ be irreducible globally $\dot{\ww}$-generic cuspidal automorphic representations of $\GL_k(\dot{\A}_{\dot{E}})$ and $\dot{G}_0^\circ(\dot{\A})$, respectively.
Then we have
\[
L^S(s, \dot{\tau} \times \dot{\sigma}) = 
\prod_{v \in S} \gamma^{\Sh}(s, \dot{\tau}_v \times \dot{\sigma}_v, \psi_{\dot{E},v}) 
\cdot L^S(1-s, \dot{\tau}^\vee \times \dot{\sigma}^\vee),
\]
where $S$ is a sufficiently large finite set of places of $\dot{F}$ 
and $L^S(s, \dot{\tau} \times \dot{\sigma}) = \prod_{v \notin S} L(s, \phi_{\dot{\tau}_v} \otimes \phi_{\dot{\sigma}_v})$ 
is the partial $L$-function (for $\re(s)$ sufficiently large).
\end{enumerate}
\par

By \eqref{TC-5}, we may assume that $F$ is non-archimedean.
By the Langlands classification, 
we may write $\tau$ and $\sigma$ as unique irreducible subrepresentations of the duals of standard modules.
Since $\tau$ and $\sigma$ are $\ww$-generic, 
the inducing data of these standard modules are also $\ww$-generic.
(Cf., \cite[Lemma 2.2]{AG}.)
Hence, by \eqref{TC-1}, \eqref{TC-3}, and the definition of $L$-parameters, 
we may assume that $\tau$ and $\sigma$ are tempered.
In this case, we may write $\tau$ and $\sigma$ as 
subrepresentations of parabolic inductions of square-integrable representations. 
By the same argument, we may assume that $\tau$ and $\sigma$ are square-integrable.
\par

First we treat the case where $\tau$ and $\sigma$ are supercuspidal.
Choose $\dot{F}, \dot{E}, \dot{G}^\circ, \dot{M}^\circ, \psi_{\dot{F}}$ as in \eqref{TC-6} 
such that $\dot{F}_{v_0} = F, \dot{E}_{v_0} = E, \dot{G}^\circ_{v_0} = G^\circ, \dot{M}^\circ_{v_0} = M^\circ$ 
for some finite place $v_0$ of $\dot{F}$.
Note that it is not always possible to find $\psi_{\dot{F}}$ such that $\psi_{\dot{F},v_0} = \psi_F$.
By the Poincar{\'e} series argument \cite[Appendice 1]{He1}, \cite[Proposition 5.1]{Sh7}, we can find $\dot{\tau}$ and $\dot{\sigma}$ as in \eqref{TC-6} 
such that $\dot{\tau}_{v_0} \cong \tau$ and $\dot{\sigma}_{v_0} \cong \sigma$, 
and such that $\dot{\tau}_v$ and $\dot{\sigma}_v$ are spherical for all finite places $v \not= v_0$.
Moreover, by \cite{CKPSS1, CKPSS2}, \cite{KK1, KK2} and \cite{CPSS}, 
the functorial lift $\dot\pi$ of $\dot{\sigma}$ is cuspidal.
In other words, 
by the global classification theorem for $\dot{G}_0^\circ$ (which we assume by the induction hypothesis), 
the global $A$-parameter of $\dot{\sigma}$ is generic 
and $\dot\pi_v$ is the functorial lift of $\dot\sigma_v$ for all places $v$. 
Then the global functional equation says that 
\[
L^S(s, \dot{\tau} \times \dot{\pi}) = 
\prod_{v \in S} \gamma(s, \dot{\tau}_v \times \dot{\pi}_v, \psi_{\dot{E},v}) 
\cdot L^S(1-s, \dot{\tau}^\vee \times \dot{\pi}^\vee),
\]
where $S$ is a sufficiently large finite set of places of $\dot{F}$, 
$L^S(s, \dot{\tau} \times \dot{\pi}) = \prod_{v \not\in S} L(s, \phi_{\dot{\tau}_v} \otimes \phi_{\dot{\pi}_v})$
is the partial $L$-function (for $\re(s)$ sufficiently large), 
and $\gamma(s, \dot{\tau}_v \times \dot{\pi}_v, \psi_{\dot{E},v})$ is the Rankin--Selberg $\gamma$-factor. 
Since $L^S(s, \dot{\tau} \times \dot{\pi}) = L^S(s, \dot{\tau} \times \dot{\sigma})$
and 
$\gamma(s, \dot{\tau}_v \times \dot{\pi}_v, \psi_{\dot{E},v}) 
= \gamma(s, \phi_{\dot{\tau}_v} \otimes \phi_{\dot{\pi}_v}, \psi_{\dot{E},v})$
(which is a desideratum of the local Langlands correspondence for general linear groups), 
we may write this equality as
\[
L^S(s, \dot{\tau} \times \dot{\sigma}) = 
\prod_{v \in S} \gamma(s, \phi_{\dot{\tau}_v} \otimes \phi_{\dot{\pi}_v}, \psi_{\dot{E},v}) 
\cdot L^S(1-s, \dot{\tau}^\vee \times \dot{\sigma}^\vee).
\]
On the other hand, by \eqref{TC-4}, \eqref{TC-5}, we have
\[
\gamma^{\Sh}(s, \dot{\tau}_v \times \dot{\sigma}_v, \psi_{\dot{E},v})
= \gamma(s, \phi_{\dot{\tau}_v} \otimes \phi_{\dot{\sigma}_v}, \psi_{\dot{E},v})
= \gamma(s, \phi_{\dot{\tau}_v} \otimes \phi_{\dot{\pi}_v}, \psi_{\dot{E},v})
\]
for all places $v \not= v_0$.
From this and \eqref{TC-6}, we can deduce that 
\[
\gamma^{\Sh}(s, \tau \times \sigma, \dot{\psi}_{E,v_0})
= \gamma(s, \phi_\tau \otimes \phi_\pi, \dot{\psi}_{E,v_0}).
\]
If we write $\psi_{\dot{F},v_0}(x) = \psi_F(ax)$ for some $a \in F^\times$, then the left-hand side is equal to 
\[
\eta(a)^{k} \omega_\tau(a)^{N_0} |a|_E^{kN_0(s-\frac{1}{2})} 
\gamma^{\Sh}(s, \tau \times \sigma, \psi_E)
\]
by \eqref{TC-2}, whereas the right-hand side is equal to 
\[
 \det(\phi_\tau \otimes \phi_\pi)(a) |a|_E^{\dim(\phi_\tau \otimes \phi_\pi)(s-\frac{1}{2})} \gamma(s, \phi_\tau \otimes \phi_\pi, \psi_E)
\]
(see \cite[Section 3.6]{Tate}).
This implies the desired equation for the $\gamma$-factors.
\par

Now suppose that $\tau$ and $\sigma$ are square-integrable.
Choose $\dot{F}, \dot{E}, \dot{G}^\circ, \dot{M}^\circ, \psi_{\dot{F}}$ as above 
and fix an auxiliary finite place $v_1 \not= v_0$ of $\dot{F}$
such that $v_1$ does not split in $\dot{E}$ when $[\dot{E}:\dot{F}] = 2$.
By the Poincar{\'e} series argument \cite[Appendice 1]{He1} 
for a central division algebra over $\dot{E}$ of degree $k$ ramified precisely at $v_0, v_1$
(where we regard $v_i$ for $i=0,1$ as the unique place of $\dot{E}$ lying over $v_i$ when $[\dot{E}:\dot{F}] = 2$), 
and the global Jacquet--Langlands correspondence \cite{Bad}, 
we can find $\dot{\tau}$ as in \eqref{TC-6} such that 
$\dot{\tau}_{v_0} \simeq \tau$, $\dot{\tau}_{v_1}$ is supercuspidal, 
and $\dot{\tau}_v$ is spherical for all finite places $v \not= v_0, v_1$.
Also, by \cite[Appendix A]{ILM}, \cite[Appendix A]{GI3}, 
we can find $\dot{\sigma}$ as in \eqref{TC-6} such that 
$\dot{\sigma}_{v_0} \simeq \sigma$, $\dot{\sigma}_{v_1}$ is supercuspidal, 
and $\dot{\sigma}_v$ is spherical for all finite places $v \not= v_0, v_1$.
(Note that \cite{GI3} only treats the case of metaplectic groups, but the same argument goes through for other classical groups.)
Since we already know that 
\[
 \gamma^{\Sh}(s, \dot{\tau}_{v_1} \times \dot{\sigma}_{v_1}, \psi_{\dot{E},v_1})
 = \gamma(s, \phi_{\dot{\tau}_{v_1}} \otimes \phi_{\dot{\pi}_{v_1}}, \psi_{\dot{E},v_1}),
\]
where $\dot\pi_{v_1}$ is the functorial lift of $\dot\sigma_{v_1}$, 
the above argument proves the analogous equation for $v_0$
and hence the desired equation for the $\gamma$-factors.
\end{proof}

\begin{rem}\label{CFGK}
Recently, 
Cai--Friedberg--Ginzburg--Kaplan \cite{CFGK}, \cite{C}, \cite{CFK} 
introduced a new family of zeta integrals, 
which generalizes the doubling zeta integrals of Piatetski-Shapiro--Rallis \cite{PSR}, \cite{LR}, 
and established an analytic theory of $\gamma$-factors
\[
\gamma^{\mathrm{CFGK}}(s, \tau \times \sigma, \psi_E)
\]
for all irreducible representations $\tau$ and $\sigma$ of $\GL_k(E)$ and $G_0^\circ$, respectively.
In particular, when $G_0^\circ$ is a split special orthogonal group or a symplectic group, 
the ``Ten Commandments'' were proved in \cite[Theorem 4.2]{CFK}.
In this case, we can modify the above argument and show that
\[
\gamma^{\mathrm{CFGK}}(s, \tau \times \sigma, \psi_E)
= \gamma(s, \phi_\tau \otimes \phi_\sigma, \psi_E)
\]
as follows: 
To globalize an irreducible square-integrable representation $\sigma$ of $G_0^\circ$, 
we can use \cite[Lemma 6.2.2]{Ar} 
to find an irreducible cuspidal automorphic representation 
$\dot{\sigma}$ of $\dot{G}_0^\circ(\dot{\A})$ with a generic global $A$-parameter 
such that $\dot{\sigma}_{v_0} \cong \sigma$ and $\dot{\sigma}_v$ is spherical for all finite places $v \not= v_0$.
\end{rem}

\begin{rem}
The second and third authors would like to take this opportunity to remark that 
the various desiderata of the local Langlands correspondence used in \cite{GI1, GI2} 
are now supplied by the results of this paper in the case of quasi-split classical groups.
Namely, in the proofs of \cite[Theorem C.5]{GI1} and \cite[Proposition B.1]{GI2}, 
they assumed the following hypothesis:
\begin{enumerate}
\item 
the equality between the local $\gamma$-factors of Shahidi and the corresponding Artin $\gamma$-factors;
\item 
the equality between the local $\gamma$-factors of Piatetski-Shapiro--Rallis 
and the corresponding Artin $\gamma$-factors;
\item 
the formula for the Plancherel measures in terms of Artin $\gamma$-factors.
\end{enumerate}
(See \cite[Section C.2]{GI1} and \cite[Section B.2]{GI2} for details.)
Now (1) follows from Proposition \ref{gamma=} and \cite{He4}, \cite{CST}, \cite{Shan}, \cite{He5},
(2) can be verified as in Remark \ref{CFGK} (where the ``Ten Commandments'' in this case were proved in \cite[Theorem 4]{LR}), 
and (3) is a consequence of the multiplicative property of the normalized intertwining operators (see Proposition \ref{multiplicative}).
They would also like to point out that the reason they gave for the correct formulation of Shahidi's formula at the end of the proof of \cite[Lemma B.2]{GI2} is not accurate: the proper justifications for the reformulation of Shahidi's results are given in Section \ref{sec.shahidi} of this paper.
\end{rem}

\begin{rem}
As our argument above uses globalization, 
this may be a good chance to comment on the globalization for unitary groups in \cite{Mok}. 
We were unable to verify Lemma 7.2.1 therein when $E/F$ is a ramified quadratic extension of $p$-adic fields 
in that it is not obvious to keep the global extension $\dot E/\dot F$ unramified at all finite places 
away from the place of interest (denoted $u$); 
if this were true, one could globalize a local unitary group with respect to $E/F$ to a global unitary group 
that is unramified outside $u$. 
Nevertheless this does not affect the globalization argument for unitary groups 
in that it is unnecessary to make $\dot E/\dot F$ unramified outside $u$, 
just like it was unnecessary in Arthur's globalization for (non-split) quasi-split $\SO_{2n}$. 
The basic reason is that one can prescribe a spherical parameter or a spherical representation 
(in the sense recalled in \cite[Section 6.1]{Ar}) 
at ramified finite places as in \cite[Lemma 7.2.3, Corollary 7.2.7]{Mok}, 
cf. \cite[Lemma 6.2.2, Corollary 6.2.4]{Ar}. 
\end{rem}

\subsection{Dependence on $\psi_F$ for Shahidi's gamma factors}\label{sec.property2}
The property \eqref{TC-2} in the proof of Proposition \ref{gamma=} 
is not stated in this form in \cite{Sh7} but can be derived as follows.
We use the notation in Section \ref{sec.maximal} for the classical group $G$.
Suppose first that we are not in the case 
where $G = \O_{2k}^\eta(F)$ with $\eta=1$, $M = \GL_k(F)$, and $k > 1$ is odd.
Put $\pi = \tau \boxtimes \sigma_0^\vee$ (regarded as a representation of $M$) 
and write $\pi_\lambda = \tau |\cdot|_E^s \boxtimes \sigma_0^\vee$ with $s \in \C$.
Recall that the local coefficient $C_P(w,\pi_\lambda)$ is given by
\[
\Omega(\pi_\lambda) = 
C_P(w,\pi_\lambda) \cdot \Omega(w\pi_\lambda) \circ J_P(w,\pi_\lambda)
\]
as in Section \ref{sec.C_P} and depends on the following choices:
\begin{itemize}
\item the Whittaker datum $\ww$ determined by $\spl$ and $\psi_F$;
\item the Weyl group representative $\tl{w} = \tl{w}_{\spl}$ of $w$ determined by $\spl$;
\item the Haar measure $du = du_{\spl,\psi_F}$ on $N_P$ determined by $\spl$ and $\psi_F$.
\end{itemize}
To indicate this dependence, we write
\begin{align*}
C_P(w,\pi_\lambda) & = C_P(w,\pi_\lambda,\spl,\psi_F), \\ 
J_P(w,\pi_\lambda) & = J_P(w,\pi_\lambda,\spl,\psi_F), \\
\Omega(\pi_\lambda) & = \Omega(\pi_\lambda,\spl,\psi_F).
\end{align*}
\par

For another additive character $\psi_F'$ given by $\psi_F'(x) = \psi_F(ax)$ with $a \in F^\times$, we take the splitting $\spl' = (B,T,\{X_\alpha'\})$ such that $X_\alpha' = a X_\alpha$ for all $\alpha$.
Then $\spl'$ and $\psi_F'$ give rise to the Whittaker datum $\ww$. 
We may write $X_\alpha' = \Ad(t_0)(X_\alpha)$, where $t_0 \in A_T(\overline{F})$ is given by 
\[
t_0 = 
\left\{
\begin{aligned}
&\diag(a^{n-\frac{1}{2}}, a^{n-\frac{3}{2}}, \dots, a^{\frac{1}{2}}, 
a^{-\frac{1}{2}}, \dots, a^{-n+\frac{3}{2}}, a^{-n+\frac{1}{2}})
\iif G = \Sp_{2n}(F), \\
&\diag(a^n, a^{n-1}, \dots, a, 1, a^{-1}, \dots, a^{-n+1}, a^{-n})
\iif G = \SO_{2n+1}(F), \\
&\diag(a^{n-1}, a^{n-2}, \dots, a, 1, 1, a^{-1}, \dots, a^{-n+2}, a^{-n+1})
\iif G = \O_{2n}^\eta(F), \\
&\diag(a^{\frac{n-1}{2}}, a^{\frac{n-3}{2}}, \dots, a^{-\frac{n-3}{2}}, a^{-\frac{n-1}{2}})
\iif G = \U_n. 
\end{aligned}
\right. 
\]
Note that the image of $t_0$ in the adjoint group of $G$ is an $F$-rational point, 
so that $\Ad(t_0)$ is an automorphism of $G$ defined over $F$.
Then we have $\tl{w}_{\spl'} = \Ad(t_0)(\tl{w}_{\spl})$.
Set $z_0 = \tl{w}_{\spl'}^{-1} \cdot \tl{w}_{\spl}$.
Recall that $\tl{w}_{\spl}$ is a representative of $w_T \in N(M^\circ,M^\circ)/T^\circ$.
It is easy to see that a representative of $w_T$ is given by 
\[
\begin{pmatrix}
&&\1_k \\ 
&\1_{G_0}&\\
\pm\1_k&&
\end{pmatrix}
\]
for some sign.
Hence
\begin{align*}
z_0 & = \Ad(t_0)(\tl{w}_{\spl})^{-1} \cdot \tl{w}_{\spl} \\
& = t_0 \cdot \Ad(\tl{w}_{\spl}^{-1})(t_0)^{-1} \\
& = \diag(a^{N_0+k+\delta} \1_k, \1_{G_0}, a^{-(N_0+k+\delta)} \1_k),
\end{align*}
where $N_0 = \dim(\St_{\widehat{G_0}})$ and 
\[
\delta = \left\{
\begin{aligned}
&1 \iif G = \SO_{2n+1}(F), \\
&{-1} \iif G = \Sp_{2n}(F),\, \O_{2n}^\eta(F), \\
&0 \iif G = \U_n.
\end{aligned}
\right. 
\]
In particular, $z_0$ belongs to the center of $M$.
Moreover, if we set $l = \dim(N_P)$, then 
\[
du_{\spl',\psi_F'}
= \delta_P(t_0)^{-1}\,du_{\spl,\psi_F'}
= \delta_P(t_0)^{-1} |a|_F^{\frac{l}{2}}\,du_{\spl,\psi_F}.
\]
\par

Since $\tl{w}^{-1}_{\spl'} = z_0 \tl{w}_{\spl}^{-1}$ and $\delta_P(t_0) = \delta_P(z_0)^{\frac{1}{2}}$, 
we obtain that 
\begin{align*}
J_P(w,\pi_\lambda,\spl',\psi_F') 
&= \omega_{\pi_\lambda}(z_0) \delta_P(z_0)^{\frac{1}{2}} \cdot \delta_P(t_0)^{-1} |a|_F^{\frac{l}{2}} 
\cdot J_P(w,\pi_\lambda,\spl,\psi_F) \\ 
&= \omega_{\pi_\lambda}(z_0) |a|_F^{\frac{l}{2}} \cdot J_P(w,\pi_\lambda,\spl,\psi_F).
\end{align*}
Similarly, we have
\begin{align*}
\Omega(\pi_\lambda,\spl',\psi_F') 
&= \omega_{\pi_\lambda}(z_0) |a|_F^{\frac{l}{2}} \cdot \Omega(\pi_\lambda,\spl,\psi_F), \\
\Omega(w\pi_\lambda,\spl',\psi_F') 
&= \omega_{w\pi_\lambda}(z_0) |a|_F^{\frac{l}{2}} \cdot \Omega(w\pi_\lambda,\spl,\psi_F).
\end{align*}
This implies that
\begin{align*}
\omega_{\pi_\lambda}(z_0) |a|_F^{\frac{l}{2}} \cdot \Omega(\pi_\lambda,\spl,\psi_F) & = \omega_{w\pi_\lambda}(z_0) |a|_F^{\frac{l}{2}} \cdot \omega_{\pi_\lambda}(z_0) |a|_F^{\frac{l}{2}} \cdot C_P(w,\pi_\lambda,\spl',\psi_F') \\
& \quad \times \Omega(w\pi_\lambda,\spl,\psi_F) \circ J_P(w,\pi_\lambda,\spl,\psi_F), 
\end{align*}
so that 
\[
C_P(w,\pi_\lambda,\spl',\psi_F')
= \omega_{w\pi_\lambda}(z_0)^{-1} |a|_F^{-\frac{l}{2}} \cdot C_P(w,\pi_\lambda,\spl,\psi_F).
\]
\par

Recall from Section \ref{sec.shahidi} that Shahidi's formula says that 
\[
C_P(w,\pi_\lambda,\spl,\psi_F)
= \lambda(w, \psi_F)^{-1} \lambda(E/F, \psi_F)^{kN_0}
\gamma^\Sh(s, \tau \times \sigma, \psi_E) \gamma^\Sh(2s, \tau, R, \psi_F)
\]
with 
\[
R = \left\{
\begin{aligned}
&\Sym^2 \iif G = \SO_{2n+1}(F), \\
&\wedge^2 \iif G = \Sp_{2n}(F),\, \O_{2n}^\eta(F), \\
&\Asai^+ \iif G=\U_n,\, n \equiv 0 \bmod 2,\\
&\Asai^- \iif G=\U_n,\, n \equiv 1 \bmod 2.
\end{aligned}
\right. 
\]
Note that 
\[
\lambda(w, \psi_F) = 
\left\{
\begin{aligned}
&\lambda(K/F, \psi_F)^k \iif G= \O_{2n}^\eta(F), \\
&\lambda(E/F, \psi_F)^{kN_0+\half{k(k-1)}} \iif G=\U_n,\, n \equiv 0 \bmod 2,\\
&\lambda(E/F, \psi_F)^{kN_0+\half{k(k+1)}} \iif G=\U_n,\, n \equiv 1 \bmod 2,\\
&1 \other,
\end{aligned}
\right. 
\]
where $K/F$ is the abelian extension corresponding to $\eta$.
Since 
\[
\lambda(K/F, \psi_F) 
= \frac{\ep(s, \Ind_{W_K}^{W_F}(\1_{W_K}), \psi_F)}{\ep(s, \1_{W_K}, \psi_F \circ \tr_{K/F})}
\]
which does not depend on $s$, 
we have
\[
\frac{\lambda(K/F, \psi_F')}{\lambda(K/F, \psi_F)}
= \det(\Ind_{W_K}^{W_F}(\1_K))(a) = \eta(a).
\]
Hence
\[
\frac{\lambda(w,\psi_F')}{\lambda(w, \psi_F)} 
= \left\{
\begin{aligned}
&\eta(a)^{k} \iif G = \O_{2n}^\eta(F), \\
&\eta'(a)^{kN_0+\half{k(k-1)}} \iif G=\U_n,\, n \equiv 0 \bmod 2,\\
&\eta'(a)^{kN_0+\half{k(k+1)}} \iif G=\U_n,\, n \equiv 1 \bmod 2,\\
&1 \other, 
\end{aligned}
\right. 
\]
where, if $G = \U_n$, we set $\eta'$ to be the quadratic character of $F^\times$
associated to $E/F$ by the class field theory.
If $G \not= \O_{2n}^\eta(F)$ (\resp $G \not= \U_n$), we simply set $\eta = \1$ (\resp $\eta'=\1$).
Since $\omega_{w\pi_\lambda}(z_0)^{-1} = \omega_\tau(a^{N_0+k+\delta})|a|_E^{(N_0+k+\delta)ks}$, 
we obtain 
\begin{align*}
\gamma^\Sh(s, \tau \times \sigma, \psi_E') \gamma^\Sh(2s, \tau, R, \psi_F')
&= 
\eta(a)^k\eta'(a)^{\half{k(k\pm1)}} \cdot \omega_\tau(a^{N_0+k+\delta})|a|_E^{(N_0+k+\delta)ks} \cdot |a|_F^{-\half{l}}
\\&\times
\gamma^\Sh(s, \tau \times \sigma, \psi_E) \gamma^\Sh(2s, \tau, R, \psi_F), 
\end{align*}
where the exponent of $\eta'(a)$ is $\half{k(k-1)}$ (\resp $\half{k(k+1)}$) 
for $G = \U_n$ with $n$ even (\resp $n$ odd). 
\par

If $G = \O_{2k}^\eta(F)$ with $\eta=1$, $M = \GL_k(F)$, and $k > 1$ is odd, 
then we take $w \in W(M^\circ, \epsilon M^\circ \epsilon^{-1})$ such that $\det(w) = 1$.
Since $t_0$ commutes with $\epsilon$, 
the above computation works 
after replacing $\omega_{w\pi_\lambda}(z_0)$ with $\omega_{w\pi_\lambda}(\epsilon z_0 \epsilon^{-1})$. 
Hence we obtain the same formula for 
$\gamma^\Sh(s, \tau \times \sigma, \psi_E') \gamma^\Sh(2s, \tau, R, \psi_F')$.
\par

In particular, if we take $(G,M) = (\SO_{2k+1}(F), \GL_{k}(F))$, 
$(\O_{2k}^\eta(F), \GL_{k}(F))$ with $\eta = 1$, and $(\U_{2k}, \GL_{k}(E))$, 
then we obtain
\[
\gamma^{\Sh}(2s, \tau, R, \psi'_F) 
= \eta'(a)^{\half{k(k-1)}} \omega_\tau(a)^{k+\delta(R)}
|a|_F^{(2s-\frac{1}{2})d(R)} \gamma^{\Sh}(2s, \tau, R, \psi_F), 
\]
where
\[
\delta(R) = \left\{
\begin{aligned}
&1 \iif R = \Sym^2, \\
&{-1} \iif R = \wedge^2, \\
&0 \iif R = \Asai^+, 
\end{aligned}
\right. 
\quad
d(R) = \left\{
\begin{aligned}
&\half{k(k+1)} \iif R = \Sym^2,\\
&\half{k(k-1)} \iif R = \wedge^2,\\
&k^2 \iif R = \Asai^+.
\end{aligned}
\right. 
\]
Moreover, in the last case, it follows from the definition (see \cite[p.~304]{Sh7}) that 
\[
\gamma^{\Sh}(2s, \tau, \Asai^-, \psi_F) 
= 
\gamma^{\Sh}(2s, \tau \otimes (\mu \circ \det), \Asai^+, \psi_F), 
\]
where $\mu$ is a character of $E^\times$ such that $\mu|_{F^\times} = \eta'$. 
Since $k^2 \equiv k \bmod 2$, we obtain that 
\[
\gamma^{\Sh}(2s, \tau, \Asai^-, \psi_F') 
= \eta'(a)^{\half{k(k+1)}} \omega_\tau(a)^{k}
|a|_F^{(2s-\frac{1}{2})k^2} \gamma^{\Sh}(2s, \tau, \Asai^-, \psi_F).
\]
This implies that 
\[
\gamma^{\Sh}(s, \tau \times \sigma, \psi'_E) 
= \eta(a)^{k} \omega_\tau(a)^{N_0} |a|_E^{kN_0(s-\frac{1}{2})} 
\gamma^{\Sh}(s, \tau \times \sigma, \psi_E),
\]
as desired.

\begin{rem}
In the above argument, we have used Shahidi's formula
\[
C_P(w,\pi_{s\tl{a}},\spl,\psi_F) = \lambda(w,\psi_F)^{-1} \prod_{i=1}^m \gamma^{\Sh}(is,\pi,r_i,\psi_F)
\]
for $s \in \C$ (see Section \ref{sec.shahidi}).
This formula implies that the local coefficient is independent of $\spl$ in the following sense.
We consider an arbitrary quasi-split connected reductive algebraic group $G$ over $F$.
Let $\spl = (B,T,\{X_\alpha\})$ and $\spl' = (B,T,\{X_\alpha'\})$ be two $F$-splittings of $G$ 
which have a common Borel pair.
We denote by $\ww = (B,\chi)$ and $\ww' = (B,\chi')$ the Whittaker data for $G$ 
determined by $\spl$ and $\spl'$, respectively, 
and a fixed additive character $\psi_F$.
Let $P = MN$ be a standard maximal parabolic subgroup of $G$, 
and let $\pi$ be an irreducible representation of $M(F)$ 
which is $\ww_M$-generic and $\ww'_M$-generic.
Then Shahidi's formula implies that 
\[
C_P(w,\pi_\lambda,\spl,\psi_F) = C_P(w,\pi_\lambda,\spl',\psi_F)
\]
for $\lambda \in \aa_{M,\C}^*$.
In fact, this equality can be proven directly as follows.
Let $G_{\ad} = G/Z$ be the adjoint group of $G$, 
where $Z$ is the center of $G$.
Choose an induced torus $\tl{Z}$ equipped with an embedding $Z \hookrightarrow \tl{Z}$.
For example, since $Z$ is a subgroup of the maximal torus $T$, 
we may take an induced torus $\tl{Z}$ such that 
there is a surjection $X^*(\tl{Z}) \twoheadrightarrow X^*(T)$ of $\Gamma$-modules.
We denote the pushout of the short exact sequence
\[
\begin{CD}
1 @>>> Z @>>> G @>>> G_{\ad} @>>> 1
\end{CD}
\]
via this embedding by 
\[
\begin{CD}
1 @>>> \tl{Z} @>>> \tl{G} @>>> G_{\ad} @>>> 1.
\end{CD}
\]
Note that if $G$ is a classical group and $\tl{Z} = \Res_{E/F} \mathbb{G}_m$, 
then $\tl{G}$ may be taken to be the corresponding similitude group.
For any subgroup $H$ of $G$ containing $Z$, 
we write $\tl{H}$ for the associated subgroup of $\tl{G}$.
Then $\tl{\spl} = (\tl{B},\tl{T},\{ X_\alpha \})$ is the $F$-splitting of $\tl{G}$ induced by $\spl$ 
and $\tl{\ww} = (\tl{B},\chi)$ is the Whittaker datum for $\tl{G}$ determined by $\tl{\spl}$ and $\psi_F$.
Choose an irreducible representation $\tl{\pi}$ of $\tl{M}(F)$ 
such that $\tl{\pi}|_{M(F)}$ contains $\pi$.
(To find such a $\tl{\pi}$, 
we first extend the central character of $\pi$ restricted to $Z(F)$ to a character of $\tl{Z}(F)$, 
and use this to extend $\pi$ to an irreducible representation $\pi^+$ of $\tl{Z}(F) M(F)$.
Note that $\tl{Z}(F) M(F)$ is a subgroup of $\tl{M}(F)$ of finite index since $F$ is of characteristic zero.
Then we can take $\tl{\pi}$ to be a suitable irreducible constituent of 
$\Ind^{\tl{M}(F)}_{\tl{Z}(F) M(F)}(\pi^+)$.)
Note that $\tl{\pi}$ is $\tl{\ww}_{\tl{M}}$-generic since $\pi$ is $\ww_M$-generic. 
Moreover, by the uniqueness of Whittaker functionals, 
$\tl{\pi}|_{M(F)}$ has a unique $\ww_M$-generic constituent, 
so that $\pi$ appears in $\tl{\pi}|_{M(F)}$ with multiplicity one.
In particular, for any non-trivial $\tl{\ww}_{\tl{M}}$-Whittaker functional on $\tl{\pi}$, 
its restriction to $\pi$ is also a non-trivial $\ww_M$-Whittaker functional on $\pi$.
Hence it follows from the definition that 
\[
C_{\tl{P}}(w,\tl{\pi}_\lambda,\tl{\spl},\psi_F) = C_P(w,\pi_\lambda,\spl,\psi_F),
\]
noting that 
\begin{itemize}
\item 
$w$ in the left-hand side is regarded as a Weyl element of $\tl{G}$;
\item 
the Weyl group representative of $w$ in $\tl{G}$ determined by $\tl{\spl}$ 
agrees with the one $\tl{w}_{\spl}$ in $G$ determined by $\spl$;
\item 
the Haar measures used to define intertwining operators 
and Jacquet integrals for $\tl{G}$ determined by $\tl{\spl}$ and $\psi_F$ 
agree with the ones $du_{\spl,\psi_F}$ for $G$ determined by $\spl$ and $\psi_F$;
\item 
the restriction of $\Ind^{\tl{G}}_{\tl{P}}(\tl{\pi}_\lambda)$ to $G$ as functions contains $I_P(\pi_\lambda)$.
\end{itemize}
Similarly, we have
\[
C_{\tl{P}}(w,\tl{\pi}_\lambda,\tl{\spl'},\psi_F) = C_P(w,\pi_\lambda,\spl',\psi_F).
\]
Thus, it remains to show that
\[
C_{\tl{P}}(w,\tl{\pi}_\lambda,\tl{\spl},\psi_F)
= C_{\tl{P}}(w,\tl{\pi}_\lambda,\tl{\spl'},\psi_F).
\]
Since the natural map $\tl{G}(F) \rightarrow G_{\ad}(F)$ is surjective, 
the set of $F$-splittings of $\tl{G}$ consists of a single $\tl{G}(F)$-conjugacy class, 
so that there is $\tl{t}_0 \in \tl{T}(F)$ such that $\tl{\spl'} = \Ad(\tl{t}_0)(\tl{\spl})$.
Then we have $\chi' = \chi \circ \Ad(\tl{t}_0)^{-1}$, $\tl{w}_{\spl'} = \Ad(\tl{t}_0)(\tl{w}_{\spl})$, 
and $du'_{\spl',\psi_F} = du_{\spl,\psi_F}$ for $u' = \Ad(\tl{t}_0)(u)$.
Put $\tl{\pi}' = \tl{\pi} \circ \Ad(\tl{t}_0)^{-1}$.
We may regard a $\tl{\ww}_{\tl{M}}$-Whittaker functional on $\tl{\pi}$ 
as a $\tl{\ww}'_{\tl{M}}$-Whittaker functional on $\tl{\pi}'$.
Then by definition, we have
\[
\Omega(\tl{\pi}'_\lambda,\tl{\spl'},\psi_F)
= \Omega(\tl{\pi}_\lambda,\tl{\spl},\psi_F) \circ \Ad(\tl{t}_0)^*,
\]
where 
$\Ad(\tl{t}_0)^* \colon 
\Ind^{\tl{G}}_{\tl{P}}(\tl{\pi}'_\lambda) \rightarrow \Ind^{\tl{G}}_{\tl{P}}(\tl{\pi}_\lambda)$ 
is the linear isomorphism given by $\Ad(\tl{t}_0)^*f(g) = f(\Ad(\tl{t}_0)(g))$.
Similarly, we have
\[
J_{\tl{P}}(w,\tl{\pi}_\lambda,\tl{\spl},\psi_F) \circ \Ad(\tl{t}_0)^*
= \Ad(\tl{t}_0)^* \circ J_{\tl{P}}(w,\tl{\pi}'_\lambda,\tl{\spl'},\psi_F).
\]
Hence we have
\begin{align*}
\Omega(\tl{\pi}'_\lambda,\tl{\spl'},\psi_F)
& = \Omega(\tl{\pi}_\lambda,\tl{\spl},\psi_F) \circ \Ad(\tl{t}_0)^* \\
& = C_{\tl{P}}(w,\tl{\pi}_\lambda,\tl{\spl},\psi_F) \cdot 
\Omega(w \tl{\pi}_\lambda,\tl{\spl},\psi_F) \circ J_{\tl{P}}(w,\tl{\pi}_\lambda,\tl{\spl},\psi_F) \circ \Ad(\tl{t}_0)^* \\
& = C_{\tl{P}}(w,\tl{\pi}_\lambda,\tl{\spl},\psi_F) \cdot 
\Omega(w \tl{\pi}_\lambda,\tl{\spl},\psi_F) \circ \Ad(\tl{t}_0)^* \circ J_{\tl{P}}(w,\tl{\pi}'_\lambda,\tl{\spl'},\psi_F) \\
& = C_{\tl{P}}(w,\tl{\pi}_\lambda,\tl{\spl},\psi_F) \cdot 
\Omega(w \tl{\pi}'_\lambda,\tl{\spl'},\psi_F) \circ J_{\tl{P}}(w,\tl{\pi}'_\lambda,\tl{\spl'},\psi_F).
\end{align*}
Since $\tl{\pi}' \cong \tl{\pi}$, this implies the desired equality.
 \end{rem}

\subsection{Proof of Proposition \ref{multiplicative}}\label{sec.ell}
Here, we give a proof of Proposition \ref{multiplicative} for classical groups in general.
Recall the setting. 
Let $G$ be a quasi-split classical group. 
For $i=1,2,3$, 
consider a standard parabolic subgroup $P_i = M_iN_i$ of $G$. 
Assume that $W(M_1^\circ, M_2^\circ) \not= \emptyset$ and $W(M_2^\circ, M_3^\circ) \not= \emptyset$.
Then for $w_1 \in W(M_1^\circ, M_2^\circ)$ and $w_2 \in W(M_2^\circ, M_3^\circ)$, 
and for an irreducible tempered representation $\pi$ of $M_1$, 
Proposition \ref{multiplicative} asserts that 
\[
R_{P_1}(w_2w_1, \pi_\lambda) = R_{P_2}(w_2, w_1 \pi_\lambda) \circ R_{P_1}(w_1, \pi_\lambda).
\]
\par

Set $w_1^{-1}P_2 = \tl{w}_1^{-1} P_2 \tl{w}_1$.
Following \cite[Section 2.3]{Ar}, 
we decompose the intertwining operator $R_{P_1}(w_1, \pi_\lambda) \colon I_{P_1}(\pi_\lambda) \rightarrow I_{P_2}(w_1\pi_\lambda)$ as
\[
R_{P_1}(w_1, \pi_\lambda) = \ell(w_1, \pi_\lambda) \circ R_{w_1^{-1}P_2|P_1}(\pi_\lambda), 
\]
where $R_{w_1^{-1}P_2|P_1}(\pi_\lambda) \colon I_{P_1}(\pi_\lambda) \rightarrow I_{w_1^{-1}P_2}(\pi_\lambda)$ 
is given by (the meromorphic continuation of) the integral
\[
(R_{w_1^{-1}P_2|P_1}(\pi_\lambda)f_\lambda)(g) = 
\frac{\gamma_A(0,\pi_\lambda, \rho_{w_1^{-1}P_2|P_1}^\vee, \psi_F)}
{\ep(1/2, \pi_\lambda, \rho_{w_1^{-1}P_2|P_1}^\vee, \psi_F)}
\int_{(N_1 \cap \tl{w}_1^{-1} N_2 \tl{w}_1) \bs \tl{w}_1^{-1} N_2 \tl{w}_1} f_\lambda(ug) du
\]
and $\ell(w_1, \pi_\lambda) \colon I_{w_1^{-1}P_2}(\pi_\lambda) \rightarrow I_{P_2}(w_1\pi_\lambda)$ is defined by 
\[
\ell(w_1, \pi_\lambda) = 
\lambda(w_1)^{-1} \ep(1/2, \pi_\lambda, \rho_{w_1^{-1}P_2|P_1}^\vee, \psi_F) L(\tl{w}_1)
\]
with $L(\tl{w}_1)f'_\lambda(g) = f'_\lambda(\tl{w}_1^{-1}g)$.
The key property of $\ell(w_1,\pi_\lambda)$ is as follows. 

\begin{lem}\label{cocycle}
The operator $\ell(w_1,\pi_\lambda)$ satisfies the condition 
\[
\ell(w_2w_1,\pi_\lambda) = \ell(w_2,w_1\pi_\lambda) \circ \ell(w_1,\pi_\lambda).
\]
\end{lem}

Before showing this lemma, we prove Proposition \ref{multiplicative}.
\begin{proof}[Proof of Proposition \ref{multiplicative}]
By \cite[Proposition 2.3.1]{Ar}, \cite[Proposition 3.3.1]{Mok} and Lemma \ref{cocycle}, 
we have 
\begin{align*}
&R_{P_2}(w_2, w_1 \pi_\lambda) \circ R_{P_1}(w_1, \pi_\lambda)
\\&= 
\left(\ell(w_2, w_1\pi_\lambda) \circ R_{w_2^{-1}P_3|P_2}(w_1\pi_\lambda)\right) 
\circ \left(\ell(w_1, \pi_\lambda) \circ R_{w_1^{-1}P_2|P_1}(\pi_\lambda) \right)
\\&= 
\ell(w_2, w_1\pi_\lambda) \circ \ell(w_1, \pi_\lambda)
\circ R_{w_1^{-1}w_2^{-1}P_3|w_1^{-1}P_2}(\pi_\lambda) \circ R_{w_1^{-1}P_2|P_1}(\pi_\lambda)
\\&= \ell(w_2w_1,\pi_\lambda) \circ R_{(w_2w_1)^{-1}P_3|P_1}(\pi_\lambda)
= R_{P_1}(w_2w_1, \pi_\lambda).
\end{align*}
This completes the proof of Proposition \ref{multiplicative}. 
\end{proof}

Therefore, Lemma \ref{cocycle} is the missing part for Proposition \ref{multiplicative}. 
For the proof of Lemma \ref{cocycle}, 
we need the following elementary fact. 
\begin{lem}\label{f(beta)}
Let $V$ be a finite dimensional real vector space. 
For $i=1,2,3$, 
let $P_i \subset V$ be a finite subset such that the union $R_i = P_i \cup -P_i$ is a disjoint union. 
For $\beta_i \in R_i$, write $\beta_i > 0$ (\resp $\beta_i < 0$) if $\beta_i \in P_i$ (\resp $-\beta_i \in P_i$).
Fix two automorphisms $w_1$ and $w_2$ of $V$ such that $w_1(R_1) = R_2$ and $w_2(R_2) = R_3$.
Then for a function $f \colon R_1 \rightarrow \C^\times$, 
we have
\[
\prod_{\substack{\beta_1 > 0 \\ w_1\beta_1 < 0}} f(\beta_1) 
\cdot 
\prod_{\substack{w_1\beta_1 > 0 \\ w_2w_1\beta_1 < 0}} f(\beta_1) 
\cdot 
\prod_{\substack{\beta_1 > 0 \\ w_2w_1\beta_1 < 0}} f(\beta_1)^{-1}
= 
\prod_{\substack{\beta_1 > 0 \\ w_1\beta_1 < 0, \, w_2w_1 \beta_1 > 0}} f(\beta_1)f(-\beta_1).
\]
\end{lem}
\begin{proof}
We write $R_1 = \sqcup_{k=1}^8 I_k$, 
where
\begin{align*}
I_1 &= \{\beta_1 \in R_1 \,|\, \beta_1 > 0, \, w_1\beta_1 > 0, \, w_2w_1\beta_1 > 0\}, \\
I_2 &= \{\beta_1 \in R_1 \,|\, \beta_1 > 0, \, w_1\beta_1 > 0, \, w_2w_1\beta_1 < 0\}, \\
I_3 &= \{\beta_1 \in R_1 \,|\, \beta_1 > 0, \, w_1\beta_1 < 0, \, w_2w_1\beta_1 > 0\}, \\
I_4 &= \{\beta_1 \in R_1 \,|\, \beta_1 > 0, \, w_1\beta_1 < 0, \, w_2w_1\beta_1 < 0\}, \\
I_5 &= \{\beta_1 \in R_1 \,|\, \beta_1 < 0, \, w_1\beta_1 > 0, \, w_2w_1\beta_1 > 0\}, \\
I_6 &= \{\beta_1 \in R_1 \,|\, \beta_1 < 0, \, w_1\beta_1 > 0, \, w_2w_1\beta_1 < 0\}, \\
I_7 &= \{\beta_1 \in R_1 \,|\, \beta_1 < 0, \, w_1\beta_1 < 0, \, w_2w_1\beta_1 > 0\}, \\
I_8 &= \{\beta_1 \in R_1 \,|\, \beta_1 < 0, \, w_1\beta_1 < 0, \, w_2w_1\beta_1 < 0\}.
\end{align*}
Then 
\begin{align*}
\prod_{\substack{\beta_1 > 0 \\ w_1\beta_1 < 0}} f(\beta_1) = \prod_{\beta_1 \in I_3 \sqcup I_4} f(\beta_1), 
&\quad
\prod_{\substack{w_1\beta_1 > 0 \\ w_2w_1\beta_1 < 0}} f(\beta_1) = \prod_{\beta_1 \in I_2 \sqcup I_6} f(\beta_1), \\
\prod_{\substack{\beta_1 > 0 \\ w_2w_1\beta_1 < 0}} f(\beta_1) = \prod_{\beta_1 \in I_2 \sqcup I_4} f(\beta_1),
&\quad
\prod_{\substack{\beta_1 > 0 \\ w_1\beta_1 < 0, \, w_2w_1 \beta_1 > 0}} f(\beta_1)f(-\beta_1)
= \prod_{\beta_1 \in I_3 \sqcup I_6} f(\beta_1).
\end{align*}
This implies the lemma.
\end{proof}

Now we prove Lemma \ref{cocycle}.
\begin{proof}[Proof of Lemma \ref{cocycle}]
When $G=G^\circ$ and $M_1 = M_2 = M_3$, the lemma is \cite[Lemma 2.3.4]{Ar} and \cite[Lemma 3.3.4]{Mok}.
The proof of the general case is essentially the same as these lemmas. 
\par

Since $\tl{w}_2 \tl{w}_1$ and $\tl{w_2w_1}$ are two representatives of $w_2w_1 \in W(M_1^\circ, M_3^\circ)$, 
we can write 
\[
\tl{w}_2 \tl{w}_1 = \tl{w_2w_1} \cdot z(w_2,w_1)
\]
for some $z(w_2,w_1) \in M_1^\circ$.
By Lemma \ref{center}, we see that 
$z(w_2,w_1)$ is in the center of $M_1$ so that
\[
L(\tl{w_2w_1}) = \omega_{\pi_\lambda}(z(w_2,w_1)) L(\tl{w}_2) \circ L(\tl{w}_1),
\]
where $\omega_{\pi_\lambda}$ is the central character of $\pi_\lambda$.
Our goal is to show that $\omega_{\pi_\lambda}(z(w_2,w_1))$ is equal to the product of 
\[
\ep(1/2, \pi_\lambda, \rho_{w_1^{-1}P_2|P_1}^\vee, \psi_F)
\ep(1/2, w_1\pi_\lambda, \rho_{w_2^{-1}P_3|P_2}^\vee, \psi_F)
\ep(1/2, \pi_\lambda, \rho_{(w_2w_1)^{-1}P_3|P_1}^\vee, \psi_F)^{-1}
\]
and
\[
\lambda(w_1)^{-1}\lambda(w_2)^{-1}\lambda(w_2w_1).
\]
\par

First, we consider the central character.
Let $R(T^\circ,G^\circ)$ be the set of roots of $T^\circ$ in $G^\circ$. 
As in the proofs of \cite[Lemma 2.3.4]{Ar} and \cite[Lemma 3.3.4]{Mok}, 
by \cite[Lemma 2.1.A]{LS}, we have
\[
z(w_2,w_1) = (-1)^{\lambda^\vee(w_2,w_1)}, 
\]
where $\lambda^\vee(w_2,w_1) = \sum_{\alpha \in R_B(w_2,w_1)} \alpha^\vee$ with 
\[
R_B(w_2,w_1) = \{\alpha \in R(T^\circ,G^\circ)\,|\, \alpha > 0,\, w_1\alpha < 0 ,\, w_2w_1\alpha > 0 \}.
\]
Moreover, $z(w_2,w_1)$ is in the split part $A_{M_1^\circ}$ of the center of $M_1^\circ$, 
i.e., $\lambda^\vee(w_2,w_1) \in X_*(A_{M_1^\circ})$. 
\par

For $i=1,2,3$, we set $R_i = R(A_{\widehat{M_i^\circ}}, \widehat{G^\circ})$. 
For $\beta_i \in R_i$, we denote by $\beta_i > 0$ if the weight space $\widehat\g_{\beta_i}$ lies in $\widehat\nn_i$.
Via the isomorphism $X_*(A_{M_1^\circ}) \cong X^*(\widehat{M_1^\circ})_F$ in \cite[Section 2.3]{Ar}, 
we can identify $\lambda^\vee \in X_*(A_{M^\circ})$ with a character of ${}^LM_1^\circ$
which is trivial on the semi-direct factor $W_F$ of ${}^LM_1^\circ$.
Then by the description of central characters in terms of the LLC for general linear groups, 
we have	
\[
\omega_{\pi_\lambda}(z^{\lambda^\vee}) 
= \lambda^\vee(\phi_\lambda(u_z)) = \lambda^\vee(m(u_z))
\]
for $\lambda^\vee \in X_*(A_{M_1^\circ}) \cong X^*(\widehat{M_1^\circ})_F$
and $u_z \in W_F$ whose image in $W_F^\ab \cong F^\times$ is equal to $z \in F^\times$, 
where $\phi_\lambda$ is the $L$-parameter for $\pi_\lambda$,
and we write $\phi_{\lambda}(u_z) = m(u_z) \rtimes u_z$ with $m(u_z) \in \widehat{M_1^\circ}$.
If $\lambda^\vee = \lambda^\vee(w_2,w_1)$, then
we can write 
\[
\lambda^\vee(w_2,w_1) 
= \sum_{\substack{\beta_1 > 0 \\ w_1\beta_1 < 0, \, w_2w_1 \beta_1 > 0}} \lambda_{\beta_1}^\vee,
\]
where $\beta_1$ runs over the elements of $R_1$ satisfying the specified conditions, 
and 
\[
\lambda_{\beta_1}^\vee = \sum_{\alpha \in R_{\beta_1}} \alpha^\vee
\]
with $R_{\beta_1}$ being the set of roots $\alpha$ of $T^\circ$ in $G^\circ$
such that $\alpha^\vee|_{A_{M_1^\circ}}$ is a positive multiple of $\beta$. 
Hence, fixing an arbitrary element $u \in W_F$ whose image in $F^\times \cong W_F^\ab$ is equal to $-1$, 
we have 
\[
\omega_{\pi_\lambda}(z(w_2,w_1))
= \prod_{\substack{\beta_1 > 0 \\ w_1\beta_1 < 0, \, w_2w_1 \beta_1 > 0}} \lambda_{\beta_1}^\vee(m(u)).
\]
\par

Next, we consider the $\ep$-factors.
For simplicity, we write $\ep(\pi_\lambda, \rho)$ for $\ep(1/2, \pi_\lambda, \rho, \psi_F)$. 
The adjoint representation $\rho_{w_1^{-1}P_2|P_1}$ of ${}^LM_1^\circ$ on 
\[
\tl{w}_1^{-1} \widehat\nn_2 \tl{w}_1 / (\tl{w}_1^{-1} \widehat\nn_2 \tl{w}_1 \cap \widehat\nn_1)
\cong \tl{w}_1^{-1} \widehat\nn_2 \tl{w}_1 \cap \widehat{\overline\nn}_1
\]
decomposes as the direct sum
\[
\bigoplus_{\substack{\beta_1 < 0 \\ w_1\beta_1 > 0}} \widehat\g_{\beta_1}, 
\]
where $\beta_1$ runs over the elements of $R_1$ satisfying the specified conditions.
If we denote the adjoint action of ${}^LM_1^\circ$ on $\widehat\g_{\beta_1}$ by $\rho_{\beta_1}$, 
then we obtain that 
\[
\ep(\pi_\lambda, \rho_{w_1^{-1}P_2|P_1}^\vee)
= \prod_{\substack{\beta_1 > 0 \\ w_1\beta_1 < 0}}\ep(\pi_\lambda, \rho_{\beta_1}).
\]
In particular, by applying Lemma \ref{f(beta)} to the function $f(\beta_1) = \ep(\pi_\lambda, \rho_{\beta_1})$
together with \cite[(3.6.8)]{Tate}, 
we see that 
\begin{align*}
&\ep(\pi_\lambda, \rho_{w_1^{-1}P_2|P_1}^\vee)
\ep(w_1\pi_\lambda, \rho_{w_2^{-1}P_3|P_2}^\vee)
\ep(\pi_\lambda, \rho_{(w_2w_1)^{-1}P_3|P_1}^\vee)^{-1}
\\&= \prod_{\substack{\beta_1 > 0 \\ w_1\beta_1 < 0}} f(\beta_1) 
\cdot 
\prod_{\substack{w_1\beta_1 > 0 \\ w_2w_1\beta_1 < 0}} f(\beta_1) 
\cdot 
\prod_{\substack{\beta_1 > 0 \\ w_2w_1\beta_1 < 0}} f(\beta_1)^{-1}
\\&= 
\prod_{\substack{\beta_1 > 0 \\ w_1\beta_1 < 0, \, w_2w_1 \beta_1 > 0}} f(\beta_1)f(-\beta_1)
\\&= 
\prod_{\substack{\beta_1 > 0 \\ w_1\beta_1 < 0, \, w_2w_1 \beta_1 > 0}} 
\ep(\pi_\lambda, \rho_{\beta_1})\ep(\pi_\lambda, \rho_{\beta_1}^\vee)
\\&=
\prod_{\substack{\beta_1 > 0 \\ w_1\beta_1 < 0, \, w_2w_1 \beta_1 > 0}} 
\det(\rho_{\beta_1} \circ \phi_\lambda(u)).
\end{align*}
Writing $\phi_\lambda(u) = m(u) \rtimes u$ with $m(u) \in \widehat{M_1^\circ}$, 
this product is equal to 
\[
\det\left(
\Ad(m(u) \rtimes u);
\bigoplus_{\substack{\beta_1 > 0 \\ w_1\beta_1 < 0, \, w_2w_1 \beta_1 > 0}} \widehat\g_{\beta_1}
\right).
\]
Since $1 \rtimes u$ acts on $A_{\widehat{M_1^\circ}}$ trivially, 
the direct sum is stable under the adjoint action of $1 \rtimes u$, 
and we may compute this determinant as 
the product of the determinants of $\Ad(m(u))$ and $\Ad(1 \rtimes u)$.
For the determinant of $\Ad(m(u))$, 
we may assume that $m(u) \in \widehat{M_1^\circ}$ is in $\widehat{T^\circ}$. 
Then we conclude that 
\[
\det\left(
\Ad(m(u));
\bigoplus_{\substack{\beta_1 > 0 \\ w_1\beta_1 < 0, \, w_2w_1 \beta_1 > 0}} \widehat\g_{\beta_1}
\right)
= \prod_{\substack{\beta_1 > 0 \\ w_1\beta_1 < 0, \, w_2w_1 \beta_1 > 0}} \lambda^\vee_{\beta_1}(m(u)), 
\]
which is equal to $\omega_{\pi_\lambda}(z(w_2,w_1))$ as we have seen above.
\par

Finally, we show that 
\[
\det\left(
\Ad(1 \rtimes u);
\bigoplus_{\substack{\beta_1 > 0 \\ w_1\beta_1 < 0, \, w_2w_1 \beta_1 > 0}} \widehat\g_{\beta_1}
\right)
= \lambda(w_1)\lambda(w_2)\lambda(w_2w_1)^{-1}.
\]
Recall that if we set $\Delta_1$ (\resp $\Delta_2$) 
to be the set of reduced roots $a \in R(A_{T^\circ}, G^\circ)$ 
with $G_{a,\sc} \cong \Res_{F_a/F} \SL_2$ (\resp $\Res_{F_a/F} \SU_{E_a/F_a}(2,1)$), 
and if we define $f_1(a) = \lambda(F_a/F, \psi_F)$ (\resp $f_2(a) = \lambda(E_a/F, \psi_F)^2 \lambda(F_a/F, \psi_F)^{-1}$),
then 
\begin{align*}
\lambda(w_1)\lambda(w_2)\lambda(w_2w_1)^{-1}
&= \prod_{\substack{a \in \Delta_1 \\ a>0, w_1a < 0}} f_1(a)
\cdot \prod_{\substack{a \in \Delta_1 \\ w_1a>0, w_2w_1a < 0}} f_1(w_1a)
\cdot \prod_{\substack{a \in \Delta_1 \\ a>0, w_2w_1a < 0}} f_1(a)^{-1}
\\&\times
\prod_{\substack{a \in \Delta_2 \\ a>0, w_1a < 0}} f_2(a)
\cdot \prod_{\substack{a \in \Delta_2 \\ w_1a>0, w_2w_1a < 0}} f_2(w_1a)
\cdot \prod_{\substack{a \in \Delta_2 \\ a>0, w_2w_1a < 0}} f_2(a)^{-1}.
\end{align*}
Here, we notice that 
$w_1$ induces a bijection on $\Delta_i$, and $f_i(w_1a) = f_i(a)$ for $i=1,2$.
Applying Lemma \ref{f(beta)} twice, we have 
\begin{align*}
\lambda(w_1)\lambda(w_2)\lambda(w_2w_1)^{-1}
&= \prod_{\substack{a \in \Delta_1 \\ a>0, w_1a < 0, w_2w_1a >0}} f_1(a)^2 
\cdot \prod_{\substack{a \in \Delta_2 \\ a>0, w_1a < 0, w_2w_1a >0}} f_2(a)^2
\end{align*}
since $f_i(-a)=f_i(a)$. 
Therefore, what we need to show is that 
$\det(\Ad(1 \rtimes u); \widehat\g_{(a^\vee)}) = f_i(a)^2$ if $a \in \Delta_i$ for $i = 1, 2$, 
where we set $\widehat\g_{(a^\vee)} = \widehat\g_{a^\vee} \oplus \widehat\g_{2a^\vee}$
and 
\[
\widehat\g_{a^\vee} 
= \bigoplus_{\substack{\alpha \in R(T^\circ,G^\circ) \\ \alpha|_{A_T^\circ} = a}} \widehat\g_{\alpha^\vee}.
\]
When $a \in \Delta_1$, 
the $W_F$-module $\widehat\g_{(a^\vee)}$ is isomorphic to $\Ind_{W_{F_a}}^{W_F}(\1_{F_a^\times})$
and hence 
\[
\det(\Ad(1 \rtimes u); \widehat\g_{(a^\vee)}) 
= \det\left(\Ind_{W_{F_a}}^{W_F}(\1_{F_a^\times})\right)(-1) = \lambda(F_a/F, \psi_F)^2.
\]
If $a \in \Delta_2$, then 
the $W_F$-module $\widehat\g_{(a^\vee)}$ is isomorphic to $\Ind_{W_{F_a}}^{W_F}(\widehat\nn_{\SU(2,1)})$, 
where $\widehat\nn_{\SU(2,1)}$ is the space of the sum of positive root spaces 
in the Lie algebra of the dual group of $\SU_{E_a/F_a}(2,1)$.
It is isomorphic to 
\[
\Ind_{W_{E_a}}^{W_{F_a}} (\1_{E_a^\times}) \oplus \eta_{E_a/F_a} 
\cong \1_{F_a^\times} \oplus \eta_{E_a/F_a} \oplus \eta_{E_a/F_a}, 
\]
where $\eta_{E_a/F_a}$ is the quadratic character of $F_a^\times$
associated to $E_a/F_a$ by the local class field theory. 
Hence
\begin{align*}
\det(\Ad(1 \rtimes u); \widehat\g_{(a^\vee)}) 
&= \det\left(\Ind_{W_{F_a}}^{W_F}(\1_{F_a^\times})\right)(-1) 
\cdot \det\left(\Ind_{W_{F_a}}^{W_F}(\eta_{E_a/F_a})\right)(-1)^2
\\&= \det\left(\Ind_{W_{F_a}}^{W_F}(\1_{F_a^\times})\right)(-1) 
= \lambda(F_a/F, \psi_F)^2.
\end{align*}
Since $\lambda(E_a/F, \psi_F)^4 = 1$, 
we have $f_2(a)^2 = \lambda(E_a/F, \psi_F)^4 \lambda(F_a/F, \psi_F)^{-2} = \lambda(F_a/F, \psi_F)^2$.
This completes the proof of Lemma \ref{cocycle}.
\end{proof}

This result, along with other results on the LIR,
will be extended to general disconnected groups in a forthcoming paper.

\section{Review of Aubert duality}\label{sec.AD}
The purpose of this appendix is to review several properties of Aubert duality, 
which we used in the main body. 
In particular, we prove the commutativity of 
the normalized intertwining operators with the Aubert involution up to scalars. 
It is a crucial result to prove 
the local intertwining relations for co-tempered $A$-packets. 
This appendix is an adaptation of an appendix in an arXiv version of \cite{KMSW}. 
\par

We also review the definition of the twisted Aubert dual 
and some of its properties that we need in this paper, following \cite{CK26}.

\subsection{Definition of a complex}\label{sec.complex}
Let $F$ be a non-archimedean local field 
and let $G$ be a connected reductive group over $F$. 
We identify $G$ with the group of $F$-points $G(F)$. 
We denote by $\Rep(G)$ the category of smooth representations of $G$ of finite length.
\par

For a parabolic subgroup $P = MN_P$ of $G$, 
where $M$ is a Levi component of $P$ and $N_P$ is the unipotent radical of $P$, 
we have the normalized parabolic induction functor
\[
\Ind_P^G \colon \Rep(M) \rightarrow \Rep(G), \, (\sigma, V_\sigma) \mapsto (\pi, V_\pi),
\]
where $V_\pi$ is the space of locally constant functions $f \colon G \rightarrow V_\sigma$
such that 
\[
f(nmg) = \delta_P(m)^{\half{1}}\sigma(m)f(g)
\]
for $n \in N_P$, $m \in M$ and $g \in G$ with $\delta_P$ the modulus character of $P$, 
and $(\pi(x)f)(g) = f(gx)$ for $x,g \in G$. 
We also have the normalized Jacquet functor 
\[
\Jac_P \colon \Rep(G) \rightarrow \Rep(M),\, (\pi, V_\pi) \mapsto (\sigma, V_\sigma), 
\]
where 
\[
V_\sigma = (V_\pi)_{N_P} = V_\pi/ \pair{\pi(n)v-v \,|\, n \in N_P,\, v \in V_\pi}, 
\]
and $\sigma(m) \overline{v} = \delta_P(m)^{-\half{1}} \overline{\pi(m)v}$ for $m \in M$
and $v \in V_\pi$ with the image $\overline{v} \in V_\sigma$.
These functors are adjoint, i.e., there is an isomorphism
\[
\Hom_G(\pi, \Ind_P^G(\sigma)) \cong \Hom_M(\Jac_P(\pi), \sigma)
\]
for $\pi \in \Rep(G)$ and $\sigma \in \Rep(M)$. 
\par

For $(\pi, V_\pi) \in \Rep(G)$ and for a parabolic subgroup $P=MN_P$ of $G$, 
we set $X_P(\pi) = \Ind_P^G(\Jac_P(\pi))$. 
This is the space of locally constant functions $\overline{f} \colon G \rightarrow (V_\pi)_{N_P}$
such that 
\[
\overline{f(nmg)} = \overline{\pi(m)f(g)}
\]
for $n \in N_P$, $m \in M$ and $g \in G$ 
with $f(g) \in V_\pi$ a representative of $\overline{f(g)} \in (V_\pi)_{N_P}$.
If $Q$ is another parabolic subgroup of $G$ with $Q \supset P$, 
since $N_Q \subset N_P$, 
we have a projection map $(V_\pi)_{N_Q} \twoheadrightarrow (V_\pi)_{N_P}$. 
We define a map $\varphi_P^Q \colon X_Q(\pi) \rightarrow X_P(\pi)$ 
by the composition of functions $G \rightarrow (V_\pi)_{N_Q}$ 
with this projection $(V_\pi)_{N_Q} \twoheadrightarrow (V_\pi)_{N_P}$. 
\par

Fix a minimal parabolic subgroup $P_0 = M_0N_{P_0}$ of $G$. 
We denote the maximal split central torus in a Levi subgroup $M$ by $A_M$, 
and set $A_0 = A_{M_0}$. 
Let $S \subset X^*(A_0)$ be the set of (relative) simple roots corresponding to $P_0$, 
and set $r = |S| = \dim(A_0/A_G)$.
We say that a parabolic subgroup $P$ of $G$ is \emph{standard} if $P \supset P_0$. 
Then there is a bijection 
\[
\{J \subset S\} \rightarrow 
\{\text{standard parabolic subgroups of $G$}\},\, 
J \mapsto P_J, 
\]
where $P_J = M_J N_{P_J}$ is such that $\Lie(P_J)$ is 
the sum of $A_0$-weight spaces for all weights that are $\Z$-linear combinations of $S$ 
with non-negative contributions of $S \setminus J$.
We write $X_J(\pi) = X_{P_J}(\pi)$ for short. 
Note that if $I \subset J$, then $P_I \subset P_J$. 
Hence we have a map
$\varphi_I^J = \varphi_{P_I}^{P_J} \colon X_J(\pi) \rightarrow X_I(\pi)$ for $\pi \in \Rep(G)$. 
\par

For $J \subset S$, we consider the $1$-dimensional vector space
\[
\Lambda_J = \bigwedge^{|S \setminus J|} (\C^{|S \setminus J|}),
\]
which is regarded as the trivial representation of $G$. 
Let $\{e_i\}_{i \in S \setminus J}$ be the standard basis of $\C^{|S \setminus J|}$.
For $I \subset J \subset S$ with $|J \setminus I| = 1$, 
letting $J \setminus I = \{j\}$, 
we define the isomorphism 
\[
\epsilon_I^J \colon \Lambda_J \rightarrow \Lambda_I,\, \omega \mapsto \omega \wedge e_j. 
\]
Consider a functor 
\[
\tl{X}_J \colon \Rep(G) \rightarrow \Rep(G)
\] 
given by 
$\tl{X}_J(\pi) = X_J(\pi) \otimes_\C \Lambda_J$ for $\pi \in \Rep(G)$. 
Then we define
\[
\tl\varphi_I^J \colon \tl{X}_J(\pi) \rightarrow \tl{X}_I(\pi)
\]
by $\tl\varphi_I^J = \varphi_I^J \otimes \epsilon_I^J$.
\par

For $\pi \in \Rep(G)$ and for $0 \leq t \leq r$, set 
\[
\tl{X}_t(\pi) = \bigoplus_{\substack{J \subset S \\ |J| = t}} \tl{X}_J(\pi).
\]
In particular, $\tl{X}_r(\pi) = \pi$. 
For $1 \leq t \leq r$, define $d_t \colon \tl{X}_t(\pi) \rightarrow \tl{X}_{t-1}(\pi)$ by 
\[
d_t\left( \sum_{\substack{J \subset S \\ |J| = t}} x_J \right) 
= \sum_{\substack{I \subset S \\ |I| = t-1}} 
\sum_{\substack{I \subset J \subset S \\ |J| = t}} \tl{\varphi}_I^J(x_J),
\]
where $x_J \in \tl{X}_J(\pi)$.
Then we have a sequence 
\[
\begin{CD}
0 @>>> \tl{X}_r(\pi) @>>> \tl{X}_{r-1}(\pi) @>>> \cdots @>>> \tl{X}_0(\pi)
\end{CD}
\]
of representations of $G$. 

\subsection{The Aubert involution}\label{sec.aubert}
For $0 \leq t \leq r$, we denote by $\Rep(G)_t$ the full subcategory of $\Rep(G)$
consisting of representations $\pi$ such that for every irreducible subquotient $\pi'$ of $\pi$, 
there is $J \subset S$ with $|J| = t$ such that $\pi'$ is a subquotient of $\Ind_{P_J}^G(\sigma)$
for an irreducible supercuspidal representation of $M_J$. 
Bernstein's decomposition implies the block decomposition
\[
\Rep(G) = \prod_{t= 0}^r \Rep(G)_t. 
\]
Note that the factor $\Rep(G)_r$ consists of direct sums of supercuspidal representations of $G$, 
and every supercuspidal representation of $G$ lies in $\Rep(G)_r$.

\begin{thm}[{\cite[Th\'eor\`eme 3.6]{Au}}]\label{thm1.6}
For $\pi \in \Rep(G)_t$, 
we have $\tl{X}_0(\pi) = \dots = \tl{X}_{t-1}(\pi) = 0$. 
Moreover, the sequence 
\[
\begin{CD}
0 @>>> \tl{X}_r(\pi) @>>> \tl{X}_{r-1}(\pi) @>>> \cdots @>>> \tl{X}_t(\pi)
\end{CD}
\]
is exact. 
\end{thm}

\begin{defi}
For $\pi \in \Rep(G)_t$, set 
\[
\hat\pi = \tl{X}_t(\pi)/d_{t+1}(\tl{X}_{t+1}(\pi))
\]
and call $\hat\pi$ the \emph{Aubert dual} of $\pi$.
\end{defi}

\begin{thm}\label{AD}
Aubert duality $\pi \mapsto \hat\pi$ satisfies the following properties. 
\begin{enumerate}
\item
The map $\Rep(G) \ni \pi \rightarrow \hat\pi \in \Rep(G)$ is an exact covariant functor. 

\item
For $\pi \in \Rep(G)_t$, we have 
\[
[\hat\pi] = (-1)^{r-t} \sum_{P = MN_P} (-1)^{\dim(A_M/A_G)} [\Ind_P^G(\Jac_P(\pi))]
\]
in the Grothendieck group $\RR(G)$, 
where $P$ runs over the set of standard parabolic subgroups of $G$. 
Here, $[\Pi]$ denotes the element in $\RR(G)$ corresponding to 
a representation $\Pi$ of $G$ of finite length. 

\item
If $\pi$ is irreducible, then $\hat\pi$ is also irreducible. 

\item
The Aubert dual of $\hat\pi$ is isomorphic to $\pi$ as representations of $G$.

\item
Let $P$ be a parabolic subgroup of $G$ with Levi component $M$, 
and denote by $\overline{P}$ the parabolic subgroup of $G$ opposite to $P$. 
Then for $\pi \in \Rep(M)$ of finite length, 
the Aubert dual of $\Ind_P^G(\pi)$ is isomorphic to $\Ind_{\overline{P}}^G(\hat\pi)$. 

\end{enumerate}
\end{thm}
\begin{proof}
For (1), see \cite[III.3]{SS}.
The assertions (2), (3) and (4) are \cite[Corollaire 3.9]{Au}. 
Finally, (5) is \cite[Theorem 31 (4)]{Ber}. 
\end{proof}

\subsection{Intertwining operators and Aubert duality}
In this subsection, 
let $G$ be one of the following quasi-split classical groups
\[
\SO_{2n+1}(F), \quad \Sp_{2n}(F), \quad
\O_{2n}(F), \quad \U_n.
\]
Let $P = MN$ be a maximal parabolic subgroup of $G$
so that $M \cong \GL_k(E) \times G_0$ for some classical group $G_0$ of the same type as $G$.
We denote by $\overline{P} = M\overline{N}$ the parabolic subgroup of $G$ opposite to $P$.
We fix $\psi \in \Psi(M)$. 

In this subsection, 
we consider Aubert dualities for $G^\circ$ and $M^\circ$, 
the connected components of the identity of $G$ and $M$, respectively.
To avoid Aubert duality for non-connected groups explained in the next subsections, 
we assume the following. 

\begin{hyp}\label{hyp_appendixB}
There are $A$-packets $\Pi_\psi$ and $\Pi_{\widehat\psi}$ 
which are multi-sets over $\Irr_\unit(M)$. 
Moreover, there is a bijection $\Pi_\psi \ni \pi \mapsto \pi' \in \Pi_{\widehat\psi}$ 
such that $\pi'|_{M^\circ}$ is the Aubert dual of $\pi|_{M^\circ}$. 
\end{hyp}

Write 
$\psi = \psi_\GL \oplus \psi_0$ with $\psi_0 \in \Psi(G_0)$ and 
$\pi = \tau \boxtimes \sigma$ for $\tau \in \Irr(\GL_k(E))$ corresponding to $\psi_\GL$ 
and $\sigma \in \Pi_{\psi_0}$.
We set $\psi_s = \psi_\GL|\cdot|_E^s \oplus \psi_0$
and $\pi_s = \tau|\cdot|_E^s \boxtimes \sigma$ for $s \in \C$. 
\par

As in Section \ref{sec.NIO}, 
for $w \in W(M^\circ)$ and $s \in \C$, 
we have the normalized intertwining operator 
\[
R_P(w, \pi_s, \psi_s) \colon I_P(\pi_s) \rightarrow I_P(w\pi_s), 
\]
which is a meromorphic family of operators. 
Since $\pi$ is unitary, $R_P(w, \pi_s, \psi_s)$ is regular at $s = 0$, 
and we obtain a well-defined operator 
$R_P(w, \pi, \psi) = R_P(w, \pi_s, \psi_s)|_{s= 0}$. 
\par

Note that $\widehat{I_P(\pi_s)} \cong I_{\overline{P}}(\hat\pi_s)$ as representations of $G^\circ$ 
by Theorem \ref{AD} (5). 
Since $\widehat{w\pi_s} \cong w\hat\pi_s$, 
by the functoriality of Aubert duality, we have
\[
\widehat{R_P(w, \pi_s, \psi_s)} \colon I_{\overline{P}}(\hat\pi_s) \rightarrow I_{\overline{P}}(w\hat\pi_s).
\]
\par

On the other hand, we can define a normalized intertwining operator 
\[
R_{\overline{P}}(w, \hat\pi_s, \widehat\psi_s) 
\colon I_{\overline{P}}(\hat\pi_s) \rightarrow I_{\overline{P}}(w\hat\pi_s), 
\]
which is regular at $s=0$. 

\begin{prop}\label{IOvsAD}
Assume Hypothesis \ref{hyp_appendixB}. 
We further assume that $R_{\overline{P}}(w, \hat\pi, \widehat\psi)$ is bijective. 

\begin{enumerate}
\item
If $G=G^\circ$, then there is $c \in \C^{\times}$ such that 
\[
\widehat{
R_P(w, \pi, \psi)
}
= c \cdot R_{\overline{P}}(w, \hat\pi, \widehat\psi).
\]
\item
Suppose that $G = \O_{2n}(F)$.
\begin{enumerate}
\item
If $P \not= P^\circ$ and if $\sigma|_{G_0^\circ}$ is irreducible, 
then there is $c \in \C^{\times}$ such that 
\[
\widehat{
R_P(w, \pi, \psi)
}
= c \cdot R_{\overline{P}}(w, \hat\pi, \widehat\psi).
\]
\item
Otherwise, there is $\pi_s^\circ \in \Irr(M^\circ)$ such that 
$I_P(\pi_s) = I_{P^\circ}(\pi_s^\circ) = \Ind_{P^\circ}^G(\pi_s^\circ)$. 
If we denote by $I_{P^\circ}^+(\pi_s^\circ)$ (\resp $I_{P^\circ}^-(\pi_s^\circ)$) 
the subspace of $I_{P^\circ}(\pi_s^\circ)$ consisting of functions $f_s$ on $G$ 
whose supports are contained in $G^\circ$ (\resp $G \setminus G^\circ$), 
then $I_{P^\circ}(\pi_s^\circ) = I_{P^\circ}^+(\pi_s^\circ) \oplus I_{P^\circ}^-(\pi_s^\circ)$. 
Moreover, there are two constants $c_+, c_- \in \C^\times$ such that 
\[
\widehat{
R_P(w, \pi, \psi)
}
f_\pm
= c_\pm \cdot R_{\overline{P}}(w, \hat\pi, \widehat\psi)f_\pm
\]
for all $f_\pm$ in the Aubert dual of $I_{P^\circ}^\pm(\pi_s^\circ)$. 
\end{enumerate}
\end{enumerate}
\end{prop}
\begin{proof}
Suppose that $G = G^\circ$. 
Then since $I_{\overline{P}}(\hat\pi_s)$ is irreducible for almost all $q_E^{-s}$, 
we have the inverse map 
\[
R_{\overline{P}}(w, \hat\pi_s, \widehat\psi_s)^{-1} 
\colon I_{\overline{P}}(w\hat\pi_s) \rightarrow I_{\overline{P}}(\hat\pi_s), 
\]
for almost all $q_E^{-s}$. 
It is regular at $s=0$ since $R_{\overline{P}}(w, \hat\pi, \widehat\psi)$ is bijective. 
We consider the composition 
\[
R_{\overline{P}}(w, \hat\pi_s, \widehat\psi_s)^{-1} \circ \widehat{R_P(w, \pi_s, \psi_s)}. 
\]
This is a meromorphic family of self-intertwining operators on $I_{\overline{P}}(\hat\pi_s)$.
Again by the irreducibility of $I_{\overline{P}}(\hat\pi_s)$ for almost all $q_E^{-s}$, 
we can find a meromorphic function $c(s)$ such that 
\[
R_{\overline{P}}(w, \hat\pi_s, \widehat\psi_s)^{-1} \circ \widehat{R_P(w, \pi_s, \psi_s)} = c(s) \cdot \id. 
\]
Since the left-hand side is regular at $s=0$, we can define $c = c(0) \in \C$. 
Then 
\[
\widehat{
R_P(w, \pi, \psi)
}
= c \cdot R_{\overline{P}}(w, \hat\pi, \widehat\psi).
\]
Since $R_P(w, \pi, \psi)$ is not identically zero, neither is its Aubert dual. 
Hence $c \not= 0$. 
This completes the proof of (1).
\par

When $G = \O_{2n}(F)$, 
the proof is essentially the same, 
but we need to consider the restrictions to the connected components of the identity.
First, we consider (2a). 
Then $\pi_s|_{M^\circ} = (\pi|_{M^\circ})_s$ is irreducible, and
$I_P(\pi_s)|_{G^\circ} \cong \Ind_{P^\circ}^{G^\circ}((\pi|_{M^\circ})_s)$
by Lemma \ref{Ovs.SO} (a). 
By regarding $R_P(w, \pi, \psi)$ and $R_{\overline{P}}(w, \hat\pi, \hat\psi)$ as $G^\circ$-homomorphisms, 
the same argument as in (1) shows that 
\[
\widehat{
R_P(w, \pi, \psi)
}
= c \cdot R_{\overline{P}}(w, \hat\pi, \widehat\psi)
\]
for some $c \in \C^\times$. 
\par

Next, we assume that $P = P^\circ$ or $\pi|_{M^\circ}$ is reducible. 
Then by Lemma \ref{Ovs.SO} (b), 
we have $I_P(\pi_s) = I_{P^\circ}(\pi_s^\circ)$ for any irreducible component $\pi^\circ$ of $\pi|_{M^\circ}$.
Moreover, by the proof of that lemma,
we have $I_{P^\circ}(\pi_s^\circ) = I_{P^\circ}^+(\pi_s^\circ) \oplus I_{P^\circ}^-(\pi_s^\circ)$. 
As a $G^\circ$-homomorphism, $R_P(w, \pi_s, \psi_s)$ can be decomposed to the direct sum of 
\[
R_P(w, \pi_s, \psi_s) \colon I_{P^\circ}^\pm(\pi_s^\circ) \rightarrow I_{P^\circ}^{\pm\delta}(w\pi_s^\circ),
\]
where $\delta = \det(w) \in \{\pm1\}$.
By the same argument as in (1), 
we can find $c_\pm \in \C^\times$ such that 
\[
\widehat{
R_P(w, \pi, \psi)
}
= c_\pm \cdot R_{\overline{P}}(w, \hat\pi, \widehat\psi)
\]
holds on the Aubert dual of $I_{P^\circ}^\pm(\pi_s^\circ)$. 
This completes the proof of (2). 
\end{proof}

Now suppose that $\tau$ is conjugate-self-dual. 
Then as in Section \ref{sec.main3}, 
we can define the normalized self-intertwining operators 
\begin{align*}
\pair{\tl{u}, \tl\pi} R_P(w_u, \tl\pi, \psi) &\colon I_P(\pi) \rightarrow I_P(\pi), \\
\pair{\tl{u}, \tl{\hat{\pi}}} R_{\overline{P}}(w_u, \tl{\hat{\pi}}, \widehat\psi) 
&\colon I_{\overline{P}}(\hat\pi) \rightarrow I_{\overline{P}}(\hat\pi).
\end{align*}
Now we have the following corollary, which is a key result to prove Theorem \ref{main3}. 

\begin{cor}\label{NIOvsAD}
Assume Hypothesis \ref{hyp_appendixB}. 
We further assume that $R_{\overline{P}}(w_u, \hat\pi, \widehat\psi)$ is bijective, 
and that $\tau$ is conjugate-self-dual. 

\begin{enumerate}
\item
If we are in the case (1) or (2a) of Proposition \ref{IOvsAD}, 
then there is $c \in \C^{\times}$ such that 
\[
(\pair{\tl{u}, \tl\pi} R_P(w_u, \tl\pi, \psi)){\widehat{\,}}
= c \cdot \pair{\tl{u}, \tl{\hat{\pi}}} R_{\overline{P}}(w_u, \tl{\hat{\pi}}, \widehat\psi).
\]

\item
If we are in the case (2b) of Proposition \ref{IOvsAD}, 
then there are two constants $c_+, c_- \in \C^\times$ such that 
\[
(\pair{\tl{u}, \tl\pi} R_P(w_u, \tl\pi, \psi)){\widehat{\,}}
= c_\pm \cdot \pair{\tl{u}, \tl{\hat{\pi}}} R_{\overline{P}}(w_u, \tl{\hat{\pi}}, \widehat\psi)
\]
holds on the Aubert dual of $I_{P^\circ}^\pm(\pi_s^\circ)$. 
\end{enumerate}
\end{cor}
\begin{proof}
Recall that the normalized intertwining operator $\pair{\tl{u}, \tl\pi} R_P(w_u, \tl\pi, \psi)$
is defined by 
\[
\pair{\tl{u}, \tl\pi} R_P(w_u, \tl\pi, \psi)f(g) 
= \pair{\tl{u}, \tl\pi}\tl{\pi}(w_u)\left(
R_P(w_u, \pi, \psi)f(g)
\right)
\]
for an isomorphism 
$\pair{\tl{u}, \tl\pi}\tl{\pi}(w_u) \colon w_u \pi \xrightarrow{\sim} \pi$. 
By Schur's lemma, its Aubert dual is equal to $\pair{\tl{u}, \tl{\hat\pi}}\tl{\hat\pi}(w_u)$
up to a nonzero constant. 
Then the claim follows from Proposition \ref{IOvsAD}.
\end{proof}

\begin{rem}
In the next subsections,
we introduce Aubert duality for non-connected groups, in particular for $\O_{2n}(F)$. 
However, an analogue of Theorem \ref{AD} (5) will not be established, 
and hence a direct approach as in Proposition \ref{IOvsAD} (1) cannot be applied.
\end{rem}

\subsection{Twisted Aubert duality as a functor}\label{sec.twistedAD}
Now we review the twisted (or disconnected) Aubert duality and some of its properties that we need in this paper, following \cite{CK26}. 
We use the same notations in Sections \ref{sec.complex} and \ref{sec.aubert}. 
\par

Let $\theta$ be an involution of $G$ 
which preserves $P_0$ and $M_0$. 
Then $\theta$ also preserves $A_0$ and acts on the set $S$ of (relative) simple roots. 
We consider the disconnected group $\tl{G} = G \rtimes \pair{\theta}$. 
Let $\Rep(\tl{G})$ be the category of smooth representations of $\tl{G}$ of finite length.
For $0 \leq t \leq r$, 
we denote by $\Rep(\tl{G})_t$ the inverse image of $\Rep(G)_t$ 
under the restriction map $\Res \colon \Rep(\tl{G}) \rightarrow \Rep(G)$. The category $\Rep(\tl{G})$ has a block decomposition $\prod_{t}\Rep(\tl{G})_t$, see \cite[Lemma 3.1.1]{CK26}.

Let $\tl\pi \in \Rep(\tl{G})_t$. 
Set $\pi = \Res(\tl\pi) \in \Rep(G)_t$ and denote the space of $\pi$ by $V_\pi$. 
Recall that for $J \subset S$, we have a representation $X_J(\pi)$ of $G$.
This is a space of functions $f \colon G \rightarrow (V_\pi)_{N_{P_J}}$. 
We have a map
\[
X_J(\theta) \colon X_J(\pi) \rightarrow X_{\theta(J)}(\pi),\,
X_J(\theta)(f)(g) = \tl\pi(\theta)(f(\theta(g))).
\]
On the other hand, the bijection $\theta \colon (S\setminus J) \rightarrow (S \setminus \theta(J))$ induces 
an isomorphism $\lambda_J(\theta) \colon \Lambda_J \rightarrow \Lambda_{\theta(J)}$. By taking the tensor product $X_J(\theta) \otimes \lambda_J(\theta)$, 
we obtain a map 
\[
\tl{X}_J(\theta) \colon \tl{X}_J(\pi) \rightarrow \tl{X}_{\theta(J)}(\pi). 
\]
By \cite[Lemma 2.7.1]{CK26}, it satisfies that 
$\tl{X}_I(\theta) \circ \tl\varphi_I^J = \tl\varphi_{\theta(I)}^{\theta(J)} \circ \tl{X}_J(\theta)$
if $|J| = |I|+1$.
Putting these together for all $J \subset S$ with $|J| = t$, 
we can define an isomorphism 
\[
\tl{X}_t(\theta) \colon \tl{X}_t(\pi) \rightarrow \tl{X}_t(\pi)
\]
by requiring the diagram
\[
\begin{CD}
\tl{X}_t(\pi) @>\tl{X}_t(\theta)>> \tl{X}_t(\pi) \\
@VVV @VVV \\
\tl{X}_{J}(\pi) @>\tl{X}_J(\theta)>> \tl{X}_{\theta(J)}(\pi) 
\end{CD}
\]
is commutative for all $J \subset S$ with $|J| = t$, 
where the vertical maps are the canonical projections.

On $\tl{X}_t(\pi)$, we have operators $\tl{X}_t(\theta)$ and $\tl{X}_t(\pi(g))$ for $g \in G$.
For $\tl{g} \in \tl{G}$, we write 
\[
\tl{X}_t(\tl\pi(\tl{g})) = 
\left\{
\begin{aligned}
&\tl{X}_t(\pi(g)) \iif \tl{g} = g \in G, \\
&\tl{X}_t(\pi(g)) \circ \tl{X}_t(\theta) \iif \tl{g} = g \rtimes \theta \in G \rtimes \theta.
\end{aligned}
\right. 
\] 
According to \cite[Lemmas 3.1.3, 3.1.4, 3.1.5]{CK26} this automorphism is well-defined, satisfies
\[
\tl{X}_t(\tl\pi(\tl{g} \cdot \tl{g}')) = \tl{X}_t(\tl\pi(\tl{g})) \circ \tl{X}_t(\tl\pi(\tl{g}'))
\]
for $\tl{g}, \tl{g}' \in \tl{G}$, and is functorial in $\tilde\pi$. It gives the $\C$-vector space $\tl{X}_t(\pi)$ the structure of a representation of $\tl{G}$, functorial in $\tilde\pi$. We denote this representation by $\tl{X}_t(\tl\pi)$. 
\par

Therefore we have a functor 
\[
\tl{X}_t \colon \Rep(\tl{G}) \rightarrow \Rep(\tl{G}), 
\]
which by construction satisfies 
\[ \tag{A1} \label{eq:aubdiscres}
\tl{X}_t(\tl\pi)|_G = \tl{X}_t(\tl\pi|_G).
\]
If $\tl\pi \in \Rep(\tl{G})_t$, 
then $\tl{X}_0(\tl\pi) = \dots = \tl{X}_{t-1}(\tl\pi) = 0$, 
and the sequence 
\[
\begin{CD}
0 @>>> \tl{X}_r(\tl\pi) @>>> \tl{X}_{r-1}(\tl\pi) @>>> \cdots @>>> \tl{X}_t(\tl\pi)
\end{CD}
\]
is an exact sequence of representations of $\tl{G}$, by \cite[Proposition 3.1.8]{CK26}. Now we can define twisted Aubert duality. 
\begin{defi}\label{def.TAD}
For $\tl\pi \in \Rep(\tl{G})_t$, set 
\[
\widehat{\tl\pi} = \tl{X}_t(\tl\pi)/d_{t+1}(\tl{X}_{t+1}(\tl\pi))
\]
and call $\widehat{\tl\pi}$ the \emph{twisted Aubert dual} of $\tl\pi$.
\end{defi}

\begin{prop}\label{TAD}
Twisted Aubert duality $\tl\pi \mapsto \widehat{\tl\pi}$ satisfies the following properties. 
\begin{enumerate}
\item
The map $\Rep(\tl{G}) \ni \tl\pi \rightarrow \widehat{\tl\pi} \in \Rep(\tl{G})$ is an exact covariant functor. 

\item
If we write $\pi = \tl\pi|_G$, then $\widehat{\tl\pi}|_G = \hat\pi$.
\item
If $\tl\pi$ is irreducible, then $\widehat{\tl\pi}$ is also irreducible. 
\end{enumerate}
\end{prop}
Indeed, (1) follows from the construction together with Theorem \ref{AD} and \eqref{eq:aubdiscres}, 
(2) is a direct consequence of \eqref{eq:aubdiscres}, and (3) is \cite[Lemma 3.1.11]{CK26}, but in our simple case can also be seen as follows. Let $\tl\pi$ be an irreducible representation of $\tl{G}$.
Set $\pi = \tl{\pi}|_G$ so that $\hat\pi = \widehat{\tl\pi}|_G$ by (2). 
Note that $\pi$ is a direct sum of at most two irreducible representations of $G$ since $(\tl{G}:G) = 2$.
If $\pi$ is irreducible, then so is $\hat\pi$ and hence $\widehat{\tl\pi}$ must be an irreducible representation of $\tl{G}$. 
Suppose that $\pi = \pi_1 \oplus \pi_2$ with $\pi_i$ irreducible. 
Then $\hat\pi = \hat\pi_1 \oplus \hat\pi_2$.
Moreover, $\tl\pi(\theta)$ gives a linear isomorphism $\pi_1 \xrightarrow{\sim} \pi_2$ as $\C$-vector spaces. 
By the construction, it induces a linear isomorphism $\hat\pi_1 \xrightarrow{\sim} \hat\pi_2$. 
Since this is nothing but $\widehat{\tl\pi}(\theta)$, 
we conclude that $\widehat{\tl\pi}$ is irreducible as a representation of $\tl{G}$.

\begin{rem}
Further properties of the twisted Aubert duality functor $\tl\pi \mapsto \widehat{\tl\pi}$, such as the fact that it is an involution, and its compatibility with the contragredient functor, 
the parabolic induction functors, and Jacquet functors, are addressed in \cite{CK26}, but are not used in this paper.
\end{rem}

If $\eta$ is a character of $G$ satisfying $\eta \circ \theta = \eta$, 
then we can extend $\eta$ to a character of $\tl{G}$ by setting $\eta(\theta) = 1$. 
Hence for $\tl\pi \in \Rep(\tl{G})$, 
one can consider the twist $\tl\pi \otimes \eta$. 

\begin{lem}\label{AD_twist}
Let $\eta$ be a character of $G$ such that $\eta \circ \theta = \eta$. 
Then for $\tl\pi \in \Rep(\tl{G})$, 
the twisted Aubert dual of $\tl\pi \otimes \eta$ is equal to $\widehat{\tl\pi} \otimes \eta$. 
\end{lem}
\begin{proof}
This follows from the construction. 
\end{proof}

\subsection{Twisted Aubert duality at the level of Grothendieck groups}\label{sec.twistedAD_G}
Let $P = MN_P$ be a standard parabolic subgroup of $G$. 
If $P$ is $\theta$-stable, we may assume that $M$ is also $\theta$-stable. 
In this case, write $\tl{P} = P \rtimes \pair{\theta}$ and $\tl{M} = M \rtimes \pair{\theta}$. 
Then the normalized parabolic induction functor
\[
\Ind_{\tl{P}}^{\tl{G}} \colon \Rep(\tl{M}) \rightarrow \Rep(\tl{G})
\]
and 
the normalized Jacquet functor
\[
\Jac_{\tl{P}} \colon \Rep(\tl{G}) \rightarrow \Rep(\tl{M})
\]
can be defined as in the connected case. 
Note that for $\tl\pi \in \Rep(\tl{G})$, we have 
\[
\tl{X}_P(\tl\pi) \cong \Ind_{\tl{P}}^{\tl{G}}(\Jac_{\tl{P}}(\tl\pi)) \otimes \Lambda_J
\]
as representations of $\tl{G}$, where $J \subset S$ is such that $P=P_J$. 
\par

Let $\RR(\tl{G})$ be the Grothendieck group of $\Rep(\tl{G})$. 
When $\tl\pi \in \Rep(\tl{G})$, we denote by $[\tl\pi] \in \RR(\tl{G})$ the corresponding element. 
The character $\Theta_{\tl\pi}$ of $\tl\pi$, which is a linear functional on $C_c^\infty(\tl{G})$, depends only on $[\tl\pi]$.
Moreover, it gives a map 
\[
\RR(\tl{G}) \ni [\tl\pi] \mapsto \Theta_{\tl\pi} \in C_c^\infty(\tl{G})^*.
\]
For $[\tl\pi_1], [\tl\pi_2] \in \RR(\tl{G})$, we write 
\[
[\tl\pi_1] \overset{\theta}{=} [\tl\pi_2]
\]
if $\Theta_{\tl\pi_1}(f) = \Theta_{\tl\pi_2}(f)$ for any $f \in C_c^\infty(G \rtimes \theta)$.
For example, for $\pi \in \Irr(G)$, if $\Ind_G^{\tl{G}}(\pi) = \tl\pi_1 \oplus \tl\pi_2$ is reducible, 
then $[\tl\pi_2] \overset{\theta}{=} -[\tl\pi_1]$. 

\begin{prop}[{\cite[Proposition 3.2.1]{CK26}}]\label{prop2.13}
Let $\tl\pi \in \Rep(\tl{G})_t$.
Then 
\[
[\widehat{\tl\pi}] \overset{\theta}{=} (-1)^{r-t} \sum_{P=MN_P}(-1)^{\dim((A_M/A_G)^\theta)}
\left[\Ind_{\tl{P}}^{\tl{G}}(\Jac_{\tl{P}}(\tl\pi))\right], 
\]
where $P$ runs over the set of $\theta$-stable standard parabolic subgroups of $G$, 
and $(A_M/A_G)^\theta$ is the subgroup of $A_M/A_G$ fixed by $\theta$. 
\end{prop}

In particular, on the level of twisted characters, the functor $\tilde\pi \mapsto \widehat{\tl\pi}$ induces the same operation as the map $\rm{inv}^\theta(\tilde\pi)$ defined in \cite[Appendix A]{X2}, see \cite[Corollary 3.2.4]{CK26}.

\section{Derivatives of tempered representations}\label{xu-atobe}
We use the notations in Sections \ref{sec.parameters} and \ref{sec.Arthur}, 
except that we denote the normalized absolute value of $E$ by $|\cdot|$ for simplicity.
In particular, 
let $G$ be one of the following quasi-split classical groups
\[
\SO_{2n+1}(F), \quad \Sp_{2n}(F), \quad
\O_{2n}(F), \quad \U_n.
\]
Throughout this appendix, we assume Hypothesis \ref{ECR-temp}. 
For convenience, we restate this hypothesis. 
\begin{hyp}\label{ECR-tempD}
For any quasi-split classical group $G'$ with $\dim(\St_{\widehat{G'}}) \leq \dim(\St_{\widehat{G}})$, 
and for any tempered $L$-parameter $\phi'$ for $G'$,
there exists a subset $\Pi_{\phi'}$ of $\Irr_{\temp}(G')$ equipped with $\pair{\cdot,\pi'}_{\phi'}$ 
satisfying \eqref{ECR1} and \eqref{ECR2} in Section \ref{sec.Arthur}.
\end{hyp}

The purpose of this appendix is to extend some results in \cite{Moe11}, \cite{X1} and \cite{At} to $G$. 
Note that in \cite{At}, the author used M{\oe}glin's construction of tempered $L$-packets, which relies on \eqref{ECR1} and \eqref{ECR2} 
for all classical groups $G'$.
Hence it is not trivial that the arguments in \cite{At} can work under our weaker hypothesis. 
To avoid M{\oe}glin's construction, we will apply the argument in \cite{X1} directly to tempered $L$-parameters. 

\subsection{Review of Tadi\'c's formula}
One of the most useful tools to study representations of $p$-adic classical groups 
is the \emph{Geometric Lemma} \cite[Theorem 5.2]{BZ}, 
or its semisimplified version called \emph{Tadi\'c's (structure) formula}.
In this subsection, we recall this formula.
\par

First, we prepare some notation for $\GL_N(E)$. 
Let $\RR(\GL_N(E))$ be the Grothendieck group of $\Rep(\GL_N(E))$, 
and set 
\[
\RR^\GL = \bigoplus_{N = 0}^\infty \RR(\GL_N(E)).
\]
It is a graded commutative algebra equipped with 
the product 
\[
m \colon \RR^\GL \otimes \RR^\GL \rightarrow \RR^\GL, \;
\tau_1 \otimes \tau_2 \mapsto \tau_1 \times \tau_2
\]
defined using the parabolic induction functors. 
The unit element of $\RR^\GL$ is 
the trivial representation $\1_{\GL_0(E)}$ of the trivial group $\GL_0(E)$.
Moreover, we have a ring homomorphism 
\[
m^* \colon \RR^\GL \rightarrow \RR^\GL \otimes \RR^\GL
\]
defined using the normalized Jacquet functors by
\[
\RR(\GL_N(E)) \ni \tau \mapsto \sum_{k=0}^N \Jac_{(k,N-k)}(\tau). 
\]
Here, $\Jac_{(k_1,k_2)}$ denotes the Jacquet functor along the standard parabolic subgroup 
$P_{(k_1,k_2)}$ of $\GL_{k_1+k_2}(E)$ with the Levi factor $\GL_{k_1}(E) \times \GL_{k_2}(E)$.
For these facts, see \cite[Proposition 1.7]{Z}.
For example, by \cite[Proposition 9.5]{Z}, we know that 
\[
m^*(\Delta([x,y]_\rho)) = \sum_{i = 0}^{y-x+1} \Delta([x+i,y]_\rho) \otimes \Delta([x, x+i-1]_\rho). 
\]
Here, we formally understand that $\Delta([x,x-1]_\rho) = \1_{\GL_0(E)}$ for any $x \in \R$. 
We define a ring homomorphism
\[
M^* \colon \RR^\GL \rightarrow \RR^\GL \otimes \RR^\GL
\]
by the composition 
\[
M^* = (m \otimes \id) \circ ({{}^c\cdot^\vee} \otimes m^*) \circ s \circ m^*, 
\]
where $s \colon \RR^\GL \otimes \RR^\GL \rightarrow \RR^\GL \otimes \RR^\GL$
denotes the transposition $s(\sum_i \tau_i \otimes \tau_i') = \sum_i \tau_i' \otimes \tau_i$.
For example, 
\[
M^*(\Delta([x,y]_\rho)) = 
\sum_{\substack{x-1 \leq z \leq y \\ z+1 \leq w \leq y+1}}
\left(\Delta([-z, -x]_{{}^c\rho^\vee}) \times \Delta([w,y]_\rho) \right)
\otimes \Delta([z+1, w-1]_\rho).
\]
Here, the sum is over $z = x-1, x, \dots, y$ and $w = z+1, z+2, \dots, y+1$.
\par

Next, let $G$ be a quasi-split classical group as above. 
We denote by $G^\circ$ the connected component of $\1 \in G$. 
Note that $G^\circ = G$ unless $G=\O_{2n}(F)$, in which case $G^\circ = \SO_{2n}(F)$.
Fix a rational Borel subgroup $B^\circ = T^\circ U$ of $G^\circ$, 
and denote the normalizer of $(T^\circ, B^\circ)$ in $G$ by $T$.
Let $P^\circ = M^\circ N_P$ be the standard parabolic subgroup of $G^\circ$ 
with Levi subgroup $M^\circ$ isomorphic to 
$\GL_{d_1}(E) \times \dots \times \GL_{d_r}(E) \times G_0^\circ$, 
where $G_0^\circ$ is a classical group of the same type as $G^\circ$. 
If $P^\circ$ is stable under the adjoint action of $T$, 
we set $P = P^\circ \cdot T$ and $M = M^\circ \cdot T$ 
so that $M \cong \GL_{d_1}(E) \times \dots \times \GL_{d_r}(E) \times G_0$. 
Otherwise,
we put $P=P^\circ$ and $M = M^\circ$. 
\par

For $\pi_0 \in \Rep(G_0)$ and $\tau_i \in \Rep(\GL_{d_i}(E))$ for $1 \leq i \leq r$, 
we denote the normalized parabolically induced representation by 
\[
\tau_1 \times \dots \times \tau_r \rtimes \pi_0
= \Ind_P^G(\tau_1 \boxtimes \dots \boxtimes \tau_r \boxtimes \pi_0).
\]
On the other hand, for $\pi \in \Rep(G)$, we have the normalized Jacquet module 
\[
\Jac_P(\pi) \in \Rep(M)
\]
along $P$. 
They are related by Frobenius reciprocity 
\[
\Hom_G(\pi, \Ind_P^G(\sigma)) \cong \Hom_M(\Jac_P(\pi), \sigma)
\]
for $\pi \in \Rep(G)$ and $\sigma \in \Rep(M)$. 
\par

Recall that $\RR(G)$ is the Grothendieck group of $\Rep(G)$. 
If $P=MN_P$ is as above,
the normalized parabolic induction and the normalized Jacquet functor induce linear maps
\begin{align*}
\Ind_P^G \colon \RR(M) \rightarrow \RR(G), \\
\Jac_P \colon \RR(G) \rightarrow \RR(M).
\end{align*}
If $(P',M')$ is conjugate to $(P,M)$ by an element in $T$, 
then we can identify $\RR(M')$ (\resp $\Ind_{P'}^G$, $\Jac_{P'}$)
with $\RR(M)$ (\resp $\Ind_{P}^G$, $\Jac_{P}$).
\par

Set 
\[
\RR^G = \bigoplus_{G'}\RR(G'), 
\]
where $G'$ runs over all quasi-split classical groups of the same type as $G$. 
It possesses a module structure defined by
\[
\rtimes \colon \RR^\GL \otimes \RR^G \rightarrow \RR^G,\,
\tau \otimes \pi \mapsto \tau \rtimes \pi
\]
and a comodule structure 
\[
\mu^* \colon \RR^G \rightarrow \RR^\GL \otimes \RR^G
\]
defined by 
\[
\RR(G) \ni \pi \mapsto \sum_{k=0}^{n}\Jac_{P_k}(\pi).
\]
Here, $n$ is the $F$-rank of $G$, and 
$\Jac_{P_k}$ denotes the normalized Jacquet functor along a standard parabolic subgroup $P_k$ of $G$
with the Levi factor of the form $\GL_k(E) \times G_0$.
\par

The Geometric Lemma at the level of Grothendieck groups is stated as follows. 
\begin{thm}[{Tadi\'c's formula (\cite[Theorems 5.4, 6.5]{T}, \cite[Theorem 7.3]{Ban})}]\label{Tformula}
For $\tau \in \RR^\GL$ and $\pi \in \RR^G$, 
we have
\[
\mu^*(\tau \rtimes \pi) = M^*(\tau) \rtimes \mu^*(\pi). 
\]
Here, for $\tau_1, \tau_2, \tau_3 \in \RR^\GL$ and $\pi_0 \in \RR^G$, we set 
\[
(\tau_1 \otimes \tau_2) \rtimes (\tau_3 \otimes \pi_0) 
= (\tau_1 \times \tau_3) \otimes (\tau_2 \rtimes \pi_0).
\]
\end{thm}

\begin{rem}
Tadi\'c's formula holds even for $\O_{2n}(F)$ (see \cite[(5.5)]{X1} or \cite[Theorem 7.3]{Ban}).
However, as in \cite[page 463]{X1},
for $\SO_{2n}(F)$, this formula needs to be modified. 
This is one of the reasons why we do not work with $\SO_{2n}(F)$ but with $\O_{2n}(F)$.
\end{rem}

The following is also a basic and useful tool. 
\begin{thm}[{Casselman's criterion (\cite[Lemma 2.4]{Konno}})]\label{Ccriterion}
Let $\pi$ be an irreducible representation of $G$. 
Then the following are equivalent. 
\begin{itemize}
\item
$\pi$ is a discrete series (\resp tempered) representation; 
\item
for any irreducible representation $\tau \otimes \pi_0$ 
appearing in $\mu^*(\pi)-\1_{\GL_0(E)} \otimes \pi$, 
if we denote the central character of $\tau$ by $\omega_\tau = \chi_u |\cdot|^s$
with $\chi_u$ unitary and $s \in \R$, 
then $s > 0$ (\resp $s \geq 0$).
\end{itemize}
\end{thm}

\subsection{Derivatives}\label{sec.derivative}
Now we introduce the notion of \emph{derivatives}. 
Note that this differs from the Bernstein--Zelevinsky derivatives in \cite{BZ}.
\par

\begin{defi}
Fix an irreducible supercuspidal representation $\rho$ of $\GL_d(E)$. 

\begin{enumerate}
\item
Suppose that $G^\circ$ has a standard parabolic subgroup $P^\circ = M^\circ N_P$ 
such that $M \cong \GL_d(E) \times G_-$. 
For $\pi \in \RR(G)$, if we write
\[
\Jac_P(\pi) = \sum_{i \in I} \tau_i \otimes \sigma_i \in \RR(\GL_d(E)) \otimes \RR(G_-)
\]
with $\tau_i \in \Irr(\GL_d(E))$ and $\sigma_i \in \Irr(G_-)$, 
we define the \emph{$\rho$-derivative} $D_\rho(\pi)$ by 
\[
D_\rho(\pi) = \sum_{\substack{i \in I \\ \tau_i \cong \rho}} \sigma_i \in \RR(G_-).
\]
If such a parabolic subgroup $P^\circ$ does not exist, we set $D_\rho(\pi) = 0$. 

\item
For $k \geq 0$, we define the \emph{$k$-th $\rho$-derivative} $D_\rho^{(k)}(\pi)$ by 
\[
D_\rho^{(k)}(\pi) = \frac{1}{k!} \underbrace{D_\rho \circ \dots \circ D_\rho}_k(\pi).
\]
In particular, $D_\rho^{(0)}(\pi) = \pi$.

\item
If $D_\rho^{(k)}(\pi) \not= 0$ but $D_\rho^{(k+1)}(\pi) = 0$, 
we say that $D_\rho^{(k)}(\pi)$ is the \emph{highest $\rho$-derivative}, 
and denote it by $D_\rho^{\max}(\pi)$.

\item
We say that $\pi$ is \emph{$\rho$-reduced} if $D_\rho(\pi) = 0$.
\end{enumerate}
\end{defi}

Note that for any $\pi \in \Irr(G)$, it follows from Frobenius reciprocity that
$D_\rho(\pi) \not= 0$ 
if and only if there is an inclusion $\pi \hookrightarrow \rho \rtimes \pi_0$ for some representation $\pi_0$. 
By \cite[Lemma 5.6]{X1}, 
if $\rho' \not\cong \rho|\cdot|^{\pm1}$, 
then $D_{\rho} \circ D_{\rho'} = D_{\rho'} \circ D_\rho$. 
\par

We write $\rho^{\times k} = \rho \times \dots \times \rho$ ($k$ times) for short.
Note that for $\pi \in \RR(G)$, 
the $k$-th derivative $D_\rho^{(k)}(\pi)$ is 
a linear combination of irreducible representations whose coefficients are non-negative integers.
In fact, it is characterized by the identity
\[
\Jac_P(\pi) = \rho^{\times k} \otimes D_\rho^{(k)}(\pi) + \sum_{i} \tau_i \boxtimes \pi_i, 
\]
where $P^\circ = M^\circ N_P$ is such that $M^\circ = \GL_{dk}(E) \times G_0^\circ$, 
and $\tau_i \boxtimes \pi_i \in \Irr(M)$ is such that $\tau_i \not\cong \rho^{\times k}$.

\begin{lem}\label{SI}
Suppose that $\rho$ is not conjugate-self-dual. 
Then for any $\pi \in \Irr(G)$, 
its highest derivative $D_\rho^{\max}(\pi)$ is also irreducible. 
Moreover, the map $\Irr(G) \ni \pi \mapsto D_\rho^{\max}(\pi)$ is injective in the following sense: 
for $\pi,\pi' \in \Irr(G)$, 
if $D_\rho^{\max}(\pi) = D_\rho^{(k)}(\pi) = D_\rho^{(k)}(\pi')$, then $\pi \cong \pi'$.
\end{lem}
\begin{proof}
Write $D_\rho^{\max}(\pi) = D_\rho^{(k)}(\pi)$.
One can take an irreducible summand $\pi_0$ of $D_\rho^{\max}(\pi)$ such that 
\[
\Jac_P(\pi) \twoheadrightarrow \underbrace{\rho \boxtimes \dots \boxtimes \rho}_k \boxtimes \pi_0
\]
in $\Rep(M)$, 
where $P ^\circ = M^\circ N_P$ is an appropriate standard parabolic subgroup of $G^\circ$. 
By Frobenius reciprocity, we have 
\[
\pi \hookrightarrow \rho^{\times k} \rtimes \pi_0.
\]
In particular, we have 
\[
\pi_0 \leq D_\rho^{(k)}(\pi) \leq D_\rho^{(k)}\left(\rho^{\times k} \rtimes \pi_0\right)
\]
in $\RR(G_0)$ for some classical group $G_0$.
(Here, for $A, B \in \RR(G_0)$, we write $A \leq B$ 
if $B-A$ is a non-negative combination of irreducible representations.)
Since $\pi_0$ is $\rho$-reduced and since $\rho$ is not conjugate-self-dual, 
by applying Tadi\'c's formula (Theorem \ref{Tformula}) to $\rho^{\times k} \rtimes \pi_0$, 
we have
\[
D_\rho^{(k)}\left(\rho^{\times k} \rtimes \pi_0\right) = \pi_0, 
\]
and hence we have $D_\rho^{(k)}(\pi) = \pi_0$.
\par

Suppose that $\pi' \in \Irr(G)$ satisfies $D_\rho^{(k)}(\pi') = \pi_0$ and $\pi' \not\cong \pi$. 
Then $\pi'$ is also an irreducible subrepresentation of $\rho^{\times k} \rtimes \pi_0$. 
However, in the Grothendieck group, since $\pi' \leq (\rho^{\times k} \rtimes \pi_0) - \pi$, 
we have
\[
\pi_0 = D_\rho^{(k)}(\pi') \leq D_\rho^{(k)}\left((\rho^{\times k} \rtimes \pi_0) - \pi\right) = 0.
\]
This is a contradiction.
\end{proof}

By the compatibility of Aubert duality with Jacquet functors, 
we obtain the following.
\begin{lem}\label{AvsD}
Fix an irreducible cuspidal representation $\rho$ of $\GL_d(E)$. 
Let $\pi$ be a representation of $G$ of finite length. 
Then
\[
\left(D_{\rho}(\pi)\right)^{\widehat{\;}} = D_{{}^c\rho^\vee}(\hat\pi).
\]
\end{lem}
\begin{proof}
This is a special case of \cite[Th\'eor\`eme 1.7 (2)]{Au} if $G = G^\circ$.
The same proof works for $G = \O_{2n}(F)$.
\end{proof}

Similarly, 
let $P = MN_P$ be a $\theta$-stable standard parabolic subgroup of $\GL_N(E)$ 
such that $M \cong \GL_d(E) \times \GL_{N_-}(E) \times \GL_d(E)$. 
For $\tl\pi \in \RR(\tl\GL_N(E))$, 
if we write
\[
\Jac_{\tl{P}}(\tl\pi) = \sum_{i \in I} \tau_i \otimes \tl\sigma_i \otimes {}^c\tau_i^\vee
+ \sum_{j \in J} (\pi_j+\pi_j \circ \theta)
\]
with $\tau_i \in \Irr(\GL_d(E))$, $\tl\sigma_i \in \Irr(\tl\GL_{N_-}(E))$
and $\pi_j \in \Irr(M)$ such that $\pi_j \not\cong \pi_j \circ \theta$, 
we define
\[
\tl{D}_\rho(\tl\pi) = \sum_{\substack{i \in I \\ \tau_i \cong \rho}} \tl\sigma_i
\in \RR(\tl\GL_{N_-}(E)). 
\]
Moreover, for $k \geq 0$, we set
\[
\tl{D}_\rho^{(k)}(\tl\pi) = \frac{1}{k!} \underbrace{\tl{D}_\rho \circ \dots \circ \tl{D}_\rho}_k(\tl\pi).
\]
Note that a priori, 
$\tl{D}_\rho^{(k)}(\tl\pi)$ is in $\RR(\tl\GL_{N_0}(E)) \otimes_\Z \Q$ for some $N_0 \geq0$.

\subsection{Compatibility of Jacquet functors with twisted endoscopy}
We denote by $\C[\Pi(G)]$ (\resp $\C^\theta[\Pi(N)]$)
the space of invariant (\resp twisted invariant) distributions on $G$ (\resp $\GL_N(E)$)
which are finite in the sense that they are finite linear combinations of irreducible characters. 
Then the representation theoretic characters provide an isomorphism
\[
\C[\Pi(G)] \cong \RR(G) \otimes_\Z \C.
\]
For any $\pi \in \RR(G)$, 
its character is denoted by $\Theta_\pi \in \C[\Pi(G)]$. 
Finally, we denote the space of stable finite linear combinations of characters of $G$
by $\C[\Pi(G)]^\st$.
\par

By the above isomorphism, one can transfer our definition of Jacquet functors 
to the spaces $\C[\Pi(G)]$ and $\C^\theta[\Pi(N)]$. 
Now we set $N = \dim(\St_{\widehat{G}})$ as in Section \ref{sec.parameters}. 
We fix an irreducible unitary supercuspidal representation $\rho$ of $\GL_d(E)$, and $x \in \R$. 
We consider Levi subgroups 
\[
\GL_d(E) \times \GL_{N_-}(E) \times \GL_d(E)
\]
of $\GL_N(E)$ and 
\[
M^\circ = \GL_d(E) \times G_-^\circ
\]
of $G^\circ$, respectively. 
Then by \cite[(6.6)]{X1}, 
we have a commutative diagram: 
\[
\begin{CD}
\C[\Pi(G)]^\st @>(G:G^\circ)^{-1} \Trans^\theta_{G^\circ}>> \C^\theta[\Pi(N)] \\
@V D_{\rho|\cdot|^x} VV @VV \tl{D}_{\rho|\cdot|^x} V \\
\C[\Pi(G_-)]^\st @>(G_-:G_-^\circ)^{-1} \Trans^\theta_{G_-^\circ}>> \C^\theta[\Pi(N_-)] 
\end{CD}
\]
where 
the above horizontal map is the composition of 
the averaged restriction map $\Theta \mapsto (G:G^\circ)^{-1} \Theta|_{C_c^\infty(G^\circ)}$
and the twisted endoscopic transfer map
$\Trans^\theta_{G^\circ} \colon \C[\Pi(G^\circ)]^\st \rightarrow \C^\theta[\Pi(N)]$.
The bottom horizontal map is defined similarly. 
This commutativity was proven for general quasi-split connected reductive groups in \cite[Appendix C]{X1},
and can be extended to the case of $\O_{2n}$ easily from the case of $\SO_{2n}$ (see Remark \ref{packet_SO} (3)).
Applying the above commutative diagram repeatedly, 
we obtain a commutative diagram 
\[
\begin{CD}
\C[\Pi(G)]^\st @>(G:G^\circ)^{-1} \Trans^\theta_{G^\circ}>> \C^\theta[\Pi(N)] \\
@V D^{(k)}_{\rho|\cdot|^x} VV @VV \tl{D}^{(k)}_{\rho|\cdot|^x} V \\
\C[\Pi(G_0)]^\st @>(G_0:G_0^\circ)^{-1} \Trans^\theta_{G_0^\circ}>> \C^\theta[\Pi(N_0)] 
\end{CD}
\]
for suitable $G_0$ and $N_0$. 
\par

Let $\phi$ be a tempered $L$-parameter for $G$, 
and let $\Pi_\phi$ be the $L$-packet associated to $\phi$. 
Then we have a stable invariant distribution 
\[
\sum_{\pi \in \Pi_\phi} \Theta_\pi \in \C[\Pi(G)]^\st, 
\]
whose image under $(G:G^\circ)^{-1} \Trans^\theta_{G^\circ}$ is 
the twisted character $\Theta_{\tl\pi_\phi} \in \C^\theta[\Pi(N)]$ by \eqref{ECR1}. 
\par

From now on, 
we suppose that 
\begin{itemize}
\item
$\rho$ is conjugate-self-dual; 
\item
$2x$ is a positive integer; 
\item
$\phi$ contains $\rho \boxtimes S_{2x+1}$ with multiplicity $k \geq 0$.
\end{itemize}
Then 
\[
\left.
\tl{D}^{(k)}_{\rho|\cdot|^x}(\tl\pi_\phi)
\right|_{\GL_{N-2dk}(E)} 
= C \cdot \pi_{\phi_0} 
\]
for some $C \in \Q$,
where 
\[
\phi_0 = \phi- (\rho \boxtimes S_{2x+1})^{\oplus k} \oplus (\rho \boxtimes S_{2x-1})^{\oplus k}.
\]
In particular, $\tl{D}^{(k+1)}_{\rho|\cdot|^x}(\tl\pi_\phi) = 0$ and 
one can define $\epsilon(\phi) \in \Q$ such that 
\[
\tl{D}^{(k)}_{\rho|\cdot|^x}(\tl\pi_\phi) \overset{\theta}{=} \epsilon(\phi) \cdot \tl\pi_{\phi_0}. 
\]
Here, we recall from Section \ref{sec.TAD_GL} that 
for $\tl\pi_1, \tl\pi_2 \in \RR(\tl\GL_N(E))$,
we write $\tl\pi_1 \overset{\theta}{=} \tl\pi_2$ if $\Theta_{\tl\pi_1}(\tl{f}) = \Theta_{\tl\pi_2}(\tl{f})$
for any $\tl{f} \in C_c^\infty(\GL_N(E) \rtimes \theta)$.
The above commutative diagram implies the following lemma. 

\begin{lem}\label{mult_k}
We have $D_{\rho|\cdot|^x}^{(k+1)}(\pi) = 0$ for any $\pi \in \Pi_\phi$.
Moreover, $\epsilon(\phi) \in \{0,1\}$ and 
\[
D_{\rho|\cdot|^x}^{(k)}\left(\sum_{\pi \in \Pi_\phi} \pi \right)
= \epsilon(\phi)
\sum_{\pi_0 \in \Pi_{\phi_0}} \pi_0.
\]
\end{lem}
\begin{proof}
By \eqref{ECR1}, we have
\begin{align*}
\frac{1}{(G_0:G_0^\circ)} \Trans^\theta_{G_0^\circ}
\circ D_{\rho|\cdot|^x}^{(k)}\left(
\sum_{\pi \in \Pi_\phi} \Theta_\pi
\right)
&= 
\tl{D}_{\rho|\cdot|^x}^{(k)} \circ \Trans^\theta_{G^\circ}
\left(
\frac{1}{(G:G^\circ)} 
\sum_{\pi \in \Pi_\phi} \Theta_\pi
\right)
\\&= \tl{D}_{\rho|\cdot|^x}^{(k)}(\Theta_{\tl\pi_\phi})
\overset{\theta}{=} \epsilon(\phi) \cdot \Theta_{\tl\pi_{\phi_0}}
\\&= \frac{\epsilon(\phi)}{(G_0:G_0^\circ)} \Trans^\theta_{G_0^\circ}
\left(
\sum_{\pi_0 \in \Pi_{\phi_0}} \Theta_{\pi_0}
\right).
\end{align*}
Hence
\[
\sum_{\pi \in \Pi_\phi} D_{\rho|\cdot|^x}^{(k)}(\Theta_\pi)
= \epsilon(\phi)\sum_{\pi_0 \in \Pi_{\phi_0}} \Theta_{\pi_0}.
\]
A similar argument shows that $\sum_{\pi \in \Pi_\phi} D_{\rho|\cdot|^x}^{(k+1)}(\Theta_\pi) = 0$, 
which implies that $D_{\rho|\cdot|^x}^{(k+1)}(\pi) = 0$ for any $\pi \in \Pi_\phi$.
Hence by Lemma \ref{SI}, $\sum_{\pi \in \Pi_\phi} D_{\rho|\cdot|^x}^{(k)}(\pi)$
is a multiplicity-free sum of irreducible representations (possibly zero).
This implies that $\epsilon(\phi) \in \{0,1\}$, 
and we obtain the last equation in the statement. 
\end{proof}

\subsection{Computation of highest derivatives}
Recall that 
we fix an irreducible unitary supercuspidal representation $\rho$ of $\GL_d(E)$, 
and $x >0$ such that $2x \in \Z$. 
We compute $D^{(k)}_{\rho|\cdot|^x}(\pi)$ for each $\pi \in \Pi_\phi$. 
\par

First, we suppose that $\rho \boxtimes S_{2x+1}$ is conjugate-self-dual of the opposite type to $\phi$. 
Then the multiplicity $k$ of $\rho \boxtimes S_{2x+1}$ is even. 
Moreover, if we set $\phi' = \phi - (\rho \boxtimes S_{2x+1})^k$, 
then $A_\phi = A_{\phi'}$ and 
\[
\pi = 
\underbrace{\Delta([-x,x]_\rho) \times \dots \times \Delta([-x,x]_\rho)}_{k/2} \rtimes \pi', 
\]
where $\pi' \in \Pi_{\phi'}$ is such that $\pair{\cdot, \pi'}_{\phi'} = \pair{\cdot, \pi}_\phi$. 
See \cite[Theorem 1.5.1]{Ar} and \cite[Theorem 2.5.1]{Mok}.
Hence by Tadi\'c's formula (Theorem \ref{Tformula}), we have
\[
D_{\rho|\cdot|^x}^{(k)}(\pi) 
= 
\underbrace{\Delta([-(x-1),x-1]_\rho) \times \dots \times \Delta([-(x-1),x-1]_\rho)}_{k/2} \rtimes \pi'. 
\]
This is equal to $\pi_0 \in \Pi_{\phi_0}$ with 
$\pair{\cdot, \pi_0}_{\phi_0} = \pair{\cdot, \pi}_\phi$ via the canonical identification $A_{\phi_0} = A_\phi$, 
where $\phi_0$ is the same as in the previous subsection.
\par

In the rest of this subsection, 
we assume that $\rho \boxtimes S_{2x+1}$ is conjugate-self-dual of the same type as $\phi$.
In this case, 
to compute $D^{(k)}_{\rho|\cdot|^x}(\pi)$ for each $\pi \in \Pi_\phi$, 
we use \eqref{ECR2}. 
Let $s \in A_\phi$. 
When $G = \O_{2n}(F)$ and $s \in A_\phi \setminus A_\phi^+$, 
we assume that $G^\circ_- \not= \{\1\}$.
Let $I_1$ be an index set such that $s = \sum_{i \in I_1} e(\rho_i,a_i,1)$, and set
\[
\phi_1 = \bigoplus_{i \in I_1} \rho_i \boxtimes S_{a_i}, 
\quad
\phi_2 = \phi - \phi_1. 
\]
For $i = 1,2$, fix a conjugate-self-dual character $\eta_i$ of $E^\times$
such that $\phi_i \otimes \eta_i \in \Phi_\temp(G_i)$ for some classical group $G_i$. 
Then $H^\circ = G_1^\circ \times G_2^\circ$ is an elliptic endoscopic group of $G^\circ$, 
and $(\eta_1,\eta_2)$ gives an $L$-homomorphism
\[
\xi \colon {}^L(G_1^\circ \times G_2^\circ) \rightarrow {}^L G^\circ. 
\]
Let 
\begin{align*}
M^\circ &= \GL_d(E) \times G_-^0, \\
M_1^\circ &= \GL_d(E) \times G_{1,-}^0, \\
M_2^\circ &= \GL_d(E) \times G_{2,-}^0
\end{align*}
be Levi subgroups of $G^\circ$, $G_1^\circ$ and $G_2^\circ$, 
respectively, 
and write 
\begin{align*}
H &= G_1 \times G_2, \\
H_{1,-} &= G_{1,-} \times G_2, \\
H_{2,-} &= G_{1} \times G_{2,-}. 
\end{align*}
Then by \cite[(6.2)--(6.4), (6.7)]{X1}, 
we have a commutative diagram: 
\[
\begin{CD}
\C[\Pi(H)]^\st @>
\Trans_{H}^{G}
>> \C[\Pi(G)] \\
@V D_{\rho\eta_1|\cdot|^x}\oplus D_{\rho\eta_2|\cdot|^x} VV @VV D_{\rho|\cdot|^x} V \\
\C[\Pi(H_{1,-})]^\st \oplus \C[\Pi(H_{2,-})]^\st @>
\sum_{i=1}^2\Trans_{H_{i,-}}^{G_-}
>> \C[\Pi(G_-)], 
\end{CD}
\]
where $\Trans_{H}^{G}$ is the transfer map normalized such that 
\[
\prod_{i=1}^2 \sum_{\pi_i \in \Pi_{\phi_i \otimes \eta_i}} \Theta_{\pi_i}
\mapsto 
\sum_{\pi \in \Pi_\phi} \pair{s, \pi}_\phi \Theta_\pi. 
\]
For the proof, see \cite[Appendix C]{X1}.
\par

We denote the multiplicity of $\rho \boxtimes S_{2x+1}$ in $\phi_i$ by $k_i$, 
and set
\[
\phi_{i,0} = \phi_i - (\rho \boxtimes S_{2x+1})^{\oplus k_i} \oplus (\rho \boxtimes S_{2x-1})^{\oplus k_i}.
\]
Let $G_{i,0}$ be the classical group such that $\phi_{i,0} \otimes \eta_i \in \Phi_\temp(G_{i,0})$.

\begin{prop}\label{DvsSECR}
Let $s_0$ be the image of $s$ under the surjection $A_\phi \twoheadrightarrow A_{\phi_0}$ defined by 
\[
e(\rho',d,1) \mapsto \left\{
\begin{aligned}
&e(\rho,2x-1,1) \iif \rho' \cong \rho,\, d=2x+1 > 2, \\
&0 \iif \rho' \cong \rho,\, d=2x+1 = 2, \\
&e(\rho',d,1) \other.
\end{aligned}
\right. 
\]
Then 
\[
D^{(k)}_{\rho|\cdot|^x}\left(
\sum_{\pi \in \Pi_\phi} \pair{s,\pi}_\phi \Theta_\pi \right)
= 
\epsilon(\phi_1 \otimes \eta_1)\epsilon(\phi_2 \otimes \eta_2)
\sum_{\pi_0 \in \Pi_{\phi_0}} \pair{s_0,\pi_0}_{\phi_0} \Theta_{\pi_0}.
\]
\end{prop}
\begin{proof}
Note that if $\phi_0 = 0$, then $2x+1 = 2$ and $\phi = (\rho \boxtimes S_2)^{\oplus k}$. 
In this case, $A_\phi^+ = A_\phi$.
In other words, if $G = \O_{2n}(F)$ and $s \in A_\phi \setminus A_\phi^+$, 
then we have $G^\circ_- \not= \{\1\}$ automatically, and we can use the above diagram. 
\par

By applying the diagram repeatedly together with \eqref{ECR2} and Lemma \ref{mult_k},
we have 
\begin{align*}
D^{(k)}_{\rho|\cdot|^x}\left(
\sum_{\pi \in \Pi_\phi} \pair{s,\pi}_\phi \Theta_\pi \right)
&= 
D^{(k)}_{\rho|\cdot|^x} \circ \Trans_{G_1 \times G_2}^{G}
\left(\prod_{i=1}^2
\sum_{\pi_i \in \Pi_{\phi_i \otimes \eta_i}} \Theta_{\pi_i} \right)
\\&=
\Trans_{G_{1,0} \times G_{2,0}}^{G_0}\left(
\prod_{i=1}^2
D^{(k_i)}_{\rho\eta_i|\cdot|^x}\left(
\sum_{\pi_i \in \Pi_{\phi_i \otimes \eta_i}} \Theta_{\pi_i}
\right)
\right)
\\&=
\Trans_{G_{1,0} \times G_{2,0}}^{G_0}\left(
\prod_{i=1}^2
\epsilon(\phi_i \otimes \eta_i)
\sum_{\pi_{i,0} \in \Pi_{\phi_{i,0} \otimes \eta_i}}
\Theta_{\pi_{i,0}}
\right)
\\&= 
\epsilon(\phi_1 \otimes \eta_1)\epsilon(\phi_2 \otimes \eta_2)
\sum_{\pi_0 \in \Pi_{\phi_0}} \pair{s_0,\pi_0}_{\phi_0} \Theta_{\pi_0}. 
\end{align*}
This completes the proof. 
\end{proof}

As a first consequence, we obtain the following important result.
\begin{cor}\label{epsilon=1}
We have $\epsilon(\phi) = 1$.
\end{cor}
\begin{proof}
Recall that $\epsilon(\phi) \in \{0,1\}$ satisfies 
$\tl{D}^{(k)}_{\rho|\cdot|^x}(\tl\pi_\phi) \overset{\theta}{=} \epsilon(\phi) \cdot \tl\pi_{\phi_0}$
and 
\[
\sum_{\pi \in \Pi_\phi} D_{\rho|\cdot|^x}^{(k)}(\Theta_\pi)
= \epsilon(\phi)\sum_{\pi_0 \in \Pi_{\phi_0}} \Theta_{\pi_0}.
\]
We show the claim by induction on $k$. 
\par

The case $k = 0$ is trivial since $\tl{D}^{(0)}_{\rho|\cdot|^x}(\tl\pi_\phi) = \tl\pi_\phi$ by definition. 
Similarly, if $k = 1$, then 
$\tl{D}^{(1)}_{\rho|\cdot|^x}(\tl\pi_\phi)|_{\GL_{N-2d}(E)} = \pi_{\phi_0}$
so that $\epsilon(\phi) \in \{\pm1\}$. 
Since $\epsilon(\phi) \in \{0,1\}$, we must have $\epsilon(\phi) = 1$. 
\par

Suppose now that $k > 1$. 
We apply Proposition \ref{DvsSECR} to $s = e(\rho, 2x+1,1)$. 
Then by the induction hypothesis, we have 
$\epsilon(\phi_i \otimes \eta_i) = 1$ for $i=1,2$
since $k_1 = 1$ and $k_2 = k-1$.
Hence
\[
D^{(k)}_{\rho|\cdot|^x}\left(\sum_{\pi \in \Pi_\phi} \pair{s,\pi}_\phi \Theta_\pi \right) \not= 0. 
\]
This implies that $\epsilon(\phi) \not= 0$ and hence $\epsilon(\phi) = 1$.
\end{proof}

In particular, by Lemma \ref{mult_k} and Corollary \ref{epsilon=1}, 
we have 
\[
D_{\rho|\cdot|^x}^{(k)}\left(\sum_{\pi \in \Pi_\phi} \pi \right)
= 
\sum_{\pi_0 \in \Pi_{\phi_0}} \pi_0.
\]
Now we compute $D_{\rho|\cdot|^x}^{(k)}(\pi)$ for $\pi \in \Pi_\phi$. 
\par

\begin{thm}\label{Dk}
Let $\phi$ be a tempered $L$-parameter for $G$. 
If $\phi$ does not contain $\rho \boxtimes S_{2x+1}$, 
then $D_{\rho|\cdot|^x}(\pi) = 0$ for any $\pi \in \Pi_\phi$. 
Suppose that $\phi$ contains $\rho \boxtimes S_{2x+1}$ with multiplicity $k>0$. 
Let $\pi \in \Pi_\phi$.
Then $D^{(k)}_{\rho|\cdot|^x}(\pi) = 0$ if and only if 
one of the following holds: 
\begin{itemize}
\item
$x \geq 1$, $\phi \supset \rho \boxtimes S_{2x-1}$ and 
$\pair{e(\rho, 2x+1,1), \pi}_\phi \not= \pair{e(\rho, 2x-1,1), \pi}_\phi$; 
\item
$x = \half{1}$ and $\pair{e(\rho, 2,1), \pi}_\phi = -1$. 
\end{itemize}
Moreover, if $D^{(k)}_{\rho|\cdot|^x}(\pi) \not= 0$, then $\pi_0 = D^{(k)}_{\rho|\cdot|^x}(\pi) \in \Pi_{\phi_0}$ and 
it is characterized by the property that $\pair{\cdot, \pi}_\phi$ is the composition of $\pair{\cdot, \pi_0}_{\phi_0}$
with the surjection $A_\phi \twoheadrightarrow A_{\phi_0}$ in Proposition \ref{DvsSECR}. 
\end{thm}
\begin{proof}
The first claim follows from Lemma \ref{mult_k} with $k=0$. 
\par

Now assume that $\phi$ contains $\rho \boxtimes S_{2x+1}$ with multiplicity $k>0$. 
First, suppose that $x \geq 1$ and $\phi \supset \rho \boxtimes S_{2x-1}$. 
We apply Proposition \ref{DvsSECR} to $s = e(\rho,2x+1,1)+e(\rho,2x-1,1)$.
Then $s_0 \in A_{\phi_0}^0$ so that 
\[
D^{(k)}_{\rho|\cdot|^x}\left(\sum_{\pi \in \Pi_\phi} \pair{s,\pi}_\phi \Theta_\pi \right)
\]
is a sum of irreducible characters with non-negative coefficients.
Hence, if $\pair{s,\pi}_\phi = \pair{e(\rho, 2x+1,1), \pi}_\phi \cdot \pair{e(\rho, 2x-1,1), \pi}_\phi = -1$, 
then $D^{(k)}_{\rho|\cdot|^x}(\pi) = 0$.
\par

Similarly, when $x = \half{1}$, 
by applying Proposition \ref{DvsSECR} to $s = e(\rho,2,1)$, 
we see that if $\pair{e(\rho, 2,1), \pi}_\phi = -1$, then $D^{(k)}_{\rho|\cdot|^x}(\pi) = 0$.
\par

Note that via the surjection $A_\phi \twoheadrightarrow A_{\phi_0}$, 
we have $z_\phi \mapsto z_{\phi_0}$, and the image of $A_\phi^0$ is included in $A_{\phi_0}^0$. 
Hence this map induces a surjection $\AA_\phi \twoheadrightarrow \AA_{\phi_0}$. 
Then $\pi$ does not satisfy the above two conditions 
if and only if 
the character $\pair{\cdot, \pi}_\phi \colon \AA_\phi \rightarrow \{\pm1\}$
factors through the surjection $\AA_\phi \twoheadrightarrow \AA_{\phi_0}$. 
In this case, we denote the character of $\AA_{\phi_0}$ associated to $\pi$ by $\eta_\pi$, i.e., 
\[
\pair{\cdot, \pi}_\phi \colon \AA_\phi \twoheadrightarrow \AA_{\phi_0} \xrightarrow{\eta_\pi} \{\pm1\}.
\]
\par

For $s_0 \in A_{\phi_0}$, choose a lift $s \in A_\phi$ of it. 
By applying Proposition \ref{DvsSECR} to $s$ together with Corollary \ref{epsilon=1}, 
we have 
\[
\sum_{\pi \in \Pi_\phi} \pair{s,\pi}_\phi D^{(k)}_{\rho|\cdot|^x}(\pi)
= 
\sum_{\pi_0 \in \Pi_{\phi_0}} \pair{s_0,\pi_0}_{\phi_0} \pi_0.
\]
This equation can be written as
\[
\sum_{\substack{\pi \in \Pi_\phi \\ D^{(k)}_{\rho|\cdot|^x}(\pi) \not= 0}} 
\eta_\pi(s_0) D^{(k)}_{\rho|\cdot|^x}(\pi)
= 
\sum_{\pi_0 \in \Pi_{\phi_0}} \pair{s_0,\pi_0}_{\phi_0} \pi_0
\]
for $s_0 \in \AA_{\phi_0}$.
This implies that 
$D^{(k)}_{\rho|\cdot|^x}(\pi) = \pi_0$
if and only if $\eta_\pi = \pair{\cdot, \pi_0}_{\phi_0}$. 
This completes the proof.
\end{proof}

When $D^{(k)}_{\rho|\cdot|^x}(\pi) = 0$, 
the highest derivative of $\pi$ is given as follows.

\begin{thm}\label{Dk-1}
Let $\phi$ be a tempered $L$-parameter for $G$. 
Suppose that $\phi$ contains $\rho \boxtimes S_{2x+1}$ with multiplicity $k>0$. 
Let $\pi \in \Pi_\phi$ and assume that $D^{(k)}_{\rho|\cdot|^x}(\pi) = 0$. 

\begin{enumerate}
\item
If $k$ is odd, then $\pi_1 = D^{(k-1)}_{\rho|\cdot|^x}(\pi)$ is nonzero, tempered and belongs to $\Pi_{\phi_1}$
with 
\[
\phi_1 = \phi - (\rho \boxtimes S_{2x+1})^{\oplus k-1} \oplus (\rho \boxtimes S_{2x-1})^{\oplus k-1}. 
\]
Moreover, there is a canonical isomorphism $\AA_{\phi_1} \cong \AA_\phi$, 
and $\pair{\cdot, \pi_1}_{\phi_1} = \pair{\cdot, \pi}_\phi$ via this isomorphism.

\item
If $k$ is even, then $D^{(k-1)}_{\rho|\cdot|^x}(\pi)$ is nonzero but not tempered. 
It is the Langlands quotient of the standard module
\[
\Delta([-(x-1),x]_\rho) \rtimes \pi_2, 
\]
where $\pi_2 \in \Pi_{\phi_2}$ with 
\[
\phi_2 = \phi - (\rho \boxtimes S_{2x+1})^{\oplus k} \oplus (\rho \boxtimes S_{2x-1})^{\oplus k-2}. 
\]
Moreover, there is a canonical inclusion $\AA_{\phi_2} \hookrightarrow \AA_\phi$, 
and $\pair{\cdot, \pi_2}_{\phi_2} = \pair{\cdot, \pi}_\phi|_{\AA_{\phi_2}}$.
\end{enumerate}
\end{thm}
\begin{proof}
First, we consider the case $k = 2$. 
Then 
\[
\pi \hookrightarrow \Delta([-x,x]_\rho) \rtimes \pi_2
\hookrightarrow \rho|\cdot|^x \times \Delta([-x,x-1]_\rho) \rtimes \pi_2.
\]
By Frobenius reciprocity, we have $D_{\rho|\cdot|^x}^{(1)}(\pi) \not= 0$. 
Since $D_{\rho|\cdot|^x}^{(2)}(\pi) = 0$, 
by Lemma \ref{SI}, we know that $D_{\rho|\cdot|^x}^{(1)}(\pi)$ is irreducible. 
Moreover, we have a nonzero equivariant map
\[
D_{\rho|\cdot|^x}^{(1)}(\pi) \rightarrow \Delta([-x,x-1]_\rho) \rtimes \pi_2. 
\]
Since $D_{\rho|\cdot|^x}^{(1)}(\pi)$ is irreducible, this map must be injective. 
Using the MVW involution, we see that $D_{\rho|\cdot|^x}^{(1)}(\pi)$ is the unique irreducible quotient of 
$\Delta([-(x-1),x]_\rho) \rtimes \pi_2$. 
See e.g., \cite[Section 2.7]{AG}.
\par

Note that when $k=1$, there is nothing to prove. 
Therefore, the remaining case is where $k \geq 3$. 
Write $k = 2l+k'$ with $k' \in \{1,2\}$. 
Consider $\pi' \in \Pi_{\phi'}$ with
\[
\phi' = \phi - (\rho \boxtimes S_{2x+1})^{\oplus 2l}
\]
so that $\AA_{\phi'} \cong \AA_\phi$, 
and $\pair{\cdot, \pi'}_{\phi'} = \pair{\cdot, \pi}_\phi|_{\AA_{\phi'}}$.
Then by \cite[Theorem 1.5.1]{Ar} and \cite[Theorem 2.5.1]{Mok}, 
the parabolically induced representation 
\[
\underbrace{\Delta([-x,x]_\rho) \times \dots \times \Delta([-x,x]_\rho)}_l \rtimes \pi'
\]
is irreducible and is equal to $\pi$.
By Tadi\'c's formula (Theorem \ref{Tformula}), 
we see that 
\begin{align*}
D^{(k-1)}_{\rho|\cdot|^x}(\pi) 
&= D^{(k-1)}_{\rho|\cdot|^x}(\Delta([-x,x]_\rho) \times \dots \times \Delta([-x,x]_\rho) \rtimes \pi')
\\&= \underbrace{\Delta([-(x-1),x-1]_\rho) \times \dots \times \Delta([-(x-1),x-1]_\rho)}_l \rtimes D^{(k'-1)}_{\rho|\cdot|^x}(\pi').
\end{align*}
This shows the assertions. 
\end{proof}

Finally, we give 
a characterization of (almost) supercuspidal representations.
(See also \cite[Theorem 2.5.1]{Moe11}.) 

\begin{cor}\label{sc}
Fix an irreducible unitary supercuspidal representation $\rho$ of $\GL_d(E)$. 
Let $\phi$ be a tempered $L$-parameter for $G$, 
and $\pi \in \Pi_\phi$. 
Then the following are equivalent: 
\begin{enumerate}
\item
$\pi$ is $\rho|\cdot|^x$-reduced 
for every real number $x \not= 0$; 
\item
the following conditions hold: 
\begin{itemize}
\item
If we denote the multiplicity of $\rho \boxtimes S_d$ in $\phi$ by $m_\phi(\rho,d)$, 
then $m_\phi(\rho,d) \leq 1$ whenever $d \geq 2$; 
\item
if $\rho \boxtimes S_d \subset \phi$ with $d > 2$, 
then $\rho \boxtimes S_{d-2} \subset \phi$ and 
\[
\pair{e(\rho,d,1), \pi}_\phi = - \pair{e(\rho,d-2,1), \pi}_\phi; 
\]
\item
if $\rho \boxtimes S_2 \subset \phi$, then 
\[
\pair{e(\rho,2,1), \pi}_\phi = -1. 
\]
\end{itemize}
\end{enumerate}
Moreover, $\pi$ is supercuspidal if and only if 
$\phi$ is a discrete parameter and 
the above conditions hold for any irreducible unitary supercuspidal representation $\rho$ of $\GL_d(E)$. 
\end{cor}
\begin{proof}
By Casselman's criterion (Theorem \ref{Ccriterion}), we know that 
any tempered representation $\pi$ is $\rho|\cdot|^x$-reduced for any $x < 0$. 
Then the first equivalence follows from Theorems \ref{Dk} and \ref{Dk-1}. 
\par

If $\pi$ is supercuspidal, then $\pi$ is a discrete series representation so that $\phi$ is discrete. 
Conversely, 
if $\pi$ satisfies the conditions in (2) but is not cuspidal, 
then one can find an irreducible cuspidal unitary representation $\rho$ of $\GL_d(E)$
such that $D_\rho(\pi) \not= 0$. 
In particular, $\pi \hookrightarrow \rho \rtimes \pi_0$ for some irreducible representation $\pi_0$.
Then Casselman's criterion (Theorem \ref{Ccriterion}) shows that $\pi_0$ is also tempered. 
Hence $\phi$ contains $\rho \oplus {}^c\rho^\vee$ so that $\phi$ is not discrete.
\end{proof}

\begin{rem}\label{0-derivative}
Fix an irreducible unitary cuspidal representation $\rho$ of $\GL_d(E)$. 
In general, for $\pi \in \Irr(G)$, 
the highest $\rho$-derivative $D_\rho^{\max}(\pi)$ is not necessarily irreducible, 
and it is difficult to describe $D_\rho^{\max}(\pi)$ completely. 
However, when $\pi$ is tempered, it is easy. 
More strongly, the next claim follows from \cite[Proposition 2.4.3]{Ar} and \cite[Proposition 3.4.4]{Mok}.
\par

Let $\phi$ be the tempered $L$-parameter of $\pi \in \Irr_\temp(G)$.
Then there exists $\pi' \in \Irr(G')$ such that 
\[
\pi \hookrightarrow \rho \rtimes \pi'
\]
if and only if $\phi$ contains $\rho \oplus {}^c\rho^\vee$.
In this case, $\pi'$ is uniquely determined by the conditions
$\pi' \in \Pi_{\phi'}$ with $\phi = \phi' \oplus \rho \oplus {}^c\rho^\vee$, 
and $\pair{\cdot, \pi'}_{\phi'} = \pair{\cdot, \pi}_\phi|_{\AA_{\phi'}}$.
In particular, if we write $\phi = \phi_0 \oplus (\rho \oplus {}^c\rho^\vee)^{\oplus k}$ such that 
$\phi_0 \not\supset \rho \oplus {}^c\rho^\vee$, 
and if $\pi_0 \in \Pi_{\phi_0}$ is such that $\pair{\cdot, \pi_0}_{\phi_0} = \pair{\cdot, \pi}_\phi|_{\AA_{\phi_0}}$, 
then $\pi \hookrightarrow \rho^{\times k} \rtimes \pi_0$ and 
\[
D_\rho^{\max}(\pi) = D_{\rho}^{(k)}(\pi) = m \cdot \pi_0
\]
with $m=2^k$ or $m=1$ according to $\rho \cong {}^c\rho^\vee$ or not.
\end{rem}

\subsection{Langlands data for certain irreducible representations}\label{sec.Ldata}
In this subsection, 
we explain how to compute the Langlands data for certain irreducible representations of classical groups.
\par

Fix an irreducible conjugate-self-dual cuspidal representation $\rho$ of $\GL_d(E)$. 
Let $\pi$ be an irreducible representation of a classical group $G$, 
and let $\II(\mm) \rtimes \tau$ be its standard module. 
Suppose that $D_{\rho|\cdot|^{-x}}(\pi) \not= 0$ for some $x > 0$.
We take the maximal $x$ with this condition. 
Let $y \geq x$ be the maximal real number with $y \equiv x \bmod \Z$ such that 
\begin{align*}
\pi'
&= D_{\rho|\cdot|^{-y}}^{\max} \circ \dots \circ D_{\rho|\cdot|^{-(x+1)}}^{\max} \circ D_{\rho|\cdot|^{-x}}^{\max}(\pi)
\\&= D_{\rho|\cdot|^{-y}}^{(k_y)} \circ \dots \circ D_{\rho|\cdot|^{-(x+1)}}^{(k_{x+1})} \circ D_{\rho|\cdot|^{-x}}^{(k_x)}(\pi)
\end{align*}
with $k_x, k_{x+1}, \dots, k_y \geq 1$.
Note that $\pi'$ is irreducible by Lemma \ref{SI}. 
We denote by $\II(\mm') \rtimes \tau'$ the standard module of $\pi'$.

\begin{prop}
With the above notations and assumptions, we have $k_x \geq k_{x+1} \geq \dots \geq k_y$. 
Moreover, $\tau \cong \tau'$ and 
\[
\mm = \mm' + \sum_{z=x}^y (k_z-k_{z+1})[x,z]_\rho
\]
hold, 
where $k_{y+1} = 0$.
Here, $\sum_{z=x}^y$ means the sum over $z = x, x+1, \dots, y$.
\end{prop}
\begin{proof}
By construction, we have an injection
\[
\pi \hookrightarrow 
(\rho|\cdot|^{-x})^{k_x} \times (\rho|\cdot|^{-(x+1)})^{k_{x+1}} \times \dots \times (\rho|\cdot|^{-y})^{k_y} \rtimes \pi'.
\]
We claim that $k_x \geq k_{x+1} \geq \dots \geq k_y$, and the above injection factors through
\[
\tag{$\natural$}\label{natural}
\pi \hookrightarrow 
\left(\bigtimes_{z = x}^y \Delta([-z,-x]_{\rho})^{k_z-k_{z+1}}\right) 
\rtimes \pi'.
\]
Fix $0 \leq i < y-x$. 
Suppose that we have proven that $k_x \geq \dots \geq k_{x+i}$
and the injection factors through 
\[
\pi \hookrightarrow 
\left(\bigtimes_{z = x}^{x+i} \Delta([-z,-x]_{\rho})^{k_z-k'_{z+1}}\right) 
\times (\rho|\cdot|^{-(x+i+1)})^{k_{x+i+1}} \times \dots \times (\rho|\cdot|^{-y})^{k_y} 
\rtimes \pi'
\]
with $k'_{z+1} = k_{z+1}$ if $z < x+i$, and $k'_{x+i+1} = 0$.
By \cite[Theorem 9.7]{Z}, $\Delta([-z,-x]_{\rho})$ commutes with $\rho|\cdot|^{-(x+i+1)}$ for $x \leq z < x+i$. 
On the other hand, by the Zelevinsky dual of \cite[Proposition 4.6]{Z}, 
if $z = x+i$, then $\Delta([-(x+i),-x]_{\rho}) \times \rho|\cdot|^{-(x+i+1)}$ 
contains $\Delta([-(x+i+1),-x]_{\rho})$ as a subrepresentation, which commutes with $\rho|\cdot|^{-(x+i+1)}$, 
and its quotient is an irreducible subrepresentation of $\rho|\cdot|^{-(x+i+1)} \times \Delta([-(x+i),-x]_{\rho})$. 
Hence if $k_{x+i} < k_{x+i+1}$, then every irreducible subquotient of 
\[
\left(\bigtimes_{z = x}^{x+i} \Delta([-z,-x]_{\rho})^{k_z-k'_{z+1}}\right) 
\times (\rho|\cdot|^{-(x+i+1)})^{k_{x+i+1}}
\]
is a subrepresentation of an induced representation of the form $\rho|\cdot|^{-(x+i+1)} \times \tau$ 
for some irreducible representation $\tau$ of some general linear group.
This implies that $D_{\rho|\cdot|^{-(x+i+1)}}(\pi) \not= 0$. 
This contradicts the definition of $x$.
Hence $k_{x+i} \geq k_{x+i+1}$, and by the same argument, 
the injection 
factors through 
\[
\pi \hookrightarrow 
\left(\bigtimes_{z = x}^{x+i+1} \Delta([-z,-x]_{\rho})^{k_z-k''_{z+1}}\right) 
\times (\rho|\cdot|^{-(x+i+2)})^{k_{x+i+2}} \times \dots \times (\rho|\cdot|^{-y})^{k_y} 
\rtimes \pi'
\]
with $k''_{z+1} = k_{z+1}$ if $z < x+i+1$, and $k''_{x+i+2} = 0$.
By induction, we obtain the claim.
\par

Since $D_{\rho|\cdot|^{-(y+1)}}(\pi') = 0$ by definition of $y$, 
the inclusion \eqref{natural} shows that $D_{\rho|\cdot|^{-x'}}(\pi') = 0$ for any $x' > y$.
Moreover, if $D_{\rho|\cdot|^{-x'}}(\pi') \not= 0$ for some $x \leq x' \leq y$, 
then  $x' \equiv x \bmod \Z$ and
\eqref{natural} shows that 
\[
D_{\rho|\cdot|^{-x'}}^{(k_{x'}+1)} 
\left(D_{\rho|\cdot|^{-(x'-1)}}^{(k_{x'-1})} \circ \dots \circ D_{\rho|\cdot|^{-x}}^{(k_x)}(\pi)\right)
\not= 0.
\]
This contradicts the definition of $k_{x'}$. 
Hence $D_{\rho|\cdot|^{-x'}}(\pi') = 0$ for any $x' \geq x$.
This means that $\mm'$ does not contain segments of the form $[x',y']_{\rho}$ for any $x' \geq x$.
\par

By applying \cite[Lemma 2.2]{AG} to \eqref{natural}, 
we obtain the surjections
\[
\left(\bigtimes_{z = x}^y \Delta([x,z]_\rho)^{k_z-k_{z+1}}\right) \times \II(\mm') \rtimes \tau'
\twoheadrightarrow 
\left(\bigtimes_{z = x}^y \Delta([x,z]_\rho)^{k_z-k_{z+1}}\right) \rtimes \pi'
\twoheadrightarrow \pi.
\]
We see that the left-hand side is a standard module 
so that it is isomorphic to $\II(\mm) \rtimes \tau$. 
Hence $\tau \cong \tau'$ and 
\[
\mm = \mm' + \sum_{z=x}^y (k_z-k_{z+1})[x,z]_\rho.
\]
This completes the proof. 
\end{proof}

As a consequence, we can describe the $L$-parameter of $\pi$
using that of $\pi'$ as follows.
This corollary is used in Sections \ref{(a)}--\ref{(c)}.
\begin{cor}\label{Lparameters}
We use the same notations and assumptions as above. 
Let $\phi_\pi$ and $\phi_{\pi'}$ be the $L$-parameters of $\pi$ and $\pi'$, respectively. 
Then 
\[
\phi_\pi = \phi_{\pi'} \oplus \left(
\bigoplus_{z = x}^y (\rho(|\cdot|^{\half{x+z}} \oplus |\cdot|^{-\half{x+z}}) \boxtimes S_{z-x+1})^{\oplus (k_z-k_{z+1})}
\right).
\]
\end{cor}

\section{The tempered $L$-packet conjecture and Arthur's Lemma 2.5.5}\label{sec.255}
In this appendix, 
we prove two results whose statements appear unrelated, but whose proofs follow from the same argument. 
\par

The first statement is that Arthur's construction of tempered $L$-packets satisfies 
the strong form of Shahidi's tempered packet conjecture (Corollary \ref{strong-shahidi-arthur}). 
This conjecture was initially formulated as \cite[Conjecture 9.4]{Sh7} 
and stipulated that every tempered $L$-packet should contain a generic representation. 
It was later strengthened to include the statement that, 
for an arbitrarily fixed Whittaker datum $\ww$, 
every tempered $L$-packet should contain exactly one $\ww$-generic member. 
For classical groups, this was proven by Konno \cite{Konno1} and Varma \cite{V2} 
assuming the twisted endoscopic transfer to $\GL_N$, 
which is the basis of the constructions of \cite{Ar} and \cite{Mok}. 
A further strengthening of the conjecture would require that 
the $\ww$-generic member is matched with the trivial character of the centralizer component group $\Sc_\phi$. 
It is this version that we formulate here (Conjecture \ref{TPC}) and prove for classical groups, 
thus providing a mild strengthening of the result of Konno and Varma for classical groups. 
We hasten to note that our proof does not replace those of Konno and Varma, but rather it uses them crucially. 
In fact, we prove a result (Theorem \ref{strong-shahidi}) that is valid for general reductive groups 
and infers the validity of this conjecture from its validity for endoscopic groups. 
This result is a strengthening of \cite[Proposition 9.6]{Sh7} and the proof follows from a similar outline, 
with a few additional arguments, and a key input from the work of Kottwitz \cite{Kot}. 
When combined with the results of Konno and Varma, 
this relative result delivers Conjecture \ref{strong-shahidi-arthur} for classical groups.
\par

The second statement (Theorem \ref{2.5.5}) is of a more technical nature. 
It infers the local intertwining relation from an a priori weaker statement. 
It was formulated for symplectic and orthogonal groups as \cite[Lemma 2.5.5]{Ar}, 
whose proof was deferred to \cite{A27}. 
We formulate and prove it here for arbitrary connected reductive groups, 
subject to assuming basic expected properties of tempered $L$-packets that are known in the setting of classical groups.
\par

Arthur suggested in \cite{Ar} that Theorem \ref{2.5.5} can be shown by applying results of Konno \cite{Konno1} 
concerning the behavior of local character expansions under the (twisted) endoscopic transfer. 
We largely follow his suggestion in the proofs below. 
However, it is also clear that the results of Konno needed to be refined and made more precise before they can be applied. 
Thankfully, in the intervening years, Varma has provided in \cite{V2} such a refined result, 
and this is the main ingredient in our proof below. 
It reduces the proof to an identity between Weil indices and $\ep$-factors, 
which follows at once from the work of Kottwitz \cite{Kot}, 
and in the case of classical groups can also be verified by hand.
Thanks to \cite{V1}, 
we can also remove the assumption in \cite[Lemma 2.5.5]{Ar} that the residual characteristic of $F$ is odd.
\par

We first employ this argument to obtain a proof of Theorem \ref{strong-shahidi}. 
The flow is as in \cite[Proposition 9.6]{Sh7}, and is in some sense opposite to that of \cite{V2}. 
Then we employ a similar argument, in a slightly more abbreviated form, to obtain Theorem \ref{2.5.5}.

\begin{rem}
While the proofs of Theorems \ref{strong-shahidi} and \ref{2.5.5} use the same
argument, the statements have different assumptions. The main assumption
for Theorem \ref{strong-shahidi} is that \eqref{ECR2} holds with respect to any endoscopic group, 
and Corollary \ref{strong-shahidi-arthur} assumes in addition that \eqref{ECR1} holds. 
Thus, in the setting of Arthur's argument, 
Theorem \ref{strong-shahidi} and Corollary \ref{strong-shahidi-arthur} 
become available after the inductive argument has been completed for the group of interest. 
On the other hand, the main assumption in Theorem \ref{2.5.5} is that 
a weakened form of the local intertwining relation holds. 
This assumption replaces \eqref{ECR2}. 
The reason is that Theorem \ref{2.5.5} will be used in the middle of the inductive proof, 
where \eqref{ECR2} is not yet available. 
In addition, the validity of Corollary \ref{strong-shahidi-arthur} is assumed for groups of lower rank. 
In the setting of Arthur's argument, this is part of the inductive assumption.
\end{rem}

\subsection{Shahidi's tempered packet conjecture for quasi-split classical groups}
Let $F$ be a local field of characteristic zero and let $G$ be a quasi-split connected reductive $F$-group. 
Fix a Whittaker datum $\ww$ for $G$. 
Let $\phi$ be a tempered $L$-parameter for $G$. 
We assume that the corresponding $L$-packet $\Pi_\phi$ and its pairing with $\Sc_\phi$ have been constructed, 
giving an injective map from $\Pi_\phi$ to the set of irreducible representations of $\Sc_\phi$; 
we use the notation $\pair{s,\pi}_\phi$ for the trace at $s$ of the representation of $\Sc_\phi$ associated to $\pi$. 
Note that the pairing $\pair{\cdot, \pi}_\phi$ depends on $\ww$.
Then Shahidi's strong tempered packet conjecture is the following statement, a strengthening of \cite[Conjecture 9.4]{Sh7}.

\begin{conj} \label{TPC}
Fix a Whittaker datum $\ww$ for $G$. 
The $L$-packet $\Pi_\phi$ contains exactly one $\ww$-generic member $\pi_\ww$, 
and it satisfies $\pair{\cdot,\pi_\ww}_\phi = \1$.
\end{conj}

This conjecture is known for archimedean base fields due to the work of Shelstad in \cite{SheTE3}, see also \cite{AK}. 
We may therefore concentrate on a non-archimedean base field $F$. 

\begin{thm}\label{strong-shahidi}
Assume $|\Pi_\phi|>1$ and that, 
for any factorization $\phi'$ of $\phi$ through a proper endoscopic group, 
the character identity \eqref{ECR2} holds and Conjecture \ref{TPC} holds for $\Pi_{\phi'}$. 
Then Conjecture \ref{TPC} holds for $\Pi_\phi$.
\end{thm}

We will show Theorem \ref{strong-shahidi} in Section \ref{sec.TPC} below. 
For now, we extract the following result in the setting of classical groups.

\begin{cor}\label{strong-shahidi-arthur}
Let $G$ be a quasi-split connected classical group. 
Assume that the local results \eqref{ECR1} and \eqref{ECR2} of \cite{Ar} and \cite{Mok} are known for the parameter $\phi$ 
and for its factorizations through endoscopic groups. 
Then Conjecture \ref{TPC} holds for $\Pi_\phi$.
\end{cor}
\begin{proof}
As already remarked, the archimedean case is known by the work of Shelstad, 
so we focus on non-archimedean $F$. 
We induct on the size of $L$-packets.
\par

Assume first that $\Pi_\phi$ is a singleton $\{\pi\}$. 
Then $\Sc_\phi$ is the trivial group, so trivially $\pair{\cdot, \pi}_\phi = \1$ 
and it remains to show that $\pi$ is $\ww$-generic. 
Let $\pi_\phi$ be the representation of $\GL_N(E)$ with parameter $\phi$, 
seen as a parameter for $\GL_N(E)$ via the standard representation of $^LG$. 
Then $\pi_\phi$ is tempered, hence generic. 
Using \eqref{ECR1}, the claim follows from \cite{Konno1} and \cite{V2}.
\par

Assume next that $\Pi_\phi$ is not a singleton. 
Then $\Sc_\phi \not= \{1\}$. 
For any $s \in S_\phi \setminus Z(\widehat{G})^\Gamma$, 
the pair $(s,\phi)$ leads to a proper endoscopic datum $G'$ and a parameter $\phi'$ for $G'$. 
The $L$-packet $\Pi_{\phi'}$ has strictly smaller size than $\Pi_\phi$, 
so Conjecture \ref{TPC} holds for $\Pi_{\phi'}$ by the induction hypothesis. 
Thus Conjecture \ref{TPC} holds for $\Pi_\phi$ by Theorem \ref{strong-shahidi}.
\end{proof}

\subsection{A weakening of the local intertwining relation}
Let $F$ be a non-archimedean local field of characteristic zero 
and let $G$ be a quasi-split connected reductive $F$-group equipped with a Whittaker datum $\ww$. 
Let $P = MN$ be a proper parabolic subgroup of $G$ 
and let $\phi_M$ be a tempered $L$-parameter for $M$. 
We assume the existence of an associated $L$-packet $\Pi_{\phi_M}$. 
Let $\phi$ be the parameter for $G$ obtained by composing $\phi_M$ with the natural inclusion ${}^LM \rightarrow {}^LG$. 
Recall that $N_\phi = N(A_{\widehat{M}}, \widehat{G}) \cap S_\phi$ 
and $\NN_\phi = \pi_0(N_\phi/Z(\widehat{G})^\Gamma)$. 
We now recall some material
from Section \ref{sec.main3} in this more general setting. 
For $f \in C_c^\infty(G(F))$ and $u \in \NN_\phi$, 
define the distribution
\[ 
f \mapsto f_G(\phi,u) 
= \sum_{\substack{\pi_M \in \Pi_{\phi_M}\\ w_u\pi_M \cong \pi_M}} 
\tr(\pair{ u, \tl\pi_M } R_P(w_u, \tl\pi_M, \phi_M) I_P(\pi_M,f)), 
\]
where $w_u$ is the Weyl element given by $u$ which preserves $M$, 
and we are using a representation $\tl\pi_M$ of the disconnected group 
\[ 
M(F) \rtimes \pair{w_u} 
\] 
that extends $\pi_M$ and the associated pairing $\pair{u,\tl\pi_M}$, 
noting that the product 
\[ 
\pair{u,\tl\pi_M}R_P(w_u,\tl\pi_M,\phi_M) 
\] 
is independent of the choice $\tl\pi_M$ extending $\pi_M$.
\par

Let $s \in \Sc_\phi$ be the image of $u$. 
From the pair $(s,\phi)$, 
we obtain an endoscopic datum $(G',s,\eta)$ and a parameter $\phi'$ for $G'$. 
We define a second distribution
\[ 
f \mapsto f_G'(\phi,s) = \Trans(S\Theta_{\phi'})(f) = S\Theta_{\phi'}(f'), 
\]
where 
\[ 
S\Theta_{\phi'} = \sum_{\pi' \in \Pi_{\phi'}} \pair{1,\pi' }_{\phi'} \Theta_{\pi'}
\]
is the stable character of $\phi'$, and $f' \in C_c^\infty(G'(F))$ is a $\Delta[\ww]$-transfer of $f$, 
i.e. the two functions $f$ and $f'$ have matching orbital integrals 
with respect to the transfer factor $\Delta[\ww]$ normalized with respect to the Whittaker datum $\ww$.
\par

We now state the assumptions under which our result holds. 
The main assumption is that there exists a constant $e(s,u)$ such that 
\[ 
f'_G(\phi,s) = e(s,u)f_G(\phi,u). 
\]
The supplementary assumptions concern expected properties of $L$-packets, as follows. 
We assume that Conjecture \ref{TPC} holds for the parameter $\phi_M$ 
as well as for the endoscopic factorizations $\phi'$ of $\phi$. 
In the setting of classical groups, this follows from Corollary \ref{strong-shahidi-arthur}. 
Thus, we know that there is a unique $\ww_M$-generic member $\pi_{M,\ww} \in \Pi_{\phi_M}$ 
and it satisfies $\pair{\cdot, \pi_{M,\ww}}_{\phi_M} = \1$.
Let $\Pi_\phi$ be the set consisting of the irreducible constituents of the representations $I_P(\pi_M)$, 
as $\pi_M$ runs over $\Pi_{\phi_M}$. 
According to the heredity property (\cite{Rod}, \cite[Corollary 1.7]{CS}), 
there is a unique $\ww$-generic member $\pi_\ww$ of $\Pi_\phi$ and it lies in $I_P(\pi_{M,\ww})$. 
Then the action of the operator $\pair{u,\tl\pi_{M,\ww}} R_P(w_u, \tl\pi_{M,\ww},\phi_M)$ on $I_P(\pi_{M,\ww})$ 
must preserve $\pi_\ww$, and hence acts on it by a scalar. 
We assume that this scalar is $1$. 
For classical groups, this follows from Theorem \ref{main1} (1) together with Lemma \ref{c3}.

\begin{thm}\label{2.5.5}
Under the above assumptions,
we have
\[ 
e(s,u)=1.
\]
In other words, Equation \eqref{A-LIR} in Section \ref{sec.main3} holds.
\end{thm}

We will prove Theorem \ref{2.5.5} in Section \ref{sec.proof-of-255} below.

\begin{rem} \label{appa-o1}
Since we are also interested in the case of the orthogonal group $G = \O_{2n}$, which is disconnected, 
we will make the following slight modification in this situation. 
We will take $f \in C_c^\infty(G^\circ(F))$, 
and $u \in \NN_\phi$ will also be taken with respect to $G^\circ$. 
The distribution $f_G(\phi,u)$ and the stable character $S\Theta_{\phi'}$ will be defined as 
\[ 
\frac{1}{(G:G^\circ)}\sum_{\substack{\pi_M \in \Pi_{\phi_M}\\ w_u\pi_M \cong \pi_M}} 
\tr( \pair{u,\tl\pi_M} R_P(w_u, \tl\pi_M,\phi_M)I_P(\pi_M,f))
\]
and 
\[ 
S\Theta_{\phi'} = \frac{1}{(G':G'^\circ)}\sum_{\pi' \in \Pi_{\phi'}} \pair{1,\pi'}_{\phi'} \Theta_{\pi'}, 
\]
respectively.
Note that, when $G$ is connected, these formulas recover the previous formulas.
Then the modifications for the case of $G = \O_{2n}$ follow 
by applying the following argument to $G^\circ = \SO_{2n}$.
See Remark \ref{packet_SO}. 
In the following proof, we will continue working with a connected reductive group $G$. 
\end{rem}

\subsection{Proof of Theorem \ref{strong-shahidi}}\label{sec.TPC}
For $\pi \in \Pi_\phi$, 
let 
\[
c(\pi,\ww) = 
\left\{
\begin{aligned}
&1 \iif \text{$\pi$ is $\ww$-generic}, \\
&0 \other. 
\end{aligned}
\right. 
\]
We claim that it is enough to show the following statement: 
For any $s \in \Sc_\phi$ with $s \not= 1$, 
we have 
\[
\tag{$\dagger$}\label{dagger}
\sum_{\pi \in \Pi_\phi} c(\pi,\ww) \pair{s,\pi}_\phi = 1.
\]
Indeed, assume that we have shown this statement. 
Let 
\[
X = \{\pi \in \Pi_\phi \,|\, c(\pi,\ww)=1\} \subset \Pi_\phi 
\]
and consider the conjugation-invariant function $f$ on $\Sc_\phi$ defined by 
\[
f(s) = \sum_{\pi \in \Pi_\phi} c(\pi,\ww) \pair{s,\pi}_\phi.
\] 
Equation \eqref{dagger} and the assumption $|\Sc_\phi|>1$ imply that $f \not= 0$, hence $X \not= \emptyset$, 
and $f(1)$ is a natural number greater than $0$ (if $\Sc_\phi$ is abelian, then $f(1) = |X|$). 
The scalar product of $f$ with the trivial character of $\Sc_\phi$ is non-trivial, 
hence $X$ contains the representation $\pi_1$ with $\pair{ \cdot ,\pi_1 }_\phi = \1$. 
Thus, $\pi_1$ is $\ww$-generic. 
Let $X^1=X \setminus \{\pi_1\}$ and let $f^1 = f - \pair{ \cdot, \pi_1 }_\phi$. 
Then $f^1(s) = 0$ for all $s \not= 1$. 
On the one hand, since $f$ is a multiplicity free sum of irreducible characters of $\Sc_\phi$, 
the construction of $f^1$ implies that 
the scalar product of $f^1$ and the trivial character of $\Sc_\phi$ is equal to zero, 
while on the other hand, this scalar product is equal to $f^1(1)$ by evaluation. 
Thus $f^1(1)=0$ and we conclude $f^1=0$, hence $X^1 = \emptyset$.
Therefore $X = \{\pi_1\}$.
\par

This reduces the proof of Theorem \ref{strong-shahidi} to the proof of Equation \eqref{dagger}. 
To prove it, let $1 \not= s \in \Sc_\phi$ 
and let $G'$ and $\phi'$ be the corresponding endoscopic datum and factored parameter, respectively. 
Consider the distributions 
\[ 
f \mapsto f_G'(\phi,s) = \Trans(S\Theta_{\phi'})(f), 
\quad S\Theta_{\phi'} = \sum_{\pi' \in \Pi_{\phi'}} \pair{1,\pi'}_{\phi'} \Theta_{\pi'} 
\]
and 
\[ f_G(\phi,s) = \sum_{\pi \in \Pi_\phi} \pair{s,\pi}_\phi \Theta_\pi(f), 
\]
both linear forms on $C_c^\infty(G(F))$. 
According to \eqref{ECR2} we have 
\[ 
f_G(\phi,s) = f'_G(\phi,s). 
\]
\par

For each $\pi \in \Pi_\phi$, we have the Harish-Chandra local character expansion
\[ 
\Theta_\pi(f) = \sum_O c(\pi,O)\widehat\mu_O(f \circ \exp )
\]
for all $f \in C_c^\infty(G(F))$ with support close to the identity, 
where $O$ runs over the set of nilpotent orbits in $\g(F)=\Lie(G)(F)$. 
To form the Fourier transform $\widehat\mu_O$, 
we have chosen arbitrarily a non-degenerate $G(F)$-invariant symmetric bilinear form $\beta$ on $\g(F)$ 
and a non-trivial unitary character $\psi_F \colon F \rightarrow \C^\times$.
Note that the measures on $G(F)$ (used to define $\Theta_\pi$) and on $O$ (used to define $\mu_O$)
also depend on the choice of $\beta$, as in \cite[Section I.8]{MW_model}.
\par

Putting these together we obtain
\[ 
f_G(\phi,s) = \sum_{\pi \in \Pi_\phi} \pair{s,\pi}_\phi \sum_O c(\pi,O)\widehat\mu_O(f \circ \exp ). 
\]
In the same way, we obtain
\[ 
f'_G(\phi,s) = \sum_{\pi' \in \Pi_{\phi'}} \pair{1,\pi'}_{\phi'} \sum_{O'} c(\pi',O')\widehat\mu_{O'}(f' \circ \exp ), 
\]
where now $O'$ runs over the set of nilpotent orbits in $\g'(F)$. 
We have used another non-degenerate $G'(F)$-invariant symmetric bilinear form $\beta'$ on $\g'(F)$, 
at the moment unrelated to $\beta$.
\par

Thus \eqref{ECR2} implies that, in a neighborhood of the identity, 
\[ 
\Trans\left(\sum_{\pi' \in \Pi_{\phi'}} \pair{ 1,\pi' }_{\phi'} \sum_{O'} c(\pi',O')\widehat\mu_{O'}\right)
= \sum_{\pi \in \Pi_\phi} \pair{s,\pi}_\phi \sum_O c(\pi,O)\widehat\mu_O, 
\]
and using the homogeneity of nilpotent orbital integrals, 
we obtain from this 
\[ 
\Trans\left(\sum_{\pi' \in \Pi_{\phi'}} \pair{1,\pi'}_{\phi'} \sum_{O':\text{regular}} c(\pi',O')\widehat\mu_{O'}\right)
= \sum_{\pi \in \Pi_\phi} \pair{s,\pi}_\phi \sum_{O:\text{regular}} c(\pi,O)\widehat\mu_O, 
\]
see \cite[pp.~325--326]{Sh7}.
\par

We now use a result of M{\oe}glin--Waldspurger \cite{MW_model} 
and its extension to the dyadic case by Varma \cite{V1}. 
There exists a certain regular nilpotent orbit $O_{\psi_F,\beta,\ww}$, 
depending on the Whittaker datum $\ww$, the form $\beta$, and the character $\psi_F$, 
such that 
\[ 
c(\pi,O_{\psi_F,\beta,\ww}) = 
\left\{
\begin{aligned}
&1 \iif \text{$\pi$ is $\ww$-generic}, \\
&0 \other.
\end{aligned}
\right.  
\]
We will specify the orbit $O_{\psi_F,\beta,\ww}$ more precisely below. 
For now, we note that we can apply this to both sides of the above identity. 
On the left-hand side, we use the uniqueness of $\ww'$-generic constituent in $\Pi_{\phi'}$ 
for each Whittaker datum $\ww'$ on $G'$, 
the fact that $\pair{1,\pi'}_{\phi'}=1$ for such a constituent 
(which is the assumption that Conjecture \ref{TPC} holds for $G'$ 
and an application of \cite[Theorem 4.3]{Kal_gc}), 
and the fact that each regular nilpotent orbit $O'$ is equal to $O_{\psi_F,\beta',\ww'}$ 
for some choice of $\ww'$, 
to conclude that 
\[ 
\sum_{\pi' \in \Pi_{\phi'}} \pair{1,\pi'}_{\phi'} \sum_{O':\text{regular}} c(\pi',O')\widehat\mu_{O'} 
= \sum_{O':\text{regular}}\widehat\mu_{O'}. 
\]
Therefore, our identity becomes 
\[ 
\Trans\left(\sum_{O':\text{regular}} \widehat\mu_{O'}\right)
= \sum_{\pi \in \Pi_\phi} \pair{s,\pi}_\phi \sum_{O:\text{regular}} c(\pi,O)\widehat\mu_O.
\]
\par

According to \cite[Corollary 5.5.B]{LS}, 
we have
\[ 
\Trans\left(\sum_{O':\text{regular}} \mu_{O'}\right) = \sum_{O:\text{regular}} \Delta[\ww](O)\mu_O, 
\]
where the transfer factor $\Delta[\ww](O)$ is defined at the end of \cite[(5.1)]{LS}, 
in which it was denoted by $\Delta(u)$, with $u$ a regular unipotent element in $G(F)$. 
We are using here the bijection between regular unipotent orbits in $G(F)$ 
and regular nilpotent orbits in $\g(F)$.
\par

Since the endoscopic transfer commutes with the Fourier transform up to an explicit scalar 
by \cite[p.~91, Conjecture 1]{W1} (which is a theorem due to \cite{W2}, \cite{W3}, \cite{CHL}, \cite{N}), 
we obtain 
\[
\Trans\left(\sum_{O':\text{regular}} \widehat\mu_{O'}\right) 
= \frac{\gamma(\g(F),\beta,\psi_F)}{\gamma(\g'(F),\beta',\psi_F)}
\sum_{O:\text{regular}} \Delta[\ww](O)\widehat\mu_O, 
\]
where $\gamma(\g(F),\beta,\psi_F)$ and $\gamma(\g'(F),\beta',\psi_F)$ are the corresponding Weil indices. 
This identity requires the forms $\beta$ and $\beta'$ to be synchronized, 
as explained in \cite[Section VIII.6]{W1}, and as we now recall. 
Extending scalars from $F$ to $\overline{F}$ identifies 
the space of non-degenerate symmetric bilinear forms on $\g(F)$ that are invariant under $G(F)$ 
with the space of non-degenerate symmetric bilinear forms on $\g(\overline{F})$ 
that are invariant under $G(\overline{F})$ and $\Gamma = \Gal(\overline{F}/F)$. 
If $T \subset G$ is an arbitrary maximal torus, 
then the restriction to its Lie algebra $\tt(\overline{F})$ identifies 
the latter space with the space of non-degenerate symmetric bilinear forms on $\tt(\overline{F})$ 
that are invariant under the absolute Weyl group and $\Gamma$. 
Note that, if $T' \subset G$ is a second maximal torus, 
then the spaces for $\tt$ and $\tt'$ are canonically identified, 
namely by $\Ad(g)$ for any $g \in G(\overline{F})$ that conjugates $T$ to $T'$. 
In this way, taking a maximal torus of $G'$ and transferring it to $G$, 
we can transfer $\beta$ to $\g'(F)$, and we take $\beta'$ to be that transfer.
\par

With this proviso, our identity becomes 
\[ 
\frac{\gamma(\g(F),\beta,\psi_F)}{\gamma(\g'(F),\beta',\psi_F)}
\sum_{O:\text{regular}} \Delta[\ww](O)\widehat\mu_O 
= \sum_{\pi \in \Pi_\phi} \pair{s,\pi}_\phi \sum_{O:\text{regular}} c(\pi,O)\widehat\mu_O.
\]
By \cite[Theorem 5.11]{HC}, we can separate terms. 
Comparing coefficients for the orbit $O_{\psi_F,\beta,\ww}$ 
and applying the results of M{\oe}glin--Waldspurger and Varma recalled above, 
this identity becomes
\[ 
\frac{\gamma(\g(F),\beta,\psi_F)}{\gamma(\g'(F),\beta',\psi_F)} \Delta[\ww](O_{\psi_F,\beta,\ww}) 
= \sum_{\pi \in \Pi_\phi} c(\pi,\ww)\pair{s,\pi}_\phi.
\]
Our goal is to show that the left-hand side of this expression equals $1$. 
As a first step, 
we will rewrite this left-hand side in a way that does not involve the form $\beta$ and the Whittaker datum $\ww$. 
For this, let $S' \subset G'$ be a maximal torus defined over $F$ and let $S \subset G$ be its transfer. 
Decompose $\g = \ss \oplus \rr$, 
where $\rr$ is the direct sum of the root spaces $\g_\alpha$ for all absolute roots $\alpha$ of $S$ in $G$. 
As discussed in \cite[Section 1.3.3]{Kot}, 
there is a canonical quadratic form $Q$ on $\rr$ 
defined as the orthogonal sum of quadratic forms $Q_{\pm\alpha}$ on each $\g_{\alpha} \oplus \g_{-\alpha}$ 
given by 
\[  
[Y_\alpha,Y_{-\alpha}] = Q_{\pm\alpha}(Y_\alpha + Y_{-\alpha}) \cdot H_\alpha, 
\]
where $H_\alpha$ is the coroot for $\alpha$.
We can thus form the Weil index $\gamma(\rr(F),Q,\psi_F)$.
Note however that this form does not always extend to a non-degenerate $G(F)$-invariant quadratic form on $\g(F)$, 
see \cite[Section I.5]{Kot}.
\par

We have the analogous decomposition $\g' = \ss' \oplus \rr'$ 
and the analogous quadratic form $Q'$ on $\rr'$, 
hence also the Weil index $\gamma(\rr'(F),Q',\psi_F)$. 
Finally, we have the maximally split maximal tori (equivalently, minimal Levi subgroups) $T \subset G$ and $T' \subset G'$ 
defined over $F$.
Note that such tori exist since $G$ is assumed to be quasi-split over $F$.

\begin{lem} \label{tfnilp1}
The identity 
\[ 
\frac{\gamma(\g(F),\beta,\psi_F)}{\gamma(\g'(F),\beta',\psi_F)} \Delta[\ww](O_{\psi_F,\beta,\ww}) 
= \frac{\gamma(\rr(F),Q,\psi_F)}{\gamma(\rr'(F),Q',\psi_F)}\ep(1/2,X^*(T)_\C-X^*(T')_\C,\psi_F)
\]
holds for any choices of character $\psi_F$, Whittaker datum $\ww$, 
and non-degenerate $G(F)$-invariant symmetric bilinear form $\beta$, 
where $\beta'$ is the transfer of $\beta$ to $\g'(F)$ according to \cite[Section VIII.6]{W1}. 
In particular, the left-hand side is independent of the choices of $\ww$ and $\beta$. 
Furthermore, both sides are independent of the choice of $\psi_F$.
\end{lem}
\begin{proof}
Following the definition of $\Delta[\ww](O_{\psi_F,\beta,\ww})$ given at the end of \cite[(5.1)]{LS}, 
we have 
\[ 
\Delta[\ww](O_{\psi_F,\beta,\ww}) 
= \frac{\Delta[\ww](\overline\gamma_{G'},\overline\gamma_G)}
{\Delta[\spl'_\infty](\overline\gamma_{G'},\overline\gamma_G)}
\pair{\inv_{\spl'}(O_{\psi_F,\beta,\ww}),s}, \]
where $\spl'$ is an arbitrarily chosen splitting of $G$ and $\spl'_\infty$ is its opposite splitting; 
the left-hand side is independent of this choice, 
as well as of the choice of related elements $\overline\gamma_{G'}$ and $\overline\gamma_G$. 
\par

We will choose $\spl'$ in such a way that $\inv_{\spl'}(O_{\psi_F,\beta,\ww})=1$. 
To see what this entails, 
we review the definition of $O_{\psi_F,\beta,\ww}$ given in \cite{MW_model} and \cite{V1}. 
Choose a splitting $\spl = (B,T,\{X_\alpha\})$ that, together with $\psi_F$, 
induces $\ww$ in the sense of \cite[Section 5.3]{KoSh}. 
Let $U$ be the unipotent radical of $B$ 
and let $\overline{U}$ be the unipotent radical of the $T$-opposite Borel subgroup to $B$. 
We write $\uu$ and $\overline{\uu}$ for the Lie algebras of $U$ and $\overline{U}$, respectively. 
Then $O_{\psi_F,\beta,\ww}$ is the $G(F)$-orbit of a regular nilpotent element 
$Y_{\psi_F,\beta,\ww} \in \overline{\uu}(F)$ whose property is that the character 
\[ 
X \mapsto \psi_F(\beta(X,Y_{\psi_F,\beta,\ww}))
\]
of $\uu(F)$ equals the composition of the exponential map $\exp \colon \uu(F) \rightarrow U(F)$ 
and the generic character $U(F) \rightarrow \C^\times$ that makes up the Whittaker datum $\ww$.  
One can check that $Y_{\psi_F,\beta,\ww}$ is given by 
\[ 
Y_{\psi_F,\beta,\ww} = \sum_\alpha \beta(X_\alpha,X_{-\alpha})^{-1}X_{-\alpha}, 
\]
where $X_{-\alpha} \in \g_{-\alpha}$ is determined by $[X_{\alpha},X_{-\alpha}] = H_\alpha$. 
Thus, if we take $\spl'$ to be $(\overline{B}, T, \{\beta(X_\alpha, X_{-\alpha})^{-1}X_{-\alpha}\})$, 
then $\inv_{\spl'}(O_{\psi_F,\beta,\ww})=1$.
Hence we have 
\[ 
\Delta[\ww](O_{\psi_F,\beta,\ww}) 
= \frac{\Delta[\ww](\overline\gamma_{G'},\overline\gamma_G)}
{\Delta[\spl'_\infty](\overline\gamma_{G'},\overline\gamma_G)}
\]
for this particular choice of $\spl'$. 
On the other hand, by the definition of the Whittaker normalization \cite[Section 5.3]{KoSh}, \cite{KoSh2},
\[
\Delta[\ww](\overline\gamma_{G'},\overline\gamma_G) 
= \ep(1/2,X^*(T)_\C-X^*(T')_\C,\psi_F) \cdot \Delta[\spl](\overline\gamma_{G'},\overline\gamma_G).
\]
Thus we need to compute the ratio between 
$\Delta[\spl](\overline\gamma_{G'},\overline\gamma_G)$ 
and $\Delta[\spl'_\infty](\overline\gamma_{G'},\overline\gamma_G)$. 
Both of these being transfer factors, 
this ratio does not depend on the elements $\overline\gamma_{G'}$ and $\overline\gamma_G$, 
as long as they are related (otherwise both factors are zero). 
So we choose arbitrarily a pair of related strongly regular semisimple elements 
$\overline\gamma_{G'} \in S'(F)$ and $\overline\gamma_G \in S(F)$.
Note that this choice determines an admissible isomorphism $S' \xrightarrow{\sim} S$, 
namely the unique one that maps $\overline\gamma_{G'}$ to $\overline\gamma_G$, 
and hence determines an inclusion of the absolute root systems $R(S',G') \rightarrow R(S,G)$.
\par

The only term of the transfer factor that depends on the splitting is the term $\Delta_{I}$, 
where the splitting enters the definition of the splitting invariant. 
We have the relation $\spl'_\infty = b \cdot \spl$, 
where $b_\alpha = \beta(X_\alpha,X_{-\alpha})$ for all $\alpha \in R(T,G)$. 
Then \cite[Lemma 5.1]{Kal_gc} shows that rescaling the splitting has the same effect as rescaling the $a$-data, 
which also enters the definition of the splitting invariant. 
On the other hand, \cite[Lemma 3.2.C]{LS} shows how the transfer factor changes when the $a$-data is rescaled. 
With this, we obtain
\[ 
\Delta[\ww](O_{\psi_F,\beta,\ww}) 
= \ep(1/2,X^*(T)_\C-X^*(T')_\C,\psi_F) 
\prod_{\alpha \in R(S,G/G')_\sym/\Gamma} \kappa_\alpha(b_\alpha), 
\]
where $R(S,G/G')$ is a shorthand notation for $R(S,G) \setminus R(S',G')$, 
the subscript ``sym'' indicates those roots that are \emph{symmetric}, 
i.e. those $\alpha$ with $-\alpha \in \Gamma \cdot \alpha$, 
and $\kappa_\alpha$ is the sign character of the quadratic extension $F_\alpha/F_{\pm\alpha}$.
\par

Using that the decompositions $\g' = \ss' \oplus \rr'$ and $\g = \ss \oplus \rr$ are orthogonal for $\beta'$ and $\beta$, 
respectively,
we see that the Weil indices decompose accordingly as products, 
and since $\beta|_\ss = \beta'|_{\ss'}$ by the synchronization of $\beta$ and $\beta'$, 
the left-hand side of the equation in the statement of the lemma becomes 
\[ 
\frac{\gamma(\rr(F),\beta,\psi_F)}{\gamma(\rr'(F),\beta',\psi_F)}
\ep(1/2,X^*(T)_\C-X^*(T')_\C,\psi_F)
\prod_{\alpha \in R(S,G/G')_\sym/\Gamma} \kappa_\alpha(b_\alpha). 
\]
According to \cite[Lemma 4.8]{Kal_ep} and the fact that $Q(X_\alpha+X_{-\alpha}) = 1$, 
we have
\[ 
\frac{\gamma(\rr(F),\beta,\psi_F)}{\gamma(\rr(F),Q,\psi_F)} 
= \prod_{\alpha \in R(S,G)_\sym/\Gamma} \kappa_\alpha(b_\alpha). 
\]
The analogous identity holds with $(G,S)$ replaced by $(G',S')$. 
This proves the identity claimed in the lemma. 
The independence of the left-hand side of the choices of $\beta$ and $\ww$ follows from that identity. 
To see the independence of both sides of the choice of $\psi_F$, 
note that any other choice of character is of the form $(a\psi_F)(x) = \psi_F(ax)$ for some $a \in F^\times$, 
but one sees directly from the definitions that replacing $\psi_F$ and $\beta$ by $a\psi_F$ and $a^{-1}\cdot\beta$ 
does not change the Weil indices or the orbit $O_{\psi_F,\beta,\ww}$.
See Section \ref{sub.weil-eps} below for the definition of the Weil indices.
\end{proof}

To complete the proof of Theorem \ref{strong-shahidi}, 
we need to show that the right-hand side of the identity of Lemma \ref{tfnilp1} equals one. 
This is accomplished by the following lemma.

\begin{lem} \label{tfnilp2}
Let $S' \subset G'$ be any maximal torus defined over $F$ and let $S \subset G$ be its transfer. 
Decompose the Lie algebras $\g = \ss \oplus \rr$ and $\g' = \ss' \oplus \rr'$. 
Let $Q$ (\resp $Q'$) be the canonical quadratic form on $\rr$ (\resp $\rr'$) 
described before the statement of Lemma \ref{tfnilp1}. 
Let $T' \subset G'$ and $T \subset G$ be minimal Levi subgroups defined over $F$. 
Then
\[ 
\frac{\gamma(\rr(F),Q,\psi_F)}{\gamma(\rr'(F),Q',\psi_F)}\ep(1/2,X^*(T)_\C-X^*(T')_\C,\psi_F) 
= 1.
\]
\end{lem}
\begin{proof}
From \cite[Theorem 1.1]{Kot}, 
we have
\[ 
\gamma(\rr(F),Q,\psi_F) = \ep(1/2,X^*(S)_\C-X^*(T)_\C,\psi_F). 
\]
Applying this result once to $G$ and the torus $S$, and once to $G'$ and the torus $S'$, 
and using that $X^*(S) \cong X^*(S')$ as $\Gamma$-modules, 
we obtain the desired result. 
\end{proof}

We obtain Equation \eqref{dagger},
and hence complete the proof of Theorem \ref{strong-shahidi}.
\par

\begin{rem}
In the case of classical groups, 
one can also compute the left-hand side in Lemma \ref{tfnilp2} explicitly 
and check that it equals $1$. 
See Section \ref{sub.weil-eps} below.
\end{rem}
    
\subsection{Proof of Theorem \ref{2.5.5}}\label{sec.proof-of-255}
The argument is very close to that of the proof of Theorem \ref{strong-shahidi}. 
Since Theorem \ref{2.5.5} is an important part of the argument of \cite{Ar}, 
we will still present all steps of the proof, 
but we will be more brief with the justifications, 
which are the same as in the previous proof.
\par

We consider the same distribution $f'_G(\phi,s)$ as in that proof, 
but replace the distribution $f_G(\phi,s)$ by the distribution $f_G(\phi,u)$ 
from the statement of Theorem \ref{2.5.5}. 
This distribution can be expanded as 
\[ 
f_G(\phi,u) = \sum_{\pi \in \Pi_\phi} c(\pi) \Theta_\pi(f), 
\]
where $c(\pi)$ is defined as follows. 
Let $\pi \otimes M_\pi$ be the maximal $\pi$-isotypic constituent of $I_P(\pi_M)$. 
By Schur's lemma $\pair{u,\tl\pi_M}R_P(w_u,\tl\pi_M,\phi_M)$ 
induces an operator on the finite-dimensional $\C$-vector space $M_\pi$ 
and $c(\pi)$ is the trace of that operator. 
Our assumptions imply that $M_{\pi_\ww}$ is $1$-dimensional and $c(\pi_\ww)=1$.
\par

Instead of \eqref{ECR2} in Theorem \ref{strong-shahidi}, 
we now use the assumed identity
\[ 
f'_G(\phi,s) = e(s,u)f_G(\phi,u).
\]
Expanding both sides using the Harish-Chandra local character expansion 
and using the homogeneity of nilpotent orbital integrals, 
we arrive at the identity
\[ 
\Trans\left(\sum_{\pi' \in \Pi_{\phi'}} \pair{1,\pi'}_{\phi'} \sum_{O':\text{regular}} c(\pi',O')\widehat\mu_{O'}\right)
= e(s,u)\sum_{\pi \in \Pi_\phi} c(\pi)\sum_{O:\text{regular}} c(\pi,O)\widehat\mu_O. 
\]
Here we have again fixed a non-trivial character $\psi_F \colon F \rightarrow \C^\times$ 
and a non-degenerate $G(F)$-invariant symmetric bilinear form $\beta$ on $\g(F)$, 
which we have transferred to a form $\beta'$ on $\g'(F)$, 
and have used these to form the Fourier transforms. 
We consider again the nilpotent $G(F)$-orbit $O_{\psi_F,\beta,\ww}$ in $\g(F)$, 
which according to the results of M{\oe}glin--Waldspurger and Varma has the property 
\[ 
c(\pi,O_{\psi_F,\beta,\ww}) =
\left\{
\begin{aligned}
&1 \iif \pi = \pi_{\ww},\\
&0 \other.
\end{aligned}
\right. 
\]
Using the assumed validity of Conjecture \ref{TPC} for $G'$, 
which gives the uniqueness of $\ww'$-generic constituent in $\Pi_{\phi'}$ for each Whittaker datum $\ww'$ on $G'$ 
and the fact that $\pair{1,\pi'}_{\phi'} = 1$ for such a constituent, 
and using further the fact that each regular nilpotent orbit $O'$ is equal to $O_{\psi_F,\beta',\ww'}$ 
for some choice of $\ww'$, we conclude that
\[ 
\sum_{\pi' \in \Pi_{\phi'}} \pair{1,\pi'}_{\phi'} \sum_{O':\text{regular}} c(\pi',O')\widehat\mu_{O'}
= \sum_{O':\text{regular}} \widehat\mu_{O'}.
\]
Therefore, our identity becomes 
\[ 
\Trans\left( \sum_{O':\text{regular}} \widehat\mu_{O'}\right)
= e(s,u)\sum_{\pi \in \Pi_\phi} c(\pi)\sum_{O:\text{regular}} c(\pi,O)\widehat\mu_O. 
\]
Applying \cite[Corollary 5.5.B]{LS}, \cite[p.~91, Conjecture 1]{W1}, \cite{W2}, \cite{W3}, \cite{CHL}, \cite{N}, 
we have
\[ 
\Trans\left( \sum_{O':\text{regular}} \widehat\mu_{O'}\right)
= \frac{\gamma(\g(F),\beta,\psi_F)}{\gamma(\g'(F),\beta',\psi_F)}
\sum_{O:\text{regular}} \Delta[\ww](O)\widehat\mu_O. 
\]
Separating terms using \cite[Theorem 5.11]{HC} 
and comparing coefficients for the orbit $O_{\psi_F,\beta,\ww}$, 
we arrive at 
\[ 
e(s,u) = \frac{\gamma(\g(F),\beta,\psi_F)}{\gamma(\g'(F),\beta',\psi_F)}\Delta[\ww](O_{\psi_F,\beta,\ww}).
\]
By Lemmas \ref{tfnilp1} and \ref{tfnilp2}, the right-hand side is equal to $1$.
This completes the proof of Theorem \ref{2.5.5}.

\subsection{An explicit computation of Weil indices and $\ep$-factors}\label{sub.weil-eps}
The proofs of Theorems \ref{strong-shahidi} and \ref{2.5.5} rely on the key Lemma \ref{tfnilp2}. 
We gave a proof of this lemma for general reductive groups using the work of Kottwitz \cite{Kot}. 
For classical groups, this lemma can also be verified by an explicit computation, 
which we shall give here. 
\par

For explicit computation, 
the left-hand side of Lemma \ref{tfnilp2} is inconvenient, 
because it stipulates that one has to work with a maximal torus of $G$ that transfers to $G'$, 
while the computation of the Weil index is most convenient when one uses as a maximal torus a minimal Levi subgroup of $G$. 
So we return to the left-hand side of Lemma \ref{tfnilp1}. 
As shown there, 
the left-hand side is independent of the choices of the character $\psi_F$, the Whittaker datum $\ww$, 
and the form $\beta$ (recall that $\beta'$ is the transfer of $\beta$). 
We are thus free to choose these objects in a convenient way. 
We first choose a character $\psi_F$ and an $F$-splitting $\spl = (B,T,\{X_{\alpha}\})$ 
and then take the Whittaker datum $\ww$ determined by $\spl$ and $\psi_F$ as in \cite[Section 5.3]{KoSh}. 
We have the $G(F)$-orbit $O_{\psi_F,\beta,\ww}$ of the regular nilpotent element 
\[
Y_{\psi_F,\beta,\ww} = \sum_{\alpha} \beta(X_\alpha,X_{-\alpha})^{-1} X_{-\alpha}.
\]
It may be a priori different from the $G(F)$-orbit $O_\infty$ of the regular nilpotent element 
\[
Y_\infty = \sum_{\alpha} X_{-\alpha}
\]
that is ``associated'' to the splitting $\spl_\infty = (\overline{B},T,\{X_{-\alpha}\})$ opposite to $\spl$. 
In our computations, we will show that $\beta$ can be chosen such that
\[
\tag{$\ddagger$}\label{ddagger}
O_{\psi_F,\beta,\ww} = O_\infty.
\]
This identity implies that $\inv_{\spl_\infty}(O_{\psi_F,\beta,\ww}) = 1$. 
We can then apply the computation of $\Delta[\ww](O_{\psi_F,\beta,\ww})$ given in the proof of Lemma \ref{tfnilp1} 
but with $\spl' = \spl_\infty$ and obtain
\[
\Delta[\ww](O_{\psi_F,\beta,\ww})
= \frac{\Delta[\ww](\overline{\gamma}_{G'},\overline{\gamma}_{G})}
{\Delta[\spl](\overline{\gamma}_{G'},\overline{\gamma}_{G})}
= \ep(1/2,X^*(T)_\C - X^*(T')_\C, \psi_F),
\]
thereby reducing the desired identity to 
\[
\frac{\gamma(\g(F),\beta,\psi_F)}{\gamma(\g'(F),\beta',\psi_F)} 
\cdot \frac{\ep(1/2,X^*(T)_\C,\psi_F)}{\ep(1/2,X^*(T')_\C,\psi_F)}
= 1
\]
for those particular choices of $\psi_F$, $\beta$ and $\spl$. 
We will show this identity for classical groups by explicitly computing all terms.
\par

First, we compute the relevant Weil indices. 
To recall the definition and properties of Weil indices, 
we introduce some notation.
Let $V$ be a finite dimensional vector space over $F$ 
equipped with a non-degenerate symmetric bilinear form $\beta$.
Following \cite[Section VIII.1]{W1}, we put
\[
\gamma(V,\beta, \psi_F) = \frac{I}{|I|}
\quad\text{with}\quad
I = \int_L \psi_F  \left( \frac{\beta(x,x)}{2} \right) dx,
\]
where $L$ is a sufficiently large lattice of $V$ and $dx$ is a Haar measure on $V$.
Note that $\gamma(V,\beta, \psi_F)$ does not depend on the choices of $L$ and $dx$.
When $Q(x) = \half{1}\beta(x,x)$ is the associated quadratic form, 
we also write $\gamma(V,Q, \psi_F) = \gamma(V,\beta, \psi_F)$.
If $V = F$ and $\beta(x,y) = 2axy$ for some $a \in F^\times$, 
then we have $\gamma(V,\beta,\psi_F) = \gamma(a \psi_F)$ with the convention in \cite[Appendix]{Rao}.
Here $a\psi_F$ is the non-trivial additive character of $F$ given by $a \psi_F(x) = \psi_F(ax)$.
Note that $\gamma(\psi_F) \gamma(-\psi_F) = 1$ and
\[
\frac{\gamma(\psi_F)}{\gamma(a \psi_F)} = \ep(1/2, \eta_a, \psi_F)  
\]
(see \cite{Kah}, \cite{Sz}).
Here $\eta_a$ is the (possibly trivial) quadratic character of $F^\times$ given by $\eta_a(x) = (a,x)_F$, 
where $(\cdot,\cdot)_F$ is the quadratic Hilbert symbol of $F$, 
so that $\eta_a$ is associated to $F(\sqrt{a})/F$ by the local class field theory.
One can prove that if $(V, \beta)$ is a direct sum of $(V_0,\beta_0)$ and the hyperbolic plane $\mathbb{H}$, 
then $\gamma(V,\beta,\psi_F) = \gamma(V_0,\beta_0,\psi_F)$.
\par

Recall that $U$ is the unipotent radical of the Borel subgroup $B = TU$, 
and $\uu$ is its Lie algebra.
Since $G$ is quasi-split, 
$\uu \oplus \overline{\uu}$ is an orthogonal direct sum of hyperbolic planes 
and hence 
\[
\gamma(\g(F),\beta, \psi_F) = \gamma(\tt(F), \beta,\psi_F).
\]
\par

Now we take the splitting $\spl$ given in Section \ref{secA.groups} 
and compute $\gamma(\g(F),\beta, \psi_F)$ explicitly. 

\begin{description}
\setlength{\itemsep}{10pt}
\item[The case of symplectic groups]
Suppose that $G(F) = \Sp_{2n}(F)$, so that 
\[
\tt(F) = \{ \diag(X_1, \dots, X_n, -X_n, \dots, -X_1) \,|\, X_1, \dots, X_n \in F \}.
\]
Take the non-degenerate $G(F)$-invariant symmetric bilinear form $\beta$ on $\g(F)$ 
given by $\beta(X, Y) = \tr(XY)$. 
Noting that $X_{-\alpha_i} = E_{i+1,i}-E_{2n+1-i,2n-i}$ for $1 \leq i \leq n-1$ 
and $X_{-\alpha_n} = E_{n+1,n}$, we have
\[
Y_{\psi_F,\beta,\ww} = 2^{-1}(X_{-\alpha_1} + \dots + X_{-\alpha_{n-1}}) + X_{-\alpha_n} 
= \Ad(t_0) \left( \sum_{i=1}^n X_{-\alpha_i} \right) \in O_\infty,
\]
where
\[
t_0 = \diag(2^{n-1}, 2^{n-2}, \dots, 1, 1, \dots, 2^{-n+2}, 2^{-n+1}) \in T(F).
\]
Hence $\beta$ satisfies the condition \eqref{ddagger}.
Since 
\[
\beta(X,Y) = 2(X_1Y_1 + \dots + X_nY_n)
\]
for $X, Y \in \tt(F)$, we have
\[
\gamma(\g(F), \beta,\psi_F) = \gamma(\psi_F)^n.
\]

\item[The case of odd special orthogonal groups]
Suppose $G(F) = \SO_{2n+1}(F)$, so that 
\[
\tt(F) = \{ \diag(X_1, \dots, X_n, 0, -X_n, \dots, -X_1) \, | \, X_1, \dots, X_n \in F \}.
\]
Take the non-degenerate $G(F)$-invariant symmetric bilinear form $\beta$ on $\g(F)$ 
given by $\beta(X,Y) = \tr(XY)$.
Noting that $X_{-\alpha_i} = E_{i+1,i}-E_{2n+2-i,2n+1-i}$ for $1 \leq i \leq n-1$ and $X_{-\alpha_n} = 2(E_{n+1,n} - E_{n+2,n+1})$, we have
\[
Y_{\psi_F,\beta,\ww} =  2^{-1}(X_{-\alpha_1} + \dots + X_{-\alpha_{n-1}}) + 2^{-2} X_{-\alpha_n}
= \Ad(t_0) \left( \sum_{i=1}^n X_{-\alpha_i} \right) \in O_\infty,
\]
where 
\[
t_0 = \diag(2^{n+1}, 2^n, \dots, 2^2, 1, 2^{-2}, \dots, 2^{-n}, 2^{-n-1}) \in T(F).
\]
Hence $\beta$ satisfies the condition \eqref{ddagger}.
Since 
\[
\beta(X,Y) = 2(X_1Y_1 + \dots + X_nY_n)
\]
for $X, Y \in \tt(F)$, we have
\[
\gamma(\g(F), \beta, \psi_F) = \gamma(\psi_F)^n.
\]

\item[The case of even special orthogonal groups]
Suppose that $G(F) = \SO_{2n}^\eta(F)$,
so that $\tt(F)$ is equal to 
\[
\left\{
\begin{aligned}
&\{ \diag(X_1, \dots, X_n, -X_n, \dots, -X_1) \, | \, X_1, \dots, X_n \in F \} \iif \eta = \1, \\
&\{ \diag(X_1, \dots, X_n, -X_n, \dots, -X_1) \, | \, X_1, \dots, X_{n-1} \in F, \, X_n \in K_0 \} \other, 
\end{aligned}
\right.
\]
where $K$ is the quadratic extension of $F$ associated to $\eta$ by the local class field theory 
and $K_0$ is the set of trace zero elements in $K$.
Take the non-degenerate $G(F)$-invariant symmetric bilinear form $\beta$ on $\g(F)$ 
given by $\beta(X,Y) = \tr(XY)$.
Noting that $X_{-\alpha_i} = E_{i+1,i}-E_{2n+1-i,2n-i}$ for $1 \leq i \leq n-1$ 
and $X_{-\alpha_n} = E_{n+1,n-1} - E_{n+2,n}$, we have
\[
Y_{\psi_F,\beta,\ww} = 2^{-1}(X_{-\alpha_1} + \dots + X_{-\alpha_n})
= \Ad(t_0) \left( \sum_{i=1}^n X_{-\alpha_i} \right) 
\in O_\infty,
\]
where 
\[
t_0 = \diag(2^{n-1}, 2^{n-2}, \dots, 1, 1, \dots, 2^{-n+2}, 2^{-n+1}) \in T(F).
\]
Hence $\beta$ satisfies the condition \eqref{ddagger}.
Since
\[
\beta(X,Y) =
\left\{
\begin{aligned}
&2(X_1Y_1 + \dots + X_nY_n) \iif \eta = \1, \\
&2(X_1Y_1 + \dots + X_{n-1}Y_{n-1}) + \tr_{K/F}(X_nY_n) \other
\end{aligned}
\right. 
\]
for $X, Y \in \tt(F)$, 
we have
\[
\gamma(\g(F), \beta, \psi_F) = \gamma(\psi_F)^{n-1} \gamma(a \psi_F),
\]
where we write $\eta = \eta_a$ with $a \in F^\times$ so that $K = F(\sqrt{a})$.
Note that $K_0=F\sqrt{a}$.

\item[The case of unitary groups]
Suppose that $G(F) = \U_n$, so that
\[
\tt(F) =
\left\{
\begin{aligned}
&\{ \diag(X_1, \dots, X_r, -\overline{X_r}, \dots, -\overline{X_1}) \,|\, X_1, \dots, X_r \in E \}, \\
&\{ \diag(X_1, \dots, X_{r+1}, -\overline{X_r}, \dots, -\overline{X_1}) \,|\, 
X_1, \dots, X_r \in E, \, X_{r+1} \in E_0 \}
\end{aligned}
\right.
\]
according to $n=2r$ or $n=2r+1$, 
where $E_0$ is the set of trace zero elements in $E$.
Write $E = F(\sqrt{a})$ with $a \in F^\times$ and put $\eta = \eta_a$.
Take the non-degenerate $G(F)$-invariant symmetric bilinear form $\beta$ on $\g(F)$ 
given by $\beta(X,Y) = \half{1} \tr_{E/F}(\tr(XY))$.
Noting that $X_{-\alpha_i} = E_{i+1,i}$ for $1 \leq i \leq n-1$ and that
\[
u_1X_{\alpha_1}+\dots+u_{n-1}X_{\alpha_{n-1}} \in \mathfrak{u}(F)
\iff \overline{u_i} = u_{n-i} \quad(1 \leq i \leq n-1)
\]
for $u_1, \dots, u_{n-1} \in E$, we have
\[
Y_{\psi_F,\beta,\ww} = \sum_{i=1}^{n-1} X_{-\alpha_i} \in O_\infty.
\]
Hence $\beta$ satisfies the condition \eqref{ddagger}.
Since
\[
\beta(X,Y) =
\left\{
\begin{aligned}
&\tr_{E/F}(X_1Y_1 + \dots + X_rY_r) \iif n = 2r, \\
&\tr_{E/F}(X_1Y_1 + \dots + X_rY_r + (1/2)X_{r+1}Y_{r+1}) \iif n=2r+1
\end{aligned}
\right. 
\]
for $X, Y \in \tt(F)$, 
we have
\[
\gamma(\g(F), \beta, \psi_F) = 
\left\{
\begin{aligned}
&\gamma(\psi_F)^r \gamma(a \psi_F)^r \iif n = 2r, \\
&\gamma(\psi_F)^r \gamma(a \psi_F)^r\gamma((a/2) \psi_F) \iif n=2r+1.
\end{aligned}
\right. 
\]
\end{description}

The computation of the Weil indices for the various classical groups is now complete. 
We turn to the proof of the equation
\[
\frac{\gamma(\g(F),\beta,\psi_F)}{\gamma(\g'(F),\beta',\psi_F)} 
\cdot \frac{\ep(1/2,X^*(T)_\C,\psi_F)}{\ep(1/2,X^*(T')_\C,\psi_F)}
= 1.
\]
Recall that $G'(F)$ is of the form 
\[
G'(F) =
\left\{
\begin{aligned}
& \Sp_{2n_1}(F) \times \SO_{2n_2}^{\eta}(F) \iif G(F) = \Sp_{2n}(F), \\
& \SO_{2n_1+1}(F) \times \SO_{2n_2+1}(F) \iif G(F) = \SO_{2n+1}(F), \\
& \SO_{2n_1}^{\eta_1}(F) \times \SO_{2n_2}^{\eta_2}(F) \iif G(F) = \SO_{2n}^\eta(F), \\
& \U_{n_1} \times \U_{n_2} \iif G(F) = \U_n
\end{aligned}
\right.
\]
with $n_1+n_2=n$ and $\eta_1 \eta_2 = \eta$ (see e.g., \cite[Section 1.8]{W5}).
Recall also that $\beta'$ is the transfer of $\beta$ to $\g'(F)$ according to \cite[Section VIII.6]{W1}.
Explicitly, if we write $G'(F) = G_1(F) \times G_2(F)$, then we have
\[
\beta' = \beta_1 \oplus \beta_2,
\]
where $\beta_i$ is the bilinear form on $\g_i(F)$ as above.
\par

\begin{description}
\setlength{\itemsep}{10pt}
\item[The case of symplectic groups]
Suppose that $G(F) = \Sp_{2n}(F)$ and $G'(F) = \Sp_{2n_1}(F) \times \SO_{2n_2}^{\eta}(F)$ with $\eta = \eta_a$.
Then we have
\[
\gamma(\g(F), \beta, \psi_F) = \gamma(\psi_F)^n, \quad 
\gamma(\g'(F), \beta', \psi_F) = \gamma(\psi_F)^{n-1} \gamma(a \psi_F)
\]
and 
\[
\ep(1/2, X^*(T)_\C,\psi_F) = 1, \quad
\ep(1/2, X^*(T')_\C,\psi_F) = \ep(1/2, \eta_a, \psi_F).
\]
Hence 
\[
\frac{\gamma(\g(F),\beta, \psi_F)}{\gamma(\g'(F), \beta', \psi_F)} 
\cdot \frac{\ep(1/2, X^*(T)_\C, \psi_F)}{\ep(1/2, X^*(T')_\C, \psi_F)}
= \frac{\gamma(\psi_F)}{\gamma(a \psi_F)} \cdot \frac{1}{\ep(1/2, \eta_a, \psi_F)} = 1.
\]

\item[The case of odd special orthogonal groups]
Suppose that $G(F) = \SO_{2n+1}(F)$ and $G'(F) = \SO_{2n_1+1}(F) \times \SO_{2n_2+1}(F)$.
Then we have
\[
\gamma(\g(F), \beta, \psi_F) = \gamma(\g'(F), \beta', \psi_F) = \gamma(\psi_F)^n
\]
and 
\[
\ep(1/2, X^*(T)_\C,\psi_F) = \ep(1/2, X^*(T')_\C,\psi_F) =  1.
\]
Hence
\[
\frac{\gamma(\g(F),\beta,\psi_F)}{\gamma(\g'(F),\beta',\psi_F)} 
\cdot \frac{\ep(1/2,X^*(T)_\C,\psi_F)}{\ep(1/2,X^*(T')_\C,\psi_F)}
= 1.
\]

\item[The case of even special orthogonal groups]
Suppose that $G(F) = \SO_{2n}^\eta(F)$ and $G'(F) = \SO_{2n_1}^{\eta_1}(F) \times \SO_{2n_2}^{\eta_2}(F)$ 
with $\eta = \eta_a$ and $\eta_i = \eta_{a_i}$.
Then we have
\[
\gamma(\g(F), \beta, \psi_F) = \gamma(\psi_F)^{n-1} \gamma(a \psi_F), 
\quad 
\gamma(\g'(F), \beta', \psi_F) = \gamma(\psi_F)^{n-2} \gamma(a_1 \psi_F) \gamma(a_2 \psi_F)
\]
and
\begin{align*}
\ep(1/2, X^*(T)_\C,\psi_F) 
&= \ep(1/2, \eta, \psi_F), 
\\
\ep(1/2, X^*(T')_\C,\psi_F) 
&= \ep(1/2, \eta_1, \psi_F) \ep(1/2, \eta_2, \psi_F).
\end{align*}
Hence 
\begin{align*}
&\frac{\gamma(\g(F), \beta, \psi_F)}{\gamma(\g'(F), \beta', \psi_F)} 
\cdot \frac{\ep(1/2, X^*(T)_\C, \psi_F)}{\ep(1/2, X^*(T')_\C, \psi_F)}
\\&= 
\frac{\gamma(\psi_F) \gamma(a \psi_F)}{\gamma(a_1 \psi_F) \gamma(a_2 \psi_F)}
\cdot \frac{\ep(1/2, \eta, \psi_F)}{\ep(1/2, \eta_1, \psi_F) \ep(1/2, \eta_2, \psi_F)}
= 1. 
\end{align*}

\item[The case of odd unitary groups]
Suppose that $G(F) = \U_n$ and $G'(F) = \U_{n_1} \times \U_{n_2}$ with $n=2r+1$.
Then we have
\[
\gamma(\g(F), \beta, \psi_F) = \gamma(\g'(F), \beta', \psi_F) 
= \gamma(\psi_F)^r \gamma(a \psi_F)^r \gamma((a/2) \psi_F)
\]
and 
\[
\ep(1/2, X^*(T)_\C,\psi_F) 
= \ep(1/2, X^*(T')_\C,\psi_F) 
= \ep(1/2, \eta, \psi_F)^{r+1}.
\]
Hence
\[
\frac{\gamma(\g(F),\beta,\psi_F)}{\gamma(\g'(F),\beta',\psi_F)} 
\cdot \frac{\ep(1/2,X^*(T)_\C,\psi_F)}{\ep(1/2,X^*(T')_\C,\psi_F)}
= 1.
\]

\item[The case of even unitary groups]
Suppose that $G(F) = \U_n$ and $G'(F) = \U_{n_1} \times \U_{n_2}$ with $n=2r$.
If $n_1$ and $n_2$ are even, then we have
\[
\gamma(\g(F), \beta, \psi_F) = \gamma(\g'(F), \beta', \psi_F) = \gamma(\psi_F)^r \gamma(a \psi_F)^r
\]
and 
\[
\ep(1/2, X^*(T)_\C,\psi_F) 
= \ep(1/2, X^*(T')_\C,\psi_F) 
= \ep(1/2, \eta, \psi_F)^r.
\]
Hence
\[
\frac{\gamma(\g(F),\beta,\psi_F)}{\gamma(\g'(F),\beta',\psi_F)} 
\cdot \frac{\ep(1/2,X^*(T)_\C,\psi_F)}{\ep(1/2,X^*(T')_\C,\psi_F)}
= 1.
\]
If $n_1$ and $n_2$ are odd, then we have
\[
\gamma(\g(F), \beta, \psi_F) = \gamma(\psi_F)^r \gamma(a \psi_F)^r, \quad
\gamma(\g'(F), \beta', \psi_F) = \gamma(\psi_F)^{r-1} \gamma(a \psi_F)^{r-1} \gamma((a/2) \psi_F)^2
\]
and 
\[
\ep(1/2, X^*(T)_\C,\psi_F) = \ep(1/2, \eta, \psi_F)^r, \quad
\ep(1/2, X^*(T')_\C,\psi_F) = \ep(1/2, \eta, \psi_F)^{r+1}.
\]
Hence 
\begin{align*}
\frac{\gamma(\g(F),\beta,\psi_F)}{\gamma(\g'(F),\beta',\psi_F)} 
\cdot \frac{\ep(1/2,X^*(T)_\C,\psi_F)}{\ep(1/2,X^*(T')_\C,\psi_F)}
&= 
\frac{\gamma(\psi_F) \gamma(a \psi_F)}{\gamma((a/2) \psi_F)^2}
\cdot \frac{1}{\ep(1/2, \eta, \psi_F)}
\\&= \frac{\ep(1/2, \eta_{a/2}, \psi_F)^2}{\ep(1/2, \eta, \psi_F)^2} 
= \frac{\eta_{a/2}(-1)}{\eta(-1)}
\\&= \frac{(a/2,-1)_F}{(a,-1)_F} 
= (2,-1)_F = 1.
\end{align*}
Here, we use the fact that $(x,1-x)_F = 1$ for $x \not= 0, 1$.
\end{description}

This completes the proof of Lemma \ref{tfnilp2} for classical groups by explicit computation.

\section{Endoscopic character relations for the archimedean case}\label{sec.TECR}
In this appendix, 
we will argue that \cite[Theorem 2.2.1(a)]{Ar} holds for $F=\R$ and tempered parameters $\psi=\phi$, 
and that \cite[Theorem 2.2.4]{Ar} holds for $F=\R$ and discrete parameters $\psi=\phi$. 
The validity of these theorems is assumed at various places in \cite{Ar} 
with the remark that they will follow from a forthcoming work of Shelstad and Mezo. 
In the meantime, the work of Shelstad \cite{She12} has appeared, 
which proves the transfer of functions in twisted endoscopy for $F=\R$, 
and the works of Mezo \cite{Mez13, Mez16}, 
which prove a weaker version of the desired theorems: 
the character identities are shown to hold up to a scalar. 
\par

In this appendix, we will use different sources. For the proof of \cite[Theorem 2.2.4]{Ar} we will use the recent work \cite{KM26} (which treats a general class of disconnected real reductive groups and $L$-parameters 
that are discrete for the identity component) and will combine this work with a global argument that reduces the general tempered case to the discrete case. 
For the proof of \cite[Theorem 2.2.1(a)]{Ar} such a strategy cannot work, 
because the group $\GL_N$ has no discrete parameters when $F=\R$ and $N>2$. 
Here we will build on \cite{AMR} and \cite{Cl}. 
This will again involve a global reduction argument, 
but now to the case that $\phi$ is discrete for the endoscopic group. 
It will further involve a local argument that allows for a variation of the $L$-homomorphism.

\subsection{Theorem 2.2.4 for archimedean discrete parameters}
We briefly review the results of \cite{KM26}. 
One considers rigid inner forms of groups of the form $G^+ = G \rtimes A$, 
where $G$ is a quasi-split connected reductive $\R$-group 
and $A$ is a finite group of automorphisms of $G$ that preserve an $\R$-pinning. 
(These groups were denoted by $\tl{G}$ in loc.~cit., 
but here we are using the notation $G^+$ of Arthur.)
In the case at hand we have $G = \SO_{2n}$, the non-trivial element of
$A = \Z/2\Z$ is acting by the unique pinned outer automorphism, and we
are interested in the trivial rigid inner form, i.e. the group $G^+ = G \rtimes A$ itself.
\par

It is shown in \cite[\S5.2]{KM26} that to each discrete parameter $\phi \colon W_\R \rightarrow {}^LG$ 
one can associate an $L$-packet $\Pi_\phi(G^+)$ of irreducible discrete series representations of $G^+(\R)$ 
and an injection of this $L$-packet into $\Irr(S_\phi^+)$, 
where $S_\phi^+$ is the centralizer of $\phi$ in $\widehat{G} \rtimes A = \SO_{2n}(\C) \rtimes \Z/2\Z$, 
where again $A$ acts by preserving a $\Gamma$-stable pinning of $\widehat{G}$. 
Note that we are using the superscript $+$ here again in the sense of Arthur, and not in the sense of \cite{KalRI}. 
This injection is normalized using a Whittaker datum $\ww$ that is $A$-admissible, 
i.e. it is a $G(\R)$-conjugacy class of pairs $(B,\chi)$ that contains an $A$-stable such pair; 
such Whittaker data always exist.
\par

When the $\widehat{G}$-conjugacy class of $\phi$ is not stable under $A$, 
then $S_\phi^+=S_\phi$ is the centralizer of $\phi$ in $\widehat{G}$. 
The setting of \cite[Theorem 2.2.4]{Ar} is the opposite, 
namely the $\widehat{G}$-conjugacy class of $\phi$ is stable under $A$, 
in which case the projection $\hat G \rtimes A \to A$ induces the exact sequence
\[ 1 \to S_\phi \to S_\phi^+ \to A \to 1. \]
\par

For any $s \in S_\phi^+$ one can associate an endoscopic datum $(G',\GG',s,\eta)$, 
which is twisted when $s \notin S_\phi$. 
The following character identity is proved in \cite[Theorem 6.1.1]{KM26}, in the form \cite[Equation (6.1.1)]{KM26}:
\[
\tag{{\bf ECR}}\label{ECR-R}
S\Theta_{\phi'}(f') = \sum_{\pi \in \Pi_\phi(G)} \pair{s,\tl\pi} \Theta_{\tl\pi}(f),
\]
where the notation is as follows: 
\begin{enumerate}
\item 
$f$ is a smooth compactly supported function on $G(\R)=\SO_{2n}(\R)$, 
which we interpret as a function on $G^+(\R)$ that is supported on the \emph{non-identity} coset, 
by means of the injection $G(\R) \rightarrow G^+(\R)$ sending $g$ to $g\theta$, 
with $\theta \in A=\Z/2\Z$ the non-identity element.

\item 
$f'$ is the transfer of $f$ to the twisted endoscopic group $G'$ 
in terms of the Kottwitz--Shelstad transfer factor normalized by the fixed $A$-admissible Whittaker datum $\ww$.

\item 
$\phi'$ is the factorization of the parameter $\phi$ through $\eta \colon {}^LG' \rightarrow {}^LG$. 

\item 
$S\Theta_{\phi'}$ is the stable character of the $L$-packet $\Pi_{\phi'}(G')$ of $G'$.

\item 
$\tl\pi$ is an arbitrary extension to $G^+(\R)$ of the representation $\pi$. 
We will argue below that such an extension always exists in the special case we are considering.

\item 
$\Theta_{\tl\pi}$ is the distribution character of $\tl\pi$. 
The product $\pair{s,\tl\pi}\Theta_{\tl\pi}(f)$ is denoted by $\tr(\pi\boxtimes\rho^\vee)^{\rm{can}}(f,s^{-1})$ in loc. cit.; it does not depend on the choice of extension $\tl\pi$ of $\pi$.
\end{enumerate}
Identity \eqref{ECR-R} is the desired identity (2.2.17) in \cite[Theorem 2.2.4]{Ar}, albeit in different notation. 
This proves part (a) of that theorem for discrete parameters over $\R$.  
\par

Part (b) is the assertion that we made  in (5) above that every $\pi$ has an extension $\tl\pi$. We shall now argue that this is the case, and also argue the independence of $\<s,\tilde\pi\>\Theta_{\tilde\pi}(f)$ of the extension $\tilde\pi$ of $\pi$. It is shown in \cite[Equation (5.2.5)]{KM26}, in the general setting of disconnected groups discussed there, 
that if $\pi \in \Pi_\phi(G)$ corresponds to $\rho \in \Irr(S_\phi)$ under the refined Langlands correspondence for connected groups, then the injection
\[ \Pi_\phi(G^+) \to \Irr(S_\phi^+)\]
induces a bijection between the set of extensions of $\pi$ to $G^+(\R)$ and the set of extensions of $\rho$ to $S_\phi^+$; in general both sets can be empty. To prove (6) we use the fact that this bijection is equivariant for the group of characters of $A$ acting on both sides by tensoring. To prove (5) it is enough to show that in the particular case of $G^+=\SO_{2n} \rtimes \Z/2\Z$ 
any $\rho \in \Irr(S_\phi)$ has an extension to $S_\phi^+$. 
\par

To see this we note that $\widehat{G} \rtimes \Z/2\Z$ is isomorphic to $\O_{2n}(\C)$. 
Therefore, the centralizer $S_\phi^+$ is isomorphic to 
\[
\prod_i \O(V_i) \times \prod_i \Sp(W_i) 
\]
and $S_\phi$ is the subgroup of index $2$ on which the product of the determinants of the individual factors is trivial. 
The representation $\rho$ kills the identity component of $S_\phi$, 
which is also the identity component of $S_\phi^+$. 
We are therefore looking for an extension of $\rho$ from $\pi_0(S_\phi)$ to $\pi_0(S_\phi^+)$. 
But the latter group is visibly abelian, in fact a $2$-group, 
so such an extension always exists.

\subsection{Theorem 2.2.1(a) for archimedean tempered parameters}
In the remainder of this appendix 
we will show that the twisted character identities in \cite[Theorem 2.2.1 (a)]{Ar} 
and their unitary group analogue over $\R$ for tempered representations are valid. 
Note that \cite{KM26} treats general disconnected real groups, 
but only those tempered parameters that are discrete for the identity component, 
and this is not sufficient for our purposes. 
Mezo has treated in \cite{Mez16} general twisted real groups and general tempered parameters,
but the desired identity is proved there only up to a scalar. 
We need to know that this scalar factor is equal to $1$. 
This is done in \cite[Appendix A]{AMR} (\resp in \cite{Cl}) for $G=\Sp_{2n}$ and $G=\SO_{n}$ (\resp for $G=\U_n$), 
but only for special $L$-embeddings ${}^LG \rightarrow \GL_N$ 
and only for $L$-parameters that are discrete for $G$. 
We will reduce the general case to that core case.

\subsection{Reductions}
For convenience we adopt the notation of \cite{Ar} and refer to its numbering. 
(These are closely followed in \cite{Mok} with minor differences.) 
\par

We can first reduce the identities to the case that $G$ is a simple twisted endoscopic group. 
Indeed, if $G = G_S \times G_O$ is not simple, 
then by \cite[Proposition 6.6.1]{Ar}, 
we may assume that $\phi = \phi_S \times \phi_O \in \tl\Phi_2(G)$ is square-integrable. 
(If $G$ possesses no square-integrable parameters, then there is nothing to prove.)
There exist stable linear forms 
\[
f^S(\phi_S), \, f^O(\phi_O),\quad f^S\in \tl\HH(G_S),\, f^O\in \tl\HH(G_O)
\]
either by the induction hypothesis or by well-known results in real endoscopy. 
(These are the stable characters associated with the discrete series $L$-packets for $\phi_S$ and $\phi_O$.)
We define a stable linear form $f(\phi)$ by the formula
\[
f(\phi) = f^S(\phi_S) f^O(\phi_O),\quad f = f^S \times f^O.
\]
Mezo \cite{Mez16} has shown that the twisted character $\tilde f_N(\phi)$ (defined in \cite[(2.2.1)]{Ar}) 
on $\tilde{f} \in \tl\HH(N)$ satisfies the following twisted character identity up to a scalar $c(\phi)\in \C^\times$, 
where $\tilde f^G$ denotes a transfer of $\tilde f$:
\[
\tilde f_N(\phi)=c(\phi) \tilde f^G(\phi).
\]

\begin{lem}\label{lem6.6.3-nonsimple}
In the above setting, $c(\phi)=1$.
\end{lem}

\begin{rem}
This lemma was assumed in the proof of \cite[Lemma 6.6.3]{Ar} and \cite[Proposition 7.7.1]{Mok} 
when the authors write ``assumed as part of the theory of twisted endoscopy'' or ``by the results of Mezo and Shelstad''. 
Unfortunately the current state of twisted endoscopy for real groups is not enough to imply the lemma as a special case. So we give a global proof in much the same way as Arthur treats $p$-adic places.
\end{rem}

\begin{proof}[Proof of Lemma \ref{lem6.6.3-nonsimple}]
Let $d>0$ be a sufficiently large integer; this number controls the number of archimedean places. 
(Arthur wants $d$ to be large in the globalization of \cite[Sections 6.2--6.3]{Ar}. 
For our purpose $d=2$ is enough because we do not need \cite[Proposition 6.3.1 (iii)]{Ar}.
We do need $d>1$ to be able to appeal to the simple trace formula.)
\par

By \cite[Proposition 6.3.1]{Ar} (disregarding condition (iii) there) 
we have the following globalization for each $i\in\{0,1\}$:
\begin{itemize}
\item 
a totally real field $\dot F_i$ such that $[\dot{F_i}:\Q]=d+i$,
\item 
a real place $u_i$ of $\dot F_i$,
\item 
a twisted endoscopic datum $\dot G_i\in \dot{\tl{\mathcal{E}}}_{\text{ell}}(N)$ over $\dot F_i$,
\item 
a parameter $\dot\phi_i\in \tl\Phi^{\text{sim}}_2(\dot G_i)$
\end{itemize}
such that $(\dot F_i,\dot G_i,\dot \phi_i)$ specializes to $(F,G,\phi)$ at the place $u_i$, 
and such that 
\begin{itemize}
\item 
$\dot G_i$ possesses discrete series at all $v\in S_{i,\infty}^{u_i}$ (in addition to $v=u_i$),
\item 
$\dot\phi_{i,v}$ is square-integrable and in relative general position at all $v\in S_{i,\infty}^{u_i}$,
\end{itemize}
where $S_{i,\infty}$ stands for the set of real places of $\dot F_i$, 
and we put $S^{u_i}_{i,\infty}=S_{i,\infty}\backslash \{u_i\}$. 
(The degree $[\dot{F_i}:\Q]$ is not prescribed in \emph{loc.~cit.}~but the existence argument there works 
for a fixed choice of $\dot F_i$ and $u_i$.) 
In fact the real groups
\[
\dot G_{i,v}
\quad (i\in \{0,1\},\, v\in S_{i,\infty})
\]
are all isomorphic to $G$ (canonically up to inner automorphism) 
as they are quasi-split real forms accommodating discrete series. 
We fix such isomorphisms.
\par

Now the point is that, in addition to the above, 
we can arrange that 
the real components of $\dot \phi_0$ and $\dot\phi_1$ away from $u_0,u_1$ are all equal, 
i.e.,
\[
\tag{$\flat$}\label{flat}
\dot \phi_{i,v}=\phi_{\gen}
\quad (i\in \{0,1\},\, v\in S^{u_i}_{i,\infty})
\]
for some $\phi_{\gen}\in \tl\Phi_2(G)$ in general position via $\dot G_{i,v} \cong G$. 
This is possible since Arthur can prescribe the parameters at $v \in S^{u_i}_{i,\infty}$ quite flexibly. 
More precisely, for $i=0,1$ and $v\in S^{u_i}_{i,\infty}$, 
his argument fixes a regular infinitesimal character $\mu_{i,v}$ (of a discrete series representation) 
and shows that there exists a global parameter $\dot\phi_i$ 
such that $\dot\phi_{i,v}$ has infinitesimal character $n\mu_{i,v}$ 
and trivial central character, as long as $n\in \Z_{>0}$ is sufficiently large. 
Thus we can achieve \eqref{flat} by starting from the same $\mu_{i,v}$ 
for all $i,v$ as in \eqref{flat} and then choosing $n$ large. 
(We use the same $n$ for all $i,v$.)
\par

The rest of the argument proceeds as in the proof of \cite[Lemma 6.6.3]{Ar}, 
with $u$ a real place rather than a finite place. 
Namely we apply \cite[Lemma 5.4.2]{Ar} 
(applicable since \cite[Assumption 5.4.1 (b)]{Ar} therein is satisfied at all archimedean places) 
to deduce that
\[
\dot{\tilde{f}}_N(\dot \phi_i) 
= \dot{\tilde{f}}^G (\dot \phi_i), 
\quad \dot{\tilde{f}}
= \prod_v \dot{\tilde{f}}_v\in \tl{\HH}(N)_{\dot F_i}.
\]
Here the subscript $\dot F_i$ is there to remind us 
that the Hecke algebra is for the twisted general linear group over $\dot F_i$.
Condition (ii) of \cite[Proposition 6.3.1]{Ar} allows us 
to cancel out the terms at all finite places from both sides, 
as in the proof of \cite[Lemma 6.6.3]{Ar}. 
Hence
\[
\prod_{v\in S_{i,\infty}} \dot{\tilde{f}}_{N,v}(\dot \phi_{i,v}) 
= \prod_{v\in S_{i,\infty}} \dot{\tilde{f}}^G_v (\dot \phi_{i,v}). 
\]
In light of the equation $\tilde f_N(\phi)=c(\phi) \tilde f^G(\phi)$, 
this implies that
\[
c(\phi) c(\phi_{\textup{gen}})^{d-1+i}=1.
\]
Since this holds for both $i \in \{0,1\}$, 
it follows that $c(\phi) = c(\phi_{\gen})=1$, 
as desired.
\end{proof}

Now we may assume that $G$ is simple. 
We can further reduce to the case that the parameter $\phi$ is square-integrable. 
Indeed, if $\phi$ is tempered but not square-integrable, 
\cite[Theorem 2.2.1]{Ar} follows as explained at the beginning of Section 6.6 in loc.~cit. 
\par

Now consider the core case that $G$ is a simple endoscopic group of twisted $\GL_N$ 
and $\phi$ is a square-integrable parameter. 
If the endoscopic datum is standard (i.e. the $L$-embedding ${}^LG \rightarrow \GL_N$ is standard), 
the result is covered by \cite[Appendix A]{AMR} for $\Sp_{2n},\SO_n$ and \cite{Cl} for $\U_n$. 
This reduces the problem to Proposition \ref{pro:eci-twist} below, 
which may be of independent interest, so we will prove it for any local field.
\par

\subsection{Set-up and statement of the result}
Let $F$ be a local field, 
$G$ a quasi-split connected reductive $F$-group, 
$(B,T,\{X_\alpha\})$ a pinning, 
$\theta$ an $F$-automorphism of $G$ preserving the pinning, 
and $\psi_F \colon F \rightarrow \C^\times$ a non-trivial character. 
To simplify matters, we assume that the derived subgroup $G_\der$ is simply connected.
Let $(H,1,\HH,\xi)$ be the principal endoscopic datum for $(G,\theta)$. 
We assume that there exists an $L$-isomorphism ${}^LH \rightarrow \HH$ 
and write ${}^L\xi \colon {}^LH \rightarrow {}^LG$ for the composition of this isomorphism with $\xi$. 
The pinning and the character $\psi_F$ lead to a Whittaker datum, which we use to normalize the transfer factor.
\par

Let $\phi \colon L_F \rightarrow {}^LH$ be a tempered parameter. 
The expected twisted character identity is then as follows:
\[ 
\sum_{\substack{\pi \in \Pi_{{}^L\xi \circ \phi}(G)\\ \pi\circ \theta \cong \pi}} \tr(\tl\pi(\tl{f})) 
= \sum_{\sigma \in \Pi_\phi(H)} \tr(\sigma(\tl{f}^H)),
\]
where $\tl{\pi}$ is the Whittaker normalized extension of $\pi$ to $G(F) \rtimes \pair{\theta}$, 
$\tl{f} \in C_c^\infty(G(F) \rtimes \theta)$, and $\tl{f}^H \in C_c^\infty(H(F))$ is its transfer.
\par

This identity is stated in the language of distributions, but can also be restated in the language of functions as
\[ 
\sum_{\substack{\pi \in \Pi_{{}^L\xi \circ \phi}(G)\\ \pi \circ \theta \cong \pi}} \tr(\tl\pi(\tl\delta)) 
= \sum_\gamma \Delta[{}^L\xi](\gamma,\tl\delta) \sum_{\sigma \in \Pi_\phi(H)} \tr(\sigma(\gamma))
\]
for any regular semisimple element $\tl\delta \in G(F) \rtimes \theta$,
where $\gamma$ runs over the set of regular semisimple elements of $H(F)$ up to stable conjugacy, 
and $\Delta[{}^L\xi](\gamma,\tl\delta)$ is the transfer factor relative to the $L$-embedding ${}^L\xi$ 
and normalized by the Whittaker datum.

\begin{prop} \label{pro:eci-twist}
Assume that the twisted character identity holds for one choice of embedding ${}^L\xi \colon {}^LH \rightarrow {}^LG$.
Then it holds for any other choice of ${}^L\xi$ with the same restriction to $\widehat{H}$.
\end{prop}

\begin{rem}
One can contemplate various generalizations of this statement. 
For example, the proof given below easily generalizes to arbitrary endoscopic groups in the ordinary setting, 
i.e. when $\theta=1$. 
It also generalizes to inner forms of $(G,\theta)$. 
What is not so clear to us is how to generalize it to arbitrary endoscopic groups when $\theta$ is non-trivial, 
but we do not need this case.
Also, we will only give the proof when the derived subgroup is simply connected.
\end{rem}

\subsection{Proof}
First, we study the possible variations of ${}^L\xi$.
Since ${}^L\xi|_{\widehat{H}}$ has been fixed, 
we use it to identify $\widehat{H}$ with a subgroup of $\widehat{G}$ to save notation. 
We have $\widehat{H}=(\widehat{G}^\theta)^\circ$.

\begin{lem}
We have $Z_{\widehat{G}}(\widehat{H})=Z(\widehat{G})$. 
\end{lem}
\begin{proof}
Let $\widehat{S} \subset \widehat{H}$ be a maximal torus. 
Then since $\widehat{T} = Z_{\widehat{G}}(\widehat{S})$ is a $\theta$-stable maximal torus of $\widehat{G}$, 
it is contained in a $\theta$-stable maximal Borel subgroup $\widehat{B}$, 
and $\widehat{S}=(\widehat{T}^\theta)^\circ$. 
The root system $R(\widehat{S}, \widehat{H})$ is the set of indivisible roots in the relative root system $R(\widehat{S},\widehat{G})$. 
The latter is the set of restrictions to $\widehat{S}$ of the absolute root system $R(\widehat{T}, \widehat{G})$. 
No such restriction vanishes, and the restriction map 
$R(\widehat{T}, \widehat{G}) \rightarrow R(\widehat{S}, \widehat{G})$ is surjective with fibers being the $\theta$-orbits.
\par

An element of $Z_{\widehat{G}}(\widehat{H})$ centralizes $\widehat{S}$, hence lies in $\widehat{T}$. 
Moreover, it acts trivially on each root space of $\widehat{S}$ in $\Lie(\widehat{H})$. 
For $\beta \in R(\widehat{S}, \widehat{H})$, 
the root space $\Lie(\widehat{H})_\beta$ is the space of $\theta$-fixed points in 
$\bigoplus_{\alpha \mapsto \beta} \Lie(\widehat{G})_\alpha$. 
Therefore, $t \in \widehat{T}$ fixes this root space if and only if $\alpha(t)=1$ for all $\alpha \mapsto \beta$. 
Since all $\widehat{B}$-simple roots $\alpha$ map to an indivisible root in $R(\widehat{S}, \widehat{H})$, 
we see that for an element $t \in \widehat{T}$ to centralize $\Lie(\widehat{H})$ 
it is necessary and sufficient that $\alpha(t)=1$ for all $\widehat{B}$-simple roots in $R(\widehat{T}, \widehat{G})$, 
which is equivalent to $t \in Z(\widehat{G})$.
\end{proof}

\begin{lem}
If ${}^L\xi \colon {}^LH \rightarrow {}^LG$ is one choice of $L$-embedding, 
then any other choice is of the form $\alpha \cdot {}^L\xi$ for some $\alpha \in Z^1(W_F,Z(\widehat{G}))$ 
whose cohomology class is $\theta$-fixed. 
\end{lem}
\begin{proof}
The actions of $\sigma \in W_F$ on $\widehat{H}$ and $\widehat{G}$ are generally different, 
and we will write $\sigma_H$ and $\sigma_G$ for them. 
Then $\sigma_H = \Ad(g_\sigma) \rtimes \sigma_G$ for some $g_\sigma \in \widehat{G}$ 
that is well-defined up to multiplication on the left by $Z_{\widehat{G}}(\widehat{H})$. 
By the preceding lemma we have $Z_{\widehat{G}}(\widehat{H}) = Z(\widehat{G})$.
\par

Since, for any ${}^L\xi$ and $\sigma \in W_F$, 
the element ${}^L\xi(1 \rtimes \sigma) \in {}^LG$ normalizes $\widehat{H}$ and acts on it by $\sigma_H$, 
we see that if one ${}^L\xi$ is fixed, 
then any other choice is of the form $\alpha \cdot {}^L\xi$ for a continuous map $\alpha \colon W_F \to Z(\widehat{G})$.
\par

The multiplicativity of ${}^L\xi$ and $\alpha \cdot {}^L\xi$ implies that $\alpha$ is a $1$-cocycle. 
By an axiom of twisted endoscopic data (\cite[(2.1.4a)]{KoSh}), 
there exist $x,y \in Z(\widehat{G})$ 
such that $\theta \circ {}^L\xi = x \cdot {}^L\xi \cdot x^{-1}$ 
and $\theta \circ (\alpha \cdot {}^L\xi) = y \cdot (\alpha \cdot {}^L\xi) \cdot y^{-1}$. 
It follows that $\alpha^{-1}(\sigma) \cdot \theta \circ \alpha(\sigma) = (xy^{-1})^{-1} \sigma(xy^{-1})$, 
i.e. $[\alpha] \in H^1(W_F,Z(\widehat G))^\theta$.
\end{proof}

Next, we reduce Proposition \ref{pro:eci-twist} to a property of the transfer factor. 
We now fix ${}^L\xi$ for which the character identity 
\[ 
\sum_{\substack{\pi \in \Pi_{{}^L\xi \circ \phi}(G)\\ \pi \circ \theta \cong \pi}} \tr(\tl\pi(\tl\delta)) 
= \sum_\gamma \Delta[{}^L\xi](\gamma,\tl\delta) \sum_{\sigma \in \Pi_\phi(H)} \tr(\sigma(\gamma))
\]
holds. 
If we replace ${}^L\xi$ by $\alpha \cdot {}^L\xi$ for some $\alpha \in Z^1(W_F,Z(\widehat{G}))$ 
whose class is $\theta$-fixed, 
then on the left-hand side of the identity, 
the packet $\Pi_{{}^L\xi \circ \phi}$ changes to $\Pi_{\alpha \cdot {}^L\xi \circ \phi}$. 
If $\chi_\alpha \colon G(F) \rightarrow \C^\times$ is the character corresponding to $\alpha$, 
then tensoring representations with $\chi_\alpha$ provides a bijection 
\[ 
\Pi_{{}^L\xi \circ \phi} \rightarrow \Pi_{\alpha \cdot {}^L\xi \circ \phi}. 
\]
Moreover, the Whittaker extension of $\chi_\alpha \otimes \pi$ equals $\tl\chi_\alpha \otimes \tl\pi$, 
where $\tl\pi$ is the Whittaker extension of $\pi$, 
and $\tl\chi_\alpha$ is the extension of $\chi_\alpha$ to $G(F) \rtimes \pair{\theta}$ 
specified by $\tl\chi_\alpha(\theta)=1$ 
(since the class of $\alpha$ is $\theta$-fixed, the character $\chi_\alpha$ is $\theta$-invariant). 
Therefore, upon replacing ${}^L\xi$ by $\alpha \cdot {}^L\xi$, 
the left-hand side of the character identity multiplies by $\tl\chi_\alpha(\tl\delta)$. 
Hence to prove Proposition \ref{pro:eci-twist}, 
it would be enough to show
\[ 
\Delta[\alpha \cdot {}^L\xi](\gamma,\tl\delta) = \tl\chi_\alpha(\tl\delta) \cdot \Delta[{}^L\xi](\gamma,\tl\delta). 
\]
\par

Finally, we shall show this property. 
The only piece of the transfer factor that depends on the choice of $L$-embedding is $\Delta_{III}$. 
Let us briefly recall its construction following \cite[Section 5.3]{KoSh}. 
There exists a $\theta$-admissible maximal $F$-torus $T \subset G$ together with an element $g \in G_\sc(\overline{F})$ 
such that $\tl\delta^* = g^{-1}\tl\delta g \in T(\overline{F}) \rtimes \theta$, 
and if we write $\tl\delta^*=\delta^* \rtimes \theta$ 
then the image $\gamma^* \in T_\theta(\overline{F})$ of $\delta^*$ lies in $T_\theta(F)$. 
Writing $z_\sigma = g^{-1}\sigma(g)$ we have $(z_\sigma^{-1}, \delta^*) \in Z^1(F,T_\sc \xrightarrow{1-\theta} T)$. 
\par

Let $\gamma \in H(F)$ be related to $\tl\delta$ 
and let $S \subset H$ be the centralizer of $\gamma$. 
There exists a unique admissible isomorphism $\xi_{\gamma,\gamma^*} \colon S \rightarrow T_\theta$ 
mapping $\gamma$ to $\gamma^*$. 
Choosing $\chi$-data for $R(T^\theta,G)$ provides $L$-embeddings 
${}^L\xi_{S,H} \colon {}^LS \rightarrow {}^LH$ and ${}^L\xi_{T,G} \colon {}^LT \rightarrow {}^LG$. 
There is a unique $1$-cocycle $\eta \colon W_F \rightarrow \widehat{T}$ 
that makes the following diagram commute
\[ 
\xymatrix{
{}^LS \ar[r]^{{}^L\xi_{S,H}} \ar[d]_{\eta \cdot {}^L\xi_{\gamma,\gamma^*}} & {}^LH \ar[d]^{{}^L\xi} \\
{}^LT \ar[r]_{{}^L\xi_{T,G}} & {}^LG.
}
\]
Here $^L\xi_{\gamma,\gamma^*}$ is the $L$-embedding obtained from the $L$-isomorphism 
$\widehat{S} \rtimes W_F \rightarrow (\widehat{T}^\theta)^\circ \rtimes W_F$ 
dual to the inverse of $\xi_{\gamma,\gamma^*}$ 
and the inclusion $(\widehat{T}^\theta)^\circ \rightarrow \widehat{T}$, 
and we have multiplied it with the $1$-cocycle $\eta$, 
composed with the projection ${}^LS \rightarrow W_F$ and the inclusion $\widehat{T} \rightarrow {}^LT$. 
Then $(\eta^{-1},1) \in Z^1(W_F,\widehat{T} \xrightarrow{1-\theta} \widehat{T}_\ad)$. 
The factor $\Delta_{III}$ is the pairing of $(z_\sigma^{-1},\delta^*)$ and $(\eta^{-1},1)$.
\par

If we replace ${}^L\xi$ by $\alpha \cdot {}^L\xi$ then $\eta$ is replaced by $\alpha \cdot \eta$. 
We are using the natural inclusion $Z(\widehat{G}) \rightarrow \widehat{T}$, 
which is equivariant under both $\Gamma$ and $\theta$, 
and induces an inclusion of complexes of $\Gamma$-modules 
$[Z(\widehat{G}) \rightarrow 1] \rightarrow [\widehat{T} \xrightarrow{1-\theta} \widehat{T}_\ad]$. 
The value of the transfer factor thus multiplies by the pairing of $(\alpha^{-1},1)$ with $(z_\sigma^{-1},\delta^*)$. 
Now $(\alpha^{-1},1)$ is included from $[Z(\widehat{G}) \rightarrow 1]$. 
Using the functoriality of the Tate--Nakayama pairing 
and the fact that $Z(\widehat{G})$ is the torus dual to $D = G/G_\der$ 
when $G_\der$ is simply connected, 
we may compute the pairing of $(\alpha^{-1},1)$ and $(z_\sigma^{-1},\delta^*)$ 
by mapping the latter under the map $[T_\sc \xrightarrow{1-\theta} T] \rightarrow [1 \rightarrow D]$. 
Now $\delta^* \rtimes \theta = \tl\delta^* = g^{-1} \tl\delta g = (g^{-1} \tl\delta g \tl\delta^{-1}) \cdot \tl\delta$ 
and the term in the parentheses lies in $G_\der$. 
Therefore we see that the images of $\delta^* \rtimes \theta$ and $\tl\delta$ in $D(\overline{F}) \rtimes \theta$ agree, 
and lie in $D(F) \rtimes \theta$. 
In particular, the image $\overline\delta$ of $\delta^*$ in $D$ is an $F$-point. 
Since $\Delta_{III}$ enters the transfer factor with its reciprocal, 
we see that changing ${}^L\xi$ to $\alpha \cdot {}^L\xi$ multiplies the transfer factor by $\chi_\alpha(\overline{\delta})$,
where $\chi_\alpha$ is the character of $D(F)$ with parameter $\alpha$. 
This character is $\theta$-invariant 
and hence extends to a character $\tl\chi_\alpha$ of $D(F) \rtimes \pair{\theta}$ with $\tl\chi_\alpha(\theta)=1$. 
Then $\chi_\alpha(\overline\delta)=\tl\chi_\alpha(\overline \delta \rtimes \theta) = \tl\chi_\alpha(\tl\delta)$, 
where we have mapped $\tl\delta$ in $D(F) \rtimes \theta$ under the natural map $G(F) \rtimes \theta \rightarrow D(F) \rtimes \theta$. 
Since the character $\chi_\alpha$ of $G(F)$ is simply the inflation to $G(F)$ of the character $\chi_\alpha$ of $D(F)$, 
the desired identity
\[ 
\Delta[\alpha \cdot {}^L\xi](\gamma,\tl\delta) = \tl\chi_\alpha(\tl\delta) \cdot \Delta[{}^L\xi](\gamma,\tl\delta) 
\]
has been established.

\section{Arithmetic Frobenius vs.~geometric Frobenius}\label{sec.Frob}
There are two conventions for normalizing class field theory in the literature, depending on whether arithmetic or geometric Frobenius elements correspond to uniformizing elements under the local non-archimedean reciprocity map.  In this appendix we shall review these conventions and discuss how they influence Arthur's endoscopic classification.
This is an expanded version of the discussion of \cite[Section 4]{KoSh2}.

\subsection{Two conventions for $L$-factors}
Consider a non-archimedean local field $F$ 
and let $(\rho,V)$ be a finite-dimensional representation of the absolute Galois group $\Gamma_F$, 
or of its Weil group $W_F \subset \Gamma_F$. 
In \cite[(9)]{ArtinL}, Artin has attached to this representation an $L$-factor 
\[ 
L_{A}(s, \rho) = \det(1-q^{-s} \cdot \rho(\Frob_\ari)|V^{\rho(I_F)})^{-1},
\]
where $I_F \subset W_F$ is the inertia subgroup and $\Frob_\ari \in W_F/I_F$ is the \emph{arithmetic Frobenius}, 
i.e. the element whose action on the algebraic closure $\overline\F_q$ of the residue field $k_F \cong \F_q$ of $F$ 
is by the transformation $x \mapsto x^q$. 
On the other hand, Deligne considers in \cite[(3.5.1)]{D} the definition (we have set $t=q^{-s}$)
\[ 
L_{D}(s, \rho) = \det(1-q^{-s} \cdot \rho(\Frob_\geo)|V^{\rho(I_F)})^{-1},
\]
where $\Frob_\geo$ is the \emph{geometric Frobenius}, 
i.e. the inverse of the arithmetic Frobenius. 
We have
\[
L_D(s, \rho) = L_{A}(s, \rho^\vee),
\]
where $(\rho^\vee,V^*)$ is the contragredient representation. 
\par

Artin's convention appears more natural from an arithmetic point of view, 
because the substitution $x \mapsto x^q$ is a primary object for any field of characteristic $p$, 
and is not always invertible. 
Deligne's convention appears more natural from a geometric point of view, 
because the geometric Frobenius acts on the $\ell$-adic cohomology of algebraic varieties over $\F_q$ 
with eigenvalues that are algebraic integers, 
and this is not the case for the arithmetic Frobenius, 
see \cite[Section 3.6]{D}.
\par

Both conventions are widely used in the literature. 
On the other hand, the $L$-factors associated to characters of $F^\times$, 
and more generally to admissible representations of  reductive $F$-groups, 
are often defined, following Tate's thesis, in terms of zeta integrals on the group and related spaces, 
and make no reference to the Galois group of $F$. 
We refer the reader to \cite{Tate1} for the group $F^\times=\GL_1(F)$, 
Godement--Jacquet \cite{GJ} for the standard $L$-function on the group $\GL_n(F)$, 
and \cite{JPSS} for the Rankin--Selberg $L$-function of a ``product'' of two representations of general linear groups. 
These definitions are compatible with parabolic induction, 
which implies that for principal series representations, they specialize to the $1$-dimensional case. 
In that case, 
the output of the construction is expressible in terms of uniformizers of $F$, 
see \cite[(3.1.3)]{Tate} for characters of $F^\times$.
\par

Therefore, as discussed in \cite[Section 3.6]{D}, 
the desire that a correspondence between representations of the Galois or Weil groups of $F$ 
and representations of connected reductive $F$-groups 
should respect the $L$-factors on both sides 
leads to two possible conventions for such a correspondence 
-- one convention adapted to Artin's definition of a Galois-theoretic $L$-function, 
and one convention adapted to Deligne's convention. 
\par

In the following, we will review these conventions 
and discuss the various parts of Arthur's argument which are affected by the choice of convention. 
We will then explain how to make consistent choices and how the main results are affected by this choice.

\subsection{Class field theory}
One of the main results of class field theory is the Artin reciprocity map $r_F \colon C_F \rightarrow \Gamma_F^\ab$, 
available for each local or global field $F$, 
where $C_F=F^\times$ when $F$ is local and $C_F=\A_F^\times/F^\times$ when $F$ is global. 
It is closely related to the invariant map $\inv \colon \Br(F) = H^2(\Gamma_F, C_{\overline{F}}) \rightarrow \Q/\Z$, 
where $C_{\overline{F}} = \overline{F}^\times$ when $F$ is local 
and $C_{\overline{F}} = \overline{\A}_F^\times/\overline{F}^\times$ when $F$ is global. 
Here $\overline{\A}_F=\varinjlim\A_E$, the colimit taken over all finite Galois extensions $E/F$.
\par

In the classical normalization of Artin, 
the local reciprocity map sends a uniformizer of $C_F=F^\times$ 
to an \emph{arithmetic Frobenius} element of $\Gamma_F^\ab$, 
i.e. an element whose projection to $\Gamma_{k_F}$ acts as $x \mapsto x^q$ on 
$\overline{k}_F \cong \overline\F_q$. 
The local and global reciprocity maps are related by the property that, 
for each place $v$ of a global field $F$, 
we have the commutative diagram
\[ 
\xymatrix{
    F_v^\times \ar[rr]^{r_{F_v}} \ar[d] &&\Gamma_{F_v}^\ab\ar[d]\\
    \A_F^\times/F^\times \ar[rr]^{r_F} &&\Gamma_F^\ab,
}
\]
where the left map embeds $x \in F_v^\times$ into $\A_F^\times$ 
as the idele whose $v$-component equals $x$ and all of whose other components equal $1$, 
while the right map is induced by identifying $\Gamma_{F_v}$ 
with the stabilizer in $\Gamma_F$ of a lift of $v$ to a ``place'' of $\overline{F}$; 
the latter identification is unique up to conjugation within $\Gamma_F$, 
hence unique after passing to abelianization.
\par

The reciprocity map induces for each finite Galois extension $E/F$, 
an isomorphism $C_F/N_{E/F}(C_E) \xrightarrow{\sim} \Gal(E/F)^\ab$. 
Conversely, these maps for all such $E/F$ glue together to the reciprocity map. 
The finite-level maps can be interpreted via Tate cohomology as the isomorphisms
\[ 
\Gal(E/F)^\ab = H^{-2}_{\Tate}(\Gal(E/F),\Z) 
\rightarrow H^0_{\Tate}(\Gal(E/F),C_E) = C_F/N_{E/F}(C_E) 
\]
obtained by taking the cup product with the fundamental class in $H^2(\Gal(E/F),C_E)$. 
In turn, the fundamental class provides an isomorphism
\[ 
\inv \colon H^2(\Gal(E/F),C_E) \rightarrow \frac{1}{[E:F]}\Z/\Z 
\]
sending the fundamental class to $[E:F]^{-1}$. 
These invariants splice together under inflation to form the absolute invariant 
$\inv \colon H^2(\Gamma_F,C_{\overline{F}}) \to \Q/\Z$.
\par

The local non-archimedean absolute invariant map can be described explicitly as follows. 
Let $F^u \subset \overline{F}$ be the maximal unramified extension of $F$. 
It is known that the inflation map 
$H^2(\Gal(F^u/F),(F^u)^\times) \rightarrow H^2(\Gamma_F, \overline{F}^\times)$ 
is an isomorphism. 
On the other hand, if $E/F$ is finite unramified, 
$\varpi \in F^\times$ is a uniformizer and $\sigma \in \Gal(E/F)$ is the arithmetic Frobenius element, 
then the assignment 
\[ 
(\sigma^i,\sigma^j) \mapsto 
\left\{
\begin{aligned}
&\varpi \iif i+j \geq [E:F], \\
&1 \iif i+j < [E:F],
\end{aligned}
\right.
\]
for $0 \leq i,j < [E:F]$ is an element of $Z^2(\Gal(E/F), E^\times)$ 
that represents the fundamental class in $H^2(\Gal(E/F), E^\times)$.
\par

The local archimedean invariant map can also be described explicitly. 
It is trivial for $F=\C$, 
and when $F=\R$ and $\sigma \in \Gal(\C/\R)$ denotes the complex conjugation, 
then the assignment 
\[
(\sigma^i,\sigma^j) \mapsto 
\left\{
\begin{aligned}
&{-1} \iif i=j=1,\\
&1 \other
\end{aligned}
\right.
\]
for $0 \leq i,j \leq 1$ is an element of $Z^2(\Gal(\C/\R),\C^\times)$ 
that represents the fundamental class in $H^2(\Gal(\C/\R),\C^\times)$.
We can also give an explicit description of the archimedean local reciprocity map. 
It is again trivial when $F=\C$, 
and when $F=\R$, 
it is the unique surjective group homomorphism $\R^\times \twoheadrightarrow \Gal(\C/\R)$ 
whose kernel is $\R_{>0}$.
\par

All the objects described so far are in the \emph{Artin} convention. To switch to the \emph{Deligne} convention, 
we compose the local and global reciprocity maps with the inversion automorphism of $C_F$. 
This has the effect that, at each non-archimedean place, 
a uniformizer of $C_F=F^\times$ is mapped to a \emph{geometric Frobenius} element, 
i.e. the inverse of an arithmetic Frobenius element. To keep the invariant maps aligned with the reciprocity maps, 
so that the reciprocity map is again obtained by the cup product with the fundamental class, 
we also invert the fundamental classes, 
which means that we also invert the local and global invariant maps. At the archimedean places the reciprocity map, the invariant map, and the fundamental class, 
remain unchanged, because the relevant groups have exponent $2$. 
\par

Note that we cannot choose an Artin convention at one non-archimedean place 
and a Deligne convention at another non-archimedean place, 
because doing so will not be compatible with any global reciprocity map. 
Thus, there are overall just two conventions, 
which hold at all local places simultaneously, as well as globally. 
At the archimedean places the Artin and Deligne conventions coincide. 

It would be next to impossible to list exhaustively all treatments of class field theory and say for each of them which convention is used. A small and useful list appears in \cite[(1.4.1)]{Tate}. In particular, \cite{AT}, \cite{CF}, and \cite{Ser} use the Artin normalization, while \cite{D} and \cite{Tate} use the Deligne normalization.

\subsection{Two conventions for $\ep$-factors} \label{ss:eps}
In \cite[\S4]{D}, Deligne proves the existence and uniqueness of a local $\ep$-factor $\ep(\rho,\psi_F,dx) \in \C^\times$ associated to a finite-dimensional representation $\rho$ of the Weil group of a local field $F$, a non-trivial character $\psi_F \colon F \to \C^\times$, and a Haar measure $dx$ of $F$, subject to the properties listed in \cite[Theorem 4.1]{D}. Following Langlands, Tate reinterprets this definition in \cite[(3.6.4)]{Tate} as a function $\ep(s,\rho,\psi_F)=\ep(\rho\otimes\omega_s,\psi_F,dx)$, where $dx$ is the unique Haar measure on $F$ that is self-dual with respect to $\psi_F$, and $\omega_s \colon W_F \to \R_{>0}$ is the character sending $w \in W_F$ to $|r_F^{-1}(\bar w)|_F^s$, where $\bar w \in W_F^\textrm{ab}$ is the image of $w$, $r_F$ is the Artin reciprocity isomorphism $F^\times \to W_F^\textrm{ab}$, and $|\cdot|_F$ is the normalized absolute value on $F$ (see \cite[(2.2)]{Tate}). 

Tate uses subscripts ``D'' and ``L'' in his article to distinguish between two possible conventions for the local $\ep$-factors (one due to Deligne and one due to Langlands), which are related by a shift by $\omega_{1/2}$, see \cite[(3.6.1)]{Tate}. This difference of conventions is \emph{not} related to the difference we are discussing here, and the two conventions lead to the same function $\ep(s,\rho,\psi_F)$. Since that function is the form of the $\ep$-factor that will be relevant for us, we do not need to distinguish between these conventions. We will write $\ep_D(s,\rho,\psi_F)$ for this function.

If $F$ is now a global field, $\rho$ a finite-dimensional complex representation of $W_F$, and $\psi_F \colon \A_F/F \to \C^\times$ a non-trivial character, then for each place $v$ of $F$ we have the factor $\ep_D(s,\rho_v,\psi_{F_v})$, where $\rho_v = \rho|_{W_{F_v}}$ and $\psi_{F_v} = \psi_F|_{F_v}$. The product $\ep_D(s,\rho)=\prod_v \ep_D(s,\rho_v,\psi_{F_v})$ is well-defined and independent of the choice of $\psi_F$, and we have 
\[ L_D(s,\rho) = \ep_D(s,\rho)L_D(1-s,\rho^\vee) \]
according to \cite[(3.5.3) Theorem]{Tate} (recall that Tate's article uses Deligne's normalization of class field theory).

Returning to a local field $F$, to switch from Deligne's to Artin's normalization of class field theory we define 
\[ \ep_A(\rho,\psi_F,dx) = \ep_D(\rho^\vee,\psi_F,dx),\quad \ep_A(s,\rho,\psi_F) = \ep_D(s,\rho^\vee,\psi_F). \]
The factor $\ep_A(\rho,\psi_F,dx)$ satisfies the conditions (1)--(4) of \cite[Theorem 4.1]{D} in terms of class field theory in Artin's normalization -- indeed, the conditions (1)--(3) make no reference to class field theory and are stable under the contragredient operation, while the condition (4) asserts that $\ep(\rho,\psi_F,dx)=\ep(\chi,\psi_F,dx)$ whenever $\rho \colon W_F \to \C^\times$ is a character and $\chi = \rho \circ r_F$. Therefore, \cite[Theorem 4.1]{D} implies the uniqueness of the collection of factors $\ep_A(\rho,\psi_F,dx)$.

Moreover, when $F$ is global and we define as above $\ep_A(s,\rho)=\prod_v \ep_A(s,\rho_v,\psi_{F,v})$, then the functional equation in Deligne's normalization implies at once
\[ L_A(s,\rho) = \ep_A(s,\rho)L_A(1-s,\rho^\vee). \]
We conclude that the local factors $\ep_A(s,\rho,\psi_F)$ provide a satisfactory theory of $\ep$-factors for the Artin normalization of class field theory.

\subsection{The local correspondence for $\GL_n$} \label{ss:llc_gln}
Let $F$ be a non-archimedean local field. 
Fix a non-trivial additive character $\psi_F \colon F \rightarrow \C^\times$. 
It was proved by Henniart \cite{He3} that there exists at most one collection of bijections 
\[ 
\rec_{F,n} \colon \Irr(\GL_n(F)) \rightarrow H^1_\cts(W_F \times \SL_2(\C), \GL_n(\C)) 
\]
indexed by natural numbers $n$ 
satisfying the following conditions:
\begin{enumerate}
\item 
The map $\rec_{F,1}$ coincides with pull-back under the inverse of the isomorphism $r_F \colon F^\times \to W_F^\ab$.
\item 
For $\pi_1 \in \Irr(\GL_{n_1}(F))$ and $\pi_2 \in \Irr(\GL_{n_2}(F))$, 
we have
\[ 
L(s, \pi_1 \times \pi_2) = L(s, \rec_{F,n_1}(\pi_1) \otimes \rec_{F,n_2}(\pi_2))
\]
and 
\[ 
\ep(s, \pi_1 \times \pi_2, \psi_F) = \ep(s, \rec_{F,n_1}(\pi_1) \otimes \rec_{F,n_2}(\pi_2), \psi_F), 
\]
where the left-hand sides are the local factors defined in \cite{JPSS}.
\item 
If $\pi \in \Irr(\GL_n(F))$ and $\chi \in \Irr(\GL_1(F))$, 
then 
\[
\rec_{F,n}(\pi \otimes (\chi \circ \det)) = \rec_{F,n}(\pi) \otimes \rec_{F,1}(\chi).
\]
\item 
If $\pi \in \Irr(\GL_n(F))$ has central character $\chi \in \Irr(\GL_1(F))$, 
then 
\[
\det(\rec_{F,n}(\pi)) = \rec_{F,1}(\chi).
\]
\item 
If $\pi \in \Irr(\GL_n(F))$, 
then $\rec_{F,n}(\pi^\vee) = \rec_{F,n}(\pi)^\vee$.
\end{enumerate}
The existence of such a bijection 
was proved by Harris--Taylor \cite{HT} and Henniart \cite{He2} for the characteristic zero case, 
and by Laumon--Rapoport--Stuhler \cite{LRS} for the positive characteristic case. 
\par

The compatibility properties (1)--(5) make it clear that 
$\rec_{F,n}$ also depends on the choice of normalization of local class field theory. 
More precisely, (1), (3) and (4) relate $\rec_{F,n}$ to $\rec_{F,1}$, 
and the latter is directly induced by the local reciprocity map, 
while (2) relates $\rec_{F,n}$ to local $L$ and $\ep$-factors, 
for which we have the two conventions of Artin and Deligne. 
Therefore, there are two normalizations of the collection of maps $\rec_{F,n}$ 
-- one compatible with the Artin normalization of $\rec_{F,1}$ and the Artin convention for local factors, 
and the other compatible with the Deligne normalization and convention.
\par

Passing from Artin to Deligne normalization composes $\rec_{F,n}$ with the contragredient operation on its source, 
equivalently (by (5)) on its target, for all $n$ simultaneously. 
Since for $n=1$, taking  contragredient is equivalent to composing with inversion, 
which is the operation that switches between the Artin and Deligne normalizations of local class field theory, 
we see that (1), (3), and (4) remain valid, and so does clearly (5). 
\par

To see that (2) also remains valid, 
we just note that on the left-hand side we have the Rankin--Selberg factors, 
which are intrinsically defined in terms of zeta integrals, 
and do not depend on the choice of normalization we are discussing, 
while on the right-hand side the factors are Galois-theoretic 
and depend on the choice of Artin vs.~Deligne convention. 
Since by the equations $L_D(s, \rho) = L_A(s, \rho^\vee)$ 
and $\ep_D(s, \rho,\psi_F) = \ep_A(s, \rho^\vee,\psi_F)$, i.e., since 
switching from one convention to the other with regards to the local factors 
has the effect of taking contragredient, 
we see from (5) that (2) will remain valid if we also compose $\rec_{F,n}$ with the contragredient operation.
\par

In the work of Harris--Taylor, the Deligne convention is used. 
If one prefers the Artin convention, 
one has to compose the result of their work with the contragredient operation.
\par

\subsection{The isomorphisms of Langlands and Tate--Nakayama}
To see how the Artin and Deligne conventions influence Arthur's endoscopic classification, 
we begin with the two most basic tools on which everything else is built: 
the Langlands correspondence for tori and the Tate--Nakayama isomorphism. Letting $\Gamma = \Gamma_F$ be the absolute Galois group of $F$, the Tate--Nakayama isomorphism is the isomorphism 
\[ 
H^{i-2}_\Tate(\Gamma, X_*(S)) \xrightarrow{\sim} H^i_\Tate(\Gamma, X_*(S) \otimes_\Z C_{\overline{F}}), 
\]
functorial in any $F$-torus $S$, 
given by taking the cup product against the fundamental class. 
Note that $X_*(S) \otimes_\Z C_{\overline{F}}$ equals $S(\overline{F})$ when $F$ is local, 
and $S(\A_{\overline{F}})/S(\overline{F})$ when $F$ is global. 
Passing from Artin to Deligne normalization inverts the fundamental class, 
hence composes the Tate--Nakayama isomorphism 
with the inversion automorphism on $X_*(S) \otimes_\Z C_{\overline{F}}$, 
equivalently with the negation automorphism of $X_*(S)$.
\par

There is also a duality version of the Tate--Nakayama isomorphism, 
where it takes the form of a pairing
\[ 
H^i(\Gamma, X^*(S)) \times H^{2-i}(\Gamma, X_*(S) \otimes_\Z C_{\overline{F}}) 
\rightarrow H^2(\Gamma,C_{\overline{F}}) \rightarrow \Q/\Z, 
\]
for $i=0,1,2$.
Here, the first map is given by taking the cup product, 
and the second map is the invariant map, see \cite[Section 3]{Kot1}. 
Passing from Artin to Deligne conventions inverts the invariant map, 
hence composes this duality pairing with the inversion automorphism of $X_*(S) \otimes_\Z C_{\overline{F}}$, 
equivalently the negation automorphism of $X^*(S)$.
\par

The Langlands correspondence for tori is the homomorphism
\[ 
H^1_\cts(W_F, \widehat{S}) \rightarrow \Hom_\cts(X_*(S) \otimes_\Z C_F, \C^\times), 
\]
functorial in any $F$-torus $S$, which is an isomorphism when $F$ is local, and surjective with kernel given by the locally-everywhere-trivial classes when $F$ is global.

We claim that passing from Artin to Deligne normalization also composes this homomorphism
with the inversion automorphism of $X_*(S) \otimes_\Z C_F$, 
equivalently the inversion automorphism of $\widehat{S}$. 
\par

A special case of this claim is an application of the discussion of Tate--Nakayama duality. 
Consider the subgroup $H^1(\Gal(E/F), \widehat{S}) \subset H^1_\cts(W_F, \widehat{S})$ 
for an arbitrary finite Galois extension $E/F$ splitting $S$. 
The restriction of the Langlands isomorphism to this subgroup can be described as follows: 
The exponential map induces the exact sequence $0 \to \Z \to \C \to \C^\times \to 1$,
which upon applying $X^*(S) \otimes_\Z -$ produces the exact sequence
\[ 
\begin{CD}
0 @>>> X^*(S) @>>> \Lie(\widehat{S}) @>>> \widehat{S} @>>> 1,
\end{CD}
\]
and the connecting homomorphism in the resulting long exact sequence of $\Gal(E/F)$-cohomology 
induces an isomorphism  $H^1(\Gal(E/F), \widehat{S}) \rightarrow H^2(\Gal(E/F),X^*(S))$. 
The cup product pairing between this group and $H^0_\Tate(\Gal(E/F), X_*(S) \otimes_\Z C_E)$ 
takes values in $H^2(\Gal(E/F),C_E)$, 
and its composition with the invariant map becomes a perfect pairing 
which identifies $H^1(\Gal(E/F), \widehat{S})$ 
with the group of characters of $X_*(S) \otimes_\Z C_F$ 
that vanish on the image of the norm map for the extension $E/F$. 
Switching from Artin to Deligne normalization inverts the invariant map, 
hence also this piece of the Langlands homomorphism.
\par

To discuss the full Langlands homomorphism, 
we need to review the Weil group $W_F$ 
and see how it is affected by the switch from Artin to Deligne normalization. 
We follow the exposition of \cite{Tate}, 
in which a ``Weil group for $F$'' is defined as a triple $(W_F, \varphi_F, \{r_E\})$, 
where $W_F$ is a topological group, 
$\varphi_F \colon W_F \rightarrow \Gamma_F$ is a continuous homomorphism with dense image, 
and for each finite extension $E/F$, 
the map $r_E$ is an isomorphism of topological groups $C_E \xrightarrow{\sim} W_E^\ab$. 
Here $W_E$ is defined as the preimage under $\varphi_F$ of $\Gamma_E \subset \Gamma_F$, $W_E^\ab$ is the quotient of $W_E$ by the closure $W_E^c$ of its commutator subgroup. 
It is required that the composition 
$\varphi_F \circ r_E \colon C_E \rightarrow \Gamma_E^\ab$ 
be the reciprocity map of class field theory. 
In particular, if $E/F$ is Galois and we define $W_{E/F} = W_F/W_E^c$, 
then we obtain an extension
\[
\begin{CD}
1 @>>> C_E @>r_E>> W_{E/F} @>\varphi_F>> \Gal(E/F) @>>> 1.
\end{CD}
\]
If we pass from Artin to Deligne normalization, 
then the reciprocity map is composed with inversion, 
and in order to keep the condition on $\varphi_F \circ r_E$ satisfied, 
we will compose $r_E$ with the inversion automorphism of $C_E$ for all $E/F$ 
(we cannot compose $\varphi_F$ with inversion, 
because $\varphi_F$ is a homomorphism between non-abelian groups). 
We emphasize again that, in a global situation, 
we have to do this for both the global field and all of its localizations, 
to keep the local--global compatibility intact.
\par

We now recall from \cite[Section 6]{Lab} that 
the Langlands homomorphism over $F$ is obtained from the commutative diagram 
\[ 
\xymatrix{
H^1_\cts(W_E, \widehat{S}) \ar[d] \ar[r]&\Hom_\cts(X_*(S)\otimes_\Z C_E, \C^\times) \ar[d]\\
H^1_\cts(W_F, \widehat{S}) \ar[r] & \Hom_\cts(X_*(S) \otimes_\Z C_F, \C^\times)
}
\]
where $E/F$ is any finite Galois extension splitting $S$, 
the left map is corestriction along the inclusion $W_E \hookrightarrow W_F$, 
and the right map is restriction along the inclusion $C_F \hookrightarrow C_E$. 
Since $\Gamma_E$ acts trivially on $\widehat{S}$, 
we have 
$H^1_\cts(W_E, \widehat{S}) = \Hom_\cts(W_E, \widehat{S}) = \Hom_\cts(W_E^\ab, \widehat S)$, 
which, via $r_E$, becomes identified with the group 
\[
\Hom_\cts(C_E, \widehat{S}) = \Hom_\cts(C_E, X^*(S) \otimes_\Z \C^\times) 
= \Hom_\cts(X_*(S) \otimes_\Z C_E, \C^\times).
\] 
Passing between Artin and Deligne normalizations composes $r_E$ with inversion, 
hence the Langlands isomorphism with inversion, as claimed.
\par

\subsection{Unramified local Langlands correspondence}\label{sec.unram}
Let $G$ be a connected reductive group over a non-archimedean local field $F$ that is \emph{unramified}, 
i.e. quasi-split and split over a finite unramified extension $E$ of $F$.
Fix a hyperspecial maximal compact subgroup $K$ of $G$ 
and consider the subset $\Irr_{\Ksph}(G)$ of $\Irr(G)$ consisting of those irreducible representations $\pi$ 
whose space $\pi^K$ of $K$-fixed vectors is non-zero.
Then the Satake isomorphism induces a bijection
\[
\Irr_{\Ksph}(G) \rightarrow \Phi^u(G), \, \pi \mapsto \phi_\pi, 
\]
where $\Phi^u(G)$ is the subset of $\Phi(G)$ consisting of $L$-parameters which are trivial on $I_F \times \SL_2(\C)$.
However, this bijection depends on the choice of Artin vs.~Deligne convention.
It also differs from Langlands' original bijection in \cite{L1}, because of a different normalization of the Satake isomorphism.
In this subsection, we recall its construction and explain how it is affected by these choices.
\par
 
Let $\HH(G,K)$ be the $\C$-algebra of bi-$K$-invariant compactly supported functions on $G$ 
equipped with the convolution product.
Then the map $\pi \mapsto \pi^K$ gives a bijection from $\Irr_{\Ksph}(G)$ 
to the set of isomorphism classes of simple $\HH(G,K)$-modules.
To describe $\HH(G,K)$ explicitly, 
we fix a Borel pair $(B,T)$ of $G$ such that $K$ is in good position with respect to $T$ (see \cite[Section 3.5]{Car}).
Let $A = A_T$ be the maximal split torus in $T$, 
so that 
\[
T/(T \cap K) \cong A/(A \cap K) \cong X_*(A).
\]
Here the first map is induced by the inclusion $A \hookrightarrow T$, 
and the second map is induced by the map $A \to X_*(A)$ 
that sends $a \in A$ to $\lambda_a \in X_*(A)=\mathrm{Hom}_\Z(X^*(A),\Z)$ 
given by $\lambda_a(\chi)=\mathrm{val}_F(\chi(a))$, 
with $\mathrm{val}_F \colon F^\times \to \Z$ being the normalized valuation.
Then the Satake isomorphism states that 
\[
\HH(G,K) \cong \C[X_*(A)]^W,
\]
where $W = N_G(T)/T$ is the Weyl group of $G$ (see \cite[Theorem 4.1]{Car}).
In particular, $\HH(G,K)$ is commutative and hence $\pi^K$ is $1$-dimensional for $\pi \in \Irr_{\Ksph}(G)$.
Thus we obtain a bijection
\begin{align*}
\Irr_{\Ksph}(G)
&\cong \Hom_{\Calg}(\HH(G,K), \C) \\
&\cong \Hom_{\Calg}(\C[X_*(A)]^W, \C) \\
&\cong (X^*(A) \otimes \C^\times)/W
=\widehat{A}/W.
\end{align*}
Now we need to choose a Frobenius element.
For a moment, we use the finite Galois form of the $L$-group $\widehat{G} \rtimes \Gal(E/F)$ 
and fix a generator $\sigma$ of the cyclic group $\Gal(E/F)$.
By \cite[Lemma 6.4]{B}, 
the inclusion $A \hookrightarrow T$ induces a surjection $\widehat{T} \twoheadrightarrow \widehat{A}$ 
and a bijection
\[
(\widehat{T} \rtimes \sigma)/\Int(N) \cong  \widehat{A}/W, 
\]
where $N$ is the inverse image of $W \cong (N_{\widehat{G}}(\widehat{T})/\widehat{T})^{\Gamma}$ 
in $N_{\widehat{G}}(\widehat{T})$.
Moreover, by \cite[Lemma 6.5]{B}, the inclusion $\widehat{T} \hookrightarrow \widehat{G}$ induces a bijection
\[
(\widehat{T} \rtimes \sigma)/\Int(N) \cong (\widehat{G} \rtimes \sigma)_{\mathrm{ss}}/\Int(\widehat{G}),
\]
where $(\widehat{G} \rtimes \sigma)_{\mathrm{ss}}$ is the set of semisimple elements in $\widehat{G} \rtimes \sigma$.
This gives rise to a bijection (depending on the choice of a generator $\Frob \in W_F/I_F$)
\[
\Irr_{\Ksph}(G) \rightarrow \Phi^u(G), \, \pi \mapsto \phi_\pi
\]
determined by
\[
\phi_\pi(\Frob) = t \rtimes \Frob
\]
up to $\widehat{G}$-conjugacy, 
where $t \in \widehat{T}$ is an element whose image in $\widehat{A}/W$ 
corresponds to $\pi$ under the canonical bijection $\Irr_{\Ksph}(G) \cong \widehat{A}/W$.
Note that this bijection is consistent with the local Langlands correspondence for tori described in the previous subsection, 
once $\Frob$ is fixed.
We write $\phi_\pi = \phi_\pi^A$ (\resp$\phi_\pi = \phi_\pi^D$) if we take $\Frob = \Frob_\ari$ (\resp$\Frob = \Frob_\geo$).
\par

Now we discuss the relation between $\phi_\pi^A$ and $\phi_\pi^D$.
For this, we need to introduce more notation.
For a moment, let $G$ be an arbitrary quasi-split connected reductive group over $F$.
Recall the Chevalley involution $\widehat{C} = C_{\widehat{G}}$ 
of the complex connected reductive group $\widehat{G}$. 
See \cite[Proposition 2.1]{AV}.
For a given pinning $(\widehat{B}, \widehat{T}, \{\widehat{X}_\alpha\})$, 
the involution $\widehat{C}$ is defined as the unique automorphism of $\widehat{G}$ 
that normalizes $\widehat{T}$ and acts as inversion on it, 
sends $\widehat{B}$ to the opposite Borel subgroup, 
and $\widehat{X}_\alpha$ to $-\widehat{X}_{-\alpha}$ 
such that $[\widehat{X}_\alpha, \widehat{X}_{-\alpha}] = \widehat{H}_\alpha$.
Since all pinnings of $\widehat{G}$ are conjugate, 
all involutions obtained this way are also conjugate\footnote{We could have also chosen to send $\widehat{X}_\alpha$ to $\widehat{X}_{-\alpha}$ and would have again obtained an involution in the same inner class.}. 
We can extend such an involution to ${}^LG$ by ensuring the pinning is $\Gamma$-stable, 
which implies that $\widehat{C}$ commutes with the $\Gamma$-action, 
and then taking the automorphism ${}^LC = \widehat{C} \rtimes \id_\Gamma$ of ${}^LG = \widehat{G} \rtimes \Gamma$. 
Since all $\Gamma$-stable pinnings are conjugate under $\widehat{G}^\Gamma$, 
the same holds for all versions of ${}^LC$ obtained this way.

\begin{lem} \label{lem:urpar}
Suppose that $G$ is unramified, and let $K$ be a hyperspecial maximal compact subgroup of $G$.
For $\pi \in \Irr_{\Ksph}(G)$, let $\phi_\pi^A$ (\resp $\phi_\pi^D$) be the $L$-parameter of $\pi$ in the Artin (\resp Deligne) convention.
Then we have
\[
 \phi^D_\pi = {}^L C \circ \phi^A_\pi.
\]
\end{lem}
\begin{proof}
Recall that
\begin{align*}
 \phi^A_\pi(\Frob_\ari) & = t \rtimes \Frob_\ari, \\
 \phi^D_\pi(\Frob_\geo) & = t \rtimes \Frob_\geo
\end{align*}
for some common $t \in \widehat{T}$ whose image in $\widehat{A}/W$ 
corresponds to $\pi$ under the canonical bijection $\Irr_{\Ksph}(G) \cong \widehat{A}/W$.
Since the natural map $\widehat{T} \rightarrow \widehat{A}$ restricted to $\widehat{T}^\Gamma$ 
is surjective by \cite[p.~37, (3)]{B}, 
we may assume that $t$ is $\Gamma$-fixed.
Then we have
\begin{align*}
 \phi^D_\pi(\Frob_\ari)
 & = \phi^D_\pi(\Frob_\geo)^{-1} = t^{-1} \rtimes \Frob_\ari \\
 & = \widehat{C}(t) \rtimes \Frob_\ari = {}^L C \circ \phi_\pi^A(\Frob_\ari).
\end{align*}
This implies the lemma.
\end{proof}

Let $(\tau, V)$ be a finite-dimensional complex representation of ${}^LG$. 
For $\pi \in \Irr_{\Ksph}(G)$ with the $L$-parameter $\phi_\pi$, 
define the \emph{unramified $L$-factors} $L_A(s, \pi, \tau)$ and $L_D(s, \pi, \tau)$ by 
\begin{align*}
L_A(s, \pi, \tau) &= \det(1-q^{-s} \cdot \tau \circ \phi_\pi^A(\Frob_\ari))^{-1}, \\
L_D(s, \pi, \tau) &= \det(1-q^{-s} \cdot \tau \circ \phi_\pi^D(\Frob_\geo))^{-1}.
\end{align*}
Then we have
\[
L_D(s, \pi, \tau) = L_A(s, \pi^\vee, \tau^\vee).
\]
Indeed, since ${}^LC \circ \phi_\pi^A \cong \phi_{\pi^\vee}^A$, we have
\begin{align*}
\tau \circ \phi_\pi^D(\Frob_\geo) 
&= \tau \circ {}^LC \circ \phi_\pi^A(\Frob_\ari)^{-1}
\\&= \tau \circ \phi_{\pi^\vee}^A(\Frob_\ari)^{-1}.
\end{align*}
Its characteristic polynomial is equal to that of $\tau^\vee \circ \phi_{\pi^\vee}^A(\Frob_\ari)$.

\begin{rem}
Let us point out that the above construction is slightly different from Langlands' original construction in \cite{L1}.
Recall that the Satake isomorphism $\HH(G,K) \cong \C[X_*(A)]^W$ is obtained from the Satake transform
\[
S \colon \HH(G,K) \rightarrow \HH(T,T \cap K)^W
\]
given by 
\[
 S(f)(t)=\delta_B(t)^{\frac{1}{2}}\int_U f(tu)du
\]
for $f \in \HH(G,K)$ and $t \in T$, together with the identification
\[
 T/(T\cap K)\cong A/(A\cap K)\cong X_*(A)
\]
fixed above.
Langlands uses the same Satake transform, 
but the opposite identification of $A/(A\cap K)$ with $X_*(A)$: 
in \cite[Chapter 2]{L1}, he uses the normalization $|\xi_\alpha(t)| = p^{\lambda(t)(\alpha)}$, 
which amounts to saying that the class of $a \in A$ is sent to the element $\lambda \in X_*(A)$ such that 
\[
\lambda(\chi) = -\mathrm{val}_F(\chi(a))
\]
for $\chi\in X^*(A)$.
Thus his bijection $\Irr_{\Ksph}(G) \cong \widehat{A}/W$ differs from ours 
by the involution $t\mapsto t^{-1}$ on $\widehat A/W$.
We write $\phi_\pi$ and $\phi_\pi^L$ for the unramified $L$-parameters corresponding to $\pi \in \Irr_{\Ksph}(G)$ 
under our bijection and Langlands' bijection, respectively, once $\Frob$ is fixed.
If $\phi_\pi(\Frob)=t \rtimes \Frob$, then the above comparison gives
\[
\phi_\pi^L(\Frob)=t^{-1}\rtimes \Frob.
\]
Hence, as in Lemma \ref{lem:urpar}, we have
\[
\phi_\pi^L={}^LC\circ \phi_\pi.
\]
\end{rem}

\subsection{Endoscopic transfer}
Geometric and spectral endoscopic transfers are governed by 
the transfer factors defined in \cite{LS} and \cite{KoSh}, 
in the setting of ordinary and twisted endoscopy, respectively. 
Since twisted endoscopy generalizes ordinary endoscopy, 
the factors of \cite{KoSh} ought to specialize to the factors of \cite{LS}. 
Moreover, since Arthur's endoscopic classification of representations uses 
both ordinary and twisted endoscopy between various pairs of groups, 
it is important for all factors to be normalized compatibly.
\par

The reader needs to be aware that the definition given in \cite{KoSh} is incorrect. 
There are two ways to correct it, 
depending on whether one uses the Artin or the Deligne convention for class field theory. 
But neither of these corrections specializes to the definition of \cite{LS}. 
Consequently, the latter definition also needs to be modified. 
All this is explained carefully in \cite[Section 5]{KoSh2} 
and we will content ourselves with a brief summary.
\par

Let $F$ be a local field and let $G$ be a quasi-split connected reductive $F$-group. 
Fix an automorphism $\theta$ of $G$ that preserves an $F$-pinning 
and a $\theta$-stable pair $(B,\chi)$ consisting of a Borel subgroup $B$ over $F$
and a generic character $\chi$ of the unipotent radical $U$ of $B$. 
Let $(H,s,\HH,\eta)$ be an endoscopic datum and let $(H_1,\eta_1)$ be a $z$-pair. 
Associated to these data, there are two normalizations of the transfer factors: 
the Artin normalization $\Delta_A$ (denoted by $\Delta'$ in \cite[Section 5]{KoSh2}) 
and the Deligne normalization $\Delta_D$. 
They are defined in \cite[(5.5.1), (5.5.2)]{KoSh2} as 
\[ 
\Delta_A = \epsilon \cdot (\Delta_I^\new \Delta_{III})^{-1} \Delta_{II} \Delta_{IV} 
\]
and 
\[ 
\Delta_D = \epsilon \cdot \Delta_I^\new \Delta_{III} \Delta_{II}^{-1} \Delta_{IV}, 
\]
where the terms $\epsilon, \Delta_{II}, \Delta_{III}, \Delta_{IV}$ are defined in \cite{KoSh}, 
and the corrected term $\Delta_I^\new$ is defined in \cite{KoSh2}. We note here that both $\Delta_A$ and $\Delta_D$ use the same $\epsilon$-factor, namely the one normalized according to Artin's convention; the construction of the remaining pieces uses 
both the Tate--Nakayama and the Langlands isomorphism, 
and in the above formulas, 
these two isomorphisms have been normalized according to the conventions in \cite{KoSh}, 
namely the Artin normalization of the Tate--Nakayama isomorphism 
and the Deligne normalization of the Langlands isomorphism. 
This is the reason why $\Delta_A$ and $\Delta_D$ are given by different formulas. 
\par

In the case of trivial twisting, i.e. $\theta=1$, 
these factors can also be described using the terms 
$\Delta_I, \Delta_{II}, \Delta_{III_1}, \Delta_{III_2}, \Delta_{IV}$ defined in \cite{LS}, 
namely as
\[ 
\Delta_A = \epsilon \cdot (\Delta_I\Delta_{III_1})^{-1}\Delta_{II}\Delta_{III_2}\Delta_{IV} 
\]
and 
\[ 
\Delta_D = \epsilon \cdot \Delta_I\Delta_{III_1}\Delta_{II}^{-1}\Delta_{III_2}^{-1}\Delta_{IV} 
= \epsilon \cdot \Delta_I\Delta_{III_1}\Delta_{II}\Delta_{III_2,D}\Delta_{IV}, 
\]
where $\Delta_{III_2,D}$ is obtained from $\Delta_{III_2}$ 
by first inverting the $\chi$-data that is being used in the definition, 
and then inverting the entire factor, see \cite[(5.1.2)]{KoSh2}. 
Note also that in the setting of ordinary endoscopy, 
the twisted $\Delta_{III}$ breaks up as $\Delta_{III_1} \cdot \Delta_{III_2}^{-1}$.
\par

\subsection{The Arthur--Langlands conjectures}
From now on, we assume that the characteristic of the base field is zero.
Before we discuss the effect of Artin and Deligne normalizations on the Arthur--Langlands conjectures, we briefly recall them. Their main thrust can be summarized as follows.

\begin{conj} \label{cnj:a}
Let $G$ be a quasi-split connected reductive group over a local or global field $F$ of characteristic zero. 
\begin{enumerate}
\item 
For $F$ global, there exists a set $\Psi_2(G)$ and for each $\psi \in \Psi_2(G)$ a subrepresentation $L^2_\psi(G) \subset L^2_\disc(G)$  such that
\[
L^2_\disc(G) = \bigoplus_{\psi \in \Psi_2(G)} L^2_\psi(G).
\]
\item 
For $F$ local, 
there is a multi-set $\Pi_\psi(G)$ over $\Irr_\unit(G)$ attached to each $\psi \in \Psi(G)$ such that
\[
\Irr_\temp(G) = \bigsqcup_{\phi \in \Phi_{\temp}(G)}\Pi_\phi(G).
\]
\item 
For $F$ local and $\psi \in \Psi(G)$, 
there is a stable distribution $f \mapsto \psi(f)$ supported on $\Pi_\psi(G)$, 
as well as a map $\Pi_\psi(G) \rightarrow \Irr(\Sc_{\psi})$, depending on a choice of Whittaker datum $\ww$ for $G$. 
We denote it by $\pi \mapsto \rho_{\ww, \pi}$. 
For each semisimple $s \in S_\psi$ we have
\[
\tag{{\bf ECR}} \label{ecr}
\sum_{\pi \in \Pi_\psi(G)} \pair{\pi, s \cdot s_\psi}_{\ww, \psi} \cdot \pi(f) = \psi'(f'),
\]
where $\pair{\pi, \cdot}_{\ww, \psi} =\tr(\rho_{\ww, \pi} (\cdot))$ and $f \mapsto f'$ is the endoscopic transfer of functions.
\item 
For $F$ global, 
there is a localization map 
\[ 
\Psi_2(G) \rightarrow \Psi(G_v),\quad \psi \mapsto \psi_v 
\] 
for any place $v$ of $F$, 
and 
\[ 
L_\psi^2(G) = \bigoplus_\pi \pi^{m(\psi,\pi)}, 
\]
where $\pi$ runs over all irreducible admissible representations $\pi=\otimes'_v\pi_v$ of $G(\A)$ 
which satisfy $\pi_v \in \Pi_{\psi_v}(G_v)$ for all places $v$, 
and 
\[ 
m(\psi,\pi) = \mathrm{mult}(\epsilon_\psi, \otimes_v (\rho_{\ww_v,\pi_v}|_{\Sc_\psi})), 
\]
where $\ww_v$ are the localizations of a fixed global Whittaker datum $\ww$ for each place $v$. 
Here, $\epsilon_\psi$ is the quadratic character of $\Sc_\psi$ defined explicitly in terms of $\psi$ as in \cite[(8.4)]{Ar2}.

\item
When $F$ is non-archimedean and $\psi \in \Psi(G)$ is unramified, 
$\Pi_\psi$ has an unramified representation $\pi$
(with respect to a fixed hyperspecial maximal compact subgroup)
with multiplicity one
such that 
\[
\phi_\psi = \phi_\pi^*,
\]
where $* = A$ (\resp $* = D$) if we use Artin's (\resp Deligne's) convention for class field theory. 
Moreover, it corresponds to the trivial representation of $\Sc_\psi$. 
\end{enumerate}
\end{conj}

Let ${}^LC = \widehat{C} \rtimes \id_\Gamma$ be the Chevalley involution of ${}^LG = \widehat{G} \rtimes \Gamma$
defined in Section \ref{sec.unram}.
Given $\psi \in \Psi(G)$, define $\psi^\vee = {}^LC \circ \psi$. 
Note that $(\psi^\vee)^\vee = \psi$ and that $\widehat  C$ induces an isomorphism $\Sc_\psi \to \Sc_{\psi^\vee}$. 
For the group $\GL_N$ and the standard pinning, 
we have $\widehat{C}(g) = {}^tg^{-1}$. In particular, if $\tau \colon \GL_N(\C) \rightarrow \GL_M(\C)$ is a representation, 
then $\tau \circ \widehat{C}$ is isomorphic to the contragredient of $\tau$, 
and if $\rho \colon \Delta \rightarrow \GL_N(\C)$ is a representation of some group $\Delta$, 
then $\widehat{C} \circ \rho$ is isomorphic to the contragredient of $\rho$. 
Therefore, for any $G$, we may think of $\psi^\vee$ as ``the contragredient'' of $\psi$. 
The following result, whose proof will be given further below, lends some credence to this thought.
\par

\begin{prop} \label{pro:contra}
Fix a quasi-split group $G$ and consider the validity of Conjecture \ref{cnj:a} for all endoscopic groups $G'$ of $G$.
\begin{enumerate}
\item 
Assume that $F$ is local and Conjecture \ref{cnj:a} (2)--(3) holds (in either convention). 
Assume further that for all $G'$ and $\psi' \in \Psi(G')$, 
the equation $\psi'^\vee(f') = \psi'(f'\circ i')$ holds, 
where $i'$ is the inversion anti-automorphism of $G'$. 
For any $\psi \in \Psi(G)$, consider $\Pi_\psi(G)^\vee = \{\pi^\vee \,|\, \pi \in \Pi_\psi(G) \}$. 
Then 
\[ 
\Pi_{\psi^\vee}(G) = \Pi_\psi(G)^\vee,
\quad 
\pair{\pi^\vee,\widehat{C}(x^{-1})}_{\ww^{-1},\psi^\vee} = \pair{\pi,x}_{\ww,\psi}
\]
for $\pi \in \Pi_\psi(G)$ and $x \in \Sc_\psi$, 
where $\ww^{-1} = (B, \chi^{-1})$ for $\ww = (B, \chi)$.
\item 
Assume that $F$ is global and Conjecture \ref{cnj:a} (1)--(4) holds (in either convention). 
Assume again that for all $G'$ and $\psi' \in \Psi(G')$, 
the equation $\psi'^\vee(f') = \psi'(f'\circ i')$ holds. 
For any $\psi \in \Psi_{2}(G)$ and $\pi \in \Pi_\psi(G)$, 
we have
\[ 
m(\pi^\vee,\psi^\vee) = m(\pi,\psi). 
\]
\item 
If $G$ is a quasi-split classical group, 
then the identity $\psi^\vee(f)=\psi(f\circ i)$ holds for Arthur's construction (in either convention). 
Since all $G'$ are products of classical groups, 
we see that parts (1) and (2) hold for $G$.
\end{enumerate}
\end{prop}

Let us now turn to the effect of the Artin and Deligne normalizations on the Arthur--Langlands conjectures. We will write $\phi^A_\pi$ and $\phi^D_\pi$ for the Langlands parameter associated to a representation $\pi$ with respect to either convention. 
We have the following cues:
\begin{enumerate}
    \item In the local correspondence for $G=\mathrm{GL}_N$, switching between the two normalizations is accomplished by taking the contragredient of the representation of $G(F)$, equivalently of its Langlands parameter. That is the same as composing the Langlands parameter with the Chevalley involution of $\mathrm{GL}_N(\C)$. Thus we have $\phi^D_\pi=\widehat C \circ \phi^A_\pi$.
    \item If $G$ is an unramified group, then we saw in Lemma \ref{lem:urpar} that $\phi_\pi^D={^LC}\circ\phi_\pi^A$. 
    \item It is expected that there should be an independent definition of an $L$-function $L(s,\pi,\tau)$ associated to an automorphic representation $\pi$ of $G(\A_F)$ and a finite-dimensional representation $\tau \colon {^LG} \to \mathrm{GL}(V)$, which does not reference the Galois side, and that $L(s,\pi,\tau)=L_*(s,\tau\circ\phi_\pi^*)$, where $\phi_\pi^*$ is the $L$-parameter of $\pi$ and $* \in \{A,D\}$ is either of the two normalization conventions. It is not known how to define $L(s,\pi,\tau)$ in general, but there are some known examples, such as that of Godement--Jacquet \cite{GJ}, where $G=\mathrm{GL}_N$ and $\tau$ is the standard representation of $\widehat{G}=\mathrm{GL}_N(\C)$ and trivial on $\Gamma$. In that setting $L(s,\pi,\tau)$ is insensitive to the normalization of class field theory. One might expect that, for general groups $G$, the definition of the automorphic $L$-function $L(s,\pi,\tau)$ might be insensitive to the normalization of class field theory \emph{provided that $^LG$ is a direct product $\widehat{G} \times \Gamma$ and $\tau$ is trivial on $\Gamma$.} This implies that we should have $L_A(s,\tau\circ\phi_\pi^A)=L_D(s,\tau\circ\phi_\pi^D)$, pointing towards $\tau\circ\phi_\pi^D=(\tau\circ\phi_\pi^A)^\vee$ as finite-dimensional representations of the Langlands group of the base field. Now on the one hand $(\tau\circ\phi_\pi^A)^\vee=\tau^\vee\circ\phi_\pi^A$, while on the other hand $\tau^\vee=\tau\circ \widehat C$; the latter can be seen by first reducing to the case that $\tau$ is irreducible and then noting that the weights of $\tau\circ\widehat C$ are exactly the negatives of those of $\tau$. This again points to $\phi_\pi^D=\widehat C \circ \phi_\pi^A$.
\end{enumerate}
With regards to (3) we want to point out that if we do not assume that $^LG = \widehat{G} \times \Gamma$ and that $\tau$ is trivial on $\Gamma$, then it is not reasonable to expect that $L(s,\pi,\tau)$ is independent of the choice of normalization of class field theory. A degenerate example would be when $G=\{1\}$, in which case $^LG=\Gamma$ and $L(s,\pi,\tau)$ is simply the $L$-function of the Galois representation $\tau$.

A less degenerate example is the unramified $L$-function $L(s,\pi,\tau)$, where both $G$ and $\pi$ are unramified, where we saw that there are two normalizations linked by the identity
\[ L_D(s, \pi, \tau) = L_A(s, \pi^\vee, \tau^\vee). \]
Note that, in this setting, if $G$ is in fact split and $\tau$ is trivial on the $\Gamma$-factor of $^LG=\widehat{G} \times \Gamma$, then 
\begin{align*}
L_D(s,\pi,\tau) 
&= L_D(s,\tau\circ\phi_\pi^D) 
\\&= L_A(s,(\tau\circ\phi_\pi^D)^\vee)
\\&= L_A(s,\tau^\vee\circ\phi_\pi^D)
\\&= L_A(s,\tau^\vee\circ\widehat C\circ\phi_\pi^A)
\\&= L_A(s,\tau\circ\phi_\pi^A)
=L_A(s,\pi,\tau),
\end{align*}
consistent with (3) above.

We are thus led to expect the following: 

\begin{conj} \label{cnj:b}
Let $G$ be a quasi-split connected reductive group over a local or global field $F$. 
\begin{enumerate}
\item 
If $F$ is local and $\Pi_\psi(G)^A$ (\resp $\Pi_\psi(G)^D$) is the $A$-packet associated to $\psi \in \Psi(G)$ 
by the Artin (\resp Deligne) convention of Conjecture \ref{cnj:a}, 
then 
\[ 
\Pi_\psi(G)^D = (\Pi_\psi(G)^A)^\vee = \Pi_{\psi^\vee}(G)^A, 
\]
and for $\pi \in \Pi_\psi(G)^D$ and $x \in \Sc_\psi$, 
we have
\[ 
\pair{\pi,x}^D_{\ww,\psi} = \pair{\pi^\vee,x^{-1}}^A_{\ww^{-1},\psi} = \pair{\pi,\widehat{C}^{-1}(x)}_{\ww, \psi^\vee}^A. 
\] 
\item 
If $F$ is global and $L^2_\psi(G)^A$ (\resp $L^2_\psi(G)^D$) is the $\psi$-constituent of $L^2_\disc(G)$ 
associated to $\psi \in \Psi_2(G)$ by the Artin (\resp Deligne) convention, 
then 
\[ 
L^2_\psi(G)^D = (L^2_{\psi}(G)^A)^\vee = L^2_{\psi^\vee}(G)^A, 
\quad m(\pi^\vee,\psi^\vee)=m(\pi,\psi). 
\]
\end{enumerate}
\end{conj}

\begin{prop} \label{pro:switch}
Let $G$ be a quasi-split connected reductive group over a local or global field $F$. 
\begin{enumerate}
\item 
Assume that $F$ is local and Conjecture \ref{cnj:a} (2)--(3) holds in Artin's convention. 
For any $\psi \in \Psi(G)$, $\pi \in \Pi_\psi(G)^A$, and $x \in \Sc_\psi$, 
define
\[ 
\Pi_\psi(G)^D = (\Pi_\psi(G)^A)^\vee,
\quad \pair{\pi^\vee, x}_{\ww,\psi}^D = \pair{\pi, x^{-1}}_{\ww^{-1},\psi}^A. 
\]
Then Conjecture \ref{cnj:a} (2)--(3) holds in Deligne's convention.
\item 
Assume that $F$ is global and Conjecture \ref{cnj:a} (1)--(4) holds in Artin's convention. 
For any $\psi \in \Psi_2(G)$, 
define 
\[ 
L^2_\psi(G)^D = (L^2_\psi(G)^A)^\vee. 
\]
Then Conjecture \ref{cnj:a} (1)--(4) holds in Deligne's convention.
\item 
In the setting of a quasi-split classical group $G$, 
let $\psi^A(f)$ and $\psi^D(f)$ be the Artin and Deligne normalizations of the stable character 
corresponding to $\psi \in \Psi(G)$ according to Arthur's construction (for quasi-split symplectic, orthogonal, or unitary groups). 
Then $\psi^D(f)=\psi^A(f \circ i_G)$, 
where $i_G$ is the inversion anti-automorphism of $G$. 
In particular, Arthur's construction satisfies Conjecture \ref{cnj:b}.
\end{enumerate}
\end{prop}

We will now give the proof of Propositions \ref{pro:contra} and \ref{pro:switch}, which are closely related. 
The generic part of Proposition \ref{pro:contra} (1) was proved in \cite{Kal_gc}. 
These arguments extend to prove both propositions, 
as was noted in the setting of unitary groups by 
Bertoloni Meli and Nguyen in \cite[Section 2.2]{BMN}.
\par

The technical heart of the proof is the following lemma, 
whose statement needs a bit of preparation. 
Given an endoscopic datum $\ee=(H,s,\HH,\eta)$ for $G \rtimes \theta$ 
we write ${}^LC(\ee)$ for the quadruple $(H,s',\HH',\eta')$, 
where $s'=\widehat{C}(s^{-1})$, 
$\HH'$ is the same group as $\HH$ but with the embedding $\widehat{H} \to \HH$ 
composed with $\widehat{C}^H$, 
and $\eta'={}^LC\circ{}^L\theta \circ \eta$. 
Given a $z$-pair $\zz=(H_1,\eta_{1})$ for $\ee$ 
we write ${}^LC^H(\zz)$ for the pair $(H_1,{}^LC^{H_1}\circ\eta_{1})$.

Recall that the endoscopic transfer of a compactly supported function on $G(F) \rtimes \theta$ to $H_1(F)$ is generally not compactly supported. Indeed, as discussed in \cite[\S5.5]{KoSh}, such a function $f_1$ transforms under the kernel of $H_1(F) \to H(F)$ by a character $\lambda=\lambda_{H_1}$, and only the image in $H(F)$ of its support is compact. We shall write $C^\infty_{\lambda,c}(H_1(F))$ for the space of these functions.

\begin{lem}\label{lem:tfsw}
Let $(G,\theta)$ be a quasi-split twisted group, 
let $\ee = (H, s, \HH,\eta)$ be an endoscopic datum, 
and let $\zz = (H_1, \eta_1)$ be a $z$-pair. 
Denote by $i_\theta$ and $i_1$ the inversion anti-automorphisms on $G \rtimes \<\theta\>$ and $H_1$, respectively. 
\begin{enumerate}
\item 
Assume that $f \in C^\infty_c(G(F) \rtimes \theta)$ and $f_1 \in C^\infty_{\lambda,c}(H_1(F))$ 
are matching functions with respect to the transfer factor $\Delta[\ww,\ee,\zz]$ (in either convention). 
Then the functions $f \circ i_\theta \in C^\infty_c(G(F) \rtimes \theta^{-1})$ 
and $f_1 \circ i_1 \in C^\infty_{\lambda,c}(H_1(F))$ are matching 
with respect to the transfer factor $\Delta[\ww^{-1}, {}^LC(\ee), {}^LC_H(\zz)]$ (in the same convention).

\item 
Assume that $f \in C^\infty_c(G(F) \rtimes \theta)$ and $f_1 \in C^\infty_{\lambda,c}(H_1(F))$ are matching functions 
with respect to the transfer factor $\Delta_A[\ww, \ee, \zz]$. 
Then the functions $f \circ  i_\theta \in C^\infty_c(G(F) \rtimes \theta^{-1})$ 
and $f_1 \circ i_1 \in C^\infty_{\lambda,c}(H_1(F))$ are matching 
with respect to the transfer factor $\Delta_D[\ww^{-1},\ee',\zz]$, 
where $\ee' = (H, s^{-1}, \HH,\eta)$.
\end{enumerate}
\end{lem}

Admitting this lemma, whose proof we will give at the end of this section, 
we complete the proofs of Propositions \ref{pro:contra} and \ref{pro:switch} as follows.
\par

Consider first parts (1) of these propositions; thus $F$ is local. 
The $A$-packets $\Pi_\psi(G)$ and the maps $\Pi_\psi(G) \rightarrow \Irr(\Sc_\psi)$ are uniquely characterized 
by the stable characters $\psi(f)$ and the identities \eqref{ecr}. 
In Proposition \ref{pro:contra} (1), we are assuming $\psi^\vee(f) = \psi(f\circ i)$ 
and it is enough to show that Equation \eqref{ecr} for a fixed $\psi$ and all $s \in \Sc_\psi$ 
implies the same equation for $\psi^\vee$ and all $s^\vee \in \Sc_{\psi^\vee}$, 
while for Proposition \ref{pro:switch} (1), 
we are defining $\psi^D(f) = \psi^A(f \circ i)$ 
and we want to show that the validity of Equation \eqref{ecr} for a fixed $\psi$ and all $s \in \Sc_\psi$ 
in Artin's convention implies its validity in Deligne's convention, 
where the two pairings are related as in the statement of Proposition \ref{pro:switch} (1).
\par

For Proposition \ref{pro:contra} (1), 
we will give a proof which works in either convention. 
We compute 
\begin{align*}
\psi'^\vee(f'\circ i') &= \psi'(f')\\
&= \sum_{\pi \in \Pi_\psi(G)} \pair{\pi, s \cdot s_\psi}_{\ww,\psi} \cdot \pi(f) \\
&= \sum_{\pi^\vee \in \Pi_\psi(G)^\vee} \pair{\pi, s \cdot s_\psi}_{\ww,\psi} \cdot \pi^\vee(f \circ i).
\end{align*}
According to Lemma \ref{lem:tfsw}(1), 
the functions $f \circ i$ and $f' \circ i'$ are matching 
with respect to $\Delta[\ww^{-1}, {}^LC(\ee), {}^LC_H(\zz)]$. This implies that the collection $\{\pi^\vee \,|\, \pi \in \Pi_\psi(G)\}$ constitutes the $A$-packet for the parameter for $G$ 
whose transfer via ${}^LC(\ee)$ and ${}^LC_H(\zz)$ equals $\psi'^\vee$, 
but this parameter is $\psi^\vee$. 
In other words, we conclude that
\[ 
\Pi_{\psi^\vee}(G) = \{\pi^\vee \,|\, \pi \in \Pi_\psi(G)\}. 
\]
Furthermore, equating scalars in the above equation with those in Equation \eqref{ecr} 
for the parameter $\psi^\vee$, 
we conclude
\[ 
\pair{\pi^\vee, \widehat{C}(x^{-1})}_{\ww^{-1},\psi^\vee} = \pair{\pi, x}_{\ww,\psi}, 
\]
which completes the proof of Proposition \ref{pro:contra} (1).
\par

The proof of Proposition \ref{pro:switch} (1) is similar. First, part (2) of Conjecture \ref{cnj:a} in Artin's convention directly implies part (2) in Deligne's convention, since the contragredient operation preserves $\Irr_{\temp}(G)$. So we consider part (3). We first need to associate to $\psi$ a stable character $\psi^D(f)$. We define $\psi^D(f) \coloneqq \psi^A(f \circ i)$, which is a stable character, in fact equal to $(\psi^\vee)^A(f)$. Next we want to verify Equation \eqref{ecr} for given $\psi$ and all $s \in \Sc_\psi$ in Deligne's convention, provided it holds in Artin's convention. Fix $s \in \Sc_\psi$ and let $\ee=(H,s,\HH,\eta)$, $\zz$, and $\psi'$, correspond to $(\psi,s)$. 
Let $f \in C^\infty_c(G(F))$ and $f' \in C^\infty_{\lambda,c}(H_1(F))$ be matching with respect to $\Delta[\ww,\ee,\zz]$. Then, as in the proof of Proposition \ref{pro:contra} (1), we obtain the identity 
\begin{align*}
\psi'^D(f' \circ i') = \psi'^A(f')
&= \sum_{\pi \in \Pi_\psi(G)^A}\pair{\pi, s \cdot s_\psi}_{\ww, \psi}^{A} \cdot \pi^\vee(f \circ i) 
\\&= \sum_{\sigma \in \Pi_\psi(G)^D}\pair{\sigma, s^{-1} \cdot s_\psi}_{\ww^{-1}, \psi}^{D} \cdot \sigma(f \circ i).
\end{align*}
According to Lemma \ref{lem:tfsw}(2), 
the functions $f \circ i$ and $f' \circ i'$ are matching with respect to $\Delta_D[\ww^{-1},\ee',\zz]$, from which we conclude that Equation \eqref{ecr} holds for $(\psi,s^{-1})$ in the Deligne convention. This completes the proof of Proposition \ref{pro:switch} (1).
\par

Next we move to the proof of parts (3) of these propositions. 
They claim that, in Arthur's construction, 
we have
\[ 
\psi^D(f)=\psi^A(f \circ i) = (\psi^A)^\vee(f), 
\]
for $F$ both local and global. 
The global case reduces immediately to the local case. 
In the local case, the stable character $\psi^*(f)$ is uniquely determined by the twisted transfer identity
\[ 
\psi^*(f) = \tl\pi_{\tau \circ \psi}^*(f^N). 
\]
Let us explain the notation. 
The superscript $*$ stands for either of the normalizations of Artin or Deligne, 
$f^N \in C^\infty_c(\GL_N(F) \rtimes \theta)$ is an arbitrary test function, 
and $f \in C^\infty_c(G(F))$ is its twisted transfer with respect to the transfer factor $\Delta_*[\ww,\ee,\tau]$, 
normalized with respect to the convention used. 
Here $\ee$ stands for the endoscopic datum that realizes $G$ as a twisted endoscopic group of $\GL_N$, 
$\tau$ is the standard representation ${}^LG \rightarrow \GL_N(\C)$ used by Arthur, 
$\pi_{\tau\circ\psi}^*$ is the representation of $\GL_N(F)$ associated to the $L$-parameter $\phi_{\tau \circ \psi}$ 
by the local Langlands correspondence for $\GL_N$ normalized according to the convention being used, 
and $\tl\pi_{\tau \circ \psi}^*$ is its Whittaker-normalized extension to $\GL_N(F) \rtimes \pair{\theta}$.
\par

Now we claim that for any irreducible $\theta$-stable representation $\pi$ of $\GL_N(F)$
associated to an $A$-parameter, the contragredient of its Whittaker-normalized extension to $\GL_N(F) \rtimes \<\theta\>$ is equal to the Whittaker-normalized extension of its contragredient.
For this, it is enough to show that 
a non-degenerate $\GL_N(F)$-invariant pairing $\pair{\cdot, \cdot} \colon \pi \times \pi^\vee \rightarrow \C$
is in fact a $\GL_N(F) \rtimes \pair{\theta}$-invariant pairing.
\par

Suppose first that $\pi$ is tempered. 
Notice that $\pi$ is unitary so that $\pi^\vee$ is isomorphic to the complex conjugate of $\pi$. 
Hence we may consider a $\GL_N(F)$-invariant inner product on $\pi$ instead of a bilinear form on $\pi \times \pi^\vee$.
Fix a non-trivial Whittaker functional $\omega$ on $\pi$. 
By \cite[Theorem 6.2, Theorem A]{Ber-P}, 
the integral 
\[
\pair{v,v'} = \int_{U \bs \mathcal{M}} \omega(\pi(g)v) \overline{\omega(\pi(g)v')} dg
\]
converges and realizes a non-degenerate $\GL_N(F)$-invariant inner product on $\pi$, 
where $\mathcal{M}$ is the standard mirabolic subgroup
and $U$ is the standard maximal unipotent subgroup. 
Then we have $\pair{\tl\pi(1 \rtimes \theta)v, \tl\pi(1 \rtimes \theta)v'} = \pair{v,v'}$ by definition.
In general, letting $\II_\pi$ be the standard module of $\pi$, 
a pairing $\pair{\cdot, \cdot}$ can be induced from 
a pairing for standard modules $\pair{\cdot, \cdot} \colon \II_\pi \times \II_\pi^\vee \rightarrow \C$.
By Lemma \ref{End}, there is a constant $c \in \C^\times$ such that
\[
\pair{\theta_W v, \theta_W^\vee v^\vee} = c \pair{v,v^\vee}
\]
for $v \in \II_\pi$ and $v^\vee \in \II_\pi^\vee$, 
where $\theta_W$ (\resp $\theta_W^\vee$) is the action of $\theta$ on $\II_\pi$ (\resp $\II_\pi^\vee$)
normalized by Whittaker functionals with respect to $\ww$ (\resp $\ww^{-1}$). 
Considering the unique irreducible tempered subrepresentation of $\II_\pi$ (Lemma \ref{psiD}), 
we see that $c=1$.
Since the diagram 
\[
\begin{CD}
\pi^\vee @>\tl\pi^\vee(1 \rtimes \theta)>> \pi^\vee \\
@VVV @VVV \\
\II_\pi^\vee @>\theta_W^\vee>> \II_\pi^\vee 
\end{CD}
\]
is commutative, which follows from Theorem \ref{main2}, 
we obtain the claim.
\par

The claim implies that $\tl\pi_{\tau\circ\psi}(f^N)=\tl\pi^\vee_{\tau\circ\psi}(f^N \circ i_N)$, where $\tl\pi^\vee_{\tau\circ\psi}$ denotes the Whittaker extension of the contragredient of $\pi_{\tau\circ\psi}$ and $i_N$ is the inversion anti-automorphism of $\GL_N(F) \rtimes \<\theta\>$. By Lemma \ref{lem:tfsw} (1), 
the functions $f^N \circ i_N$ and $f \circ i$ have matching orbital integrals 
with respect to $\Delta[\ww^{-1}, C_N(\ee), {}^LC_G \circ \tau]$, 
where $C_N$ (\resp ${}^LC_G$) is the Chevalley involution on $\GL_N$ (\resp ${}^LG$). 
Therefore (dropping the superscript * and agreeing that we are using an arbitrary but fixed convention) 
\begin{align*}
\psi(f) &= \tl\pi_{\tau\circ\psi}(f^N)\\
&= \tl\pi_{\tau\circ\psi}^\vee(f^N \circ i_N)\\
&= \tl\pi_{C_N\circ\tau\circ\psi}(f^N \circ i_N)\\
&= ({^LC_G}\circ\psi)(f \circ i) 
= \psi^\vee(f \circ i),
\end{align*}
proving Proposition \ref{pro:contra} (3).
\par

The proof of Proposition \ref{pro:switch} (3) is quite similar. 
Namely, if $f^N$ and $f$ match with respect to $\Delta_A[\ww,\ee,\zz]$, 
then $f^N \circ i_N$ and $f \circ i$ match with respect to $\Delta_D[\ww^{-1},\ee',\zz]$ 
according to Lemma \ref{lem:tfsw} (2). 
The switch from $\ww$ to $\ww^{-1}$ is now irrelevant 
since they are equivalent on $\GL_N$, 
and the switch from $\ee$ to $\ee'$ is also irrelevant, 
since the endoscopic element $s$ has order $2$. 
Hence we see that
\[ 
\psi^D(f\circ i) = \tl\pi_{\tau\circ\psi}^D(f^N \circ i_N) 
= \tl\pi^A_{\tau\circ\psi}(f^N) = \psi^A(f), 
\]
as desired.
\par

We now come to the global statements (2) of Propositions \ref{pro:contra} and \ref{pro:switch}. 
For Proposition \ref{pro:contra} (2) we have, using part (1) and  noting that $\epsilon_\psi$ is a quadratic character,
\begin{align*}
m(\pi,\psi) &= |\Sc_\psi|^{-1}\sum_{x \in \Sc_\psi}\epsilon_\psi(x)\pair{\pi,x}_\psi\\
&= |\Sc_\psi|^{-1}\sum_{x \in \Sc_\psi}\epsilon_\psi(x)\pair{\pi^\vee, \widehat{C}(x)}_{\psi^\vee}\\
&= |\Sc_{\psi^\vee}|^{-1}\sum_{x \in \Sc_{\psi^\vee}}\epsilon_\psi(\widehat{C}^{-1}(x))\pair{\pi^\vee,x}_{\psi^\vee},
\end{align*}
where we are using the global pairing $\pair{\pi,x}_\psi = \prod_v \pair{\pi_v,x}_{\ww,\psi_v}$ 
and suppress the Whittaker datum from the notation on the left, 
because its influence on the local pairings obeys a global product formula, 
so the global pairing is independent of it. 
\par

We claim that $\epsilon_\psi(\widehat{C}^{-1}(x)) = \epsilon_{\psi^\vee}(x)$ for all $x \in \Sc_{\psi^\vee}$. 
We recall a formula for $\epsilon_\psi$ in \cite[Section 8.3.5]{CL}.
Decompose 
\[
\psi = \tau_1[d_1] \boxplus \dots \boxplus \tau_r[d_r], 
\]
where $\tau_i$ is an irreducible (conjugate-)self-dual unitary cuspidal automorphic representation of $\GL_{m_i}(\A_E)$.
Then $\Sc_\psi$ is an elementary abelian $2$-group generated by elements $\{\alpha_{\tau_i[d_i]}\}_{i=1,\dots,r}$
with $\alpha_{\tau_i[d_i]}$ corresponding to $\tau_i[d_i]$, 
and 
\[
\epsilon_\psi(\alpha_{\tau_i[d_i]}) = \prod_{j \not= i} \ep(\tau_i \times \tau_j)^{\min\{d_i,d_j\}}, 
\]
where $\ep(\tau_i \times \tau_j) = \ep(1/2, \tau_i \times \tau_j) \in \{\pm1\}$ 
is the central value of the Rankin--Selberg epsilon factor.
Since $\widehat{C}^{-1}(\alpha_{\tau_i[d_i]}) = \alpha_{\tau_i^\vee[d_i]} \in \Sc_{\psi^\vee}$ and 
$\ep(\tau_i \times \tau_j)\ep(\tau_i^\vee \times \tau_j^\vee) = 1$, 
we have 
$\epsilon_\psi(\widehat{C}^{-1}(x)) = \epsilon_{\psi^\vee}(x)$ for $x = \alpha_{\tau_i[d_i]}$.
\par

Therefore, 
\[
m(\pi,\psi) = |\Sc_{\psi^\vee}|^{-1}\sum_{x \in \Sc_{\psi^\vee}}\epsilon_{\psi^\vee}(x)\pair{\pi^\vee,x}_{\psi^\vee}
= m(\pi^\vee,\psi^\vee), 
\]
which proves Proposition \ref{pro:contra} (2). 
Note in particular that it implies $L^2_{\psi^\vee}(G) = L^2_\psi(G)^\vee$.
\par

The computation for the proof of Proposition \ref{pro:switch} (2) is almost identical. 
Consider an irreducible constituent $\pi$ of $L^2_\psi(G)^D = (L^2_\psi(G)^A)^\vee$. 
The multiplicity $m^D(\pi,\psi)$ of $\pi$ in $L^2_\psi(G)^D \subset L^2_\disc(G)$ 
is equal (tautologically) to the multiplicity $m^A(\pi^\vee,\psi)$ of $\pi^\vee$ in $L^2_\psi(G)^A \subset L^2_\disc(G)$.
Therefore
\[ 
m^D(\pi,\psi) = |\Sc_\psi|^{-1}\sum_{x \in \Sc_\psi} \epsilon_\psi(x) \pair{\pi^\vee,x}_\psi^A 
= |\Sc_\psi|^{-1}\sum_{x \in \Sc_\psi}\epsilon_\psi(x)\pair{\pi,x^{-1}}_\psi^D.
\]
Since $\epsilon_\psi$ is a quadratic character, 
we have $\epsilon_\psi(x)=\epsilon_\psi(x^{-1})$, 
and we can reindex the sum to replace $x^{-1}$ by $x$. 
This completes the proofs of Propositions \ref{pro:contra} and \ref{pro:switch}, 
modulo the proof of Lemma \ref{lem:tfsw}, which we now give.

\begin{proof}[Proof of Lemma \ref{lem:tfsw}]
Part (1) is proved as \cite[Corollary 5.5]{Kal_gc}. Note just the slight change in notational convention. In \cite{Kal_gc} we are using test functions on the group $G(F)$ and are letting $G(F)$ act on itself by $\theta$-twisted conjugation, i.e. $hg\theta(h)^{-1}$, while now we are using test functions on the coset $G(F) \rtimes \theta$ in the group $G(F) \rtimes \<\theta\>$ and we are using the action of $G(F)$ on this coset coming from the action of $G(F)$ on $G(F) \rtimes \<\theta\>$ by usual conjugation. The translation between the two set-ups is by the map $g \mapsto g \rtimes \theta$, and we have $(f \circ i_\theta)( g \rtimes \theta) = f(\theta^{-1}(g^{-1}) \rtimes \theta^{-1})$.

The proof of part (2) is reduced in the same way to the identity
\[ 
\Delta_D[\ww^{-1},\ee',\zz](\gamma_1^{-1},\theta^{-1}(\delta^{-1})) 
= \Delta_A[\ww,\ee,\zz](\gamma_1,\delta). 
\]
To prove the latter, 
we follow the discussion of \cite[Proposition 5.4]{Kal_gc}. 
We write out as in (5.2) of loc.~cit.~
\[ 
\Delta_A[\ww,\ee,\zz](\gamma_1,\delta) 
= 
\epsilon(V_{G,H},\psi_F) \cdot \Delta_I^\new[\ee,\spl,a]^{-1}
\cdot \Delta_{II}[a,\chi] \cdot \Delta_{III}[\ee,\zz,\chi]^{-1} \cdot \Delta_{IV}, 
\]
where we have chosen arbitrarily $a$-data and $\chi$-data. 
We will now alter the pieces, without changing their product, in such a way as to see that 
their product equals $\Delta_D[\ww^{-1},\ee',\zz](\gamma_1^{-1},\theta^{-1}(\delta)^{-1})$.
\par

As discussed in \cite[(5.3)]{Kal_gc}, 
we have
\[ 
\Delta_{IV}(\gamma_1,\delta_1) = \Delta_{IV}(\gamma_1^{-1},\theta^{-1}(\delta_1)^{-1}).
\]
We will also keep $\psi_F$ fixed, 
which keeps $\epsilon(V_{G,H},\psi_F)$ unchanged.
\par

Next, since the choice of $a$-data is arbitrary, 
we may replace $a$ by $-a$ without changing the product of all pieces. 
Of course the individual pieces $\Delta_I$ and $\Delta_{II}$ change. 
We have 
\[ 
\Delta_I^\new[\ee,\spl,-a]^{-1} = \Delta_I^\new[\ee,-\spl,a]^{-1} = \Delta_I^\new[\ee',-\spl,a], 
\] 
where the first equality is according to \cite[Lemma 5.1]{Kal_gc}, 
and the second is by construction. 
The passage from $(\gamma_1,\delta)$ to $(\gamma_1^{-1},\theta^{-1}(\delta^{-1}))$ makes no difference to that piece. 
On the other hand, we have
\[ 
\Delta_{II}[-a,\chi](\gamma_1,\delta_1)
= \Delta_{II}[a,-\chi](\gamma_1^{-1},\theta^{-1}(\delta_1)^{-1}) 
= \Delta_{II}[a,\chi](\gamma_1^{-1},\theta^{-1}(\delta_1)^{-1})^{-1} 
\]
where the first identity is discussed in \cite[(5.5)]{Kal_gc} and the second is by construction.
\par

Next we claim
\[ 
\Delta_{III}[\ee,\zz,\chi](\gamma_1,\delta)^{-1} 
= \Delta_{III}[\ee',\zz,\chi](\gamma_1^{-1},\theta^{-1}(\delta^{-1})). 
\]
The proof is similar to that of \cite[(5.9)]{Kal_gc}, 
so we will just give a minimalistic sketch using the notation from that reference. 
The term $\Delta_{III}[\ee,\zz,\chi](\gamma_1,\delta)$ is obtained by pairing the cohomology classes 
$\inv(\gamma_1,\delta) \in H^1(F,S_\sc \xrightarrow{1-\theta_1} S_1)$, 
obtained as the pair $(\sigma(g)g^{-1},\delta_1^*)$, 
and $A_0[\ee,\zz,\chi] \in H^1(W_F,\widehat{S}_1 \xrightarrow{1-\widehat\theta_1} \widehat{S}_\ad)$, 
obtained as a pair $(a_S[\chi]^{-1}, s_S)$, via a hypercohomology version of Tate--Nakayama duality 
which blends the usual Tate--Nakayama isomorphism (in Artin normalization) 
and the usual Langlands isomorphism (in Deligne normalization). 
Switching $(\gamma_1,\delta)$ to $(\gamma_1^{-1},\theta^{-1}(\delta^{-1}))$ does not change $\sigma(g)g^{-1}$, 
but inverts $\delta_1^*$. 
Switching $\ee$ to $\ee'$ does not change $a_S[\chi]$, but inverts $s_S$. 
In other words, we are now pairing $(\sigma(g)g^{-1},\delta_1^{*,-1})$ with $(a_S[\chi]^{-1},s_S^{-1})$, 
both of which lie in the corresponding hypercohomology groups but with $\theta_1$ replaced by $\theta_1^{-1}$. 
Since the pairing is built by pairing $\delta_1^*$ with $a_S[\chi]^{-1}$ and $\sigma(g)g^{-1}$ with $s_S$, 
we see that the total outcome of the pairing is inverted.
\par

Combining the changes of the individual pieces, 
we arrive at
\[ 
\Delta_A[\ww,\ee,\zz](\gamma_1,\delta) 
= \epsilon(V_{G,H},\psi_F) \cdot \Delta_I^\new[\ee',-\spl,a] 
\cdot \Delta_{II}[a,\chi]^{-1} \cdot \Delta_{III}[\ee',\zz,\chi] \cdot \Delta_{IV}, 
\]
where now all pieces are built for the related pair $(\gamma_1^{-1},\theta^{-1}(\delta^{-1}))$. 
Noting that $\psi_F$ and $-\spl$ lead to the Whittaker datum $\ww^{-1}$, 
we see that the above product equals 
$\Delta_D[\ww^{-1},\ee',\zz](\gamma_1^{-1},\theta^{-1}(\delta)^{-1})$, 
as claimed.
\end{proof}


\end{document}